\documentclass[11pt,reqno]{amsart}

\usepackage{amssymb,amsmath,amsthm,amsfonts}
\usepackage{hyperref, multirow, bm, color, marginnote, graphicx, wrapfig, subcaption,esint}
\usepackage{diagbox}

\makeatletter
\renewcommand{\tocsection}[3]{%
  \indentlabel{\@ifnotempty{#2}{\bfseries\ignorespaces#1 #2\quad}}\bfseries#3}
\renewcommand{\tocsubsection}[3]{%
  \indentlabel{\@ifnotempty{#2}{\ignorespaces#1 #2\quad}}#3}

\newcommand\@dotsep{4.5}
\def\@tocline#1#2#3#4#5#6#7{\relax
  \ifnum #1>\c@tocdepth 
  \else
    \par \addpenalty\@secpenalty\addvspace{#2}%
    \begingroup \hyphenpenalty\@M
    \@ifempty{#4}{%
      \@tempdima\csname r@tocindent\number#1\endcsname\relax
    }{%
      \@tempdima#4\relax
    }%
    \parindent\z@ \leftskip#3\relax \advance\leftskip\@tempdima\relax
    \rightskip\@pnumwidth plus1em \parfillskip-\@pnumwidth
    #5\leavevmode\hskip-\@tempdima{#6}\nobreak
    \leaders\hbox{$\m@th\mkern \@dotsep mu\hbox{.}\mkern \@dotsep mu$}\hfill
    \nobreak
    \hbox to\@pnumwidth{\@tocpagenum{\ifnum#1=1\bfseries\fi#7}}\par
    \nobreak
    \endgroup
  \fi}
\AtBeginDocument{%
\expandafter\renewcommand\csname r@tocindent0\endcsname{0pt}
}
\def\l@subsection{\@tocline{2}{0pt}{2.5pc}{5pc}{}}
\makeatother

\usepackage[margin=1in]{geometry}

\newcommand{\R}{\mathbb{R}}
\newcommand{\Z}{\mathbb{Z}}
\newcommand{\N}{\mathbb{N}}

\renewcommand{\P}{\mathbb{P}}

\newcommand{\T}{\mathbb{T}}
\newcommand{\A}{\mathcal{A}}
\newcommand{\E}{\mathcal{E}}

\newcommand{\bR}{\bm{R}}
\newcommand{\barR}{\overline{\bm{R}}}

\newcommand{\bartheta}{\wh{\theta}}
\newcommand{\bars}{\wh s}

\newcommand{\bu}{\bm{u}}
\newcommand{\bw}{\bm{w}}
\newcommand{\bx}{\bm{x}}
\newcommand{\by}{\bm{y}}

\newcommand{\X}{\bm{X}}
\newcommand{\Y}{\bm{Y}}
\newcommand{\be}{\bm{e}}

\newcommand{\bv}{\bm{v}}
\newcommand{\p}{\partial}
\renewcommand{\div}{{\rm{div}\,}}

\newcommand{\abs}[1]{\left\lvert #1 \right\rvert}
\newcommand{\norm}[1]{\left\lVert #1 \right\rVert}
\newcommand{\wh}[1]{\widehat{#1}}
\newcommand{\wt}[1]{\widetilde{#1}}
\newcommand{\mc}[1]{\mathcal{#1}}
\DeclareMathOperator*{\esssup}{ess\,sup}

\newtheorem{theorem}{Theorem}[section]
\newtheorem{lemma}[theorem]{Lemma}
\newtheorem{proposition}[theorem]{Proposition}
\newtheorem{corollary}[theorem]{Corollary}

\theoremstyle{definition}

\allowdisplaybreaks

\begin{document}
\title{A free boundary problem for an immersed filament in 3D Stokes flow}

\author{Laurel Ohm}
\address{Department of Mathematics, University of Wisconsin - Madison, Madison, WI 53706}
\email{lohm2@wisc.edu}

\begin{abstract} 
We consider a simplified extensible version of a dynamic free boundary problem for a thin filament with radius $\epsilon>0$ immersed in 3D Stokes flow. The 3D fluid is coupled to the quasi-1D filament dynamics via a novel type of angle-averaged Neumann-to-Dirichlet operator for the Stokes equations, and much of the difficulty in the analysis lies in understanding this operator. Here we show that the main part of this angle-averaged NtD map about a closed, curved filament is the corresponding operator about a straight filament, for which we can derive an explicit symbol. Remainder terms due to curvature are lower order with respect to regularity or size in $\epsilon$. Using this operator decomposition, it is then possible to show that the simplified free boundary evolution is a third-order parabolic equation and is locally well posed. This establishes a more complete mathematical foundation for the myriad computational results based on slender body approximations for thin immersed elastic structures.
\end{abstract}

\maketitle

\tableofcontents

\setlength{\parskip}{8pt}
\section{Introduction}\label{sec:intro}
Thin elastic filaments immersed in very viscous fluids play a central role in many biological processes, including mitotic spindle positioning during cell division \cite{nazockdast2017cytoplasmic}, cytoskeletal mechanics \cite{maxian2021simulations}, sperm motility \cite{hines1978bend, camalet1999self}, and ciliary transport \cite{chakrabarti2022multiscale}. 
To capture the effective interaction between a thin semiflexible filament and the surrounding viscous fluid, it is natural to seek a lower dimensional representation of immersed filament dynamics. This is the basic tenet of slender body theories, which play an important role in many computational models.  

Given an immersed filament with finite radius $\epsilon>0$ and centerline position $\X(s,t):\R/\Z\times[0,T]\to \R^3$, the fundamental aim is to capture the filament evolution using an equation of the form 
\begin{equation}\label{eq:full_evolution}
\frac{\p\X}{\p t} = -\mc{L}_\epsilon(\X)\big[\p_s^4\X-\p_s(\tau\p_s\X)\big] \,, \quad \abs{\p_s\X}=1\,.
\end{equation}
Here $\mc{L}_\epsilon(\X)$ is an operator mapping force data defined along a 1D curve to a 3D velocity field governing the position of the filament centerline. This operator must also incorporate the effects of the surrounding 3D fluid on the filament evolution. The forcing terms $-\p_s^4\X+\p_s(\tau\p_s\X)$ come from Euler-Bernoulli beam theory and model a simple elastic response along the filament, with $\tau(s,t)$, the filament tension, serving as a Lagrange multiplier to enforce the inextensibility constraint $\abs{\p_s\X} = 1$. The key feature of this type of formulation \eqref{eq:full_evolution} is that the filament dynamics are 1D. A major difficulty is coupling these 1D dynamics and the surrounding 3D fluid via the operator $\mc{L}_\epsilon$ in both a physically and a mathematically reasonable way.

A simple and mathematically reasonable choice for this coupling operator is \emph{local slender body theory}, $\mc{L}_{\epsilon,{\rm loc}}^{\rm SB}(\X)=c\abs{\log\epsilon}({\bf I}+\p_s\X\otimes\p_s\X)$ \cite{gray1955propulsion, cox1970motion, batchelor1970slender, hines1978bend, wiggins1998flexive, camalet2000generic, camalet1999self}. For this very simple choice of coupling, the analogue of \eqref{eq:full_evolution} is a fourth order parabolic equation for the filament position coupled with a 1D elliptic equation for the filament tension $\tau$. In addition to establishing well-posedness in this setting, the author and Mori \cite{mori2023well,ohm2024well} prove that the filament can generate net forward motion (i.e. swim) with a suitable choice of time-periodic forcing. However, $\mc{L}_{\epsilon,{\rm loc}}^{\rm SB}$ incorporates only a very simplified description of the surrounding fluid physics on the filament evolution.

 To better capture the physics of the fluid-structure interaction, a natural attempt would be to use \emph{nonlocal} slender body theory $\mc{L}^{\rm SB}_{\epsilon,{\rm nloc}}(\X)$ \cite{keller1976slender,johnson1980improved, tornberg2004simulating}, which, in addition to the terms from $\mc{L}_{\epsilon,{\rm loc}}^{\rm SB}$, includes a 1D integral operator incorporating nonlocal interactions along the filament. This is a popular choice for computational models \cite{li2013sedimentation,spagnolie2011comparative,tornberg2006numerical,lauga2009hydrodynamics,cortez2012slender,NorwayPoF,norway_SBT,maxian2021integral,maxian2022hydrodynamics} because it is more physically realistic than $\mc{L}^{\rm SB}_{\epsilon,{\rm loc}}$ and the evolution \eqref{eq:full_evolution} is still relatively straightforward to compute numerically. However, the classical $\mc{L}^{\rm SB}_{\epsilon,{\rm nloc}}$ operator is known to have issues at high wavenumbers $\sim 1/\epsilon$ leading to nonsensical mapping properties \cite{gotz2000interactions, inverse, tornberg2004simulating, shelley2000stokesian}, even in analytic regularity. This issue does not typically affect numerics, as these high frequencies are truncated by discretization. Even so, from a PDE perspective, $\mc{L}^{\rm SB}_{\epsilon,{\rm nloc}}$ cannot give rise to a well-posed evolution equation \eqref{eq:full_evolution}.

Thus, a central question is how to incorporate more of the fluid physics into the operator $\mc{L}_\epsilon$ while still yielding a well-posed PDE evolution \eqref{eq:full_evolution}. In \cite{closed_loop,free_ends}, the author along with Mori and Spirn propose a novel candidate for the coupling operator $\mc{L}_\epsilon$ in the form of an angle-averaged Neumann-to-Dirichlet operator for the 3D Stokes equations. This notion of \emph{slender body Neumann-to-Dirichlet (NtD) map} is defined as the solution to a natural type of boundary value problem for the Stokes equations in the exterior of the thin filament (see equation \eqref{eq:SB_PDE}). It therefore incorporates much more of the fluid physics than local slender body theory but also has more reasonable mapping properties than nonlocal slender body theory. In \cite{closed_loop,free_ends}, we establish well-posedness for the static slender body boundary value problem in a natural energy space and moreover show that the solution is close in a suitable sense to nonlocal slender body theory. We have since extended this notion of slender body NtD map in various directions \cite{rigid,regularized,inverse,mitchell2022single,laplace}, but all analysis has thus far been in the static setting.

Here we give a more complete answer to the central question above by establishing local-in-time well-posedness for a simplified version of \eqref{eq:full_evolution} (see equation \eqref{eq:curve_evolution}) using the slender body NtD map $\mc{L}_\epsilon$. In our simplified dynamics, the filament is allowed to change length over time, although not necessarily in a physically realistic way. This allows us to develop the theory for the main part of the evolution \eqref{eq:full_evolution} without also needing to solve for the filament tension $\tau$. Developing this well-posedness theory requires a detailed understanding of the mapping properties of the map $\mc{L}_\epsilon$. The bulk of this paper is therefore devoted to proving Theorem \ref{thm:decomp}, a technically demanding decomposition of $\mc{L}_\epsilon$ into a main part plus lower order remainders which are smoother or small with respect to $\epsilon$. 

Using a layer potential argument, we show that the principal behavior of $\mc{L}_\epsilon$ is that of the corresponding operator about a straight cylinder with periodic boundary conditions, for which we can derive an explicit symbol. Using this explicit symbol, we show that the toy extensible version of \eqref{eq:full_evolution} is third-order parabolic and locally well-posed. This provides a proof-of-concept for using $\mc{L}_\epsilon$ to solve a dynamic free boundary problem, and represents an important step toward establishing well-posedness for the full evolution \eqref{eq:full_evolution} and a more complete mathematical justification of slender body theories of filament dynamics in Stokes flow.

Our extraction of the straight operator as the leading order behavior of $\mc{L}_\epsilon$ is a type of small-scale decomposition analogous to similar Dirichlet-to-Neumann operator decompositions used in other interfacial fluids problems. These include Hele-Shaw and water waves problems \cite{alazard2014cauchy, alazard2009paralinearization, beale1993growth,hou1994removing,lannes2005well}, the Muskat problem \cite{alazard2020paralinearization, flynn2021vanishing, nguyen2020paradifferential}, and the Peskin problem for the evolution of a 1D filament in 2D Stokes flow \cite{cameron2021critical, chen2021peskin, gancedo2020global, garcia2020peskin, lin2019solvability, kuo2023tension, mori2019well, tong2021regularized, tong2023geometric, garcia2023critical}. In many of these other interfacial fluids problems, however, the principal part of the operator is simpler to state and analyze. In contrast, the slender body NtD operator $\mc{L}_\epsilon$ requires a cylindrical geometry to be well-defined, and even for the straight cylinder, the symbol, while explicit, is not a simple expression. 
Considering a curved, closed filament as a perturbation of the straight filament and showing that the local effects due to curvature are lower order has also been used as a strategy to study vortex filament solutions of the Ginzburg-Landau, Navier-Stokes, and Euler equations \cite{bedrossian2018vortex, davila2022interacting, davila2022travelling}. We will also frequently refer to the author's earlier work \cite{laplace} deriving an analogous operator decomposition in the much simpler setting of the Laplace equation.

\subsection{Geometry}\label{subsec:geometry}
We begin by defining the filament geometry at a fixed instant in time (for example, the initial configuration of the fiber). The geometric quantities described here will play a major role in the eventual extraction of the principal behavior for the slender body NtD map (Theorem \ref{thm:decomp}). At this fixed time, we let $\X: \,\T:=\R/\Z \to \R^3$ denote the arclength parameterization of a closed, unit-length curve $\Gamma_0$ in $\R^3$ (see figure \ref{fig:filament}). We require that the curve belongs to $C^{2,\beta}(\T)$, $0<\beta<1$, and satisfies a non-self-intersection condition 
\begin{equation}\label{eq:cGamma}
 \inf_{s\neq s'}\frac{\abs{\X(s)-\X(s')}}{\abs{s-s'}} \ge c_\Gamma 
\end{equation}
for some constant $c_\Gamma>0$. Taking $\be_{\rm t}(s)=\frac{d\X}{ds}(s)$ to be the unit tangent vector to the curve at $s$, points sufficiently close to $\Gamma_0$ may be described by a $C^{1,\beta}$ orthonormal frame $(\be_{\rm t}(s),\be_{\rm n_1}(s),\be_{\rm n_2}(s))$ satisfying the ODEs
\begin{equation}\label{eq:frame}
\frac{d}{ds}
\begin{pmatrix}
\be_{\rm t}(s)\\
\be_{\rm n_1}(s)\\
\be_{\rm n_2}(s)
\end{pmatrix}
 = \begin{pmatrix}
 0 & \kappa_1(s) & \kappa_2(s) \\
 -\kappa_1(s) & 0 & \kappa_3 \\
-\kappa_2(s) & -\kappa_3& 0
 \end{pmatrix}\begin{pmatrix}
\be_{\rm t}(s)\\
\be_{\rm n_1}(s)\\
\be_{\rm n_2}(s)
\end{pmatrix}\,.
\end{equation}
The coefficients $\kappa_1$ and $\kappa_2$ are related to the shape of the curve via
\begin{align*}
\kappa_1^2(s) +\kappa_2^2(s) = \kappa^2(s)\,,
\end{align*}
where $\kappa^2(s)=\abs{\X_{ss}}^2$ is the squared curvature of $\X(s)$. The coefficient $\kappa_3$ may be taken to be a constant satisfying $\abs{\kappa_3}\le \pi$ (see \cite[Lemma 1.1]{closed_loop}). Throughout, we will use the notation
 \begin{equation}\label{eq:kappastar}
 \kappa_* := \sup_{s\in\T} \abs{\kappa(s)}, \quad \kappa_{*,\beta} := \sup_{s\neq s'\in\T} \frac{\abs{\X_{ss}(s)-\X_{ss}(s')}}{\abs{s-s'}^\beta} + \kappa_* \,.
 \end{equation}
 Note that, using \eqref{eq:frame}, the coefficients $\kappa_\ell$, $\ell=1,2$, satisfy
 \begin{align*}
   \norm{\kappa_\ell}_{C^{0,\beta}(\T)} = \norm{\X_{ss}\cdot\be_{{\rm n}_\ell}}_{C^{0,\beta}} \le \kappa_{*,\beta}\,.
 \end{align*}
 Given \eqref{eq:frame}, it will also be convenient to define the following cylindrical basis vectors: 
\begin{equation}\label{eq:er_etheta}
\be_r(s,\theta) = \cos\theta\be_{\rm n_1}(s) + \sin\theta\be_{\rm n_2}(s)\,,
\quad \be_\theta(s,\theta) = -\sin\theta\be_{\rm n_1}(s) + \cos\theta\be_{\rm n_2}(s)\,.
\end{equation}
Now, for $\Gamma_0$ as above, there exists a maximal radius
\begin{equation}\label{eq:rstar}
  0<r_* = r_*(c_\Gamma,\kappa_*) \le \min\bigg\{\frac{1}{2\kappa_*},\frac{c_\Gamma}{2}\bigg\}
\end{equation}
such that we may uniquely parameterize points $\bx$ satisfying ${\rm dist}(\bx,\Gamma_0)< r_*$ as
\begin{align*} 
\bx = \X(s) + r\be_r(s,\theta), \quad 0\le r <r_* \,.
\end{align*}
For any $0<\epsilon<\frac{r_*}{4}$, we may define a slender filament with uniform radius $\epsilon$ as
\begin{equation}\label{eq:SB_def} 
\Sigma_\epsilon:= \big\{\bx\in \R^3 \; : \; \bx = \X(s) + r\be_r(s,\theta)\,,  \; s\in\T\,, \; r < \epsilon\,, \; \theta\in2\pi\T \big\}\,.
\end{equation}
The filament surface $\Gamma_\epsilon = \p \Sigma_\epsilon$ is given by
\begin{align*}
 \Gamma_\epsilon :=  \big\{\bx\in \R^3 \; : \; \bx = \X(s) + \epsilon\be_r(s,\theta)\,,  \; s\in\T\,, \; \theta\in2\pi\T \big\}\,,
 \end{align*}
 with surface element $dS_\epsilon=\mc{J}_\epsilon(s,\theta)\,d\theta ds$ where $\mc{J}_\epsilon(s,\theta)$ is given by 
\begin{equation}\label{eq:jacfac}
\mc{J}_\epsilon(s,\theta) = \epsilon(1-\epsilon\wh\kappa(s,\theta))\,, \qquad
\wh\kappa(s,\theta):= \kappa_1(s)\cos\theta+\kappa_2(s)\sin\theta\,. 
\end{equation}

\begin{figure}[!ht]
\centering
\includegraphics[scale=0.3]{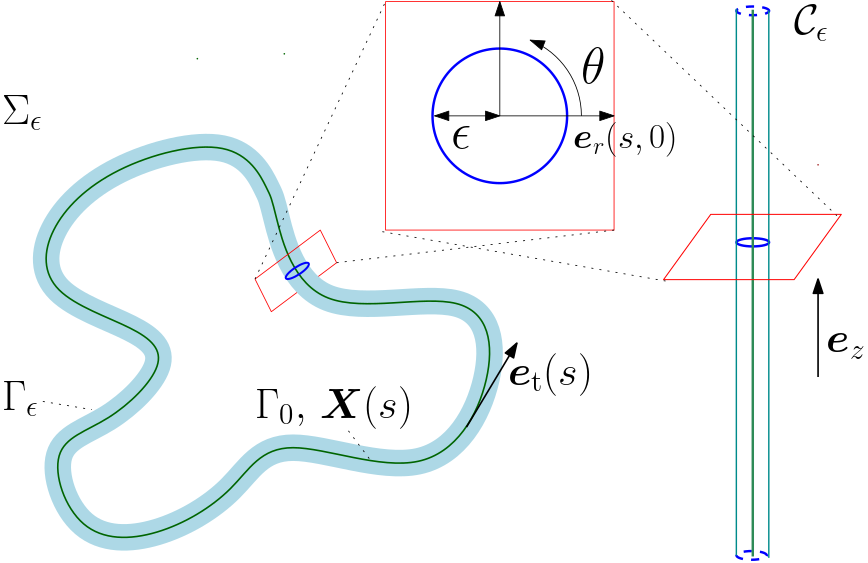}
\caption{A typical curved filament $\Sigma_\epsilon$ as in \eqref{eq:SB_def} in comparison with the straight filament $\mc{C}_\epsilon$ given by \eqref{eq:Cepsilon}.}
\label{fig:filament}
\end{figure}

Let $\Sigma_\epsilon^{(a)}$ and $\Sigma_\epsilon^{(b)}$ be two closed filaments as in \eqref{eq:SB_def} with centerlines parameterized by $\X^{(a)}(s)$ and $\X^{(b)}(s)$, respectively. We will be concerned with pairs of curves that are ``close" in the following sense. Choose the basepoints $\X^{(a)}(0)$, $\X^{(b)}(0)$ to be the nearest points on the two curves, i.e.
\begin{align*}
\min_{s_1,s_2\in \T} \abs{\X^{(a)}(s_1)-\X^{(b)}(s_2)}\,.
\end{align*}
Note that there may be multiple pairs of nearest points, in which case we may then select one such pair arbitrarily. Given this choice of basepoint, we consider pairs of curves satisfying  
\begin{equation}\label{eq:XaXb_close}
\norm{\X^{(a)}-\X^{(b)}}_{C^{2,\beta}(\T)} =\delta
\end{equation}
for $\delta\ge 0$ sufficiently small. 
For two such nearby curves, the corresponding orthonormal frames may also be chosen to be close. In particular, from \eqref{eq:XaXb_close} we have that $\be_{\rm t}^{(a)}=\frac{d\X^{(a)}}{ds}$ and $\be_{\rm t}^{(b)}=\frac{d\X^{(b)}}{ds}$ satisfy $\norm{\be_{\rm t}^{(a)}-\be_{\rm t}^{(b)}}_{C^{1,\beta}(\T)}\le \delta$, and there exists a systematic choice of normal vectors along each curve such that the frames $(\be_{\rm t}^{(a)},\be_{\rm n_1}^{(a)},\be_{\rm n_2}^{(a)})$, $(\be_{\rm t}^{(b)},\be_{\rm n_1}^{(b)},\be_{\rm n_2}^{(b)})$ both satisfy \eqref{eq:frame} and the following lemma holds. 
\begin{lemma}[Nearby parameterizations for nearby curves]\label{lem:XaXb_C2beta}
Given two closed filaments with centerlines $\X^{(a)}(s)$, $\X^{(b)}(s)$ satisfying \eqref{eq:XaXb_close}, there exists a choice of orthonormal frames $(\be_{\rm t}^{(a)},\be_{\rm n_1}^{(a)},\be_{\rm n_2}^{(a)})$ and $(\be_{\rm t}^{(b)},\be_{\rm n_1}^{(b)},\be_{\rm n_2}^{(b)})$, respectively, satisfying \eqref{eq:frame} such that 
\begin{equation}\label{eq:XaXb_C2beta}
\sum_{\ell=1}^3\norm{\kappa_\ell^{(a)}-\kappa_\ell^{(b)}}_{C^{0,\beta}(\T)} \le c(\kappa_{*,\beta}^{(a)},\kappa_{*,\beta}^{(b)})\,\norm{\X^{(a)}-\X^{(b)}}_{C^{2,\beta}(\T)}\,.
\end{equation}
\end{lemma}
The proof of Lemma \ref{lem:XaXb_C2beta} appears in Appendix \ref{app:XaXb}. Given two sufficiently nearby filaments, we will always choose a parameterization satisfying the estimate \eqref{eq:XaXb_C2beta}.

\subsection{Slender body Neumann-to-Dirichlet map}\label{subsec:SB_NtD}
To study an evolution equation of the form \eqref{eq:full_evolution}, we must first define the map $\mc{L}_\epsilon$ being used to map force data defined along a 1D curve to a 3D velocity field governing the motion of the filament. Here we define $\mc{L}_\epsilon$ as the solution to the following \emph{slender body boundary value problem}, first introduced by the author along with Mori and Spirn in \cite{closed_loop}. Given force data $\bm{f}(s)$, $s\in\T$, we consider $(\bu,p)$ satisfying
\begin{equation}\label{eq:SB_PDE}
\begin{aligned}
-\Delta \bu +\nabla p &= 0\,, \quad \div\bu=0 \qquad \text{in }\Omega_\epsilon = \R^3\backslash \overline{\Sigma_\epsilon} \\
\int_0^{2\pi}(\bm{\sigma}\bm{n}) \, \mc{J}_\epsilon(s,\theta)\,d\theta &= \bm{f}(s) \qquad\qquad\qquad\; \text{on }\Gamma_\epsilon \\
\bu\big|_{\Gamma_\epsilon} &= \bv(s)\,, \qquad\qquad\quad\;\;\, \text{unknown but independent of }\theta\,,   
\end{aligned}
\end{equation}
and $\abs{\bu}\to 0$ as $\abs{\bx}\to\infty$. Here $\bm{\sigma}=\bm{\sigma}[\bu]= \frac{1}{2}(\nabla\bu+\nabla\bu^{\rm T})-p{\bf I}$ is the viscous stress tensor and $\bm{n}(\bx)$ is the unit normal vector at $\bx\in\Gamma_\epsilon$. The force data $\bm{f}(s)$ is understood as the total surface stress $\bm{\sigma}[\bu]\bm{n}$ over each cross section of the filament, weighted by the fiber surface area via the Jacobian factor \eqref{eq:jacfac}. The corresponding Dirichlet boundary value $\bv(s)$ along the filament surface is \emph{a priori} unknown but constrained to be constant over each cross section of the fiber.

In \cite{closed_loop,free_ends,rigid}, the slender body BVP is shown to be well posed in a natural energy space (see section \ref{sec:SBNtD_holder} for further details). Moreover, given $\bm{f}\in C^1(\T)$, the solution $\bu(\bx)$ to \eqref{eq:SB_PDE} and its Dirichlet boundary value $\bv(s)$ are shown to be close in a quantitative sense to the corresponding nonlocal slender body theory expressions. In particular,
\begin{align*}
\norm{\bv -\mc{L}_{\epsilon,\rm nloc}^{\rm SB}[\bm{f}]}_{L^2(\T)}\le c(\kappa_*,c_\Gamma)\,\epsilon\abs{\log\epsilon}^{3/2}\norm{\bm{f}}_{C^1(\T)}\,.   
\end{align*}   
We define the Stokes \emph{slender body Neumann-to-Dirichlet (NtD) map} $\mc{L}_\epsilon$ as the operator
\begin{equation}\label{eq:SB_NtD}
\mc{L}_\epsilon : \bm{f}(s) \mapsto \bm{v}(s)
\end{equation}
mapping the total force per cross section $\bm{f}(s)$ to the $\theta$-independent velocity $\bv(s)$ along $\Gamma_\epsilon$ by solving the boundary value problem \eqref{eq:SB_PDE}. Our analysis hinges on a detailed understanding of this map.
Given the nonlocal nature of the boundary conditions in \eqref{eq:SB_PDE}, our analysis of $\mc{L}_\epsilon$ will rely heavily on the fact that the inverse map, the \emph{slender body Dirichlet-to-Neumann (DtN) map},
\begin{equation}\label{eq:SB_DtN}
\mc{L}_\epsilon^{-1} : \bm{v}(s) \mapsto \bm{f}(s)\,,
\end{equation}
is much simpler to study. In particular, $\mc{L}_\epsilon^{-1}$ is just the usual Stokes Dirichlet-to-Neumann map along a surface using $\theta$-independent Dirichlet data, with an additional step of angle-averaging the resulting Neumann data.

To study the maps $\mc{L}_\epsilon$ and $\mc{L}_\epsilon^{-1}$ in detail, we will work in the setting of H\"older spaces along the filament surface $\Gamma_\epsilon$, as this will be a natural setting for analyzing the boundary integral formulation of \eqref{eq:SB_PDE} considered in section \ref{subsubsec:SB_DtN_boundaryintegral}. For a vector-valued function $\bm{h}$ defined along $\Gamma_\epsilon$, we denote 
\begin{equation}\label{eq:norms}
\norm{\bm{h}}_{L^\infty(\Gamma_\epsilon)} = \esssup_{\bx\in\Gamma_\epsilon}\abs{\bm{h}(\bx)}\,, \qquad
\abs{\bm{h}}_{\dot C^{0,\alpha}(\Gamma_\epsilon)} = \sup_{\bx\neq\bx'\in\Gamma_\epsilon}\frac{\abs{ \bm{h}(\bx)-\bm{h}(\bx')}}{\abs{\bx-\bx'}^\alpha}\,, \quad 0<\alpha<1\,, 
\end{equation}
and define the H\"older spaces $C^{k,\alpha}(\Gamma_\epsilon)$ as
\begin{equation}\label{eq:Ckalpha}
C^{k,\alpha}(\Gamma_\epsilon) = \{ \bm{h} \; : \; \norm{\bm{h}}_{C^{k,\alpha}(\Gamma_\epsilon)} <\infty\}\,, \qquad 
\norm{\bm{h}}_{C^{k,\alpha}(\Gamma_\epsilon)} = \sum_{\ell=0}^k\norm{\nabla^\ell \bm{h}}_{L^\infty(\Gamma_\epsilon)}+\abs{\nabla^k \bm{h}}_{\dot C^{0,\alpha}(\Gamma_\epsilon)}\,.
\end{equation}

\subsubsection{Evolution problem and theorem statement}\label{subsubsec:thm_statement}
Given the map $\mc{L}_\epsilon$ defined by solving the slender body boundary value problem \eqref{eq:SB_PDE}, we consider a simplified version of \eqref{eq:full_evolution} without the inextensibility constraint. This simplifies many aspects of the evolution, as we avoid the need to simultaneously solve a \emph{tension determination problem} for $\tau(s,t)$ along with the evolution equation for $\X$ \cite{kuo2023tension,mori2023well,ohm2024well}. This allows us to focus on the behavior of $\mc{L}_\epsilon$ as the main difficulty in this paper, and develop the theory for the main part of the evolution \eqref{eq:full_evolution}, namely, $\mc{L}_\epsilon$ applied to a fourth derivative of the filament centerline. The tension determination problem with $\mc{L}_\epsilon$ will be considered in future work.

However, we now have to consider a filament that is possibly growing or shrinking in length over time in a non-uniform way. To treat this possible filament length change in the simplest way, we choose to artificially introduce some rescalings such that the main operator always acts on a unit length curve. We emphasize that this choice of dynamics has the same main term as \eqref{eq:full_evolution}.


For an extensible filament, we will consider the evolution of $\X^\sigma(\sigma,t)$ where $0\le\sigma\le 1$ is now a material coordinate rather than an arclength coordinate. Let
\begin{equation}\label{eq:lambda_def}
\lambda(t) = \int_0^1\abs{\p_\sigma\X^\sigma(\sigma,t)}\,d\sigma
\end{equation}
denote the length of the filament at time $t$. We will require the filament to have length 1 at the initial time, with $\abs{\p_\sigma(\X^\sigma)^{\rm in}}=1$ at each $\sigma$. 

The NtD operator $\mc{L}_\epsilon$ depends on two arguments: a curve in arclength parameterization and a function of the arclength parameter. In \eqref{eq:SB_PDE}, the problem is introduced for a unit length curve, but may be readily extended to more general curves. By making the change of variables $\wt{\bx}=\lambda\bx$ in \eqref{eq:SB_PDE}, we can see how $\mc{L}_\epsilon$ changes under uniform rescaling. For now, we suppress the time dependence in our notation. Let $\wt s=\lambda s$ for $0\le s\le 1$ and denote $\wt{\bv}(\wt s)=\bv(s)$, $\wt{\bm{f}}(\wt s)=\bm{f}(s)$, and $\wt\X(\wt s)=\X(s)$. We then have
\begin{equation}\label{eq:rescaling}
\mc{L}_{\lambda\epsilon}(\wt\X)[\wt{\bm{f}}](\wt s)=\lambda^{-1}\wt\bv(\wt s) = \lambda^{-1}\bv(s) = \lambda^{-1}\mc{L}_\epsilon(\X)[\bm{f}](s)\,.
\end{equation}
We will consider the evolution
\begin{equation}\label{eq:pre_evo}
\frac{\p\X^\sigma}{\p t}(\sigma,t) = \mc{L}_{\lambda\epsilon}(\wt\X)[\wt{\bm{f}}(\wt s,t)](\lambda\sigma) \,,
\end{equation}
where $\X^\sigma$ parameterizes the curve of interest and $\wt{\X}$ is the same curve parameterized by arclength $\wt s$. Here we will use the forcing term
\begin{equation}\label{eq:wtf_def}
\wt{\bm{f}}(\wt s,t) = -(\p_\sigma^4\X^\sigma)\big|_{\sigma=\wt s/\lambda}\,,
\end{equation}
i.e., we take the fourth derivative with respect to the material parameter and evaluate at the rescaled arclength parameter $0\le \wt s/\lambda(t)\le 1$, where the rescaling occurs around a fixed basepoint $\X(0,t)$. Note that the material points and rescaled arclength parameter do not necessarily match up precisely after the initial time. An additional nonphysical feature is that for convenience we will take a fixed \emph{ratio} $\epsilon$ between the filament radius and length, rather than having a fixed filament radius.

Putting \eqref{eq:rescaling}, \eqref{eq:pre_evo}, and \eqref{eq:wtf_def} together, we obtain the curve evolution
\begin{equation}\label{eq:curve_evolution}
\begin{aligned}
\frac{\p\X^\sigma}{\p t}(\sigma,t) &= -\lambda(t)^{-1}\mc{L}_\epsilon(\wt R_1[\X^\sigma])\big[\p_\sigma^4\X^\sigma\big]\,, \\
\X^\sigma(\sigma,0)&= (\X^\sigma)^{\rm in}(\sigma)\,, \qquad \abs{\p_\sigma(\X^\sigma)^{\rm in}}=1\,,\\
\lambda(t)&= \int_0^1\abs{\p_\sigma\X^\sigma}\,d\sigma\,.
\end{aligned}
\end{equation}
Here we use $\wt R_1[\X^\sigma](s)$ to denote the arclength parameterization of the uniform rescaling of the curve $\X^\sigma$ to length 1. 
We show the following. 
\begin{theorem}[Local well-posedness for filament evolution]\label{thm:main}
Given $(\X^\sigma)^{\rm in}\in C^{4,\alpha}(\T)$ satisfying \eqref{eq:cGamma}, there exists $\epsilon_0>0$ depending only on $\norm{(\X^\sigma)^{\rm in}}_{C^{4,\alpha}(\T)}$ and $c_\Gamma$ such that for $0<\epsilon<\epsilon_0$, there exists a time $T>0$ depending on $\epsilon$, $\norm{(\X^\sigma)^{\rm in}}_{C^{4,\alpha}(\T)}$, and $c_\Gamma$ such that the evolution \eqref{eq:curve_evolution} admits a unique solution $\X^\sigma(\sigma,t)\in L^\infty(0,T;C^{4,\alpha}(\T))$.
\end{theorem}
The proof of Theorem \ref{thm:main} in section \ref{subsec:pf_dynamics} is relatively straightforward by design. It serves mainly as a proof-of-concept for the use of the slender body NtD operator $\mc{L}_\epsilon$ for a dynamic curve evolution problem. The main content of this paper is a detailed decomposition of $\mc{L}_\epsilon$ that we will introduce in the next sections and summarize in Theorem \ref{thm:decomp}. The decomposition of $\mc{L}_\epsilon=\mc{L}_\epsilon(\X)$ will be presented for a unit length centerline curve $\X(s)$ in arclength parameterization. We will typically suppress the curve dependence of the operator $\mc{L}_\epsilon$ in our notation. The material parameter $\sigma$ and curve $\X^\sigma$ will only be used to define the evolution \eqref{eq:curve_evolution} and prove Theorem \ref{thm:main}; everything else will be in terms of $\X(s)$.

\subsubsection{Slender body NtD map for the straight cylinder}\label{subsubsec:SB_NtD_straight}
The proof of Theorem \ref{thm:main} relies heavily on the fact that we can derive very detailed information about the slender body NtD and its inverse, the slender body DtN, for the simple geometry of a straight cylinder with periodic boundary conditions. We define 
\begin{equation}\label{eq:Cepsilon} 
\mc{C}_\epsilon :=  \big\{\bx\in \R^2\times\T \; : \; \bx = s\be_z + \epsilon\be_r(\theta)\,,  \; s\in\T\,, \; \theta\in2\pi\T \big\}\,.
\end{equation}

Let $\overline{\mc{L}}_\epsilon$ and $\overline{\mc{L}}_\epsilon^{-1}$ denote the slender body NtD and DtN maps along $\mc{C}_\epsilon$. Due to symmetry, the behavior of $\overline{\mc{L}}_\epsilon$ and $\overline{\mc{L}}_\epsilon^{-1}$ in this simple geometry decouples into directions tangential ($\be_z$) and normal ($\be_x,\be_y$) to the filament centerline. The eigenfunctions of $\overline{\mc{L}}_\epsilon$ and $\overline{\mc{L}}_\epsilon^{-1}$ in each direction are then just $e^{2\pi iks}\be_z$, $e^{2\pi iks}\be_x$, $e^{2\pi iks}\be_y$, and we may consider the eigenvalues $m_{\epsilon,{\rm t}}^{-1}(k)$ and $m_{\epsilon,{\rm n}}^{-1}(k)$ of the slender body DtN $\overline{\mc{L}}_\epsilon^{-1}$ satisfying  
\begin{equation}\label{eq:eval_prob}
\overline{\mc{L}}_\epsilon^{-1}[e^{2\pi iks}\be_z] = m_{\epsilon,{\rm t}}^{-1}(k)e^{2\pi iks}\be_z\,, \qquad 
\overline{\mc{L}}_\epsilon^{-1}[e^{2\pi iks}\be_x] = m_{\epsilon,{\rm n}}^{-1}(k)e^{2\pi iks}\be_x\,.
\end{equation} 
Note that by symmetry the $\be_y$ direction is identical to the $\be_x$ direction. 
In \cite{inverse}, the following explicit expressions are calculated for $m_{\epsilon,{\rm t}}^{-1}(k)$ and $m_{\epsilon,{\rm n}}^{-1}(k)$.
\begin{proposition}[Slender body DtN spectrum]\label{prop:Leps_spectrum}
The eigenvalues $m_{\epsilon,{\rm t}}^{-1}(k)$, $m_{\epsilon,{\rm n}}^{-1}(k)$ of the straight slender body DtN map $\overline{\mc{L}}_\epsilon^{-1}$ satisfying \eqref{eq:eval_prob} are given by
\begin{align}
\label{eq:eigsT}
m_{\epsilon,{\rm t}}^{-1}(k) &= \frac{ 8\pi^2\epsilon \abs{k} K_1^2}{2K_0K_1 + 2\pi\epsilon \abs{k} \big( K_0^2 - K_1^2 \big) } \\
\label{eq:eigsN}
m_{\epsilon,{\rm n}}^{-1}(k) &= 
4\pi^2\epsilon \abs{k}\frac{4K_1^2K_2+2\pi\epsilon \abs{k} K_1(K_1^2-K_0K_2)}{2K_0K_1K_2 + 2\pi\epsilon \abs{k} \big(K_1^2(K_0+K_2)-2K_0^2K_2 \big)}
\end{align}
where each $K_j=K_j(2\pi\epsilon \abs{k})$, $j=0,1,2$, is a $j^{\text th}$ order modified Bessel function of the second kind. 
\end{proposition}
In \cite{inverse}, both $m_{\epsilon,{\rm t}}^{-1}(k)$ and $m_{\epsilon,{\rm n}}^{-1}(k)$ are shown to grow like $\epsilon\abs{k}$ as $\abs{k}\to\infty$; in particular, $\overline{\mc{L}}_\epsilon^{-1}(\T)$ maps the Sobolev space $H^m$ to $H^{m-1}(\T)$. In addition, at low wavenumbers $\abs{k}\lesssim\frac{1}{\epsilon}$, both $m_{\epsilon,{\rm t}}^{-1}(k)$ and $m_{\epsilon,{\rm n}}^{-1}(k)$ scale as $\abs{\log(\epsilon |k|)}^{-1}$ (see figure \ref{fig:multipliers}). Here we will be working in H\"older spaces and thus will require a more detailed characterization of the behavior of $m_{\epsilon,{\rm t}}^{-1}(k)$ and $m_{\epsilon,{\rm n}}^{-1}(k)$. We use the explicit expressions \eqref{eq:eigsT} and \eqref{eq:eigsN} to show the following.
\begin{lemma}[Mapping properties of SB NtD and DtN along $\mc{C}_\epsilon$]\label{lem:straight_Leps_Holder}
Let $\overline{\mc{L}}_\epsilon$ denote the slender body NtD map \eqref{eq:SB_NtD} along the straight filament $\mc{C}_\epsilon$. Given total force data $\bm{f}(s)\in C^{0,\alpha}(\T)$ along $\mc{C}_\epsilon$ with $\int_\T\bm{f}(s)\,ds=0$, we have 
\begin{equation}\label{eq:holder_NtD}
\norm{\overline{\mc{L}}_\epsilon[\bm{f}]}_{C^{1,\alpha}(\T)} \le c\abs{\log\epsilon}\big(\norm{\bm{f}}_{L^\infty(\T)}+\epsilon^{-1}\abs{\bm{f}}_{\dot C^{0,\alpha}(\T)} \big)\,.
\end{equation}
Letting $\overline{\mc{L}}_\epsilon^{-1}$ denote the SB DtN map \eqref{eq:SB_DtN} along $\mc{C}_\epsilon$, given a $\theta$-independent velocity field $\bv(s)\in C^{1,\alpha}(\T)$ along $\mc{C}_\epsilon$, we have 
\begin{equation}\label{eq:holder_DtN}
\norm{\overline{\mc{L}}_\epsilon^{-1}[\bv]}_{C^{0,\alpha}(\T)} \le c\abs{\log\epsilon}^{-1}\norm{\bv}_{C^{1,\alpha}(\T)}\,.
\end{equation}
\end{lemma}
The proof of Lemma \ref{lem:straight_Leps_Holder} appears in section \ref{subsec:straight_mapping}.

\subsubsection{Decomposition of the slender body DtN and NtD maps}\label{subsubsec:thm_decomp}
The proof of Theorem \ref{thm:main} relies on a detailed decomposition of the slender body DtN and NtD maps for arbitrarily curved filaments. In particular, given a curved filament, we extract the straight operators $\overline{\mc{L}}_\epsilon^{-1}$, $\overline{\mc{L}}_\epsilon$ as the leading order behavior and show that the effects of curvature are lower order: either small in $\epsilon$ or more regular. Given the explicit expressions for the behavior of $\overline{\mc{L}}_\epsilon^{-1}$ and $\overline{\mc{L}}_\epsilon$, this decomposition will allow us to go on to prove local well-posedness for the evolution \eqref{eq:curve_evolution}. 

In order to accomplish this, we will need to define a way of identifying vector valued functions defined on curvilinear basis vectors $(\be_{\rm t}(s),\be_{\rm n_1}(s),\be_{\rm n_2}(s))$ along $\Gamma_\epsilon$ with vector valued functions defined on $(\be_z,\be_x,\be_y)$ along $\mc{C}_\epsilon$. We define a map $\Phi$ identifying tangential and normal components of vectors along $\mc{C}_\epsilon$ with tangential and normal components of vectors along $\Gamma_\epsilon$. In particular, given a vector-valued function $\bm{g}$ along the straight filament $\mc{C}_\epsilon$ and vector-valued $\bm{h}$ about the curved filament $\Gamma_\epsilon$, we identify the tangential and normal components in the straight versus curved setting via 
\begin{equation}\label{eq:mapPhi_def}
\begin{aligned}
\Phi\bm{g} &= (\bm{g}\cdot\be_z)\be_{\rm t} + (\bm{g}\cdot\be_x)\be_{\rm n_1} + (\bm{g}\cdot\be_y)\be_{\rm n_2} \\
\Phi^{-1}\bm{h} &= (\bm{h}\cdot\be_{\rm t})\be_z + (\bm{h}\cdot\be_{\rm n_1})\be_x + (\bm{h}\cdot\be_{\rm n_2})\be_y\,.
\end{aligned}
\end{equation}
For any $\bm{h}$ defined along $\Gamma_\epsilon$, we define
\begin{equation}\label{eq:subtract_mean}
\bm{h}_0^\Phi := \bm{h}-\Phi\int_\T(\Phi^{-1}\bm{h})\,ds\,;
\end{equation}
in particular, $\Phi^{-1}\bm{h}_0^\Phi$ is mean-zero in $s$ along the straight filament $\mc{C}_\epsilon$.

Given the definition \eqref{eq:mapPhi_def} of $\Phi$, we show the following. 
\begin{theorem}[Decomposition of SB DtN and NtD]\label{thm:decomp}
Let $0<\alpha<\gamma<\beta<1$ and consider a closed filament $\Sigma_\epsilon$ with centerline $\X(s)\in C^{2,\beta}(\T)$. Let $\bv(s)\in C^{1,\alpha}(\T)$ be $\theta$-independent Dirichlet data along $\Gamma_\epsilon$. There exists $\epsilon_{\rm d}=\epsilon_{\rm d}(c_\Gamma,\kappa_*)>0$ such that for $0<\epsilon\le\epsilon_{\rm d}$, the slender body Dirichlet-to-Neumann operator $\mc{L}_\epsilon^{-1}$ as defined in \eqref{eq:SB_DtN} may be decomposed as 
\begin{equation}\label{eq:thm_DtN_decomp}
\mc{L}_\epsilon^{-1}[\bv](s) = \Phi\overline{\mc{L}_\epsilon}^{-1}[\Phi^{-1}\bv(s)] + \mc{R}_{\rm d,\epsilon}[\bv(s)]+\mc{R}_{\rm d,+}[\bv(s)]
\end{equation}
where the remainder terms satisfy 
\begin{equation}\label{eq:thm_DtN_ests}
\begin{aligned}
\norm{\mc{R}_{\rm d,\epsilon}[\bv]}_{C^{0,\alpha}(\T)} &\le c(\kappa_{*,\alpha^+},c_\Gamma)\, \epsilon^{2-\alpha^+}\norm{\bv}_{C^{1,\alpha}(\T)} \\
\norm{\mc{R}_{\rm d,+}[\bv]}_{C^{0,\gamma}(\T)} &\le c(\epsilon,\kappa_{*,\gamma^+},c_\Gamma)\norm{\bv}_{C^{1,\alpha}(\T)}
\end{aligned}
\end{equation}
for any $\alpha^+\in(\alpha,\beta]$ and $\gamma^+\in(\gamma,\beta]$. 
Furthermore, given two nearby filaments with centerlines $\X^{(a)}(s)$, $\X^{(b)}(s)$ satisfying Lemma \ref{lem:XaXb_C2beta}, the differences between their corresponding remainder terms $\mc{R}_{\rm d,\epsilon}^{(a)}-\mc{R}_{\rm d,\epsilon}^{(b)}$ and $\mc{R}_{\rm d,+}^{(a)}-\mc{R}_{\rm d,+}^{(b)}$ satisfy 
\begin{equation}\label{eq:thm_DtN_ests_lip}
\begin{aligned}
\norm{(\mc{R}_{\rm d,\epsilon}^{(a)}-\mc{R}_{\rm d,\epsilon}^{(b)})[\bv]}_{C^{0,\alpha}(\T)} &\le c(\kappa_{*,\alpha^+}^{(a)},\kappa_{*,\alpha^+}^{(b)},c_\Gamma)\, \epsilon^{2-\alpha^+}\norm{\X^{(a)}-\X^{(b)}}_{C^{2,\alpha^+}(\T)}\norm{\bv}_{C^{1,\alpha}(\T)} \\
\norm{(\mc{R}_{\rm d,+}^{(a)}-\mc{R}_{\rm d,+}^{(b)})[\bv]}_{C^{0,\gamma}(\T)} &\le c(\epsilon,\kappa_{*,\gamma^+}^{(a)},\kappa_{*,\gamma^+}^{(b)},c_\Gamma)\norm{\X^{(a)}-\X^{(b)}}_{C^{2,\gamma^+}(\T)}\norm{\bv}_{C^{1,\alpha}(\T)}\,.
\end{aligned}
\end{equation}

For a filament $\Sigma_\epsilon$ with centerline $\X(s)\in C^{3,\alpha}(\T)$, let $\bm{f}(s)\in C^{0,\alpha}(\T)$ be slender body force data. There exists $\epsilon_{\rm n}=\epsilon_{\rm n}(c_\Gamma,\kappa_{*,\gamma})>0$ such that for $0<\epsilon\le\epsilon_{\rm n}$, the slender body Neumann-to-Dirichlet operator $\mc{L}_\epsilon$ given by \eqref{eq:SB_NtD} may be similarly decomposed as 
\begin{equation}\label{eq:thm_NtD_decomp}
\mc{L}_\epsilon[\bm{f}](s) = \big({\bf I} + \mc{R}_{\rm n,\epsilon}\big)\big[\Phi\overline{\mc{L}}_\epsilon[\Phi^{-1}\bm{f}_0^\Phi(s)]\big] +\mc{R}_{\rm n,+}[\bm{f}(s)]
\end{equation}
where the remainder terms satisfy
\begin{equation}\label{eq:thm_NtD_ests}
\begin{aligned}
\norm{\mc{R}_{\rm n,\epsilon}[\bm{g}]}_{C^{1,\alpha}(\T)}&\le c(\kappa_{*,\alpha^+},c_\Gamma)\,\epsilon^{1-\alpha^+}\abs{\log\epsilon}\norm{\bm{g}}_{C^{1,\alpha}(\T)}\\
\norm{\mc{R}_{\rm n,+}[\bm{f}]}_{C^{1,\gamma}(\T)} &\le c(\epsilon,\|\X_{ss}\|_{C^{1,\alpha}},c_\Gamma)\norm{\bm{f}}_{C^{0,\alpha}(\T)}
\end{aligned}
\end{equation}
for any $\alpha^+\in(\alpha,\beta]$.
Given two nearby filaments with centerlines $\X^{(a)}(s)$, $\X^{(b)}(s)$ in $C^{3,\alpha}$ and satisfying Lemma \ref{lem:XaXb_C2beta}, we also have
\begin{equation}\label{eq:thm_NtD_ests_lip}
\begin{aligned}
\norm{(\mc{R}_{\rm n,\epsilon}^{(a)}-\mc{R}_{\rm n,\epsilon}^{(b)})[\bm{g}]}_{C^{1,\alpha}(\T)} &\le c(\kappa_{*,\alpha^+}^{(a)},\kappa_{*,\alpha^+}^{(b)},c_\Gamma)\, \epsilon^{1-\alpha^+}\abs{\log\epsilon}\norm{\X^{(a)}-\X^{(b)}}_{C^{2,\alpha^+}}\norm{\bm{g}}_{C^{1,\alpha}(\T)} \\
\norm{(\mc{R}_{\rm n,+}^{(a)}-\mc{R}_{\rm n,+}^{(b)})[\bm{f}]}_{C^{1,\gamma}(\T)}
&\le c(\epsilon,\|\X_{ss}^{(a)}\|_{C^{1,\alpha}},\|\X_{ss}^{(b)}\|_{C^{1,\alpha}},c_\Gamma)\norm{\X^{(a)}-\X^{(b)}}_{C^{2,\alpha^+}}\norm{\bm{f}}_{C^{0,\alpha}(\T)}\,.
\end{aligned}
\end{equation}
\end{theorem}

For both $\mc{L}_\epsilon^{-1}$ and $\mc{L}_\epsilon$, the $\epsilon$-dependence is explicit in the bound for the remainder terms $\mc{R}_{j,\epsilon}$ but not for the smoother remainders $\mc{R}_{j,+}$.
Although only the decomposition \eqref{eq:thm_NtD_decomp} of $\mc{L}_\epsilon$ is required to show Theorem \ref{thm:main}, the proof of this representation is dependent on the decomposition \eqref{eq:thm_DtN_decomp} of $\mc{L}_\epsilon^{-1}$, which is simpler to study and forms the bulk of this paper. We thus state both results together. 

Given force data $\bm{f}(s)$ involving multiple derivatives in $s$, as in \eqref{eq:full_evolution}, we may refine the decomposition \eqref{eq:thm_NtD_decomp} of $\mc{L}_\epsilon$ as follows. 
\begin{corollary}[SB NtD decomposition involving derivatives]\label{cor:four_derivs}
For $0<\alpha<\gamma<1$, given $\X(s)\in C^{3,\gamma}(\T)$ and slender body force data of the form $\bm{f}(s)=\p_s^4\bm{F}$ for some $\bm{F}\in C^{4,\alpha}(\T)$, the decomposition \eqref{eq:thm_NtD_decomp} of $\mc{L}_\epsilon$ simplifies to 
\begin{equation}\label{eq:new_NtD_decomp}
\begin{aligned}
\mc{L}_\epsilon[\p_s^4\bm{F}]&= ({\bf I}+\mc{R}_{\rm n,\epsilon})\big[\overline{\mc{L}}_\epsilon[\p_s^4\bm{F}]\big] +\wt{\mc{R}}_{\rm n,+}[\p_s^4\bm{F}]\,,\\
\norm{\mc{R}_{\rm n,\epsilon}\big[\overline{\mc{L}}_\epsilon[\p_s^4\bm{F}]\big]}_{C^{1,\alpha}(\T)} &\le c(\|\X\|_{C^{2,\alpha^+}},c_\Gamma)\, \epsilon^{1-\alpha^+}\abs{\log\epsilon}\norm{\overline{\mc{L}}_\epsilon[\p_s^4\bm{F}]}_{C^{1,\alpha}(\T)}\\
\norm{\wt{\mc{R}}_{\rm n,+}[\p_s^4\bm{F}]}_{C^{1,\gamma}(\T)} &\le c(\epsilon,\|\X\|_{C^{3,\gamma}},c_\Gamma)\norm{\bm{F}}_{C^{4,\alpha}(\T)}\,,
\end{aligned}
\end{equation}
where $\alpha<\alpha^+<1$.
Furthermore, given two nearby filaments with centerlines $\X^{(a)}(s)$, $\X^{(b)}(s)$ in $C^{3,\gamma}$ satisfying Lemma \ref{lem:XaXb_C2beta}, we have
\begin{equation}\label{eq:new_NtD_ests_lip}
\begin{aligned}
&\norm{(\mc{R}_{\rm n,\epsilon}^{(a)}-\mc{R}_{\rm n,\epsilon}^{(b)})[\overline{\mc{L}}_\epsilon[\p_s^4\bm{F}]]}_{C^{1,\alpha}(\T)} \\
&\quad\le c(\|\X^{(a)}\|_{C^{2,\alpha^+}},\|\X^{(b)}\|_{C^{2,\alpha^+}},c_\Gamma)\, \epsilon^{1-\alpha^+}\abs{\log\epsilon}\norm{\X^{(a)}-\X^{(b)}}_{C^{2,\alpha^+}(\T)}\norm{\overline{\mc{L}}_\epsilon[\p_s^4\bm{F}]}_{C^{1,\alpha}(\T)} \\
&\norm{(\wt{\mc{R}}_{\rm n,+}^{(a)}-\wt{\mc{R}}_{\rm n,+}^{(b)})[\p_s^4\bm{F}]}_{C^{1,\gamma}(\T)}\\
&\quad\le c(\epsilon,\|\X^{(a)}\|_{C^{3,\gamma}},\|\X^{(b)}\|_{C^{3,\gamma}},c_\Gamma)\norm{\X^{(a)}-\X^{(b)}}_{C^{3,\gamma}(\T)}\norm{\bm{F}}_{C^{4,\alpha}(\T)}\,.
\end{aligned}
\end{equation}
\end{corollary}

The proof of Theorem \ref{thm:decomp} relies on a series of key lemmas, which we will state first before proving Theorem \ref{thm:decomp} in section \ref{subsec:pf_decomp}. The proof of Corollary \ref{cor:four_derivs} will appear immediately after the proof of Theorem \ref{thm:decomp}. The proofs of each of these key lemmas then form the remainder of the paper.

\subsubsection{Boundary integral formulation of slender body DtN map}\label{subsubsec:SB_DtN_boundaryintegral}
Motivated by our analysis in the Laplace setting \cite{laplace}, the strategy for understanding the slender body NtD map will rely on an explicit boundary integral representation for the inverse map, the slender body DtN map, for which we can more easily obtain estimates. 

We use $\mc{G}(\bx,\bx')$ to denote the Stokeslet, the free-space Green's function for the Stokes equations, given by the matrix 
\begin{equation}\label{eq:stokeslet}
\mc{G}(\bx,\bx') = \frac{1}{8\pi}\bigg(\frac{{\bf I}}{\abs{\bx-\bx'}}+\frac{(\bx-\bx')\otimes(\bx-\bx')}{\abs{\bx-\bx'}^3}\bigg).
\end{equation}
For $\bx\in\Gamma_\epsilon$, let $\bm{n}(\bx)$ denote the unit normal vector to $\Gamma_\epsilon$ pointing \emph{out from} the slender body $\Sigma_\epsilon$, \emph{into} the fluid domain $\Omega_\epsilon$. We additionally define the stresslet kernel\footnote{Some sources \cite{pozrikidis1992boundary} instead use the convention that the double layer kernel is given by $-\bm{K}_{\mc D}(\bx',\bx;\bm{n}(\bx'))$.}
\begin{equation}\label{eq:stresslet}
\bm{K}_{\mc D}(\bx,\bx';\bm{n}(\bx')) = \frac{3}{4\pi}\frac{(\bx-\bx')\otimes(\bx-\bx')}{\abs{\bx-\bx'}^5}(\bx-\bx')\cdot\bm{n}(\bx')\,.
\end{equation}
Given a continuous $\bm{\varphi}:\Gamma_\epsilon \to \R^3$, for $\bx\in\Omega_\epsilon$, we define the single layer operator $\mc{S}$ and double layer operator $\mc{D}$ as 
\begin{equation}\label{eq:S_and_D}
\mc{S}[\bm{\varphi}](\bx) = \int_{\Gamma_\epsilon}\mc{G}(\bx,\bx')\bm{\varphi}(\bx')\,dS_{x'}\,, \qquad 
\mc{D}[\bm{\varphi}](\bx) = \int_{\Gamma_\epsilon}\bm{K}_{\mc D}(\bx,\bx';\bm{n}(\bx'))\bm{\varphi}(\bx')\,dS_{x'}\,.
\end{equation}
As in the Laplace setting \cite{kress1989linear,steinbach2007numerical,folland1995introduction}, the single layer operator is continuous up to $\Gamma_\epsilon$, while the double layer operator satisfies the exterior jump relation (see \cite[eq. (2.3.12)]{pozrikidis1992boundary}):
\begin{equation}\label{eq:ext_jump}
\lim_{h\to 0} \mc{D}[\bm{\varphi}](\bx+h\bm{n}(\bx)) =\mc{D}[\bm{\varphi}](\bx)+ \frac{1}{2}\bm{\varphi}(\bx)\,.
\end{equation}
Following \cite[eq. (2.3.10)]{pozrikidis1992boundary}, the Stokes flow $\bu(\bx)$ at a point $\bx\in\Omega_\epsilon$ may be represented by the boundary integral expression 
\begin{equation}\label{eq:BI_rep}
\bu(\bx) = \mc{S}[\bm{w}](\bx) + \mc{D}[\bv](\bx) 
\end{equation}
where, for $\bx\in\Gamma_\epsilon$, 
\begin{equation}\label{eq:BI_rep2}
\begin{aligned}
\bm{w}(\bx) &= -\bm{\sigma}[\bu]\bm{n}(\bx)\big|_{\Gamma_\epsilon}\,,  \\ 
\bv(\bx) &= \bu\big|_{\Gamma_\epsilon}  \,.
\end{aligned}
\end{equation}
The negative sign in $\bm{w}$ arises due to the convention that the normal vector $\bm{n}(\bx)$ points into the fluid domain $\Omega_\epsilon$, making $-\bm{\sigma}[\bu]\bm{n}(\bx)$ the Neumann boundary value for $\bu$. Using \eqref{eq:ext_jump} in \eqref{eq:BI_rep}, on the filament surface $\Gamma_\epsilon$, we obtain a relation between the Dirichlet ($\bv$) and Neumann ($\bm{w}$) boundary values for $\bu$:
\begin{equation}\label{eq:full_DtN}
\textstyle \frac{1}{2}\bv(\bx) - \mc{D}[\bv](\bx) = \mc{S}[\bm{w}](\bx)\,, \qquad \bx\in \Gamma_\epsilon\,.
\end{equation}
Using \eqref{eq:full_DtN}, we may thus relate the slender body Dirichlet data $\bv(s)$ and angle-averaged Neumann data $\bm{f}(s)$ from \eqref{eq:SB_PDE} as
\begin{equation}\label{eq:SB_DtN_layer}
\begin{aligned}
\textstyle \frac{1}{2}\bv(s) - \mc{D}[\bv(s)] &= \mc{S}[\bm{w}(s,\theta)]\\
\int_0^{2\pi}\bm{w}(s,\theta) \,\mc{J}_\epsilon(s,\theta)\,d\theta &= \bm{f}(s)
\end{aligned}
\end{equation}
As in the Laplace setting \cite{laplace}, we consider $\bm{w}(s,\theta)=\bm{w}(\bx(s,\theta))$ as a function of the parameters $s$ and $\theta$ along the surface $\Gamma_\epsilon$.

Accordingly, to better compare across different filament centerline geometries, we use the following characterization of the $\abs{\cdot}_{\dot C^{0,\alpha}}$ seminorm along $\Gamma_\epsilon$:  
\begin{equation}\label{eq:dot_Calpha_eps}
\abs{\bm{h}}_{\dot C^{0,\alpha}(\Gamma_\epsilon)} = \sup_{(s-s')^2+(\theta-\theta')^2\neq 0}\frac{\abs{\bm{h}(s,\theta)-\bm{h}(s',\theta')}}{\sqrt{(s-s')^2+\epsilon^2(\theta-\theta')^2}^{\,\alpha}}\,.
\end{equation}
The seminorm \eqref{eq:dot_Calpha_eps} may be seen to be an equivalent characterization of $\dot C^{0,\alpha}(\Gamma_\epsilon)$ to \eqref{eq:Ckalpha}
(see \cite[Lemma 3.1]{laplace}). In addition, it will be useful to define the following spaces measuring additional H\"older regularity in $s$ only: 
\begin{equation}\label{eq:CalphaS_norm}
\begin{aligned}
\norm{\bm{h}}_{C^{0,\alpha}_{\rm s}(\Gamma_\epsilon)} &:= \norm{\bm{h}}_{L^\infty(\Gamma_\epsilon)}+ \sup_{\theta\in [0,2\pi)}\bigg(\sup_{s\neq s'} \frac{\abs{\bm{h}(s,\theta)-\bm{h}(s',\theta)}}{\abs{s-s'}^\alpha}\bigg) \\
\norm{\bm{h}}_{C^{m,\alpha}_{\rm s}(\Gamma_\epsilon)} &:= \sum_{\ell=0}^m\norm{\p_s^\ell\bm{h}}_{L^\infty(\Gamma_\epsilon)} + \sup_{\theta\in [0,2\pi)}\bigg(\sup_{s\neq s'} \frac{\abs{\p_s^m\bm{h}(s,\theta)-\p_s^m\bm{h}(s',\theta)}}{\abs{s-s'}^\alpha}\bigg)\,.
\end{aligned}
\end{equation}
Note that if $\bm{h}\in C^{m,\alpha}(\Gamma_\epsilon)$ is a function of $s$ only, then both \eqref{eq:dot_Calpha_eps} and \eqref{eq:CalphaS_norm} reduce to 
\begin{align*}
\abs{\bm{h}(s)}_{\dot C^{0,\alpha}(\Gamma_\epsilon)} 
= \abs{\bm{h}(s)}_{\dot C^{0,\alpha}_s(\Gamma_\epsilon)} 
= \abs{\bm{h}(s)}_{\dot C^{0,\alpha}(\T)} 
= \sup_{s\neq s'}\frac{\abs{\bm{h}(s)-\bm{h}(s')}}{\abs{s-s'}^{\,\alpha}}\,.
\end{align*}

\subsubsection{Single and double layer operators about $\mc{C}_\epsilon$}\label{subsubsec:straight_SLDL}
In the Laplace setting \cite{laplace}, we were able to calculate explicit Fourier multiplier expressions for the corresponding single and double layer potentials about the straight cylinder $\mc{C}_\epsilon$. These expressions were then applied as maps between $\theta$-independent scalar-valued functions on $\Gamma_\epsilon$.

In the Stokes setting, we again may calculate explicit Fourier multiplier expressions for the single and double layer operators about the straight filament $\mc{C}_\epsilon$, which we will denote by $\overline{\mc{S}}$ and $\overline{\mc{D}}$, respectively. However, the situation is far more complicated. In particular, the multipliers are now matrix-valued expressions that must be applied between vector-valued functions living on cylindrical basis vectors $\be_r$, $\be_\theta$, $\be_z$. 

Furthermore, the $\theta$-independent functions on $\mc{C}_\epsilon$ are represented not just by zero Fourier modes in $\theta$, but also by one modes: given a $\theta$-independent vector-valued function $\bm{f}(s)$ along the straight cylinder $\mc{C}_\epsilon$, for $s\in \T$ we may write 
\begin{align*}
\bm{f}(s) &= f^z(s)\,\be_z + f^x(s)\,\be_x + f^y(s)\,\be_y \\
&= f^z(s)\,\be_z + f^x(s)(\cos\theta\be_r-\sin\theta\be_\theta) + f^y(s)(\sin\theta\be_r+\cos\theta\be_\theta)\,.
\end{align*}
Along $\mc{C}_\epsilon$, we may treat the tangential ($\be_z$) and normal ($\be_x$ and $\be_y$) directions along the filament separately. Expanding each of $f^z(s)$, $f^x(s)$, and $f^y(s)$ as a Fourier series in $s$, we show that the double and single layer operators $\overline{\mc{D}}$ and $\overline{\mc{S}}$ act in the following way.

\begin{lemma}[Single and double layer multipliers about $\mc{C}_\epsilon$]\label{lem:straight_SD}
When applied to $\theta$-independent data in the tangential direction along the straight cylinder $\mc{C}_\epsilon$, the double layer operator $\overline{\mc{D}}$ is given by 
\begin{equation}\label{eq:DL_tangential}
\begin{aligned}
\overline{\mc{D}}[e^{2\pi iks}\be_z] 
&= Q_{\rm tE}(2\pi\epsilon\abs{k})\, e^{2\pi iks}\be_z + i\,Q_{\rm tF}(2\pi\epsilon\abs{k})\,e^{2\pi iks}\be_r\,, \\
Q_{\rm tE}(z) &= \textstyle z^2 \big( I_0K_0-I_1K_1\big) -\frac{z}{2}\big(I_0K_1-I_1K_0\big)\,,\\
Q_{\rm tF}(z) &= -z^2\big(I_1K_0-I_0K_1\big)- z I_1K_1 \,.
\end{aligned}
\end{equation}
Here and throughout, each $I_j=I_j(z)$ and $K_j=K_j(z)$ are modified Bessel functions of the first and second kind, respectively.
Given the form \eqref{eq:DL_tangential} of the double layer operator, let
\begin{equation}\label{eq:Fz}
\bm{F}_z = \big(f_z^z\be_z+f_r^z\be_r\big) e^{2\pi i k s}\,.
\end{equation}
When applied to data of the form \eqref{eq:Fz}, the single layer operator $\overline{\mc{S}}$ is invertible and satisfies
\begin{equation}\label{eq:Sinv_tangential1}
\begin{aligned}
\overline{\mc{S}}^{-1}[\bm{F}_z]&= \epsilon^{-1}\bm{M}_{\rm S,t}^{-1}(k)\bm{F}_z\,, \\
\bm{M}_{\rm S,t}^{-1}(k)&=
m_{\epsilon,{\rm tA}}(k)\,\be_z\otimes\be_z  +i\,m_{\epsilon,{\rm tB}}(k)(\be_r\otimes\be_z-\be_z\otimes\be_r) + m_{\epsilon,{\rm tC}}(k)\,\be_r\otimes\be_r\,. 
\end{aligned}
\end{equation}
The components of the matrix-valued multiplier $\bm{M}_{\rm S,t}^{-1}(k)$ are given by 
\begin{equation}\label{eq:Sinv_tangential2}
\begin{aligned}
 m_{\epsilon,{\rm t}j}(k) &= \frac{Q_{{\rm t}j}(2\pi\epsilon\abs{k})}{Q_{\rm tD}(2\pi\epsilon\abs{k})}\,, \quad j={\rm A,B,C}  \\
Q_{\rm tA}(z) &= \textstyle \frac{1}{I_0K_0}\left(1+ \frac{z}{2}\left(\frac{K_0}{K_1}-\frac{I_0}{I_1} \right)\right) \,, \quad
Q_{\rm tC}(z) = \frac{1}{I_1K_1}\left(1+ \frac{z}{2}\left(\frac{I_1}{I_0}-\frac{K_1}{K_0} \right)\right) \,, \\
Q_{\rm tB}(z) &= \textstyle \frac{z}{2}\left(\frac{1}{I_1K_1}-\frac{1}{I_0K_0}\right)\,, \quad
Q_{\rm tD}(z) = \left(1+ \frac{z}{2}\left(\frac{I_1}{I_0}-\frac{I_0}{I_1} \right)\right)\left(1+ \frac{z}{2}\left(\frac{K_0}{K_1}-\frac{K_1}{K_0} \right)\right)\,. 
\end{aligned}
\end{equation}

When applied to $\theta$-independent data in the normal direction to the straight cylinder $\mc{C}_\epsilon$, the double layer operator $\overline{\mc{D}}$ is given by 
\begin{equation}\label{eq:DL_normal}
\begin{aligned}
\overline{\mc{D}}[e^{2\pi iks}\be_x] &= \overline{\mc{D}}[e^{2\pi iks}\big(\cos\theta\be_r - \sin\theta\be_\theta\big)] \\
&= \left( Q_{\rm nN}(2\pi\epsilon\abs{k})\cos\theta\be_r - Q_{\rm nO}(2\pi\epsilon\abs{k})\sin\theta\be_\theta + iQ_{\rm nP}(2\pi\epsilon\abs{k})\cos\theta\be_z\right)e^{2\pi iks}\,, \\
\overline{\mc{D}}[e^{2\pi iks}\be_y] &= \overline{\mc{D}}[e^{2\pi iks}\big(\sin\theta\be_r + \cos\theta\be_\theta\big)] \\
&= \left( Q_{\rm nN}(2\pi\epsilon\abs{k})\sin\theta\be_r + Q_{\rm nO}(2\pi\epsilon\abs{k})\cos\theta\be_\theta + iQ_{\rm nP}(2\pi\epsilon\abs{k})\sin\theta\be_z\right)e^{2\pi iks}\,; \\
Q_{\rm nN}(z) &= z^2(I_1K_1-I_0K_0)+\textstyle \frac{3}{2}z(I_1K_0-I_0K_1) + 2I_1K_1 \,, \\
Q_{\rm nO}(z) &= \textstyle -\frac{z}{2}(I_1K_0-I_0K_1)-2I_1K_1\,,\qquad 
Q_{\rm nP}(z) = z^2(I_0K_1-I_1K_0) - zI_1K_1 \,.
\end{aligned}
\end{equation}
Given the form of the expression \eqref{eq:DL_normal}, let
\begin{equation}\label{eq:FxFy}
\begin{aligned}
\bm{F}_x(s,\theta) &= \big(f_r^x\cos\theta \be_r + f_\theta^x\sin\theta \be_\theta + f_z^x\cos\theta\be_z\big) e^{2\pi iks} \\
\bm{F}_y(s,\theta) &= \big(f_r^y\sin\theta \be_r - f_\theta^y\cos\theta \be_\theta + f_z^y\sin\theta\be_z \big) e^{2\pi iks}\,.
\end{aligned}
\end{equation}
When applied to data of the form \eqref{eq:FxFy}, the single layer operator $\overline{\mc{S}}$ is again invertible and satisfies 
\begin{equation}\label{eq:Sinv_normal1}
\begin{aligned}
\overline{\mc{S}}^{-1}[\bm{F}_x] &= \epsilon^{-1}{\bm M}_{\rm S,n}^{-1}(k)\bm{F}_x\,, \qquad  
\overline{\mc{S}}^{-1}[\bm{F}_y] = \epsilon^{-1}{\bm M}_{\rm S,n}^{-1}(k)\bm{F}_y\,, 
\end{aligned}
\end{equation}
where ${\bm M}_{\rm S,n}^{-1}$ is a matrix-valued multiplier of the form
\begin{equation}\label{eq:Sinv_normal2}
\begin{aligned}
&{\bm M}_{\rm S,n}^{-1}(k) = 
m_{\epsilon,{\rm nA}}(k)\,\be_r\otimes\be_r 
+ m_{\epsilon,{\rm nB}}(k)\,(\be_r\otimes\be_\theta + \be_\theta\otimes\be_r) 
+ m_{\epsilon,{\rm nD}}(k)\,\be_\theta\otimes\be_\theta \\
&\quad + i\,m_{\epsilon,{\rm nC}}(k)\,(\be_r\otimes\be_z-\be_z\otimes\be_r ) 
+i\,m_{\epsilon,{\rm nE}}(k)\,(\be_z\otimes\be_\theta -\be_\theta\otimes\be_z) +m_{\epsilon,{\rm nF}}(k)\,\be_z\otimes\be_z\,.
\end{aligned}
\end{equation}
The components of ${\bm M}_{\rm S,n}^{-1}(k)$ are given explicitly by   
\begin{equation}\label{eq:Sinv_components}
\begin{aligned}
m_{\epsilon,{\rm n}j}(k) &= \frac{Q_{{\rm n}j}(2\pi\epsilon\abs{k})}{Q_{\rm nG}(2\pi\epsilon\abs{k})}\,, \quad j={\rm A,B,C,D,E,F}\,,\\
Q_{\rm nA}(z) &= \textstyle \frac{1}{I_1K_1} \left( 4+\frac{4}{z^2}\frac{I_1}{I_0}\frac{K_1}{K_0} + \frac{2}{z}\left(\frac{K_1}{K_0} -\frac{I_1}{I_0}\right) -2z\left(\frac{I_0}{I_1} -\frac{K_0}{K_1}\right) -2\left(\frac{I_0}{I_1}\frac{K_1}{K_0} +\frac{I_1}{I_0}\frac{K_0}{K_1}\right) \right)\\
Q_{\rm nB}(z) &= \textstyle \frac{1}{I_1K_1} \left( 2\frac{K_1}{K_0}\left(\frac{I_0}{I_1}-\frac{I_1}{I_0}\right)+ 2\left(\frac{I_1}{I_0}\frac{K_0}{K_1}-1\right) +\frac{2}{z}\left(\frac{I_1}{I_0}-\frac{K_1}{K_0}\right) -\frac{4}{z^2}\frac{I_1}{I_0}\frac{K_1}{K_0} \right) \\
Q_{\rm nC}(z) &=\textstyle \frac{2}{I_1K_1} \left(\frac{2}{z}\frac{I_1}{I_0}\frac{K_1}{K_0} + 2\left(\frac{I_0}{I_1}- \frac{K_0}{K_1}\right) +\left(\frac{I_1}{I_0}-\frac{K_1}{K_0}\right)+z\frac{I_0}{I_1}\frac{K_0}{K_1} -z-\frac{4}{z} \right) \\
Q_{\rm nD}(z) &= \textstyle \frac{1}{I_1K_1} \bigg( \frac{4}{z^2}\frac{I_1}{I_0}\frac{K_1}{K_0}+ \frac{2}{z}\left(\frac{K_1}{K_0}-\frac{I_1}{I_0}\right)+2\left(2\frac{I_1}{I_0}\frac{K_1}{K_0}- \frac{I_1}{I_0}\frac{K_0}{K_1}- \frac{I_0}{I_1}\frac{K_1}{K_0}\right) \\
&\qquad \textstyle +z^2\left(\frac{I_0}{I_1}-\frac{I_1}{I_0}\right)\left(\frac{K_0}{K_1}-\frac{K_1}{K_0}\right) \bigg)\\
Q_{\rm nE}(z) &= \textstyle \frac{1}{I_1K_1} \left( \frac{4}{z}\frac{I_1}{I_0}\frac{K_1}{K_0}- \frac{4}{z} +2\left(\frac{I_0}{I_1}-\frac{K_0}{K_1}\right) +z\left(\frac{I_0}{I_1}-\frac{I_1}{I_0}\right)\left(\frac{K_0}{K_1}-\frac{K_1}{K_0}\right) \right)\\
Q_{\rm nF}(z) &= \textstyle \frac{1}{I_1K_1} \bigg( \frac{12}{z^2} + \frac{6}{z}\left(\frac{K_0}{K_1}-\frac{I_0}{I_1}+\frac{K_1}{K_0}-\frac{I_1}{I_0}\right) -3\left(\frac{I_0}{I_1}\left(\frac{K_0}{K_1}+\frac{K_1}{K_0}\right)+\frac{I_1}{I_0}\frac{K_0}{K_1}\right) \\
&\qquad \textstyle + \left(\frac{I_1}{I_0}\frac{K_1}{K_0} +8\right) + 2z\left(\frac{K_0}{K_1}-\frac{I_0}{I_1}\right) \bigg)  \\
Q_{\rm nG}(z) &= \textstyle \left(\frac{2}{z} + \left(1- z\frac{I_0}{I_1}\right)\left(\frac{I_0}{I_1}-\frac{I_1}{I_0}\right) \right) \left(\frac{2}{z}- \left(1+ z\frac{K_0}{K_1}\right)\left(\frac{K_0}{K_1}-\frac{K_1}{K_0}\right) \right) \,.
\end{aligned}
\end{equation}
\end{lemma}

The calculation of the explicit symbols in Lemma \ref{lem:straight_SD} appears in Appendix \ref{app:symbol_calc}.
For data of the form $\bm{F}_z$ \eqref{eq:Fz} or $\bm{F}_x,\bm{F}_y$ \eqref{eq:FxFy}, instead of working with the full inverse single layer operator $\overline{\mc{S}}^{-1}$, it will be convenient to define the following angle-averaged inverse single layer along $\mc{C}_\epsilon$: 
\begin{equation}\label{eq:Aeps}
\overline{\mc{A}}_\epsilon[\bm{F}_j] := \int_0^{2\pi}\overline{\mc{S}}^{-1}[\bm{F}_j]\,\epsilon \,d\theta\,.
\end{equation}
Using that 
\begin{align*}
\int_0^{2\pi} \be_z \,d\theta &= 2\pi\be_z\,, \quad \int_0^{2\pi}\be_r\,d\theta=\int_0^{2\pi}\be_\theta\,d\theta=\int_0^{2\pi}\cos\theta\be_z\,d\theta=0\,, \\
\int_0^{2\pi}\cos\theta\be_r\,d\theta&=-\int_0^{2\pi}\sin\theta\be_\theta\,d\theta=\pi \be_x \,, \qquad 
\int_0^{2\pi}\sin\theta\be_r\,d\theta=\int_0^{2\pi}\cos\theta\be_\theta\,d\theta=\pi \be_y\,,
\end{align*}
we immediately obtain the following corollary to Lemma \ref{lem:straight_SD}. 
\begin{corollary}[Angle-averaged single layer]\label{cor:Aeps}
When applied to functions $\bm{F}_\mu$ of the form \eqref{eq:Fz} or \eqref{eq:FxFy}, the angle-averaged inverse single layer operator $\overline{\mc{A}}_\epsilon$ given by \eqref{eq:Aeps} satisfies
\begin{equation}
\overline{\mc{A}}_\epsilon[\bm{F}_\mu] = \pi \big( m_{\epsilon,r}^\mu(k)f_r^\mu + m_{\epsilon,\theta}^\mu(k)f_\theta^\mu + m_{\epsilon,z}^\mu(k)f_z^\mu \big)\, e^{2\pi iks}\be_\mu\,, \qquad \mu=x,y,z\,,
\end{equation}
where the tangential multipliers $m_{\epsilon,\ell}^z$ are given by
\begin{equation}\label{eq:Aeps_mt}
m_{\epsilon,z}^z(k) = 2\,m_{\epsilon,{\rm tA}}(k)\,, \qquad 
m_{\epsilon,\theta}^z(k) = 0\,, \qquad
m_{\epsilon,r}^z(k) = -2i\,m_{\epsilon,{\rm tB}}(k)\,;
\end{equation}
and the normal direction multipliers $m_{\epsilon,\ell}^x=m_{\epsilon,\ell}^y$ are given by 
\begin{equation}\label{eq:Aeps_mn}
\begin{aligned}
m_{\epsilon,r}^x(k)&=m_{\epsilon,r}^y(k) = m_{\epsilon,{\rm nA}}(k)-m_{\epsilon,{\rm nB}}(k)\,, \\
m_{\epsilon,\theta}^x(k)&=m_{\epsilon,\theta}^y(k) = m_{\epsilon,{\rm nB}}(k)-m_{\epsilon,{\rm nD}}(k)\,, \\
m_{\epsilon,z}^x(k)&=m_{\epsilon,z}^y(k) = i\,(m_{\epsilon,{\rm nC}}(k)+m_{\epsilon,{\rm nE}}(k)) \,.
\end{aligned}
\end{equation}
\end{corollary}

It is ultimately the angle-averaged operator $\overline{\mc{A}}_\epsilon$ for which we will need detailed information. Before we turn to the mapping properties of $\overline{\mc{A}}_\epsilon$, we may perform a sanity check to demonstrate that we indeed have the correct expressions for the single and double layer multipliers on $\mc{C}_\epsilon$. 
In particular, we can verify that the double and single layer expressions of Lemma \ref{lem:straight_SD} combine to yield the multipliers $m_{\epsilon,{\rm t}}^{-1}$ and $m_{\epsilon,{\rm n}}^{-1}$ given by \eqref{eq:eigsT} and \eqref{eq:eigsN}. 
For the tangential direction, we may combine \eqref{eq:DL_tangential} and \eqref{eq:Aeps_mt} to obtain 
\begin{equation}\label{eq:recover_mt}
\begin{aligned}
\overline{\mc{L}}_\epsilon^{-1}[e^{2\pi iks}\be_z] &= \overline{\mc{A}}_\epsilon\big[\textstyle(\frac{1}{2}{\bf I}-\overline{\mc{D}})[e^{2\pi iks}\be_z]\big]\\
&= 2\pi\big(m_{\epsilon,{\rm tA}}(k)Q_{\rm tE}(2\pi\epsilon\abs{k})+ m_{\epsilon,{\rm tB}}(k)Q_{\rm tF}(2\pi\epsilon\abs{k})\big)\, e^{2\pi iks}\be_z\,.
\end{aligned}
\end{equation}
For the normal direction, we may combine \eqref{eq:DL_normal} and \eqref{eq:Aeps_mn} to obtain
\begin{equation}\label{eq:recover_mn}
\begin{aligned}
\overline{\mc{L}}_\epsilon^{-1}[e^{2\pi iks}\be_x] &= 
\overline{\mc{A}}_\epsilon\big[\textstyle(\frac{1}{2}{\bf I}-\overline{\mc{D}})[e^{2\pi iks}\be_x]\big] \\
&= \pi \bigg( \big(m_{\epsilon,{\rm nA}}(k)-m_{\epsilon,{\rm nB}}(k) \big)Q_{\rm nN}(2\pi\epsilon\abs{k}) - \big(m_{\epsilon,{\rm nB}}(k)-m_{\epsilon,{\rm nD}}(k) \big)Q_{\rm nO}(2\pi\epsilon\abs{k}) \\
&\qquad - \big(m_{\epsilon,{\rm nC}}(k)+m_{\epsilon,{\rm nE}}(k)\big)Q_{\rm nP}(2\pi\epsilon\abs{k}) \bigg) \, e^{2\pi iks}\be_x\,.
\end{aligned}
\end{equation}
The $\be_y$ direction is identical. Here each of the $m_{\epsilon,j}$ and $Q_j$ are as in Lemma \ref{lem:straight_SD}. Just by looking at the components of the expressions \eqref{eq:recover_mt} and \eqref{eq:recover_mn}, it is not clear that these are the same as the multipliers $m_{\epsilon,{\rm t}}^{-1}(k)$ and $m_{\epsilon,{\rm n}}^{-1}(k)$ of Proposition \ref{prop:Leps_spectrum}. However, by plotting (see figure \ref{fig:multipliers}), we can verify that \eqref{eq:recover_mt} and \eqref{eq:recover_mn} are in fact the same as \eqref{eq:eigsT} and \eqref{eq:eigsN} from Proposition \ref{prop:Leps_spectrum}. 
\begin{figure}[!ht]
\centering
\includegraphics[scale=0.6]{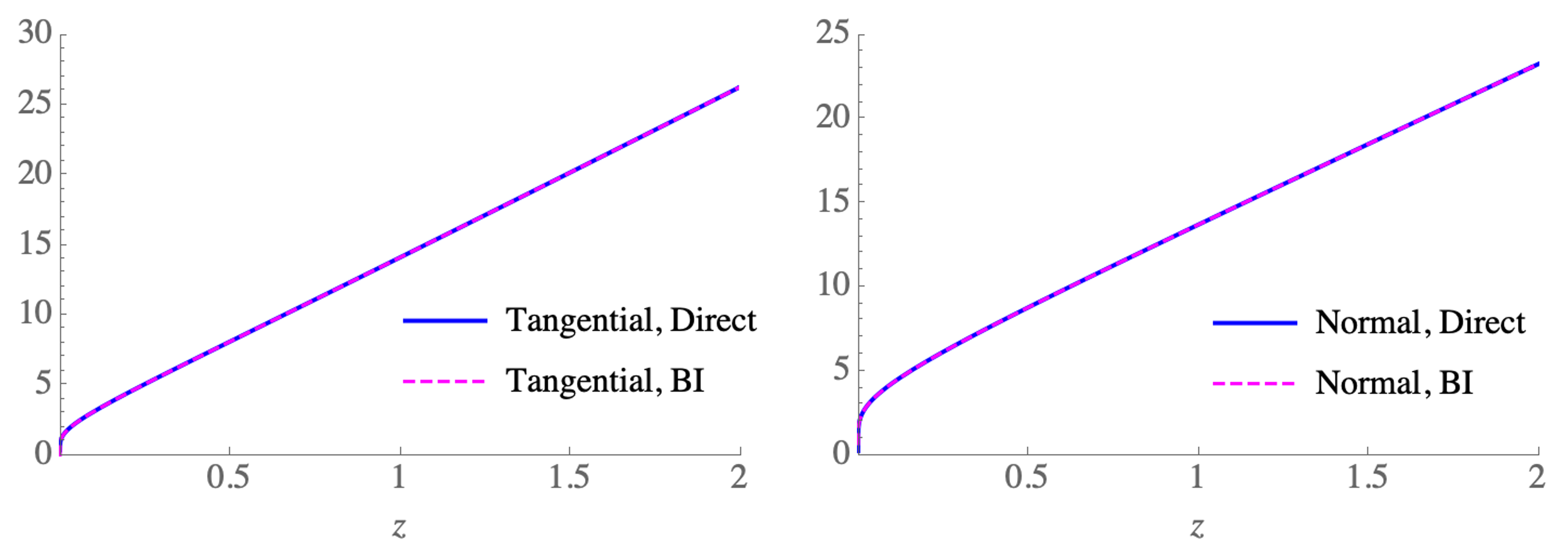}
\caption{Plot of the eigenvalue expressions in directions tangential (left) and normal (right) to the straight filament $\mc{C}_\epsilon$. Here the horizontal axis is $z=2\pi\epsilon \abs{k}$. In both plots, the solid blue curves (``Direct'') are given by the expressions in Proposition \ref{prop:Leps_spectrum} while the dotted curves (``BI'') are the boundary integral expressions \eqref{eq:recover_mt} and \eqref{eq:recover_mn}, respectively.}
\label{fig:multipliers}
\end{figure}

The important takeaway from the explicit expressions in Lemma \ref{lem:straight_SD} is that we can then directly calculate the mapping properties of the angle-averaged operator $\overline{\mc{A}}_\epsilon$ when applied to functions consisting of zero- and one-modes in $\theta$.
This is in direct analogy to the Laplace setting \cite{laplace}, but the calculation here is much more involved.

For a general vector-valued function $\bm{g}(s,\theta)$ defined along the straight cylinder $\mc{C}_\epsilon$, we write  
\begin{align*}
\bm{g}(s,\theta) = \sum_{j\in\{r,\theta,z\}}\bigg(g_0^j(s)+\sum_{\ell=0}^\infty g^j_{{\rm c},\ell}(s) \cos(\ell\theta) +g^j_{\rm{s},\ell}(s) \sin(\ell\theta) \bigg)\be_j\,.
\end{align*}
We define the projection $\P_{01}$ onto zero- and one-modes in $\theta$ by
\begin{equation}\label{eq:P01_def}
\begin{aligned}
\P_{01}\bm{g} &= g_0^z\,\be_z + g_0^r\,\be_r + g_{{\rm c},1}^r\cos\theta\be_r + g_{{\rm s},1}^\theta\sin\theta\be_\theta + g_{{\rm c},1}^z\cos\theta\be_z \\
&\qquad + g_{{\rm s},1}^r\sin\theta\be_r + g_{{\rm c},1}^\theta\cos\theta\be_\theta  + g_{{\rm s},1}^z\sin\theta\be_z\,.
\end{aligned}
\end{equation}
\begin{lemma}\label{lem:Sinv_mapping}
Let $\overline{\mc{A}}_\epsilon$ be the angle-averaged inverse single layer operator about $\mc{C}_\epsilon$, given by \eqref{eq:Aeps}. Given $\bm{g}\in C^{1,\alpha}_s(\Gamma_\epsilon)$ and the projection operator \eqref{eq:P01_def}, we may decompose $\overline{\mc{A}}_\epsilon\P_{01}$ as
\begin{align*}
\overline{\mc{A}}_\epsilon[\P_{01}\bm{g}] = \overline{\mc{A}}_\epsilon^\epsilon[\bm{g}] + \overline{\mc{A}}_\epsilon^+[\bm{g}]
\end{align*}
where
\begin{equation}
\begin{aligned}
\norm{\overline{\mc{A}}_\epsilon^\epsilon[\bm{g}]}_{C^{0,\alpha}(\T)}
&\le c\norm{\bm{g}}_{C_s^{0,\alpha}(\Gamma_\epsilon)}+c\,\epsilon\abs{\bm{g}}_{\dot C_s^{1,\alpha}(\Gamma_\epsilon)}\,, \\
\norm{\overline{\mc{A}}_\epsilon^+[\bm{g}]}_{C^{1,\alpha}(\T)} &\le  
c\,\epsilon^{-1}\norm{\bm{g}}_{C_s^{1,\alpha}(\Gamma_\epsilon)}\,.
\end{aligned}
\end{equation}
\end{lemma}
The proof of Lemma \ref{lem:Sinv_mapping} appears in section \ref{subsec:straight_mapping}.

\subsubsection{Decomposition of SB DtN}\label{subsubsec:SB_DtN_decomp}
Given a curved filament $\Sigma_\epsilon$ with centerline $\X(s)\in C^{2,\beta}(\T)$, we consider the slender body Dirichlet-to-Neumann map $\mc{L}_\epsilon^{-1}$ along $\Gamma_\epsilon$ and aim to extract the corresponding map $\overline{\mc{L}}_\epsilon^{-1}$ about the straight filament $\mc{C}_\epsilon$ as the leading order behavior. 

Recalling the definition \eqref{eq:mapPhi_def} of the operator $\Phi$ identifying the normal/tangential directions along the curved and straight filaments, we begin by decomposing the boundary integral representation \eqref{eq:SB_DtN_layer} as 
\begin{equation}\label{eq:decomp_step1}
\textstyle (\frac{1}{2}{\bf I} - \overline{\mc{D}})[\Phi^{-1}\bv(s)]-\mc{R}_{\mc{D}}[\bv(s)] = \overline{\mc{S}}[\Phi^{-1}\bw_0^\Phi]+\mc{R}_{\mc{S}}[\bw_0^\Phi] + \Phi^{-1}\mc{S}\displaystyle\bigg[\Phi\int_\T\Phi^{-1}\bw(s,\theta)\,ds \bigg]\,,
\end{equation}
where $\bw_0^\Phi(s,\theta)$ is as in \eqref{eq:subtract_mean} and we define
\begin{equation}\label{eq:RS_RD_def}
\begin{aligned}
\mc{R}_{\mc{D}}[\bv(s)] &:= \Phi^{-1}\mc{D}[\bv(s)]-\overline{\mc{D}}[\Phi^{-1}\bv(s)]\,, \\
\mc{R}_{\mc{S}}[\bw_0^\Phi(s,\theta)] &:= \Phi^{-1}\mc{S}[\bw_0^\Phi(s,\theta)]-\overline{\mc{S}}[\Phi^{-1}\bw_0^\Phi(s,\theta)]\,.
\end{aligned}
\end{equation}
We next would like to use Lemma \ref{lem:straight_SD} to apply $\overline{\mc{S}}^{-1}$ to \eqref{eq:decomp_step1}, but in order to do so we must first apply the projection operator $\P_{01}$ \eqref{eq:P01_def} onto zero- and one-modes in $\theta$. 
Note that $\P_{01}$ commutes with the straight single and double layer operators $\overline{\mc{S}}$ and $\overline{\mc{D}}$ about $\mc{C}_\epsilon$. In particular, applying $\P_{01}$ to \eqref{eq:decomp_step1}, we have 
\begin{align*}
\textstyle (\frac{1}{2}{\bf I} - \overline{\mc{D}})[\Phi^{-1}\bv(s)]-\P_{01}\mc{R}_{\mc{D}}[\bv(s)] &= \overline{\mc{S}}[\P_{01}\Phi^{-1}\bw_0^\Phi]+ \P_{01}\mc{R}_{\mc{S}}[\bw_0^\Phi] + \P_{01}\Phi^{-1}\mc{S}\displaystyle\bigg[\Phi\int_\T\Phi^{-1}\bw(s,\theta)\,ds \bigg]\,.
\end{align*}
We may now apply $\overline{\mc{S}}^{-1}$ to the above equation, multiply by $\epsilon$, and integrate in $\theta$ to obtain 
\begin{align*}
\textstyle \overline{\mc{A}}_\epsilon\big[(\frac{1}{2}{\bf I}-\overline{\mc{D}})[\Phi^{-1}\bv(s)] \big] &= \displaystyle\int_0^{2\pi}\P_{01}(\Phi^{-1}\bw_0^\Phi)\,\epsilon\,d\theta
+ \overline{\A}_\epsilon\big[\P_{01}\mc{R}_{\mc{D}}[\bv]\big]
+ \overline{\A}_\epsilon\big[\P_{01}\mc{R}_{\mc{S}}[\bm{w}_0^\Phi]\big]  \\
&\qquad + \mc{A}_\epsilon\bigg[\P_{01}\Phi^{-1}\mc{S}\displaystyle\bigg[\Phi\int_\T\Phi^{-1}\bw(s,\theta)\,ds \bigg]\bigg]\\
&= \Phi^{-1}\int_0^{2\pi}\bm{w}(s,\theta)\,\epsilon\,d\theta
-\int_\T \Phi^{-1}\int_0^{2\pi}\bm{w}(s,\theta)\,\epsilon\, d\theta ds
+ \overline{\A}_\epsilon\big[\P_{01}\mc{R}_{\mc{D}}[\bv]\big]
 \\
&\quad + \overline{\A}_\epsilon\big[\P_{01}\mc{R}_{\mc{S}}[\bm{w}_0^\Phi]\big]+ \overline{\A}_\epsilon\bigg[\P_{01}\Phi^{-1}\mc{S}\displaystyle\bigg[\Phi\int_\T\Phi^{-1}\bw(s,\theta)\,ds \bigg]\bigg]\\
&= (\Phi^{-1}\bm{f})(s)+\Phi^{-1}\int_0^{2\pi}\bm{w}(s,\theta)\,\epsilon^2\wh\kappa\,d\theta -\int_\T \Phi^{-1}\int_0^{2\pi}\bm{w}(s,\theta)\,\epsilon\, d\theta ds\\
&\quad
+ \overline{\A}_\epsilon\big[\P_{01}\mc{R}_{\mc{D}}[\bv]\big]
+ \overline{\A}_\epsilon\big[\P_{01}\mc{R}_{\mc{S}}[\bm{w}_0^\Phi]\big] 
+ \overline{\A}_\epsilon\bigg[\P_{01}\Phi^{-1}\mc{S}\displaystyle\bigg[\Phi\int_\T\Phi^{-1}\bw(s,\theta)\,ds \bigg]\bigg]
\,.
\end{align*}
Here we have used the expression for $\bm{f}(s)$ from \eqref{eq:SB_DtN_layer}.
The left hand side is exactly $\overline{\mc{L}}_\epsilon^{-1}[\Phi^{-1}\bv(s)]$, and using that $\bm{f}(s) = \mc{L}_\epsilon^{-1}[\bv(s)]$, we thus obtain the following decomposition of the slender body DtN map: 
\begin{equation}\label{eq:SB_DtN_decomp}
\begin{aligned}
\mc{L}_\epsilon^{-1}[\bv(s)] &= \Phi\overline{\mc{L}}_\epsilon^{-1}[\Phi^{-1}\bv(s)]
- \int_0^{2\pi}\bm{w}(s,\theta)\,\epsilon^2\wh\kappa\,d\theta + \Phi\int_\T \Phi^{-1}\int_0^{2\pi}\bm{w}(s,\theta)\,\epsilon\, d\theta ds \\
&\quad 
- \Phi\overline{\A}_\epsilon\big[\P_{01}\mc{R}_{\mc{D}}[\bv]\big]
- \Phi\overline{\A}_\epsilon\big[\P_{01}\mc{R}_{\mc{S}}[\bm{w}_0^\Phi]\big] 
- \Phi\overline{\A}_\epsilon\bigg[\P_{01}\Phi^{-1}\mc{S}\displaystyle\bigg[\Phi\int_\T\Phi^{-1}\bw(s,\theta)\,ds \bigg]\bigg]\,.
\end{aligned}
\end{equation}

We will rely on this representation of $\mc{L}_\epsilon^{-1}$ to prove Theorem \ref{thm:decomp}. In particular, we show that the five remainder terms appearing in \eqref{eq:SB_DtN_decomp} are lower order with respect to regularity or size in $\epsilon$. These terms may each be bounded using a combination of Lemma \ref{lem:Sinv_mapping} and the following series of key lemmas.

The first two lemmas give the mapping properties of the single and double layer remainders $\mc{R}_\mc{S}$ and $\mc{R}_\mc{D}$, respectively.
\begin{lemma}[Single layer remainder]\label{lem:single_layer}
Let $0<\alpha<\beta<1$ and consider a filament $\Sigma_\epsilon$ with centerline $\X(s)\in C^{2,\beta}(\T)$. Let the single layer remainder $\mc{R}_{\mc{S}}$ be as defined in \eqref{eq:RS_RD_def}, and let the map $\Phi$ be as in \eqref{eq:mapPhi_def}. Given $\bm{\varphi}\in C^{0,\alpha}(\Gamma_\epsilon)$ with $\int_\T \Phi^{-1}\bm{\varphi}(s,\theta)\,ds=0$, $\mc{R}_{\mc{S}}$ may be decomposed as 
\begin{align*}
\mc{R}_{\mc{S}}[\bm{\varphi}] = \mc{R}_{\mc{S},\epsilon}[\bm{\varphi}] + \mc{R}_{\mc{S},+}[\bm{\varphi}]
\end{align*}
where
\begin{equation}\label{eq:RS_ests1}
\begin{aligned}
\norm{\mc{R}_{\mc{S},\epsilon}[\bm{\varphi}]}_{C^{0,\alpha}} &\le c(\kappa_{*,\alpha},c_\Gamma)\,\epsilon^{2-\alpha}\norm{\bm{\varphi}}_{L^\infty}\,, \quad
\abs{\mc{R}_{\mc{S},\epsilon}[\bm{\varphi}]}_{\dot C_s^{1,\alpha}} \le c(\kappa_{*,\alpha^+},c_\Gamma)\,\epsilon^{1-\alpha^+}\norm{\bm{\varphi}}_{C^{0,\alpha}}\\
\norm{\mc{R}_{\mc{S},+}[\bm{\varphi}]}_{C^{0,\alpha}} &\le c(\kappa_{*,\alpha},c_\Gamma)\,\epsilon^{1-\alpha}\norm{\bm{\varphi}}_{L^\infty}\,, \quad
\abs{\mc{R}_{\mc{S},+}[\bm{\varphi}]}_{\dot C_s^{1,\beta}} \le c(\kappa_{*,\beta},c_\Gamma)\,\epsilon^{1-2\beta}\norm{\bm{\varphi}}_{C^{0,\alpha}}
\end{aligned}
\end{equation}
for any $\alpha^+\in(\alpha,\beta]$.
Furthermore, given two nearby filaments with centerlines $\X^{(a)}(s)$, $\X^{(b)}(s)$ satisfying Lemma \ref{lem:XaXb_C2beta}, let $\bm{\varphi}_0^{\Phi^{(a)}}=\bm{\varphi}-\int_\T(\Phi^{(a)})^{-1}\bm{\varphi}(s,\theta)\,ds$ and $\bm{\varphi}_0^{\Phi^{(b)}}=\bm{\varphi}-\int_\T(\Phi^{(b)})^{-1}\bm{\varphi}(s,\theta)\,ds$. The difference between the corresponding single layer remainders $\mc{R}_{\mc{S}}^{(a)}$ and $\mc{R}_{\mc{S}}^{(b)}$ may be decomposed as above, and the components satisfy
\begin{equation}\label{eq:RS_ests2}
\begin{aligned}
\norm{\mc{R}_{\mc{S},\epsilon}^{(a)}[\bm{\varphi}_0^{\Phi^{(a)}}]-\mc{R}_{\mc{S},\epsilon}^{(b)}[\bm{\varphi}_0^{\Phi^{(b)}}]}_{C^{0,\alpha}} &\le c(\kappa_{*,\alpha}^{(a)},\kappa_{*,\alpha}^{(b)},c_\Gamma)\,\epsilon^{2-\alpha}\norm{\X^{(a)}-\X^{(b)}}_{C^{2,\alpha}}\norm{\bm{\varphi}}_{L^\infty} \\
\abs{\mc{R}_{\mc{S},\epsilon}^{(a)}[\bm{\varphi}_0^{\Phi^{(a)}}]-\mc{R}_{\mc{S},\epsilon}^{(b)}[\bm{\varphi}_0^{\Phi^{(b)}}]}_{\dot C^{1,\alpha}} &\le c(\kappa_{*,\alpha^+}^{(a)},\kappa_{*,\alpha^+}^{(b)},c_\Gamma)\,\epsilon^{1-\alpha^+}\norm{\X^{(a)}-\X^{(b)}}_{C^{2,\alpha^+}}\norm{\bm{\varphi}}_{C^{0,\alpha}}\\
\norm{\mc{R}_{\mc{S},+}^{(a)}[\bm{\varphi}_0^{\Phi^{(a)}}]-\mc{R}_{\mc{S},+}^{(b)}[\bm{\varphi}_0^{\Phi^{(b)}}]}_{C^{0,\alpha}} &\le c(\kappa_{*,\alpha}^{(a)},\kappa_{*,\alpha}^{(b)},c_\Gamma)\,\epsilon^{1-\alpha}\norm{\X^{(a)}-\X^{(b)}}_{C^{2,\alpha}}\norm{\bm{\varphi}}_{L^\infty} \\
\abs{\mc{R}_{\mc{S},+}^{(a)}[\bm{\varphi}_0^{\Phi^{(a)}}]-\mc{R}_{\mc{S},+}^{(b)}[\bm{\varphi}_0^{\Phi^{(b)}}]}_{\dot C^{1,\beta}} &\le c(\kappa_{*,\beta}^{(a)},\kappa_{*,\beta}^{(b)},c_\Gamma)\,\epsilon^{1-2\beta}\norm{\X^{(a)}-\X^{(b)}}_{C^{2,\beta}}\norm{\bm{\varphi}}_{C^{0,\alpha}}\,.
\end{aligned}
\end{equation}
\end{lemma}
The proof of Lemma \ref{lem:single_layer} appears in section \ref{subsec:single_layer}. 
%
%
\begin{lemma}[Double layer remainder]\label{lem:double_layer}
Let $0<\gamma<\beta<1$ and consider a filament $\Sigma_\epsilon$ with centerline $\X(s)\in C^{2,\beta}(\T)$. Let the double layer remainder $\mc{R}_{\mc{D}}$ be as defined in \eqref{eq:RS_RD_def}. Given $\bm{\psi}\in C^{0,\gamma}(\Gamma_\epsilon)$, we have that $\mc{R}_{\mc{D}}$ satisfies
\begin{equation}\label{eq:RD_est}
\norm{\mc{R}_{\mc{D}}[\bm{\psi}]}_{C^{1,\gamma}}\le c(\kappa_{*,\gamma^+},c_\Gamma)\,\epsilon^{-\gamma^+}\norm{\bm{\psi}}_{C^{0,\gamma}} 
\end{equation}
for any $\gamma^+\in(\gamma,\beta]$.
Furthermore, given two nearby filaments with centerlines $\X^{(a)}(s)$, $\X^{(b)}(s)$ satisfying Lemma \ref{lem:XaXb_C2beta}, the difference between their corresponding double layer remainders $\mc{R}_{\mc{D}}^{(a)}$ and $\mc{R}_{\mc{D}}^{(b)}$ satisfies 
\begin{equation}\label{eq:RD_est_lip}
\norm{\mc{R}_{\mc{D}}^{(a)}[\bm{\psi}]-\mc{R}_{\mc{D}}^{(b)}[\bm{\psi}]}_{C^{1,\gamma}}\le c(\kappa_{*,\gamma^+}^{(a)},\kappa_{*,\gamma^+}^{(b)},c_\Gamma)\,\epsilon^{-\gamma^+}\norm{\X^{(a)}-\X^{(b)}}_{C^{2,\gamma^+}}\norm{\bm{\psi}}_{C^{0,\gamma}}\,. 
\end{equation}
\end{lemma}
The proof of Lemma \ref{lem:double_layer} appears in section \ref{subsec:double_layer}. In addition to Lemmas \ref{lem:single_layer} and \ref{lem:double_layer}, we will require a bound for the single layer operator applied to constant-in-$s$ functions along $\Gamma_\epsilon$. We show the following.

\begin{lemma}[Single layer applied to constant-in-$s$]\label{lem:single_const_in_s}
Let $0<\alpha<\beta<1$ and consider a filament $\Sigma_\epsilon$ with centerline $\X(s)\in C^{2,\beta}(\T)$. Given $\overline{\bm{h}}(\theta)\in C^{0,\alpha}(2\pi\T)$ of the form $\overline{\bm{h}}=\overline h_1(\theta)\be_z+\overline h_2(\theta)\be_x+\overline h_3(\theta)\be_y$, and recalling the map $\Phi$ \eqref{eq:mapPhi_def}, we may decompose
\begin{align*}
\mc{S}[\Phi\overline{\bm{h}}(\theta)] = \mc{H}_\epsilon[\Phi\overline{\bm{h}}(\theta)] + \mc{H}_+[\Phi\overline{\bm{h}}(\theta)]
\end{align*}
where
\begin{equation}
\begin{aligned}
\norm{\mc{H}_\epsilon[\Phi\overline{\bm{h}}]}_{C^{0,\alpha}} &\le c(\kappa_{*,\alpha},c_\Gamma)\,\epsilon^{2-\alpha}\norm{\overline{\bm{h}}}_{L^\infty}\,, \quad
\abs{\mc{H}_\epsilon[\Phi\overline{\bm{h}}]}_{\dot C_s^{1,\alpha}} \le c(\kappa_{*,\alpha^+},c_\Gamma)\,\epsilon^{1-\alpha^+}\norm{\overline{\bm{h}}}_{C^{0,\alpha}} \\
\norm{\mc{H}_+[\Phi\overline{\bm{h}}]}_{C^{0,\alpha}} &\le c(\kappa_{*,\alpha},c_\Gamma)\,\epsilon^{1-\alpha}\norm{\overline{\bm{h}}}_{L^\infty}\,, \quad
\abs{\mc{H}_+[\Phi\overline{\bm{h}}]}_{\dot C_s^{1,\beta}} \le c(\kappa_{*,\beta},c_\Gamma)\,\epsilon^{1-2\beta}\norm{\overline{\bm{h}}}_{C^{0,\alpha}}
\end{aligned}
\end{equation}
for any $\alpha^+\in(\alpha,\beta]$.
In addition, given two nearby curves $\X^{(a)}$, $\X^{(b)}$ satisfying Lemma \ref{lem:XaXb_C2beta}, the corresponding differences $\mc{H}_\epsilon^{(a)}-\mc{H}_\epsilon^{(b)}$ and $\mc{H}_+^{(a)}-\mc{H}_+^{(b)}$ satisfy
\begin{equation}
\begin{aligned}
\norm{\mc{H}_\epsilon^{(a)}[\Phi^{(a)}\overline{\bm{h}}]-\mc{H}_\epsilon^{(b)}[\Phi^{(b)}\overline{\bm{h}}]}_{C^{0,\alpha}} &\le c(\kappa_{*,\alpha}^{(a)},\kappa_{*,\alpha}^{(b)},c_\Gamma)\,\epsilon^{2-\alpha}\norm{\X^{(a)}-\X^{(b)}}_{C^{2,\alpha}}\norm{\overline{\bm{h}}}_{L^\infty} \\
\abs{\mc{H}_\epsilon^{(a)}[\Phi^{(a)}\overline{\bm{h}}]-\mc{H}_\epsilon^{(b)}[\Phi^{(b)}\overline{\bm{h}}]}_{\dot C_s^{1,\alpha}} &\le c(\kappa_{*,\alpha^+}^{(a)},\kappa_{*,\alpha^+}^{(b)},c_\Gamma)\,\epsilon^{1-\alpha^+}\norm{\X^{(a)}-\X^{(b)}}_{C^{2,\alpha^+}}\norm{\overline{\bm{h}}}_{C^{0,\alpha}} \\
\norm{\mc{H}_+^{(a)}[\Phi^{(a)}\overline{\bm{h}}]-\mc{H}_+^{(b)}[\Phi^{(b)}\overline{\bm{h}}]}_{C^{0,\alpha}} &\le c(\kappa_{*,\alpha}^{(a)},\kappa_{*,\alpha}^{(b)},c_\Gamma)\,\epsilon^{1-\alpha}\norm{\X^{(a)}-\X^{(b)}}_{C^{2,\alpha}}\norm{\overline{\bm{h}}}_{L^\infty}\\
\abs{\mc{H}_+^{(a)}[\Phi^{(a)}\overline{\bm{h}}]-\mc{H}_+^{(b)}[\Phi^{(b)}\overline{\bm{h}}]}_{\dot C_s^{1,\beta}} &\le c(\kappa_{*,\beta}^{(a)},\kappa_{*,\beta}^{(b)},c_\Gamma)\,\epsilon^{1-2\beta}\norm{\X^{(a)}-\X^{(b)}}_{C^{2,\beta}}\norm{\overline{\bm{h}}}_{C^{0,\alpha}}\,.
\end{aligned}
\end{equation}
\end{lemma}
The proof of Lemma \ref{lem:single_const_in_s} appears in section \ref{subsec:single_layer_mean}.


The next lemma concerns bounds for the full Neumann data $\bw(s,\theta)$ given by \eqref{eq:BI_rep2} and appearing in the decomposition \eqref{eq:SB_DtN_decomp}.
\begin{lemma}[Bounds for $\bw(s,\theta)$]\label{lem:full_neumann}
Let $0<\alpha<\gamma<\beta<1$ and consider a slender body $\Sigma_\epsilon$ with centerline $\X(s)\in C^{2,\beta}(\T)$.
Given $\theta$-independent Dirichlet data $\bv(s)\in C^{1,\alpha}(\T)$ along the filament surface $\Gamma_\epsilon$, let $\bu(\bx)$, $\bx\in\Omega_\epsilon$, denote the corresponding solution to the Stokes equations in $\Omega_\epsilon$, and let $\bw(s,\theta)$ denote the corresponding Neumann boundary value $\bw=-\bm{\sigma}[\bu]\bm{n}\big|_{\Gamma_\epsilon}$. Then $\bw$ may be decomposed as 
\begin{align*}
\bw(s,\theta) = \bw_\epsilon(s,\theta) + \bw_+(s,\theta)
\end{align*}
where
\begin{equation}\label{eq:fullN}
\begin{aligned}
\norm{\bw_\epsilon}_{C^{0,\alpha}(\Gamma_\epsilon)} &\le c(\kappa_*,c_\Gamma)\norm{\bv}_{C^{1,\alpha}(\T)}\\
\norm{\bw_+}_{C^{0,\gamma}(\Gamma_\epsilon)}&\le c(\epsilon,\kappa_{*,\gamma^+},c_\Gamma)\,\norm{\bv}_{C^{0,\gamma}(\T)}
\end{aligned}
\end{equation}
for any $\gamma^+\in(\gamma,\beta]$.
In addition, given two nearby filaments $\Sigma_\epsilon^{(a)},\Sigma_\epsilon^{(b)}$ with centerlines $\X^{(a)}$ and $\X^{(b)}$ satisfying Lemma \ref{lem:XaXb_C2beta}, we have
\begin{equation}\label{eq:fullN_lip}
\begin{aligned}
\norm{\bw_\epsilon^{(a)}-\bw_\epsilon^{(b)}}_{C^{0,\alpha}(\Gamma_\epsilon)}&\le c(\kappa_*^{(a)},\kappa_*^{(b)},c_\Gamma)\norm{\X^{(a)}-\X^{(b)}}_{C^2(\T)}\norm{\bv}_{C^{1,\alpha}(\T)}\\
\norm{\bw_+^{(a)}-\bw_+^{(b)}}_{C^{0,\gamma}(\Gamma_\epsilon)}&\le c(\epsilon,\kappa_{*,\gamma^+}^{(a)},\kappa_{*,\gamma^+}^{(b)},c_\Gamma)\,\norm{\X^{(a)}-\X^{(b)}}_{C^{2,\gamma^+}(\T)}\norm{\bv}_{C^{0,\gamma}(\T)}\,,
\end{aligned}
\end{equation}
where $\bw^{(a)},\bw^{(b)}$ denote the Neumann data along filaments $\Sigma_\epsilon^{(a)}$ and $\Sigma_\epsilon^{(b)}$, respectively.
\end{lemma}
In both \eqref{eq:fullN} and \eqref{eq:fullN_lip}, the $\epsilon$-dependence is explicit in the bounds for $\bw_\epsilon$ but not for $\bw_+$.
The proof of Lemma \ref{lem:full_neumann} is given in section \ref{sec:full_neumann}.

The final key lemma concerns H\"older space estimates for the slender body NtD map, which will be used in inverting the representation \eqref{eq:SB_DtN_decomp} of $\mc{L}_\epsilon^{-1}$ to obtain a decomposition of $\mc{L}_\epsilon$ instead.
\begin{lemma}[Slender body NtD in H\"older spaces]\label{lem:holder_NtD}
For $0<\alpha<\alpha^+<1$, given a filament $\Sigma_\epsilon$ with centerline $\X(s)\in C^{3,\alpha}(\T)$ and slender body force data $\bm{f}(s)\in C^{0,\alpha}(\T)$, the slender body NtD map $\mc{L}_\epsilon[\bm{f}]$ \eqref{eq:SB_NtD} satisfies
\begin{equation}\label{eq:holderNTD1}
\norm{\mc{L}_\epsilon[\bm{f}]}_{C^{1,\alpha}(\T)} \le c(\epsilon,\norm{\X_{ss}}_{C^{1,\alpha}} ,c_\Gamma)\norm{\bm{f}}_{C^{0,\alpha}(\T)}\,. 
\end{equation}
In addition, given two nearby filaments with centerlines $\X^{(a)}(s)$, $\X^{(b)}(s)$ satisfying Lemma \ref{lem:XaXb_C2beta}, the difference between the corresponding NtD maps $(\mc{L}_\epsilon^{(a)}-\mc{L}_\epsilon^{(b)})[\bm{f}]$ satisfies
\begin{equation}\label{eq:holderNTD2}
\norm{(\mc{L}_\epsilon^{(a)}-\mc{L}_\epsilon^{(b)})[\bm{f}]}_{C^{1,\alpha}(\T)} \le c(\epsilon,\|\X_{ss}^{(a)}\|_{C^{1,\alpha}},\|\X_{ss}^{(b)}\|_{C^{1,\alpha}},c_\Gamma)\norm{\X^{(a)}-\X^{(b)}}_{C^{2,\alpha^+}(\T)}\norm{\bm{f}}_{C^{0,\alpha}(\T)}\,.
\end{equation}
\end{lemma}
Note that additional $C^{3,\alpha}$ regularity is needed on the filament centerline to show Lemma \ref{lem:holder_NtD}. Furthermore, the $\epsilon$-dependence in the bounds \eqref{eq:holderNTD1} and \eqref{eq:holderNTD2} is not explicit. The proof of Lemma \ref{lem:holder_NtD} appears in section \ref{sec:SBNtD_holder}.


\section{Proofs of Theorems \ref{thm:decomp} and \ref{thm:main}}\label{sec:pf_maintheorems}
Given the decomposition \eqref{eq:SB_DtN_decomp} of $\mc{L}_\epsilon^{-1}$ along with Lemmas \ref{lem:Sinv_mapping}, \ref{lem:single_layer}, \ref{lem:double_layer}, \ref{lem:single_const_in_s}, \ref{lem:full_neumann}, and \ref{lem:holder_NtD}, we may proceed immediately to the proof of Theorem \ref{thm:decomp} in section \ref{subsec:pf_decomp}. Given Theorem \ref{thm:decomp} and Lemma \ref{lem:straight_Leps_Holder}, we may then prove Theorem \ref{thm:main} for the simple evolution \eqref{eq:curve_evolution} in section \ref{subsec:pf_dynamics}. The remainder of this paper is devoted to the proofs of each of the above key lemmas.

\subsection{Proof of Theorem \ref{thm:decomp}: operator decomposition}\label{subsec:pf_decomp}
We begin by stating some properties of the map $\Phi$ \eqref{eq:mapPhi_def} used to identify the tangential and normal directions along curved $\Gamma_\epsilon$ with the tangential and normal directions along the straight filament $\mc{C}_\epsilon$. 
Note that the map $\Phi$ commutes with derivatives in $\theta$ along $\Gamma_\epsilon$ but does not commute with derivatives in $s$. In particular, we have
\begin{equation}\label{eq:Phi_commutator0}
\begin{aligned}
\p_s(\Phi\bm{g}) &= \Phi(\p_s\bm{g}) - \bm{g}\cdot(\kappa_1\be_x+\kappa_2\be_y)\be_{\rm t} 
+ \bm{g}\cdot(\kappa_1\be_z-\kappa_3\be_y)\be_{\rm n_1} \\
&\qquad + \bm{g}\cdot(\kappa_2\be_z+\kappa_3\be_x)\be_{\rm n_2} \\
\p_s(\Phi^{-1}\bm{h}) &= \Phi^{-1}(\p_s\bm{h}) + \bm{h}\cdot(\kappa_1\be_{\rm n_1}+\kappa_2\be_{\rm n_2})\be_z + \bm{h}\cdot(\kappa_3\be_{\rm n_2}-\kappa_1\be_{\rm t})\be_x \\
&\qquad  -\bm{h}\cdot(\kappa_3\be_{\rm n_1}+\kappa_2\be_{\rm t})\be_y\,.
\end{aligned}
\end{equation}
From \eqref{eq:Phi_commutator0}, we can see that
\begin{equation}\label{eq:Phi_commutator_est}
\begin{aligned}
\norm{[\p_s,\Phi]\,\bm{g}}_{L^\infty} &\le c(\kappa_*)\norm{\bm{g}}_{L^\infty}\,, \qquad 
\norm{[\p_s,\Phi]\,\bm{g}}_{C^{0,\alpha}} \le c(\kappa_{*,\alpha})\norm{\bm{g}}_{C^{0,\alpha}} \\
\norm{[\p_s,\Phi^{-1}]\,\bm{h}}_{L^\infty}&\le c(\kappa_*)\norm{\bm{h}}_{L^\infty}\,, \qquad 
\norm{[\p_s,\Phi^{-1}]\,\bm{h}}_{C^{0,\alpha}} \le c(\kappa_{*,\alpha})\norm{\bm{h}}_{C^{0,\alpha}}\,.
\end{aligned}
\end{equation}
Furthermore, given two nearby filaments with centerlines $\X^{(a)}$ and $\X^{(b)}$ satisfying Lemma \ref{lem:XaXb_C2beta}, we note the following form of the difference between the corresponding maps $\Phi^{(a)}$ and $\Phi^{(b)}$:
\begin{equation}\label{eq:Phi_lip_1}
\begin{aligned}
\Phi^{(a)}\bm{g}-\Phi^{(b)}\bm{g} &= (\bm{g}\cdot\be_z)(\be_{\rm t}^{(a)}-\be_{\rm t}^{(b)}) + (\bm{g}\cdot\be_x)(\be_{\rm n_1}^{(a)}-\be_{\rm n_1}^{(b)}) + (\bm{g}\cdot\be_y)(\be_{\rm n_2}^{(a)}-\be_{\rm n_2}^{(b)})\\
(\Phi^{(a)})^{-1}\bm{h}-(\Phi^{(b)})^{-1}\bm{h} &= \bm{h}\cdot(\be_{\rm t}^{(a)}-\be_{\rm t}^{(b)})\be_z + \bm{h}\cdot(\be_{\rm n_1}^{(a)}-\be_{\rm n_1}^{(b)})\be_x + \bm{h}\cdot(\be_{\rm n_2}^{(a)}-\be_{\rm n_2}^{(b)})\be_y\,.
\end{aligned}
\end{equation}
Recalling the definition \eqref{eq:subtract_mean} of $\bm{h}_0^\Phi$, we note that the subtracted mean differs slightly for different curves. However, using the definition \eqref{eq:mapPhi_def} of $\Phi$, we may write
\begin{equation}\label{eq:Phi_lip_2}
\begin{aligned}
(\Phi^{(a)})^{-1}\bm{h}_0^{\Phi^{(a)}}-(\Phi^{(b)})^{-1}\bm{h}_0^{\Phi^{(b)}} &= (\Phi^{(a)})^{-1}\bm{h}-(\Phi^{(b)})^{-1}\bm{h}
- \be_z\int_\T\bm{h}\cdot\big(\be_{\rm t}^{(a)}-\be_{\rm t}^{(b)}\big)\,ds \\
&\quad -\be_x\int_\T\bm{h}\cdot\big(\be_{\rm n_1}^{(a)}-\be_{\rm n_1}^{(b)}\big)\,ds  -\be_y\int_\T\bm{h}\cdot\big(\be_{\rm n_2}^{(a)}-\be_{\rm n_2}^{(b)}\big)\,ds\,.
\end{aligned}
\end{equation}
Combining \eqref{eq:Phi_lip_1} and \eqref{eq:Phi_lip_2}, we thus have
\begin{equation}\label{eq:Phi_lip_est}
\begin{aligned}
\norm{(\Phi^{(a)})^{-1}\bm{h}_0^{\Phi^{(a)}}-(\Phi^{(b)})^{-1}\bm{h}_0^{\Phi^{(b)}}}_{L^\infty} &\le c(\kappa_*^{(a)},\kappa_*^{(b)})\norm{\X^{(a)}-\X^{(b)}}_{C^2}\norm{\bm{h}}_{L^\infty} \\
\norm{(\Phi^{(a)})^{-1}\bm{h}_0^{\Phi^{(a)}}-(\Phi^{(b)})^{-1}\bm{h}_0^{\Phi^{(b)}}}_{C^{0,\alpha}} &\le c(\kappa_{*,\alpha}^{(a)},\kappa_{*,\alpha}^{(b)})\norm{\X^{(a)}-\X^{(b)}}_{C^{2,\alpha}}\norm{\bm{h}}_{C^{0,\alpha}} \,.
\end{aligned}
\end{equation}
In addition, from \eqref{eq:Phi_commutator0}, we obtain
\begin{equation}\label{eq:Phi_comm_lip_est}
\begin{aligned}
\norm{\big[\p_s,(\Phi^{(a)})^{-1}\big]\bm{h}_0^{\Phi^{(a)}}-\big[\p_s,(\Phi^{(b)})^{-1}\big]\bm{h}_0^{\Phi^{(b)}}}_{L^\infty} &\le c(\kappa_*^{(a)},\kappa_*^{(b)})\norm{\X^{(a)}-\X^{(b)}}_{C^2}\norm{\bm{h}}_{L^\infty} \\
\norm{\big[\p_s,(\Phi^{(a)})^{-1}\big]\bm{h}_0^{\Phi^{(a)}}-\big[\p_s,(\Phi^{(b)})^{-1}\big]\bm{h}_0^{\Phi^{(b)}}}_{C^{0,\alpha}} &\le c(\kappa_{*,\alpha}^{(a)},\kappa_{*,\alpha}^{(b)})\norm{\X^{(a)}-\X^{(b)}}_{C^{2,\alpha}}\norm{\bm{h}}_{C^{0,\alpha}} \,.
\end{aligned}
\end{equation}

We may now proceed with the decomposition \eqref{eq:thm_DtN_decomp} of $\mc{L}_\epsilon^{-1}$. Recalling the representation \eqref{eq:SB_DtN_decomp} of $\mc{L}_\epsilon^{-1}$, we first separate the remainder terms into terms which are small in $\epsilon$ and terms which are smoother. Using the decompositions $\overline{\mc{A}}_\epsilon^{\epsilon/+}$, $\mc{R}_{\mc{S},\epsilon/+}$, $\mc{H}_{\epsilon/+}$ and $\bm{w}_{\epsilon/+}$ from Lemmas \ref{lem:Sinv_mapping}, \ref{lem:single_layer}, \ref{lem:single_const_in_s}, and \ref{lem:full_neumann}, respectively, we define the terms $\mc{R}_{\rm d,\epsilon}$ and $\mc{R}_{\rm d,+}$ in \eqref{eq:thm_DtN_decomp} as 
\begin{align*}
\mc{R}_{\rm d,\epsilon}[\bv(s)] &= - \int_0^{2\pi}\bm{w}_\epsilon(s,\theta)\,\epsilon^2\wh\kappa\,d\theta
- \Phi\overline{\A}_\epsilon^\epsilon\big[\P_{01}\mc{R}_{\mc{S},\epsilon}[(\bm{w}_\epsilon)_0^\Phi]\big]- \Phi\overline{\A}_\epsilon^\epsilon\bigg[\P_{01}\Phi^{-1}\mc{H}_\epsilon\bigg[\Phi\int_\T\Phi^{-1}\bw_\epsilon\,ds \bigg]\bigg] \\
\mc{R}_{\rm d,+}[\bv(s)] &= - \int_0^{2\pi}\bm{w}_+(s,\theta)\,\epsilon^2\wh\kappa\,d\theta + \Phi\int_\T \Phi^{-1}\int_0^{2\pi}\bm{w}(s,\theta)\,\epsilon\, d\theta ds
- \Phi\overline{\A}_\epsilon\big[\P_{01}\mc{R}_{\mc{D}}[\bv]\big]\\
&\qquad
- \Phi\overline{\A}_\epsilon^+\big[\P_{01}\mc{R}_{\mc{S}}[\bm{w}_0^\Phi]\big] 
- \Phi\overline{\A}_\epsilon^\epsilon\big[\P_{01}\mc{R}_{\mc{S},+}[\bm{w}_0^\Phi]\big] 
- \Phi\overline{\A}_\epsilon^\epsilon\big[\P_{01}\mc{R}_{\mc{S},\epsilon}[(\bm{w}_+)_0^\Phi]\big]  \\
&\qquad
- \Phi\overline{\A}_\epsilon^+\bigg[\P_{01}\Phi^{-1}\mc{S}\displaystyle\bigg[\Phi\int_\T\Phi^{-1}\bw\,ds \bigg]\bigg]
-\Phi\overline{\A}_\epsilon^\epsilon\bigg[\P_{01}\Phi^{-1}\mc{H}_+\displaystyle\bigg[\Phi\int_\T\Phi^{-1}\bw\,ds \bigg]\bigg]\\
&\qquad -\Phi\overline{\A}_\epsilon^\epsilon\bigg[\P_{01}\Phi^{-1}\mc{H}_\epsilon\displaystyle\bigg[\Phi\int_\T\Phi^{-1}\bw_+\,ds \bigg]\bigg]
\,.
\end{align*}

We first bound $\mc{R}_{\rm d,\epsilon}$.
Using Lemmas \ref{lem:Sinv_mapping}, \ref{lem:single_layer}, and \ref{lem:full_neumann}, we make note of the following estimates for the second term:
\begin{align*}
&\norm{\Phi\overline{\A}_\epsilon^\epsilon\big[\P_{01}\mc{R}_{\mc{S},\epsilon}[(\bm{w}_\epsilon)_0^\Phi]\big]}_{C^{0,\alpha}} 
\le 
c(\kappa_*)\big(\norm{\mc{R}_{\mc{S},\epsilon}[(\bm{w}_\epsilon)_0^\Phi]}_{C_s^{0,\alpha}} + \epsilon\abs{\mc{R}_{\mc{S},\epsilon}[(\bm{w}_\epsilon)_0^\Phi]}_{\dot C_s^{1,\alpha}} \big) \\
&\qquad\le 
 c(\kappa_{*,\alpha^+},c_\Gamma)\,\epsilon^{2-\alpha^+}\norm{\bm{w}_\epsilon}_{C^{0,\alpha}} 
 \le c(\kappa_{*,\alpha^+},c_\Gamma)\,\epsilon^{2-\alpha^+}\norm{\bv}_{C^{1,\alpha}(\T)} \,, \\
&\norm{\Phi^{(a)}\overline{\A}_\epsilon^\epsilon\big[\P_{01}\mc{R}_{\mc{S},\epsilon}^{(a)}[(\bm{w}_\epsilon^{(a)})_0^{\Phi^{(a)}}]\big]-\Phi^{(b)}\overline{\A}_\epsilon^\epsilon\big[\P_{01}\mc{R}_{\mc{S},\epsilon}^{(b)}[(\bm{w}_\epsilon^{(b)})_0^{\Phi^{(b)}}]\big]}_{C^{0,\alpha}}\\
&\qquad \le c(\kappa_*^{(a)},\kappa_*^{(b)}) \bigg(\norm{\X^{(a)}-\X^{(b)}}_{C^2}\big(\norm{\mc{R}_{\mc{S},\epsilon}^{(a)}[(\bm{w}_\epsilon^{(a)})_0^{\Phi^{(a)}}]\big]}_{C_s^{0,\alpha}}+ \epsilon\abs{\mc{R}_{\mc{S},\epsilon}^{(a)}[(\bm{w}_\epsilon^{(a)})_0^{\Phi^{(a)}}]\big]}_{\dot C_s^{1,\alpha}} \big)\\
&\qquad \quad 
+\norm{\mc{R}_{\mc{S},\epsilon}^{(a)}[(\bm{w}_\epsilon^{(a)})_0^{\Phi^{(a)}}]-\mc{R}_{\mc{S},\epsilon}^{(b)}[(\bm{w}_\epsilon^{(b)})_0^{\Phi^{(b)}}]}_{C_s^{0,\alpha}}\
+\epsilon\abs{\mc{R}_{\mc{S},\epsilon}^{(a)}[(\bm{w}_\epsilon^{(a)})_0^{\Phi^{(a)}}]-\mc{R}_{\mc{S},\epsilon}^{(b)}[(\bm{w}_\epsilon^{(b)})_0^{\Phi^{(b)}}]}_{C_s^{1,\alpha}} \bigg) \\
&\qquad\le c(\kappa_{*,\alpha^+}^{(a)},\kappa_{*,\alpha^+}^{(b)},c_\Gamma)\,\epsilon^{2-\alpha^+}\big(\norm{\X^{(a)}-\X^{(b)}}_{C^{2,\alpha^+}}\norm{\bw_\epsilon^{(a)}}_{C^{0,\alpha}} + \norm{\bw_\epsilon^{(a)}-\bw_\epsilon^{(b)}}_{C^{0,\alpha}}\big)\\
&\qquad\le c(\kappa_{*,\alpha^+}^{(a)},\kappa_{*,\alpha^+}^{(b)},c_\Gamma)\,\epsilon^{2-\alpha^+}\norm{\X^{(a)}-\X^{(b)}}_{C^{2,\alpha^+}}\norm{\bv}_{C^{1,\alpha}(\T)}\,.
\end{align*}
For the third term of $\mc{R}_{\rm d,\epsilon}$, we may use Lemmas \ref{lem:Sinv_mapping}, \ref{lem:single_const_in_s}, and \ref{lem:full_neumann} to estimate
\begin{align*}
&\norm{\Phi\overline{\A}_\epsilon^\epsilon\bigg[\P_{01}\Phi^{-1}\mc{H}_\epsilon\bigg[\Phi\int_\T\Phi^{-1}\bw_\epsilon\,ds \bigg]\bigg]}_{C^{0,\alpha}}\\
&\quad\le c(\kappa_*)\bigg(\norm{\mc{H}_\epsilon\bigg[\Phi\int_\T\Phi^{-1}\bw_\epsilon\,ds \bigg]}_{C_s^{0,\alpha}} +\epsilon\abs{\mc{H}_\epsilon\bigg[\Phi\int_\T\Phi^{-1}\bw_\epsilon\,ds \bigg]}_{\dot C_s^{1,\alpha}} \bigg)\\
&\quad \le c(\kappa_{*,\alpha^+},c_\Gamma)\,\epsilon^{2-\alpha^+}\norm{\bw_\epsilon}_{C^{0,\alpha}}
\le c(\kappa_{*,\alpha^+},c_\Gamma)\,\epsilon^{2-\alpha^+}\norm{\bv}_{C^{1,\alpha}(\T)}\,,\\
&\norm{(\Phi^{(a)}-\Phi^{(b)})\overline{\A}_\epsilon^\epsilon\bigg[\P_{01}(\Phi^{(a)})^{-1}\mc{H}_\epsilon^{(a)}\bigg[\Phi^{(a)}\int_\T(\Phi^{(a)})^{-1}\bw_\epsilon^{(a)}\,ds \bigg]\bigg]}_{C^{0,\alpha}}\\
&+ \norm{\Phi^{(b)}\overline{\A}_\epsilon^\epsilon\bigg[\P_{01}\bigg((\Phi^{(a)})^{-1}\mc{H}_\epsilon^{(a)}\bigg[\Phi^{(a)}\int_\T(\Phi^{(a)})^{-1}\bw_\epsilon^{(a)}\,ds \bigg]-(\Phi^{(b)})^{-1}\mc{H}_\epsilon^{(b)}\bigg[\Phi^{(b)}\int_\T(\Phi^{(b)})^{-1}\bw_\epsilon^{(b)}\,ds \bigg]\bigg)\bigg]}_{C^{0,\alpha}}\\
&\quad\le c(\kappa_*^{(a)},\kappa_*^{(b)})\bigg(\norm{\X^{(a)}-\X^{(b)}}_{C^{2,\alpha}}\norm{\mc{H}_\epsilon^{(a)}\bigg[\Phi^{(a)}\int_\T(\Phi^{(a)})^{-1}\bw_\epsilon^{(a)}\,ds \bigg]}_{C_s^{0,\alpha}}\\
&\qquad 
+ \norm{\X^{(a)}-\X^{(b)}}_{C^{2,\alpha}}\,\epsilon\abs{\mc{H}_\epsilon^{(a)}\bigg[\Phi^{(a)}\int_\T(\Phi^{(a)})^{-1}\bw_\epsilon^{(a)}\,ds \bigg]}_{\dot C_s^{1,\alpha}}\\
&\qquad
 + \norm{(\Phi^{(a)})^{-1}\mc{H}_\epsilon^{(a)}\bigg[\Phi^{(a)}\int_\T(\Phi^{(a)})^{-1}\bw_\epsilon^{(a)}\,ds \bigg]-(\Phi^{(b)})^{-1}\mc{H}_\epsilon^{(b)}\bigg[\Phi^{(b)}\int_\T(\Phi^{(b)})^{-1}\bw_\epsilon^{(b)}\,ds \bigg]}_{C_s^{0,\alpha}}\\
 &\qquad
 + \epsilon\abs{(\Phi^{(a)})^{-1}\mc{H}_\epsilon^{(a)}\bigg[\Phi^{(a)}\int_\T(\Phi^{(a)})^{-1}\bw_\epsilon^{(a)}\,ds \bigg]-(\Phi^{(b)})^{-1}\mc{H}_\epsilon^{(b)}\bigg[\Phi^{(b)}\int_\T(\Phi^{(b)})^{-1}\bw_\epsilon^{(b)}\,ds \bigg]}_{\dot C_s^{1,\alpha}}\bigg) \\
 &\quad \le c(\kappa_{*,\alpha^+}^{(a)},\kappa_{*,\alpha^+}^{(b)},c_\Gamma)\,\epsilon^{2-\alpha^+}\big(\norm{\X^{(a)}-\X^{(b)}}_{C^{2,\alpha^+}}\norm{\bw_\epsilon^{(a)}}_{C^{0,\alpha}}+\norm{\bw_\epsilon^{(a)}-\bw_\epsilon^{(b)}}_{C^{0,\alpha}}\big)\\
 &\quad \le c(\kappa_{*,\alpha^+}^{(a)},\kappa_{*,\alpha^+}^{(b)},c_\Gamma)\,\epsilon^{2-\alpha^+}\norm{\X^{(a)}-\X^{(b)}}_{C^{2,\alpha^+}}\norm{\bv}_{C^{1,\alpha}(\T)}\,.
\end{align*}
Altogether, we may bound $\mc{R}_{\rm d,\epsilon}$ as
\begin{align*}
\norm{\mc{R}_{\rm d,\epsilon}[\bv]}_{C^{0,\alpha}(\T)} 
&\le c(\kappa_{*,\alpha}) \epsilon^{2-\alpha}\norm{\bm{w}_\epsilon}_{C^{0,\alpha}} + c(\kappa_{*,\alpha^+},c_\Gamma)\,\epsilon^{2-\alpha^+}\norm{\bv}_{C^{1,\alpha}(\T)}\\
&\le c(\kappa_{*,\alpha^+},c_\Gamma)\,\epsilon^{2-\alpha^+}\norm{\bv}_{C^{1,\alpha}(\T)} \,,\\
\norm{(\mc{R}_{\rm d,\epsilon}^{(a)}-\mc{R}_{\rm d,\epsilon}^{(b)})[\bv]}_{C^{0,\alpha}(\T)} &\le \epsilon^2\norm{\int_0^{2\pi}\big(\bm{w}_\epsilon^{(a)}\,\wh\kappa^{(a)}-\bm{w}_\epsilon^{(b)}\,\wh\kappa^{(b)}\big)\,d\theta}_{C^{0,\alpha}}\\
&\quad + c(\kappa_{*,\alpha^+}^{(a)},\kappa_{*,\alpha^+}^{(b)},c_\Gamma)\,\epsilon^{2-\alpha^+}\norm{\X^{(a)}-\X^{(b)}}_{C^{2,\alpha^+}}\norm{\bv}_{C^{1,\alpha}(\T)} \\
&\le c(\kappa_{*,\alpha^+}^{(a)},\kappa_{*,\alpha^+}^{(b)},c_\Gamma)\,\epsilon^{2-\alpha^+}\norm{\X^{(a)}-\X^{(b)}}_{C^{2,\alpha^+}}\norm{\bv}_{C^{1,\alpha}(\T)}\,.
\end{align*}

We next consider $\mc{R}_{\rm d,+}$. For $0<\alpha<\gamma<\beta<1$, we may use Lemmas \ref{lem:Sinv_mapping}, \ref{lem:double_layer}, and \ref{lem:full_neumann} to estimate the third term involving the double layer remainder as 
\begin{align*}
&\norm{\Phi\overline{\A}_\epsilon\big[\P_{01}\mc{R}_{\mc{D}}[\bv]\big]}_{C^{0,\gamma}} \le c(\kappa_*)\,\epsilon^{-1}\norm{\mc{R}_{\mc{D}}[\bv]}_{C_s^{1,\gamma}} \le c(\kappa_{*,\gamma^+},c_\Gamma)\,\epsilon^{-1-\gamma^+}\norm{\bv}_{C^{0,\gamma}(\T)}\,,\\
&\norm{\Phi^{(a)}\overline{\A}_\epsilon\big[\P_{01}\mc{R}_{\mc{D}}^{(a)}[\bv]\big]-\Phi^{(b)}\overline{\A}_\epsilon\big[\P_{01}\mc{R}_{\mc{D}}^{(b)}[\bv]\big]}_{C^{0,\gamma}} \\
&\qquad\qquad\le c(\kappa_*^{(a)},\kappa_*^{(b)})\,\epsilon^{-1}\big(\norm{\X^{(a)}-\X^{(b)}}_{C^2}\norm{\mc{R}_{\mc{D}}^{(a)}[\bv]}_{C_s^{1,\gamma}} 
+ \norm{(\mc{R}_{\mc{D}}^{(a)}-\mc{R}_{\mc{D}}^{(b)})[\bv]}_{C_s^{1,\gamma}} \big)\\
&\qquad\qquad \le c(\kappa_{*,\gamma^+}^{(a)},\kappa_{*,\gamma^+}^{(b)},c_\Gamma)\,\epsilon^{-1-\gamma^+}\norm{\X^{(a)}-\X^{(b)}}_{C^{2,\gamma^+}}\norm{\bv}_{C^{0,\gamma}(\T)}\,.
\end{align*}
Letting $\Phi\mc{M}_1=\Phi(\overline{\A}_\epsilon^+\big[\P_{01}\mc{R}_{\mc{S}}[\bm{w}_0^\Phi]\big] 
+\overline{\A}_\epsilon^\epsilon\big[\P_{01}\mc{R}_{\mc{S},+}[\bm{w}_0^\Phi]\big] 
+\overline{\A}_\epsilon^\epsilon\big[\P_{01}\mc{R}_{\mc{S},\epsilon}[(\bm{w}_+)_0^\Phi]\big])$ denote the middle three terms of $\mc{R}_{\rm d,+}$, we may use Lemmas \ref{lem:Sinv_mapping}, \ref{lem:single_layer}, and \ref{lem:full_neumann} to estimate
\begin{align*}
&\norm{\Phi\mc{M}_1}_{C^{0,\gamma}} \le c(\kappa_*)\big(\epsilon^{-1}\norm{\mc{R}_{\mc{S}}[\bm{w}_0^\Phi]}_{C_s^{1,\alpha}} 
+\norm{\mc{R}_{\mc{S},+}[\bm{w}_0^\Phi]}_{C_s^{1,\gamma}} 
+\norm{\mc{R}_{\mc{S},\epsilon}[(\bm{w}_+)_0^\Phi]}_{C_s^{1,\gamma}}\big)\\
&\qquad\qquad\quad\, \le c(\epsilon,\kappa_{*,\gamma^+},c_\Gamma)\big(\norm{\bm{w}}_{C^{0,\alpha}} 
+\norm{\bm{w}_+}_{C^{0,\gamma}}\big)
\le c(\epsilon,\kappa_{*,\gamma^+},c_\Gamma)\norm{\bv}_{C^{1,\alpha}(\T)}\,, \\
&\norm{\Phi^{(a)}\mc{M}_1^{(a)}-\Phi^{(b)}\mc{M}_1^{(b)}}_{C^{0,\gamma}}\\ 
&\quad \le 
c(\epsilon,\kappa_*^{(a)},\kappa_*^{(b)})\bigg(\norm{\X^{(a)}-\X^{(b)}}_{C^2}\bigg(\norm{\mc{R}_{\mc{S}}^{(a)}[(\bm{w}_0^{(a)})^{\Phi^{(a)}}]}_{C_s^{1,\alpha}} 
+\norm{\mc{R}_{\mc{S},+}^{(a)}[(\bm{w}_0^{(a)})^{\Phi^{(a)}}]}_{C_s^{1,\gamma}}\\ 
&\qquad +\norm{\mc{R}_{\mc{S},\epsilon}^{(a)}[(\bm{w}_+^{(a)})_0^{\Phi^{(a)}}]}_{C_s^{1,\gamma}} \bigg)
 + \norm{\mc{R}_{\mc{S}}^{(a)}[(\bm{w}_0^{(a)})^{\Phi^{(a)}}]-\mc{R}_{\mc{S}}^{(b)}[(\bm{w}_0^{(b)})^{\Phi^{(b)}}]}_{C_s^{1,\alpha}} \\
&\qquad +\norm{\mc{R}_{\mc{S},+}^{(a)}[(\bm{w}_0^{(a)})^{\Phi^{(a)}}]-\mc{R}_{\mc{S},+}^{(b)}[(\bm{w}_0^{(b)})^{\Phi^{(b)}}]}_{C_s^{1,\gamma}}
 +\norm{\mc{R}_{\mc{S},\epsilon}^{(a)}[(\bm{w}_+^{(a)})_0^{\Phi^{(a)}}]-\mc{R}_{\mc{S},\epsilon}^{(b)}[(\bm{w}_+^{(b)})_0^{\Phi^{(b)}}]}_{C_s^{1,\gamma}}\bigg)\\
&\quad \le c(\epsilon,\kappa_{*,\gamma^+}^{(a)},\kappa_{*,\gamma^+}^{(b)},c_\Gamma)\bigg(\norm{\X^{(a)}-\X^{(b)}}_{C^{2,\gamma^+}}\big(\norm{\bm{w}_0^{(a)}}_{C^{0,\alpha}} +\norm{\bm{w}_+^{(a)}}_{C^{0,\gamma}} \big) \\
 &\qquad + \norm{\bm{w}_0^{(a)}-\bm{w}_0^{(b)}}_{C^{0,\alpha}} +\norm{\bm{w}_+^{(a)}-\bm{w}_+^{(b)}}_{C^{0,\gamma}}\bigg)\\
&\quad \le c(\epsilon,\kappa_{*,\gamma^+}^{(a)},\kappa_{*,\gamma^+}^{(b)},c_\Gamma)\norm{\X^{(a)}-\X^{(b)}}_{C^{2,\gamma^+}}\norm{\bv}_{C^{1,\alpha}(\T)}\,. 
\end{align*}
Finally, letting $\Phi\mc{M}_2= \Phi\big(\overline{\A}_\epsilon^+\big[\P_{01}\Phi^{-1}\mc{S}\big[\Phi\int_\T\Phi^{-1}\bw\,ds \big]\big]$
$+\overline{\A}_\epsilon^\epsilon\big[\P_{01}\Phi^{-1}\mc{H}_+\big[\Phi\int_\T\Phi^{-1}\bw\,ds \big]\big]$ \\
$+\overline{\A}_\epsilon^\epsilon\big[\P_{01}\Phi^{-1}\mc{H}_\epsilon\big[\Phi\int_\T\Phi^{-1}\bw_+\,ds \big]\big]\big)$ denote the final three terms of $\mc{R}_{\rm d,+}$, we may use Lemmas \ref{lem:Sinv_mapping}, \ref{lem:single_const_in_s}, and \ref{lem:full_neumann} to obtain 
\begin{align*}
&\norm{\Phi\mc{M}_2}_{C^{0,\gamma}} \le c(\kappa_*)\bigg(\norm{\mc{S}\big[\Phi\int_\T\Phi^{-1}\bw\,ds \big]\big]}_{C_s^{1,\alpha}}
+\norm{\mc{H}_+\big[\Phi\int_\T\Phi^{-1}\bw\,ds \big]}_{C_s^{1,\gamma}}\\
&\hspace{4cm}+\norm{\mc{H}_\epsilon\big[\Phi\int_\T\Phi^{-1}\bw_+\,ds \big]}_{C_s^{1,\gamma}}\bigg)\\
&\hspace{2cm} \le c(\epsilon,\kappa_{*,\gamma^+},c_\Gamma)\big(\norm{\bw}_{C^{0,\alpha}}+\norm{\bw_+}_{C^{0,\gamma}}\big) 
\le c(\epsilon,\kappa_{*,\gamma^+},c_\Gamma)\norm{\bv}_{C^{1,\alpha}(\T)}\,, \\
&\norm{\Phi^{(a)}\mc{M}_2^{(a)}-\Phi^{(b)}\mc{M}_2^{(b)}}_{C^{0,\gamma}} \\
&\quad \le c(\kappa_*^{(a)},\kappa_*^{(b)})\bigg(\norm{\X^{(a)}-\X^{(b)}}_{C^2}\bigg(\norm{\mc{S}^{(a)}\big[\Phi^{(a)}\int_\T(\Phi^{(a)})^{-1}\bw^{(a)}\,ds \big]\big]}_{C_s^{1,\alpha}}\\
&\qquad +\norm{\mc{H}_+^{(a)}\big[\Phi^{(a)}\int_\T(\Phi^{(a)})^{-1}\bw^{(a)}\,ds \big]}_{C_s^{1,\gamma}}
+\norm{\mc{H}_\epsilon^{(a)}\big[\Phi^{(a)}\int_\T(\Phi^{(a)})^{-1}\bw_+^{(a)}\,ds \big]}_{C_s^{1,\gamma}}\bigg) \\
&\qquad +\norm{\mc{S}^{(a)}\big[\Phi^{(a)}\int_\T(\Phi^{(a)})^{-1}\bw^{(a)}\,ds \big]\big]-\mc{S}^{(b)}\big[\Phi^{(b)}\int_\T(\Phi^{(b)})^{-1}\bw^{(b)}\,ds \big]\big]}_{C_s^{1,\alpha}}\\
&\qquad +\norm{\mc{H}_+^{(a)}\big[\Phi^{(a)}\int_\T(\Phi^{(a)})^{-1}\bw^{(a)}\,ds \big]-\mc{H}_+^{(b)}\big[\Phi^{(b)}\int_\T(\Phi^{(b)})^{-1}\bw^{(b)}\,ds \big]}_{C_s^{1,\gamma}}\\
&\qquad+\norm{\mc{H}_\epsilon^{(a)}\big[\Phi^{(a)}\int_\T(\Phi^{(a)})^{-1}\bw_+^{(a)}\,ds \big]-\mc{H}_\epsilon^{(b)}\big[\Phi^{(b)}\int_\T(\Phi^{(b)})^{-1}\bw_+^{(b)}\,ds \big]}_{C_s^{1,\gamma}}\bigg) \\
&\quad \le c(\epsilon,\kappa_{*,\gamma^+}^{(a)},\kappa_{*,\gamma^+}^{(b)},c_\Gamma)\bigg(\norm{\X^{(a)}-\X^{(b)}}_{C^{2,\gamma^+}}\big(\norm{\bw^{(a)}}_{C^{0,\alpha}}
+\norm{\bw_+^{(a)}}_{C^{0,\gamma}}\big) \\
&\qquad +\norm{\bw^{(a)}-\bw^{(b)}}_{C^{0,\alpha}}
+\norm{\bw_+^{(a)}-\bw_+^{(b)}}_{C^{0,\gamma}}\bigg) \\
&\quad \le c(\epsilon,\kappa_{*,\gamma^+}^{(a)},\kappa_{*,\gamma^+}^{(b)},c_\Gamma)\norm{\X^{(a)}-\X^{(b)}}_{C^{2,\gamma^+}}\norm{\bv}_{C^{1,\alpha}(\T)}\,.
\end{align*}
Altogether, we may estimate $\mc{R}_{\rm d,+}$ as
\begin{equation}\label{eq:rd_plus}
\begin{aligned}
\norm{\mc{R}_{\rm d,+}[\bv]}_{C^{0,\gamma}(\T)}
&\le c(\kappa_{*,\gamma})\,\epsilon^{2-\gamma} \norm{\bm{w}_+}_{C^{0,\gamma}}+\epsilon\norm{\bw}_{L^\infty} 
+ c(\epsilon,\kappa_{*,\gamma^+},c_\Gamma)\norm{\bv}_{C^{1,\alpha}(\T)} \\
&\le c(\epsilon,\kappa_{*,\gamma^+},c_\Gamma)\norm{\bv}_{C^{1,\alpha}(\T)}\,, \\
\norm{(\mc{R}_{\rm d,+}^{(a)}-\mc{R}_{\rm d,+}^{(b)})[\bv]}_{C^{0,\gamma}(\T)}&\le  \epsilon^2\norm{\bm{w}_+^{(a)}(\wh\kappa^{(a)}-\wh\kappa^{(b)})}_{C^{0,\gamma}} + \epsilon^2\norm{\wh\kappa^{(b)}(\bm{w}_+^{(a)}-\bm{w}_+^{(b)})}_{C^{0,\gamma}}  \\
&\quad + \norm{\bigg(\Phi^{(a)}\int_\T(\Phi^{(a)})^{-1}-\Phi^{(b)}\int_\T(\Phi^{(b)})^{-1}\bigg)\int_0^{2\pi}\bw^{(a)}\,\epsilon \,d\theta ds}_{C^{0,\gamma}}\\
&\quad+ c(\kappa_{*,\gamma}^{(b)}) \,\epsilon \norm{\bw^{(a)}-\bw^{(b)}}_{L^\infty}\\ 
&\quad 
+ c(\epsilon,\kappa_{*,\gamma^+}^{(a)},\kappa_{*,\gamma^+}^{(b)},c_\Gamma)\norm{\X^{(a)}-\X^{(b)}}_{C^{2,\gamma^+}}\norm{\bv}_{C^{1,\alpha}(\T)}\\
&\le c(\epsilon,\kappa_{*,\gamma^+}^{(a)},\kappa_{*,\gamma^+}^{(b)},c_\Gamma)\norm{\X^{(a)}-\X^{(b)}}_{C^{2,\gamma^+}}\norm{\bv}_{C^{1,\alpha}(\T)}\,.
\end{aligned}
\end{equation}
In total we obtain the estimates \eqref{eq:thm_DtN_ests} and \eqref{eq:thm_DtN_ests_lip} for the slender body DtN map decomposition \eqref{eq:thm_DtN_decomp}. \\

We next turn to the decomposition \eqref{eq:thm_NtD_decomp} of the slender body NtD operator $\mc{L}_\epsilon$. We will rely on the DtN decomposition \eqref{eq:thm_DtN_decomp} to prove the subsequent bounds \eqref{eq:thm_NtD_ests} and \eqref{eq:thm_NtD_ests_lip}. In particular, given $\X(s)$ and corresponding map $\Phi$, for $\bm{f}(s)\in C^{0,\alpha}(\T)$ and $\bm{f}_0^\Phi(s)=\bm{f}(s)-\Phi\int_\T \Phi^{-1}\bm{f}(s)\,ds$, we may use \eqref{eq:thm_DtN_decomp} to write 
\begin{align*}
\Phi^{-1}\bm{v}(s) &= ({\bf I}+\overline{\mc{L}}_\epsilon(\mc{R}_{\rm d,\epsilon}\Phi)_0^\Phi)^{-1}\overline{\mc{L}}_\epsilon[\Phi^{-1}\bm{f}_0^\Phi(s)] + ({\bf I}+\overline{\mc{L}}_\epsilon(\mc{R}_{\rm d,\epsilon}\Phi)_0^\Phi)^{-1}\overline{\mc{L}}_\epsilon\big[ (\mc{R}_{\rm d,+}[\bv(s)])_0^\Phi\big]\\
&= ({\bf I}+\overline{\mc{L}}_\epsilon(\mc{R}_{\rm d,\epsilon}\Phi)_0^\Phi)^{-1}\overline{\mc{L}}_\epsilon[\Phi^{-1}\bm{f}_0^\Phi(s)] + ({\bf I}+\overline{\mc{L}}_\epsilon(\mc{R}_{\rm d,\epsilon}\Phi)_0^\Phi)^{-1}\overline{\mc{L}}_\epsilon\big[ (\mc{R}_{\rm d,+}[\mc{L}_\epsilon[\bm{f}(s)]])_0^\Phi\big]\,.
\end{align*}
Now, using Lemma \ref{lem:straight_Leps_Holder} and the estimate \eqref{eq:thm_DtN_ests} for $\mc{R}_{\rm d,\epsilon}$, for any $\overline{\bm{g}}\in C^{1,\alpha}(\T)$ we have
\begin{align*}
\norm{\overline{\mc{L}}_\epsilon(\mc{R}_{\rm d,\epsilon}[\Phi \overline{\bm{g}}])_0^\Phi}_{C^{1,\alpha}(\T)} 
&\le c\abs{\log\epsilon}\big(\norm{\mc{R}_{\rm d,\epsilon}[\Phi \overline{\bm{g}}]}_{L^\infty(\T)} + \epsilon^{-1}\big|\mc{R}_{\rm d,\epsilon}[\Phi \overline{\bm{g}}]\big|_{\dot C^{0,\alpha}(\T)} \big)\\
&\le c(\kappa_{*,\alpha^+},c_\Gamma)\,\epsilon^{1-\alpha^+}\abs{\log\epsilon}\norm{\Phi \overline{\bm{g}}}_{C^{1,\alpha}(\T)}\,.
\end{align*}
Thus for $\epsilon$ sufficiently small, we may expand $({\bf I}+\overline{\mc{L}}_\epsilon(\mc{R}_{\rm d,\epsilon}\Phi)_0^\Phi)^{-1}$ as a Neumann series:
\begin{equation}\label{eq:Rn_epsilon}
({\bf I}+\overline{\mc{L}}_\epsilon(\mc{R}_{\rm d,\epsilon}\Phi)_0^\Phi)^{-1}
= {\bf I} + \sum_{j=1}^\infty\big(\overline{\mc{L}}_\epsilon(\mc{R}_{\rm d,\epsilon}\Phi)_0^\Phi \big)^j 
=: {\bf I} + \mc{R}_{\rm n,\epsilon}\,,
\end{equation}
with
\begin{equation}\label{eq:Rn0_epsilon_est}
\norm{\mc{R}_{\rm n,\epsilon}[\bm{g}]}_{C^{1,\alpha}(\T)}
\le c(\kappa_{*,\alpha^+},c_\Gamma)\,\epsilon^{1-\alpha^+}\abs{\log\epsilon}\norm{\bm{g}}_{C^{1,\alpha}(\T)}\,.
\end{equation}
We thus obtain the first part of the decomposition \eqref{eq:thm_NtD_decomp}. 
Furthermore, using the Lipschitz estimate \eqref{eq:thm_DtN_ests_lip} for $\mc{R}_{\rm d,\epsilon}$, we also have 
\begin{equation}\label{eq:Rn0_eps_lip}
\norm{(\mc{R}_{\rm n,\epsilon}^{(a)}-\mc{R}_{\rm n,\epsilon}^{(b)})[\bm{g}]}_{C^{1,\alpha}(\T)} \le c(\kappa_{*,\alpha^+}^{(a)},\kappa_{*,\alpha^+}^{(b)},c_\Gamma)\, \epsilon^{1-\alpha^+}\abs{\log\epsilon}\norm{\X^{(a)}-\X^{(b)}}_{C^{2,\alpha^+}(\T)}\norm{\bm{g}}_{C^{1,\alpha}(\T)}\,.
\end{equation}

In addition, defining
\begin{align*}
\mc{R}_{\rm n,+} &:= ({\bf I}+\overline{\mc{L}}_\epsilon(\mc{R}_{\rm d,\epsilon}\Phi)_0^\Phi)^{-1}\overline{\mc{L}}_\epsilon\big[ (\mc{R}_{\rm d,+}[\mc{L}_\epsilon[\cdot]])_0^\Phi\big]\\
&= ({\bf I}+\mc{R}_{\rm n0,\epsilon})\overline{\mc{L}}_\epsilon\big[ (\mc{R}_{\rm d,+}[\mc{L}_\epsilon[\cdot]])_0^\Phi\big]\,,
\end{align*}
we may use \eqref{eq:Rn0_epsilon_est}, Lemma \ref{lem:straight_Leps_Holder}, \eqref{eq:rd_plus}, and Lemma \ref{lem:holder_NtD} to estimate
\begin{align*}
\norm{\mc{R}_{\rm n,+}[\bm{f}]}_{C^{1,\gamma}(\T)} &\le c(\epsilon,\kappa_{*,\gamma^+},c_\Gamma)\norm{\overline{\mc{L}}_\epsilon\big[ (\mc{R}_{\rm d,+}[\mc{L}_\epsilon[\bm{f}]])_0^\Phi\big]}_{C^{1,\gamma}} 
\le c(\epsilon,\kappa_{*,\gamma^+},c_\Gamma)\norm{\mc{R}_{\rm d,+}[\mc{L}_\epsilon[\bm{f}]]}_{C^{0,\gamma}} \\
&\le c(\epsilon,\kappa_{*,\gamma^+},c_\Gamma)\norm{\mc{L}_\epsilon[\bm{f}]}_{C^{1,\alpha}}
\le c(\epsilon,\|\X_{ss}\|_{C^{1,\alpha}},c_\Gamma)\norm{\bm{f}}_{C^{0,\alpha}(\T)}\,.
\end{align*}
Using the Lipschitz estimates \eqref{eq:Rn0_eps_lip}, \eqref{eq:thm_DtN_ests_lip}, and \eqref{eq:holderNTD2} for $\mc{R}_{\rm n,\epsilon}$, $\mc{R}_{\rm d,+}$, and $\mc{L}_\epsilon$, we obtain 
\begin{align*}
\norm{(\mc{R}_{\rm n,+}^{(a)}-\mc{R}_{\rm n,+}^{(b)})[\bm{f}]}_{C^{1,\gamma}(\T)}
&\le c(\epsilon,\|\X_{ss}^{(a)}\|_{C^{1,\alpha}},\|\X_{ss}^{(b)}\|_{C^{1,\alpha}},c_\Gamma)\norm{\X^{(a)}-\X^{(b)}}_{C^{2,\alpha^+}(\T)}\norm{\bm{f}}_{C^{0,\alpha}(\T)}\,.
\end{align*}
In total, we obtain the decomposition \eqref{eq:thm_NtD_decomp} and estimates \eqref{eq:thm_NtD_ests}, \eqref{eq:thm_NtD_ests_lip} for the slender body NtD map $\mc{L}_\epsilon$.
\hfill \qedsymbol

\begin{proof}[Proof of Corollary \ref{cor:four_derivs}]
Given slender body force data of the form $\bm{f}(s)=\p_s^4\bm{F}$, we may first use the commutator bounds \eqref{eq:Phi_commutator_est} for $\Phi$ to write
\begin{align*}
\Phi^{-1}\p_s^4\bm{F} &= \p_s^2\big(\Phi^{-1}\p_s^2\bm{F} \big)+\mc{R}_{\X,1}\,,\\
\norm{\mc{R}_{\X,1}}_{C^{0,\alpha}}&= \norm{\p_s\big([\p_s,\Phi^{-1}]\p_s^2\bm{F} \big)+[\p_s,\Phi^{-1}]\p_s^3\bm{F}}_{C^{0,\alpha}}\le c(\norm{\X_{ss}}_{C^{1,\alpha}})\norm{\bm{F}}_{C^{3,\alpha}}\,.
\end{align*}
For two nearby filaments with centerlines $\X^{(a)}$ and $\X^{(b)}$ satisfying Lemma \ref{lem:XaXb_C2beta}, using the commutator estimates \eqref{eq:Phi_comm_lip_est}, we additionally have 
\begin{align*}
\norm{\mc{R}_{\X,1}^{(a)}-\mc{R}_{\X,1}^{(b)}}_{C^{0,\alpha}}&\le 
c(\|\X^{(a)}\|_{C^{3,\alpha}},\|\X^{(b)}\|_{C^{3,\alpha}})\,\norm{\X^{(a)}-\X^{(b)}}_{C^{3,\alpha}}\norm{\bm{F}}_{C^{3,\alpha}}\,.
\end{align*}

We next note the following commutator estimate for $\overline{\mc{L}}_\epsilon\p_s^2$. 
\begin{proposition}[Commutator estimate for $\overline{\mc{L}}_\epsilon\p_s^2$]\label{prop:Phi_comm}
Given a matrix-valued function $\bm{A}\in C^{n+1,\alpha}(\T)$ and a vector-valued function $\bm{g}\in C^{n,\alpha}(\T)$ for $0\le n\in \Z$, the commutator $[\overline{\mc{L}}_\epsilon\p_s^2,\bm{A}]\bm{g}$ satisfies
\begin{equation}\label{eq:Phi_comm_est}
\norm{[\overline{\mc{L}}_\epsilon\p_s^2,\bm{A}]\bm{g}}_{C^{n,\alpha}(\T)} \le c\,\epsilon^{-1}\abs{\log\epsilon}\,\norm{\bm{A}}_{C^{n+1,\alpha}}\norm{\bm{g}}_{C^{n,\alpha}(\T)}\,.
\end{equation}
\end{proposition}
The proof of Proposition \ref{prop:Phi_comm} appears in Appendix \ref{app:Phi_comm}.
Applying Proposition \ref{prop:Phi_comm} in the case $\bm{A}=\Phi^{-1}=\be_z\otimes\be_{\rm t}+\be_x\otimes\be_{\rm n_1}+\be_y\otimes\be_{\rm n_2}$ and $n=1$, we have
\begin{align*}
(\overline{\mc{L}}_\epsilon\p_s^2)[\Phi^{-1}\p_s^2\bm{F}] &= \Phi^{-1}\overline{\mc{L}}_\epsilon\p_s^4\bm{F} + \mc{R}_{\X,2}\,,\\
\norm{\mc{R}_{\X,2}}_{C^{1,\gamma}}&= \norm{[\overline{\mc{L}}_\epsilon\p_s^2,\Phi^{-1}]\p_s^2\bm{F}}_{C^{1,\gamma}} \le c(\epsilon,\|\X\|_{C^{3,\gamma}})\norm{\bm{F}}_{C^{3,\gamma}}\,.
\end{align*}
Furthermore, for two nearby curves $\X^{(a)}$ and $\X^{(b)}$, we have
\begin{align*}
\norm{\mc{R}_{\X,2}^{(a)}-\mc{R}_{\X,2}^{(b)}}_{C^{1,\gamma}}& \le c(\epsilon,\|\X^{(a)}\|_{C^{3,\gamma}},\|\X^{(b)}\|_{C^{3,\gamma}})\,\norm{\X^{(a)}-\X^{(b)}}_{C^{3,\gamma}}\norm{\bm{F}}_{C^{3,\gamma}}\,.
\end{align*}

In total, we may rewrite 
\begin{equation}\label{eq:new_forcing}
\Phi\overline{\mc{L}}_\epsilon[\Phi^{-1}(\p_s^4\bm{F})_0^\Phi] = \overline{\mc{L}}_\epsilon[\p_s^4\bm{F}] + \Phi\overline{\mc{L}}_\epsilon(\mc{R}_{\X,1})_0^\Phi+\Phi\mc{R}_{\X,2}\,.
\end{equation}
Defining $\wt{\mc{R}}_{\rm n,+}=\mc{R}_{\rm n,+}+ ({\bf I}+\mc{R}_{\rm n,\epsilon})\Phi\overline{\mc{L}}_\epsilon(\mc{R}_{\X,1})_0^\Phi+ ({\bf I}+\mc{R}_{\rm n,\epsilon})\Phi\mc{R}_{\X,2}$, we obtain Corollary \ref{cor:four_derivs}.
\end{proof}

\subsection{Proof of Theorem \ref{thm:main}: local well-posedness}\label{subsec:pf_dynamics}
We return to the notation $\X^\sigma(\sigma,t)$ of section \ref{subsubsec:thm_statement} for an evolving filament that is possibly changing length over time. 
Using Corollary \ref{cor:four_derivs}, the evolution \eqref{eq:curve_evolution} may be written as 
\begin{align*}
\frac{\p\X^\sigma}{\p t}(\sigma,t) &= -\lambda^{-1}\big(({\bf I} + \mc{R}_{\rm n,\epsilon})\overline{\mc{L}}_\epsilon[\p_\sigma^4\X] +\wt{\mc{R}}_{\rm n,+}[\p_\sigma^4\X]\big)\,,\\
\X^\sigma(\sigma,0)&= (\X^\sigma)^{\rm in}(\sigma)\,, \qquad \abs{\p_\sigma(\X^\sigma)^{\rm in}}=1\,,\\
\lambda(t)&= \int_0^1\abs{\p_\sigma\X^\sigma}\,d\sigma\,.
\end{align*}
We will consider the evolution of the difference $\Y(\sigma,t) = \X^\sigma(\sigma,t)-(\X^\sigma)^{\rm in}(\sigma)$, which satisfies
\begin{equation}\label{eq:Y_pde}
\frac{\p\Y}{\p t}(\sigma,t) = -\lambda^{-1}\bigg(\big({\bf I}+\mc{R}_{\rm n,\epsilon}\big)\overline{\mc{L}}_\epsilon\big[\p_\sigma^4(\Y+(\X^\sigma)^{\rm in})\big] +\wt{\mc{R}}_{\rm n,+}\big[\p_\sigma^4(\Y+(\X^\sigma)^{\rm in})\big]\bigg)\,.
\end{equation}
Using Duhamel's formula, we may write $\Y(\sigma,t)$ as 
\begin{equation}\label{eq:Y_duhamel}
\begin{aligned}
\Y(\sigma,t) &= e^{-\int_{0}^t\lambda^{-1}(\tau)d\tau\;\overline{\mc{L}}_\epsilon\p_\sigma^4}(\X^\sigma)^{\rm in}-(\X^\sigma)^{\rm in}\\
&\quad-\int_0^t e^{-\int_{t'}^t\lambda^{-1}(\tau)d\tau\;\overline{\mc{L}}_\epsilon\p_\sigma^4}\,\lambda^{-1}\bigg(\mc{R}_\epsilon\big[\overline{\mc{L}}_\epsilon\p_\sigma^4(\Y+(\X^\sigma)^{\rm in})\big]+\wt{\mc{R}}_+\big[\p_\sigma^4(\Y+(\X^\sigma)^{\rm in})\big] \bigg)\, dt'\,.
\end{aligned}
\end{equation}
We will rely on the following two lemmas to estimate $\Y$ given by \eqref{eq:Y_duhamel}.
\begin{lemma}[Semigroup estimates]\label{lem:semigroup}
Let $n,m$ be nonnegative integers and $\alpha,\gamma\in(0,1)$ with $n+\alpha\ge m+\gamma$.
For $\bm{V}(s)\in C^{m,\gamma}(\T)$, we have 
\begin{equation}\label{eq:semigroup}
\norm{e^{-t\overline{\mc{L}}_\epsilon\p_s^4}\bm{V}}_{C^{n,\alpha}(\T)} 
\le c\, t^{-\frac{n+\alpha - (m +\gamma)}{3}} \norm{\bm{V}}_{C^{m,\gamma}(\T)}\,.
\end{equation}
In each case, $c$ is an absolute constant bounded independent of $\epsilon$ as $\epsilon\to 0$.
\end{lemma}
For the next lemma, given $T>0$ and $0\le k\in \Z$, we define the function spaces 
\begin{equation}\label{eq:Ykspace_def}
\mc{Y}_k = \big\{ \bm{V}\in L^\infty\big(0,T;C^{k,\alpha}(\T)\big)\,:\,\norm{\bm{V}}_{\mc{Y}_k}<\infty \}\,, \qquad \norm{\cdot}_{\mc{Y}_k}=\esssup_{t\in [0,T]}\norm{\cdot}_{C^{k,\alpha}(\T)} \,.
\end{equation}
We have the following.
\begin{lemma}[Maximal regularity bound]\label{lem:max_reg}
Given $\bm{g}\in \mc{Y}_1$, let
\begin{align*}
\bm{h}(s,t)= \int_0^t e^{-(t-t')\overline{\mc{L}}_\epsilon\p_s^4}\;\bm{g}(s,t')\,dt'\,.
\end{align*}
Then, for some absolute constant $c$,
\begin{equation}\label{eq:max_semi}
\norm{\bm{h}}_{\mc{Y}_4} \le \big(\epsilon + T^{1/4}\abs{\log\epsilon}^{-3/4}+T \big)\,c\norm{\bm{g}}_{\mc{Y}_1} \,.
\end{equation}
\end{lemma}
The proofs of Lemmas \ref{lem:semigroup} and \eqref{lem:max_reg} appear in section \ref{subsec:semigroup}.\\

Now, to track the non-self-intersection property \eqref{eq:cGamma}, we will define the quantity
\begin{align*}
  \abs{\X^\sigma(\cdot, t)}_\star:=\inf_{\sigma\neq\sigma'}\frac{\abs{\X^\sigma(\sigma,t)-\X^\sigma(\sigma',t)}}{\abs{\sigma-\sigma'}}\, .
\end{align*}
Using the definition \eqref{eq:Ykspace_def} of the spaces $\mc{Y}_k$, we consider the closed ball 
\begin{align*}
\mc{B}_{M,\frac{1}{2}}(T)= \big\{ \Y\in \mc{Y}_4\,:\,\norm{\Y}_{\mc{Y}_4}\le M\,, \; \norm{\Y}_{\mc{Y}_1}\le \textstyle \min\{\frac{1}{2}\abs{(\X^\sigma)^{\rm in}}_\star,\frac{1}{2}\} \big\}\,.
\end{align*}
The $\mc{Y}_1$ bound will be used to propagate a non-self-intersection condition and to consider only situations where the filament length remains close to its initial unit length. Note that this bound ensures that
\begin{align*}
   \esssup_{t\in[0,T]}\abs{\X^\sigma(\cdot,t)}_\star \ge \abs{(\X^\sigma)^{\rm in}}_\star - \norm{\Y}_{\mc{Y}_1}
   \ge \frac{1}{2}\abs{(\X^\sigma)^{\rm in}}_\star\,.
\end{align*} 
Furthermore, if $\Y$ satisfies $\norm{\Y}_{\mc{Y}_1}\le \frac{1}{2}$, then, using that $\abs{\p_\sigma(\X^\sigma)^{\rm in}}=1$ and 
\begin{equation}\label{eq:lambda_Y}
\lambda = \int_0^1\abs{\p_\sigma\Y+\p_\sigma(\X^\sigma)^{\rm in}}\,d\sigma\,,
\end{equation}
we have $\frac{1}{2}\le \lambda \le \frac{3}{2}$ and, in particular,
\begin{equation}\label{eq:lambda_inv_bds}
\frac{2}{3}(t-t')\le \int_{t'}^t\lambda^{-1}(\tau)\,d\tau \le 2(t-t')\,.
\end{equation}

Now, for $\Y\in \mc{B}_{M,\frac{1}{2}}(T)$, let $\Lambda[\Y]$ denote the map
\begin{equation}\label{eq:duhamel_map}
\begin{aligned}
\Lambda[\Y] &= e^{-\int_{0}^t\lambda^{-1}(\tau)d\tau\;\overline{\mc{L}}_\epsilon\p_\sigma^4}(\X^\sigma)^{\rm in}-(\X^\sigma)^{\rm in}\\
&\quad -\int_0^t e^{-\int_{t'}^t\lambda^{-1}(\tau)d\tau\;\overline{\mc{L}}_\epsilon\p_\sigma^4}\,\lambda^{-1}\bigg(\mc{R}_\epsilon\big[\overline{\mc{L}}_\epsilon\p_\sigma^4(\Y+(\X^\sigma)^{\rm in})\big]+\wt{\mc{R}}_+\big[\p_\sigma^4(\Y+(\X^\sigma)^{\rm in})\big] \bigg)\,dt' \,.
\end{aligned}
\end{equation}
We show that the map $\Lambda$ takes $\mc{B}_{M,\frac{1}{2}}(T)$ to itself for some choice of $\epsilon,T>0$.

We first estimate the $\mc{Y}_4$ contribution from the initial condition. Note that by \eqref{eq:lambda_inv_bds} and Lemma \ref{lem:semigroup}, we have
\begin{equation}\label{eq:Mchoice}
  \norm{e^{-\int_{0}^t\lambda^{-1}(\tau)d\tau\;\overline{\mc{L}}_\epsilon\p_\sigma^4}(\X^\sigma)^{\rm in}-(\X^\sigma)^{\rm in}}_{C^{4,\alpha}(\T)}\le
  c\norm{(\X^\sigma)^{\rm in}}_{C^{4,\alpha}(\T)}
\end{equation}
for some absolute constant $c$. We will take $M=3c\norm{(\X^\sigma)^{\rm in}}_{C^{4,\alpha}(\T)}$.

Next, using \eqref{eq:lambda_inv_bds}, Lemma \ref{lem:semigroup} and Corollary \ref{cor:four_derivs}, we may estimate the smooth remainder term involving $\wt{\mc{R}}_+$ as 
\begin{align*}
&\norm{\int_0^t e^{-\int_{t'}^t\lambda^{-1}(\tau)d\tau\;\overline{\mc{L}}_\epsilon\p_\sigma^4}\,\lambda^{-1}\wt{\mc{R}}_+[\p_\sigma^4(\Y+(\X^\sigma)^{\rm in})]\, dt'}_{C^{4,\alpha}(\T)} \\
&\qquad\le \int_0^t c\,(t-t')^{-(1-\frac{\gamma-\alpha}{3})}\norm{\wt{\mc{R}}_+[\p_\sigma^4(\Y+(\X^\sigma)^{\rm in})]}_{C^{1,\gamma}(\T)} \,dt'\\
&\qquad \le \int_0^t (t-t')^{-(1-\frac{\gamma-\alpha}{3})} c(\epsilon,\norm{\Y+(\X^\sigma)^{\rm in}}_{C^{3,\gamma}},\abs{(\X^\sigma)^{\rm in}}_\star)\big(\norm{\Y}_{C^{4,\alpha}(\T)} +\norm{(\X^\sigma)^{\rm in}}_{C^{4,\alpha}(\T)}\big)\,dt' \\
&\qquad \le t^{\frac{\gamma-\alpha}{3}}\, c(\epsilon,\norm{\Y}_{\mc{Y}_4},\norm{(\X^\sigma)^{\rm in}}_{\mc{Y}_4},\abs{(\X^\sigma)^{\rm in}}_\star)\big(\norm{\Y}_{\mc{Y}_4} +\norm{(\X^\sigma)^{\rm in}}_{\mc{Y}_4}\big)\\ 
&\qquad\le t^{\frac{\gamma-\alpha}{3}}\, c(\epsilon,M,\abs{(\X^\sigma)^{\rm in}}_\star)\,M\,.
\end{align*}
In addition, we may estimate the smooth remainder term in $\mc{Y}_1$ using
\begin{align*}
&\norm{\int_0^t e^{-\int_{t'}^t\lambda^{-1}(\tau)d\tau\;\overline{\mc{L}}_\epsilon\p_\sigma^4}\,\lambda^{-1}\wt{\mc{R}}_+[\p_\sigma^4(\Y+(\X^\sigma)^{\rm in})]\, dt'}_{C^{1,\alpha}(\T)} 
\le c\int_0^t \norm{\wt{\mc{R}}_+[\p_\sigma^4(\Y+(\X^\sigma)^{\rm in})]}_{C^{1,\alpha}(\T)} \,dt'\\
&\qquad \le t\, c(\epsilon,\norm{\Y}_{\mc{Y}_4},\norm{(\X^\sigma)^{\rm in}}_{\mc{Y}_4},\abs{(\X^\sigma)^{\rm in}}_\star)\big(\norm{\Y}_{\mc{Y}_4} +\norm{(\X^\sigma)^{\rm in}}_{\mc{Y}_4}\big) 
\le t\, c(\epsilon,M,\abs{(\X^\sigma)^{\rm in}}_\star)\,M\,.
\end{align*}

Furthermore, for the small-in-$\epsilon$ remainder term, we may use \eqref{eq:lambda_inv_bds} and Lemma \ref{lem:max_reg} along with Corollary \ref{cor:four_derivs} and Lemma \ref{lem:straight_Leps_Holder} to estimate 
\begin{align*}
&\norm{\int_0^t e^{-\int_{t'}^t\lambda^{-1}(\tau)d\tau\;\overline{\mc{L}}_\epsilon\p_\sigma^4}\,\lambda^{-1}\mc{R}_\epsilon[\overline{\mc{L}}_\epsilon\p_\sigma^4(\Y+(\X^\sigma)^{\rm in})]\, dt'}_{\mc{Y}_4} \\
&\qquad\le \big(\epsilon+T^{1/4}\abs{\log\epsilon}^{-3/4}+T \big)\,c\norm{\mc{R}_\epsilon[\overline{\mc{L}}_\epsilon\p_\sigma^4(\Y+(\X^\sigma)^{\rm in})]}_{\mc{Y}_1}\\
&\qquad\le \epsilon^{1-\alpha^+}\abs{\log\epsilon}\,\big(\epsilon+T^{1/4}\abs{\log\epsilon}^{-3/4}+T \big)\,c(\norm{\Y+(\X^\sigma)^{\rm in}}_{\mc{Y}_4},\abs{(\X^\sigma)^{\rm in}}_\star)\norm{\overline{\mc{L}}_\epsilon\p_\sigma^4(\Y+(\X^\sigma)^{\rm in})}_{\mc{Y}_1}\\
&\qquad\le \epsilon^{1-\alpha^+}\abs{\log\epsilon}\,\big(\epsilon+T^{1/4}\abs{\log\epsilon}^{-3/4}+T \big)\,c(M,\abs{(\X^\sigma)^{\rm in}}_\star)\; \big(\epsilon^{-1}\abs{\log\epsilon}\,M \big)\\
&\qquad \le \epsilon^{1-\alpha^+}\abs{\log\epsilon}^2\,c(M,\abs{(\X^\sigma)^{\rm in}}_\star)\,M
+\big(T^{1/4}+T \big)\,c(\epsilon,M,\abs{(\X^\sigma)^{\rm in}}_\star)\,M \,.
\end{align*}
In addition, by \eqref{eq:lambda_inv_bds} and Lemma \ref{lem:semigroup}, we have the $C^{1,\alpha}(\T)$ bound
\begin{align*}
&\norm{\int_0^t e^{-\int_{t'}^t\lambda^{-1}(\tau)d\tau\;\overline{\mc{L}}_\epsilon\p_\sigma^4}\,\lambda^{-1}\mc{R}_\epsilon[\overline{\mc{L}}_\epsilon\p_\sigma^4(\Y+(\X^\sigma)^{\rm in})]\, dt'}_{C^{1,\alpha}(\T)} \\
&\qquad\le c\int_0^t \norm{\mc{R}_\epsilon[\overline{\mc{L}}_\epsilon\p_\sigma^4(\Y+(\X^\sigma)^{\rm in})]}_{C^{1,\alpha}(\T)} \,dt'\\
&\qquad \le  t\, c(\epsilon,\norm{\Y}_{\mc{Y}_4},\norm{(\X^\sigma)^{\rm in}}_{\mc{Y}_4},\abs{(\X^\sigma)^{\rm in}}_\star)\,\norm{\overline{\mc{L}}_\epsilon\p_\sigma^4(\Y+(\X^\sigma)^{\rm in})}_{\mc{Y}_1}
\le t\, c(\epsilon,M,\abs{(\X^\sigma)^{\rm in}}_\star)\,M\,.
\end{align*}

Finally, to estimate the $\mc{Y}_1$ contribution from the initial condition, we consider the equation 
\begin{align*}
  \frac{\p\bm{Z}}{\p t}(\sigma, t) = -\lambda(t)^{-1}\overline{\mc{L}}_\epsilon\p_\sigma^4\bm{Z} \,, \qquad 
  \bm{Z}^{\rm in}=(\X^\sigma)^{\rm in}(\sigma)\,.
\end{align*}
Using the mapping properties of $\overline{\mc{L}}_\epsilon$ (Lemma \ref{lem:straight_Leps_Holder}), at each $t$ we have
\begin{align*}
  \norm{\p_t\bm{Z}}_{C^{1,\alpha}(\T)} \le c(\epsilon)\norm{\bm{Z}}_{C^{4,\alpha}(\T)} 
  \le c(\epsilon)\norm{(\X^\sigma)^{\rm in}}_{C^{4,\alpha}(\T)}\,.
\end{align*}
Thus 
\begin{align*}
  \norm{\bm{Z}-(\X^\sigma)^{\rm in}}_{C^{1,\alpha}(\T)} = 
  \norm{\int_0^t\p_t\bm{Z}\,dt'}_{C^{1,\alpha}(\T)} 
  \le c(\epsilon)\,t\,\norm{(\X^\sigma)^{\rm in}}_{C^{4,\alpha}(\T)}\,;
\end{align*}
in particular, we have
\begin{equation}
  \norm{e^{-\int_{0}^t\lambda^{-1}(\tau)d\tau\;\overline{\mc{L}}_\epsilon\p_\sigma^4}(\X^\sigma)^{\rm in}-(\X^\sigma)^{\rm in}}_{\mc{Y}_1} 
  \le c(\epsilon)\,T\norm{(\X^\sigma)^{\rm in}}_{C^{4,\alpha}(\T)} 
  \le c(\epsilon)\,T\,M\,.
\end{equation}

Altogether, we have
\begin{align*}
\norm{\Lambda[\Y]}_{\mc{Y}_4}&\le \bigg(\frac{1}{3}+ \big(T^{\frac{\gamma-\alpha}{3}}+T^{1/4}+T\big)\, c(\epsilon,M,\abs{(\X^\sigma)^{\rm in}}_\star) +  \epsilon^{1-\alpha^+}\abs{\log\epsilon}^2\,c(M,\abs{(\X^\sigma)^{\rm in}}_\star)\bigg) M \,,\\
\norm{\Lambda[\Y]}_{\mc{Y}_1}&\le T\, c(\epsilon,M,\abs{(\X^\sigma)^{\rm in}}_\star) \, M\,.
\end{align*}
We first take $\epsilon$ small enough that 
\begin{equation}\label{eq:eps_small1}
\epsilon^{1-\alpha^+}\abs{\log\epsilon}^2\,c(M,\abs{(\X^\sigma)^{\rm in}}_\star)\le \frac{1}{3}
\quad \text{and} \quad
\epsilon\le \min\bigg\{\frac{1}{2}\textstyle r_*(\frac{|\X^{\rm in}|_\star}{2},M)\,,\epsilon_{\rm n}(\frac{|\X^{\rm in}|_\star}{2},M)\bigg\}\,,
\end{equation}
where $r_*$ is as in \eqref{eq:rstar} and $\epsilon_{\rm n}$ is as in Theorem \ref{thm:decomp}.
Then, taking $T$ small enough that 
\begin{equation}\label{eq:Tsmall1}
\begin{aligned}
 &(T^{\frac{\gamma-\alpha}{3}}+T^{1/4}+T)\, c(\epsilon,M,\abs{(\X^\sigma)^{\rm in}}_\star)\le \frac{1}{3} 
 \quad \text{and} \\
 &T\, c(\epsilon,M,\abs{(\X^\sigma)^{\rm in}}_\star)\le \min\bigg\{\frac{1}{2}\abs{(\X^\sigma)^{\rm in}}_\star,\frac{1}{2}\bigg\}\,,
\end{aligned}
\end{equation}
we obtain 
\begin{align*}
\norm{\Lambda[\Y]}_{\mc{Y}_4}\le M\,, \qquad \norm{\Lambda[\Y]}_{\mc{Y}_1}\le\min\bigg\{\frac{1}{2}\abs{(\X^\sigma)^{\rm in}}_\star,\frac{1}{2}\bigg\}\,.
\end{align*}

We next show that the map $\Lambda[\Y]$ admits a unique fixed point on $\mc{B}_{M,\frac{1}{2}}(T)$.
We now consider two nearby curves $\X^{(a)}$ and $\X^{(b)}$, and define $\Y^{(a)}=\X^{(a)}-(\X^\sigma)^{\rm in}$, $\Y^{(b)}=\X^{(b)}-(\X^\sigma)^{\rm in}$.
Note that the lengths $\lambda^{(a)}$ and $\lambda^{(b)}$ of curve (a) and (b), respectively, satisfy 
\begin{align*}
\lambda^{(a)}-\lambda^{(b)}&=\int_0^1\bigg(\abs{\p_\sigma\Y^{(a)}+\p_\sigma(\X^\sigma)^{\rm in}}-\abs{\p_\sigma\Y^{(b)}+\p_\sigma(\X^\sigma)^{\rm in}}\bigg)\,d\sigma \\
&= \int_0^1 \frac{\p_\sigma(\Y^{(a)}-\Y^{(b)})\cdot(\p_\sigma\Y^{(a)}+\p_\sigma\Y^{(b)}+2\p_\sigma(\X^\sigma)^{\rm in})}{|\p_\sigma\Y^{(a)}+\p_\sigma(\X^\sigma)^{\rm in}|+|\p_\sigma\Y^{(b)}+\p_\sigma(\X^\sigma)^{\rm in}|}\,d\sigma\,;
\end{align*}
in particular, in $\mc{B}_{M,\frac{1}{2}}(T)$, we have
\begin{equation}\label{eq:lambda_lip}
\begin{aligned}
\norm{\Y^{(a)}-\Y^{(b)}}_{C^{1,\alpha}(\T)} &\le\abs{\lambda^{(a)}-\lambda^{(b)}} \le 3\norm{\Y^{(a)}-\Y^{(b)}}_{C^{1,\alpha}(\T)}\,, \\
\frac{4}{9}\norm{\Y^{(a)}-\Y^{(b)}}_{C^{1,\alpha}(\T)} &\le \abs{\frac{1}{\lambda^{(a)}}-\frac{1}{\lambda^{(b)}}} \le 12 \norm{\Y^{(a)}-\Y^{(b)}}_{C^{1,\alpha}(\T)}\,.
\end{aligned}
\end{equation}
Let
\begin{align*}
\bm{W}_+^{(j)} &= (\lambda^{(j)})^{-1}\wt{\mc{R}}_+^{(j)}[\p_\sigma^4(\Y^{(j)}+(\X^\sigma)^{\rm in})] \\
\bm{W}_\epsilon^{(j)} &= (\lambda^{(j)})^{-1}\mc{R}_\epsilon^{(j)}[\overline{\mc{L}}_\epsilon\p_\sigma^4(\Y^{(j)}+(\X^\sigma)^{\rm in})]\,, \qquad j=a,b\,.
\end{align*}
Using \eqref{eq:lambda_lip} and \eqref{eq:new_NtD_ests_lip}, we may estimate 
\begin{equation}\label{eq:BW_plus_lip}
\begin{aligned}
\norm{\bm{W}_+^{(a)} -\bm{W}_+^{(b)} }_{C^{1,\gamma}}
&\le c(\epsilon,M,\abs{(\X^\sigma)^{\rm in}}_\star)\bigg(\norm{\Y^{(a)}-\Y^{(b)}}_{C^{1,\alpha}}
+ \norm{\Y^{(a)}-\Y^{(b)}}_{C^{2,\beta}}\norm{\Y^{(a)}+(\X^\sigma)^{\rm in}}_{C^{0,\alpha}}\\
&\qquad + \norm{\Y^{(a)}-\Y^{(b)}}_{C^{4,\alpha}}\bigg)
\le c(\epsilon,M,\abs{(\X^\sigma)^{\rm in}}_\star) \norm{\Y^{(a)}-\Y^{(b)}}_{\mc{Y}_4}\,.
\end{aligned}
\end{equation}
Furthermore, using \eqref{eq:lambda_lip}, \eqref{eq:new_NtD_ests_lip}, and Lemma \ref{lem:straight_Leps_Holder}, we have 
\begin{equation}\label{eq:BW_eps_lip}
\begin{aligned}
\norm{\bm{W}_\epsilon^{(a)} -\bm{W}_\epsilon^{(b)}}_{C^{1,\alpha}}
&\le \epsilon^{1-\alpha^+}\abs{\log\epsilon}\,c(M,\abs{(\X^\sigma)^{\rm in}}_\star)\bigg(
\norm{\Y^{(a)}-\Y^{(b)}}_{C^{2,\beta}}\norm{\overline{\mc{L}}_\epsilon\p_\sigma^4(\Y^{(a)}+(\X^\sigma)^{\rm in})}_{C^{1,\alpha}} \\
&\qquad + \norm{\overline{\mc{L}}_\epsilon\p_\sigma^4(\Y^{(a)}-\Y^{(b)})}_{C^{1,\alpha}}\bigg) \\
&\le \epsilon^{-\alpha^+}\abs{\log\epsilon}^2\,c(M,\abs{(\X^\sigma)^{\rm in}}_\star)\norm{\Y^{(a)}-\Y^{(b)}}_{\mc{Y}_4}\,.
\end{aligned}
\end{equation}

Now, integrating \eqref{eq:lambda_lip} in time, we have 
\begin{equation}\label{eq:lambda_time_lip}
\frac{4}{9}(t-t')\norm{\Y^{(a)}-\Y^{(b)}}_{\mc{Y}_1}
\le \int_{t'}^t\bigg(\frac{1}{\lambda^{(a)}(\tau)}-\frac{1}{\lambda^{(b)}(\tau)}\bigg)\,d\tau 
\le 12(t-t')\norm{\Y^{(a)}-\Y^{(b)}}_{\mc{Y}_1}\,.
\end{equation}
In particular, combining \eqref{eq:lambda_time_lip}, \eqref{eq:BW_plus_lip}, and Lemma \ref{lem:semigroup}, we may estimate
\begin{align*}
&\norm{\int_0^t \bigg(e^{-\int_{t'}^t(\lambda^{(a)})^{-1}(\tau)d\tau\;\overline{\mc{L}}_\epsilon\p_\sigma^4}\bm{W}_+^{(a)}- e^{-\int_{t'}^t(\lambda^{(b)})^{-1}(\tau)d\tau\;\overline{\mc{L}}_\epsilon\p_\sigma^4}\bm{W}_+^{(b)}\bigg)\, dt'}_{C^{4,\alpha}(\T)} \\
&\quad\le \int_0^t c\,(t-t')^{-(1-\frac{\gamma-\alpha}{3})}\bigg(\norm{\Y^{(a)}-\Y^{(b)}}_{\mc{Y}_1}\norm{\bm{W}_+^{(a)}}_{C^{1,\gamma}(\T)} + \norm{\bm{W}_+^{(a)}-\bm{W}_+^{(b)}}_{C^{1,\gamma}(\T)}\bigg) \,dt'\\
&\quad \le t^{\frac{\gamma-\alpha}{3}}\,c(\epsilon,M,\abs{(\X^\sigma)^{\rm in}}_\star)\,\norm{\Y^{(a)}-\Y^{(b)}}_{\mc{Y}_4}\,.
\end{align*}

Furthermore, combining \eqref{eq:lambda_time_lip}, \eqref{eq:BW_eps_lip}, and Lemma \ref{lem:max_reg}, we have
\begin{align*}
&\norm{\int_0^t \bigg(e^{-\int_{t'}^t(\lambda^{(a)})^{-1}(\tau)d\tau\;\overline{\mc{L}}_\epsilon\p_\sigma^4}\bm{W}_\epsilon^{(a)}- e^{-\int_{t'}^t(\lambda^{(b)})^{-1}(\tau)d\tau\;\overline{\mc{L}}_\epsilon\p_\sigma^4}\bm{W}_\epsilon^{(b)}\bigg)\, dt'}_{\mc{Y}_4} \\
&\quad\le \big(\epsilon + T^{1/4}\abs{\log\epsilon}^{-3/4}+T \big)\,c\bigg(\norm{\Y^{(a)}-\Y^{(b)}}_{\mc{Y}_1}\norm{\bm{W}_\epsilon^{(a)}}_{\mc{Y}_1} + \norm{\bm{W}_\epsilon^{(a)}-\bm{W}_\epsilon^{(b)}}_{\mc{Y}_1}\bigg)\\
&\quad \le \epsilon^{-\alpha^+}\abs{\log\epsilon}^2\big(\epsilon + T^{1/4}\abs{\log\epsilon}^{-3/4}+T \big)\,c(M,\abs{(\X^\sigma)^{\rm in}}_\star)\,\norm{\Y^{(a)}-\Y^{(b)}}_{\mc{Y}_4} \,.
\end{align*}
Altogether, we obtain 
\begin{align*}
\norm{\Lambda[\Y^{(a)}]-\Lambda[\Y^{(b)}]}_{\mc{Y}_4}
&\le \bigg(\big(T^{\frac{\gamma-\alpha}{3}}+T^{1/4}+T\big)\,c(\epsilon,M,\abs{(\X^\sigma)^{\rm in}}_\star) \\
&\qquad +\epsilon^{1-\alpha^+}\abs{\log\epsilon}^2\,c(M,\abs{(\X^\sigma)^{\rm in}}_\star) \bigg)\norm{\Y^{(a)}-\Y^{(b)}}_{\mc{Y}_4}\,.
\end{align*}
Taking $\epsilon$ small enough that $\epsilon^{1-\alpha^+}\abs{\log\epsilon}^2\,c(M,\abs{(\X^\sigma)^{\rm in}}_\star)\le \frac{1}{3}$ in addition to \eqref{eq:eps_small1}, and then taking $T$ small enough that $\big(T^{\frac{\gamma-\alpha}{3}}+T^{1/4}+T\big)\,c(\epsilon,M,\abs{(\X^\sigma)^{\rm in}}_\star)\le\frac{1}{3}$ in addition to \eqref{eq:Tsmall1}, we have that the map $\Lambda$ admits a unique fixed point in $\mc{B}_{M,\frac{1}{2}}(T)$. 
\hfill \qedsymbol


\section{Mapping properties about the straight cylinder}\label{sec:straight_periodic}
This section is devoted to the proofs of Lemmas \ref{lem:straight_Leps_Holder} and \ref{lem:Sinv_mapping} regarding the mapping properties of the slender body NtD and DtN operators $\overline{\mc{L}}_\epsilon$ and $\overline{\mc{L}}_\epsilon^{-1}$ and the angle-averaged single layer operator $\overline{\mc{A}}_\epsilon$ about the straight filament $\mc{C}_\epsilon$.

\subsection{Preliminaries}\label{subsec:besov}
As in the Laplace setting \cite[section 2]{laplace}, we will consider the H\"older spaces $C^{k,\alpha}(\T)$ and $C^{k,\alpha}_s(\Gamma_\epsilon)$ as subsets of $C^{k,\alpha}(\R)$ and $C^{k,\alpha}_s(\R\times 2\pi\T)$ and will rely on a Besov space characterization of $C^{k,\alpha}(\R)$ and $C^{k,\alpha}_s(\R\times 2\pi\T)$. 

Following \cite[Chapter 2.2]{bahouri2011fourier}, we consider a smooth cutoff function $\phi(\xi)$ supported on the annulus $\{\frac{3}{4}<\abs{\xi}<2\}$ and satisfying $\sum_{j\in\Z}\phi(2^{-j}\xi)=1$ for $\xi\neq 0$. We define
\begin{equation}\label{eq:phij_def}
\phi_j(\xi) = \phi(2^{-j}\xi) 
\end{equation} 
and note that 
\begin{align*}
{\rm supp}(\phi_j)\cap {\rm supp}(\phi_k) &= \emptyset\,, \qquad \abs{j-k}\ge 2\,.
\end{align*}
For a function $\bm{g}(s,\theta)$, we will use $\mc{F}[\bm{g}]$ to denote the Fourier transform in $s$ only: 
\begin{align*}
\mc{F}[\bm{g}](\xi,\theta) = \int_{-\infty}^\infty \bm{g}(s,\theta)\,e^{-2\pi i\xi s}\, ds \,.
\end{align*}
We define the Littlewood-Paley projection in $s$, $P_j\bm{g}$, by
\begin{equation}\label{eq:littlewoodp}
\mc{F}[P_j\bm{g}](\xi) = \phi_j(\xi)\mc{F}[\bm{g}](\xi,\theta)\,;
\end{equation}
i.e. the projection of $\bm{g}$ onto frequencies in the $s$-variable which are supported within annulus $j$. In addition, we will use the notation 
\begin{align*}
P_{\le j_*} := \sum_{j\le j_*}P_j\,.
\end{align*}

Given $\nu\in\R$, we define the seminorm 
\begin{align*}
\abs{\bm{g}}_{(\dot B^\nu_{\infty,\infty})_s} &= \sup_{j\in \Z}2^{j\nu}\norm{P_j\bm{g}}_{L^\infty} \,.
\end{align*}
Note that if $\bm{g}$ is a function of $s$ only, then this is just the Besov seminorm $\abs{\cdot}_{\dot B^\nu_{\infty,\infty}(\R)}$.
We define the full norm
\begin{equation}\label{eq:besov}
\norm{\bm{g}}_{(B^\nu_{\infty,\infty})_s} = \sup\bigg(\norm{P_{\le 0}\bm{g}}_{L^\infty}\,,\,\sup_{j>0}2^{j\nu}\norm{P_j\bm{g}}_{L^\infty}\bigg)\,.
\end{equation}
Letting $\lfloor\nu\rfloor$ denote the integer part of $\nu$, it may be shown that the $(B^\nu_{\infty,\infty})_s$ norm \eqref{eq:besov} is equivalent to the $C^{\lfloor\nu\rfloor,\nu-\lfloor\nu\rfloor}_s$ H\"older norm \eqref{eq:CalphaS_norm}. Furthermore, if $\bm{g}$ is a function of $s$ only, then $\norm{\bm{g}}_{(B^\nu_{\infty,\infty})_s}=\norm{\bm{g}}_{B^\nu_{\infty,\infty}(\R)}$ is equivalent to $\norm{\bm{g}}_{C^{\lfloor\nu\rfloor,\nu-\lfloor\nu\rfloor}(\R)}$.

Given a matrix-valued Fourier multiplier $\wh{\bm{M}}(\xi)$, we denote 
\begin{equation}\label{eq:T_m_def}
T_{\wh{\bm{M}}}\bm{g} = \mc{F}^{-1}[\wh{\bm{M}}\,\mc{F}[\bm{g}]]\,, \qquad P_jT_{\wh{\bm{M}}}\bm{g} = \mc{F}^{-1}[\phi_j \wh{\bm{M}}\,\mc{F}[\bm{g}]]\,.
\end{equation}
For $\phi_j$ as in \eqref{eq:phij_def}, we may define $\bm{M}_j=\mc{F}^{-1}[\phi_j\wh{\bm{M}}]$ so that we may write $P_jT_{\wh{\bm{M}}}\bm{g}=\bm{M}_j*\bm{g}$, where the convolution is with respect to $s$ only:
\begin{align*}
 \bm{M}_j*\bm{g} = \int_\R \bm{M}_j(s-s')\,\bm{g}(s',\theta)\,ds'\,.
\end{align*} 
By Young's inequality for convolutions, we may estimate 
\begin{align*}
\norm{P_jT_{\wh{\bm{M}}}\bm{g}}_{L^\infty} = \norm{\bm{M}_j*\bm{g}}_{L^\infty}
\le \norm{\bm{M}_j}_{L^1}\norm{\bm{g}}_{L^\infty}\,.
\end{align*}
Using the definition \eqref{eq:besov} of the $(B^\nu_{\infty,\infty})_s$ norm and its equivalence with $C^{\lfloor\nu\rfloor,\nu-\lfloor\nu\rfloor}_s$, we may thus obtain the mapping properties of an operator $T_{\wh{\bm{M}}}\bm{g}$ by obtaining $L^1$ bounds for the function $\bm{M}_j$ in physical space. One way to prove such bounds is via the following lemma, which is just the matrix-valued version of \cite[Lemma 2.1]{laplace}. 
\begin{lemma}[Physical space $L^1$ bounds for multipliers]\label{lem:besov}
For fixed $j\in \Z$, let $\wh{\bm{M}}_j(\xi)\in C^\infty_0(\R)$ be a matrix-valued Fourier multiplier supported in the annulus $2^{j-1}<\abs{\xi}<2^{j+1}$ and satisfying
\begin{equation}
  \abs{\wh{\bm{M}}_j} \le A\,, \qquad \abs{\p_\xi^2\wh{\bm{M}}_j} \le B
\end{equation} 
for some numbers $A$, $B$. Then $\bm{M}_j = \mc{F}^{-1}[\wh{\bm{M}}_j]$ satisfies the $L^1$ bound
\begin{equation}
\norm{\bm{M}_j}_{L^1(\R)} \lesssim 2^j \sqrt{AB}\,.
\end{equation}
\end{lemma}
The strategy for showing the mapping properties in Lemmas \ref{lem:straight_Leps_Holder} and \ref{lem:Sinv_mapping} will thus be to prove $L^\infty$ bounds for the zeroth and second derivatives of the components of the Fourier multipliers corresponding to $\overline{\mc{L}}_\epsilon$, $\overline{\mc{L}}_\epsilon^{-1}$, and $\overline{\mc{A}}_\epsilon$.

\subsection{Bounds for SB DtN and NtD multipliers}
We begin by considering the tangential and normal direction multipliers $m_{\epsilon,{\rm t}}^{-1}$ and $m_{\epsilon,{\rm n}}^{-1}$ corresponding to $\overline{\mc{L}}_\epsilon^{-1}$, given by Proposition \ref{prop:Leps_spectrum}. Given the form of these multipliers, it will be essential to have bounds for ratios of second-kind modified Bessel functions $K_0$ and $K_1$ as well as their derivatives. We will make use of the following proposition. 

\begin{proposition}\label{prop:besselK_bounds}
Let $K_0(z)$ and $K_1(z)$ denote zeroth and first order modified Bessel functions of the second kind. The following bounds hold: 
\begin{align}
0&<\frac{K_1(z)}{K_0(z)} - 1 - \frac{1}{2z}+\frac{1}{8z^2} \le \frac{1}{8z^3} \,, \qquad z\ge 1  \label{eq:K1_K0} \\
-\frac{4}{11z^3}&\le \frac{K_0(z)}{K_1(z)}- 1+\frac{1}{2z}-\frac{3}{8z^2}< 0\,, \hspace{1.25cm} z\ge 1  \label{eq:K0_K1} \\
\frac{c}{z\abs{\log(z/3)}}&<\frac{K_1(z)}{K_0(z)}<\frac{c}{z\abs{\log(z/2)}}\,, \hspace{2.25cm} 0<z<1  \label{eq:K1_K0_smallz} \\
0&\le \frac{K_1}{K_0}-\frac{K_0}{K_1} \le \frac{1}{z}\,, \hspace{3.3cm} \text{all }z>0\,.  \label{eq:K12_K02_est}
\end{align}
\end{proposition}
A proof of the quantitative bounds \eqref{eq:K1_K0} and \eqref{eq:K0_K1} for $z\ge 1$ appears below. These quantitative bounds will be useful for bounding the complicated functions \eqref{eq:eigsT}, \eqref{eq:eigsN} of the ratios $\frac{K_1}{K_0}$ and $\frac{K_0}{K_1}$ at intermediate $z\sim 1$. For small $0<z<1$, such quantitative bounds are not necessary, and \eqref{eq:K1_K0_smallz} follows from the well-known small-$z$ asymptotics of $K_0(z)$ and $K_1(z)$ \cite[Chapter 10]{NIST:DLMF}. The general bound \eqref{eq:K12_K02_est} is due to estimate (77) in \cite{inverse}.

Noting that the first and second derivatives of the ratios in Proposition \ref{prop:besselK_bounds} are given by 
\begin{align*}
\bigg(\frac{K_1}{K_0}\bigg)' &= \frac{K_1^2}{K_0^2}-\frac{K_1}{zK_0}-1\,, \qquad
\bigg(\frac{K_1}{K_0}\bigg)'' =2\frac{K_1}{K_0}\bigg(\frac{K_1}{K_0}\bigg)'+\frac{1}{z^2}\frac{K_1}{K_0}-\frac{1}{z}\bigg(\frac{K_1}{K_0}\bigg)' \\
\bigg(\frac{K_0}{K_1}\bigg)' &= \frac{K_0^2}{K_1^2}+\frac{K_0}{zK_1}-1\,, \qquad
\bigg(\frac{K_0}{K_1}\bigg)'' =2\frac{K_0}{K_1}\bigg(\frac{K_0}{K_1}\bigg)'-\frac{1}{z^2}\frac{K_0}{K_1}+\frac{1}{z}\bigg(\frac{K_0}{K_1}\bigg)'\,,
\end{align*}
as an immediate consequence of Proposition \ref{prop:besselK_bounds}, we obtain the following corollary.
\begin{corollary}\label{cor:Kderivs}
For $z\ge 1$, the following bounds hold for derivatives of ratios of $K_1(z)$ and $K_0(z)$:
\begin{equation}\label{eq:Kderivs_bigZ}
\begin{aligned}
\abs{\bigg(\frac{K_1}{K_0}\bigg)'+\frac{1}{2z^2}} 
&\le \frac{3}{4z^3}\,, \qquad
\abs{\bigg(\frac{K_1}{K_0}\bigg)''} 
\le \frac{9}{2z^3}\,, \\
\abs{\bigg(\frac{K_0}{K_1}\bigg)'-\frac{1}{2z^2}} 
&\le \frac{2}{z^3}\,, \qquad
\abs{\bigg(\frac{K_0}{K_1}\bigg)''} 
\le \frac{25}{2z^3}\,.
\end{aligned}
\end{equation}

For $0<z<1$, we have the bounds
\begin{equation}\label{eq:Kderivs_smallZ}
\begin{aligned}
\abs{\bigg(\frac{K_1}{K_0}\bigg)'} &\le \frac{c}{z^2\abs{\log(z/2)}}\,, \qquad
\abs{\bigg(\frac{K_1}{K_0}\bigg)''} \le \frac{c}{z^3\abs{\log(z/2)}}\,, \\
\abs{\bigg(\frac{K_0}{K_1}\bigg)'} &\le c\abs{\log(z/3)}\,, \qquad
\abs{\bigg(\frac{K_0}{K_1}\bigg)''} \le \frac{c}{z}\abs{\log(z/3)}\,.
\end{aligned}
\end{equation}
\end{corollary}

\begin{proof}[Proof of Proposition \ref{prop:besselK_bounds}, estimates \eqref{eq:K1_K0} and \eqref{eq:K0_K1}]

We first outline a general strategy for obtaining bounds of this form.
Consider a function $B(z)$ satisfying the ODE
\begin{equation}
B'(z) = F(z,B(z))\,, \qquad B(z_0) = B_0
\end{equation}
for some $F$. To obtain an upper bound on $B$, we seek $b(z)$ satisfying 
\begin{equation}\label{eq:B_upper}
\begin{cases}
b'(z) < F(z,b(z)) \quad \text{ for } z\ge z_0\ge 0\\
\exists z_1>z_0 \text{ such that } b(z)>B(z) \text{ for } z\ge z_1\,.
\end{cases}
\end{equation}
Then $b(z)>B(z)$ for all $z>z_0$.
Similarly, to obtain a lower bound for $B$, we seek $h(z)$ satisfying
\begin{equation}\label{eq:B_lower}
\begin{cases}
h'(z) > F(z,h(z)) \quad \text{ for } z\ge z_0\ge 0\\
\exists z_1>z_0 \text{ such that } h(z)<B(z) \text{ for } z\ge z_1\,.
\end{cases}
\end{equation}
Then $h(z)<B(z)$ for all $z>z_0$. \\

Consider the function $B_K(z) = z\frac{K_1(z)}{K_0(z)}-z-\frac{1}{2}+\frac{1}{8z}$. It can be seen that $B_K$ satisfies the ODE
\begin{equation}\label{eq:BK_ODE}
B_K'(z) = \frac{1}{z}\bigg(B_K^2+B_K-\frac{1}{4z}B_K+\frac{1}{64z^2}-\frac{1}{4z} \bigg) + 2B_K\,, \qquad B_K(1) = \frac{K_1(1)}{K_0(1)}-\frac{11}{8}\approx 0.055\,.
\end{equation}
Let $b(z)=\frac{1}{8z^2}$. We have that 
\begin{align*}
b'(z) - \frac{1}{z}\bigg(b^2+b-\frac{1}{4z}b+\frac{1}{64z^2}-\frac{1}{4z} \bigg) - 2b 
= -\frac{1}{64z^5}+\frac{1}{32z^4}-\frac{25}{64z^3}  <0
\end{align*}
for $z\ge 1$. Using the large-$z$ asymptotics of $K_0(z)$ and $K_1(z)$ \cite{NIST:DLMF}, we have 
\begin{equation}\label{eq:largeZ}
z\frac{K_1(z)}{K_0(z)} = z + \frac{1}{2}-\frac{1}{8z}+\frac{1}{8z^2} -\frac{25}{128z^3} + O\bigg(\frac{1}{z^4}\bigg)\,, \quad z\to \infty\,.
\end{equation}
In particular, for sufficiently large $z$, we have $B_K(z)<b(z)$. Thus by \eqref{eq:B_upper}, $B_K(z)<b(z)$ for all $z\ge 1$.

Now let $h(z)=0$. We have
\begin{align*}
h'(z) - \frac{1}{z}\bigg(h^2+h-\frac{1}{4z}h+\frac{1}{64z^2}-\frac{1}{4z} \bigg) - 2h 
= \frac{1}{4z^2} -\frac{1}{64z^3} >0 
 \end{align*} 
for $z\ge 1$. Again by the large-$z$ asymptotics \eqref{eq:largeZ} for $B_K$, we have that for sufficiently large $z$, $B_K>0$. By \eqref{eq:B_lower}, we thus have $B_K>0$ for all $z\ge 1$.

Note that using \eqref{eq:K1_K0} we immediately obtain the bound \eqref{eq:K0_K1}. In particular, we have
\begin{align*}
1-\frac{1}{2z}+\frac{3}{8z^2}-\frac{4}{11z^3}\le \frac{1}{1+\frac{1}{2z}-\frac{1}{8z^2}+\frac{1}{8z^3}} < \frac{K_0}{K_1}\le \frac{1}{1+\frac{1}{2z}-\frac{1}{8z^2}} \le 1-\frac{1}{2z}+\frac{3}{8z^2}\,.
\end{align*}
\end{proof}


We may use the bounds of Proposition \ref{prop:besselK_bounds} and Corollary \ref{cor:Kderivs} to obtain the following estimates for the tangential direction multiplier $m_{\epsilon,{\rm t}}^{-1}$ given by Proposition \ref{prop:Leps_spectrum}.

\begin{lemma}[Bounds for $m_{\epsilon,{\rm t}}^{-1}$ and $m_{\epsilon,{\rm t}}$]\label{lem:mt_bounds}
Let $m_{\epsilon,{\rm t}}^{-1}(\xi)$ be as in Proposition \ref{prop:Leps_spectrum}. For $\ell=0,1,2$, the following bounds hold: 
\begin{equation}\label{eq:mtINV_bds}
\begin{aligned}
\abs{\p_{\xi}^\ell m_{\epsilon,{\rm t}}^{-1}(\xi)} &\le 
\begin{cases}
c\,\epsilon\abs{\xi}^{1-\ell}\, , & \abs{\xi}\ge\frac{1}{2\pi\epsilon} \,, \\
c\,\abs{\log\epsilon}^{-1}\abs{\xi}^{-\ell}\,, & \abs{\xi}<\frac{1}{2\pi\epsilon} \,.
\end{cases}
\end{aligned}
\end{equation}
Furthermore, we have that $m_{\epsilon,{\rm t}}(\xi)=\frac{1}{m_{\epsilon,{\rm t}}^{-1}(\xi)}$ satisfies 
\begin{equation}\label{eq:mt_bds}
\begin{aligned}
\abs{\p_{\xi}^\ell m_{\epsilon,{\rm t}}(\xi)} &\le 
\begin{cases}
c\,\epsilon^{-1} \abs{\xi}^{-\ell-1}\,, & \abs{\xi}\ge\frac{1}{2\pi\epsilon} \,, \\
c\,\abs{\log\epsilon}\abs{\xi}^{-\ell} \,, & \abs{\xi}<\frac{1}{2\pi\epsilon}\,.
\end{cases}
\end{aligned}
\end{equation}
\end{lemma}



\begin{proof}[Proof of Lemma \ref{lem:mt_bounds}]
Consider the function 
\begin{equation}\label{eq:mtINV_def}
\begin{aligned}
m_{\rm t}^{-1}(z) &= \frac{m_{\rm t1}(z)}{m_{\rm t2}(z)}\,; \\ 
m_{\rm t1}(z) := z\frac{K_1(z)}{K_0(z)}\,, \quad 
&m_{\rm t2}(z) := 1+\frac{z}{2}\bigg(\frac{K_0(z)}{K_1(z)}-\frac{K_1(z)}{K_0(z)}\bigg)
\end{aligned}
\end{equation}
and note that $m_{\epsilon,{\rm t}}^{-1}(\xi)=2\pi \, m_{\rm t}^{-1}(2\pi\epsilon\abs{\xi})$ for $m_{\epsilon,{\rm t}}^{-1}(\xi)$ as in Proposition \ref{prop:Leps_spectrum}.
We will additionally consider 
\begin{equation}\label{eq:mt_def}
m_{\rm t}(z) = \frac{1}{m_{\rm t}^{-1}(z)} = \frac{m_{\rm t2}(z)}{m_{\rm t1}(z)}\,.
\end{equation}
The first and second derivatives of $m_{\rm t}^{-1}(z)$ and $m_{\rm t}(z)$ satisfy
\begin{equation}\label{eq:mt_derivs}
\begin{aligned}
(m_{\rm t}^{-1})'(z) &= \frac{m_{\rm t1}'m_{\rm t2}-m_{\rm t1}m_{\rm t2}'}{m_{\rm t2}^2}\,, \qquad  
(m_{\rm t}^{-1})''(z) =\frac{m_{\rm t2}m_{\rm t1}''-m_{\rm t1}m_{\rm t2}''}{m_{\rm t2}^2} + 2\frac{m_{\rm t2}'}{m_{\rm t2}}(m_{\rm t}^{-1})'\\
m_{\rm t}'(z) &= \frac{m_{\rm t2}'m_{\rm t1}-m_{\rm t2}m_{\rm t1}'}{m_{\rm t1}^2}\,, \qquad \qquad  
m_{\rm t}''(z) =\frac{m_{\rm t1}m_{\rm t2}''-m_{\rm t2}m_{\rm t1}''}{m_{\rm t1}^2} + 2\frac{m_{\rm t1}'}{m_{\rm t1}}\,m_{\rm t}'\,.
\end{aligned}
\end{equation}
From Proposition \ref{prop:besselK_bounds}, for $z\ge1$ we have
\begin{equation}\label{eq:mt_bounds0}
\abs{m_{\rm t1}-z-\frac{1}{2}+\frac{1}{8z}} \le \frac{1}{8z^2}\,, \qquad
\abs{m_{\rm t2}-\frac{1}{2}-\frac{1}{8z}} \le \frac{5}{16z^2}\,.
\end{equation}
Furthermore, using Corollary \ref{cor:Kderivs}, for $z\ge 1$ the first and second derivatives of $m_{\rm tA}$ and $m_{\rm tB}$ satisfy 
\begin{equation}\label{eq:mt_bounds}
\begin{aligned}
\abs{m_{\rm t1}' - 1} &= \abs{\frac{K_1}{K_0}+z\bigg(\frac{K_1}{K_0}\bigg)'-1} \le \frac{1}{z^2} \\
\abs{m_{\rm t2}' } &= \abs{\frac{1}{2}\bigg(\frac{K_0}{K_1}-\frac{K_1}{K_0}\bigg)+\frac{z}{2}\bigg(\frac{K_0}{K_1}-\frac{K_1}{K_0}\bigg)' } \le \frac{7}{4z^2}\\
\abs{m_{\rm t1}''} &= \abs{2\bigg(\frac{K_1}{K_0}\bigg)'+z\bigg(\frac{K_1}{K_0}\bigg)''} \le \frac{7}{z^2} \\
\abs{m_{\rm t2}''} &= \abs{\bigg(\frac{K_0}{K_1}-\frac{K_1}{K_0}\bigg)'+\frac{z}{2}\bigg(\frac{K_0}{K_1}-\frac{K_1}{K_0}\bigg)'' } \le \frac{13}{z^2}\,.
\end{aligned}
\end{equation}
Using \eqref{eq:mt_bounds0} and \eqref{eq:mt_bounds} in the expressions \eqref{eq:mtINV_def}, \eqref{eq:mt_def}, and \eqref{eq:mt_derivs}, we obtain the following bounds for $m_{\rm t}^{-1}$ and $m_{\rm t}$ when $z\ge 1$:
\begin{equation}\label{eq:mt_bigZ}
\begin{aligned}
\abs{m_{\rm t}^{-1}(z)} &\le c\,z \,, \qquad \abs{(m_{\rm t}^{-1})'(z)} \le c \,, \qquad \abs{(m_{\rm t}^{-1})''(z)} \le \frac{c}{z}\,, \\
\abs{m_{\rm t}(z)} &\le \frac{c}{z} \,, \qquad \abs{m_{\rm t}'(z)} \le \frac{c}{z^2}\,, \qquad \abs{m_{\rm t}''(z)} \le \frac{c}{z^3}\,.
\end{aligned}
\end{equation}

For $0<z<1$, we may use the bounds \eqref{eq:K1_K0_smallz} and \eqref{eq:K12_K02_est} to obtain
\begin{equation}\label{eq:mt_bounds_smallZ}
\begin{aligned}
\frac{c}{\abs{\log(z/3)}}\le m_{\rm t1}(z) &\le \frac{c}{\abs{\log(z/2)}}\,, \qquad  
\frac{1}{2}\le m_{\rm t2}(z)  \le 1 \\
\abs{m_{\rm t1}'(z)} &\le \frac{c}{z\abs{\log(z/2)}}\,, \qquad  
\abs{m_{\rm t2}'(z)} \le \frac{c}{z} \\
\abs{m_{\rm t1}''(z)} &\le \frac{c}{z^2\abs{\log(z/2)}}\,, \qquad  
\abs{m_{\rm t2}''(z)} \le \frac{c}{z^2} \,.
\end{aligned}
\end{equation}
Using \eqref{eq:mt_bounds_smallZ} in the expressions \eqref{eq:mtINV_def}, \eqref{eq:mt_def}, and \eqref{eq:mt_derivs}, for $0<z<1$ we obtain the bounds
\begin{equation}\label{eq:mt_smallZ}
\begin{aligned}
\abs{m_{\rm t}^{-1}(z)} &\le \frac{c}{\abs{\log(z/2)}} \,, \qquad \abs{(m_{\rm t}^{-1})'(z)} \le \frac{c}{z\abs{\log(z/2)}} \,, \qquad \abs{(m_{\rm t}^{-1})''(z)} \le \frac{c}{z^2\abs{\log(z/2)}}\,, \\
\abs{m_{\rm t}(z)} &\le c\abs{\log(z/2)} \,, \qquad \abs{m_{\rm t}'(z)} \le \frac{c\abs{\log(z/2)}}{z}\,, \qquad \abs{m_{\rm t}''(z)} \le \frac{c\abs{\log(z/2)}}{z^2}\,.
\end{aligned}
\end{equation}

Combining the $z\ge1$ bounds \eqref{eq:mt_bigZ} with the $0<z<1$ bounds \eqref{eq:mt_smallZ} and using that $m_{\epsilon,{\rm t}}^{-1}(\xi)=2\pi \,m_{\rm t}^{-1}(2\pi\epsilon\,\abs{\xi})$, we obtain Lemma \ref{lem:mt_bounds}.
\end{proof}


We next consider the normal direction multiplier $m_{\epsilon,{\rm n}}^{-1}$ given by Proposition \ref{prop:Leps_spectrum} and use the bounds of Proposition \ref{prop:besselK_bounds} and Corollary \ref{cor:Kderivs} to obtain the following lemma.

\begin{lemma}[Bounds for $m_{\epsilon,{\rm n}}^{-1}$ and $m_{\epsilon,{\rm n}}$]\label{lem:mn_bounds}
Let $m_{\epsilon,{\rm n}}^{-1}(\xi)$ be as in Proposition \ref{prop:Leps_spectrum}. For $\ell=0,1,2$, the following bounds hold: 
\begin{equation}\label{eq:mnINV_bds}
\begin{aligned}
\abs{\p_{\xi}^\ell m_{\epsilon,{\rm n}}^{-1}(\xi)} &\le 
\begin{cases}
c\,\epsilon\abs{\xi}^{1-\ell}\, , & \abs{\xi}\ge\frac{1}{2\pi\epsilon} \,, \\
c\,\abs{\log\epsilon}^{-1}\abs{\xi}^{-\ell}\,, & \abs{\xi}<\frac{1}{2\pi\epsilon} \,.
\end{cases}
\end{aligned}
\end{equation}

Furthermore, we have that $m_{\epsilon,{\rm n}}(\xi)=\frac{1}{m_{\epsilon,{\rm n}}^{-1}(\xi)}$ satisfies
\begin{equation}\label{eq:mn_bds}
\begin{aligned}
\abs{\p_{\xi}^\ell m_{\epsilon,{\rm n}}(\xi)} &\le 
\begin{cases}
c\,\epsilon^{-1} \abs{\xi}^{-\ell-1}\,, & \abs{\xi}\ge\frac{1}{2\pi\epsilon} \,, \\
c\,\abs{\log\epsilon}\abs{\xi}^{-\ell} \,, & \abs{\xi}<\frac{1}{2\pi\epsilon}\,.
\end{cases}
\end{aligned}
\end{equation}
\end{lemma}

\begin{proof}[Proof of Lemma \ref{lem:mn_bounds}]
We consider the function 
\begin{equation}\label{eq:mnINV_def}
\begin{aligned}
m_{\rm n}^{-1}(z) &=\frac{m_{\rm n1}(z)}{m_{\rm n2}(z)}\,, \\
m_{\rm n1}(z) &= 2+ \frac{8}{z}\frac{K_1(z)}{K_0(z)} +z\bigg(\frac{K_1(z)}{K_0(z)}- \frac{K_0(z)}{K_1(z)}\bigg)\,, \\
m_{\rm n2}(z) &= \frac{4}{z^2}+ 2\bigg(1- \frac{K_0^2(z)}{K_1^2(z)}\bigg) + \frac{2}{z}\bigg( \frac{K_1(z)}{K_0(z)}-\frac{K_0(z)}{K_1(z)}\bigg)\,.
\end{aligned}
\end{equation}
Using the Bessel function identity
\begin{equation}\label{eq:K2_ID}
K_2(z) = K_0(z) + \frac{2}{z}K_1(z)\,,
\end{equation}
it can be seen that in fact $m_{\epsilon,{\rm n}}^{-1}(\xi)=2\pi\, m_{\rm n}^{-1}(2\pi\epsilon\abs{\xi})$ where $m_{\epsilon,{\rm n}}^{-1}(\xi)$ is given by Proposition \ref{prop:Leps_spectrum}. In addition, we define 
\begin{equation}\label{eq:mn_def}
m_{\rm n}(z) = \frac{1}{m_{\rm n}^{-1}(z)} = \frac{m_{\rm n2}(z)}{m_{\rm n1}(z)}\,.
\end{equation}
The first and second derivatives of $m_{\rm n}^{-1}$ and $m_{\rm n}$ then satisfy
\begin{equation}\label{eq:mn_derivs}
\begin{aligned}
(m_{\rm n}^{-1})'(z) &= \frac{m_{\rm n1}'m_{\rm n2}-m_{\rm n1}m_{\rm n2}'}{m_{\rm n2}^2}\,, \qquad  
(m_{\rm n}^{-1})''(z) =\frac{m_{\rm n2}m_{\rm n1}''-m_{\rm n1}m_{\rm n2}''}{m_{\rm n2}^2} + 2\frac{m_{\rm n2}'}{m_{\rm n2}}(m_{\rm n}^{-1})'\\
m_{\rm n}'(z) &= \frac{m_{\rm n2}'m_{\rm n1}-m_{\rm n2}m_{\rm n1}'}{m_{\rm n1}^2}\,, \qquad \qquad  
m_{\rm n}''(z) =\frac{m_{\rm n1}m_{\rm n2}''-m_{\rm n2}m_{\rm n1}''}{m_{\rm n1}^2} + 2\frac{m_{\rm n1}'}{m_{\rm n1}}\,m_{\rm n}'\,.
\end{aligned}
\end{equation}
Using Proposition \ref{prop:besselK_bounds} and Corollary \ref{cor:Kderivs}, for $z\ge1$ we may obtain the following bounds for the components $m_{\rm nA}$ and $m_{\rm nB}$:
\begin{equation}\label{eq:mn_bounds}
\begin{aligned}
 \abs{m_{\rm n1}-3-\frac{15}{2z}} &\le  \textstyle\abs{\frac{8}{z}\left(1+\frac{1}{2z}-\frac{1}{8z^2}+\frac{1}{8z^3}\right)-\frac{8}{z}}+\abs{z\left(\frac{1}{z}-\frac{1}{2z^2}+\frac{5}{8z^3}\right) -1+\frac{1}{2z}} \le \displaystyle \frac{7}{z^2}\\
 \abs{m_{\rm n2} - \frac{2}{z}-\frac{4}{z^2}} &\le \textstyle \abs{\left(\frac{1}{z}-\frac{3}{4z^2}\right)\left(2-\frac{1}{2z}+\frac{3}{8z^2} \right)- \frac{2}{z}+\frac{2}{z^2} }+ \frac{1}{z^3}+ \frac{5}{4z^4} \le \displaystyle \frac{3}{z^3}\\
\abs{m_{\rm n1}'} &= \textstyle \abs{-\frac{8}{z^2}\frac{K_1}{K_0}+\frac{8}{z}\left(\frac{K_1}{K_0}\right)'+\left(\frac{K_1}{K_0}-\frac{K_0}{K_1}\right)+z\left(\frac{K_1}{K_0}-\frac{K_0}{K_1}\right)'} \le \displaystyle \frac{32}{z} \\
\abs{m_{\rm n2}'} &= \textstyle \abs{-\frac{8}{z^3}-4\frac{K_0}{K_1}\left(\frac{K_0}{K_1}\right)'-\frac{2}{z^2}\left(\frac{K_1}{K_0}-\frac{K_0}{K_1}\right)+\frac{2}{z}\left(\frac{K_1}{K_0}-\frac{K_0}{K_1}\right)'} \le \displaystyle \frac{47}{z^2} \\
\abs{m_{\rm n1}''} &=  \textstyle \bigg|\frac{16}{z^3}\frac{K_1}{K_0}-\frac{16}{z^2}\left(\frac{K_1}{K_0}\right)' +\frac{8}{z}\left(\frac{K_1}{K_0}\right)''
+ 2\left(\frac{K_1}{K_0}-\frac{K_0}{K_1}\right)'+z\left(\frac{K_1}{K_0}-\frac{K_0}{K_1}\right)''\bigg|
\le \displaystyle \frac{103}{z^2}\\
\abs{m_{\rm n2}''} &=  \textstyle \bigg|\frac{24}{z^4}-4\left(\frac{K_0}{K_1}\right)'^{\;2}-4\frac{K_0}{K_1}\left(\frac{K_0}{K_1}\right)''+\frac{4}{z^3}\left(\frac{K_1}{K_0}-\frac{K_0}{K_1}\right)\\
&\qquad\qquad\qquad \textstyle -\frac{4}{z^2}\left(\frac{K_1}{K_0}-\frac{K_0}{K_1}\right)'+\frac{2}{z}\left(\frac{K_1}{K_0}-\frac{K_0}{K_1}\right)''\bigg| 
\le \displaystyle \frac{222}{z^3}\,.
\end{aligned}
\end{equation}
In particular, using \eqref{eq:mn_bounds} in \eqref{eq:mn_derivs}, for $z\ge 1$ we have 
\begin{equation}\label{eq:mn_bigZ}
\begin{aligned}
\abs{m_{\rm n}^{-1}(z)} &\le c\,z \,, \qquad \abs{(m_{\rm n}^{-1})'(z)} \le c\,,\qquad \abs{(m_{\rm n}^{-1})''(z)} \le \frac{c}{z} \\
\abs{m_{\rm n}(z)} &\le \frac{c}{z}\,, \qquad \abs{m_{\rm n}'(z)} \le \frac{c}{z^2}\,,\qquad \abs{m_{\rm n}''(z)} \le \frac{c}{z^3}\,.
\end{aligned}
\end{equation} 

For $0<z<1$, using Proposition \ref{prop:besselK_bounds}, Corollary \ref{cor:Kderivs}, and the forms of the derivatives in \eqref{eq:mn_bounds}, we may calculate
\begin{equation}\label{eq:mn_bounds_smallZ}
\begin{aligned}
\frac{c}{z^2\abs{\log(z/3)}}\le\abs{m_{\rm n1}} &\le \frac{c}{z^2\abs{\log(z/2)}}\,, \qquad
\frac{4}{z^2} \le \abs{m_{\rm n2}} \le \frac{8}{z^2}\\
\abs{m_{\rm n1}'} 
&\le \frac{c}{z^3\abs{\log(z/2)}} \,, \qquad\qquad
\abs{m_{\rm n2}'} 
\le \frac{c}{z^3} \\
\abs{m_{\rm n1}''} 
&\le \frac{c}{z^4\abs{\log(z/2)}}\,, \qquad \qquad
\abs{m_{\rm n2}''} 
\le \frac{c}{z^4}\,.
\end{aligned}
\end{equation}
Using \eqref{eq:mn_bounds_smallZ} with the forms of $m_{\rm n}^{-1}(z)$, $m_{\rm n}(z)$, and their derivatives \eqref{eq:mn_derivs}, we obtain
\begin{equation}\label{eq:mn_smallZ}
\begin{aligned}
\abs{m_{\rm n}^{-1}(z)} &\le \frac{c}{\abs{\log(z/2)}}\,, \qquad \abs{(m_{\rm n}^{-1})'(z)} \le \frac{c}{z\abs{\log(z/2)}}\,,\qquad \abs{(m_{\rm n}^{-1})''(z)} \le \frac{c}{z^2\abs{\log(z/2)}} \\
\abs{m_{\rm n}(z)} &\le c\abs{\log(z/2)}\,, \qquad \abs{m_{\rm n}'(z)} \le \frac{c\abs{\log(z/2)}}{z}\,,\qquad \abs{m_{\rm n}''(z)} \le \frac{c\abs{\log(z/2)}}{z^2}\,.
\end{aligned}
\end{equation} 

Again combining the $z\ge1$ bounds \eqref{eq:mn_bigZ} with the $0<z<1$ bounds \eqref{eq:mn_smallZ} and using that $m_{\epsilon,{\rm n}}^{-1}(\xi)=2\pi \,m_{\rm n}^{-1}(2\pi\epsilon\,\abs{\xi})$, we obtain Lemma \ref{lem:mn_bounds}.
\end{proof}

\subsection{Bounds for single layer multipliers}
We now turn to the angle-averaged inverse single layer operator $\overline{\mc{A}}_\epsilon$, defined in \eqref{eq:Aeps}, and its corresponding multipliers, given by Corollary \ref{cor:Aeps}. To bound these multipliers, in addition to Proposition \ref{prop:besselK_bounds} and Corollary \ref{cor:Kderivs}, we will need bounds on ratios of first-kind modified Bessel functions $I_0$ and $I_1$.

We first make note of a useful identity that is a consequence of a more general property of the Wronskian between first and second kind modified Bessel functions \cite{NIST:DLMF}. In particular, first- and second-order modified Bessel functions of the first and second kind may be related via the following: 
\begin{equation}\label{eq:I_K_id}
K_1(z)I_0(z) + K_0(z)I_1(z) =\frac{1}{z} \,, \qquad z> 0\,. 
\end{equation}
In addition, we have the following proposition. 
\begin{proposition}\label{prop:besselI_bounds}
Let $I_0(z)$ and $I_1(z)$ denote zeroth and first order modified Bessel functions of the first kind. The following bounds each hold for some $c>0$:
\begin{align}
\abs{\frac{I_1}{I_0}-1+\frac{1}{2z}+\frac{1}{8z^2}}&\le \frac{c}{z^3} \,, \qquad z\ge 1  \label{eq:I1_I0}\\
\abs{\frac{I_0}{I_1}-1-\frac{1}{2z}-\frac{3}{8z^2}}&\le \frac{c}{z^3} \,, \qquad z\ge 1  \label{eq:I0_I1}\\
\abs{\frac{I_1}{I_0}-\frac{z}{2}}&\le c z^3 \,, \qquad 0<z<1  \label{eq:I1_I0smallZ} \\
\abs{\frac{I_0}{I_1}-\frac{2}{z}-\frac{z}{4}}&\le c z^3 \,, \qquad 0<z<1 \,.  \label{eq:I0_I1smallZ} 
\end{align}
\end{proposition}
Each of the bounds in Proposition \ref{prop:besselI_bounds} may be seen immediately using the large and small-$z$ asymptotics of $I_0(z)$ and $I_1(z)$; in particular, from \cite[Chap. 10]{NIST:DLMF}, we have 
\begin{equation}\label{eq:Iasymptotics}
\begin{aligned}
I_0(z) &\sim \frac{e^z}{\sqrt{2\pi z}}\big(1+ \frac{1}{8z}+\frac{9}{128z^2}+\frac{75}{124z^3}+\cdots\big)\,,\\
I_1(z) &\sim \frac{e^z}{\sqrt{2\pi z}}\big(1- \frac{3}{8z}-\frac{15}{128z^2}-\frac{105}{1025z^3}-\cdots\big)\,, \quad z\to \infty\,; \\
I_0(z) &\sim 1+\frac{z^2}{4}+\frac{z^4}{64}+\cdots\,,\quad
I_1(z) \sim \frac{z}{2}+\frac{z^3}{16}+\cdots\,, \quad z\to 0\,.
\end{aligned}
\end{equation}
Note that the bounds \eqref{eq:I1_I0} and \eqref{eq:I0_I1} are stated with an unknown constant $c$, unlike in Proposition \ref{prop:besselK_bounds}. Because of this, we will need an additional refinement in order to prove all of the bounds we will eventually need (see Proposition \ref{prop:besselI_2}). Before stating and proving this refinement, however, we note an immediate corollary of Proposition \ref{prop:besselI_bounds}. 

Calculating the first and second derivatives of the quantities in Proposition \ref{prop:besselI_bounds} yields
\begin{align*}
\bigg(\frac{I_1}{I_0} \bigg)' &= -\frac{I_1^2}{I_0^2}-\frac{I_1}{zI_0}+1\,, \qquad
\bigg(\frac{I_1}{I_0} \bigg)'' = -2\frac{I_1}{I_0}\bigg(\frac{I_1}{I_0} \bigg)' +\frac{1}{z^2}\frac{I_1}{I_0}-\frac{1}{z}\bigg(\frac{I_1}{I_0}\bigg)'\\
\bigg(\frac{I_0}{I_1} \bigg)' &= -\frac{I_0^2}{I_1^2}+\frac{I_0}{zI_1}+1\,, \qquad
\bigg(\frac{I_0}{I_1} \bigg)'' = -2\frac{I_0}{I_1}\bigg(\frac{I_0}{I_1}\bigg)'-\frac{1}{z^2}\frac{I_0}{I_1}+\frac{1}{z}\bigg(\frac{I_0}{I_1}\bigg)'\,,
\end{align*}
from which we may again immediately obtain the following.

\begin{corollary}\label{cor:Iderivs}
For $z\ge 1$, the following bounds hold for derivatives of ratios of $I_1(z)$ and $I_0(z)$:
\begin{equation}
\begin{aligned}
\abs{\bigg(\frac{I_1}{I_0} \bigg)'-\frac{1}{2z^2}} &\le \frac{c}{z^3}\,, \qquad
\abs{\bigg(\frac{I_1}{I_0} \bigg)''} \le \frac{c}{z^3} \\
\abs{\bigg(\frac{I_0}{I_1} \bigg)'+\frac{1}{2z^2}} &\le \frac{c}{z^3}\,, \qquad
\abs{\bigg(\frac{I_0}{I_1} \bigg)''} \le \frac{c}{z^3}\,.
\end{aligned}
\end{equation}
For $0<z<1$, we have the bounds
\begin{equation}\label{eq:Iderivs_smallZ}
\begin{aligned}
\abs{\bigg(\frac{I_1}{I_0} \bigg)'-\frac{1}{2}} &\le c\,z^2\,, \qquad
\abs{\bigg(\frac{I_1}{I_0} \bigg)''} \le c\,z \\
\abs{\bigg(\frac{I_0}{I_1} \bigg)'+\frac{2}{z^2}-\frac{1}{4}} &\le c\,z^2 \,, \qquad
\abs{\bigg(\frac{I_0}{I_1} \bigg)''-\frac{4}{z^3}} \le c\,z\,.
\end{aligned}
\end{equation}
\end{corollary}

Since we do not have precise values for the constants $c$ appearing in \eqref{eq:I1_I0} and \eqref{eq:I0_I1}, we will additionally need the following refinement.
Using the small-$z$ asymptotics \eqref{eq:I1_I0smallZ} and \eqref{eq:I0_I1smallZ}, note that we may choose $0<\delta< \frac{1}{4}$ small enough that
\begin{equation}\label{eq:delta_def}
\abs{\frac{I_1(z)}{I_0(z)}-\frac{z}{2}} \le \frac{z^2}{8} \quad \text{and} \quad \abs{\frac{I_0(z)}{I_1(z)}-\frac{2}{z}-\frac{z}{4}} \le \frac{z^2}{8}\,, \qquad 0<z\le\delta\,.
\end{equation}
With this choice of $\delta$, we may show the following proposition.

\begin{proposition}\label{prop:besselI_2}
Let $\delta$ be as in \eqref{eq:delta_def}. For $z\ge\delta$, the following bound holds between the difference of ratios of first kind modified Bessel functions:
\begin{equation}\label{eq:Idiff_bd}
\frac{1}{z}\le \frac{I_0(z)}{I_1(z)}-\frac{I_1(z)}{I_0(z)} \le \frac{c_\delta}{z}\,, \qquad c_\delta<2.
\end{equation}
In addition, for $z\ge \delta$, we may show that the following function is bounded away from zero: 
\begin{equation}\label{eq:Idenom_bd}
\frac{2}{z}+\bigg(1-z\frac{I_0(z)}{I_1(z)}\bigg)\bigg(\frac{I_0(z)}{I_1(z)}-\frac{I_1(z)}{I_0(z)}\bigg) < -c 
\end{equation}
for some $c>0$. 
\end{proposition}

\begin{proof}
We consider the function
\begin{equation}\label{eq:Bz}
B(z):= z\bigg( \frac{I_0(z)}{I_1(z)}-\frac{I_1(z)}{I_0(z)} \bigg)\,,
\end{equation}
which satisfies the ODE
\begin{align*}
B'(z) = -B(z)\bigg( \frac{I_0(z)}{I_1(z)}+\frac{I_1(z)}{I_0(z)} \bigg) + 2\frac{I_0(z)}{I_1(z)}
\end{align*}
with initial condition $B(\delta) = \delta\frac{I_0(\delta)}{I_1(\delta)}-\frac{I_1(\delta)}{I_0(\delta)}$. Using the bounds \eqref{eq:delta_def}, we have that
\begin{equation}\label{eq:Bdelta}
2-\frac{\delta^2}{4}(1+\delta)\le B(\delta)\le 2-\frac{\delta^2}{4}(1-\delta)\,.
\end{equation}
Noting that $I_0(z)>I_1(z)$ for all $z$ \cite{NIST:DLMF}, we have 
\begin{align*}
B'(z) &\le \big(-B(z) + 2\big)\frac{I_0(z)}{I_1(z)}\,,\\
B'(z) &\ge \big(-B(z) +1\big)\bigg( \frac{I_0(z)}{I_1(z)}+\frac{I_1(z)}{I_0(z)} \bigg) \,.
\end{align*}

Let $b(z)$ be the solution to the ODE 
\begin{align*}
b'(z) = \big(-b(z) + 2\big)\frac{I_0(z)}{I_1(z)}\,, \qquad b(\delta)= B(\delta) \,.
\end{align*}
Defining $s=\log(zI_1(z))$, we have
\begin{align*}
\frac{db}{ds} = -b(s) + 2\,, \qquad b(s(\delta)) = B(\delta)\,,
\end{align*}
which we can integrate directly to yield $b(s) = 2+ (B(\delta)-2) e^{-(s-s(\delta))}$, or
\begin{equation}\label{eq:b_s_z}
b(s(z)) = 2+ (B(\delta)-2)\frac{\delta I_1(\delta)}{zI_1(z)}\,.
\end{equation}
Using the large-$z$ asymptotics \eqref{eq:I1_I0} and \eqref{eq:I0_I1} of $\frac{I_1(z)}{I_0(z)}$ and $\frac{I_0(z)}{I_1(z)}$, we note that there exists $\eta>1$ such that 
\begin{equation}\label{eq:Bz_eta}
B(z)\le \frac{3}{2}\,, \qquad z\ge \eta\,.
\end{equation}
Meanwhile, $b(s(z))\to 2$ as $z\to\infty$, so for $z$ sufficiently large we have $b(s(z))>B(z)$. Thus by \eqref{eq:B_upper}, we have $B(z)\le b(s(z))$ for $z\ge \delta$. In particular, for $\delta\le z\le \eta$, we have 
\begin{equation}\label{eq:B_upperbd}
B(z)\le b(s(\eta)) \le 2 - \frac{\delta}{8}\frac{\delta I_1(\delta)}{\eta I_1(\eta)} <2\,. 
\end{equation}
Combining \eqref{eq:B_upperbd} and \eqref{eq:Bz_eta}, we obtain the upper bound of \eqref{eq:Idiff_bd}.

Similarly, let $h(z)$ be the solution to the ODE
\begin{align*}
 h'(z) =  \big(-h(z)+1\big)\bigg( \frac{I_0(z)}{I_1(z)}+\frac{I_1(z)}{I_0(z)} \bigg)\,, \qquad h(\delta) = B(\delta)\,.
 \end{align*} 
Defining $t=\log(zI_1(z)I_0(z))$, we have 
\begin{align*}
\frac{dh}{dt} = -h(t)+1\,, \qquad h(t(\delta)) = B(\delta)\,.
\end{align*}
This can again be directly integrated to yield $h(t) = 1 + (B(\delta)-1)e^{-(t-t(\delta))}$, or
\begin{equation}\label{eq:h_t_z}
h(t(z)) = 1 + (B(\delta)-1)\frac{\delta I_0(\delta)I_1(\delta)}{zI_0(z)I_1(z)}\,.
\end{equation}
Since both $I_1(z)$ and $I_0(z)$ grow exponentially in $z$ \cite{NIST:DLMF}, for sufficiently large $z$, we have $h(t(z))<B(z)$. Thus by \eqref{eq:B_lower}, we have
\begin{equation}\label{eq:B_lowerbd}
B(z)\ge h(t(z))\ge 1 \,, \qquad z\ge \delta\,. 
\end{equation}
This yields the lower bound of \eqref{eq:Idiff_bd}. \\

To show the second bound \eqref{eq:Idenom_bd}, we let 
\begin{align*}
G(z) = \frac{2}{z}+\bigg(1-z\frac{I_0(z)}{I_1(z)}\bigg)\bigg(\frac{I_0(z)}{I_1(z)}-\frac{I_1(z)}{I_0(z)}\bigg) \,.
\end{align*}
For $z\ge3$, since $\frac{I_0(z)}{I_1(z)}>1$ for all $z$, we may use \eqref{eq:Idiff_bd} to immediately obtain
\begin{equation}\label{eq:Gz_ge3}
G(z) \le \frac{2}{z}+\bigg(1-z\frac{I_0(z)}{I_1(z)}\bigg)\frac{1}{z}
= \frac{3}{z}-\frac{I_0(z)}{I_1(z)} \le -c <0\,.
\end{equation}
For $\delta\le z\le 3$, it suffices to show that $H(z) = zG(z)$ satisfies $H(z)\le -c$ for some $c>0$. Writing 
\begin{align*}
H(z) = 2+\bigg(1-z\frac{I_0(z)}{I_1(z)}\bigg)B(z)
\end{align*}
for $B(z)$ as in \eqref{eq:Bz}, we may calculate that $H(z)$ satisfies 
\begin{align*}
H'(z) &= \bigg(-H(z)+ \frac{B(z)(2-\frac{B(z)}{2})-2}{B(z)-1} \bigg)\frac{B(z)-1}{B(z)}\bigg(2\frac{I_0(z)}{I_1(z)}+\frac{I_1(z)}{I_0(z)} \bigg)\\
&\qquad -\frac{I_1(z)}{I_0(z)}\bigg(1-\frac{B(z)}{2}\bigg)-z(B(z)-1)\,.
\end{align*}
Noting that $B(z)\le c_\delta <2$ implies $B(z)(2-\frac{B(z)}{2})\le c<2$, we have $\frac{B(z)(2-\frac{B(z)}{2})-2}{B(z)-1} \le -c_B$ for some constant $c_B>0$. In particular, 
\begin{align*}
H'(z) &\le \big(-H(z)-c_B \big)\frac{B(z)-1}{B(z)}\bigg(2\frac{I_0(z)}{I_1(z)}+\frac{I_1(z)}{I_0(z)} \bigg)\,.
\end{align*}
Now, using the bounds \eqref{eq:delta_def}, for $0<z\le \delta$, we have
\begin{align*}
\abs{H(z) + \frac{z^2}{4}} \le \frac{3z^3}{4}\,. 
\end{align*}
In particular, $H(\delta)\le -\frac{\delta^2}{16}<0$, since $\delta\le \frac{1}{4}$. 
Letting $s(z)$ be such that $s'(z) = \frac{B(z)-1}{B(z)}\big(2\frac{I_0(z)}{I_1(z)}+\frac{I_1(z)}{I_0(z)} \big)>0$ for all $z$, we define $w$ to be the solution to the ODE
\begin{align*}
\frac{dw}{ds} = -w(s) - c_B\,, \qquad w(s(\delta))=H(\delta)\,.
\end{align*}
We may solve for $w$ explicitly as $w(s) = -c_B+ (c_B+H(\delta))e^{-s-s(\delta)}$. Using the large-$z$ asymptotics \eqref{eq:I1_I0} and \eqref{eq:I0_I1} of $\frac{I_1}{I_0}$ and $\frac{I_0}{I_1}$, along with \eqref{eq:B_upper}, we thus have $H(z)\le w(s(z))\le -c<0$ for all $z\ge \delta$; in particular, for $\delta\le z\le 3$. Combined with the bound \eqref{eq:Gz_ge3} for $z\ge 3$, we obtain the bound \eqref{eq:Idenom_bd}.
\end{proof}


Using Propositions \ref{prop:besselI_bounds}, \ref{prop:besselI_2}, and Corollary \ref{cor:Iderivs}, we may obtain bounds for the tangential multipliers $m_{\epsilon,z}^z$ and $m_{\epsilon,r}^z$ for the angle-averaged single layer operator $\overline{\mc{A}}_\epsilon$, defined in Corollary \ref{cor:Aeps}. 
\begin{lemma}[Bounds for $\overline{\mc{A}}_\epsilon$ multipliers, tangential direction]\label{lem:SLbounds_tang}
Consider the multipliers $m_{\epsilon,z}^z(\xi)$ and $m_{\epsilon,r}^z(\xi)$ given by \eqref{eq:Aeps_mt} in Corollary \ref{cor:Aeps}. 
For $\ell=0,1,2$, these multipliers satisfy: 
\begin{equation}
\begin{aligned}
\abs{\p_\xi^\ell m_{\epsilon,z}^z(\xi)} &\le 
\begin{cases}
c\,\epsilon\abs{\xi}^{1-\ell}\,, & \abs{\xi}\ge \frac{1}{2\pi\epsilon}\,,\\ 
c\,\abs{\xi}^{-\ell}\,, & \abs{\xi}< \frac{1}{2\pi\epsilon}\,;
\end{cases} \\
\abs{\p_\xi^\ell m_{\epsilon,r}^z(\xi)} &\le 
\begin{cases}
c\,, & \abs{\xi}\ge \frac{1}{2\pi\epsilon}\,, \; \ell=0\,,\\ 
c\,\abs{\xi}^{-1} \,, & \abs{\xi}\ge \frac{1}{2\pi\epsilon}\,, \; \ell=1,2\,,\\ 
c\,\epsilon^{-1}\abs{\xi}^{-1-\ell}\,, & \abs{\xi}< \frac{1}{2\pi\epsilon}\,.
\end{cases}
%
%
\end{aligned}
\end{equation}
\end{lemma}

\begin{proof}
We begin by defining 
\begin{align*}
m_{{\rm tA}}(z) = \frac{Q_{\rm tA}(z)}{Q_{\rm tD}(z)}\,, \quad 
m_{{\rm tB}}(z) = \frac{Q_{\rm tB}(z)}{Q_{\rm tD}(z)}
\end{align*}
where $Q_{{\rm t}j}$ are as in \eqref{eq:Sinv_tangential2} of Lemma \ref{lem:straight_SD}. Note that $m_{{\rm tA}}(2\pi\epsilon\abs{\xi})=\frac{1}{2}m_{\epsilon,z}^z(\xi)$ and $m_{{\rm tB}}(2\pi\epsilon\abs{\xi})=\frac{i}{2}m_{\epsilon,r}(\xi)$.
For both $j={\rm A,B}$, we have 
\begin{equation}\label{eq:mtj_derivs}
m_{{\rm t}j}'(z) = \frac{Q_{{\rm t}j}'Q_{\rm tD}- Q_{{\rm t}j}Q_{\rm tD}'}{Q_{\rm tD}^2}\,, \qquad 
m_{{\rm t}j}''(z) =\frac{Q_{\rm tD}Q_{{\rm t}j}''-Q_{{\rm t}j}Q_{\rm tD}''}{Q_{\rm tD}^2} + 2\frac{Q_{\rm tD}'}{Q_{\rm tD}}\,m_{{\rm t}j}'\,.
\end{equation}
It thus remains to bound each of $Q_{{\rm tA}}$, $Q_{{\rm tB}}$, $Q_{{\rm tD}}$, and their derivatives.

We begin with $Q_{\rm tD}$. Using Proposition \ref{prop:besselI_2} and Proposition \ref{prop:besselK_bounds}, we may obtain a lower bound for $Q_{\rm tD}(z)$ when $z\ge1$: in particular, 
\begin{equation}\label{eq:QtD_bds0}
Q_{\rm tD}(z) = \left(1+ \frac{z}{2}\left(\frac{I_1}{I_0}-\frac{I_0}{I_1} \right)\right)\left(1+ \frac{z}{2}\left(\frac{K_0}{K_1}-\frac{K_1}{K_0} \right)\right) \ge c>0\,, \qquad z\ge 1\,.
\end{equation}
Furthermore, using Propositions \ref{prop:besselK_bounds} and \ref{prop:besselI_bounds}, for $z\ge 1$ we have 
\begin{equation}\label{eq:QtD_bds}
\begin{aligned}
\textstyle \abs{Q_{\rm tD}(z)-\frac{1}{4}} &\le \frac{c}{z^2} \\
\abs{Q_{\rm tD}'(z)} &= \textstyle \bigg|\left( \frac{1}{2}\left(\frac{I_1}{I_0}-\frac{I_0}{I_1} \right)+\frac{z}{2}\left(\frac{I_1}{I_0}-\frac{I_0}{I_1} \right)'\right)\left(1+ \frac{z}{2}\left(\frac{K_0}{K_1}-\frac{K_1}{K_0} \right)\right)\\
&\qquad \textstyle +\left(1+ \frac{z}{2}\left(\frac{I_1}{I_0}-\frac{I_0}{I_1} \right)\right)\left(\frac{1}{2}\left(\frac{K_0}{K_1}-\frac{K_1}{K_0} \right)+\frac{z}{2}\left(\frac{K_0}{K_1}-\frac{K_1}{K_0} \right)'\right)\bigg| 
\le \displaystyle \frac{c}{z^2}\\
\abs{Q_{\rm tD}''(z)} &= \textstyle \bigg|\left( \left(\frac{I_1}{I_0}-\frac{I_0}{I_1} \right)'+\frac{z}{2}\left(\frac{I_1}{I_0}-\frac{I_0}{I_1} \right)''\right)\left(1+ \frac{z}{2}\left(\frac{K_0}{K_1}-\frac{K_1}{K_0} \right)\right)\\
&\qquad \textstyle + \left( \frac{I_1}{I_0}-\frac{I_0}{I_1} +z\left(\frac{I_1}{I_0}-\frac{I_0}{I_1} \right)'\right)\left(\frac{1}{2}\left(\frac{K_0}{K_1}-\frac{K_1}{K_0} \right)+\frac{z}{2}\left(\frac{K_0}{K_1}-\frac{K_1}{K_0} \right)'\right)\\
&\qquad \textstyle +\left(1+ \frac{z}{2}\left(\frac{I_1}{I_0}-\frac{I_0}{I_1} \right)\right)\left(\left(\frac{K_0}{K_1}-\frac{K_1}{K_0} \right)'+\frac{z}{2}\left(\frac{K_0}{K_1}-\frac{K_1}{K_0} \right)''\right)\bigg| 
 \le \displaystyle \frac{c}{z^2}\,.
\end{aligned}
\end{equation}
In addition, for $0<z<1$, using cancellations in the asymptotics \eqref{eq:I0_I1smallZ}, \eqref{eq:Iderivs_smallZ} for $\frac{I_0}{I_1}$, we have 
\begin{equation}\label{eq:QtD_smallz}
\frac{z^2}{c}\le \abs{Q_{\rm tD}(z)} \le c\,z^2 \,, \qquad
\abs{Q_{\rm tD}'(z)} \le c\,z \,, \qquad
\abs{Q_{\rm tD}''(z)} \le c\,.
\end{equation}

For $Q_{\rm tA}(z)$ and $Q_{\rm tB}(z)$, we first note that, using the identity \eqref{eq:I_K_id}, we may write
\begin{equation}\label{eq:rewrite_IDs}
\frac{1}{I_0K_0} = z\bigg(\frac{I_1}{I_0}+\frac{K_1}{K_0}\bigg)\,, \qquad
\frac{1}{I_1K_1} = z\bigg(\frac{I_0}{I_1}+\frac{K_0}{K_1}\bigg)\,.
\end{equation}
For $z\ge 1$, we may then use Propositions \ref{prop:besselI_bounds} and \ref{prop:besselK_bounds} to obtain
\begin{equation}\label{eq:QtA_bds}
\begin{aligned}
&\abs{Q_{\rm tA}(z)-z} = \textstyle \abs{z\left(\frac{I_1}{I_0}+\frac{K_1}{K_0}\right)\left(1+ \frac{z}{2}\left(\frac{K_0}{K_1}-\frac{I_0}{I_1} \right)\right)-z} \le \displaystyle \frac{c}{z} \\
&\abs{Q_{\rm tA}'(z)-1} = \textstyle \bigg|\left(\frac{I_1}{I_0}+\frac{K_1}{K_0}\right)\left(1+ \frac{z}{2}\left(\frac{K_0}{K_1}-\frac{I_0}{I_1} \right)\right)
+z\left(\frac{I_1}{I_0}+\frac{K_1}{K_0}\right)'\left(1+ \frac{z}{2}\left(\frac{K_0}{K_1}-\frac{I_0}{I_1} \right)\right) \\
&\qquad\qquad\qquad \textstyle +z\left(\frac{I_1}{I_0}+\frac{K_1}{K_0}\right)\left(\frac{1}{2}\left(\frac{K_0}{K_1}-\frac{I_0}{I_1} \right)+ \frac{z}{2}\left(\frac{K_0}{K_1}-\frac{I_0}{I_1} \right)'\right)-1\bigg| \le \displaystyle \frac{c}{z}\\
&\abs{Q_{\rm tA}''(z)} = \textstyle \bigg|2\left(\frac{I_1}{I_0}+\frac{K_1}{K_0}\right)'\left(1+ \frac{z}{2}\left(\frac{K_0}{K_1}-\frac{I_0}{I_1} \right)\right)
+2\left(\frac{I_1}{I_0}+\frac{K_1}{K_0}\right)\left(\frac{1}{2}\left(\frac{K_0}{K_1}-\frac{I_0}{I_1} \right)+ \frac{z}{2}\left(\frac{K_0}{K_1}-\frac{I_0}{I_1} \right)'\right)\\
&\qquad\quad \textstyle +z\left(\frac{I_1}{I_0}+\frac{K_1}{K_0}\right)''\left(1+ \frac{z}{2}\left(\frac{K_0}{K_1}-\frac{I_0}{I_1} \right)\right) 
+ 2z\left(\frac{I_1}{I_0}+\frac{K_1}{K_0}\right)'\left(\frac{1}{2}\left(\frac{K_0}{K_1}-\frac{I_0}{I_1} \right)+ \frac{z}{2}\left(\frac{K_0}{K_1}-\frac{I_0}{I_1} \right)'\right) \\
&\qquad\quad \textstyle +z\left(\frac{I_1}{I_0}+\frac{K_1}{K_0}\right)\left(\left(\frac{K_0}{K_1}-\frac{I_0}{I_1} \right)'+ \frac{z}{2}\left(\frac{K_0}{K_1}-\frac{I_0}{I_1} \right)''\right)\bigg| \le \displaystyle \frac{c}{z}\,.
\end{aligned}
\end{equation}
Similarly, for $0<z<1$, using the small-$z$ asymptotics of Proposition \ref{prop:besselI_bounds} and the bounds of Proposition \ref{prop:besselK_bounds}, we have 
\begin{equation}\label{eq:QtA_smallz}
 \abs{Q_{\rm tA}(z)} \le c\,z^2 \,, \quad
\abs{Q_{\rm tA}'(z)} \le  c\,z\,, \quad
\abs{Q_{\rm tA}''(z)} \le c\,.
\end{equation}

For $Q_{\rm tB}$, again using \eqref{eq:rewrite_IDs} and Propositions \ref{prop:besselI_bounds} and \ref{prop:besselK_bounds}, for $z\ge 1$ we have
\begin{equation}\label{eq:QtB_bds}
\begin{aligned}
\textstyle\abs{Q_{\rm tB}(z)-\frac{1}{2}} &= \textstyle \abs{\frac{z^2}{2}\left(\frac{I_0}{I_1}- \frac{I_1}{I_0} +\frac{K_0}{K_1}-\frac{K_1}{K_0}\right)-\frac{1}{2}} \le \displaystyle \frac{c}{z} \\
\abs{Q_{\rm tB}'(z)} &= \textstyle \abs{z\left(\frac{I_0}{I_1}- \frac{I_1}{I_0} +\frac{K_0}{K_1}-\frac{K_1}{K_0}\right)+ \frac{z^2}{2}\left(\frac{I_0}{I_1}- \frac{I_1}{I_0} +\frac{K_0}{K_1}-\frac{K_1}{K_0}\right)'} \le \displaystyle \frac{c}{z} \\
\abs{Q_{\rm tB}''(z)} &= \textstyle \bigg|\left(\frac{I_0}{I_1}- \frac{I_1}{I_0} +\frac{K_0}{K_1}-\frac{K_1}{K_0}\right)+2z\left(\frac{I_0}{I_1}- \frac{I_1}{I_0} +\frac{K_0}{K_1}-\frac{K_1}{K_0}\right)'\\
&\qquad +\textstyle  \frac{z^2}{2}\left(\frac{I_0}{I_1}- \frac{I_1}{I_0} +\frac{K_0}{K_1}-\frac{K_1}{K_0}\right)''\bigg| \le \displaystyle \frac{c}{z} \,.
\end{aligned}
\end{equation}
For $0<z<1$, we have
\begin{equation}\label{eq:QtB_smallz}
\abs{Q_{\rm tB}(z)} \le c\,z \,, \qquad
\abs{Q_{\rm tB}'(z)} \le c \,, \qquad
\abs{Q_{\rm tB}''(z)} \le \frac{c}{z}\,.
\end{equation}

Using the lower and upper bounds \eqref{eq:QtD_bds0} and \eqref{eq:QtD_bds} for $Q_{\rm tD}$, along with the bounds \eqref{eq:QtA_bds} and \eqref{eq:QtB_bds} for $Q_{\rm tA}$ and $Q_{\rm tB}$, for $z\ge 1$ the formulas \eqref{eq:mtj_derivs} yield
\begin{align*}
\abs{m_{\rm tA}(z)} &\le c\,z\,, \qquad \abs{m_{\rm tA}'(z)} \le c\,, \qquad \abs{m_{\rm tA}''(z)} \le \frac{c}{z}\,; \\
\abs{m_{\rm tB}(z)} &\le c\,, \qquad \abs{m_{\rm tB}'(z)} \le \frac{c}{z}\,, \qquad \abs{m_{\rm tB}''(z)} \le \frac{c}{z}\,. 
%
\end{align*}
Similarly, using \eqref{eq:QtD_smallz} with each of \eqref{eq:QtA_smallz}, \eqref{eq:QtB_smallz} in the formulas \eqref{eq:mtj_derivs}, for $0<z<1$ we have
\begin{align*}
\abs{m_{\rm tA}(z)} &\le c\,, \qquad \abs{m_{\rm tA}'(z)} \le \frac{c}{z}\,, \qquad \abs{m_{\rm tA}''(z)} \le \frac{c}{z^2}\,; \\
\abs{m_{\rm tB}(z)} &\le \frac{c}{z}\,, \qquad \abs{m_{\rm tB}'(z)} \le \frac{c}{z^2}\,, \qquad \abs{m_{\rm tB}''(z)} \le \frac{c}{z^3}\,.
%
\end{align*}
Using that $m_{{\rm tA}}(2\pi\epsilon\abs{\xi})=\frac{1}{2}m_{\epsilon,z}^z(\xi)$ and $m_{{\rm tB}}(2\pi\epsilon\abs{\xi})=\frac{i}{2}m_{\epsilon,r}(\xi)$, we obtain Lemma \ref{lem:SLbounds_tang}.
\end{proof}


We next consider the normal direction multipliers $m_{\epsilon,r}^x$, $m_{\epsilon,\theta}^x$, $m_{\epsilon,z}^x$ for the angle-averaged single layer operator $\overline{\mc{A}}_\epsilon$, defined in Corollary \ref{cor:Aeps}. Using Propositions \ref{prop:besselK_bounds}, \ref{prop:besselI_bounds}, \ref{prop:besselI_2}, and Corollaries \ref{cor:Kderivs} and \ref{cor:Iderivs}, we may obtain the following bounds for the normal direction multipliers and their derivatives.

\begin{lemma}[Bounds for $\overline{\mc{A}}_\epsilon$ multipliers, normal direction]\label{lem:SLbounds_norm}
Consider the multipliers $m_{\epsilon,r}^x(\xi)$, $m_{\epsilon,\theta}^x(\xi)$, and $m_{\epsilon,z}^x(\xi)$ given by \eqref{eq:Aeps_mn} in Corollary \ref{cor:Aeps}. For $\ell=0,1,2$, these multipliers satisfy: 
\begin{equation}
\begin{aligned}
\abs{\p_\xi^\ell m_{\epsilon,r}^x(\xi)} &\le
\begin{cases}
c\,\epsilon\abs{\xi}^{1-\ell}\,, & \abs{\xi}\ge \frac{1}{2\pi\epsilon}\,,\\ 
c\,\abs{\xi}^{-\ell} \,, & \abs{\xi}< \frac{1}{2\pi\epsilon}\,;
\end{cases} \\
\abs{\p_\xi^\ell m_{\epsilon,\theta}^x(\xi)} &\le
\begin{cases}
c\,\epsilon\abs{\xi}^{1-\ell}\,, & \abs{\xi}\ge \frac{1}{2\pi\epsilon}\,,\\ 
c\,\abs{\xi}^{-\ell} \,, & \abs{\xi}< \frac{1}{2\pi\epsilon}\,;
\end{cases} \\
\abs{\p_\xi^\ell m_{\epsilon,z}^x(\xi)} &\le
\begin{cases}
c\,\abs{\xi}^{-\ell} \,, & \abs{\xi}\ge \frac{1}{2\pi\epsilon}\,,\\ 
c\,\epsilon\abs{\log\epsilon}\abs{\xi}^{1-\ell} \,, & \abs{\xi}< \frac{1}{2\pi\epsilon}\,.
\end{cases} 
\end{aligned}
\end{equation}
\end{lemma}

\begin{proof}
We begin by defining
\begin{equation}\label{eq:mx_defs}
m_r^x(z) = \frac{Q_{\rm nA}(z)-Q_{\rm nB}(z)}{Q_{\rm nG}(z)}\,, \quad
m_\theta^x(z) = \frac{Q_{\rm nB}(z)-Q_{\rm nD}(z)}{Q_{\rm nG}(z)}\,, \quad
m_z^x(z) = \frac{Q_{\rm nC}(z)+Q_{\rm nE}(z)}{Q_{\rm nG}(z)}\,, 
\end{equation}
where the functions $Q_{{\rm n}j}(z)$ are as in \eqref{eq:Sinv_components}. Note that $m_{\epsilon,r}^x(\xi)=m_r^x(2\pi\epsilon\abs{\xi})$, $m_{\epsilon,\theta}^x(\xi)=m_\theta^x(2\pi\epsilon\abs{\xi})$, and $m_{\epsilon,z}^x(\xi)=i\,m_z^x(2\pi\epsilon\abs{\xi})$ for the expressions \eqref{eq:Aeps_mn} in Corollary \ref{cor:Aeps}.

The first and second derivatives of $m_r^x$ are given by
\begin{equation}\label{eq:mx_derivs}
\begin{aligned}
(m_r^x)'(z) &= \frac{(Q_{\rm nA}'-Q_{\rm nB}')Q_{\rm nG}-(Q_{\rm nA}-Q_{\rm nB})Q_{\rm nG}'}{Q_{\rm nG}^2} \\
(m_r^x)''(z) &= \frac{(Q_{\rm nA}''-Q_{\rm nB}'')Q_{\rm nG}-(Q_{\rm nA}-Q_{\rm nB})Q_{\rm nG}''}{Q_{\rm nG}^2} + 2\frac{Q_{\rm nG}'}{Q_{\rm nG}}(Q_{\rm nA}'-Q_{\rm nB}') \,,
\end{aligned}
\end{equation}
and the first and second derivatives of $m_\theta^x$ and $m_z^x$ may be written analogously. It thus remains to obtain bounds for the numerators and denominator of each multiplier.

We begin with bounds for $Q_{\rm nG}$. Using Proposition \ref{prop:besselI_2} and Proposition \ref{prop:besselK_bounds}, we may obtain the following lower bound for $\abs{Q_{\rm nG}(z)}$ when $z\ge1$: 
\begin{equation}\label{eq:QnG_lower}
\abs{Q_{\rm nG}(z)} = \abs{\textstyle \left(\frac{2}{z} + \left(1- z\frac{I_0}{I_1}\right)\left(\frac{I_0}{I_1}-\frac{I_1}{I_0}\right) \right) \left(\frac{2}{z}- \left(1+ z\frac{K_0}{K_1}\right)\left(\frac{K_0}{K_1}-\frac{K_1}{K_0}\right) \right)} \ge c >0\,. 
\end{equation}
In addition, by Propositions \ref{prop:besselI_bounds} and \ref{prop:besselK_bounds}, for $z\ge1$ we have the upper bounds 
\begin{equation}\label{eq:QnG_upper}
\begin{aligned}
\abs{Q_{\rm nG}(z)} &\le c \\
\abs{Q_{\rm nG}'(z)} &= \bigg|\textstyle \left(\frac{2}{z^2} + \left(\frac{I_0}{I_1} + z\left(\frac{I_0}{I_1}\right)'\right)\left(\frac{I_0}{I_1}-\frac{I_1}{I_0}\right) - \left(1- z\frac{I_0}{I_1}\right)\left(\frac{I_0}{I_1}-\frac{I_1}{I_0}\right)' \right)\times \\
&\qquad \times \textstyle \left(\frac{2}{z}- \left(1+ z\frac{K_0}{K_1}\right)\left(\frac{K_0}{K_1}-\frac{K_1}{K_0}\right) \right)  +\textstyle \left(\frac{2}{z} + \left(1- z\frac{I_0}{I_1}\right)\left(\frac{I_0}{I_1}-\frac{I_1}{I_0}\right) \right) \times \\
&\qquad \times \textstyle \left(\frac{2}{z^2} + \left(\frac{K_0}{K_1}+z\left(\frac{K_0}{K_1}\right)'\right)\left(\frac{K_0}{K_1}-\frac{K_1}{K_0}\right) 
+ \left(1+ z\frac{K_0}{K_1}\right)\left(\frac{K_0}{K_1}-\frac{K_1}{K_0}\right)'\right) \bigg| 
\le \displaystyle \frac{c}{z} \\
\abs{Q_{\rm nG}''(z)} &= \bigg|\textstyle \bigg(-\frac{4}{z^3} + \left(2\left(\frac{I_0}{I_1}\right)'+\left(\frac{I_0}{I_1}\right)''\right)\left(\frac{I_0}{I_1}-\frac{I_1}{I_0}\right)
+ 2\left(\frac{I_0}{I_1} + z\left(\frac{I_0}{I_1}\right)'\right)\left(\frac{I_0}{I_1}-\frac{I_1}{I_0}\right)' \\
&\quad \textstyle -\left(1- z\frac{I_0}{I_1}\right)\left(\frac{I_0}{I_1}-\frac{I_1}{I_0}\right)'' \bigg) \textstyle \left(\frac{2}{z}- \left(1+ z\frac{K_0}{K_1}\right)\left(\frac{K_0}{K_1}-\frac{K_1}{K_0}\right) \right) \\
&\quad +\textstyle \left(\frac{2}{z} + \left(1- z\frac{I_0}{I_1}\right)\left(\frac{I_0}{I_1}-\frac{I_1}{I_0}\right) \right)  \textstyle \bigg(-\frac{4}{z^3} + \left(2\left(\frac{K_0}{K_1}\right)'+\left(\frac{K_0}{K_1}\right)''\right)\left(\frac{K_0}{K_1}-\frac{K_1}{K_0}\right) \\
&\quad
+ \textstyle 2\left(\frac{K_0}{K_1}+z\left(\frac{K_0}{K_1}\right)'\right)\left(\frac{K_0}{K_1}-\frac{K_1}{K_0}\right)' 
+ \left(1+ z\frac{K_0}{K_1}\right)\left(\frac{K_0}{K_1}-\frac{K_1}{K_0}\right)''\bigg) \\
&\quad + \textstyle 2\left(\frac{2}{z^2} + \left(\frac{I_0}{I_1} + z\left(\frac{I_0}{I_1}\right)'\right)\left(\frac{I_0}{I_1}-\frac{I_1}{I_0}\right) - \left(1- z\frac{I_0}{I_1}\right)\left(\frac{I_0}{I_1}-\frac{I_1}{I_0}\right)' \right)\times \\
&\qquad \times \textstyle \left(\frac{2}{z^2} + \left(\frac{K_0}{K_1}+z\left(\frac{K_0}{K_1}\right)'\right)\left(\frac{K_0}{K_1}-\frac{K_1}{K_0}\right) 
+ \left(1+ z\frac{K_0}{K_1}\right)\left(\frac{K_0}{K_1}-\frac{K_1}{K_0}\right)'\right)\bigg| 
\le \displaystyle \frac{c}{z^2}\,.
\end{aligned}
\end{equation}
For $0<z<1$, we may use cancellations in the small-$z$ asymptotics \eqref{eq:I0_I1smallZ} for $\frac{I_0}{I_1}$ along with the small-$z$ bounds of Propositions \ref{prop:besselK_bounds} and \ref{prop:besselI_bounds} to obtain
\begin{equation}\label{eq:QnG_smallZ}
\frac{1}{c}\le \abs{Q_{\rm nG}(z)} \le c\,, \qquad
\abs{Q_{\rm nG}'(z)} \le \frac{c}{z}\,, \qquad 
\abs{Q_{\rm nG}''(z)} \le \frac{c}{z^2}\,, \qquad  0<z<1\,.
\end{equation}

We next consider the numerator of $m_r^x(z)$ in \eqref{eq:mx_defs}. Using the form \eqref{eq:Sinv_components} of $Q_{\rm nA}$ and $Q_{\rm nB}$, and using \eqref{eq:rewrite_IDs} to rewrite $\frac{1}{I_1K_1}$, for $z\ge 1$ we may use Propositions \ref{prop:besselK_bounds}, \ref{prop:besselI_bounds} and Corollaries \ref{cor:Kderivs}, \ref{cor:Iderivs} to obtain:
\begin{equation}\label{eq:mrx_bigZ}
\begin{aligned}
\abs{Q_{\rm nA}-Q_{\rm nB}} &= \bigg|\textstyle z\left(\frac{I_0}{I_1}+\frac{K_0}{K_1}\right) \bigg( 6+\left(\frac{8}{z^2}+2\right)\frac{I_1}{I_0}\frac{K_1}{K_0} + \frac{4}{z}\left(\frac{K_1}{K_0} -\frac{I_1}{I_0}\right) \\
&\qquad \textstyle -2z\left(\frac{I_0}{I_1} -\frac{K_0}{K_1}\right) -4\left(\frac{I_0}{I_1}\frac{K_1}{K_0} +\frac{I_1}{I_0}\frac{K_0}{K_1}\right) \bigg) \bigg| 
\le c\,z\,, \\
\abs{Q_{\rm nA}'-Q_{\rm nB}'} &\le c\,, \qquad
\abs{Q_{\rm nA}''-Q_{\rm nB}''} \le \frac{c}{z}\,.
\end{aligned}
\end{equation}
We have suppressed the full expressions for the first and second derivatives above, but note that they can easily be computed from the first expression.  

For $0<z<1$, we may use cancellations in the small-$z$ asymptotics \eqref{eq:I0_I1smallZ}, \eqref{eq:I1_I0smallZ} of $\frac{I_0}{I_1}$ and $\frac{I_1}{I_0}$ along with Proposition \eqref{prop:besselK_bounds} to obtain
\begin{equation}\label{eq:mrx_smallZ}
\begin{aligned}
\abs{Q_{\rm nA}-Q_{\rm nB}} \le c\,, \qquad
\abs{Q_{\rm nA}'-Q_{\rm nB}'} \le \frac{c}{z}\,, \qquad 
\abs{Q_{\rm nA}''-Q_{\rm nB}''} \le \frac{c}{z^2}\,.
\end{aligned}
\end{equation}
Combining \eqref{eq:QnG_lower}, \eqref{eq:QnG_upper}, \eqref{eq:QnG_smallZ}, \eqref{eq:mrx_bigZ}, and \eqref{eq:mrx_smallZ} with the definition \eqref{eq:mx_defs} of $m_r^x(z)$ and the expression \eqref{eq:mx_derivs} of its derivatives, we obtain
\begin{equation}\label{eq:mrx_bds0}
\begin{aligned}
\abs{m_r^x(z)} &\le \begin{cases}
c\,z\,,  & z\ge 1 \\
c\,, & 0<z<1\,,
\end{cases} \qquad 
\abs{(m_r^x)'(z)} \le \begin{cases}
c\,,  & z\ge 1 \\
\frac{c}{z}\,, & 0<z<1\,,
\end{cases} \\
\abs{(m_r^x)''(z)} &\le \begin{cases}
\frac{c}{z}\,,  & z\ge 1 \\
\frac{c}{z^2}\,, & 0<z<1\,.
\end{cases}
\end{aligned}
\end{equation}

We now turn to the numerator of $m_\theta^x(z)$ from \eqref{eq:mx_defs}. We recall the form \eqref{eq:Sinv_components} of $Q_{\rm nB}$ and $Q_{\rm nD}$ and again use \eqref{eq:rewrite_IDs} to rewrite $\frac{1}{I_1K_1}$. By Propositions \ref{prop:besselK_bounds}, \ref{prop:besselI_bounds} and Corollaries \ref{cor:Kderivs}, \ref{cor:Iderivs}, for $z\ge 1$ we have
\begin{equation}\label{eq:mthetax_bigZ}
\begin{aligned}
\abs{Q_{\rm nB}-Q_{\rm nD}} &= \bigg|\textstyle z\left(\frac{I_0}{I_1}+\frac{K_0}{K_1}\right) \bigg( 4\frac{K_1}{K_0}\frac{I_0}{I_1}-6\frac{K_1}{K_0}\frac{I_1}{I_0} +4\frac{I_1}{I_0}\frac{K_0}{K_1}-2 +\frac{4}{z}\left(\frac{I_1}{I_0}-\frac{K_1}{K_0}\right)\\
&\qquad \textstyle -\frac{8}{z^2}\frac{I_1}{I_0}\frac{K_1}{K_0} -z^2\left(\frac{I_0}{I_1}-\frac{I_1}{I_0}\right)\left(\frac{K_0}{K_1}-\frac{K_1}{K_0}\right) \bigg) \bigg| 
\le c\,z\,, \\
\abs{Q_{\rm nB}'-Q_{\rm nD}'} &\le c\,, \qquad
\abs{Q_{\rm nB}''-Q_{\rm nD}''} \le \frac{c}{z}\,.
\end{aligned}
\end{equation}
Again we have suppressed the full expressions for the first and second derivatives, but again they may be easily computed from the first expression.
For $0<z<1$, again using Proposition \eqref{prop:besselK_bounds} and cancellations in the small-$z$ asymptotics \eqref{eq:I0_I1smallZ}, \eqref{eq:I1_I0smallZ} of $\frac{I_0}{I_1}$ and $\frac{I_1}{I_0}$, we have
\begin{equation}\label{eq:mthetax_smallZ}
\abs{Q_{\rm nB}-Q_{\rm nD}} \le c\,, \qquad
\abs{Q_{\rm nB}'-Q_{\rm nD}'} \le \frac{c}{z}\,, \qquad
\abs{Q_{\rm nB}''-Q_{\rm nD}''} \le \frac{c}{z^2}\,.
\end{equation}
Combining \eqref{eq:QnG_lower}, \eqref{eq:QnG_upper}, \eqref{eq:QnG_smallZ}, \eqref{eq:mthetax_bigZ}, and \eqref{eq:mthetax_smallZ} with the definition \eqref{eq:mx_defs} of $m_\theta^x(z)$ and the analogous expression \eqref{eq:mx_derivs} of its derivatives, we obtain
\begin{equation}\label{eq:mthetax_bds0}
\begin{aligned}
\abs{m_\theta^x(z)} &\le \begin{cases}
c\,z\,,  & z\ge 1 \\
c\,, & 0<z<1\,,
\end{cases} \qquad 
\abs{(m_\theta^x)'(z)} \le \begin{cases}
c\,,  & z\ge 1 \\
\frac{c}{z}\,, & 0<z<1\,,
\end{cases} \\
\abs{(m_\theta^x)''(z)} &\le \begin{cases}
\frac{c}{z}\,,  & z\ge 1 \\
\frac{c}{z^2}\,, & 0<z<1\,.
\end{cases}
\end{aligned}
\end{equation}

Finally, we consider the numerator of $m_z^x(z)$ from \eqref{eq:mx_defs}. Using the form \eqref{eq:Sinv_components} of $Q_{\rm nC}$ and $Q_{\rm nE}$ along with the identity \eqref{eq:rewrite_IDs}, for $z\ge 1$ we have
\begin{equation}\label{eq:mzx_bigZ}
\begin{aligned}
\abs{Q_{\rm nC}+Q_{\rm nE}} &= \bigg|\textstyle z\left(\frac{I_0}{I_1}+\frac{K_0}{K_1}\right)\bigg(\frac{8}{z}\frac{I_1}{I_0}\frac{K_1}{K_0} + 6\left(\frac{I_0}{I_1}- \frac{K_0}{K_1}\right) +2\left(\frac{I_1}{I_0}-\frac{K_1}{K_0}\right)+2z\left(\frac{I_0}{I_1}\frac{K_0}{K_1} -1\right) \\
&\qquad \textstyle -\frac{12}{z} +z\left(\frac{I_0}{I_1}-\frac{I_1}{I_0}\right)\left(\frac{K_0}{K_1}-\frac{K_1}{K_0}\right) \bigg) \bigg| 
\le c\,,\\
\abs{Q_{\rm nC}'+Q_{\rm nE}'} &\le \frac{c}{z}\,, \qquad
\abs{Q_{\rm nC}''+Q_{\rm nE}''} \le \frac{c}{z^2}\,.
\end{aligned}
\end{equation}
Again we suppress the full expressions for the first and second derivatives, as they are straightforward to compute. For $0<z<1$, we again use Proposition \ref{prop:besselK_bounds} along with cancellations in the small-$z$ asymptotics \eqref{eq:I0_I1smallZ}, \eqref{eq:I1_I0smallZ} of $\frac{I_0}{I_1}$ and $\frac{I_1}{I_0}$ to obtain
\begin{equation}\label{eq:mzx_smallZ}
\textstyle 
\abs{Q_{\rm nC}+Q_{\rm nE}} \le c\,z\abs{\log(\frac{z}{2})}\,, \quad
\abs{Q_{\rm nC}'+Q_{\rm nE}'} \le c\,\abs{\log(\frac{z}{2})}\,, \quad
\abs{Q_{\rm nC}''+Q_{\rm nE}''} \le \displaystyle c\,\frac{\abs{\log(\frac{z}{2})}}{z}\,.
\end{equation}
Combining \eqref{eq:QnG_lower}, \eqref{eq:QnG_upper}, \eqref{eq:QnG_smallZ}, \eqref{eq:mzx_bigZ}, and \eqref{eq:mzx_smallZ} with the definition \eqref{eq:mx_defs} of $m_z^x(z)$ and the analogous expression \eqref{eq:mx_derivs} of its derivatives, we obtain
\begin{equation}\label{eq:mzx_bds0}
\begin{aligned}
\abs{m_z^x(z)} &\le \begin{cases}
c\,,  & z\ge 1 \\
c\,z\abs{\log(\frac{z}{2})}\,, & 0<z<1\,,
\end{cases} \qquad 
\abs{(m_z^x)'(z)} \le \begin{cases}
\frac{c}{z}\,,  & z\ge 1 \\
c\,\abs{\log(\frac{z}{2})}\,, & 0<z<1\,,
\end{cases} \\
\abs{(m_z^x)''(z)} &\le \begin{cases}
\frac{c}{z^2}\,,  & z\ge 1 \\
c\,\frac{\abs{\log(z/2)}}{z}\,, & 0<z<1\,.
\end{cases}
\end{aligned}
\end{equation}

Using the bounds \eqref{eq:mrx_bds0}, \eqref{eq:mthetax_bds0}, and \eqref{eq:mzx_bds0} along with the fact that $m_{\epsilon,r}^x(\xi)=m_r^x(2\pi\epsilon\abs{\xi})$, $m_{\epsilon,\theta}^x(\xi)=m_\theta^x(2\pi\epsilon\abs{\xi})$, and $m_{\epsilon,z}^x(\xi)=i\,m_z^x(2\pi\epsilon\abs{\xi})$, we obtain Lemma \ref{lem:SLbounds_norm}. 
\end{proof}

\subsection{Proofs of Lemmas \ref{lem:straight_Leps_Holder} and \ref{lem:Sinv_mapping}}\label{subsec:straight_mapping}
Given the multiplier bounds in Lemmas \ref{lem:mt_bounds}, \ref{lem:mn_bounds}, \ref{lem:SLbounds_tang}, \ref{lem:SLbounds_norm}, we may now use Lemma \ref{lem:besov} to prove Lemmas \ref{lem:straight_Leps_Holder} and \ref{lem:Sinv_mapping}. We begin with Lemma \ref{lem:straight_Leps_Holder}.

\subsubsection{Proof of Lemma \ref{lem:straight_Leps_Holder}}\label{subsubsec:Meps}
Recalling the definitions \eqref{eq:eigsT} and \eqref{eq:eigsN} of $m_{\epsilon,{\rm t}}$ and $m_{\epsilon,{\rm n}}$, we define the matrix-valued Fourier multipliers 
\begin{equation}\label{eq:Meps12_defs}
\begin{aligned}
\wh{\bm{M}}_{\epsilon}^{(1)}(\xi) &=  m_{\epsilon,{\rm t}}(\xi)\, \be_z\otimes\be_z + m_{\epsilon,{\rm n}}(\xi)\, (\be_x\otimes\be_x +\be_y\otimes\be_y) \,, \qquad
\wh{\bm{M}}_{\epsilon,j}^{(1)}(\xi) = \phi_j(\xi)\wh{\bm{M}}_{\epsilon}^{(1)}(\xi) \\
\wh{\bm{M}}_{\epsilon}^{(2)}(\xi) &=  m_{\epsilon,{\rm t}}^{-1}(\xi)\, \be_z\otimes\be_z + m_{\epsilon,{\rm n}}^{-1}(\xi)\, (\be_x\otimes\be_x +\be_y\otimes\be_y)\,, \qquad
\wh{\bm{M}}_{\epsilon,j}^{(2)}(\xi) = \phi_j(\xi)\wh{\bm{M}}_{\epsilon}^{(2)}(\xi) \,.
\end{aligned}
\end{equation}
Here $\phi_j(\xi)$ is as in \eqref{eq:phij_def}. Letting $j_\epsilon=\frac{\abs{\log(2\pi\epsilon)}}{\log(2)}$, by Lemmas \ref{lem:mt_bounds} and \ref{lem:mn_bounds}, we have 
\begin{align*}
\norm{\p_\xi^\ell\wh{\bm{M}}_{\epsilon,j}^{(1)}}_{L^\infty} &\le \begin{cases}
c\,\epsilon^{-1}\,2^{-j(\ell+1)}  \,, & j\ge j_\epsilon \\
c\abs{\log\epsilon}\,2^{-j\ell}\,, &j<j_\epsilon \,,
\end{cases} \qquad \ell=0,1,2\,; \\
\norm{\p_\xi^\ell\wh{\bm{M}}_{\epsilon,j}^{(2)}}_{L^\infty} &\le \begin{cases}
c\,\epsilon \,2^{j(1-\ell)}\,, &j\ge j_\epsilon \\
c\abs{\log\epsilon}^{-1} 2^{-j\ell}\,, &j<j_\epsilon\,,
\end{cases} \qquad \ell=0,1,2\,.
\end{align*}
We may then use Lemma \ref{lem:besov} to obtain physical space estimates for $\bm{M}_{\epsilon,j}^{(1)}$ and $\bm{M}_{\epsilon,j}^{(2)}$: 
\begin{align*}
\norm{\bm{M}_{\epsilon,j}^{(1)}}_{L^1} &\le \begin{cases}
c\,\epsilon^{-1}\,2^{-j} \,, &j\ge j_\epsilon  \\
c\abs{\log\epsilon} \,, &j<j_\epsilon\,,
\end{cases}  \qquad
\norm{\bm{M}_{\epsilon,j}^{(2)}}_{L^1} \le \begin{cases}
c\,\epsilon\, 2^j\,, &j\ge j_\epsilon \\
c\abs{\log\epsilon}^{-1} \,, &j<j_\epsilon\,.
\end{cases}
\end{align*}
Given $\bm{h}\in C^{0,\alpha}(\T)$ with $\int_\T\bm{h}(s)\,ds = 0$ and $\bm{g}\in C^{1,\alpha}(\T)$, we may use the dyadic partition of unity \eqref{eq:phij_def} to write
\begin{align*}
\overline{\mc{L}}_\epsilon[\bm{h}] &= T_{\wh{\bm{M}}_\epsilon^{(1)}}\bm{h} = \sum_j P_jT_{\wh{\bm{M}}_\epsilon^{(1)}}\bm{h}\,, \qquad
\overline{\mc{L}}_\epsilon^{-1}[\bm{g}] = T_{\wh{\bm{M}}_\epsilon^{(2)}}\bm{g} = \sum_j P_jT_{\wh{\bm{M}}_\epsilon^{(2)}}\bm{g}\,.
\end{align*}

On $\T$, we have $P_{\le 0}T_{\wh{\bm{M}}_\epsilon^{(1)}}\bm{h} = P_0T_{\wh{\bm{M}}_\epsilon^{(1)}}\bm{h}$ and $P_{\le 0}T_{\wh{\bm{M}}_\epsilon^{(2)}}\bm{g} = P_0T_{\wh{\bm{M}}_\epsilon^{(2)}}\bm{g}$. We may then estimate
\begin{equation}\label{eq:Meps0mode}
\begin{aligned}
\norm{P_{\le 0}T_{\wh{\bm{M}}_\epsilon^{(1)}}\bm{h}}_{L^\infty} &= 
\norm{\bm{M}_{\epsilon,0}^{(1)}*\bm{h}}_{L^\infty} \le
\norm{\bm{M}_{\epsilon,0}^{(1)}}_{L^1}\norm{\bm{h}}_{L^\infty}
\le c\abs{\log\epsilon}\norm{\bm{h}}_{L^\infty} \\
\norm{P_{\le 0}T_{\wh{\bm{M}}_\epsilon^{(2)}}\bm{g}}_{L^\infty} &= 
\norm{\bm{M}_{\epsilon,0}^{(2)}*\bm{g}}_{L^\infty} \le
\norm{\bm{M}_{\epsilon,0}^{(2)}}_{L^1}\norm{\bm{g}}_{L^\infty}
\le c\abs{\log\epsilon}^{-1}\norm{\bm{g}}_{L^\infty}\, .
\end{aligned}
\end{equation}
For $j>0$ we introduce the notation $\wt P_j= P_{j-1}+P_j+P_{j+1}$ and note that $P_jT_{\wh{\bm{M}}_\epsilon^{(1)}}\bm{h}=P_jT_{\wh{\bm{M}}_\epsilon^{(1)}}\wt P_{j}\bm{h}$ and $P_jT_{\wh{\bm{M}}_\epsilon^{(2)}}\bm{g}=P_jT_{\wh{\bm{M}}_\epsilon^{(2)}}\wt P_{j}\bm{g}$. For low frequencies $0<j< j_\epsilon$, we estimate
\begin{equation}\label{eq:Meps_lowmode}
\begin{aligned}
\sup_{0<j< j_\epsilon} 2^{j(1+\alpha)}\norm{P_jT_{\wh{\bm{M}}_\epsilon^{(1)}}\bm{h}}_{L^\infty} &= \sup_{0<j< j_\epsilon} 2^{j(1+\alpha)}\norm{P_jT_{\wh{\bm{M}}_\epsilon^{(1)}}\wt P_j\bm{h}}_{L^\infty} \\
&= \sup_{0<j< j_\epsilon} 2^{j(1+\alpha)}\norm{\bm{M}_{\epsilon,j}^{(1)}*(\wt P_j\bm{h})}_{L^\infty} \\
&\le 2^{j_\epsilon}\sup_{0<j< j_\epsilon} \norm{\bm{M}_{\epsilon,j}^{(1)}}_{L^1}\,2^{j\alpha}\norm{\wt P_j\bm{h}}_{L^\infty} 
\le c\,\epsilon^{-1}\abs{\log\epsilon}\abs{\bm{h}}_{\dot B^\alpha_{\infty,\infty}}\\
\sup_{0<j< j_\epsilon} 2^{j\alpha}\norm{P_jT_{\wh{\bm{M}}_\epsilon^{(2)}}\bm{g}}_{L^\infty} &= \sup_{0<j< j_\epsilon} 2^{j\alpha}\norm{P_jT_{\wh{\bm{M}}_\epsilon^{(2)}}\wt P_j\bm{g}}_{L^\infty} \\
&= \sup_{0<j< j_\epsilon} 2^{j\alpha}\norm{\bm{M}_{\epsilon,j}^{(2)}*(\wt P_j\bm{g})}_{L^\infty} \\
&\le \sup_{0<j< j_\epsilon} \norm{\bm{M}_{\epsilon,j}^{(2)}}_{L^1}\,2^{j\alpha}\norm{\wt P_j\bm{g}}_{L^\infty} 
\le c \abs{\log\epsilon}^{-1} \abs{\bm{g}}_{\dot B^\alpha_{\infty,\infty}}\,.
\end{aligned}
\end{equation}
For high frequencies $j\ge j_\epsilon$, we have
\begin{equation}\label{eq:Meps_highmode}
\begin{aligned}
\sup_{j\ge j_\epsilon} 2^{j(1+\alpha)}\norm{P_jT_{\wh{\bm{M}}_\epsilon^{(1)}}\bm{h}}_{L^\infty} &\le 
\sup_{j\ge j_\epsilon} 2^{j(1+\alpha)}\norm{\bm{M}_{\epsilon,j}^{(1)}}_{L^1}\norm{\wt P_j\bm{h}}_{L^\infty}  \\
&\le c\,\epsilon^{-1}\sup_{j\ge j_\epsilon} 2^{j\alpha}\norm{\wt P_j\bm{h}}_{L^\infty} 
\le c\,\epsilon^{-1}\abs{\bm{h}}_{\dot B^\alpha_{\infty,\infty}}\\
\sup_{j\ge j_\epsilon} 2^{j\alpha}\norm{P_jT_{\wh{\bm{M}}_\epsilon^{(2)}}\bm{g}}_{L^\infty} &\le 
\sup_{j\ge j_\epsilon} 2^{j\alpha}\norm{\bm{M}_{\epsilon,j}^{(2)}}_{L^1}\norm{\wt P_j\bm{g}}_{L^\infty} \\
&\le c\,\epsilon\,\sup_{j\ge j_\epsilon} 2^{j(1+\alpha)}\norm{\wt P_j\bm{g}}_{L^\infty}
\le c\,\epsilon\abs{\bm{g}}_{\dot B^{1+\alpha}_{\infty,\infty}}\,.
\end{aligned}
\end{equation}
Combining the estimates \eqref{eq:Meps0mode}, \eqref{eq:Meps_lowmode}, and \eqref{eq:Meps_highmode} and using the norm definition \eqref{eq:besov}, we obtain Lemma \ref{lem:straight_Leps_Holder}.
\hfill\qedsymbol \\

\subsubsection{Proof of Lemma \ref{lem:Sinv_mapping} }
The mapping properties of the operator $\overline{\mc{A}}_\epsilon\P_{01}$ may be shown by similar means to the previous section. 
Recall the definition \eqref{eq:phij_def} of $\phi_j$ and the forms of the five distinct multipliers $m_{\epsilon,z}^z$, $m_{\epsilon,r}^z$, $m_{\epsilon,z}^x$, $m_{\epsilon,r}^x$, $m_{\epsilon,\theta}^x$ appearing in Corollary \ref{cor:Aeps} as the components of $\overline{\mc{A}}_\epsilon\P_{01}$. We define the scalar-valued Fourier multipliers 
\begin{align*}
\wh M_{\epsilon,z,j}^z(\xi) &= \phi_j(\xi) m_{\epsilon,z}^z(\xi)\,, \quad \wh M_{\epsilon,r,j}^z(\xi) = \phi_j(\xi) m_{\epsilon,r}^z(\xi)\,, \\
\wh M_{\epsilon,z,j}^x(\xi) &= \phi_j(\xi) m_{\epsilon,z}^x(\xi)\,, \quad \wh M_{\epsilon,r,j}^x(\xi) = \phi_j(\xi) m_{\epsilon,r}^x(\xi)\,, \quad \wh M_{\epsilon,\theta,j}^x(\xi) = \phi_j(\xi) m_{\epsilon,\theta}^x(\xi)\,.
\end{align*}
As in the proof of Lemma \ref{lem:straight_Leps_Holder}, we define $j_\epsilon = \frac{\abs{\log(2\pi\epsilon)}}{\log(2)}$. Using Lemmas \ref{lem:SLbounds_tang} and \ref{lem:SLbounds_norm}, we have that three of the multipliers $\wh M_{\epsilon,z,j}^z$, $\wh M_{\epsilon,r,j}^x$, and $\wh M_{\epsilon,\theta,j}^x$ each satisfy the same set of estimates, namely 
\begin{align*}
 \norm{\p_\xi^\ell\wh M_{\epsilon,z,j}^z}_{L^\infty} \,, 
 \norm{\p_\xi^\ell\wh M_{\epsilon,r,j}^x}_{L^\infty} \,,
\norm{\p_\xi^\ell\wh M_{\epsilon,\theta,j}^x}_{L^\infty} 
&\le \begin{cases}
c\,\epsilon \,2^{j(1-\ell)}\,, & j\ge j_\epsilon \\
c\, 2^{-j\ell}\,, & j<j_\epsilon\,,
\end{cases}
\quad
\ell=0,1,2\,.
 \end{align*}
These multipliers will comprise $\overline{\mc{A}}_\epsilon^\epsilon$, the leading order behavior (with respect to regularity) of the angle-averaged single layer operator $\overline{\mc{A}}_\epsilon\P_{01}$. The remaining multipliers  
$\wh M_{\epsilon,r,j}^z$ and $\wh M_{\epsilon,z,j}^x$ satisfy a different set of estimates:
 \begin{align*}
  \norm{\p_\xi^\ell\wh M_{\epsilon,r,j}^z}_{L^\infty}
  &\le 
\begin{cases}
c\,, & j\ge j_\epsilon\,, \;\ell=0 \\
c\,2^{-j}\,, & j\ge j_\epsilon\,, \;\ell=1,2 \\
c\, \epsilon^{-1}2^{-j(1+\ell)}\,, & j<j_\epsilon\,,\; \ell=0,1,2\,,
\end{cases}   \\
\norm{\p_\xi^\ell\wh M_{\epsilon,z,j}^x}_{L^\infty}
  &\le 
\begin{cases}
c\,2^{-j\ell}\,, & j\ge j_\epsilon \\
c\, \epsilon\abs{\log\epsilon}2^{j(1-\ell)}\,, & j<j_\epsilon\,.
\end{cases} 
 \end{align*}
These multipliers will make up the operator $\overline{\mc{A}}_\epsilon^+$, which has `bad' $\epsilon$-dependence but is smoother.

Using the scalar-valued version of Lemma \ref{lem:besov} (see \cite[Lemma 2.1]{laplace}), we may obtain the following physical space estimates for the inverse Fourier transforms of each $\wh M_{\epsilon,j}$: 
\begin{equation}\label{eq:M_L1ests}
\begin{aligned}
 \norm{M_{\epsilon,z,j}^z}_{L^1} \,, 
 \norm{M_{\epsilon,r,j}^x}_{L^1} \,,
\norm{M_{\epsilon,\theta,j}^x}_{L^1} 
&\le \begin{cases}
c\,\epsilon \,2^j\,, & j\ge j_\epsilon \\
c\,, & j<j_\epsilon\,,
\end{cases} \\
\norm{M_{\epsilon,r,j}^z}_{L^1}
  &\le 
\begin{cases}
c\,2^{-j/2}\,, & j\ge j_\epsilon \\
c\,\epsilon^{-1}2^{-j}\,, & j<j_\epsilon\,,
\end{cases}   \\
\norm{M_{\epsilon,z,j}^x}_{L^1}
  &\le 
\begin{cases}
c\,, & j\ge j_\epsilon \\
c\, \epsilon\abs{\log\epsilon}2^j\,, & j<j_\epsilon\,.
\end{cases} 
 \end{aligned}
 \end{equation}

It will now be useful to further decompose the projection operator $\P_{01}$ as follows. Given $\bm{g}(s,\theta)$ defined along $\mc{C}_\epsilon$, we may consider $\P_{01}\bm{g}$ as in \eqref{eq:P01_def} and we may further define the following projections:
\begin{align*}
\P_0^z \,\bm{g} &= g_0^z\be_z\,, \qquad \P_0^r \bm{g} = g_0^r\be_r\,, \\
\P_1^z\,\bm{g}&=g_{\rm c,1}^z\cos\theta\be_z + g_{\rm s,1}^z\sin\theta\be_z \,,\\
\P_1^r\,\bm{g} &= g_{\rm c,1}^r\cos\theta \be_r + g_{\rm s,1}^r\sin\theta \be_r\,, \\
\P_1^\theta\,\bm{g} &= g_{\rm s,1}^\theta\sin\theta \be_\theta + g_{\rm c,1}^\theta\cos\theta \be_\theta\,.
\end{align*}
We then decompose $\overline{\mc{A}}_\epsilon \P_{01}$ as $\overline{\mc{A}}_\epsilon \P_{01}=\overline{\mc{A}}_\epsilon^\epsilon+\overline{\mc{A}}_\epsilon^+$where
\begin{align*}
\overline{\mc{A}}_\epsilon^\epsilon &= \overline{\mc{A}}_\epsilon(\P_0^z+\P_1^r+\P_1^\theta)\,, 
\qquad \overline{\mc{A}}_\epsilon^+ = \overline{\mc{A}}_\epsilon(\P_0^r + \P_1^z)\,.
\end{align*}

For $\bm{g}\in C_s^{1,\alpha}(\Gamma_\epsilon)$, we use the dyadic partition of unity \eqref{eq:phij_def} to write 
\begin{equation}\label{eq:Aeps_components}
\begin{aligned}
\overline{\mc{A}}_\epsilon^\epsilon\bm{g} &=  T_{M_{\epsilon,z,j}^z}\P_0^z\bm{g}+ T_{M_{\epsilon,r,j}^x}\P_1^r\bm{g} 
+ T_{M_{\epsilon,\theta,j}^x}\P_1^\theta\bm{g} \\
&= \sum_j P_j\big(T_{M_{\epsilon,z,j}^z}\P_0^z\bm{g}+ T_{M_{\epsilon,r,j}^x}\P_1^r\bm{g} 
+ T_{M_{\epsilon,\theta,j}^x}\P_1^\theta\bm{g} \big)\,, \\
\overline{\mc{A}}_\epsilon^+\bm{g} &=  T_{M_{\epsilon,r,j}^z}\P_0^r\bm{g}+ T_{M_{\epsilon,z,j}^x}\P_1^z\bm{g}  \\
&= \sum_j P_j\big(T_{M_{\epsilon,r,j}^z}\P_0^r\bm{g}+ T_{M_{\epsilon,z,j}^x}\P_1^z\bm{g} \big)\,.
\end{aligned}
\end{equation}
We may now obtain bounds for $\overline{\mc{A}}_\epsilon^\epsilon\bm{g}$ and $\overline{\mc{A}}_\epsilon^+\bm{g}$. For each multiplier $M$, we have $P_{\le 0}T_M \bm{g}=P_0T_M\bm{g}$, so using the $L^1$ bounds \eqref{eq:M_L1ests} we may estimate each of the components of $\overline{\mc{A}}_\epsilon^\epsilon$ as 
\begin{align*}
\norm{P_{\le 0}T_{M_{\epsilon,z,j}^z}\P_0^z\bm{g}}_{L^\infty} 
&= \norm{M_{\epsilon,z,0}^z*(\P_0^z\bm{g})}_{L^\infty} 
\le c\norm{M_{\epsilon,z,0}^z}_{L^1}\norm{\bm{g}}_{L^\infty} 
\le c\norm{\bm{g}}_{L^\infty} \,,\\
\norm{P_{\le 0}T_{M_{\epsilon,r,j}^x}\P_1^r\bm{g}}_{L^\infty} &\le c\norm{\bm{g}}_{L^\infty} \,, \qquad 
\norm{P_{\le 0}T_{M_{\epsilon,\theta,j}^x}\P_1^\theta\bm{g}}_{L^\infty} \le c\norm{\bm{g}}_{L^\infty} \,.
\end{align*}
Here and throughout, the convolution $*$ is with respect to $s$ only. The components of $\overline{\mc{A}}_\epsilon^+$ may be similarly estimated as 
\begin{align*}
\norm{P_{\le 0}T_{M_{\epsilon,r,j}^z}\P_0^r\bm{g}}_{L^\infty} 
&= \norm{M_{\epsilon,r,0}^z*(\P_0^r\bm{g})}_{L^\infty} 
\le c\norm{M_{\epsilon,r,0}^z}_{L^1}\norm{\bm{g}}_{L^\infty} 
\le c\,\epsilon^{-1}\norm{\bm{g}}_{L^\infty} \,,\\
\norm{P_{\le 0}T_{M_{\epsilon,z,j}^x}\P_1^z\bm{g}}_{L^\infty} 
&= \norm{M_{\epsilon,z,0}^x*(\P_1^z\bm{g})}_{L^\infty} 
\le c\norm{M_{\epsilon,z,0}^x}_{L^1}\norm{\bm{g}}_{L^\infty} 
\le c\,\epsilon\abs{\log\epsilon}\norm{\bm{g}}_{L^\infty} \,.
\end{align*}

We next bound the low frequencies $0<j< j_\epsilon$. As in the proof of Lemma \ref{lem:straight_Leps_Holder}, we will use the notation $\wt P_j\bm{g}=(P_{j-1}+P_j+P_{j+1})\bm{g}$ and again note that $P_jT_m\bm{g}=P_jT_m\wt P_j\bm{g}$.

Using \eqref{eq:M_L1ests}, we may estimate the first component of $\overline{\mc{A}}_\epsilon^\epsilon$ as
\begin{align*}
\sup_{0<j< j_\epsilon} 2^{j\alpha}\norm{P_j T_{M_{\epsilon,z,j}^z}\P_0^z\bm{g}}_{L^\infty} 
&\le \sup_{0<j< j_\epsilon} 2^{j\alpha}\norm{P_j T_{M_{\epsilon,z,j}^z}\wt P_j(\P_0^z\bm{g})}_{L^\infty} \\
&= \sup_{0<j< j_\epsilon} 2^{j\alpha}\norm{M_{\epsilon,z,j}^z*(\wt P_j\P_0^z\bm{g})}_{L^\infty} \\
&\le \sup_{0<j< j_\epsilon} 2^{j\alpha}\norm{M_{\epsilon,z,j}^z}_{L^1}\norm{\wt P_j\P_0^z\bm{g}}_{L^\infty}\\
&\le c\sup_{0<j< j_\epsilon} 2^{j\alpha}\norm{\wt P_j\P_0^z\bm{g}}_{L^\infty}
\le c\,\abs{\bm{g}}_{(\dot B^\alpha_{\infty,\infty})_s} \,.
\end{align*}
Likewise, the remaining two terms of $\overline{\mc{A}}_\epsilon^\epsilon$ satisfy
\begin{align*}
\sup_{0<j< j_\epsilon} 2^{j\alpha}\norm{P_j T_{M_{\epsilon,r,j}^x}\P_1^r\bm{g}}_{L^\infty} 
\le c\,\abs{\bm{g}}_{(\dot B^\alpha_{\infty,\infty})_s}\,,
\qquad  
\sup_{0<j< j_\epsilon} 2^{j\alpha}\norm{P_j T_{M_{\epsilon,\theta,j}^x}\P_1^\theta\bm{g}}_{L^\infty} 
\le c\,\abs{\bm{g}}_{(\dot B^\alpha_{\infty,\infty})_s}\,.
\end{align*}
In addition, the two components of $\overline{\mc{A}}_\epsilon^+$ satisfy
\begin{align*}
\sup_{0<j< j_\epsilon} 2^{j(3/2+\alpha)}\norm{P_j T_{M_{\epsilon,r,j}^z}\P_0^r\bm{g}}_{L^\infty} 
&\le \sup_{0<j< j_\epsilon} 2^{j(3/2+\alpha)}\norm{M_{\epsilon,r,j}^z}_{L^1}\norm{\wt P_j\P_0^r\bm{g}}_{L^\infty}\\
&\le c\,\epsilon^{-1}\sup_{0<j< j_\epsilon} 2^{j(1/2+\alpha)}\norm{\wt P_j\P_0^r\bm{g}}_{L^\infty}
\le c\,\epsilon^{-1}\abs{\bm{g}}_{(\dot B^{1/2+\alpha}_{\infty,\infty})_s} \,,\\
\sup_{0<j< j_\epsilon} 2^{j(1+\alpha)}\norm{P_j T_{M_{\epsilon,z,j}^x}\P_1^z\bm{g}}_{L^\infty} 
&\le \sup_{0<j< j_\epsilon} 2^{j(1+\alpha)}\norm{M_{\epsilon,z,j}^x}_{L^1}\norm{\wt P_j\P_1^z\bm{g}}_{L^\infty}\\
&\le c\,\epsilon\abs{\log\epsilon}\sup_{0<j< j_\epsilon} 2^{j(2+\alpha)}\norm{\wt P_j\P_1^z\bm{g}}_{L^\infty}\\
&\le c\,\epsilon \,2^{j_\epsilon}\abs{\log\epsilon}\sup_{0<j< j_\epsilon} 2^{j(1+\alpha)}\norm{\wt P_j\P_1^z\bm{g}}_{L^\infty}
\le c\abs{\log\epsilon}\abs{\bm{g}}_{(\dot B^{1+\alpha}_{\infty,\infty})_s} \,.
\end{align*}
Here we have use that $2^{j_\epsilon}=\frac{1}{2\pi\epsilon}$.

Finally, for high frequencies $j\ge j_\epsilon$, we may use the $L^1$ estimates \eqref{eq:M_L1ests} to show that the components of $\overline{\mc{A}}_\epsilon^\epsilon$ each satisfy
\begin{align*}
\sup_{j\ge j_\epsilon} 2^{j\alpha}\norm{P_j T_{M_{\epsilon,z,j}^z}\P_0^z\bm{g}}_{L^\infty} 
&\le \sup_{j\ge j_\epsilon} 2^{j\alpha}\norm{P_j T_{M_{\epsilon,z,j}^z}\wt P_j(\P_0^z\bm{g})}_{L^\infty} \\
&= \sup_{j\ge j_\epsilon} 2^{j\alpha}\norm{M_{\epsilon,z,j}^z*(\wt P_j\P_0^z\bm{g})}_{L^\infty} \\
&\le \sup_{j\ge j_\epsilon} 2^{j\alpha}\norm{M_{\epsilon,z,j}^z}_{L^1}\norm{\wt P_j\P_0^z\bm{g}}_{L^\infty}\\
&\le c\,\epsilon\sup_{j\ge j_\epsilon} 2^{j(1+\alpha)}\norm{\wt P_j\P_0^z\bm{g}}_{L^\infty}
\le c\,\epsilon\abs{\bm{g}}_{(\dot B^{1+\alpha}_{\infty,\infty})_s} \,,\\
\sup_{j\ge j_\epsilon} 2^{j\alpha}\norm{P_j T_{M_{\epsilon,r,j}^x}\P_1^r\bm{g}}_{L^\infty} 
&\le c\,\epsilon\abs{\bm{g}}_{(\dot B^{1+\alpha}_{\infty,\infty})_s}\,,
\qquad  
\sup_{j\ge j_\epsilon} 2^{j\alpha}\norm{P_j T_{M_{\epsilon,\theta,j}^x}\P_1^\theta\bm{g}}_{L^\infty} 
\le c\,\epsilon\abs{\bm{g}}_{(\dot B^{1+\alpha}_{\infty,\infty})_s}\,.
\end{align*}
For $\overline{\mc{A}}_\epsilon^+$, we have
\begin{align*}
\sup_{j\ge j_\epsilon} 2^{j(3/2+\alpha)}\norm{P_j T_{M_{\epsilon,r,j}^z}\P_0^r\bm{g}}_{L^\infty} 
&\le \sup_{j\ge j_\epsilon} 2^{j(3/2+\alpha)}\norm{M_{\epsilon,r,j}^z}_{L^1}\norm{\wt P_j\P_0^r\bm{g}}_{L^\infty}\\
&\le c\sup_{j\ge j_\epsilon} 2^{j(1+\alpha)}\norm{\wt P_j\P_0^r\bm{g}}_{L^\infty}
\le c\,\abs{\bm{g}}_{(\dot B^{1+\alpha}_{\infty,\infty})_s} \,,\\
\sup_{j\ge j_\epsilon} 2^{j(1+\alpha)}\norm{P_j T_{M_{\epsilon,z,j}^x}\P_1^z\bm{g}}_{L^\infty} 
&\le \sup_{j\ge j_\epsilon} 2^{j(1+\alpha)}\norm{M_{\epsilon,z,j}^x}_{L^1}\norm{\wt P_j\P_1^z\bm{g}}_{L^\infty}\\
&\le c\sup_{j\ge j_\epsilon} 2^{j(1+\alpha)}\norm{\wt P_j\P_1^z\bm{g}}_{L^\infty}
\le c\,\abs{\bm{g}}_{(\dot B^{1+\alpha}_{\infty,\infty})_s} \,.
\end{align*}

Returning to the definitions \eqref{eq:Aeps_components} of $\overline{\mc{A}}_\epsilon^\epsilon$ and $\overline{\mc{A}}_\epsilon^+$, we may combine the above estimates and use the Besov characterization \eqref{eq:besov} of $C^{0,\alpha}_s$ to obtain 
\begin{align*}
\norm{\overline{\mc{A}}_\epsilon^\epsilon\bm{g}}_{C^{0,\alpha}(\T)} &\le \norm{T_{M_{\epsilon,z,j}^z}\P_0^z\bm{g}}_{C^{0,\alpha}(\T)}+ \norm{T_{M_{\epsilon,r,j}^x}\P_1^r\bm{g}}_{C^{0,\alpha}(\T)} 
+ \norm{T_{M_{\epsilon,\theta,j}^x}\P_1^\theta\bm{g}}_{C^{0,\alpha}(\T)} \\
&\le c\norm{\bm{g}}_{C^{0,\alpha}_s(\mc{C}_\epsilon)}+c\epsilon\norm{\bm{g}}_{C^{1,\alpha}_s(\mc{C}_\epsilon)}\,, \\
\norm{\overline{\mc{A}}_\epsilon^+\bm{g}}_{C^{1,\alpha}(\T)} &\le  \norm{T_{M_{\epsilon,r,j}^z}\P_0^r\bm{g}}_{C^{1,\alpha}(\T)}+ \norm{T_{M_{\epsilon,z,j}^x}\P_1^z\bm{g}}_{C^{1,\alpha}(\T)}  \le c\,\epsilon^{-1}\norm{\bm{g}}_{C^{1,\alpha}_s(\mc{C}_\epsilon)}\,.
\end{align*}
\hfill \qedsymbol

\subsection{Proofs of semigroup lemmas}\label{subsec:semigroup}
We may additionally use the Littlewood-Paley decomposition \eqref{eq:littlewoodp} to show the estimates \eqref{eq:semigroup} and \eqref{eq:max_semi} for the semigroup $e^{-t\overline{\mc{L}}_\epsilon\p_s^4}$ stated in Lemmas \ref{lem:semigroup} and \ref{lem:max_reg}. 

\begin{proof}[Proof of Lemma \ref{lem:semigroup}]
We will require bounds for the Littlewood--Paley projection $P_j e^{-t\overline{\mc{L}}_\epsilon\p_s^4}\bm{V}=e^{-t\overline{\mc{L}}_\epsilon\p_s^4}P_j\bm{V}$ for each $j$. Using the notation of section \ref{subsubsec:Meps}, we first recall the definition of the matrix-valued Fourier multiplier
\begin{align*}
\wh{\bm{M}}_\epsilon^{(1)}(k) = m_{\epsilon,{\rm t}}(k) \be_z\otimes\be_z + 
m_{\epsilon,{\rm n}}(k) (\be_x\otimes\be_x+\be_y\otimes\be_y)\,.
\end{align*}
At each time $t$, we may write the semigroup $e^{-t\overline{\mc{L}}_\epsilon\p_s^4}$ as a Fourier multiplier in $s$:
\begin{align*}
\mc{F}[e^{-t\overline{\mc{L}}_\epsilon\p_s^4}\bm{V}](k) &= e^{-t\,k^4\wh{\bm{M}}_\epsilon^{(1)}(k)}\mc{F}[\bm{V}](k)\,.
\end{align*}
Recalling the definition \eqref{eq:phij_def} of $\phi_j$, at each $t$ we consider \begin{align*}
\wh{\bm{L}}_{\epsilon,j}(\xi,t) = e^{-t\,\xi^4\wh{\bm{M}}_{\epsilon,j}(\xi)}\,, \qquad \wh{\bm{M}}_{\epsilon,j} = (\phi_{j-1}(\xi)+ \phi_j(\xi) +\phi_{j+1}(\xi)) \wh{\bm{M}}_{\epsilon}^{(1)}(\xi)\,.
\end{align*}
Using the growth bounds of Lemmas \ref{lem:mt_bounds} and \ref{lem:mn_bounds}, for $j_\epsilon=\frac{\abs{\log(2\pi\epsilon)}}{\log(2)}$, we have 
\begin{align*}
\abs{\wh{\bm{L}}_{\epsilon,j}} &\le \begin{cases}
 e^{-c\,t\epsilon^{-1}2^{3j}} \,, & j\ge j_\epsilon\\
 e^{-c\,t\abs{\log\epsilon}2^{4j}} \,, & j<j_\epsilon\,,
\end{cases} \\
\abs{\p_\xi^2\wh{\bm{L}}_{\epsilon,j}} &\le \begin{cases}
c \big(t\epsilon^{-1}2^{j}+(t\epsilon^{-1}2^{2j})^2 \big)e^{-c\,t\epsilon^{-1}2^{3j}} \,, & j\ge j_\epsilon\\
c \big(t\abs{\log\epsilon}2^{2j}+ (t\abs{\log\epsilon} 2^{3j})^2 \big) e^{-c\,t\abs{\log\epsilon}2^{4j}} \,, & j<j_\epsilon\,.
\end{cases}
\end{align*}
Using Lemma \ref{lem:besov}, we thus obtain 
\begin{align*}
\norm{\bm{L}_{\epsilon,j}(\cdot,t)}_{L^1} \le 
\begin{cases}
c\,\big(t \epsilon^{-1} 2^{3j}(1+t \epsilon^{-1} 2^{3j})\big)^{1/2}e^{-c\,t\epsilon^{-1}2^{3j}} \,, & j\ge j_\epsilon \\
c\,\big(t\abs{\log\epsilon} 2^{4j}(1+t\abs{\log\epsilon} 2^{4j})\big)^{1/2} e^{-c\,t\abs{\log\epsilon}2^{4j}} \,, & j<j_\epsilon\,.
\end{cases}
\end{align*}

For $\bm{V}\in C^{m,\gamma}(\T)$, we use the definitions \eqref{eq:T_m_def}, \eqref{eq:littlewoodp} of $T_{\wh{\bm{M}}}$ and $P_j$ to write 
\begin{align*}
e^{-t\overline{\mc{L}}_\epsilon\p_s^4}\bm{V} &= 
\sum_{j} T_{\wh{\bm{L}}_{\epsilon,j}}P_j\bm{V}
\,.
\end{align*}
We will let $\omega=n+\alpha$ and $\nu=m+\gamma$ and use the characterization \eqref{eq:besov} of $C^{n,\alpha}(\T)$ and $C^{m,\gamma}(\T)$. Noting that $P_{\le 0}=P_0$ on $\T$, we may first estimate the lowest part of the norm \eqref{eq:besov} as 
\begin{equation}\label{eq:Vest1}
\begin{aligned}
\norm{T_{\wh{\bm{L}}_{\epsilon,0}}P_{\le 0}\bm{V}}_{L^\infty} 
&= \norm{\bm{L}_{\epsilon,0}*(P_0\bm{V})}_{L^\infty}
\le \norm{\bm{L}_{\epsilon,0}}_{L^1}\norm{\bm{V}}_{L^\infty}\\
&\le c\,\big(t\abs{\log\epsilon}(1+t\abs{\log\epsilon})\big)^{1/2} e^{-c\,t\abs{\log\epsilon}} \norm{\bm{V}}_{L^\infty}
\le c\,\norm{\bm{V}}_{L^\infty}\,.
\end{aligned}
\end{equation}
For the low frequencies $0<j<j_\epsilon$, we have 
\begin{equation}\label{eq:Vest2}
\begin{aligned}
&\sup_{0<j<j_\epsilon} 2^{j\omega}\norm{T_{\wh{\bm{L}}_{\epsilon,j}}P_j\bm{V}}_{L^\infty} = \sup_{0<j<j_\epsilon} 2^{j\omega}\norm{\bm{L}_{\epsilon,j}*(P_j\bm{V})}_{L^\infty} \\
&\qquad\le \sup_{0<j<j_\epsilon} 2^{j(\omega-\nu)}\norm{\bm{L}_{\epsilon,j}}_{L^1}2^{j\nu}\norm{P_j\bm{V}}_{L^\infty} \\
&\qquad \le \sup_{0<j<j_\epsilon} 2^{j(\omega-\nu)}\bigg(\big(t\abs{\log\epsilon} 2^{4j}(1+t\abs{\log\epsilon} 2^{4j})\big)^{1/2}e^{-c\,t\abs{\log\epsilon}2^{4j}}\bigg)\,2^{j\nu}\norm{P_j\bm{V}}_{L^\infty} \\
&\qquad \le c\, t^{-\frac{(\omega-\nu)}{4}}\abs{\log\epsilon}^{-\frac{(\omega-\nu)}{4}}\abs{\bm{V}}_{\dot B^\nu_{\infty,\infty}}\,.
\end{aligned}
\end{equation}
For the high frequencies $j\ge j_\epsilon$, we may similarly estimate 
\begin{equation}\label{eq:Vest3}
\begin{aligned}
\sup_{j\ge j_\epsilon} 2^{j\omega}\norm{T_{\wh{\bm{L}}_{\epsilon,j}}P_j\bm{V}}_{L^\infty} 
&\le \sup_{j\ge j_\epsilon} 2^{j\omega}\norm{\bm{L}_{\epsilon,j}}_{L^1}\norm{P_j\bm{V}}_{L^\infty} \\
&\le \sup_{j\ge j_\epsilon} 2^{j(\omega-\nu)}
\bigg(\big(t \epsilon^{-1} 2^{3j}(1+t \epsilon^{-1} 2^{3j})\big)^{1/2}e^{-c\,t\epsilon^{-1}2^{3j}}\bigg)
 2^{j\nu}\norm{P_j\bm{V}}_{L^\infty} \\
 &\le c\,t^{-\frac{(\omega-\nu)}{3}}\epsilon^{\frac{\omega-\nu}{3}}
\abs{\bm{V}}_{\dot B^\nu_{\infty,\infty}} \,.
\end{aligned}
\end{equation}
Combining the estimates \eqref{eq:Vest1}, \eqref{eq:Vest2}, and \eqref{eq:Vest3} with the characterization \eqref{eq:besov} of $C^{n,\alpha}$ and $C^{m,\gamma}$, we obtain Lemma \ref{lem:semigroup}.
\end{proof}

We next show Lemma \ref{lem:max_reg} regarding a maximal regularity estimate for the semigroup $e^{-t\overline{\mc{L}}_\epsilon\p_s^4}$.
\begin{proof}[Proof of Lemma \ref{lem:max_reg}]
We first note that, using the characterization \eqref{eq:besov} of $C^{k,\alpha}(\T)$ and the definition \eqref{eq:Ykspace_def} of the spaces $\mc{Y}_k$, we may write $\norm{\bm{g}}_{\mc{Y}_1}$ as 
\begin{align*}
\norm{\bm{g}}_{\mc{Y}_1} &= \esssup_{t\in[0,T]}\;\sup\bigg(\norm{P_{\le 0}\bm{g}(\cdot,t)}_{L^\infty}\,,\,\sup_{j>0}\,2^{j(1+\alpha)}\norm{P_j\bm{g}(\cdot,t)}_{L^\infty}\bigg)\\
&= \sup \bigg(\esssup_{t\in[0,T]}\norm{P_{\le 0}\bm{g}(\cdot,t)}_{L^\infty}\,,\, \sup_{j>0}\,2^{j(1+\alpha)}\;\esssup_{t\in[0,T]}\norm{P_j\bm{g}(\cdot,t)}_{L^\infty}\bigg)\,.
\end{align*}
Again taking $j_\epsilon=\frac{\abs{\log(2\pi\epsilon)}}{\log(2)}$ and using Lemmas \ref{lem:mt_bounds} and \ref{lem:mn_bounds}, for $j\ge j_\epsilon$ we have 
\begin{align*}
\esssup_{t\in[0,T]}\;& 2^{j(4+\alpha)}\norm{P_j\bm{h}(\cdot,t)}_{L^\infty(\T)} \\
&= \esssup_{t\in[0,T]}\;2^{j(4+\alpha)}\norm{\int_0^t e^{-(t-t')\overline{\mc{L}}_\epsilon\p_s^4}P_j\bm{g}(\cdot,t')\,dt'}_{L^\infty(\T)} \\
& \le c\,\esssup_{t\in[0,T]}\;2^{j(4+\alpha)}\int_0^t e^{-c\,(t-t')\epsilon^{-1}2^{3j}}\norm{P_j\bm{g}(\cdot,t')}_{L^\infty(\T)}\,dt' \\
&\le c\,\bigg(\esssup_{t\in[0,T]}\;2^{3j}\int_0^t e^{-c\,(t-t')\epsilon^{-1}2^{3j}}\,dt'\bigg) \bigg(\esssup_{t\in[0,T]}2^{j(1+\alpha)}\norm{P_j\bm{g}(\cdot,t)}_{L^\infty(\T)} \bigg) \\
&\le c\,\epsilon\norm{\bm{g}}_{\mc{Y}_1}\,.
\end{align*}
In the first line we have used that $P_j$ commutes with the semigroup $e^{-t\overline{\mc{L}}_\epsilon\p_s^4}$.

For low modes with $0<j<j_\epsilon$, we have
\begin{align*}
\esssup_{t\in[0,T]}&\;2^{j(4+\alpha)}\norm{P_j\bm{h}(\cdot,t)}_{L^\infty(\T)} 
\le c\,\esssup_{t\in[0,T]}\;2^{j(4+\alpha)}\int_0^t e^{-c\,(t-t')\abs{\log\epsilon}\,2^{4j}}\norm{P_j\bm{g}(\cdot,t')}_{L^\infty(\T)}\,dt' \\
&\le c\,\bigg(\esssup_{t\in[0,T]}\;2^{3j}\int_0^t e^{-c\,(t-t')\abs{\log\epsilon}\,2^{4j}}\,dt'\bigg) \bigg(\esssup_{t\in[0,T]}2^{j(1+\alpha)}\norm{P_j\bm{g}(\cdot,t)}_{L^\infty(\T)} \bigg)\\
&\le c\,\abs{\log\epsilon}^{-3/4}\bigg(\esssup_{t\in[0,T]}\,\int_0^t (t-t')^{-3/4}\,dt'\bigg) \norm{\bm{g}}_{\mc{Y}_1}
\le c\, T^{1/4}\abs{\log\epsilon}^{-3/4}\norm{\bm{g}}_{\mc{Y}_1}\,.
\end{align*}
Finally, for the zero-mode projection $P_{\le 0}=P_0$, we have 
\begin{align*}
\esssup_{t\in[0,T]}\norm{P_0\bm{h}(\cdot,t)}_{L^\infty(\T)} 
&\le c\,\esssup_{t\in[0,T]}\int_0^t\norm{P_0\bm{g}(\cdot,t')}_{L^\infty(\T)}\,dt' %
\le c\,T\norm{\bm{g}}_{\mc{Y}_1}\,.
\end{align*}
In total, we obtain Lemma \ref{lem:max_reg}.
\end{proof}


\section{Bounds for single and double layer remainders}\label{sec:SandD_remainders}
\subsection{Setup and tools}\label{subsec:setup}
This section is devoted to the proofs of Lemmas \ref{lem:single_layer}, \ref{lem:double_layer}, and \ref{lem:single_const_in_s} regarding the mapping properties of the single and double layer remainder terms. We first require a more detailed understanding of the kernels \eqref{eq:stokeslet} and \eqref{eq:stresslet} along $\Gamma_\epsilon$ as functions of the surface parameters $s$ and $\theta$.

Throughout, we will be proving bounds for a single filament $\Sigma_\epsilon$ with centerline $\X(s)$, as well as corresponding Lipschitz estimates for two nearby filaments $\Sigma_\epsilon^{(a)}$ and $\Sigma_\epsilon^{(b)}$ with centerlines $\X^{(a)}(s)$ and $\X^{(b)}(s)$ satisfying Lemma \ref{lem:XaXb_C2beta}. When comparing two nearby filaments, we will use the superscripts $(a)$ and $(b)$ to denote quantities belonging to filament $\Sigma_\epsilon^{(a)}$ and $\Sigma_\epsilon^{(b)}$, respectively.

Given the single and double layer operators for an arbitrarily curved filament, we will be extracting the corresponding single and double layer operators about the straight filament $\mc{C}_\epsilon$. We will continue to denote quantities defined along $\mc{C}_\epsilon$ using overline notation. For $\overline{\bx},\overline{\bx}'\in \mc{C}_\epsilon$, we define 
\begin{equation}\label{eq:barR0}
\barR := \overline{\bx}-\overline{\bx}' = (s-s')\be_z+\epsilon(\be_r(\theta)-\be_r(\theta')) = \textstyle (s-s')\be_z+2\epsilon\sin(\frac{\theta-\theta'}{2})\be_\theta(\frac{\theta+\theta'}{2})\,.
\end{equation}
Defining $\bars:=s-s'$ and $\bartheta:=\theta-\theta'$, we may work in terms of $s,\bars,\theta,\bartheta$ rather than $s,s',\theta,\theta'$, and write
\begin{equation}\label{eq:barR}
\barR = \textstyle \bars\be_z+2\epsilon\sin(\frac{\bartheta}{2})\be_\theta(\theta-\frac{\bartheta}{2})\,, \qquad \bars=s-s'\,,\; \bartheta=\theta-\theta'\,.
\end{equation}
Noting that $\overline{\bm{n}}(\overline{\bx})=\be_r(\theta)$ and $\overline{\bm{n}}(\overline{\bx}')=\be_r(\theta-\bartheta)$, we may write
\begin{align*}
|\barR(\bars,\bartheta)|^2 &=\textstyle \bars^2+4\epsilon^2\sin^2(\frac{\bartheta}{2})\,,\\
\barR\cdot\overline{\bm{n}} &= \textstyle - \barR\cdot\overline{\bm{n}}' = 2\epsilon\sin^2(\frac{\bartheta}{2})\,.
\end{align*}

Along a curved filament $\Sigma_\epsilon$, we similarly consider
\begin{align*}
\bR:=\bx-\bx'
\end{align*}
as a function of $(s,\bars,\theta,\bartheta)$. Using the orthonormal frame ODEs \eqref{eq:frame}, we note the following expansions with respect to arclength $s$:
\begin{equation}\label{eq:s_expand}
\begin{aligned}
\X(s)-\X(s-\bars) &= \bars\be_{\rm t}(s) - \bars^2\bm{Q}_{\rm t}(s,\bars) \\
\be_{\rm t}(s) - \be_{\rm t}(s-\bars) &= \bars\bm{Q}_{\rm t}(s,\bars)\\
\be_{\rm n_1}(s)-\be_{\rm n_1}(s-\bars) &= \bars(-\kappa_1(s)\be_{\rm t}(s) + \kappa_3\be_{\rm n_2}(s)) + \bars^2\bm{Q}_{\rm n_1}(s,\bars) \\
\be_{\rm n_2}(s)-\be_{\rm n_2}(s-\bars) &= \bars(-\kappa_2(s)\be_{\rm t}(s) - \kappa_3\be_{\rm n_1}(s)) + \bars^2\bm{Q}_{\rm n_2}(s,\bars) \,.
\end{aligned}
\end{equation}
Given two filaments with centerlines $\X^{(a)}$, $\X^{(b)}$ in $C^{2,\beta}(\T)$, each of the remainder terms $\bm{Q}_i$ in \eqref{eq:s_expand} satisfies 
\begin{equation}\label{eq:Qdiff_bds}
\norm{\bm{Q}_i^{(a)}\cdot\be_j^{(a)}-\bm{Q}_i^{(b)}\cdot\be_j^{(b)}}_Y 
\le c\,\sum_{\ell=1}^3\norm{\kappa_\ell^{(a)}-\kappa_\ell^{(b)}}_Y \,, \qquad i,j\in \{{\rm t},{\rm n_1},{\rm n_2}\}\,,
\end{equation}
where the norms $\norm{\cdot}_Y$ are defined for functions of both $(s,\theta)$ and $(\bars,\bartheta)\in \T\times 2\pi\T$ by any of the following: 
\begin{equation}\label{eq:Y_norms}
\begin{aligned}
\norm{Q}_{L^\infty}&:=\esssup_{s,\theta,\bars,\bartheta}\abs{Q(s,\theta,\bars,\bartheta)}, \quad
\norm{Q}_{C^{0,\beta}_1}:=\esssup_{\bars,\bartheta}\norm{Q(\cdot,\cdot,\bars,\bartheta)}_{C^{0,\beta}}\,,\\
\norm{Q}_{C^{0,\beta}_2}&:=\esssup_{s,\theta}\norm{Q(s,\theta,\cdot,\cdot)}_{C^{0,\beta}}\,.
\end{aligned}
\end{equation}
Note that for $Q=Q(s,\theta)$ we have $\norm{Q}_{C^{0,\beta}_1}=\norm{Q}_{C^{0,\beta}}$, and for $Q=Q(\bars,\bartheta)$ we have $\norm{Q}_{C^{0,\beta}_2}=\norm{Q}_{C^{0,\beta}}$.
For some combinations of $\bm{Q}_i\cdot\be_j$, we note the following more refined expansions:
\begin{equation}\label{eq:refined_Qdiff}
\begin{aligned}
\bm{Q}_{\rm t}(s,\bars)\cdot\be_{\rm t}(s)&=\bars Q_{\rm t}(s,\bars)\,, \\ 
\bm{Q}_{{\rm n}_j}(s,\bars)\cdot\be_{\rm t}(s) + \bm{Q}_{\rm t}(s,\bars)\cdot\be_{{\rm n}_j}(s) &= \bars Q_{{\rm tn}_j}(s,\bars)\,, \quad j=1,2\,; \\
\norm{Q_i^{(a)}-Q_i^{(b)}}_Y &\le c\,\sum_{\ell=1}^3\norm{\kappa_\ell^{(a)}-\kappa_\ell^{(b)}}_Y \,, \quad i\in \{{\rm t},{\rm tn_1},{\rm tn_2}\}\,.
\end{aligned}
\end{equation}
Using \eqref{eq:s_expand} and the definition \eqref{eq:er_etheta} of $\be_r(s,\theta)$ and $\be_\theta(s,\theta)$, we may write 
\begin{equation}\label{eq:QrQtheta}
\be_r(s,\theta)-\be_r(s-\bars,\theta) = \bars\bm{Q}_r(s,\bars,\theta)\,, \qquad
\be_\theta(s,\theta)-\be_\theta(s-\bars,\theta) = \bars\bm{Q}_\theta(s,\bars,\theta) \,,
\end{equation}
where, using \eqref{eq:Qdiff_bds} and \eqref{eq:refined_Qdiff}, $\bm{Q}_r$ and $\bm{Q}_\theta$ satisfy
\begin{equation}\label{eq:Qexpand}
\begin{aligned}
\be_{\rm t}(s)\cdot\bm{Q}_r(s,\bars,\theta) &= -\wh\kappa(s,\theta) + \bars Q_{0,1}(s,\theta,\bars)  \\
\be_{\rm t}(s)\cdot\bm{Q}_\theta(s,\theta,\bars)+\bm{Q}_{\rm t}(s,\bars)\cdot\be_\theta(s,\theta) &= \bars Q_{0,2}(s,\theta,\bars)\\
\textstyle \be_\theta(s,\theta-\frac{\bartheta}{2})\cdot\bm{Q}_r(s,\bars,\theta-\bartheta)&= \textstyle \kappa_3\cos(\frac{\bartheta}{2}) + \bars Q_{0,3}(s,\theta,\bars,\bartheta) \\
\abs{\bm{Q}_r(s,\theta,\bars)}^2&= \wh\kappa(s,\theta)^2+\kappa_3^2 +\bars Q_{0,4}(s,\theta,\bars) \\
\be_r(s,\theta)\cdot\bm{Q}_r(s,\theta,\bars) &= \bars Q_{0,5}(s,\theta,\bars)\,.
\end{aligned}
\end{equation}
For each $Q_{0,j}$ above, using the norms $\norm{\cdot}_Y$ defined in \eqref{eq:Y_norms}, we have 
\begin{equation}\label{eq:Q_0j_norms}
\norm{Q_{0,j}^{(a)}-Q_{0,j}^{(b)}}_Y \le c\,\epsilon^{-\beta}\sum_{\ell=1}^3\norm{\kappa_\ell^{(a)}-\kappa_\ell^{(b)}}_Y \,.
\end{equation}
Here the $\epsilon^{-\beta}$ arises due to the $\theta$-dependence of these terms along $\Gamma_\epsilon$, and we have $\beta=0$ for the $Y=L^\infty$ estimate. 
In addition, we note that the following identities hold along $\Gamma_\epsilon$:
\begin{equation}\label{eq:er_etheta_IDs}
\begin{aligned}
\be_r(s,\theta)-\be_r(s,\theta-\bartheta) &= \textstyle 2\sin(\frac{\bartheta}{2})\be_\theta(s,\theta-\frac{\bartheta}{2}) \,, \qquad
\textstyle\be_\theta(s,\theta-\frac{\bartheta}{2})\cdot\be_r(s,\theta) = \textstyle\sin(\frac{\bartheta}{2}) \\
\bm{Q}_r(s,\theta,\bars)-\bm{Q}_r(s,\theta-\bartheta,\bars) 
&= \textstyle 2\sin(\frac{\bartheta}{2})\bm{Q}_\theta(s,\theta-\frac{\bartheta}{2},\bars)\,.
\end{aligned}
\end{equation}
Using the above expansions, we may write (see \cite[Section 3]{laplace}): 
\begin{equation}\label{eq:curvedR}
\begin{aligned}
\bR(s,\theta,\bars,\bartheta) &= \X(s) - \X(s-\bars) + \epsilon\big(\be_r(s,\theta)-\be_r(s-\bars,\theta-\bartheta) \big) \\
&= \textstyle \bars\be_{\rm t}(s) + 2\epsilon\sin(\frac{\bartheta}{2})\be_\theta(s,\theta-\frac{\bartheta}{2})-\bars^2\bm{Q}_{\rm t}(s,\bars)  + \epsilon\bars\bm{Q}_r(s,\bars,\theta-\bartheta)\,.
\end{aligned}
\end{equation}
Recalling the form \eqref{eq:barR} of $\barR$, we have 
\begin{equation}\label{eq:Rsq}
\begin{aligned}
\abs{\bm{R}}^2 
&=\textstyle
|\barR|^2 + \epsilon\bars^2Q_{R,0}(s,\theta) + 2\kappa_3\epsilon^2\bars\sin(\bartheta) +\bars^4Q_{R,1}(s,\bars) \\
&\hspace{2cm} \textstyle + \epsilon \bars^3Q_{R,2}(s,\theta,\bars,\bartheta)+ \epsilon^2\bars^2\sin(\frac{\bartheta}{2})Q_{R,3}(s,\theta,\bars,\bartheta) \,,
\end{aligned}
\end{equation}
where $Q_{R,j}$ are given by
\begin{equation}\label{eq:Q_Rj}
\begin{aligned}
Q_{R,0}(s,\theta) &= -2\wh\kappa(s,\theta)+\epsilon\wh\kappa(s,\theta)^2+\epsilon\kappa_3^2\\
Q_{R,1}(s,\bars) &= 3Q_{\rm t}(s,\bars) + \abs{\bm{Q}_{\rm t}(s,\bars)}^2 \\
Q_{R,2}(s,\theta,\bars,\bartheta) &= \textstyle 2Q_{0,1}(s,\bars,\theta) +4\sin(\frac{\bartheta}{2})Q_{0,2}(s,\bars,\theta-\frac{\bartheta}{2}) \\
&\hspace{2cm}+ 2\bm{Q}_{\rm t}(s,\bars)\cdot\bm{Q}_r(s,\bars,\theta-\bartheta)  +\epsilon Q_{0,4}(s,\bars,\theta) \\
Q_{R,3}(s,\theta,\bars,\bartheta) &= \textstyle 4Q_{0,3}(s,\theta,\bars,\bartheta)  -4\bm{Q}_\theta(s,\bars,\theta-\frac{\bartheta}{2})\cdot\bm{Q}_r(s,\bars,\theta) +4\sin(\frac{\bartheta}{2})|\bm{Q}_\theta(s,\bars,\theta-\frac{\bartheta}{2})|^2 \,.
\end{aligned}
\end{equation}
Here we emphasize that $Q_{R,0}(s,\theta)$ does not depend on $\bars$ and $\bartheta$, as this will be used in the proof of Lemma \ref{lem:new_alpha_est} later on. Using the definition \eqref{eq:Y_norms} of the norms $\norm{\cdot}_Y$ and the bounds \eqref{eq:Q_0j_norms}, the remaining $Q_{R,j}$ can be seen to satisfy 
\begin{align*}
\norm{Q_{R,1}^{(a)}-Q_{R,1}^{(b)}}_Y &\le c(\kappa_{*,\beta}^{(a)},\kappa_{*,\beta}^{(b)})\sum_{\ell=1}^3\norm{\kappa_\ell^{(a)}-\kappa_\ell^{(b)}}_Y\\
\norm{Q_{R,j}^{(a)}-Q_{R,j}^{(b)}}_Y &\le \epsilon^{-\beta} \, c(\kappa_{*,\beta}^{(a)},\kappa_{*,\beta}^{(b)})\sum_{\ell=1}^3\norm{\kappa_\ell^{(a)}-\kappa_\ell^{(b)}}_Y \,, \quad j=2,3\,.
\end{align*}

For two filaments with centerlines $\X^{(a)}$, $\X^{(b)}$ in $C^{2,\beta}(\T)$, we have  
\begin{equation}\label{eq:Ra_minus_Rb}
\begin{aligned}
\bR^{(a)}-\bR^{(b)} &= \textstyle \bars\big(\be_{\rm t}^{(a)}(s)-\be_{\rm t}^{(b)}(s)\big) + 2\epsilon\sin(\frac{\bartheta}{2})\big(\be_\theta^{(a)}(s,\theta-\frac{\bartheta}{2})-\be_\theta^{(b)}(s,\theta-\frac{\bartheta}{2})\big)\\
&\quad -\bars^2\big(\bm{Q}_{\rm t}^{(a)}(s,\bars)-\bm{Q}_{\rm t}^{(b)}(s,\bars)\big)  + \epsilon\bars\big(\bm{Q}_r^{(a)}(s,\bars,\theta-\bartheta)-\bm{Q}_r^{(b)}(s,\bars,\theta-\bartheta)\big)\,.
\end{aligned}
\end{equation}
Furthermore, we may write the differences
\begin{equation}\label{eq:RaRb_diff}
\begin{aligned}
\frac{1}{\abs{\bR}}-\frac{1}{|\barR|} &= \frac{\epsilon\bars^2Q_{R,0} +\epsilon^2\bars\sin(\bartheta)\kappa_3 +\bars^4Q_{R,1} +\epsilon\bars^3 Q_{R,2}+ \epsilon^2 \bars^2\sin(\frac{\bartheta}{2})Q_{R,3}}{\abs{\bR}|\barR|(\abs{\bR}+|\barR|)}\\
\frac{1}{\abs{\bm{R}^{(a)}}}-\frac{1}{\abs{\bm{R}^{(b)}}} &= \frac{\epsilon\bars^2(Q_{R,0}^{(a)}-Q_{R,0}^{(b)}) +\epsilon^2\bars\sin(\bartheta)(\kappa_3^{(a)}-\kappa_3^{(b)}) }{\abs{\bR^{(a)}}\abs{\bR^{(b)}}(\abs{\bR^{(a)}}+\abs{\bR^{(b)}})} \\
&\quad +\frac{\bars^4(Q_{R,1}^{(a)}-Q_{R,1}^{(b)}) +\epsilon\bars^3 (Q_{R,2}^{(a)}-Q_{R,2}^{(b)})+ \epsilon^2 \bars^2\sin(\frac{\bartheta}{2})(Q_{R,3}^{(a)}-Q_{R,3}^{(b)})}{\abs{\bR^{(a)}}\abs{\bR^{(b)}}(\abs{\bR^{(a)}}+\abs{\bR^{(b)}})}
\,.
\end{aligned}
\end{equation}
We may use \eqref{eq:RaRb_diff} in expanding differences of higher powers $k\in \N$ as
\begin{equation}\label{eq:Rdiffk}
\begin{aligned}
\frac{1}{\abs{\bR}^k}-\frac{1}{|\barR|^k} = \bigg(\frac{1}{\abs{\bR}}-\frac{1}{|\barR|}\bigg)\sum_{\ell=0}^{k-1}\frac{1}{\abs{\bR}^\ell|\barR|^{k-1-\ell}}\,,\\
\frac{1}{\abs{\bR^{(a)}}^k}-\frac{1}{\abs{\bR^{(b)}}^k} = \bigg(\frac{1}{\abs{\bm{R}^{(a)}}}-\frac{1}{\abs{\bm{R}^{(b)}}}\bigg)\sum_{\ell=0}^{k-1}\frac{1}{\abs{\bR^{(a)}}^\ell\abs{\bR^{(b)}}^{k-1-\ell}}\,.
\end{aligned}
\end{equation}
In addition, we note the following expansions for $\bR$ in the direction of the normal vectors $\bm{n}(\bx')$ and $\bm{n}(\bx)$. Noting that $\bm{n}':=\bm{n}(\bx')=\be_r(s-\bars,\theta-\bartheta)$ and $\bm{n}=\bm{n}(\bx)=\be_r(s,\theta)$, we have
\begin{equation}\label{eq:RdotN}
\begin{aligned}
\bR\cdot\bm{n}' &= -\textstyle 2\epsilon\sin^2(\frac{\bartheta}{2})+ \bars^2Q_{\rm Rn'}(s,\bars,\theta-\bartheta) 
= \barR\cdot\overline{\bm{n}}' + \bars^2Q_{\rm Rn'}\,,  \\
\bR\cdot\bm{n} &= \textstyle 2\epsilon\sin^2(\frac{\bartheta}{2})+ \bars^2Q_{\rm Rn}(s,\bars,\theta-\bartheta) 
= \barR\cdot\overline{\bm{n}} + \bars^2Q_{\rm Rn}\,.
\end{aligned}
\end{equation}
Furthermore, using \eqref{eq:Qexpand}, we have
\begin{equation}\label{eq:Rdots}
\begin{aligned}
\bR\cdot\be_{\rm t}(s) &= \bars + \epsilon\bars\wh\kappa(s,\theta) + \bars^2Q_{\rm Rt}(s,\bars,\theta,\bartheta)
= \barR\cdot\be_z + \epsilon\bars\wh\kappa(s,\theta) + \bars^2Q_{\rm Rt}\\
\bR\cdot\be_\theta(s,\theta) &= \epsilon\sin(\bartheta) + \epsilon\bars \kappa_3\cos(\theta-\bartheta) + \bars^2Q_{\rm R\theta}(s,\bars,\theta,\bartheta)\\
&= \barR\cdot\be_\theta(\theta) + \epsilon\bars \kappa_3\cos(\theta-\bartheta) + \bars^2Q_{\rm R\theta}\\
\bR\cdot\bm{\varphi} &= \textstyle \bars\be_{\rm t}(s)\cdot\bm{\varphi} + 2\epsilon\sin(\frac{\bartheta}{2})\be_\theta(s,\theta-\frac{\bartheta}{2})\cdot\bm{\varphi}-\bars^2\bm{Q}_{\rm t}\cdot\bm{\varphi}  + \epsilon\bars\bm{Q}_r\cdot\bm{\varphi}\\
&= \barR\cdot(\Phi^{-1}\bm{\varphi}) -\bars^2\bm{Q}_{\rm t}\cdot\bm{\varphi}  + \epsilon\bars\bm{Q}_r\cdot\bm{\varphi}\,.
\end{aligned}
\end{equation}
Here, given two curves $\X^{(a)}$ and $\X^{(b)}$ in $C^{2,\beta}$, each of the remainder terms appearing in \eqref{eq:RdotN} and \eqref{eq:Rdots} satisfies an estimate of the form
\begin{equation}\label{eq:Qj_Cbeta}
\norm{Q^{(a)}-Q^{(b)}}_Y
\le \epsilon^{-\beta}\, c(\kappa_{*,\beta}^{(a)},\kappa_{*,\beta}^{(b)})\sum_{\ell=1}^3\norm{\kappa_\ell^{(a)}-\kappa_\ell^{(b)}}_Y
\end{equation}
where the norms $\norm{\cdot}_Y$ are given by \eqref{eq:Y_norms} and $\beta=0$ for the $L^\infty$ bound.
Here and throughout, we note that if the two curves $\X^{(a)}$ and $\X^{(b)}$ are sufficiently close in the sense of Lemma \ref{lem:XaXb_C2beta}, then the bounds of the form \eqref{eq:Qj_Cbeta} imply that 
\begin{equation}\label{eq:Qj_Cbeta2}
\norm{Q^{(a)}-Q^{(b)}}_Y
\le \epsilon^{-\beta}\, c(\kappa_{*,\beta}^{(a)},\kappa_{*,\beta}^{(b)})\norm{\X^{(a)}-\X^{(b)}}_{C^{2,\beta}(\T)}\,.
\end{equation}

Given all of the above expansions, we now state a series of lemmas that will be used to estimate the actual boundary integral expressions appearing in the decomposition \eqref{eq:SB_DtN_decomp} and to obtain Lemmas \ref{lem:single_layer}, \ref{lem:double_layer}, and \ref{lem:single_const_in_s}. Each of the following (Lemmas \ref{lem:Rests} - \ref{lem:alpha_est}) comes directly from \cite[section 3]{laplace} and will thus be stated here without proof.

\begin{lemma}[Relating $\bR$ and $\barR$]\label{lem:Rests}
Given $\barR$, $\bR$ as in \eqref{eq:barR}, \eqref{eq:curvedR}, respectively, the following bounds hold for $\epsilon$ sufficiently small:
\begin{align}
\abs{\abs{\bR}-|\barR|} &\le \frac{\kappa_*}{2}\bars^2 + c(\kappa_*)\,\epsilon\abs{\bars} \label{eq:xest1}\\
\abs{\bR} &\ge c(\kappa_*,c_\Gamma)|\barR| \,. \label{eq:xest2}
\end{align}
\end{lemma}

In addition to Lemma \ref{lem:Rests}, we have the following series of integral estimates. 
\begin{lemma}[Integral bounds for $\bR$ and $\barR$]\label{lem:basic_est}
Given $\bR$, $\barR$ as in \eqref{eq:curvedR}, \eqref{eq:barR}, respectively, and given $0\le\alpha<1$, let $k$ be an integer satisfying $k-\alpha< 2$. We have 
\begin{equation}
\int_{-1/2}^{1/2}\int_{-\pi}^{\pi}\frac{1}{|\bR|^{k-\alpha}}\,\epsilon d\bartheta d\bars\, \le c(\kappa_*,c_\Gamma)\int_{-1/2}^{1/2}\int_{-\pi}^{\pi}\frac{1}{\abs{\overline{\bm{R}}}^{k-\alpha}}\,\epsilon d\bartheta d\bars \; \le \begin{cases}
c\,\epsilon^{2-k+\alpha}\,, & 1<k-\alpha<2\\
c\, \epsilon\,, & k-\alpha\le 1 \,.
\end{cases}
\end{equation}
\end{lemma}

\begin{lemma}[Bound for integrands with odd powers]\label{lem:odd_nm}
Let $\barR$, $\bR$ be as in \eqref{eq:barR}, \eqref{eq:curvedR} and let $\ell,k,n,m$ be nonnegative integers with $\ell+k+2=n+m$ and $\ell+k$ odd. Given $g,\varphi\in C^{0,\alpha}(\Gamma_\epsilon)$, we have 
\begin{equation}\label{eq:oddlem}
\abs{{\rm p.v.}\int_{-1/2}^{1/2}\int_{-\pi}^{\pi}\frac{\bars^\ell(\epsilon\sin(\frac{\bartheta}{2}))^k g(\bars,\bartheta) }{|\barR|^m|\bR|^n}\varphi(s-\bars,\theta-\bartheta)\, \epsilon d\bartheta d\bars} \le 
c(\kappa_*,c_\Gamma)\,\epsilon^\alpha\norm{g}_{C^{0,\alpha}}\norm{\varphi}_{C^{0,\alpha}} \,.
\end{equation}
\end{lemma}
Here ${\rm p.v.}$ is the Cauchy principal value, given along $\Gamma_\epsilon$ by 
\begin{align*}
{\rm p.v.}\int_{-1/2}^{1/2}\int_{-\pi}^{\pi} h(\bars,\bartheta)\,d\bartheta d\bars = \lim_{\delta\to 0^+}\bigg(\int_{-1/2}^{-\delta}+\int_{\delta}^{1/2} \bigg)\bigg(\int_{-\pi}^{-2\pi\delta}+\int_{2\pi\delta}^{\pi}\bigg)h(\bars,\bartheta)\,d\bartheta d\bars\,.
\end{align*}

Finally, we state the following bound for the $\dot C^{0,\alpha}$ seminorm of integral operators along $\Gamma_\epsilon$. 
\begin{lemma}[Estimating $\abs{\cdot}_{\dot C^{0,\alpha}}$ seminorms]\label{lem:alpha_est}
Let $\ell,k,m,n$ be nonnegative integers satisfying either
\begin{enumerate}
\item $\ell+k=m+n-1$ or
\item $\ell+k=m+n-2$ and $\ell+k$ is odd.
\end{enumerate} 
Given $g(s,\theta,s-\bars,\theta-\bartheta)\in C^{0,\alpha^+}(\Gamma_\epsilon\times\Gamma_\epsilon)$ for any $0<\alpha<\alpha^+<1$, for $\barR$ and $\bR$ as in \eqref{eq:barR}, \eqref{eq:curvedR}, let $K_{\ell kmn}(s,\theta,\bars,\bartheta)$ denote
\begin{equation}\label{eq:K_lkmn}
 K_{\ell kmn}(s,\theta,\bars,\bartheta):= \frac{\bars^\ell(\epsilon\sin(\frac{\bartheta}{2}))^k g(s,\theta,s-\bars,\theta-\bartheta)}{\abs{\barR(s,\theta,\bars,\bartheta)}^{m}\abs{\bR(s,\theta,\bars,\bartheta)}^n}\,.
\end{equation} 
Suppose $\varphi(s,\theta)\in C^{0,\alpha}(\Gamma_\epsilon)$. Then for any $(s_0,\theta_0)\in[-1/2,1/2]\times[-\pi,\pi]$ with $s_0^2+\theta_0^2\neq 0$ we have 
\begin{equation}\label{eq:alphalem}
\begin{aligned}
&\bigg|{\rm p.v.}\int_{-s_0-1/2}^{-s_0+1/2}\int_{-\theta_0-\pi}^{-\theta_0+\pi} K_{\ell kmn}(s_0+s,\theta_0+\theta,s_0+\bars,\theta_0+\bartheta)\varphi(s-\bars,\theta-\bartheta) \, \epsilon d\bartheta d\bars \\
&\qquad\quad -  {\rm p.v.}\int_{-1/2}^{1/2}\int_{-\pi}^{\pi}K_{\ell kmn}(s,\theta,\bars,\bartheta)\varphi(s-\bars,\theta-\bartheta) \, \epsilon d\bartheta d\bars\bigg| \\
&\hspace{2cm} \le \begin{cases}
c(\kappa_*,c_\Gamma)\,\epsilon^{1-\alpha}\norm{g}_{C^{0,\alpha}_1}\norm{\varphi}_{L^\infty}\sqrt{s_0^2+\epsilon^2\theta_0^2}^{\,\alpha} & \text{in case (1)} \\
c(\kappa_*,c_\Gamma)\big(\norm{g}_{C^{0,\alpha^+}_1}+\norm{g}_{C^{0,\alpha}_2}\big)\norm{\varphi}_{C^{0,\alpha}}\sqrt{s_0^2+\epsilon^2\theta_0^2}^{\,\alpha} & \text{in case (2)}\,, 
\end{cases}
\end{aligned}
\end{equation}
where $\norm{\cdot}_{C^{0,\alpha}_1}$ and $\norm{\cdot}_{C^{0,\alpha}_2}$ are as in \eqref{eq:Y_norms}.
\end{lemma}


\subsection{Single layer remainder terms}\label{subsec:single_layer}
Given the tools of section \ref{subsec:setup}, we may proceed to the proof of Lemma \ref{lem:single_layer} regarding the mapping properties of the single layer remainder $\mc{R}_\mc{S}$, defined in \eqref{eq:RS_RD_def}. We restate the lemma here for convenience.

\begin{lemma}[Single layer remainder]\label{lem:single_layer0}
Let $0<\alpha<\beta$ and consider a filament $\Sigma_\epsilon$ with centerline $\X(s)\in C^{2,\beta}(\T)$. Let the single layer remainder $\mc{R}_{\mc{S}}$ be as defined in \eqref{eq:RS_RD_def}, and let the map $\Phi$ be as in \eqref{eq:mapPhi_def}. Given $\bm{\varphi}\in C^{0,\alpha}(\Gamma_\epsilon)$ with $\int_\T \Phi^{-1}\bm{\varphi}(s,\theta)\,ds=0$, $\mc{R}_{\mc{S}}$ may be decomposed as 
\begin{align*}
\mc{R}_{\mc{S}}[\bm{\varphi}] = \mc{R}_{\mc{S},\epsilon}[\bm{\varphi}] + \mc{R}_{\mc{S},+}[\bm{\varphi}]
\end{align*}
where
\begin{equation}\label{eq:RS_ests1_0}
\begin{aligned}
\norm{\mc{R}_{\mc{S},\epsilon}[\bm{\varphi}]}_{C^{0,\alpha}} &\le c(\kappa_{*,\alpha},c_\Gamma)\,\epsilon^{2-\alpha}\norm{\bm{\varphi}}_{L^\infty}\,, \quad
\abs{\mc{R}_{\mc{S},\epsilon}[\bm{\varphi}]}_{\dot C_s^{1,\alpha}} \le c(\kappa_{*,\alpha^+},c_\Gamma)\,\epsilon^{1-\alpha^+}\norm{\bm{\varphi}}_{C^{0,\alpha}}\\
\norm{\mc{R}_{\mc{S},+}[\bm{\varphi}]}_{C^{0,\alpha}} &\le c(\kappa_{*,\alpha},c_\Gamma)\,\epsilon^{1-\alpha}\norm{\bm{\varphi}}_{L^\infty}\,, \quad
\abs{\mc{R}_{\mc{S},+}[\bm{\varphi}]}_{\dot C_s^{1,\beta}} \le c(\kappa_{*,\beta},c_\Gamma)\,\epsilon^{1-2\beta}\norm{\bm{\varphi}}_{C^{0,\alpha}}\,.
\end{aligned}
\end{equation}

Furthermore, given two nearby filaments with centerlines $\X^{(a)}(s)$, $\X^{(b)}(s)$ satisfying Lemma \ref{lem:XaXb_C2beta}, let $\bm{\varphi}_0^{\Phi^{(a)}}=\bm{\varphi}-\int_\T(\Phi^{(a)})^{-1}\bm{\varphi}(s,\theta)\,ds$ and $\bm{\varphi}_0^{\Phi^{(b)}}=\bm{\varphi}-\int_\T(\Phi^{(b)})^{-1}\bm{\varphi}(s,\theta)\,ds$. The difference between the corresponding single layer remainders $\mc{R}_{\mc{S}}^{(a)}$ and $\mc{R}_{\mc{S}}^{(b)}$ may be decomposed as above, and the components satisfy
\begin{equation}\label{eq:RS_ests2_0}
\begin{aligned}
\norm{\mc{R}_{\mc{S},\epsilon}^{(a)}[\bm{\varphi}_0^{\Phi^{(a)}}]-\mc{R}_{\mc{S},\epsilon}^{(b)}[\bm{\varphi}_0^{\Phi^{(b)}}]}_{C^{0,\alpha}} &\le c(\kappa_{*,\alpha}^{(a)},\kappa_{*,\alpha}^{(b)},c_\Gamma)\,\epsilon^{2-\alpha}\norm{\X^{(a)}-\X^{(b)}}_{C^{2,\alpha}}\norm{\bm{\varphi}}_{L^\infty} \\
\abs{\mc{R}_{\mc{S},\epsilon}^{(a)}[\bm{\varphi}_0^{\Phi^{(a)}}]-\mc{R}_{\mc{S},\epsilon}^{(b)}[\bm{\varphi}_0^{\Phi^{(b)}}]}_{\dot C^{1,\alpha}} &\le c(\kappa_{*,\alpha^+}^{(a)},\kappa_{*,\alpha^+}^{(b)},c_\Gamma)\,\epsilon^{1-\alpha^+}\norm{\X^{(a)}-\X^{(b)}}_{C^{2,\alpha^+}}\norm{\bm{\varphi}}_{C^{0,\alpha}}\\
\norm{\mc{R}_{\mc{S},+}^{(a)}[\bm{\varphi}_0^{\Phi^{(a)}}]-\mc{R}_{\mc{S},+}^{(b)}[\bm{\varphi}_0^{\Phi^{(b)}}]}_{C^{0,\alpha}} &\le c(\kappa_{*,\alpha}^{(a)},\kappa_{*,\alpha}^{(b)},c_\Gamma)\,\epsilon^{1-\alpha}\norm{\X^{(a)}-\X^{(b)}}_{C^{2,\alpha}}\norm{\bm{\varphi}}_{L^\infty} \\
\abs{\mc{R}_{\mc{S},+}^{(a)}[\bm{\varphi}_0^{\Phi^{(a)}}]-\mc{R}_{\mc{S},+}^{(b)}[\bm{\varphi}_0^{\Phi^{(b)}}]}_{\dot C^{1,\beta}} &\le c(\kappa_{*,\beta}^{(a)},\kappa_{*,\beta}^{(b)},c_\Gamma)\,\epsilon^{1-2\beta}\norm{\X^{(a)}-\X^{(b)}}_{C^{2,\beta}}\norm{\bm{\varphi}}_{C^{0,\alpha}}\,.\\
\end{aligned}
\end{equation}
\end{lemma}

\begin{proof}[Proof of Lemma \ref{lem:single_layer}]
Recall the definition of $\mc{R}_{\mc{S}}$ from \eqref{eq:RS_RD_def} and note that by assumption $\bm{\varphi}(s,\theta)=\bm{\varphi}_0^\Phi(s,\theta)$ since we consider $\int_\T \Phi^{-1}\bm{\varphi}(s,\theta)\,ds=0$. Rewriting each integrand in terms of the surface parameters $s,\bars,\theta,\bartheta$, we may decompose $\mc{R}_{\mc{S}}$ as follows:  
\begin{equation}\label{eq:RS_defs}
\begin{aligned}
\mc{R}_{\mc{S}}[\bm{\varphi}(s,\theta)] &= \Phi^{-1}\mc{S}[\bm{\varphi}(s,\theta)]-\overline{\mc{S}}[\Phi^{-1}\bm{\varphi}(s,\theta)] \\
&= \mc{R}_{\mc{S},0}[\bm{\varphi}] + \mc{R}_{\mc{S},1}[\bm{\varphi}] + \mc{R}_{\mc{S},2}[\bm{\varphi}]+ \mc{R}_{\mc{S},3}[\bm{\varphi}]\,, \\
\mc{R}_{\mc{S},0}[\bm{\varphi}] &= -\bigg(\int_{-\infty}^{-1/2}+\int_{1/2}^\infty\bigg)\int_{-\pi}^\pi \overline{\mc{G}}(s,\bars,\theta,\bartheta)\,\Phi^{-1}\bm{\varphi}(s-\bars,\theta-\bartheta) \, \epsilon d\bartheta d\bars \\
\mc{R}_{\mc{S},1}[\bm{\varphi}] &= 
\frac{1}{8\pi}\bigg(\Phi^{-1}\int_{-1/2}^{1/2}\int_{-\pi}^{\pi} \frac{1}{\abs{\bR}}\bm{\varphi}(s-\bars,\theta-\bartheta)\,\epsilon \,d\bars\,d\bartheta \\
&\qquad\qquad - 
\int_{-1/2}^{1/2}\int_{-\pi}^{\pi} \frac{1}{|\barR|}(\Phi^{-1}\bm{\varphi})(s-\bars,\theta-\bartheta)\,\epsilon\, d\bars\,d\bartheta\bigg) \\
\mc{R}_{\mc{S},2}[\bm{\varphi}] &= 
\frac{1}{8\pi}\bigg(\Phi^{-1}\int_{-1/2}^{1/2}\int_{-\pi}^{\pi} \frac{\bR\otimes\bR}{\abs{\bR}^3}\bm{\varphi}(s-\bars,\theta-\bartheta)\,\epsilon \,d\bars\,d\bartheta \\
&\qquad\qquad - 
\int_{-1/2}^{1/2}\int_{-\pi}^{\pi} \frac{\barR\otimes\barR}{|\barR|^3}(\Phi^{-1}\bm{\varphi})(s-\bars,\theta-\bartheta)\,\epsilon\, d\bars\,d\bartheta\bigg) \\
\mc{R}_{\mc{S},3}[\bm{\varphi}] &= -\Phi^{-1}\int_{-1/2}^{1/2}\int_{-\pi}^\pi \mc{G}(s,\bars,\theta,\bartheta)\bm{\varphi}(s-\bars,\theta-\bartheta)\, \epsilon^2\wh\kappa\,d\bartheta d\bars\,.
\end{aligned}
\end{equation}

We begin with bounds for the smooth remainder $\mc{R}_{\mc{S},0}$ away from the singularity at $(\bars,\bartheta)=(0,0)$. Since $\int_\T \Phi^{-1}\bm{\varphi}(s,\theta)\,ds=0$, we may write $\Phi^{-1}\bm{\varphi}(s,\theta)=\p_s\overline{\bm{\varphi}}(s,\theta)$ for some $\overline{\bm{\varphi}}$ with $\overline{\bm{\varphi}}(0,\theta)=0$. 
Using $\p_s\overline{\bm{\varphi}}(s,\theta)$ in the expression for $\mc{R}_{\mc{S},0}$, we may then integrate by parts in $\bars$ to obtain 
\begin{align*}
\mc{R}_{\mc{S},0}[\bm{\varphi}] &= -\bigg(\int_{-\infty}^{-1/2}+\int_{1/2}^\infty\bigg)\int_{-\pi}^\pi \big(\p_{\bars}\,\overline{\mc{G}}\big)\;\overline{\bm{\varphi}}(s-\bars,\theta-\bartheta) \, \epsilon d\bartheta d\bars \\
&\qquad + \int_{-\pi}^\pi \overline{\mc{G}}\;\overline{\bm{\varphi}}(s-\bars,\theta-\bartheta) \, \epsilon d\bartheta d\bars \bigg|_{\bars=-1/2}^{\bars=1/2}\,.
\end{align*}
Using the form \eqref{eq:stokeslet} of the Stokeslet kernel $\mc{G}$, we have
\begin{align*}
\abs{\mc{R}_{\mc{S},0}} \le c\,\epsilon\norm{\overline{\bm{\varphi}}}_{L^\infty}\int_{1/2}^\infty \frac{1}{\bars^2}\,d\bars + c\,\epsilon\norm{\overline{\bm{\varphi}}}_{L^\infty} 
\le c\,\epsilon\norm{\overline{\bm{\varphi}}}_{L^\infty} 
= c\,\epsilon\norm{\int_0^s(\Phi^{-1}\bm{\varphi})\,ds'}_{L^\infty} 
\le c\,\epsilon\norm{\bm{\varphi}}_{L^\infty} \,.
\end{align*}
Furthermore, we may estimate 
\begin{align*}
\norm{\p_s^2\mc{R}_{\mc{S},0}}_{L^\infty} \le c\,\epsilon\norm{\bm{\varphi}}_{L^\infty}\int_{1/2}^\infty \bigg(\frac{1}{\abs{\bars}^3} + \frac{\epsilon}{\abs{\bars}^4}+ \frac{\epsilon^2}{\abs{\bars}^5}\bigg)\,d\bars 
\le c\,\epsilon\norm{\bm{\varphi}}_{L^\infty} \,.
\end{align*}
Finally, given two nearby curves $\X^{(a)}$ and $\X^{(b)}$ satisfying Lemma \ref{lem:XaXb_C2beta}, the corresponding remainder pieces $\mc{R}_{\mc{S},0}^{(a)}$ and $\mc{R}_{\mc{S},0}^{(b)}$ are nearly identical except for slight differences in the $s$-means $\int_\T(\Phi^{(a)})^{-1}\bm{\varphi}\,ds$ and $\int_\T(\Phi^{(b)})^{-1}\bm{\varphi}\,ds$. We may write
\begin{align*}
&\abs{\mc{R}_{\mc{S},0}^{(a)}[\bm{\varphi}_0^{\Phi^{(a)}}]-\mc{R}_{\mc{S},0}^{(b)}[\bm{\varphi}_0^{\Phi^{(b)}}]}\\
&\qquad = \abs{\bigg(\int_{-\infty}^{-1/2}+\int_{1/2}^\infty\bigg)\int_{-\pi}^\pi \overline{\mc{G}}\,((\Phi^{(a)})^{-1}\bm{\varphi}_0^{\Phi^{(a)}}-(\Phi^{(b)})^{-1}\bm{\varphi}_0^{\Phi^{(b)}}) \, \epsilon d\bartheta d\bars}\\
&\qquad\le \abs{\bigg(\int_{-\infty}^{-1/2}+\int_{1/2}^\infty\bigg)\int_{-\pi}^\pi \big(\p_{\bars}\,\overline{\mc{G}}\big)\;(\overline{\bm{\varphi}}^{(a)}-\overline{\bm{\varphi}}^{(b)}) \, \epsilon d\bartheta d\bars} \\
&\qquad \qquad + \abs{\int_{-\pi}^\pi \overline{\mc{G}}\;(\overline{\bm{\varphi}}^{(a)}-\overline{\bm{\varphi}}^{(b)}) \, \epsilon d\bartheta d\bars \bigg|_{\bars=-1/2}^{\bars=1/2}}\\
&\qquad \le c\,\epsilon\norm{\overline{\bm{\varphi}}^{(a)}-\overline{\bm{\varphi}}^{(b)}}_{L^\infty} 
\le c\,\epsilon\norm{\int_0^s\big((\Phi^{(a)})^{-1}\bm{\varphi}_0^{\Phi^{(a)}}-(\Phi^{(b)})^{-1}\bm{\varphi}_0^{\Phi^{(b)}}\big)\,ds'}_{L^\infty}  \\
&\qquad \le c\,\epsilon\norm{(\Phi^{(a)})^{-1}\bm{\varphi}_0^{\Phi^{(a)}}-(\Phi^{(b)})^{-1}\bm{\varphi}_0^{\Phi^{(b)}}}_{L^\infty}
\le c(\kappa_*^{(a)},\kappa_*^{(b)})\,\epsilon \norm{\X^{(a)}-\X^{(b)}}_{C^2}\norm{\bm{\varphi}}_{L^\infty}\,,
\end{align*}
by the bound \eqref{eq:Phi_lip_est}. By a similar calculation, we additionally obtain 
\begin{align*}
\norm{\p_s^2\mc{R}_{\mc{S},0}^{(a)}[\bm{\varphi}_0^{\Phi^{(a)}}]-\p_s^2\mc{R}_{\mc{S},0}^{(b)}[\bm{\varphi}_0^{\Phi^{(b)}}]}_{L^\infty} &\le c\,\epsilon\norm{(\Phi^{(a)})^{-1}\bm{\varphi}_0^{\Phi^{(a)}}-(\Phi^{(b)})^{-1}\bm{\varphi}_0^{\Phi^{(b)}}}_{L^\infty}\\
&\le c(\kappa_*^{(a)},\kappa_*^{(b)})\,\epsilon \norm{\X^{(a)}-\X^{(b)}}_{C^2}\,\norm{\bm{\varphi}}_{L^\infty}\,.
\end{align*}
In summary, the remainder term $\mc{R}_{\mc{S},0}$ satisfies the estimates 
\begin{equation}\label{eq:RS0_bds}
\begin{aligned}
\norm{\mc{R}_{\mc{S},0}[\bm{\varphi}]}_{C^2} &\le c\,\epsilon\norm{\bm{\varphi}}_{L^\infty} \\
\norm{\mc{R}_{\mc{S},0}^{(a)}[\bm{\varphi}_0^{\Phi^{(a)}}]-\mc{R}_{\mc{S},0}^{(b)}[\bm{\varphi}_0^{\Phi^{(b)}}]}_{C^2} 
&\le c(\kappa_*^{(a)},\kappa_*^{(b)})\,\epsilon \norm{\X^{(a)}-\X^{(b)}}_{C^2}\,\norm{\bm{\varphi}}_{L^\infty}\,.
\end{aligned}
\end{equation}

We now turn to the remainders $\mc{R}_{\mc{S},1}$ and $\mc{R}_{\mc{S},2}$, which both contain some components which are small in $\epsilon$ and some components which are more regular. To separate these components, we define an auxiliary function
\begin{equation}\label{eq:Rtdef}
\bR_{\rm t}(s,\bars,\theta,\bartheta) = \bars\be_{\rm t}(s) + \epsilon\big(\be_r(s,\theta)-\be_r(s-\bars,\theta-\bartheta)\big)\,, 
\end{equation}
as in \cite[Section 3.2]{laplace}, and note that 
\begin{equation}\label{eq:Rtsq}
\begin{aligned}
|\bR_{\rm t}|^2 
&= \textstyle
|\barR|^2 + \epsilon\bars^2Q_{S,1} + \epsilon^2\bars\sin(\frac{\bartheta}{2})Q_{S,2} \\
|\bR|^2 
&= \textstyle
|\bR_{\rm t}|^2 + \bars^3 Q_{S,3} + \epsilon \bars^2\sin(\frac{\bartheta}{2}) Q_{S,4}\,,
\end{aligned}
\end{equation}
where the terms $Q_{S,j}(s,\theta,\bars,\bartheta)$ are given by 
\begin{align*}
Q_{S,1} &= \textstyle -2\wh\kappa(s,\theta)+\epsilon\wh\kappa(s,\theta)^2+\epsilon\kappa_3^2 -4\sin(\frac{\bartheta}{2})\be_{\rm t}(s)\cdot\bm{Q}_\theta(s,\bars,\theta-\frac{\bartheta}{2}) \\
&\qquad +  \textstyle 2\bars Q_{0,1}(s,\bars,\theta) +\epsilon\bars Q_{0,4}(s,\bars,\theta) \\
Q_{S,2} &=\textstyle 4\kappa_3\cos(\frac{\bartheta}{2})+4\bars Q_{0,3}(s,\theta,\bars,\bartheta) -4\bars \bm{Q}_\theta\cdot\bm{Q}_r(s,\bars,\theta)+4\bars \sin(\frac{\bartheta}{2})\abs{\bm{Q}_\theta}^2\\
Q_{S,3} &= \textstyle 2\be_{\rm t}\cdot\bm{Q}_{\rm t} +2\epsilon \bm{Q}_{\rm t}(s,\bars)\cdot\bm{Q}_r(s,\bars,\theta-\bartheta) +\bars\abs{\bm{Q}_{\rm t}}^2 \\
Q_{S,4} &= \textstyle - 4\bm{Q}_{\rm t}(s,\bars)\cdot\be_\theta(s,\theta-\frac{\bartheta}{2})
\end{align*}
for $Q_{0,j}$ as in \eqref{eq:Qexpand}.
In addition, each $Q_{S,j}$ satisfies an estimate of the form \eqref{eq:Qj_Cbeta}.
We emphasize that the difference between $\bR_{\rm t}$ and $\bR$ is higher order in $\bars$ and the difference between $\bR_{\rm t}$ and $\barR$ is higher order in $\epsilon$. 
Furthermore, $\bR_{\rm t}$ may be shown to satisfy the same estimates as $\bR$ in each of Lemmas \ref{lem:Rests}, \ref{lem:basic_est}, \ref{lem:odd_nm}, and \ref{lem:alpha_est} (see \cite{laplace}).

For integers $k\ge 1$, we additionally note the following expansions:
\begin{equation}\label{eq:Rt_inv_diff}
\begin{aligned}
\frac{1}{|\bR|^k}-\frac{1}{|\bR_{\rm t}|^k} &= \frac{\bars^3 Q_{S,3} + \epsilon \bars^2\sin(\frac{\bartheta}{2})Q_{S,4}}{\abs{\bR}|\bR_{\rm t}|(\abs{\bR}+|\bR_{\rm t}|)}\sum_{\ell=0}^{k-1}\frac{1}{|\bR|^\ell|\bR_{\rm t}|^{k-1-\ell}}\\
\frac{1}{|\bR_{\rm t}|^k}-\frac{1}{|\barR|^k} &= \frac{\epsilon\bars^2Q_{S,1} + \epsilon^2\bars\sin(\frac{\bartheta}{2})Q_{S,2}}{|\bR_{\rm t}||\barR|(|\bR_{\rm t}|+|\barR|)}\sum_{\ell=0}^{k-1}\frac{1}{|\bR_{\rm t}|^\ell|\barR|^{k-1-\ell}}\\
\frac{1}{\big|\bR_{\rm t}^{(a)}\big|^k}-\frac{1}{\big|\bR_{\rm t}^{(b)}\big|^k} &= \frac{\epsilon\bars^2(Q_{S,1}^{(a)}-Q_{S,1}^{(b)}) + \epsilon^2\bars\sin(\frac{\bartheta}{2})(Q_{S,2}^{(a)}-Q_{S,2}^{(b)})}{\big|\bR_{\rm t}^{(a)}\big| \big|\bR_{\rm t}^{(b)}\big|\big(\big|\bR_{\rm t}^{(a)}\big|+ \big|\bR_{\rm t}^{(b)}\big|\big)}\sum_{\ell=0}^{k-1}\frac{1}{\big|\bR_{\rm t}^{(a)}\big|^\ell\big|\bR_{\rm t}^{(b)}\big|^{k-1-\ell}}\,.
\end{aligned}
\end{equation}
Here, as usual, the superscripts $(a)$ and $(b)$ are used to denote quantities corresponding to two different but nearby filament centerlines $\X^{(a)}$ and $\X^{(b)}$, respectively, satisfying Lemma \ref{lem:XaXb_C2beta}.

Using $\bR_{\rm t}$, we further subdivide $\mc{R}_{\mc{S},1}$ as 
\begin{align*}
\mc{R}_{\mc{S},1}[\bm{\varphi}] &= \frac{1}{8\pi} \big(\mc{R}_{\mc{S},1,1}[\bm{\varphi}] + \mc{R}_{\mc{S},1,2}[\bm{\varphi}] \big)\,, \\
\mc{R}_{\mc{S},1,1}[\bm{\varphi}] &= \Phi^{-1}\int_{-1/2}^{1/2}\int_{-\pi}^\pi \bigg(\frac{1}{\abs{\bR}} - \frac{1}{|\bR_{\rm t}|} \bigg) \bm{\varphi}(s-\bars,\theta-\bartheta)\,\epsilon \,d\bartheta d\bars \\
\overline{\be}_j\cdot\mc{R}_{\mc{S},1,2}[\bm{\varphi}] &= \int_{-1/2}^{1/2}\int_{-\pi}^\pi\bigg(\frac{1}{|\bR_{\rm t}|}-\frac{1}{|\barR|}\bigg)\be_j\cdot\bm{\varphi}(s-\bars,\theta-\bartheta)\,\epsilon \,d\bartheta d\bars\,.
\end{align*}
Here we take the subscript $j\in \{{\rm t},r,\theta\}$ and use $\overline\be_j$ to denote the corresponding basis vector with respect to the straight frame about $\mc{C}_\epsilon$. Note that in this notation, $\overline\be_{\rm t}=\be_z$. 
We additionally subdivide $\mc{R}_{\mc{S},2}$ as
\begin{align*}
\mc{R}_{\mc{S},2}[\bm{\varphi}] &= \frac{1}{8\pi} \big(\mc{R}_{\mc{S},2,1}[\bm{\varphi}] + \mc{R}_{\mc{S},2,2}[\bm{\varphi}] \big)\,, \\
\overline{\be}_j\cdot\mc{R}_{\mc{S},2,1}[\bm{\varphi}] &= \int_{-1/2}^{1/2}\int_{-\pi}^\pi \bigg(\frac{(\be_j\cdot\bR)(\bR\cdot\bm{\varphi})}{|\bR|^3} - \frac{(\be_j\cdot\bR_{\rm t})(\bR_{\rm t}\cdot\bm{\varphi})}{|\bR_{\rm t}|^3} \bigg) \,\epsilon \,d\bartheta d\bars \\
\overline{\be}_j\cdot\mc{R}_{\mc{S},2,2}[\bm{\varphi}] &= \int_{-1/2}^{1/2}\int_{-\pi}^\pi \bigg(\frac{(\be_j\cdot\bR_{\rm t})(\bR_{\rm t}\cdot\bm{\varphi})}{|\bR_{\rm t}|^3} - \frac{(\overline{\be}_j\cdot\barR)(\barR\cdot(\Phi^{-1}\bm{\varphi}))}{|\barR|^3} \bigg) \,\epsilon \,d\bartheta d\bars\,.
\end{align*}
Using the above expressions along with the forms of $\barR$ \eqref{eq:barR}, $\bR$ \eqref{eq:curvedR}, $\bR_{\rm t}$ \eqref{eq:Rtdef}, and the differences \eqref{eq:RaRb_diff}, \eqref{eq:Rdiffk}, and \eqref{eq:Rt_inv_diff}, by Lemma \ref{lem:basic_est} we have 
\begin{equation}\label{eq:RS_lower_est1}
\begin{aligned}
\abs{\mc{R}_{\mc{S},1,1}[\bm{\varphi}]} &\le c(\kappa_*)\,\norm{\bm{\varphi}}_{L^\infty}\int_{-1/2}^{1/2}\int_{-\pi}^{\pi}\,\epsilon\,d\bartheta d\bars \le c(\kappa_*)\,\epsilon \norm{\bm{\varphi}}_{L^\infty}\\
\abs{\mc{R}_{\mc{S},1,2}[\bm{\varphi}]} &\le c(\kappa_*,c_\Gamma)\,\norm{\bm{\varphi}}_{L^\infty}\int_{-1/2}^{1/2}\int_{-\pi}^{\pi}\frac{\epsilon}{|\barR|}\,\epsilon\,d\bartheta d\bars \le c(\kappa_*,c_\Gamma)\,\epsilon^2\norm{\bm{\varphi}}_{L^\infty}\\
\abs{\mc{R}_{\mc{S},2,1}[\bm{\varphi}]} &\le c(\kappa_*)\,\norm{\bm{\varphi}}_{L^\infty}\int_{-1/2}^{1/2}\int_{-\pi}^{\pi}\,\epsilon\,d\bartheta d\bars \le c(\kappa_*)\,\epsilon \norm{\bm{\varphi}}_{L^\infty}\\
\abs{\mc{R}_{\mc{S},2,2}[\bm{\varphi}]} &\le c(\kappa_*,c_\Gamma)\,\norm{\bm{\varphi}}_{L^\infty}\int_{-1/2}^{1/2}\int_{-\pi}^{\pi}\frac{\epsilon}{|\barR|}\,\epsilon\,d\bartheta d\bars \le c(\kappa_*,c_\Gamma)\,\epsilon^2\norm{\bm{\varphi}}_{L^\infty}\\
\abs{\mc{R}_{\mc{S},3}[\bm{\varphi}]} &\le c(\kappa_*)\,\norm{\bm{\varphi}}_{L^\infty}\int_{-1/2}^{1/2}\int_{-\pi}^{\pi}\frac{1}{\abs{\bR}}\,\epsilon^2\,d\bartheta d\bars \le c(\kappa_*,c_\Gamma)\,\epsilon^2 \norm{\bm{\varphi}}_{L^\infty}  \,.
\end{aligned}
\end{equation}
In addition, using case (1) of Lemma \ref{lem:alpha_est}, we may estimate 
\begin{equation}\label{eq:RS_lower_est2}
\begin{aligned}
\abs{\mc{R}_{\mc{S},1,1}[\bm{\varphi}]}_{\dot C^{0,\alpha}} &\le c(\kappa_{*,\alpha},c_\Gamma)\,\epsilon^{1-\alpha} \norm{\bm{\varphi}}_{L^\infty} \\
\abs{\mc{R}_{\mc{S},1,2}[\bm{\varphi}]}_{\dot C^{0,\alpha}} &\le c(\kappa_{*,\alpha},c_\Gamma)\,\epsilon^{2-\alpha} \norm{\bm{\varphi}}_{L^\infty} \\
\abs{\mc{R}_{\mc{S},2,1}[\bm{\varphi}]}_{\dot C^{0,\alpha}} &\le c(\kappa_{*,\alpha},c_\Gamma)\,\epsilon^{1-\alpha} \norm{\bm{\varphi}}_{L^\infty} \\
\abs{\mc{R}_{\mc{S},2,2}[\bm{\varphi}]}_{\dot C^{0,\alpha}} &\le c(\kappa_{*,\alpha},c_\Gamma)\,\epsilon^{2-\alpha} \norm{\bm{\varphi}}_{L^\infty} \\
\abs{\mc{R}_{\mc{S},3}[\bm{\varphi}]}_{\dot C^{0,\alpha}} &\le c(\kappa_{*,\alpha},c_\Gamma)\,\epsilon^{2-\alpha} \norm{\bm{\varphi}}_{L^\infty} \,.
\end{aligned}
\end{equation}
Using that the remainder coefficients in the expansions \eqref{eq:curvedR}, \eqref{eq:Rtsq}, \eqref{eq:Rtsq}, and \eqref{eq:Rt_inv_diff} each satisfy a Lipschitz bound of the form \eqref{eq:Qj_Cbeta}, and using the bounds \eqref{eq:Phi_lip_est}, given two curves $\X^{(a)}$ and $\X^{(b)}$ satisfying Lemma \ref{lem:XaXb_C2beta}, the forms of the above kernels also yield 
\begin{equation}\label{eq:RS_lower_est3}
\begin{aligned}
\norm{\mc{R}_{\mc{S},1,1}^{(a)}[\bm{\varphi}_0^{\Phi^{(a)}}]-\mc{R}_{\mc{S},1,1}^{(b)}[\bm{\varphi}_0^{\Phi^{(b)}}]}_{\dot C^{0,\alpha}} &\le c(\kappa_{*,\alpha}^{(a)},\kappa_{*,\alpha}^{(b)},c_\Gamma)\,\epsilon^{1-\alpha}\norm{\X^{(a)}-\X^{(b)}}_{C^{2,\alpha}} \norm{\bm{\varphi}}_{L^\infty} \\
\norm{\mc{R}_{\mc{S},1,2}^{(a)}[\bm{\varphi}_0^{\Phi^{(a)}}]-\mc{R}_{\mc{S},1,2}^{(b)}[\bm{\varphi}_0^{\Phi^{(b)}}]}_{\dot C^{0,\alpha}} &\le c(\kappa_{*,\alpha}^{(a)},\kappa_{*,\alpha}^{(b)},c_\Gamma)\,\epsilon^{2-\alpha}\norm{\X^{(a)}-\X^{(b)}}_{C^{2,\alpha}} \norm{\bm{\varphi}}_{L^\infty} \\
\norm{\mc{R}_{\mc{S},2,1}^{(a)}[\bm{\varphi}_0^{\Phi^{(a)}}]-\mc{R}_{\mc{S},2,1}^{(b)}[\bm{\varphi}_0^{\Phi^{(b)}}]}_{\dot C^{0,\alpha}} &\le c(\kappa_{*,\alpha}^{(a)},\kappa_{*,\alpha}^{(b)},c_\Gamma)\,\epsilon^{1-\alpha}\norm{\X^{(a)}-\X^{(b)}}_{C^{2,\alpha}} \norm{\bm{\varphi}}_{L^\infty} \\
\norm{\mc{R}_{\mc{S},2,2}^{(a)}[\bm{\varphi}_0^{\Phi^{(a)}}]-\mc{R}_{\mc{S},2,2}^{(b)}[\bm{\varphi}_0^{\Phi^{(b)}}]}_{\dot C^{0,\alpha}} &\le c(\kappa_{*,\alpha}^{(a)},\kappa_{*,\alpha}^{(b)},c_\Gamma)\,\epsilon^{2-\alpha} \norm{\X^{(a)}-\X^{(b)}}_{C^{2,\alpha}}\norm{\bm{\varphi}}_{L^\infty} \\
\norm{\mc{R}_{\mc{S},3}^{(a)}[\bm{\varphi}_0^{\Phi^{(a)}}]-\mc{R}_{\mc{S},3}^{(b)}[\bm{\varphi}_0^{\Phi^{(b)}}]}_{\dot C^{0,\alpha}} &\le c(\kappa_{*,\alpha}^{(a)},\kappa_{*,\alpha}^{(b)},c_\Gamma)\,\epsilon^{2-\alpha}\norm{\X^{(a)}-\X^{(b)}}_{C^{2,\alpha}} \norm{\bm{\varphi}}_{L^\infty} \,.
\end{aligned}
\end{equation}

We now turn to bounds for the derivatives of $\mc{R}_{\mc{S},i}$, $i=1,2,3$, along the surface $\Gamma_\epsilon$. We begin with $s$-derivatives and note the expansions 
\begin{equation}\label{eq:psR}
\begin{aligned}
\p_s\bR &= (1-\epsilon\wh\kappa(s,\theta))\be_{\rm t}(s) + \epsilon\kappa_3\be_\theta(s,\theta)\\
\p_s\bR_{\rm t} &= (1-\epsilon\wh\kappa(s,\theta))\be_{\rm t}(s) +\bars(\kappa_1\be_{\rm n_1}(s)+\kappa_2\be_{\rm n_2}(s)) + \epsilon\kappa_3\be_\theta(s,\theta)\\
\p_s\bR\cdot\bR &= \textstyle \bars+2\epsilon^2\sin(\frac{\bartheta}{2})\sin(\bartheta)+ \epsilon\bars Q_{\rm Rs1} + \bars^2Q_{\rm Rs2}\\
\p_s\bR \cdot\bm{n}' &= \p_s\bR\cdot\be_r(s-\bars,\theta-\bartheta) = \bars\wh{\kappa}(s,\theta) - \kappa_3\epsilon\sin(\bartheta)+\epsilon\bars Q_{\rm Rs3} \\ 
\p_s\bR \cdot\bm{n} &= \p_s\bR\cdot\be_r(s,\theta) = 0\,, \qquad 
\p_s\bR \cdot\be_{\rm t}(s) = (1-\epsilon\wh\kappa(s,\theta))\,, \qquad
\p_s\bR \cdot\be_\theta(s,\theta) = \epsilon\kappa_3 \,.
\end{aligned}
\end{equation}
Again the remainders $Q_{{\rm Rs}j}$ each satisfy estimates of the form \eqref{eq:Qj_Cbeta}. 
%
%
We may then write
\begin{equation}\label{eq:psRt_expand}
\begin{aligned}
\p_s\bR_{\rm t}\cdot\bR_{\rm t} &= \p_s\barR\cdot\barR +\epsilon\bars Q_{S,5}(s,\bars,\theta,\bartheta) -\epsilon^2\kappa_3\sin(\bartheta) \\
&= \p_s\bR\cdot\bR + \bars^2 Q_{S,6}(s,\bars,\theta,\bartheta) \,,
\end{aligned}
\end{equation}
where $Q_{S,5}$ and $Q_{S,6}$ also satisfy \eqref{eq:Qj_Cbeta}.

We begin by estimating $\p_s\mc{R}_{\mc{S},1}$, which we again consider in two pieces as $\p_s\mc{R}_{\mc{S},1,1}+\p_s\mc{R}_{\mc{S},1,2}$. Note that we will treat the commutator term $[\p_s,\Phi^{-1}]$ as part of the smoother remainder $\p_s\mc{R}_{\mc{S},1,1}$. 
For each $j={\rm t},r,\theta$, we may write 
\begin{align*}
\overline{\be}_j\cdot \p_s\mc{R}_{\mc{S},1,1} &= {\rm p.v.}\int_{-1/2}^{1/2}\int_{-\pi}^\pi S_{1,1} \,\epsilon\,d\bartheta d\bars + \frac{\overline\be_j}{8\pi}\cdot\big[\p_s,\Phi^{-1} \big]\int_{-1/2}^{1/2}\int_{-\pi}^\pi\frac{1}{\abs{\bR}}\bm{\varphi}\,\epsilon \,d\bartheta d\bars\,,\\
S_{1,1} &= \bigg(\frac{\p_s\bR_{\rm t}\cdot\bR_{\rm t}-\p_s\bR\cdot\bR}{|\bR|^3}+\bigg(\frac{1}{|\bR_{\rm t}|^3}-\frac{1}{|\bR|^3}\bigg)(\p_s\bR_{\rm t}\cdot\bR_{\rm t} ) \bigg) \be_j\cdot\bm{\varphi} \\
&= \bigg(\frac{\bars^2Q_{S,6}}{|\bR|^3}+\bigg(\frac{1}{|\bR_{\rm t}|^3}-\frac{1}{|\bR|^3}\bigg)(\bars+\epsilon\bars Q_{S,5}-\epsilon^2\kappa_3\sin(\bartheta) ) \bigg) \be_j\cdot\bm{\varphi} \,.
\end{align*}
Using \eqref{eq:Rt_inv_diff} and Lemma \ref{lem:basic_est} along with the commutator estimate \eqref{eq:Phi_commutator_est}, we have 
\begin{align*}
\abs{\p_s\mc{R}_{\mc{S},1,1}[\bm{\varphi}]} &\le c(\kappa_*)\norm{\bm{\varphi}}_{L^\infty}\int_{-1/2}^{1/2}\int_{-\pi}^\pi \frac{1}{\abs{\bR}} \,\epsilon\,d\bartheta d\bars
\le c(\kappa_*,c_\Gamma)\,\epsilon\norm{\bm{\varphi}}_{L^\infty}\,.
\end{align*}
In addition, for $\alpha<\beta<1$, by case (1) of Lemma \ref{lem:alpha_est} we have
\begin{equation}\label{eq:psRS11_Cbeta}
\begin{aligned}
\abs{\p_s\mc{R}_{\mc{S},1,1}[\bm{\varphi}]}_{\dot C^{0,\beta}} &\le c(\kappa_*,c_\Gamma)\,\epsilon^{1-\beta}\left(\max_i\norm{Q_i}_{C_1^{0,\beta}} \right)\norm{\bm{\varphi}}_{L^\infty}
\le c(\kappa_{*,\beta},c_\Gamma)\,\epsilon^{1-2\beta}\norm{\bm{\varphi}}_{L^\infty}\,.
\end{aligned}
\end{equation}
Here the $Q_i$ include the remainder coefficients $Q_{S,3}$, $Q_{S,4}$, $Q_{S,5}$, and $Q_{S,6}$. 
Furthermore, given two filaments with centerlines $\X^{(a)}$, $\X^{(b)}$ as in Lemma \ref{lem:XaXb_C2beta}, we have 
\begin{align*}
S_{1,1}^{(a)}-S_{1,1}^{(b)} &=\bigg(\frac{\bars^2(Q_{S,6}^{(a)}-Q_{S,6}^{(b)})}{|\bR^{(a)}|^3}+ \bigg(\frac{1}{|\bR^{(a)}|^3}-\frac{1}{|\bR^{(b)}|^3}\bigg)\bars^2Q_{S,6}^{(b)} \\
&\quad +\bigg(\frac{1}{\big|\bR_{\rm t}^{(a)}\big|^3}-\frac{1}{|\bR^{(a)}|^3}\bigg)(\epsilon\bars (Q_{S,5}^{(a)}-Q_{S,5}^{(b)})-\epsilon^2(\kappa_3^{(a)}-\kappa_3^{(b)})\sin(\bartheta) ) \\
&\quad +\bigg(\frac{1}{\big|\bR_{\rm t}^{(a)}\big|^3}-\frac{1}{|\bR^{(a)}|^3}-\frac{1}{\big|\bR_{\rm t}^{(b)}\big|^3}+\frac{1}{|\bR^{(b)}|^3}\bigg)(\bars+\epsilon\bars Q_{S,5}^{(b)}-\epsilon^2\kappa_3^{(b)}\sin(\bartheta) )\bigg) \be_j^{(a)}\cdot\bm{\varphi}_0^{\Phi^{(a)}} \\
&\quad + \bigg(\frac{\bars^2Q_{S,6}^{(b)}}{|\bR^{(b)}|^3}+\bigg(\frac{1}{|\bR_{\rm t}^{(b)}|^3}-\frac{1}{|\bR^{(b)}|^3}\bigg)(\bars+\epsilon\bars Q_{S,5}^{(b)}-\epsilon^2\kappa_3^{(b)}\sin(\bartheta) ) \bigg) \big(\be_j^{(a)}\cdot\bm{\varphi}_0^{\Phi^{(a)}}-\be_j^{(b)}\cdot\bm{\varphi}_0^{\Phi^{(b)}}\big) \,.
\end{align*}
Since each remainder coefficient satisfies the Lipschitz estimate \eqref{eq:Qj_Cbeta}, using \eqref{eq:RaRb_diff}, \eqref{eq:Rt_inv_diff}, and Lemma \ref{lem:alpha_est}, case (1), along with the bounds \eqref{eq:Phi_lip_est} and \eqref{eq:Phi_comm_lip_est}, we have 
\begin{equation}\label{eq:psRS11_Cbeta_lip}
\begin{aligned}
\abs{\p_s\mc{R}_{\mc{S},1,1}^{(a)}[\bm{\varphi}_0^{\Phi^{(a)}}]-\p_s\mc{R}_{\mc{S},1,1}^{(b)}[\bm{\varphi}_0^{\Phi^{(b)}}]}_{\dot C^{0,\beta}} &\le
 c(\kappa_{*,\beta}^{(a)},\kappa_{*,\beta}^{(b)},c_\Gamma)\,\epsilon^{1-2\beta}\,\norm{\X^{(a)}-\X^{(b)}}_{C^{2,\beta}}\norm{\bm{\varphi}}_{L^\infty}\,.
\end{aligned}
\end{equation}
Similarly, for $j={\rm t},r,\theta$, we write
\begin{align*}
\overline{\be}_j\cdot \p_s\mc{R}_{\mc{S},1,2} &= {\rm p.v.}\int_{-1/2}^{1/2}\int_{-\pi}^\pi S_{1,2} \,\epsilon\,d\bartheta d\bars\,,\\
S_{1,2} &= \bigg(\frac{\p_s\barR\cdot\barR-\p_s\bR_{\rm t}\cdot\bR_{\rm t}}{|\bR_{\rm t}|^3}+\bigg(\frac{1}{|\barR|^3}-\frac{1}{|\bR_{\rm t}|^3}\bigg)\p_s\barR\cdot\barR \bigg) \be_j\cdot\bm{\varphi} \\
&= \bigg(-\frac{\epsilon\bars Q_{S,5}-\epsilon^2\kappa_3\sin(\bartheta)}{|\bR|^3}+\bars\bigg(\frac{1}{|\barR|^3}-\frac{1}{|\bR_{\rm t}|^3}\bigg) \bigg) \be_j\cdot\bm{\varphi} \,.
\end{align*}
Using \eqref{eq:Rt_inv_diff} and Lemmas \ref{lem:odd_nm} and \ref{lem:basic_est}, we have
\begin{align*}
\abs{\p_s\mc{R}_{\mc{S},1,2}[\bm{\varphi}]} &\le c(\kappa_*,c_\Gamma)\,\epsilon^{1+\alpha}\left( \max_i\norm{Q_i}_{C_2^{0,\alpha}}\right)\norm{\bm{\varphi}}_{C^{0,\alpha}}
\le c(\kappa_{*,\alpha},c_\Gamma)\,\epsilon\norm{\bm{\varphi}}_{C^{0,\alpha}}\,,
\end{align*}
where the $Q_i$ include $Q_{S,1}$, $Q_{S,2}$, and $Q_{S,5}$.
In addition, using case (2) of Lemma \ref{lem:alpha_est} along with \eqref{eq:Phi_lip_est} and \eqref{eq:Phi_comm_lip_est}, we have
\begin{equation}\label{eq:psRS12_Cbeta}
\begin{aligned}
\abs{\p_s\mc{R}_{\mc{S},1,2}[\bm{\varphi}]}_{\dot C^{0,\alpha}} &\le c(\kappa_*,c_\Gamma)\,\epsilon \,\max_i\left(\norm{Q_i}_{C_1^{0,\alpha^+}}+\norm{Q_i}_{C_2^{0,\alpha}}\right)\norm{\bm{\varphi}}_{C^{0,\alpha}}\\
&\le c(\kappa_{*,\alpha^+},c_\Gamma)\,\epsilon^{1-\alpha^+}\norm{\bm{\varphi}}_{C^{0,\alpha}}\,, \\
\abs{\p_s\mc{R}_{\mc{S},1,2}^{(a)}[\bm{\varphi}_0^{\Phi^{(a)}}]-\p_s\mc{R}_{\mc{S},1,2}^{(b)}[\bm{\varphi}_0^{\Phi^{(b)}}]}_{\dot C^{0,\alpha}} &\le 
c(\kappa_{*,\alpha^+}^{(a)},\kappa_{*,\alpha^+}^{(b)},c_\Gamma)\,\epsilon^{1-\alpha^+}\,\norm{\X^{(a)}-\X^{(b)}}_{C^{2,\alpha^+}}\norm{\bm{\varphi}}_{C^{0,\alpha}}\,.
\end{aligned}
\end{equation}

We next consider $\p_s\mc{R}_{\mc{S},2}$, which we again treat in two parts as $\p_s\mc{R}_{\mc{S},2,1}$ and $\p_s\mc{R}_{\mc{S},2,2}$. Again, we will treat the commutator term $[\p_s,\Phi^{-1}]$ as part of the smoother remainder $\p_s\mc{R}_{\mc{S},2,1}$.  
For each $j\in\{{\rm t},r,\theta\}$, we expand the components of $\p_s\mc{R}_{\mc{S},2,1}$ as 
\begin{align*}
\overline{\be}_j\cdot \p_s\mc{R}_{\mc{S},2,1} &= {\rm p.v.}\int_{-1/2}^{1/2}\int_{-\pi}^\pi \big( S_{2,1,j}+S_{2,2,j} +S_{2,3,j}\big) \,\epsilon\,d\bartheta d\bars  \\
&\qquad\qquad  + \frac{\overline\be_j}{8\pi}\cdot\big[\p_s,\Phi^{-1} \big]\int_{-1/2}^{1/2}\int_{-\pi}^\pi\frac{\bR(\bR\cdot\bm{\varphi})}{\abs{\bR}^3}\,\epsilon \,d\bartheta d\bars\,,\\
S_{2,1,j} &= \frac{(\be_j\cdot\p_s\bR)(\bR\cdot\bm{\varphi})}{|\bR|^3} - \frac{(\be_j\cdot\p_s\bR_{\rm t})(\bR_{\rm t}\cdot\bm{\varphi})}{|\bR_{\rm t}|^3} \\
S_{2,2,j} &= \frac{(\be_j\cdot\bR)(\p_s\bR\cdot\bm{\varphi})}{|\bR|^3} - \frac{(\be_j\cdot\bR_{\rm t})(\p_s\bR_{\rm t}\cdot\bm{\varphi})}{|\bR_{\rm t}|^3} \\
S_{2,3,j} &= -3\frac{(\be_j\cdot\bR)(\bR\cdot\bm{\varphi})(\p_s\bR\cdot\bR)}{|\bR|^5} + 3\frac{(\be_j\cdot\bR_{\rm t})(\bR_{\rm t}\cdot\bm{\varphi})(\p_s\bR_{\rm t}\cdot\bR_{\rm t})}{|\bR_{\rm t}|^5}\,.
\end{align*}
Using the expansions \eqref{eq:Qexpand}, \eqref{eq:curvedR}, \eqref{eq:psR}, and \eqref{eq:psRt_expand}, we may further expand the tangential direction integrand as 
\begin{align*}
S_{2,1,\rm t} &= \frac{-\bars^2(1-\epsilon\wh\kappa)(\bm{Q}_{\rm t}\cdot\bm{\varphi})}{|\bR|^3} + \bigg(\frac{1}{|\bR|^3}-\frac{1}{|\bR_{\rm t}|^3}\bigg)\textstyle(1-\epsilon\wh\kappa)(\bars\be_{\rm t}+2\epsilon\sin(\frac{\bartheta}{2})\be_\theta+\epsilon\bars\bm{Q}_r)\cdot\bm{\varphi} \\
S_{2,2,\rm t} &= \frac{-\bars^3Q_{\rm t}((1-\epsilon\wh\kappa)\be_{\rm t}+\epsilon\kappa_3\be_\theta)\cdot\bm{\varphi}}{|\bR|^3} 
- \frac{\bars^2(\kappa_1\be_{\rm n_1}+\kappa_2\be_{\rm n_2})\cdot\bm{\varphi}\,(1+\epsilon\bm{Q}_r\cdot\be_{\rm t})}{|\bR|^3} \\
&\qquad + \bigg(\frac{1}{|\bR|^3}-\frac{1}{|\bR_{\rm t}|^3}\bigg)\bars(1+\epsilon \bm{Q}_r\cdot\be_{\rm t})\big((1-\epsilon\wh\kappa)\be_{\rm t}+\bars(\kappa_1\be_{\rm n_1}+\kappa_2\be_{\rm n_2})+\epsilon\kappa_3\be_\theta\big)\cdot\bm{\varphi}  \\
S_{2,3,\rm t} &= \frac{3\bars^3Q_{S,6}(1+\epsilon \bm{Q}_r\cdot\be_{\rm t})(\bars\be_{\rm t}+2\epsilon\sin(\frac{\bartheta}{2})\be_\theta+\epsilon\bars\bm{Q}_r)\cdot\bm{\varphi}}{|\bR|^5} \\
& 
+ \frac{3\bars^3\big[Q_{\rm t}(\bars\wt{\bm{Q}}_{\rm t}+2\epsilon\sin(\frac{\bartheta}{2})\be_\theta) + (1+\epsilon\bm{Q}_r\cdot\be_{\rm t})\bm{Q}_{\rm t}\big]\cdot\bm{\varphi}\,(\bars\wt Q_{\rm Rs1}+2\epsilon^2\sin(\frac{\bartheta}{2})\sin(\bartheta))}{|\bR|^5} \\
& -3\bigg(\frac{1}{|\bR|^5}-\frac{1}{|\bR_{\rm t}|^5}\bigg)\textstyle \bars(1+\epsilon \bm{Q}_r\cdot\be_{\rm t})(\bars\be_{\rm t}+2\epsilon\sin(\frac{\bartheta}{2})\be_\theta+\epsilon\bars\bm{Q}_r)\cdot\bm{\varphi}\,(\bars\wt Q_{\rm Rs2}+2\epsilon^2\sin(\frac{\bartheta}{2})\sin(\bartheta))\,,
\end{align*}
where we define $\wt{\bm{Q}}_{\rm t}=\be_{\rm t}-\bars\bm{Q}_{\rm t}+\epsilon\bm{Q}_r$, $\wt Q_{\rm Rs1}=1+\epsilon Q_{\rm Rs1}+ \bars Q_{\rm Rs2}$, and $\wt Q_{\rm Rs2}=1+\epsilon Q_{\rm Rs1}+ \bars(Q_{\rm Rs2}-Q_{S,6})$. Using the expansion \eqref{eq:Rt_inv_diff} for $\frac{1}{|\bR|^3}-\frac{1}{|\bR_{\rm t}|^3}$ and $\frac{1}{|\bR|^5}-\frac{1}{|\bR_{\rm t}|^5}$, along with the commutator bound \eqref{eq:Phi_commutator_est}, we may apply Lemma \ref{lem:alpha_est}, case (1), to obtain, for $\alpha<\beta<1$,
\begin{align*}
\abs{\be_z\cdot \p_s\mc{R}_{\mc{S},2,1}[\bm{\varphi}]}_{\dot C^{0,\beta}}
\le c(\kappa_*,c_\Gamma)\,\epsilon^{1-\beta}\left(\max_i\norm{Q_i}_{C_1^{0,\beta}} \right)\norm{\bm{\varphi}}_{L^\infty}
\le c(\kappa_{*,\beta},c_\Gamma)\,\epsilon^{1-2\beta}\norm{\bm{\varphi}}_{L^\infty}\,.
\end{align*}
Here $Q_i$ is used as a catch-all notation for each of the coefficients appearing in the expressions above, including $\wh\kappa(s,\theta)$, $\kappa_j(s)$, $Q_{\rm t}$, $\bm{Q}_{\rm t}$, $\bm{Q}_r$, and additional $Q$ remainders.
In addition, using \eqref{eq:Qj_Cbeta}, \eqref{eq:Rdiffk}, \eqref{eq:Rt_inv_diff}, \eqref{eq:Phi_lip_est}, and \eqref{eq:Phi_comm_lip_est}, by Lemma \ref{lem:alpha_est}, we have
\begin{align*}
\abs{\be_z\cdot \p_s\mc{R}_{\mc{S},2,1}^{(a)}[\bm{\varphi}_0^{\Phi^{(a)}}]-\be_z\cdot\p_s\mc{R}_{\mc{S},2,1}^{(b)}[\bm{\varphi}_0^{\Phi^{(b)}}]}_{\dot C^{0,\beta}}
\le c(\kappa_{*,\beta}^{(a)},\kappa_{*,\beta}^{(b)},c_\Gamma)\,\epsilon^{1-2\beta}\,\norm{\X^{(a)}-\X^{(b)}}_{C^{2,\beta}} \norm{\bm{\varphi}}_{L^\infty}\,.
\end{align*}

We next expand the $\be_r$-direction integrand of the first integral in $\p_s\mc{R}_{\mc{S},2,1}$ using \eqref{eq:Qexpand}, \eqref{eq:curvedR}, \eqref{eq:psR}, and \eqref{eq:psRt_expand}. We have
\begin{align*}
S_{2,1,r} &= \frac{-\bars\wh\kappa\big(\bars\wt{\bm{Q}}_{\rm t}+2\epsilon\sin(\frac{\bartheta}{2})\be_\theta+\bars^2\bm{Q}_{\rm t}\big)\cdot\bm{\varphi}}{|\bR|^3} 
+ \bigg(\frac{1}{|\bR|^3}-\frac{1}{|\bR_{\rm t}|^3}\bigg)\textstyle \bars\wh\kappa(\bars\be_{\rm t}+2\epsilon\sin(\frac{\bartheta}{2})\be_\theta+\epsilon\bars\bm{Q}_r)\cdot\bm{\varphi} \\
S_{2,2,r} &= \frac{-\bars^2(\be_r\cdot\bm{Q}_{\rm t})((1-\epsilon\wh\kappa)\be_{\rm t}+\epsilon\kappa_3\be_\theta)\cdot\bm{\varphi}}{|\bR|^3} 
- \frac{\bars(2\epsilon\sin^2(\frac{\bartheta}{2})+\epsilon\bars^2Q_{0,5})(\kappa_1\be_{\rm n_1}+\kappa_2\be_{\rm n_2})\cdot\bm{\varphi}}{|\bR|^3} \\
&\qquad + \bigg(\frac{1}{|\bR|^3}-\frac{1}{|\bR_{\rm t}|^3}\bigg)\textstyle (2\epsilon\sin^2(\frac{\bartheta}{2})+\epsilon\bars^2 Q_{0,5})((1-\epsilon\wh\kappa)\be_{\rm t}+\bars(\kappa_1\be_{\rm n_1}+\kappa_2\be_{\rm n_2})+\epsilon\kappa_3\be_\theta)\cdot\bm{\varphi}  \\
S_{2,3,r} &= \frac{3\bars^3\wh\kappa Q_{S,6}(\bars\be_{\rm t}+2\epsilon\sin(\frac{\bartheta}{2})\be_\theta+\epsilon\bars\bm{Q}_r)\cdot\bm{\varphi}}{|\bR|^5} \\
&+ \frac{3\bars^2\big[(\be_r\cdot\bm{Q}_{\rm t})(\bars\wh{\bm{Q}}_{\rm t}+2\epsilon\sin(\frac{\bartheta}{2})\be_\theta) + (2\epsilon\sin^2(\frac{\bartheta}{2})+\epsilon\bars^2Q_{0,5})\bm{Q}_{\rm t}\big]\cdot\bm{\varphi}\,(\bars\wt Q_{\rm Rs1}+2\epsilon^2\sin(\frac{\bartheta}{2})\sin(\bartheta))}{|\bR|^5} \\
&
-3\bigg(\frac{1}{|\bR|^5}-\frac{1}{|\bR_{\rm t}|^5}\bigg)\textstyle (2\epsilon\sin^2(\frac{\bartheta}{2})+\epsilon\bars^2Q_{0,5})(\bars\be_{\rm t}+2\epsilon\sin(\frac{\bartheta}{2})\be_\theta+\epsilon\bars\bm{Q}_r)\cdot\bm{\varphi}\,(\bars \wt Q_{\rm Rs2}+2\epsilon^2\sin(\frac{\bartheta}{2})\sin(\bartheta) )\,,
\end{align*}
where $\wt{\bm{Q}}_{\rm t}$, $\wt Q_{\rm Rs1}$, and $\wt Q_{\rm Rs2}$ are as defined above. We again use the expansion \eqref{eq:Rt_inv_diff} for $\frac{1}{|\bR|^3}-\frac{1}{|\bR_{\rm t}|^3}$ and $\frac{1}{|\bR|^5}-\frac{1}{|\bR_{\rm t}|^5}$, along with the commutator estimate \eqref{eq:Phi_commutator_est}, and apply case (1) of Lemma \ref{lem:alpha_est} to obtain, for $\alpha<\beta<1$,
\begin{align*}
\abs{\overline{\be}_r\cdot \p_s\mc{R}_{\mc{S},2,1}[\bm{\varphi}]}_{\dot C^{0,\beta}}
\le c(\kappa_*,c_\Gamma)\epsilon^{1-\beta}\left(\max_i\norm{Q_i}_{C_1^{0,\beta}} \right)\norm{\bm{\varphi}}_{L^\infty}
\le c(\kappa_{*,\beta},c_\Gamma)\epsilon^{1-2\beta}\norm{\bm{\varphi}}_{L^\infty}\,,\\
\abs{\overline{\be}_r\cdot \p_s\mc{R}_{\mc{S},2,1}^{(a)}[\bm{\varphi}_0^{\Phi^{(a)}}]-\overline{\be}_r\cdot \p_s\mc{R}_{\mc{S},2,1}^{(b)}[\bm{\varphi}_0^{\Phi^{(b)}}]}_{\dot C^{0,\beta}}
\le c(\kappa_{*,\beta}^{(a)},\kappa_{*,\beta}^{(b)},c_\Gamma)\,\epsilon^{1-2\beta}\,\norm{\X^{(a)}-\X^{(b)}}_{C^{2,\beta}} \norm{\bm{\varphi}}_{L^\infty}\,.
\end{align*}

Finally, we expand the $\be_\theta$-direction integrand of the first integral of $\p_s\mc{R}_{\mc{S},2,1}$ using \eqref{eq:Qexpand}, \eqref{eq:curvedR}, \eqref{eq:psR}, and \eqref{eq:psRt_expand}. We have
\begin{align*}
S_{2,1,\theta} &= \frac{-\bars\,\p_\theta\wh\kappa(\bars\wt{\bm{Q}}_{\rm t}+2\epsilon\sin(\frac{\bartheta}{2})\be_\theta)\cdot\bm{\varphi}}{|\bR|^3}
-\frac{\bars^2(\bm{Q}_{\rm t}\cdot\bm{\varphi})(\epsilon\kappa_3+\bars\p_\theta\wh\kappa)}{|\bR|^3} \\
&\qquad + \bigg(\frac{1}{|\bR|^3}-\frac{1}{|\bR_{\rm t}|^3}\bigg)\textstyle (\epsilon\kappa_3+\bars\p_\theta\wh\kappa)\big(\bars\be_{\rm t}+2\epsilon\sin(\frac{\bartheta}{2})\be_\theta+\epsilon\bars\bm{Q}_r\big)\cdot\bm{\varphi} \\
S_{2,2,\theta} &= \frac{-\bars^2(\be_\theta\cdot\bm{Q}_{\rm t})((1-\epsilon\wh\kappa)\be_{\rm t}+\epsilon\kappa_3\be_\theta)\cdot\bm{\varphi}}{|\bR|^3} 
-\frac{\bars(\kappa_1\be_{\rm n_1}+\kappa_2\be_{\rm n_2})\cdot\bm{\varphi}\,(\epsilon\sin(\bartheta)+\epsilon\bars\bm{Q}_r\cdot\be_\theta)}{|\bR|^3}  \\
&\qquad + \bigg(\frac{1}{|\bR|^3}-\frac{1}{|\bR_{\rm t}|^3}\bigg)\textstyle (\epsilon\sin(\bartheta)+\epsilon\bars\bm{Q}_r\cdot\be_\theta)((1-\epsilon\wh\kappa)\be_{\rm t}+\bars(\kappa_1\be_{\rm n_1}+\kappa_2\be_{\rm n_2})+\epsilon\kappa_3\be_\theta)\cdot\bm{\varphi}  \\
S_{2,3,\theta} &= \frac{3\bars^2Q_{S,6}(\epsilon\sin(\bartheta)+\epsilon\bars\bm{Q}_r\cdot\be_\theta)(\bars\be_{\rm t}+2\epsilon\sin(\frac{\bartheta}{2})\be_\theta+\epsilon\bars\bm{Q}_r)\cdot\bm{\varphi}}{|\bR|^5}  \\
&+ \frac{3\bars^2\big[\be_\theta\cdot\bm{Q}_{\rm t}(\bars\wt{\bm{Q}}_{\rm t}+2\epsilon\sin(\frac{\bartheta}{2})\be_\theta)+ \epsilon(\sin(\bartheta)+\bars\bm{Q}_r\cdot\be_\theta)\bm{Q}_{\rm t}\big]\cdot\bm{\varphi}\,(\bars\wt Q_{\rm Rs1}+2\epsilon^2\sin(\frac{\bartheta}{2})\sin(\bartheta))}{|\bR|^5}\\
&-3\bigg(\frac{1}{|\bR|^5}-\frac{1}{|\bR_{\rm t}|^5}\bigg)\textstyle \epsilon(\sin(\bartheta)+\bars\bm{Q}_r\cdot\be_\theta)(\bars\be_{\rm t}+2\epsilon\sin(\frac{\bartheta}{2})\be_\theta+\epsilon\bars\bm{Q}_r)\cdot\bm{\varphi}\,(\bars\wt Q_{\rm Rs2}+2\epsilon^2\sin(\frac{\bartheta}{2})\sin(\bartheta))\,,
\end{align*}
with $\wt{\bm{Q}}_{\rm t}$, $\wt Q_{\rm Rs1}$, and $\wt Q_{\rm Rs2}$ as before. Again expanding $\frac{1}{|\bR|^3}-\frac{1}{|\bR_{\rm t}|^3}$ and $\frac{1}{|\bR|^5}-\frac{1}{|\bR_{\rm t}|^5}$ using \eqref{eq:Rt_inv_diff}, and using the commutator estimate \eqref{eq:Phi_commutator_est}, by Lemma \ref{lem:alpha_est}, case (1), we have, for $\alpha<\beta<1$,
\begin{align*}
\abs{\overline{\be}_\theta\cdot \p_s\mc{R}_{\mc{S},2,1}[\bm{\varphi}]}_{\dot C^{0,\beta}}
\le c(\kappa_*,c_\Gamma)\,\epsilon^{1-\beta}\left(\max_i\norm{Q_i}_{C_1^{0,\beta}} \right)\norm{\bm{\varphi}}_{L^\infty}
\le c(\kappa_*,c_\Gamma)\,\epsilon^{1-2\beta}\norm{\bm{\varphi}}_{L^\infty}\,,\\
\abs{\overline{\be}_\theta\cdot \p_s\mc{R}_{\mc{S},2,1}^{(a)}[\bm{\varphi}_0^{\Phi^{(a)}}]-\overline{\be}_\theta\cdot \p_s\mc{R}_{\mc{S},2,1}^{(b)}[\bm{\varphi}_0^{\Phi^{(b)}}]}_{\dot C^{0,\beta}}
\le c(\kappa_{*,\beta}^{(a)},\kappa_{*,\beta}^{(b)},c_\Gamma)\,\epsilon^{1-2\beta}\,\norm{\X^{(a)}-\X^{(b)}}_{C^{2,\beta}} \norm{\bm{\varphi}}_{L^\infty}\,.
\end{align*}

In total, for $\alpha<\beta<1$, we have
\begin{equation}\label{eq:psRS21_C0beta}
\begin{aligned}
\abs{\p_s\mc{R}_{\mc{S},2,1}[\bm{\varphi}]}_{\dot C^{0,\beta}}&\le c(\kappa_*,c_\Gamma)\epsilon^{1-2\beta}\norm{\bm{\varphi}}_{L^\infty} \\
\abs{\p_s\mc{R}_{\mc{S},2,1}^{(a)}[\bm{\varphi}_0^{\Phi^{(a)}}]-\p_s\mc{R}_{\mc{S},2,1}^{(b)}[\bm{\varphi}_0^{\Phi^{(b)}}]}_{\dot C^{0,\beta}}
&\le c(\kappa_{*,\beta}^{(a)},\kappa_{*,\beta}^{(b)},c_\Gamma)\,\epsilon^{1-2\beta}\,\norm{\X^{(a)}-\X^{(b)}}_{C^{2,\beta}} \norm{\bm{\varphi}}_{L^\infty}\,. \\
\end{aligned}
\end{equation}

We now turn to $\p_s\mc{R}_{\mc{S},2,2}$. For $j\in\{{\rm t},r,\theta\}$, we again expand the components of $\p_s\mc{R}_{\mc{S},2,2}$ as
\begin{align*}
\overline{\be}_j\cdot \p_s\mc{R}_{\mc{S},2,2} &= {\rm p.v.}\int_{-1/2}^{1/2}\int_{-\pi}^\pi \big( S_{2,4,j}+S_{2,5,j} +S_{2,6,j}\big) \,\epsilon\,d\bartheta d\bars\,,\\
S_{2,4,j} &= \frac{(\be_j\cdot\p_s\bR_{\rm t})(\bR_{\rm t}\cdot\bm{\varphi})}{|\bR_{\rm t}|^3} - \frac{(\overline{\be}_j\cdot\p_s\barR)(\barR\cdot(\Phi^{-1}\bm{\varphi}))}{|\barR|^3}  \\
S_{2,5,j} &= \frac{(\be_j\cdot\bR_{\rm t})(\p_s\bR_{\rm t}\cdot\bm{\varphi})}{|\bR_{\rm t}|^3} - \frac{(\overline{\be}_j\cdot\barR)(\p_s\barR\cdot(\Phi^{-1}\bm{\varphi}))}{|\barR|^3} \\
 S_{2,6,j} &= -3\frac{(\be_j\cdot\bR_{\rm t})(\bR_{\rm t}\cdot\bm{\varphi})(\p_s\bR_{\rm t}\cdot\bR_{\rm t})}{|\bR_{\rm t}|^5} + 3\frac{(\overline{\be}_j\cdot\barR)(\barR\cdot(\Phi^{-1}\bm{\varphi}))(\p_s\barR\cdot\barR)}{|\barR|^5}\,.
\end{align*}
Similar to before, we use \eqref{eq:barR}, \eqref{eq:Qexpand}, and \eqref{eq:psRt_expand} to expand the tangential direction integrand as 
\begin{align*}
S_{2,4,\rm t} &= \frac{\epsilon\wh\kappa(\bars\be_{\rm t}+2\epsilon\sin(\frac{\bartheta}{2})\be_\theta+\epsilon\bars\bm{Q}_r)\cdot\bm{\varphi}+\epsilon\bars(\bm{Q}_r\cdot\bm{\varphi})}{|\bR_{\rm t}|^3}  \\
&\qquad + \bigg(\frac{1}{|\bR_{\rm t}|^3}-\frac{1}{|\barR|^3}\bigg)\textstyle(\bars\be_{\rm t}+2\epsilon\sin(\frac{\bartheta}{2})\be_\theta)\cdot\bm{\varphi}  \\
S_{2,5,\rm t} &= \frac{\epsilon\bars(\bm{Q}_r\cdot\be_{\rm t})(\be_{\rm t}\cdot\bm{\varphi})}{|\bR_{\rm t}|^3}+\frac{\bars(1+\epsilon\bm{Q}_r\cdot\be_{\rm t})(-\epsilon\wh\kappa\be_{\rm t}+\bars(\kappa_1\be_{\rm n_1}+\kappa_2\be_{\rm n_2})+\epsilon\kappa_3\be_\theta)\cdot\bm{\varphi}}{|\bR_{\rm t}|^3} \\
&\qquad + \bigg(\frac{1}{|\bR_{\rm t}|^3}-\frac{1}{|\barR|^3} \bigg)\bars(\be_{\rm t}\cdot\bm{\varphi})\\
 S_{2,6,\rm t} &= -\frac{3\epsilon\bars(\bm{Q}_r\cdot\be_{\rm t})(\bars\be_{\rm t}+2\epsilon\sin(\frac{\bartheta}{2})\be_\theta+\epsilon\bars\bm{Q}_r)\cdot\bm{\varphi}\,(\bars+\epsilon\bars Q_{S,5}-\epsilon^2\kappa_3\sin(\bartheta))}{|\bR_{\rm t}|^5} \\
& -\frac{3\epsilon\bars\big[\bars(\bars+\epsilon\bars Q_{S,5}-\epsilon^2\kappa_3\sin(\bartheta))\bm{Q}_r +(\bars Q_{S,5}-\epsilon\kappa_3\sin(\bartheta))(\bars\be_{\rm t}+2\epsilon\sin(\frac{\bartheta}{2})\be_\theta)\big]\cdot\bm{\varphi}}{|\bR_{\rm t}|^5} \\
&\qquad - 3\bigg(\frac{1}{|\bR_{\rm t}|^5}- \frac{1}{|\barR|^5}\bigg)\textstyle \bars^2(\bars\be_{\rm t}+2\epsilon\sin(\frac{\bartheta}{2})\be_\theta)\cdot\bm{\varphi}\,.
\end{align*}
Here most terms appear with an additional factor of $\epsilon$ compared to those appearing in $\p_s\mc{R}_{\mc{S},2,1}$, including the terms involving $\frac{1}{|\bR_{\rm t}|^3}- \frac{1}{|\barR|^3}$ and $\frac{1}{|\bR_{\rm t}|^5}- \frac{1}{|\barR|^5}$, by \eqref{eq:Rt_inv_diff}. The second term in $S_{2,5,\rm t}$ does not carry an additional factor of $\epsilon$ in all components. In particular, there is a term of the form $\frac{\bars^2(\kappa_1\be_{\rm n_1}+\kappa_2\be_{\rm n_2})\cdot\bm{\varphi}}{\abs{\bR_{\rm t}}^3}$. We emphasize that this term instead includes an additional factor of $\bars$ in the numerator, and that $\kappa_1\be_{\rm n_1}+\kappa_2\be_{\rm n_2}$ depends only on $s$ and not $\theta$. By case (1) of Lemma \ref{lem:alpha_est}, this term satisfies
\begin{align*}
 \abs{\int_{-1/2}^{1/2}\int_{-\pi}^\pi \frac{\bars^2(\kappa_1\be_{\rm n_1}+\kappa_2\be_{\rm n_2})\cdot\bm{\varphi}}{\abs{\bR_{\rm t}}^3}\,\epsilon\,d\bartheta d\bars }_{\dot C^{0,\alpha}}
 &\le c(\kappa_*,c_\Gamma)\,\epsilon^{1-\alpha}\norm{\X_{ss}}_{C^{0,\alpha}(\T)}\norm{\bm{\varphi}}_{L^\infty}\\
&\le c(\kappa_{*,\alpha},c_\Gamma)\,\epsilon^{1-\alpha}\norm{\bm{\varphi}}_{L^\infty}\,.
\end{align*} 
In particular, since $\kappa_1\be_{\rm n_1}+\kappa_2\be_{\rm n_2}$ depends only on $s$, we do not lose an additional factor of $\epsilon^{-\alpha}$.
The remaining expressions making up $\be_z\cdot \p_s\mc{R}_{\mc{S},2,2}$ may be estimated using case (2) of Lemma \ref{lem:alpha_est}. In total, we obtain
\begin{align*}
\abs{\be_z\cdot \p_s\mc{R}_{\mc{S},2,2}}_{\dot C^{0,\alpha}} &\le c(\kappa_{*,\alpha},c_\Gamma)\,\epsilon^{1-\alpha}\norm{\bm{\varphi}}_{L^\infty}
+ c(\kappa_*,c_\Gamma)\,\epsilon \,\max_i\left(\norm{Q_i}_{C_1^{0,\alpha^+}}+\norm{Q_i}_{C_2^{0,\alpha}}\right)\norm{\bm{\varphi}}_{C^{0,\alpha}}\\
&\le c(\kappa_{*,\alpha^+},c_\Gamma)\,\epsilon^{1-\alpha^+}\norm{\bm{\varphi}}_{C^{0,\alpha}}\,.
\end{align*}
By similar arguments and estimate \eqref{eq:Phi_lip_est}, for two filaments with centerlines $\X^{(a)}$, $\X^{(b)}$ satisfying Lemma \ref{lem:XaXb_C2beta}, we may additionally obtain the Lipschitz estimate
\begin{align*}
\abs{\be_z\cdot \p_s\mc{R}_{\mc{S},2,2}^{(a)}-\be_z\cdot \p_s\mc{R}_{\mc{S},2,2}^{(b)}}_{\dot C^{0,\alpha}} 
&\le c(\kappa_{*,\alpha^+}^{(a)},\kappa_{*,\alpha^+}^{(b)},c_\Gamma)\,\epsilon^{1-\alpha^+}\,\norm{\X^{(a)}-\X^{(b)}}_{C^{2,\alpha^+}}\norm{\bm{\varphi}}_{C^{0,\alpha}}\,.
\end{align*}

We next consider $\p_s\mc{R}_{\mc{S},2,2}$ in directions normal to the filament centerline. It will be convenient to consider $S_{2,4,{\rm n_1}}$ and $S_{2,4,{\rm n_2}}$ (in the directions $\be_{\rm n_1}$ and $\be_{\rm n_2}$) instead of $S_{2,4,r}$ and $S_{2,4,\theta}$ to avoid introducing artificial $\theta$-dependence resulting in a loss of $\epsilon^{-\alpha}$. The expressions $S_{2,4,{\rm n_1}}$ and $S_{2,4,{\rm n_2}}$ are given by 
\begin{align*}
S_{2,4,{\rm n_1}} &= \frac{(\bars\kappa_1(s)-\epsilon\kappa_3\sin(\theta))(\bars\be_{\rm t}+2\epsilon\sin(\frac{\bartheta}{2})\be_\theta+\epsilon\bars\bm{Q}_r)\cdot\bm{\varphi}}{|\bR_{\rm t}|^3} \\
S_{2,4,{\rm n_2}} &= \frac{(\bars\kappa_2(s)+\epsilon\kappa_3\cos(\theta))(\bars\be_{\rm t}+2\epsilon\sin(\frac{\bartheta}{2})\be_\theta+\epsilon\bars\bm{Q}_r)\cdot\bm{\varphi}}{|\bR_{\rm t}|^3} \,.
\end{align*}
Both expressions contain terms proportional to $\frac{\bars^2}{|\bR_{\rm t}|^3}$, $\frac{\bars \epsilon\sin(\frac{\bartheta}{2})}{|\bR_{\rm t}|^3}$, and terms which are higher order in $\epsilon$. For the terms of the first form, using case (1) of Lemma \ref{lem:alpha_est}, we have
\begin{align*}
 \abs{\int_{-1/2}^{1/2}\int_{-\pi}^\pi \frac{\bars^2 \kappa_i(s)\be_{\rm t}\cdot\bm{\varphi}}{\abs{\bR_{\rm t}}^3}\,\epsilon\,d\bartheta d\bars }_{\dot C^{0,\alpha}}
 &\le c(\kappa_*,c_\Gamma)\,\epsilon^{1-\alpha}\norm{\X_{ss}}_{C^{0,\alpha}(\T)}\norm{\bm{\varphi}}_{L^\infty}\\
&\le c(\kappa_{*,\alpha},c_\Gamma)\,\epsilon^{1-\alpha}\norm{\bm{\varphi}}_{L^\infty}\,, 
\end{align*} 
since $\kappa_i(s)\be_{\rm t}(s)$, $i=1,2$, depends only on $s$ and not $\theta$ or $\bartheta$. In particular, we do not lose an additional factor of $\epsilon^{-\alpha}$ from estimating these terms in $C^{0,\alpha}$. For terms proportional to $\frac{\bars \epsilon\sin(\frac{\bartheta}{2})}{|\bR_{\rm t}|^3}$, we treat $\sin(\frac{\bartheta}{2})$ as a coefficient, i.e. let $\bm{Q}(s,\bars,\theta,\bartheta)=2\kappa_i(s)\sin(\frac{\bartheta}{2})\be_\theta(s,\theta)$, and use case (2) of Lemma \ref{lem:alpha_est} to estimate
\begin{equation}\label{eq:sine_treatment}
\begin{aligned}
 \abs{{\rm p.v.}\int_{-1/2}^{1/2}\int_{-\pi}^\pi \frac{\epsilon \bars \bm{Q}\cdot\bm{\varphi}}{\abs{\bR_{\rm t}}^3}\,\epsilon\,d\bartheta d\bars }_{\dot C^{0,\alpha}}
 &\le c(\kappa_*,c_\Gamma)\,\epsilon\,\left(\norm{\bm{Q}}_{C_1^{0,\alpha^+}}+\norm{\bm{Q}}_{C_2^{0,\alpha}} \right)\norm{\bm{\varphi}}_{C^{0,\alpha}}\\
&\le c(\kappa_{*,\alpha^+},c_\Gamma)\,\epsilon^{1-\alpha^+}\norm{\bm{\varphi}}_{C^{0,\alpha}}\,. 
\end{aligned}
\end{equation}
The remaining terms may be treated using case (2) of Lemma \ref{lem:alpha_est}, yielding, for $j\in\{{\rm n_1,n_2}\}$, 
\begin{align*}
\abs{{\rm p.v.}\int_{-1/2}^{1/2}\int_{-\pi}^\pi S_{2,4,j}\,\epsilon\,d\bartheta d\bars }_{\dot C^{0,\alpha}}
\le c(\kappa_{*,\alpha^+},c_\Gamma)\,\epsilon^{1-\alpha^+}\norm{\bm{\varphi}}_{C^{0,\alpha}}\,.
\end{align*}
In addition, using \eqref{eq:Qj_Cbeta}, \eqref{eq:Rdiffk}, \eqref{eq:Rt_inv_diff}, and \eqref{eq:Phi_lip_est}, by Lemma \ref{lem:alpha_est}, case (2), we obtain 
\begin{align*}
\abs{{\rm p.v.}\int_{-1/2}^{1/2}\int_{-\pi}^\pi \big(S_{2,4,j}^{(a)}-S_{2,4,j}^{(b)}\big)\,\epsilon\,d\bartheta d\bars }_{\dot C^{0,\alpha}}
\le c(\kappa_{*,\alpha^+}^{(a)},\kappa_{*,\alpha^+}^{(b)},c_\Gamma)\,\epsilon^{1-\alpha^+}\,\norm{\X^{(a)}-\X^{(b)}}_{C^{2,\alpha^+}} \norm{\bm{\varphi}}_{C^{0,\alpha}}\,.
\end{align*}

We next use \eqref{eq:barR}, \eqref{eq:Qexpand}, and \eqref{eq:psRt_expand} to expand the remaining terms in the $\be_r$- and $\be_\theta$-direction integrands as 
\begin{align*}
S_{2,5,r} &= \frac{\epsilon\bars^2Q_{0,5}(\be_{\rm t}\cdot\bm{\varphi})}{|\bR_{\rm t}|^3}+\frac{(\epsilon\bars^2Q_{0,5}+ 2\epsilon\sin^2(\frac{\bartheta}{2}))(-\epsilon\wh\kappa\be_{\rm t}+\bars(\kappa_1(s)\be_{\rm n_1}(s)+\kappa_2(s)\be_{\rm n_2}(s))+\epsilon\kappa_3\be_\theta)\cdot\bm{\varphi}}{|\bR_{\rm t}|^3} \\
 &\qquad + \bigg(\frac{1}{|\bR_{\rm t}|^3}-\frac{1}{|\barR|^3} \bigg)\textstyle 2\epsilon\sin^2(\frac{\bartheta}{2}) (\be_{\rm t}\cdot\bm{\varphi})\\
 S_{2,6,r} &= -\frac{3\epsilon\bars\big[\bars Q_{0,5}(\bars\be_{\rm t}+2\epsilon\sin(\frac{\bartheta}{2})\be_\theta+\epsilon\bars\bm{Q}_r) + 2\epsilon\sin^2(\frac{\bartheta}{2})\bm{Q}_r\big]\cdot\bm{\varphi}\,(\bars +\epsilon\bars Q_{S,5} -\epsilon^2\kappa_3\sin(\bartheta))}{|\bR_{\rm t}|^5} \\
&\qquad -\frac{6\epsilon^2\sin^2(\frac{\bartheta}{2})(\bars Q_{S,5}-\epsilon\kappa_3\sin(\bartheta))(\bars\be_{\rm t}+2\epsilon\sin(\frac{\bartheta}{2})\be_\theta)\cdot\bm{\varphi}}{|\bR_{\rm t}|^5} \\
&\qquad - \bigg(\frac{1}{|\bR_{\rm t}|^5}- \frac{1}{|\barR|^5}\bigg)\textstyle 6\bars\epsilon\sin^2(\frac{\bartheta}{2})(\bars\be_{\rm t}+2\epsilon\sin(\frac{\bartheta}{2})\be_\theta)\cdot\bm{\varphi} \\
S_{2,5,\theta} &= \frac{\epsilon\bars(\bm{Q}_r\cdot\be_\theta)\be_{\rm t}\cdot\bm{\varphi}}{|\bR_{\rm t}|^3}
+ \frac{\epsilon(\bars(\bm{Q}_r\cdot\be_\theta) + \sin(\bartheta))(-\epsilon\wh\kappa\be_{\rm t}+\bars(\kappa_1\be_{\rm n_1}+\kappa_2\be_{\rm n_2})+\epsilon\kappa_3\be_\theta)\cdot\bm{\varphi}}{|\bR_{\rm t}|^3} \\
&\qquad + \bigg(\frac{1}{|\bR_{\rm t}|^3}-\frac{1}{|\barR|^3} \bigg)\epsilon\sin(\bartheta)(\be_{\rm t}\cdot\bm{\varphi})\\
 S_{2,6,\theta} &= -\frac{3\epsilon\bars\big[(\bm{Q}_r\cdot\be_\theta)(\bars\be_{\rm t}+2\epsilon\sin(\frac{\bartheta}{2})\be_\theta+\epsilon\bars\bm{Q}_r)+ \epsilon\sin(\bartheta)\bm{Q}_r\big]\cdot\bm{\varphi}\,(\bars +\epsilon\bars Q_{S,5} -\epsilon^2\kappa_3\sin(\bartheta))}{|\bR_{\rm t}|^5} \\
&\qquad -\frac{3\epsilon^2\sin(\bartheta)(\bars Q_{S,5}-\epsilon\kappa_3\sin(\bartheta))(\bars\be_{\rm t}+2\epsilon\sin(\frac{\bartheta}{2})\be_\theta)\cdot\bm{\varphi}}{|\bR_{\rm t}|^5} \\
&\qquad - 3\bigg(\frac{1}{|\bR_{\rm t}|^5}- \frac{1}{|\barR|^5}\bigg)\textstyle \epsilon\sin(\bartheta)\bars(\bars\be_{\rm t}+2\epsilon\sin(\frac{\bartheta}{2})\be_\theta)\cdot\bm{\varphi} \,.
\end{align*}
The second and third terms in $S_{2,5,r}$ and $S_{2,6,r}$ both contain components that must be treated as in \eqref{eq:sine_treatment} to avoid loss of $\epsilon^{-\alpha}$. In total, using case (2) of Lemma \ref{lem:alpha_est}, for both $j=r$ and $j=\theta$ we obtain the bound
\begin{align*}
&\abs{{\rm p.v.}\int_{-1/2}^{1/2}\int_{-\pi}^\pi \big( S_{2,5,j}+ S_{2,6,j}\big)\,\epsilon\,d\bartheta d\bars }_{\dot C^{0,\alpha}}
\le c(\kappa_{*,\alpha^+},c_\Gamma)\,\epsilon^{1-\alpha^+}\norm{\bm{\varphi}}_{C^{0,\alpha}}\,,\\
&\abs{{\rm p.v.}\int_{-1/2}^{1/2}\int_{-\pi}^\pi \big( S_{2,5,j}^{(a)}-S_{2,5,j}^{(b)}+ S_{2,6,j}^{(a)}-S_{2,6,j}^{(b)}\big)\,\epsilon\,d\bartheta d\bars }_{\dot C^{0,\alpha}}\\
&\qquad\le c(\kappa_{*,\alpha^+}^{(a)},\kappa_{*,\alpha^+}^{(b)},c_\Gamma)\,\epsilon^{1-\alpha^+}\,\norm{\X^{(a)}-\X^{(b)}}_{C^{2,\alpha^+}}\norm{\bm{\varphi}}_{C^{0,\alpha}}\,.
\end{align*}

Altogether, for $\alpha^+$ satisfying $\alpha<\alpha^+<1$, we have
\begin{equation}\label{eq:psRS22_C0alpha}
\begin{aligned}
\abs{\p_s\mc{R}_{\mc{S},2,2}[\bm{\varphi}]}_{\dot C^{0,\alpha}}&\le c(\kappa_{*,\alpha^+},c_\Gamma)\,\epsilon^{1-\alpha^+}\norm{\bm{\varphi}}_{C^{0,\alpha}}\,, \\
\abs{\p_s\mc{R}_{\mc{S},2,2}^{(a)}[\bm{\varphi}_0^{\Phi^{(a)}}]-\p_s\mc{R}_{\mc{S},2,2}^{(b)}[\bm{\varphi}_0^{\Phi^{(b)}}]}_{\dot C^{0,\alpha}} &\le c(\kappa_{*,\alpha^+}^{(a)},\kappa_{*,\alpha^+}^{(b)},c_\Gamma)\,\epsilon^{1-\alpha^+}\,\norm{\X^{(a)}-\X^{(b)}}_{C^{2,\alpha^+}}\norm{\bm{\varphi}}_{C^{0,\alpha}}\,. 
\end{aligned}
\end{equation}

Finally, we derive bounds for $\p_s\mc{R}_{\mc{S},3}$, given by
\begin{align*}
\p_s\mc{R}_{\mc{S},3}[\bm{\varphi}] &= -\Phi^{-1}\,{\rm p.v.}\int_{-1/2}^{1/2}\int_{-\pi}^\pi (\p_s\mc{G})\,\bm{\varphi}\, \epsilon^2\wh\kappa \,d\bartheta d\bars
- [\p_s,\Phi^{-1}]\int_{-1/2}^{1/2}\int_{-\pi}^\pi \mc{G}\,\bm{\varphi}\, \epsilon^2\wh\kappa\,d\bartheta d\bars\,.
\end{align*}
 We may write the integrand of the first term of $\p_s\mc{R}_{\mc{S},3}$ as
\begin{align*}
&\big(\p_s \mc{G}\,\bm{\varphi}\big)\epsilon^2\wh\kappa =  \bigg(-\frac{(\p_s\bR\cdot\bR)\bm{\varphi}}{|\bR|^3}+\frac{\p_s\bR(\bR\cdot\bm{\varphi})}{|\bR|^3} + \frac{\bR(\p_s\bR\cdot\bm{\varphi})}{|\bR|^3} -3\frac{\bR(\bR\cdot\bm{\varphi})(\p_s\bR\cdot\bR)}{|\bR|^5}\bigg)\epsilon^2\wh\kappa \\
&=  \bigg(-\frac{(\bars(1-\epsilon\wh\kappa)\be_{\rm t}\cdot\wt{\bm{Q}}_{\rm t}+2\bars\epsilon^2\kappa_3\be_\theta\cdot\wt{\bm{Q}}_{\rm t}+\epsilon^2\kappa_3\sin(\bartheta))\bm{\varphi}}{|\bR|^3}\\
&+ \frac{((1-\epsilon\wh\kappa)\be_{\rm t}+\epsilon\kappa_3\be_\theta)(\bars\wt{\bm{Q}}_{\rm t}+2\epsilon\sin(\frac{\bartheta}{2})\be_\theta)\cdot\bm{\varphi}}{|\bR|^3} + \frac{(\bars\wt{\bm{Q}}_{\rm t}+2\epsilon\sin(\frac{\bartheta}{2})\be_\theta)((1-\epsilon\wh\kappa)\be_{\rm t}+\epsilon\kappa_3\be_\theta)\cdot\bm{\varphi}}{|\bR|^3} \\
& -3\frac{(\bars\wt{\bm{Q}}_{\rm t}+2\epsilon\sin(\frac{\bartheta}{2})\be_\theta)(\bars\wt{\bm{Q}}_{\rm t}+2\epsilon\sin(\frac{\bartheta}{2})\be_\theta)\cdot\bm{\varphi}\,(\bars(1-\epsilon\wh\kappa)\be_{\rm t}\cdot\wt{\bm{Q}}_{\rm t}+2\bars\epsilon^2\kappa_3\be_\theta\cdot\wt{\bm{Q}}_{\rm t}+\epsilon^2\kappa_3\sin(\bartheta))}{|\bR|^5}\bigg)\epsilon^2\wh\kappa \,.
\end{align*} 
Noting the additional factor of $\epsilon$ due to \eqref{eq:jacfac}, by Lemma \ref{lem:alpha_est}, case (2), we then have 
\begin{align*}
\abs{\Phi^{-1}\,{\rm p.v.}\int_{-1/2}^{1/2}\int_{-\pi}^\pi (\p_s\mc{G})\,\bm{\varphi}\, \epsilon^2\wh\kappa \,d\bartheta d\bars}_{\dot C^{0,\alpha}} &\le c(\kappa_{*,\alpha^+},c_\Gamma)\,\epsilon^{1-\alpha^+}\norm{\bm{\varphi}}_{C^{0,\alpha}}\,.
\end{align*}
Using the commutator bound \eqref{eq:Phi_commutator_est} to estimate the second term of $\p_s\mc{R}_{\mc{S},3}$, we then have
\begin{equation}\label{eq:psRS3_C0alpha}
\begin{aligned}
\abs{\p_s\mc{R}_{\mc{S},3}[\bm{\varphi}]}_{\dot C^{0,\alpha}} &\le c(\kappa_{*,\alpha^+},c_\Gamma)\,\epsilon^{1-\alpha^+}\norm{\bm{\varphi}}_{C^{0,\alpha}}\,,\\
\abs{\p_s\mc{R}_{\mc{S},3}^{(a)}[\bm{\varphi}_0^{\Phi^{(a)}}]-\p_s\mc{R}_{\mc{S},3}^{(b)}[\bm{\varphi}_0^{\Phi^{(b)}}]}_{\dot C^{0,\alpha}} &\le c(\kappa_{*,\alpha^+}^{(a)},\kappa_{*,\alpha^+}^{(b)},c_\Gamma)\,\epsilon^{1-\alpha^+}\,\norm{\X^{(a)}-\X^{(b)}}_{C^{2,\alpha^+}} \norm{\bm{\varphi}}_{C^{0,\alpha}}\,. 
\end{aligned}
\end{equation}

We may now put everything together. Define 
\begin{align*}
\mc{R}_{\mc{S},+}[\bm{\varphi}] &= \mc{R}_{\mc{S},0}[\bm{\varphi}]+\mc{R}_{\mc{S},1,1}[\bm{\varphi}] + \mc{R}_{\mc{S},2,1}[\bm{\varphi}] \,, \\
\mc{R}_{\mc{S},\epsilon}[\bm{\varphi}] &=\mc{R}_{\mc{S},1,2}[\bm{\varphi}] + \mc{R}_{\mc{S},2,2}[\bm{\varphi}] + \mc{R}_{\mc{S},3}[\bm{\varphi}]\,.
\end{align*}
Then, combining the bounds \eqref{eq:RS0_bds}, \eqref{eq:RS_lower_est1}, \eqref{eq:RS_lower_est2}, \eqref{eq:RS_lower_est3}, \eqref{eq:psRS11_Cbeta}, \eqref{eq:psRS11_Cbeta_lip}, \eqref{eq:psRS21_C0beta}, we obtain the estimates \eqref{eq:RS_ests1} and \eqref{eq:RS_ests2} for $\mc{R}_{\mc{S},+}$. 
Additionally, combining the bounds \eqref{eq:RS_lower_est1}, \eqref{eq:RS_lower_est2}, \eqref{eq:RS_lower_est3}, \eqref{eq:psRS12_Cbeta}, \eqref{eq:psRS22_C0alpha}, and \eqref{eq:psRS3_C0alpha}, we obtain the estimates \eqref{eq:RS_ests1} and \eqref{eq:RS_ests2} for $\mc{R}_{\mc{S},\epsilon}$. In total, we obtain Lemma \ref{lem:single_layer}.
\end{proof}


\subsection{Double layer remainder terms}\label{subsec:double_layer}
We next prove Lemma \ref{lem:double_layer} regarding the mapping properties of the double layer remainder $\mc{R}_\mc{D}$, defined in \eqref{eq:RS_RD_def}. We again restate the lemma here for convenience.

\begin{lemma}[Double layer remainder]\label{lem:double_layer0}
Let $0<\gamma<\beta<1$ and consider a filament $\Sigma_\epsilon$ with centerline $\X(s)\in C^{2,\beta}(\T)$. Let the double layer remainder $\mc{R}_{\mc{D}}$ be as defined in \eqref{eq:RS_RD_def}. Given $\bm{\psi}\in C^{0,\gamma}(\Gamma_\epsilon)$, we have that $\mc{R}_{\mc{D}}$ satisfies
\begin{equation}\label{eq:RD_est0}
\norm{\mc{R}_{\mc{D}}[\bm{\psi}]}_{C^{1,\gamma}}\le c(\kappa_{*,\gamma^+},c_\Gamma)\,\epsilon^{-\gamma^+}\norm{\bm{\psi}}_{C^{0,\gamma}}
\end{equation}
for any $\gamma^+\in(\gamma,\beta]$.
Furthermore, given two nearby filaments with centerlines $\X^{(a)}(s)$, $\X^{(b)}(s)$ satisfying Lemma \ref{lem:XaXb_C2beta}, the difference between their corresponding double layer remainders $\mc{R}_{\mc{D}}^{(a)}$ and $\mc{R}_{\mc{D}}^{(b)}$ satisfies 
\begin{equation}\label{eq:RD_est_lip0}
\norm{\mc{R}_{\mc{D}}^{(a)}[\bm{\psi}]-\mc{R}_{\mc{D}}^{(b)}[\bm{\psi}]}_{C^{1,\gamma}}\le c(\kappa_{*,\gamma^+}^{(a)},\kappa_{*,\gamma^+}^{(b)},c_\Gamma)\,\epsilon^{-\gamma^+}\norm{\X^{(a)}-\X^{(b)}}_{C^{2,\gamma^+}}\norm{\bm{\psi}}_{C^{0,\gamma}}\,. 
\end{equation}
\end{lemma}

\begin{proof}[Proof of Lemma \ref{lem:double_layer}]
Note that although we do not actually need an estimate for the full $C^{1,\gamma}(\Gamma_\epsilon)$ norm of $\mc{R}_{\mc{D}}$, as $C_s^{1,\gamma}$ will suffice, we show the full bound here because the double layer estimates will be used later on in bounding the full Neumann data $\bw$ (Lemma \ref{lem:full_neumann}).

We begin by introducing some notation. Given a vector-valued density $\bm{\psi}\in C^{0,\gamma}(\Gamma_\epsilon)$ and matrix-valued kernel $\bm{H}$ defined along $\Gamma_\epsilon$, we denote 
\begin{equation}\label{eq:Wdef}
\begin{aligned}
\bm{W}(\bm{\psi};\bm{H}) &=\Phi^{-1}\,{\rm p.v.}\int_{-1/2}^{1/2}\int_{-\pi}^{\pi} \bm{H}(s,\bars,\theta,\bartheta)\bm{\psi}(s-\bars,\theta-\bartheta)\,\epsilon \,d\bartheta\,d\bars \\
&\qquad - 
{\rm p.v.}\int_{-1/2}^{1/2}\int_{-\pi}^{\pi} \overline{\bm{H}}(s,\bars,\theta,\bartheta)(\Phi^{-1}\bm{\psi})(s-\bars,\theta-\bartheta)\,\epsilon\, d\bartheta\,d\bars\,,
\end{aligned}
\end{equation}
where $\overline{\bm{H}}$ is the corresponding matrix kernel defined along the straight filament $\p\mc{C}_\epsilon$. Noting that $\bm{W}(\bm{\psi};\bm{H})$ lives on straight basis vectors, we use $W_j$ to denote the scalar-valued kernel resulting from dotting $\bm{W}(\bm{\psi};\bm{H})$ with the straight basis vector $\overline{\be}_j$, $j\in \{{\rm t},r,\theta\}$: 
\begin{equation}\label{eq:Wj_def}
\begin{aligned}
\bm{W}(\bm{\psi};\bm{H})\cdot\overline{\be}_j &= {\rm p.v.}\int_{-1/2}^{1/2}\int_{-\pi}^{\pi}
W_j(\bm{\psi};\bm{H})\,\epsilon\, d\bartheta d\bars\,;\\ 
W_j(\bm{\psi};\bm{H})&:= (\bm{H}\bm{\psi})\cdot\be_j(s)-\overline{\be}_j\cdot(\overline{\bm{H}}\Phi^{-1}\bm{\psi})\,.
\end{aligned}
\end{equation}
We again note that $\overline{\be}_{\rm t}=\be_z$, but here we will use the overline notation for convenience. \\

Recalling the definition of $\mc{R}_{\mc{D}}$ from \eqref{eq:RS_RD_def} and rewriting the integrands in terms of $(s,\bars,\theta,\bartheta)$, we may write
\begin{align*}
\mc{R}_{\mc{D}}[\bm{\psi}(s,\theta)] &= \Phi^{-1}\mc{D}[\bm{\psi}(s,\theta)]-\overline{\mc{D}}[(\Phi^{-1}\bm{\psi})(s,\theta)] \\
&= \mc{R}_{\mc{D},0}[\bm{\psi}] + \mc{R}_{\mc{D},1}[\bm{\psi}] + \mc{R}_{\mc{D},2}[\bm{\psi}]\,, \\
\mc{R}_{\mc{D},0}[\bm{\psi}] &= -\bigg(\int_{-\infty}^{-1/2}+\int_{1/2}^\infty\bigg)\int_{-\pi}^\pi \overline{\bm{K}}_{\mc D}(s,\bars,\theta,\bartheta)(\Phi^{-1}\bm{\psi})(s-\bars,\theta-\bartheta) \, \epsilon \,d\bartheta d\bars \\
\mc{R}_{\mc{D},1}[\bm{\psi}] &= \bm{W}(\bm{\psi};\bm{K}_{\mc D}) \\
\mc{R}_{\mc{D},2}[\bm{\psi}] &= -\Phi^{-1}\int_{-1/2}^{1/2}\int_{-\pi}^\pi \bm{K}_{\mc D}(s,\bars,\theta,\bartheta)\bm{\psi}(s-\bars,\theta-\bartheta)\, \epsilon^2\wh\kappa\,d\bartheta d\bars\,.
\end{align*}
Here the form of $\mc{R}_{\mc{D},1}$ is as defined in \eqref{eq:Wdef} and the terms $\mc{R}_{\mc{D},0}$ and $\mc{R}_{\mc{D},2}$ arise due to the integration limits in the straight setting and the Jacobian factor \eqref{eq:jacfac}, respectively. 

As we did for the single layer, we begin with bounds for the smooth remainder $\mc{R}_{\mc{D},0}$ away from the singularity at $(\bars,\bartheta)=(0,0)$. Using the form \eqref{eq:stresslet} of the straight kernel $\overline{\bm{K}}_{\mc D}$ and the expression \eqref{eq:barR} for $\barR$, we have 
\begin{equation}\label{eq:RD0_bd}
\norm{\mc{R}_{\mc{D},0}[\bm{\psi}]}_{C^2} \le c\,\epsilon\norm{\bm{\psi}}_{L^\infty}\int_{1/2}^{\infty} \bigg( \frac{1}{\abs{\bars}^2}+ \frac{1}{\abs{\bars}^3} +\frac{1}{\abs{\bars}^4}\bigg) \, d\bars \le c\,\epsilon\norm{\bm{\psi}}_{L^\infty}\,.
\end{equation}
In addition, given two nearby curves $\X^{(a)}$, $\X^{(b)}$ satisfying Lemma \ref{lem:XaXb_C2beta}, the corresponding remainders $\mc{R}_{\mc{D},0}^{(a)}$ and $\mc{R}_{\mc{D},0}^{(b)}$ are nearly identical except for the map $\Phi$ identifying the tangential and normal components of $\bm{\psi}$ with straight basis vectors. In particular, we may estimate 
\begin{equation}\label{eq:RD0_bd_lip}
\begin{aligned}
\norm{\mc{R}_{\mc{D},0}^{(a)}[\bm{\psi}]-\mc{R}_{\mc{S},0}^{(b)}[\bm{\psi}]}_{C^2}
&= \norm{\bigg(\int_{-\infty}^{-1/2}+\int_{1/2}^\infty\bigg)\int_{-\pi}^\pi \overline{\bm{K}}_\mc{D}\,((\Phi^{(a)})^{-1}\bm{\psi}-(\Phi^{(b)})^{-1}\bm{\psi}) \, \epsilon d\bartheta d\bars}_{C^2}\\
&\le c\,\epsilon\norm{(\Phi^{(a)})^{-1}\bm{\psi}-(\Phi^{(b)})^{-1}\bm{\psi}}_{L^\infty} \\
&\le c(\kappa_*^{(a)},\kappa_*^{(b)})\,\epsilon \norm{\X^{(a)}-\X^{(b)}}_{C^2}\norm{\bm{\psi}}_{L^\infty}\,,
\end{aligned}
\end{equation}
by \eqref{eq:Phi_lip_1}. 

For the remainders $\mc{R}_{{\mc D},1}$ and $\mc{R}_{{\mc D},2}$, we start with $L^\infty$ bounds. Using the form \eqref{eq:stresslet} of $\bm{K}_{\mc D}$, we may write the integrands $W_j$ given by \eqref{eq:Wj_def} as 
\begin{align*}
&W_j(\bm{\psi};\bm{K}_{\mc D}) = \frac{3}{4\pi}\bigg(
\frac{\be_j(s)\cdot\bR}{|\bR|^5}\,(\bR\cdot\bm{n}')\big(\bR\cdot\bm{\psi} - \barR\cdot(\Phi^{-1}\bm{\psi})\big) \\
&\qquad + \frac{\be_j(s)\cdot\bR}{|\bR|^5} \big(\bR\cdot\bm{n}' - \barR\cdot\overline{\bm{n}}'\big)\,\barR\cdot(\Phi^{-1}\bm{\psi}) 
+ \frac{(\be_j(s)\cdot\bR)-(\overline{\be}_j\cdot\barR)}{|\bR|^5} (\barR\cdot\overline{\bm{n}}')\,\barR\cdot(\Phi^{-1}\bm{\psi})\\
&\qquad\qquad\qquad +\bigg(\frac{1}{|\bR|^5} - \frac{1}{|\barR|^5} \bigg)(\overline{\be}_j\cdot\barR)(\barR\cdot\overline{\bm{n}}')\,\barR\cdot(\Phi^{-1}\bm{\psi}) \bigg)\,.
\end{align*}
Using the expansions \eqref{eq:RdotN} and \eqref{eq:Rdots} along with \eqref{eq:Rdiffk} and Lemma \ref{lem:Rests}, for each $j=t,r,\theta$, we may bound 
\begin{align*}
\abs{W_j(\bm{\psi};\bm{K}_{\mc D})} &\le c(\kappa_*,c_\Gamma)\,\norm{\bm{\psi}}_{L^\infty}\frac{1}{|\barR|}\,.
\end{align*}
Using the above integrands in the definition \eqref{eq:Wj_def} of $\bm{W}(\bm{\psi};\bm{K}_{\mc D})=\mc{R}_{\mc{D},1}$, by Lemma \ref{lem:basic_est} we have
\begin{equation}\label{eq:RD1_inftybd}
\norm{\mc{R}_{\mc{D},1}[\bm{\psi}]}_{L^\infty} \le c(\kappa_*,c_\Gamma)\,\epsilon \norm{\bm{\psi}}_{L^\infty}\,.
\end{equation}
In addition, for two filaments with centerlines $\X^{(a)}$ and $\X^{(b)}$ satisfying Lemma \ref{lem:XaXb_C2beta}, we may use the above expansion of $W_j(\bm{\psi};\bm{K}_{\mc D})$ along with \eqref{eq:Rdiffk}, \eqref{eq:RdotN}, \eqref{eq:Rdots}, and the bounds \eqref{eq:Qj_Cbeta} for each remainder term $Q$ appearing in the expansion to obtain 
\begin{equation}\label{eq:RD1_inftybd_lip}
\norm{\mc{R}_{\mc{D},1}^{(a)}[\bm{\psi}]-\mc{R}_{\mc{D},1}^{(b)}[\bm{\psi}]}_{L^\infty} \le c(\kappa_*^{(a)},\kappa_*^{(b)},c_\Gamma)\,\epsilon \norm{\X^{(a)}-\X^{(b)}}_{C^2}\norm{\bm{\psi}}_{L^\infty}\,,
\end{equation}
by Lemma \ref{lem:basic_est}. 
Furthermore, using the form \eqref{eq:stresslet} of $\bm{K}_{\mc D}$ and \eqref{eq:RdotN}, we have
\begin{equation}\label{eq:RD2_linfty}
\begin{aligned}
\norm{\mc{R}_{\mc{D},2}[\bm{\psi}]}_{L^\infty} &\le c(\kappa_*,c_\Gamma)\,\norm{\bm{\varphi}}_{L^\infty}\int_{-1/2}^{1/2}\int_{-\pi}^\pi \frac{\epsilon^{-1}}{|\barR|} \,\epsilon^2\,d\bartheta d\bars 
\le \epsilon \, c(\kappa_*,c_\Gamma)\,\norm{\bm{\psi}}_{L^\infty}\,,\\
\norm{\mc{R}_{\mc{D},2}^{(a)}[\bm{\psi}]-\mc{R}_{\mc{D},2}^{(b)}[\bm{\psi}]}_{L^\infty} &\le  c(\kappa_*^{(a)},\kappa_*^{(b)},c_\Gamma)\,\epsilon \norm{\X^{(a)}-\X^{(b)}}_{C^2}\norm{\bm{\psi}}_{L^\infty}\,.
\end{aligned}
\end{equation}

We now turn to $\dot C^{0,\gamma}$ estimates for the components $\frac{1}{\epsilon}\p_\theta\mc{R}_{{\mc D},1}$, $\frac{1}{\epsilon}\p_\theta\mc{R}_{{\mc D},2}$, $\p_s\mc{R}_{{\mc D},1}$, $\p_s\mc{R}_{{\mc D},2}$ of the gradient of $\mc{R}_{\mc{D}}$ along the surface $\Gamma_\epsilon$.
We note the following expansions:
\begin{equation}\label{eq:pthetaR}
\begin{aligned}
\textstyle \frac{1}{\epsilon}\p_\theta\bR &= \textstyle \frac{1}{\epsilon}\p_\theta\bR_{\rm t}=  \be_\theta(s,\theta)\\
\textstyle \frac{1}{\epsilon}\p_\theta\bR\cdot\bR &= \epsilon\sin(\bartheta)+\bars^2Q_{\rm R\theta1}+ \epsilon\bars Q_{\rm R\theta2}\\
\textstyle \frac{1}{\epsilon}\p_\theta\bR \cdot\bm{n}' &= \textstyle \frac{1}{\epsilon}\p_\theta\bR\cdot\be_r(s-\bars,\theta-\bartheta) = -\sin(\bartheta)+\bars Q_{\rm R\theta3} \\ 
\textstyle \frac{1}{\epsilon}\p_\theta\bR \cdot\bm{n} &= \textstyle \frac{1}{\epsilon}\p_\theta\bR\cdot\be_r(s,\theta) = 0 \,, 
\qquad \frac{1}{\epsilon}\p_\theta\bR \cdot\be_{\rm t}(s) = 0\,,
\qquad \frac{1}{\epsilon}\p_\theta\bR \cdot\be_\theta(s,\theta) = 1 \,.
\end{aligned}
\end{equation}
Here the terms $Q_{{\rm R\theta}j}$ each satisfy bounds of the form \eqref{eq:Qj_Cbeta}.

We start with the $\theta$-derivatives $\frac{1}{\epsilon}\p_\theta\mc{R}_{\mc{D},1}$ and $\frac{1}{\epsilon}\p_\theta\mc{R}_{\mc{D},2}$, which may be written as follows: 
\begin{align*}
\textstyle \frac{1}{\epsilon}\p_\theta\mc{R}_{\mc{D},1}[\bm{\psi}]&= \bm{W}[\bm{\psi};\textstyle \frac{1}{\epsilon}\p_\theta\bm{K}_{\mc D}] \\
\textstyle \frac{1}{\epsilon}\p_\theta\mc{R}_{\mc{D},2}[\bm{\psi}]&= -\Phi^{-1}\,{\rm p.v.}\int_{-1/2}^{1/2}\int_{-\pi}^\pi \textstyle \frac{1}{\epsilon}\p_\theta\bm{K}_{\mc D}(s,\bars,\theta,\bartheta)\bm{\psi}(s-\bars,\theta-\bartheta)\, \epsilon^2\wh\kappa\,d\bartheta d\bars\,.
\end{align*}
We begin by estimating $\frac{1}{\epsilon}\p_\theta\mc{R}_{\mc{D},1}[\bm{\psi}]$. Using the definition \eqref{eq:Wj_def} of the integrands $W_j$, we may write $W_{\rm t}(\bm{\psi};\frac{1}{\epsilon}\p_\theta\bm{K}_{\mc D})$ as
\begin{align*} 
W_{\rm t}(\bm{\psi};&\textstyle \frac{1}{\epsilon}\p_\theta\bm{K}_{\mc D}) = \displaystyle \frac{3}{4\pi}\big(W_{\rm t,1}+W_{\rm t,2}+W_{\rm t,3} \big)\,, \\
W_{\rm t,1} &= \bigg(\frac{(\bR\cdot\be_{\rm t})(\bR\cdot\bm{n}')}{|\bR|^5} - \frac{(\barR\cdot\be_z)(\barR\cdot\overline{\bm{n}}')}{|\barR|^5}\bigg)(\be_\theta\cdot\bm{\psi}) \\
W_{\rm t,2} &= \frac{(\bR\cdot\be_{\rm t})(\bR\cdot\bm{\psi}) (\frac{1}{\epsilon}\p_\theta\bR\cdot\bm{n}')}{|\bR|^5} - \frac{(\barR\cdot\be_z)(\barR\cdot(\Phi^{-1}\bm{\psi})) (\frac{1}{\epsilon}\p_\theta\barR\cdot\overline{\bm{n}}')}{|\barR|^5} \\
W_{\rm t,3} &= -\frac{5(\bR\cdot\be_{\rm t})(\bR\cdot\bm{\psi}) (\bR\cdot\bm{n}')(\frac{1}{\epsilon}\p_\theta\bR\cdot\bR)}{|\bR|^7} + \frac{5(\barR\cdot\be_z)(\barR\cdot(\Phi^{-1}\bm{\psi})) (\barR\cdot\overline{\bm{n}}')(\frac{1}{\epsilon}\p_\theta\barR\cdot\barR)}{|\barR|^7}\,.
\end{align*}
Using the expansions \eqref{eq:RdotN}, \eqref{eq:Rdots}, and \eqref{eq:pthetaR}, we may write the components $W_{{\rm t},j}$ in more detail:
\begin{align*}
W_{\rm t,1} &=\bigg(\frac{(\epsilon\bars\wh\kappa + \bars^2Q_{\rm Rt})(-2\epsilon\sin^2(\frac{\bartheta}{2})+\bars^2Q_{\rm Rn'})+ \bars^3Q_{\rm Rn'}}{|\bR|^5} - \bigg(\frac{1}{|\bR|^5}-\frac{1}{|\barR|^5} \bigg) \textstyle 2\epsilon \bars \sin^2(\frac{\bartheta}{2})\bigg)(\be_\theta\cdot\bm{\psi}) \\
W_{\rm t,2} 
&=\frac{(\epsilon\bars\wh\kappa+\bars^2Q_{\rm Rt})(-\sin(\bartheta)+\bars Q_{\rm R\theta3}) + \bars^2 Q_{\rm R\theta3}}{|\bR|^5}\textstyle (\bars\wt{\bm{Q}}_{\rm t}\cdot\bm{\psi}+ 2\epsilon\sin(\frac{\bartheta}{2})\be_\theta\cdot\bm{\psi})  \\
&\quad+\frac{\bars\sin(\bartheta)(\bars^2\bm{Q}_{\rm t}\cdot{\bm\psi}-\epsilon\bars\bm{Q}_r\cdot\bm{\psi}) }{|\bR|^5} 
-\bigg(\frac{1}{|\bR|^5}-\frac{1}{|\barR|^5} \bigg) \textstyle \bars\sin(\bartheta)\big(\bars\be_{\rm t}\cdot\bm{\psi} + 2\epsilon\sin(\frac{\bartheta}{2})\be_\theta\cdot\bm{\psi}\big)   \\
W_{\rm t,3} 
&= \frac{10\epsilon\bars\sin^2(\frac{\bartheta}{2})(\bars\be_{\rm t}\cdot\bm{\varphi} + 2\epsilon\sin(\frac{\bartheta}{2})\be_\theta\cdot\bm{\varphi})(\bars^2Q_{\rm R\theta1}+\epsilon\bars Q_{\rm R\theta2}) }{|\bR|^7} \\
&\quad -\frac{5\bars^3Q_{\rm Rn'}(\bars\be_{\rm t}\cdot\bm{\varphi} + 2\epsilon\sin(\frac{\bartheta}{2})\be_\theta\cdot\bm{\varphi}) (\epsilon\sin(\bartheta)+\bars^2Q_{\rm R\theta1}+ \epsilon\bars Q_{\rm R\theta2})}{|\bR|^7} \\
&\quad +\frac{5\bars (\bars^2\bm{Q}_{\rm t}\cdot\bm{\psi}-\epsilon\bars\bm{Q}_r\cdot\bm{\psi})(-2\epsilon\sin^2(\frac{\bartheta}{2})+\bars^2Q_{\rm n'})(\epsilon\sin(\bartheta)+\bars^2Q_{\rm R\theta1}+ \epsilon\bars Q_{\rm R\theta2})}{|\bR|^7} \\
& + \frac{5(\epsilon\bars\wh\kappa+\bars^2Q_{\rm Rt})(\bars\wt{\bm{Q}}_{\rm t}\cdot\bm{\psi}+ 2\epsilon\sin(\frac{\bartheta}{2})\be_\theta\cdot\bm{\psi}) (2\epsilon\sin^2(\frac{\bartheta}{2})-\bars^2Q_{\rm n'})(\epsilon\sin(\bartheta)+\bars^2Q_{\rm R\theta1}+ \epsilon\bars Q_{\rm R\theta2})}{|\bR|^7}  \\
&\quad -\bigg(\frac{1}{|\bR|^7}-\frac{1}{|\barR|^7} \bigg)\textstyle 10\epsilon^2\bars\sin^2(\frac{\bartheta}{2})\sin(\bartheta)(\bars\be_{\rm t}\cdot\bm{\psi} + 2\epsilon\sin(\frac{\bartheta}{2})\be_\theta\cdot\bm{\psi}) \,,
\end{align*}
where $\wt{\bm{Q}}_{\rm t}= \be_{\rm t}-\bars\bm{Q}_{\rm t} + \epsilon\bm{Q}_r$. Using the expansion \eqref{eq:RaRb_diff}, \eqref{eq:Rdiffk} for $\frac{1}{|\bR|^5}-\frac{1}{|\barR|^5}$ and $\frac{1}{|\bR|^7}-\frac{1}{|\barR|^7}$, we may apply Lemma \ref{lem:alpha_est} to obtain 
\begin{equation}\label{eq:Wz_theta}
\begin{aligned}
\textstyle \abs{\bm{W}(\bm{\psi};\frac{1}{\epsilon}\p_\theta\bm{K}_{\mc D})\cdot\be_z}_{\dot C^{0,\gamma}} 
&\le c(\kappa_*,c_\Gamma)\max_j\big(\norm{Q_j}_{C^{0,\gamma^+}_1}+\norm{Q_j}_{C^{0,\gamma}_2} \big)\norm{\bm{\psi}}_{C^{0,\gamma}} \\
&\le c(\kappa_{*,\gamma^+},c_\Gamma)\,\epsilon^{-\gamma^+}\norm{\bm{\psi}}_{C^{0,\gamma}}\,.
\end{aligned}
\end{equation}
Here the $\max_j$ is taken over all forms of remainder terms appearing in the above expansions. We also note that terms of the form, for example, $\frac{\bars^4\epsilon\sin^2(\frac{\bartheta}{2})Q}{\abs{\bR}}$ may be treated as $\frac{\bars^4\epsilon\sin(\frac{\bartheta}{2})\wt Q}{\abs{\bR}}$, $\wt Q=\sin(\frac{\bartheta}{2})Q$, so as not to lose a factor of $\epsilon^{-\gamma}$ from using case (1) of Lemma \ref{lem:alpha_est}. 

Furthermore, using the above expansions of $W_{{\rm t},j}$ and the Lipschitz estimates \eqref{eq:Qj_Cbeta} for each remainder $Q_j$, we additionally note that, given two curves $\X^{(a)}$ and $\X^{(b)}$ satisfying Lemma \ref{lem:XaXb_C2beta}, we have  
\begin{equation}\label{eq:Wz_theta_lip}
\begin{aligned}
\textstyle \abs{\bm{W}(\bm{\psi};\frac{1}{\epsilon}\p_\theta(\bm{K}_{\mc D}^{(a)}-\bm{K}_{\mc D}^{(b)}))\cdot\be_z}_{\dot C^{0,\gamma}}
&\le c(\kappa_{*,\gamma^+}^{(a)},\kappa_{*,\gamma^+}^{(b)},c_\Gamma)\,\epsilon^{-\gamma^+}\norm{\X^{(a)}-\X^{(b)}}_{C^{2,\gamma^+}}\norm{\bm{\psi}}_{C^{0,\gamma}}\,.
\end{aligned}
\end{equation}

We next consider the integrand $W_r(\bm{\psi};\frac{1}{\epsilon}\p_\theta\bm{K}_{\mc D})$ as defined in \eqref{eq:Wj_def}. We may write
\begin{align*} 
\textstyle W_r(\bm{\psi};\frac{1}{\epsilon}\p_\theta\bm{K}_{\mc D}) &= \frac{3}{4\pi}\big(W_{r,1}+W_{r,2}+W_{r,3} \big)\,, \\
W_{r,1} &= \bigg(\frac{(\bR\cdot\be_r)\bR\cdot\bm{n}'}{|\bR|^5} - \frac{(\barR\cdot\overline{\be}_r)\barR\cdot\overline{\bm{n}}'}{|\barR|^5}\bigg)(\be_\theta\cdot\bm{\psi}) \\
W_{r,2} &= \frac{(\bR\cdot\be_r)(\bR\cdot\bm{\psi}) (\frac{1}{\epsilon}\p_\theta\bR\cdot\bm{n}')}{|\bR|^5}-\frac{(\barR\cdot\overline{\be}_r)(\barR\cdot(\Phi^{-1}\bm{\psi})) (\frac{1}{\epsilon}\p_\theta\barR\cdot\overline{\bm{n}}')}{|\barR|^5}   \\
W_{r,3} &= -\frac{5(\bR\cdot\be_r)(\bR\cdot\bm{\psi}) (\bR\cdot\bm{n}')(\frac{1}{\epsilon}\p_\theta\bR\cdot\bR)}{|\bR|^7} + \frac{5(\barR\cdot\overline{\be}_r)(\barR\cdot(\Phi^{-1}\bm{\psi})) (\barR\cdot\overline{\bm{n}}')(\frac{1}{\epsilon}\p_\theta\barR\cdot\barR)}{|\barR|^7}\,.
\end{align*}
Again by \eqref{eq:RdotN}, \eqref{eq:Rdots}, and \eqref{eq:pthetaR}, we may further expand each $W_{r,j}$ as
\begin{align*}
W_{r,1} &= \bigg(\frac{-2\epsilon\bars^2\sin^2(\frac{\bartheta}{2})(Q_{\rm Rn}+Q_{\rm Rn'})+\bars^4Q_{\rm Rn}Q_{\rm Rn'} }{|\bR|^5}
  - \textstyle 4\epsilon^2\sin^4(\frac{\bartheta}{2}) \displaystyle\bigg(\frac{1}{|\bR|^5} - \frac{1}{|\barR|^5}\bigg)\bigg)(\be_\theta\cdot\bm{\psi})\\
W_{r,2} &= \frac{\bars^2Q_{\rm Rn}(\bars\wt{\bm{Q}}_{\rm t}\cdot\bm{\psi}+ 2\epsilon\sin(\frac{\bartheta}{2})\be_\theta\cdot\bm{\psi}) (-\sin(\bartheta)+\bars Q_{\rm R\theta3})}{|\bR|^5} \\
&\quad + \frac{2\epsilon\sin^2(\frac{\bartheta}{2})\big[(\bars^2\bm{Q}_{\rm t}\cdot\bm{\psi}  - \epsilon\bars\bm{Q}_r\cdot\bm{\psi}) (\sin(\bartheta)-\bars Q_{\rm R\theta3})+ \bars Q_{\rm R\theta3}(\bars\be_{\rm t}\cdot\bm{\psi} + 2\epsilon\sin(\frac{\bartheta}{2})\be_\theta\cdot\bm{\psi}) \big]}{|\bR|^5} \\
&\qquad - \bigg(\frac{1}{|\bR|^5}-\frac{1}{|\barR|^5} \bigg)\textstyle 2\epsilon\sin^2(\frac{\bartheta}{2})\sin(\bartheta)(\bars\be_{\rm t}\cdot\bm{\psi} + 2\epsilon\sin(\frac{\bartheta}{2})\be_\theta\cdot\bm{\psi})  \\
W_{r,3}&= \frac{5\bars^2Q_{\rm Rn}(\bars\wt{\bm{Q}}_{\rm t}\cdot\bm{\psi}+ 2\epsilon\sin(\frac{\bartheta}{2})\be_\theta\cdot\bm{\psi})(2\epsilon\sin^2(\frac{\bartheta}{2})-\bars^2Q_{\rm Rn'})(\epsilon\sin(\bartheta)+\bars^2Q_{\rm R\theta1}+ \epsilon\bars Q_{\rm R\theta2})}{|\bR|^7} \\
&\quad -\frac{10\epsilon\sin^2(\frac{\bartheta}{2})(\bars^2\bm{Q}_{\rm t}\cdot\bm{\psi} - \epsilon\bars\bm{Q}_r\cdot\bm{\psi})(2\epsilon\sin^2(\frac{\bartheta}{2})-\bars^2Q_{\rm Rn'})(\epsilon\sin(\bartheta)+\bars^2Q_{\rm R\theta1}+ \epsilon\bars Q_{\rm R\theta2})}{|\bR|^7}\\
&\qquad -\frac{10\epsilon\sin^2(\frac{\bartheta}{2})\bars^2Q_{\rm Rn'}(\bars\be_{\rm t}\cdot\bm{\psi} + 2\epsilon\sin(\frac{\bartheta}{2})\be_\theta\cdot\bm{\psi})(\epsilon\sin(\bartheta)+\bars^2Q_{\rm R\theta1}+ \epsilon\bars Q_{\rm R\theta2})}{|\bR|^7} \\
&\qquad + \frac{20\epsilon^2\sin^4(\frac{\bartheta}{2})(\bars^2Q_{\rm R\theta1}+ \epsilon\bars Q_{\rm R\theta2})(\bars\be_{\rm t}\cdot\bm{\psi} + 2\epsilon\sin(\frac{\bartheta}{2})\be_\theta\cdot\bm{\psi}) }{|\bR|^7} \\
&\qquad +\bigg(\frac{1}{|\bR|^7}-\frac{1}{|\barR|^7}\bigg)\textstyle 20\epsilon^3\sin^4(\frac{\bartheta}{2})\sin(\bartheta) (\bars\be_{\rm t}\cdot\bm{\psi} + 2\epsilon\sin(\frac{\bartheta}{2})\be_\theta\cdot\bm{\psi}) \,.
\end{align*}
Here again $\wt{\bm{Q}}_{\rm t}= \be_{\rm t}-\bars\bm{Q}_{\rm t} + \epsilon\bm{Q}_r$, and we may use the expansion \eqref{eq:RaRb_diff}, \eqref{eq:Rdiffk} for $\frac{1}{|\bR|^5}-\frac{1}{|\barR|^5}$ and $\frac{1}{|\bR|^7}-\frac{1}{|\barR|^7}$. We may then apply Lemma \ref{lem:alpha_est} to obtain 
\begin{equation}\label{eq:Wr_theta}
\begin{aligned}
\textstyle \abs{\bm{W}(\bm{\psi};\frac{1}{\epsilon}\p_\theta\bm{K}_{\mc D})\cdot\overline{\be}_r}_{\dot C^{0,\gamma}} 
&\le c(\kappa_*,c_\Gamma)\max_j\big(\norm{Q_j}_{C^{0,\gamma^+}_1}+\norm{Q_j}_{C^{0,\gamma}_2} \big)\norm{\bm{\psi}}_{C^{0,\gamma}} \\
&\le c(\kappa_{*,\gamma^+},c_\Gamma)\,\epsilon^{-\gamma^+}\norm{\bm{\psi}}_{C^{0,\gamma}}\,, \\
\textstyle \abs{\bm{W}(\bm{\psi};\frac{1}{\epsilon}\p_\theta(\bm{K}_{\mc D}^{(a)}-\bm{K}_{\mc D}^{(b)}))\cdot\overline{\be}_r }_{\dot C^{0,\gamma}} 
&\le c(\kappa_{*,\gamma^+}^{(a)},\kappa_{*,\gamma^+}^{(b)},c_\Gamma)\,\epsilon^{-\gamma^+}\norm{\X^{(a)}-\X^{(b)}}_{C^{2,\gamma^+}}\norm{\bm{\psi}}_{C^{0,\gamma}}\,.
\end{aligned} 
\end{equation}
Where necessary, we again may exchange some regularity for a factor of $\epsilon^\gamma$ by treating terms of the form, e.g., $\frac{\bars^4\epsilon\sin^2(\frac{\bartheta}{2})Q}{\abs{\bR}}$ as $\frac{\bars^4\epsilon\sin(\frac{\bartheta}{2})\wt Q}{\abs{\bR}}$ with $\wt Q=\sin(\frac{\bartheta}{2})Q$. Then case (2) of Lemma \ref{lem:alpha_est} yields the bound \eqref{eq:Wr_theta}.

Finally we turn to the integrand $W_\theta(\bm{\psi};\frac{1}{\epsilon}\p_\theta\bm{K}_{\mc D})$ as defined in \eqref{eq:Wj_def}. We may write
\begin{align*} 
 W_\theta(\bm{\psi};&\textstyle\frac{1}{\epsilon}\p_\theta\bm{K}_{\mc D}) = \displaystyle\frac{3}{4\pi}\big(W_{\theta,1}+W_{\theta,2}+W_{\theta,3}+W_{\theta,4} \big)\,, \\
W_{\theta,1} &= \frac{(\bR\cdot\bm{\psi})(\bR\cdot\bm{n}')}{|\bR|^5}-\frac{(\barR\cdot(\Phi^{-1}\bm{\psi}))(\barR\cdot\overline{\bm{n}}')}{|\barR|^5} \\
W_{\theta,2} &= \frac{(\bR\cdot\be_\theta)(\frac{1}{\epsilon}\p_\theta\bR\cdot\bm{\psi})(\bR\cdot\bm{n}')}{|\bR|^5} - \frac{(\barR\cdot\overline{\be}_\theta)(\frac{1}{\epsilon}\p_\theta\barR\cdot(\Phi^{-1}\bm{\psi}))(\barR\cdot\overline{\bm{n}}')}{|\barR|^5} \\
W_{\theta,3} &= \frac{(\bR\cdot\be_\theta)(\bR\cdot\bm{\psi}) (\frac{1}{\epsilon}\p_\theta\bR\cdot\bm{n}')}{|\bR|^5} - \frac{(\barR\cdot\overline{\be}_\theta)(\barR\cdot(\Phi^{-1}\bm{\psi})) (\frac{1}{\epsilon}\p_\theta\barR\cdot\overline{\bm{n}}')}{|\barR|^5} \\
W_{\theta,4} &= -\frac{5(\bR\cdot\be_\theta)(\bR\cdot\bm{\psi}) (\bR\cdot\bm{n}')(\frac{1}{\epsilon}\p_\theta\bR\cdot\bR)}{|\bR|^7}
+ \frac{5(\barR\cdot\overline{\be}_\theta)(\barR\cdot(\Phi^{-1}\bm{\psi})) (\barR\cdot\overline{\bm{n}}')(\frac{1}{\epsilon}\p_\theta\barR\cdot\barR)}{|\barR|^7}\,.
\end{align*}
We may again use \eqref{eq:RdotN}, \eqref{eq:Rdots}, and \eqref{eq:pthetaR} to expand each $W_{\theta,j}$ as
\begin{align*}
W_{\theta,1} &= \frac{(2\epsilon\sin^2(\frac{\bartheta}{2})-\bars^2Q_{\rm Rn'})(\bars^2\bm{Q}_{\rm t}\cdot\bm{\psi}- \epsilon\bars\bm{Q}_r\cdot\bm{\psi}) +\bars^2Q_{\rm Rn'}(\bars\be_{\rm t}\cdot\bm{\psi} + 2\epsilon\sin(\frac{\bartheta}{2})\be_\theta\cdot\bm{\psi})}{|\bR|^5} \\
&\qquad - \bigg(\frac{1}{|\bR|^5}-\frac{1}{|\barR|^5}\bigg)\textstyle 2\epsilon\sin^2(\frac{\bartheta}{2})(\bars\be_{\rm t}\cdot\bm{\psi} + 2\epsilon\sin(\frac{\bartheta}{2})\be_\theta\cdot\bm{\psi}) \\
W_{\theta,2} &= \frac{(\epsilon\bars\kappa_3\cos(\theta-\bartheta)+\bars^2Q_{\rm R\theta})(-2\epsilon\sin^2(\frac{\bartheta}{2})+\bars^2Q_{\rm Rn'})+\epsilon\sin(\bartheta)\bars^2Q_{\rm Rn'}}{|\bR|^5}(\bm{\psi}\cdot\be_\theta) \\
&\qquad +\bigg(\frac{1}{|\bR|^5}-\frac{1}{|\barR|^5} \bigg)\textstyle \epsilon\sin(\bartheta)(-2\epsilon\sin^2(\frac{\bartheta}{2})+\bars^2Q_{\rm Rn'})(\bm{\psi}\cdot\be_\theta) \\
W_{\theta,3} &= \frac{(\epsilon\bars\kappa_3\cos(\theta-\bartheta)+\bars^2Q_{\rm R\theta})(\bars\wt{\bm{Q}}_{\rm t}\cdot\bm{\psi}+ 2\epsilon\sin(\frac{\bartheta}{2})\be_\theta\cdot\bm{\psi}) (-\sin(\bartheta)+\bars Q_{\rm R\theta3})}{|\bR|^5} \\
&\quad + \frac{ \epsilon\sin(\bartheta)\bars\big[(\sin(\bartheta)-\bars Q_{\rm R\theta3})(\bars\bm{Q}_{\rm t}\cdot\bm{\psi}- \epsilon\bm{Q}_r\cdot\bm{\psi}) + Q_{\rm R\theta3}(\bars\be_{\rm t}\cdot\bm{\psi} + 2\epsilon\sin(\frac{\bartheta}{2})\be_\theta\cdot\bm{\psi})\big]}{|\bR|^5} \\
&\qquad - \bigg(\frac{1}{|\bR|^5}-\frac{1}{|\barR|^5}\bigg)\textstyle \epsilon\sin^2(\bartheta)(\bars\be_{\rm t}\cdot\bm{\psi} + 2\epsilon\sin(\frac{\bartheta}{2})\be_\theta\cdot\bm{\psi})  \\
W_{\theta,4} &= \frac{5(\bars\wt{\bm{Q}}_{\rm t}\cdot\bm{\psi}+ 2\epsilon\sin(\frac{\bartheta}{2})\be_\theta\cdot\bm{\psi})(\epsilon\bars\kappa_3\cos(\theta-\bartheta)+\bars^2Q_{\rm R\theta}) (2\epsilon\sin^2(\frac{\bartheta}{2})-\bars^2Q_{\rm Rn'})(\epsilon\sin(\bartheta)+\bars^2Q_{\rm R\theta1}+ \epsilon\bars Q_{\rm R\theta2})}{|\bR|^7} \\
&\qquad +\frac{5\epsilon\sin(\bartheta)(\bars^2\bm{Q}_{\rm t}\cdot\bm{\psi}- \epsilon\bars\bm{Q}_r\cdot\bm{\psi})(-2\epsilon\sin^2(\frac{\bartheta}{2})+\bars^2Q_{\rm Rn'})(\epsilon\sin(\bartheta)+\bars^2Q_{\rm R\theta1}+ \epsilon\bars Q_{\rm R\theta2})}{|\bR|^7} \\
&\qquad -\frac{5\epsilon\sin(\bartheta)\bars^2Q_{\rm Rn'}(\bars\be_{\rm t}\cdot\bm{\psi} + 2\epsilon\sin(\frac{\bartheta}{2})\be_\theta\cdot\bm{\psi}) (\epsilon\sin(\bartheta)+\bars^2Q_{\rm R\theta1}+ \epsilon\bars Q_{\rm R\theta2})}{|\bR|^7} \\
&\qquad +\frac{10\epsilon^2\sin^2(\frac{\bartheta}{2})\sin(\bartheta)(\bars^2Q_{\rm R\theta1}+ \epsilon\bars Q_{\rm R\theta2})(\bars\be_{\rm t}\cdot\bm{\psi} + 2\epsilon\sin(\frac{\bartheta}{2})\be_\theta\cdot\bm{\psi}) }{|\bR|^7} \\
&\qquad +\bigg(\frac{1}{|\bR|^7}-\frac{1}{|\barR|^7}\bigg) \textstyle 10\epsilon^3\sin^2(\frac{\bartheta}{2})\sin^2(\bartheta)(\bars\be_{\rm t}\cdot\bm{\psi} + 2\epsilon\sin(\frac{\bartheta}{2})\be_\theta\cdot\bm{\psi}) \,.
\end{align*}
As before, using the expansions \eqref{eq:RaRb_diff}, \eqref{eq:Rdiffk} for $\frac{1}{|\bR|^5}-\frac{1}{|\barR|^5}$ and $\frac{1}{|\bR|^7}-\frac{1}{|\barR|^7}$ and applying Lemma \ref{lem:alpha_est} yields the bounds
\begin{equation}\label{eq:Wthe_theta}
\begin{aligned}
\textstyle \abs{\bm{W}(\bm{\psi};\frac{1}{\epsilon}\p_\theta\bm{K}_{\mc D})\cdot\overline{\be}_\theta}_{\dot C^{0,\gamma}} 
&\le c(\kappa_*,c_\Gamma)\max_j\big(\norm{Q_j}_{C^{0,\gamma^+}_1}+\norm{Q_j}_{C^{0,\gamma}_2} \big)\norm{\bm{\psi}}_{C^{0,\gamma}} \\
&\le c(\kappa_{*,\gamma^+},c_\Gamma)\,\epsilon^{-\gamma^+}\norm{\bm{\psi}}_{C^{0,\gamma}}\,,\\
\textstyle \abs{\bm{W}(\bm{\psi};\frac{1}{\epsilon}\p_\theta(\bm{K}_{\mc D}^{(a)}-\bm{K}_{\mc D}^{(b)}))\cdot\overline{\be}_\theta}_{\dot C^{0,\gamma}} 
&\le c(\kappa_{*,\gamma^+}^{(a)},\kappa_{*,\gamma^+}^{(b)},c_\Gamma)\,\epsilon^{-\gamma^+}\norm{\X^{(a)}-\X^{(b)}}_{C^{2,\gamma^+}} \norm{\bm{\psi}}_{C^{0,\gamma}}\,.
\end{aligned}
\end{equation}
Combining \eqref{eq:Wz_theta}, \eqref{eq:Wz_theta_lip}, \eqref{eq:Wr_theta}, and \eqref{eq:Wthe_theta}, we obtain 
\begin{equation}\label{eq:pthetaRD1_Calpha}
\begin{aligned}
\textstyle \abs{\frac{1}{\epsilon}\p_\theta\mc{R}_{\mc{D},1}[\bm{\psi}]}_{\dot C^{0,\gamma}} &= \textstyle \abs{\bm{W}(\bm{\psi};\frac{1}{\epsilon}\p_\theta\bm{K}_{\mc D})}_{\dot C^{0,\gamma}} 
\le c(\kappa_{*,\gamma^+},c_\Gamma)\,\epsilon^{-\gamma^+}\norm{\bm{\psi}}_{C^{0,\gamma}}\,,\\
\textstyle \abs{\frac{1}{\epsilon}\p_\theta\mc{R}_{\mc{D},1}^{(a)}[\bm{\psi}] - \frac{1}{\epsilon}\p_\theta\mc{R}_{\mc{D},1}^{(b)}[\bm{\psi}]}_{\dot C^{0,\gamma}} &\le c(\kappa_{*,\gamma^+}^{(a)},\kappa_{*,\gamma^+}^{(b)},c_\Gamma)\,\epsilon^{-\gamma^+}\norm{\X^{(a)}-\X^{(b)}}_{C^{2,\gamma^+}}\norm{\bm{\psi}}_{C^{0,\gamma}} \,.
\end{aligned}
\end{equation}

We next turn to $\frac{1}{\epsilon}\p_\theta\mc{R}_{\mc{D},2}[\bm{\psi}]$, which we note may be written 
\begin{align*}
\textstyle \frac{1}{\epsilon}\p_\theta\mc{R}_{\mc{D},2}[\bm{\psi}] &= -\bm{W}[\bm{\psi};\wh\kappa \p_\theta\bm{K}_{\mc D}] - {\rm p.v.}\int_{-1/2}^{1/2}\int_{-\pi}^\pi \textstyle \wh\kappa\p_\theta\overline{\bm{K}}_{\mc D}\,(\Phi^{-1}\bm{\psi})(s-\bars,\theta-\bartheta)\, \epsilon \,d\bartheta d\bars\,.
\end{align*}
Note the additional factor of $\epsilon\wh\kappa$ is due to the Jacobian factor \eqref{eq:jacfac}.
Using that \eqref{eq:pthetaRD1_Calpha} implies
\begin{align*}
\abs{\bm{W}[\bm{\psi};\wh\kappa \p_\theta\bm{K}_{\mc D}]}_{\dot C^{0,\gamma}} &\le c(\kappa_{*,\gamma^+},c_\Gamma)\,\epsilon^{1-\gamma^+}\norm{\bm{\psi}}_{C^{0,\gamma}}\,,\\
\abs{\bm{W}[\bm{\psi};\wh\kappa^{(a)} \p_\theta\bm{K}_{\mc D}^{(a)}-\wh\kappa^{(b)} \p_\theta\bm{K}_{\mc D}^{(b)}]}_{\dot C^{0,\gamma}} &\le c(\kappa_{*,\gamma^+}^{(a)},\kappa_{*,\gamma^+}^{(b)},c_\Gamma)\,\epsilon^{1-\gamma^+}\norm{\X^{(a)}-\X^{(b)}}_{C^{2,\gamma^+}}\norm{\bm{\psi}}_{C^{0,\gamma}}\,,
\end{align*}
it remains to bound the second term of the $\frac{1}{\epsilon}\p_\theta\mc{R}_{\mc{D},2}[\bm{\psi}]$ expression. Using the form \eqref{eq:barR} of $\barR$, we may write 
\begin{align*}
\textstyle \be_z\cdot\big(\wh\kappa\p_\theta \overline{\bm{K}}_{\mc D}(\Phi^{-1}\bm{\psi})\big) 
&= \frac{3\epsilon\wh\kappa}{4\pi} \bigg( -\frac{2\epsilon\sin^2(\frac{\bartheta}{2})\bars(\be_\theta\cdot\bm{\psi})}{|\barR|^5}-\frac{\bars\sin(\bartheta)(\bars\be_{\rm t}\cdot\bm{\psi}+2\epsilon\sin(\frac{\bartheta}{2})\be_\theta\cdot{\bm\psi}) }{|\barR|^5} \\
&\qquad +\frac{10\epsilon^2\sin^2(\frac{\bartheta}{2})\sin(\wh\theta)\bars(\bars\be_{\rm t}\cdot\bm{\psi}+2\epsilon\sin(\frac{\bartheta}{2})\be_\theta\cdot{\bm\psi}) }{|\barR|^7} \bigg)\\
\textstyle \overline{\be}_r\cdot\big(\wh\kappa\p_\theta \overline{\bm{K}}_{\mc D}(\Phi^{-1}\bm{\psi})\big) 
&= \frac{3\epsilon\wh\kappa}{4\pi} \bigg( -\frac{2\epsilon\sin^2(\frac{\bartheta}{2})\big(\bars\be_{\rm t}\cdot\bm{\psi}+2\epsilon\sin(\frac{\bartheta}{2})\be_\theta\cdot\bm{\psi} + \epsilon\sin(\bartheta)\be_\theta\cdot\bm{\psi}\big)}{|\barR|^5} \\
&\hspace{-2cm} -\frac{\epsilon\sin^2(\bartheta)\big(\bars\be_{\rm t}\cdot\bm{\psi}+2\epsilon\sin(\frac{\bartheta}{2})\be_\theta\cdot\bm{\psi}\big) }{|\barR|^5} 
+\frac{10\epsilon^3\sin^2(\frac{\bartheta}{2})\sin^2(\bartheta)\big(\bars\be_{\rm t}\cdot\bm{\psi}+2\epsilon\sin(\frac{\bartheta}{2})\be_\theta\cdot\bm{\psi}\big) }{|\barR|^7} \bigg)\\
\textstyle \overline{\be}_\theta\cdot\big(\wh\kappa\p_\theta \overline{\bm{K}}_{\mc D}(\Phi^{-1}\bm{\psi})\big) 
&= \frac{3\epsilon\wh\kappa}{4\pi} \bigg( -\frac{4\epsilon^2\sin^4(\frac{\bartheta}{2})(\be_\theta\cdot\bm{\psi})}{|\barR|^5}
-\frac{2\epsilon\sin^2(\frac{\bartheta}{2})\sin(\bartheta)\big(\bars\be_{\rm t}\cdot\bm{\psi}+2\epsilon\sin(\frac{\bartheta}{2})\be_\theta\cdot\bm{\psi}\big)}{|\barR|^5} \\
&\qquad +\frac{20\epsilon^3\sin^4(\frac{\bartheta}{2})\sin(\bartheta)\big(\bars\be_{\rm t}\cdot\bm{\psi}+2\epsilon\sin(\frac{\bartheta}{2})\be_\theta\cdot\bm{\psi}\big) }{|\barR|^7} \bigg)\,.
\end{align*}
Using Lemma \ref{lem:alpha_est}, we thus obtain 
\begin{equation}\label{eq:pthetaRD2_Calpha}
\begin{aligned}
\textstyle \abs{\frac{1}{\epsilon}\p_\theta\mc{R}_{\mc{D},2}[\bm{\psi}]}_{\dot C^{0,\gamma}} &\le \abs{\bm{W}[\bm{\psi};\wh\kappa \p_\theta\bm{K}_{\mc D}]}_{\dot C^{0,\gamma}} + \bigg|\int_{-1/2}^{1/2}\int_{-\pi}^\pi \textstyle \wh\kappa\p_\theta\overline{\bm{K}}_{\mc D}\,(\Phi^{-1}\bm{\psi})\, \epsilon \,d\bartheta d\bars\bigg|_{\dot C^{0,\gamma}} \\
&\le c(\kappa_{*,\gamma^+},c_\Gamma)\,\epsilon^{-\gamma^+}\norm{\bm{\psi}}_{C^{0,\gamma}}\,, \\
\textstyle \abs{\frac{1}{\epsilon}\p_\theta\mc{R}_{\mc{D},2}^{(a)}[\bm{\psi}]-\frac{1}{\epsilon}\p_\theta\mc{R}_{\mc{D},2}^{(b)}[\bm{\psi}]}_{\dot C^{0,\gamma}} 
&\le c(\kappa_{*,\gamma^+}^{(a)},\kappa_{*,\gamma^+}^{(b)},c_\Gamma)\,\epsilon^{-\gamma^+}\norm{\X^{(a)}-\X^{(b)}}_{C^{2,\gamma^+}}\norm{\bm{\psi}}_{C^{0,\gamma}} \,. \\
\end{aligned}
\end{equation}


We now turn to $\dot C^{0,\gamma}$ estimates for the $s$-derivatives $\p_s\mc{R}_{{\mc D},1}$ and $\p_s\mc{R}_{{\mc D},2}$ along $\Gamma_\epsilon$. These may be written as 
\begin{align*}
\p_s\mc{R}_{\mc{D},1}[\bm{\psi}]&= \bm{W}[\bm{\psi};\p_s\bm{K}_{\mc D}] + [\p_s,\Phi^{-1}]\int_{-1/2}^{1/2}\int_{-\pi}^\pi\bm{K}_{\mc D}\,\bm{\psi}\,\epsilon\,d\bartheta d\bars  \\
\p_s\mc{R}_{\mc{D},2}[\bm{\psi}]&= 
 -\Phi^{-1}{\rm p.v.}\int_{-1/2}^{1/2}\int_{-\pi}^\pi \p_s\bm{K}_{\mc D}\,\bm{\psi}\, \epsilon^2\wh\kappa\,d\bartheta d\bars 
-[\p_s,\Phi^{-1}]\int_{-1/2}^{1/2}\int_{-\pi}^\pi \bm{K}_{\mc D}\,\bm{\psi}\, \epsilon^2\wh\kappa\,d\bartheta d\bars\,.
\end{align*}

We may estimate $\bm{W}[\bm{\psi};\p_s\bm{K}_{\mc D}]$ by first using the definition \eqref{eq:Wj_def} of the integrands $W_j$ to write 
\begin{align*}
W_{\rm t}(\bm{\psi};&\p_s\bm{K}_{\mc D}) = \frac{3}{4\pi}\big(W_{\rm t,1}+ W_{\rm t,2}+ W_{\rm t,3} \big)\,, \\
W_{\rm t,1}&= \frac{\big((1-\epsilon\wh\kappa)(\bR\cdot\bm{\psi})+(\bR\cdot\be_{\rm t})(\p_s\bR\cdot\bm{\psi})\big)\bR\cdot\bm{n}'}{|\bR|^5} 
- \frac{\big(\barR\cdot(\Phi^{-1}\bm{\psi})+(\barR\cdot\be_z)(\be_{\rm t}\cdot\bm{\psi})\big)\barR\cdot\overline{\bm{n}}'}{|\barR|^5} \\
W_{\rm t,2}&= \frac{(\bR\cdot\be_{\rm t})(\bR\cdot\bm{\psi}) (\p_s\bR\cdot\bm{n}')}{|\bR|^5}
-\frac{(\barR\cdot\be_z)(\barR\cdot(\Phi^{-1}\bm{\psi})) (\p_s\barR\cdot\overline{\bm{n}}')}{|\barR|^5} \\
W_{\rm t,3}&= -\frac{5(\bR\cdot\be_{\rm t})(\bR\cdot\bm{\psi}) (\bR\cdot\bm{n}')\bR\cdot\p_s\bR}{|\bR|^7}
+\frac{5(\barR\cdot\be_z)(\barR\cdot(\Phi^{-1}\bm{\psi})) (\barR\cdot\overline{\bm{n}}')\barR\cdot\p_s\barR}{|\barR|^7}\,.
\end{align*}
Using the expansions \eqref{eq:RdotN}, \eqref{eq:Rdots}, and \eqref{eq:psR}, we may expand each $W_{\rm t,j}$ as follows:
\begin{align*}
W_{\rm t,1}&= \frac{\big( \bars^2\bm{Q}_{\rm t}\cdot\bm{\psi}-\epsilon\bars\bm{Q}_r\cdot\bm{\psi} +\epsilon\wh\kappa(\bars\wt{\bm{Q}}_{\rm t}\cdot\bm{\psi}- 2\epsilon\sin(\frac{\bartheta}{2})\be_\theta\cdot\bm{\psi})\big)(2\epsilon\sin^2(\frac{\bartheta}{2})- \bars^2Q_{\rm Rn'})}{|\bR|^5}\\
& - \frac{\big((\epsilon\bars\wh\kappa+\bars^2Q_{\rm Rt})((1-\epsilon\wh\kappa)\be_{\rm t}\cdot{\bm \psi} + \epsilon\kappa_3\be_\theta\cdot\bm{\psi})-\epsilon\bars(\wh\kappa\be_{\rm t}\cdot{\bm \psi} -\kappa_3\be_\theta\cdot\bm{\psi})\big)(2\epsilon\sin^2(\frac{\bartheta}{2})- \bars^2Q_{\rm Rn'})}{|\bR|^5}\\
& + \frac{2\bars^2Q_{\rm Rn'}\big(\bars\be_{\rm t}\cdot\bm{\psi} + \epsilon\sin(\frac{\bartheta}{2})\be_\theta\cdot\bm{\psi}\big)}{|\bR|^5}
-\bigg(\frac{1}{|\bR|^5}-\frac{1}{|\barR|^5}\bigg) \textstyle 4\epsilon\sin^2(\frac{\bartheta}{2})\big(\bars\be_{\rm t}\cdot\bm{\psi} + \epsilon\sin(\frac{\bartheta}{2})\be_\theta\cdot\bm{\psi}\big)
 \\
W_{\rm t,2}&= \frac{(\epsilon\bars\wh\kappa+\bars^2Q_{\rm Rt})(\bars\wt{\bm{Q}}_{\rm t}\cdot\bm{\psi}+ 2\epsilon\sin(\frac{\bartheta}{2})\be_\theta\cdot\bm{\psi}) (\bars\wh{\kappa} - \kappa_3\epsilon\sin(\bartheta)+\epsilon\bars Q_{\rm Rs3})}{|\bR|^5} \\
&\qquad - \frac{\bars (\bars^2\bm{Q}_{\rm t}\cdot\bm{\psi}-\epsilon\bars\bm{Q}_r\cdot\bm{\psi})(\bars\wh{\kappa} - \kappa_3\epsilon\sin(\bartheta)+\epsilon\bars Q_{\rm Rs3})}{|\bR|^5} \\
&\qquad + \bigg(\frac{1}{|\bR|^5}-\frac{1}{|\barR|^5} \bigg)\textstyle \bars(\bars\be_{\rm t}\cdot\bm{\psi} + 2\epsilon\sin(\frac{\bartheta}{2})\be_\theta\cdot\bm{\psi}) (\bars\wh{\kappa} - \kappa_3\epsilon\sin(\bartheta)+\epsilon\bars Q_{\rm Rs3}) \\
W_{\rm t,3}&= \frac{5(\epsilon\bars\wh\kappa+\bars^2Q_{\rm Rt})(\bars\wt{\bm{Q}}_{\rm t}\cdot\bm{\psi}+ 2\epsilon\sin(\frac{\bartheta}{2})\be_\theta\cdot\bm{\psi}) (2\epsilon\sin^2(\frac{\bartheta}{2})- \bars^2Q_{\rm Rn'})(\bars\wt Q_{\rm R}+2\epsilon^2\sin(\frac{\bartheta}{2})\sin(\bartheta))}{|\bR|^7} \\
&\qquad -\frac{5\bars(\bars^2\bm{Q}_{\rm t}\cdot\bm{\psi}-\epsilon\bars\bm{Q}_r\cdot\bm{\psi})(2\epsilon\sin^2(\frac{\bartheta}{2})- \bars^2Q_{\rm Rn'})(\bars\wt Q_{\rm R}+2\epsilon^2\sin(\frac{\bartheta}{2})\sin(\bartheta))}{|\bR|^7} \\
&\qquad -\frac{5\bars^3Q_{\rm Rn'}(\bars\be_{\rm t}\cdot\bm{\psi} + 2\epsilon\sin(\frac{\bartheta}{2})\be_\theta\cdot\bm{\psi})(\bars\wt Q_{\rm R}+2\epsilon^2\sin(\frac{\bartheta}{2})\sin(\bartheta))}{|\bR|^7} \\
&\qquad + \frac{10\bars\epsilon\sin^2(\frac{\bartheta}{2})(\epsilon\bars Q_{\rm Rs1} + \bars^2Q_{\rm Rs2})(\bars\be_{\rm t}\cdot\bm{\psi} + 2\epsilon\sin(\frac{\bartheta}{2})\be_\theta\cdot\bm{\psi})}{|\bR|^7} \\
&\qquad +\bigg(\frac{1}{|\bR|^7}-\frac{1}{|\barR|^7} \bigg)\textstyle 10\bars\epsilon\sin^2(\frac{\bartheta}{2})(\bars\be_{\rm t}\cdot\bm{\psi} + 2\epsilon\sin(\frac{\bartheta}{2})\be_\theta\cdot\bm{\psi}) (\bars+2\epsilon^2\sin(\frac{\bartheta}{2})\sin(\bartheta))\,.
\end{align*}
As above, $\wt{\bm Q}_{\rm t}=\be_{\rm t}-\bars\bm{Q}_{\rm t}+\epsilon\bm{Q}_r$, and here we define $\wt Q_{\rm R}=1+ \epsilon Q_{\rm Rs1} + \bars Q_{\rm Rs2}$. Expanding $\frac{1}{|\bR|^5}-\frac{1}{|\barR|^5}$ and $\frac{1}{|\bR|^7}-\frac{1}{|\barR|^7}$ using \eqref{eq:RaRb_diff} and \eqref{eq:Rdiffk}, we may use Lemma \ref{lem:alpha_est} to estimate
\begin{equation}\label{eq:Wz_s}
\begin{aligned}
\abs{\bm{W}(\bm{\psi};\p_s\bm{K}_{\mc D})\cdot\be_z}_{\dot C^{0,\gamma}} &\le c(\kappa_*,c_\Gamma)\max_j\big(\norm{Q_j}_{C^{0,\gamma^+}_1}+\norm{Q_j}_{C^{0,\gamma}_2} \big)\norm{\bm{\psi}}_{C^{0,\gamma}} \\
&\le c(\kappa_{*,\gamma^+},c_\Gamma)\,\epsilon^{-\gamma^+}\norm{\bm{\psi}}_{C^{0,\gamma}}\,.
\end{aligned}
\end{equation}
Furthermore, using that each remainder term $Q_i$ satisfies a Lipschitz estimate of the form \eqref{eq:Qj_Cbeta}, by the above expansions and Lemma \ref{lem:alpha_est} we have 
\begin{equation}\label{eq:Wz_s_lip}
\begin{aligned}
\abs{\bm{W}(\bm{\psi};\p_s\bm{K}_{\mc D}^{(a)}-\p_s\bm{K}_{\mc D}^{(b)})\cdot\be_z}_{\dot C^{0,\gamma}} 
&\le c(\kappa_{*,\gamma^+}^{(a)},\kappa_{*,\gamma^+}^{(b)},c_\Gamma)\,\epsilon^{-\gamma^+}\norm{\X^{(a)}-\X^{(b)}}_{C^{2,\gamma^+}} \norm{\bm{\psi}}_{C^{0,\gamma}}\,.
\end{aligned}
\end{equation}

We next consider $W_r(\bm{\psi};\p_s\bm{K}_{\mc D})$ as given by \eqref{eq:Wj_def}. We write
\begin{align*}
W_r(\bm{\psi};&\p_s\bm{K}_{\mc D}) = \frac{3}{4\pi}\big(W_{r,1}+ W_{r,2}+ W_{r,3} \big)\,, \\
W_{r,1}&= \frac{(\bR\cdot\be_r)(\p_s\bR\cdot\bm{\psi})\bR\cdot\bm{n}'}{|\bR|^5}
-\frac{(\barR\cdot\overline\be_r)(\be_{\rm t}\cdot\bm{\psi})\barR\cdot\overline{\bm{n}}'}{|\barR|^5}\\
W_{r,2}&= \frac{(\bR\cdot\be_r)(\bR\cdot\bm{\psi}) (\p_s\bR\cdot\bm{n}')}{|\bR|^5}
-\frac{(\barR\cdot\overline\be_r)(\barR\cdot(\Phi^{-1}\bm{\psi})) (\p_s\barR\cdot\overline{\bm{n}}')}{|\barR|^5} \\
W_{r,3}&= -\frac{5(\bR\cdot\be_r)(\bR\cdot\bm{\psi}) (\bR\cdot\bm{n}')\bR\cdot\p_s\bR}{|\bR|^7}
+\frac{5(\barR\cdot\overline\be_r)(\barR\cdot(\Phi^{-1}\bm{\psi})) (\barR\cdot\overline{\bm{n}}')\barR\cdot\p_s\barR}{|\barR|^7}\,.
\end{align*}
Using \eqref{eq:RdotN}, \eqref{eq:Rdots}, and \eqref{eq:psR}, we may expand each $W_{r,j}$ as
\begin{align*}
W_{r,1}&= -\frac{\bars^2Q_{\rm Rn}((1-\epsilon\wh\kappa)\be_{\rm t}\cdot{\bm \psi} + \epsilon\kappa_3\be_\theta\cdot\bm{\psi})(2\epsilon\sin^2(\frac{\bartheta}{2})-\bars^2Q_{\rm Rn'})}{|\bR|^5} \\
&\qquad + \frac{2\epsilon\sin^2(\frac{\bartheta}{2})(\epsilon\wh\kappa\be_{\rm t}\cdot{\bm \psi} - \epsilon\kappa_3\be_\theta\cdot\bm{\psi})(2\epsilon\sin^2(\frac{\bartheta}{2})-\bars^2Q_{\rm Rn'})}{|\bR|^5}\\
&\qquad + \frac{2\epsilon\sin^2(\frac{\bartheta}{2})\bars^2Q_{\rm Rn'}(\be_{\rm t}\cdot\bm{\psi})}{|\bR|^5} - \bigg(\frac{1}{|\bR|^5}-\frac{1}{|\barR|^5}\bigg)\textstyle4\epsilon^2\sin^4(\frac{\bartheta}{2})(\be_{\rm t}\cdot\bm{\psi})
\\
W_{r,2}&= \frac{\bars^2Q_{\rm Rn}(\bars\wt{\bm{Q}}_{\rm t}\cdot\bm{\psi}+ 2\epsilon\sin(\frac{\bartheta}{2})\be_\theta\cdot\bm{\psi}) (\bars\wh{\kappa} - \kappa_3\epsilon\sin(\bartheta)+\epsilon\bars Q_{\rm Rs3})}{|\bR|^5} \\
&\qquad  - \frac{2\epsilon\sin^2(\frac{\bartheta}{2})(\bars^2\bm{Q}_{\rm t}\cdot\bm{\psi}-\epsilon\bars\bm{Q}_r\cdot\bm{\psi})(\bars\wh{\kappa} - \kappa_3\epsilon\sin(\bartheta)+\epsilon\bars Q_{\rm Rs3})}{|\bR|^5}\\
&\qquad +\bigg(\frac{1}{|\bR|^5}-\frac{1}{|\barR|^5}\bigg)\textstyle 2\epsilon\sin^2(\frac{\bartheta}{2})(\bars\be_{\rm t}\cdot\bm{\psi} + 2\epsilon\sin(\frac{\bartheta}{2})\be_\theta\cdot\bm{\psi}) (\bars\wh{\kappa} - \kappa_3\epsilon\sin(\bartheta)+\epsilon\bars Q_{\rm Rs3})
 \\
W_{r,3}&= \frac{5\bars^2Q_{\rm Rn}(\bars\wt{\bm{Q}}_{\rm t}\cdot\bm{\psi}+ 2\epsilon\sin(\frac{\bartheta}{2})\be_\theta\cdot\bm{\psi}) (2\epsilon\sin^2(\frac{\bartheta}{2})- \bars^2Q_{\rm Rn'})(\bars\wt Q_{\rm R}+2\epsilon^2\sin(\frac{\bartheta}{2})\sin(\bartheta))}{|\bR|^7} \\
&\qquad -\frac{10\epsilon\sin^2(\frac{\bartheta}{2})(\bars^2\bm{Q}_{\rm t}\cdot\bm{\psi}-\epsilon\bars\bm{Q}_r\cdot\bm{\psi})(2\epsilon\sin^2(\frac{\bartheta}{2})- \bars^2Q_{\rm Rn'})(\bars\wt Q_{\rm R} +2\epsilon^2\sin(\frac{\bartheta}{2})\sin(\bartheta))}{|\bR|^7} \\
&\qquad -\frac{10\epsilon\sin^2(\frac{\bartheta}{2})\bars^2Q_{\rm Rn'}(\bars\be_{\rm t}\cdot\bm{\psi} + 2\epsilon\sin(\frac{\bartheta}{2})\be_\theta\cdot\bm{\psi})(\bars\wt Q_{\rm R}+2\epsilon^2\sin(\frac{\bartheta}{2})\sin(\bartheta))}{|\bR|^7} \\
&\qquad +\frac{20\epsilon^2\sin^4(\frac{\bartheta}{2})(\epsilon\bars Q_{\rm Rs1} + \bars^2Q_{\rm Rs2})(\bars\be_{\rm t}\cdot\bm{\psi} + 2\epsilon\sin(\frac{\bartheta}{2})\be_\theta\cdot\bm{\psi})}{|\bR|^7} \\
&\qquad +\bigg(\frac{1}{|\bR|^7}-\frac{1}{|\barR|^7}\bigg)\textstyle 20\epsilon^2\sin^4(\frac{\bartheta}{2})(\bars\be_{\rm t}\cdot\bm{\psi} + 2\epsilon\sin(\frac{\bartheta}{2})\be_\theta\cdot\bm{\psi}) (\bars+2\epsilon^2\sin(\frac{\bartheta}{2})\sin(\bartheta))\,.
\end{align*}
Again we have $\wt{\bm Q}_{\rm t}=\be_{\rm t}-\bars\bm{Q}_{\rm t}+\epsilon\bm{Q}_r$ and $\wt Q_{\rm R}=1+ \epsilon Q_{\rm Rs1} + \bars Q_{\rm Rs2}$. Expanding $\frac{1}{|\bR|^5}-\frac{1}{|\barR|^5}$ and $\frac{1}{|\bR|^7}-\frac{1}{|\barR|^7}$ via \eqref{eq:RaRb_diff} and \eqref{eq:Rdiffk}, by Lemma \ref{lem:alpha_est} we have
\begin{equation}\label{eq:Wr_s}
\begin{aligned}
\abs{\bm{W}(\bm{\psi};\p_s\bm{K}_{\mc D})\cdot\overline\be_r}_{\dot C^{0,\gamma}} &\le c(\kappa_*,c_\Gamma)\max_j\big(\norm{Q_j}_{C^{0,\gamma^+}_1}+\norm{Q_j}_{C^{0,\gamma}_2} \big)\norm{\bm{\psi}}_{C^{0,\gamma}} \\
&\le c(\kappa_{*,\gamma^+},c_\Gamma)\,\epsilon^{-\gamma^+}\norm{\bm{\psi}}_{C^{0,\gamma}}\,,\\
\abs{\bm{W}(\bm{\psi};\p_s\bm{K}_{\mc D}^{(a)}-\p_s\bm{K}_{\mc D}^{(b)})\cdot\overline{\be}_r}_{\dot C^{0,\gamma}} 
&\le c(\kappa_{*,\gamma^+}^{(a)},\kappa_{*,\gamma^+}^{(b)},c_\Gamma)\,\epsilon^{-\gamma^+}\norm{\X^{(a)}-\X^{(b)}}_{C^{2,\gamma^+}}\norm{\bm{\psi}}_{C^{0,\gamma}}\,.
\end{aligned}
\end{equation}

We finally turn to $W_\theta(\bm{\psi};\p_s\bm{K}_{\mc D})$ as in \eqref{eq:Wj_def}, which we may write as
\begin{align*}
W_\theta(\bm{\psi};&\p_s\bm{K}_{\mc D}) = \frac{3}{4\pi}\big( W_{\theta,1}+W_{\theta,2}+W_{\theta,3} \big)\,, \\
W_{\theta,1}&= \frac{\big(\epsilon\kappa_3(\bR\cdot\bm{\psi})+(\bR\cdot\be_\theta)(\p_s\bR\cdot\bm{\psi})\big)\bR\cdot\bm{n}'}{|\bR|^5}
-\frac{(\barR\cdot\overline\be_\theta)(\be_{\rm t}\cdot\bm{\psi})\barR\cdot\overline{\bm{n}}'}{|\barR|^5}\\
W_{\theta,2}&= \frac{(\bR\cdot\be_\theta)(\bR\cdot\bm{\psi}) (\p_s\bR\cdot\bm{n}')}{|\bR|^5}
-\frac{(\barR\cdot\overline\be_\theta)(\barR\cdot(\Phi^{-1}\bm{\psi})) (\p_s\barR\cdot\overline{\bm{n}}')}{|\barR|^5} \\
W_{\theta,3}&= -\frac{5(\bR\cdot\be_\theta)(\bR\cdot\bm{\psi}) (\bR\cdot\bm{n}')\bR\cdot\p_s\bR}{|\bR|^7}
+\frac{5(\barR\cdot\overline\be_\theta)(\barR\cdot(\Phi^{-1}\bm{\psi})) (\barR\cdot\overline{\bm{n}}')\barR\cdot\p_s\barR}{|\barR|^7}\,.
\end{align*}
By \eqref{eq:RdotN}, \eqref{eq:Rdots}, and \eqref{eq:psR}, we may expand $W_{\theta,j}$ as 
\begin{align*}
W_{\theta,1}&= - \frac{(\epsilon\bars\kappa_3\cos(\theta-\bartheta) + \bars^2Q_{\rm R\theta})((1-\epsilon\wh\kappa)\be_{\rm t}\cdot{\bm \psi} + \epsilon\kappa_3\be_\theta\cdot\bm{\psi})(2\epsilon\sin^2(\frac{\bartheta}{2})-\bars^2Q_{\rm Rn'})}{|\bR|^5} \\
&\qquad -\frac{\epsilon\kappa_3(\bars\wt{\bm{Q}}_{\rm t}\cdot\bm{\psi}+ 2\epsilon\sin(\frac{\bartheta}{2})\be_\theta\cdot\bm{\psi})(2\epsilon\sin^2(\frac{\bartheta}{2})- \bars^2Q_{\rm Rn'})}{|\bR|^5} \\
&\qquad + \frac{\epsilon\sin(\bartheta)(\epsilon\wh\kappa\be_{\rm t}\cdot{\bm \psi} + \epsilon\kappa_3\be_\theta\cdot\bm{\psi})(2\epsilon\sin^2(\frac{\bartheta}{2})- \bars^2Q_{\rm Rn'})}{|\bR|^5} \\
&\qquad + \frac{\epsilon\sin(\bartheta)\bars^2Q_{\rm Rn'}(\be_{\rm t}\cdot\bm{\psi})}{|\bR|^5} - \bigg(\frac{1}{|\bR|^5}-\frac{1}{|\barR|^5}\bigg)\textstyle 2\epsilon^2\sin^2(\frac{\bartheta}{2})\sin(\bartheta)(\be_{\rm t}\cdot\bm{\psi})
\\
W_{\theta,2}&= \frac{(\epsilon\bars \kappa_3\cos(\theta-\bartheta) + \bars^2Q_{\rm R\theta})(\bars\wt{\bm{Q}}_{\rm t}\cdot\bm{\psi}+ 2\epsilon\sin(\frac{\bartheta}{2})\be_\theta\cdot\bm{\psi}) (\bars\wh{\kappa} - \kappa_3\epsilon\sin(\bartheta)+\epsilon\bars Q_{\rm Rs3})}{|\bR|^5} \\
&\qquad - \frac{\epsilon\sin(\bartheta)(\bars^2\bm{Q}_{\rm t}\cdot\bm{\psi}-\epsilon\bars\bm{Q}_r\cdot\bm{\psi})(\bars\wh{\kappa} - \kappa_3\epsilon\sin(\bartheta)+\epsilon\bars Q_{\rm Rs3})}{|\bR|^5} \\
&\qquad + \bigg(\frac{1}{|\bR|^5}-\frac{1}{|\barR|^5}\bigg)\textstyle \epsilon\sin(\bartheta)(\bars\be_{\rm t}\cdot\bm{\psi} + 2\epsilon\sin(\frac{\bartheta}{2})\be_\theta\cdot\bm{\psi})(\bars\wh{\kappa} - \kappa_3\epsilon\sin(\bartheta)+\epsilon\bars Q_{\rm Rs3}) \\
W_{\theta,3}&= \frac{5(\epsilon\bars \kappa_3\cos(\theta-\bartheta) + \bars^2Q_{\rm R\theta})(\bars\wt{\bm{Q}}_{\rm t}\cdot\bm{\psi}+ 2\epsilon\sin(\frac{\bartheta}{2})\be_\theta\cdot\bm{\psi}) (2\epsilon\sin^2(\frac{\bartheta}{2})- \bars^2Q_{\rm Rn'})(\bars\wt Q_{\rm R}+2\epsilon^2\sin(\frac{\bartheta}{2})\sin(\bartheta))}{|\bR|^7} \\
&\qquad -\frac{5\epsilon\sin(\bartheta)(\bars^2\bm{Q}_{\rm t}\cdot\bm{\psi}-\epsilon\bars\bm{Q}_r\cdot\bm{\psi})(2\epsilon\sin^2(\frac{\bartheta}{2})- \bars^2Q_{\rm Rn'})(\bars\wt Q_{\rm R}+2\epsilon^2\sin(\frac{\bartheta}{2})\sin(\bartheta))}{|\bR|^7} \\
&\qquad -\frac{5\epsilon\sin(\bartheta)\bars^2Q_{\rm Rn'}(\bars\be_{\rm t}\cdot\bm{\psi} + 2\epsilon\sin(\frac{\bartheta}{2})\be_\theta\cdot\bm{\psi})(\bars \wt Q_{\rm R}+2\epsilon^2\sin(\frac{\bartheta}{2})\sin(\bartheta))}{|\bR|^7} \\
&\qquad +\frac{10\epsilon^2\sin^2(\frac{\bartheta}{2})\sin(\bartheta)(\epsilon\bars Q_{\rm Rs1} + \bars^2Q_{\rm Rs2})(\bars\be_{\rm t}\cdot\bm{\psi} + 2\epsilon\sin(\frac{\bartheta}{2})\be_\theta\cdot\bm{\psi})}{|\bR|^7} \\
&\qquad +\bigg(\frac{1}{|\bR|^7}-\frac{1}{|\barR|^7}\bigg)\textstyle 10\epsilon^2\sin^2(\frac{\bartheta}{2})\sin(\bartheta)(\bars\be_{\rm t}\cdot\bm{\psi} + 2\epsilon\sin(\frac{\bartheta}{2})\be_\theta\cdot\bm{\psi}) (\bars+2\epsilon^2\sin(\frac{\bartheta}{2})\sin(\bartheta))\,.
\end{align*}
We again expand $\frac{1}{|\bR|^7}-\frac{1}{|\barR|^7}$ and $\frac{1}{|\bR|^7}-\frac{1}{|\barR|^7}$ using \eqref{eq:RaRb_diff} and \eqref{eq:Rdiffk}, and apply Lemma \ref{lem:alpha_est} to obtain
\begin{equation}\label{eq:Wthe_s}
\begin{aligned}
\abs{\bm{W}(\bm{\psi};\p_s\bm{K}_{\mc D})\cdot\overline\be_\theta}_{\dot C^{0,\gamma}} &\le  c(\kappa_*,c_\Gamma)\max_j\big(\norm{Q_j}_{C^{0,\gamma^+}_1}+\norm{Q_j}_{C^{0,\gamma}_2} \big)\norm{\bm{\psi}}_{C^{0,\gamma}} \\
&\le c(\kappa_{*,\gamma^+},c_\Gamma)\,\epsilon^{-\gamma^+}\norm{\bm{\psi}}_{C^{0,\gamma}}\,,\\
\abs{\bm{W}(\bm{\psi};\p_s\bm{K}_{\mc D}^{(a)}-\p_s\bm{K}_{\mc D}^{(b)})\cdot\overline\be_\theta}_{\dot C^{0,\gamma}} 
&\le c(\kappa_{*,\gamma^+}^{(a)},\kappa_{*,\gamma^+}^{(b)},c_\Gamma)\,\epsilon^{-\gamma^+}\norm{\X^{(a)}-\X^{(b)}}_{C^{2,\gamma^+}}\norm{\bm{\psi}}_{C^{0,\gamma}} \,.
\end{aligned}
\end{equation}

Next, noting that, by the form \eqref{eq:RdotN} of $\bR\cdot\bm{n}'$, we have 
\begin{equation}\label{eq:KD_form}
\begin{aligned}
\bm{K}_\mc{D} = \frac{3}{4\pi}\frac{\bR\otimes\bR (\bR\cdot\bm{n}')}{\abs{\bR}^5} = \frac{3}{4\pi}\frac{\bR\otimes\bR (-2\epsilon\sin^2(\frac{\bartheta}{2})+\bars^2Q_{\rm Rn'})}{\abs{\bR}^5}\,,
\end{aligned}
\end{equation}
we may use the commutator bounds \eqref{eq:Phi_commutator_est} and Lemmas \ref{lem:basic_est} and \ref{lem:alpha_est} to estimate the second term of $\mc{R}_{\mc{D},1}$ involving $[\p_s, \Phi^{-1}]$: 
\begin{equation}\label{eq:psRD1_comm_est}
\begin{aligned}
\abs{[\p_s, \Phi^{-1}]\int_{-1/2}^{1/2}\int_{-\pi}^\pi\bm{K}_{\mc D}\,\bm{\psi}\,\epsilon\,d\bartheta d\bars}_{\dot C^{0,\gamma}} &\le
c(\kappa_{*,\gamma})\norm{\int_{-1/2}^{1/2}\int_{-\pi}^\pi\bm{K}_{\mc D}\,\bm{\psi}\,\epsilon\,d\bartheta d\bars}_{C^{0,\gamma}}\\
&\le c(\kappa_{*,\gamma},c_\Gamma)\,\epsilon^{-\gamma}\norm{\bm{\psi}}_{L^\infty}  \\
\bigg|[\p_s, (\Phi^{(a)})^{-1}]\int_{-1/2}^{1/2}\int_{-\pi}^\pi\bm{K}_{\mc D}^{(a)}\,\bm{\psi}\,\epsilon\,d\bartheta d\bars &-  [\p_s, (\Phi^{b})^{-1}]\int_{-1/2}^{1/2}\int_{-\pi}^\pi\bm{K}_{\mc D}^{(b)}\,\bm{\psi}\,\epsilon\,d\bartheta d\bars \bigg|_{\dot C^{0,\gamma}}\\
&\le c(\kappa_{*,\gamma}^{(a)},\kappa_{*,\gamma}^{(b)},c_\Gamma)\,\epsilon^{-\gamma}\norm{\X^{(a)}-\X^{(b)}}_{C^{2,\gamma}}\norm{\bm{\psi}}_{L^\infty}\,.
\end{aligned}
\end{equation}
We may then combine \eqref{eq:Wz_s}, \eqref{eq:Wz_s_lip}, \eqref{eq:Wr_s}, \eqref{eq:Wthe_s}, and \eqref{eq:psRD1_comm_est} to obtain 
\begin{equation}\label{eq:psRD1_Calpha}
\begin{aligned}
\abs{\p_s\mc{R}_{\mc{D},1}[\bm{\psi}]}_{\dot C^{0,\gamma}} &\le  
 \abs{\bm{W}(\bm{\psi};\p_s\bm{K}_{\mc D})}_{\dot C^{0,\gamma}}
 + \abs{[\p_s, \Phi^{-1}]\int_{-1/2}^{1/2}\int_{-\pi}^\pi\bm{K}_{\mc D}\,\bm{\psi}\,\epsilon\,d\bartheta d\bars}_{\dot C^{0,\gamma}} \\
&\le c(\kappa_{*,\gamma^+},c_\Gamma)\,\epsilon^{-\gamma^+}\norm{\bm{\psi}}_{C^{0,\gamma}}\,,\\
\abs{\p_s\mc{R}_{\mc{D},1}^{(a)}[\bm{\psi}]-\p_s\mc{R}_{\mc{D},1}^{(b)}[\bm{\psi}]}_{\dot C^{0,\gamma}} &\le 
 c(\kappa_{*,\gamma^+}^{(a)},\kappa_{*,\gamma^+}^{(b)},c_\Gamma)\,\epsilon^{-\gamma^+}\norm{\X^{(a)}-\X^{(b)}}_{C^{2,\gamma^+}}\norm{\bm{\psi}}_{C^{0,\gamma}} \,.
\end{aligned}
\end{equation}

We finally turn to $\p_s\mc{R}_{\mc{D},2}[\bm{\psi}]$, which we write as 
\begin{align*}
\p_s\mc{R}_{\mc{D},2}[\bm{\psi}] &= -\bm{W}[\bm{\psi};\epsilon\wh\kappa\p_s\bm{K}_{\mc D}] - {\rm p.v.}\int_{-1/2}^{1/2}\int_{-\pi}^{\pi}\epsilon\wh\kappa \p_s\overline{\bm{K}}_{\mc D}(\Phi^{-1}\bm{\psi})(s-\bars,\theta-\bartheta)\,\epsilon\,d\bartheta d\bars\\
&\qquad - [\p_s,\Phi^{-1}]\int_{-1/2}^{1/2}\int_{-\pi}^\pi\bm{K}_{\mc D}\,\bm{\psi}\,\epsilon^2\wh\kappa\,d\bartheta d\bars \,.
\end{align*}
We may use that \eqref{eq:Wz_s}, \eqref{eq:Wr_s}, and \eqref{eq:Wthe_s} imply that 
\begin{align*}
\abs{\bm{W}[\bm{\psi};\epsilon\wh\kappa\p_s\bm{K}_{\mc D}]}_{\dot C^{0,\gamma}} &\le c(\kappa_{*,\gamma^+},c_\Gamma)\,\epsilon^{1-\gamma^+}\,\norm{\bm{\psi}}_{C^{0,\gamma}}\,,\\
\abs{\bm{W}\big[\bm{\psi};\epsilon\wh\kappa^{(a)}\p_s\bm{K}_{\mc D}^{(a)}-\epsilon\wh\kappa^{(b)}\p_s\bm{K}_{\mc D}^{(b)}\big]}_{\dot C^{0,\gamma}} 
&\le c(\kappa_{*,\gamma^+}^{(a)},\kappa_{*,\gamma^+}^{(b)},c_\Gamma)\,\epsilon^{1-\gamma^+}\norm{\X^{(a)}-\X^{(b)}}_{C^{2,\gamma^+}}\norm{\bm{\psi}}_{C^{0,\gamma}} \,.
\end{align*}
Furthermore, using \eqref{eq:Phi_commutator_est}, \eqref{eq:KD_form}, and Lemma \ref{lem:basic_est}, the third term involving the commutator $[\p_s,\Phi^{-1}]$ satisfies  
\begin{align*}
\abs{[\p_s,\Phi^{-1}]\int_{-1/2}^{1/2}\int_{-\pi}^\pi\bm{K}_{\mc D}\,\bm{\psi}\,\epsilon^2\wh\kappa\,d\bartheta d\bars}_{\dot C^{0,\gamma}} 
&\le c(\kappa_{*,\gamma})\norm{\int_{-1/2}^{1/2}\int_{-\pi}^\pi\bm{K}_{\mc D}\,\bm{\psi}\,\epsilon^2\wh\kappa\,d\bartheta d\bars}_{\dot C^{0,\gamma}}\\
&\le c(\kappa_{*,\gamma},c_\Gamma)\,\epsilon^{1-\gamma}\norm{\bm{\psi}}_{L^\infty}\,,\\
\bigg|[\p_s,(\Phi^{(a)})^{-1}]\int_{-1/2}^{1/2}\int_{-\pi}^\pi\bm{K}_{\mc D}^{(a)}\,\bm{\psi}\,\epsilon^2\wh\kappa^{(a)}\,d\bartheta d\bars&- [\p_s,(\Phi^{(b)})^{-1}]\int_{-1/2}^{1/2}\int_{-\pi}^\pi\bm{K}_{\mc D}^{(b)}\,\bm{\psi}\,\epsilon^2\wh\kappa^{(b)}\,d\bartheta d\bars\bigg|_{\dot C^{0,\gamma}} \\
&\le c(\kappa_{*,\gamma}^{(a)},\kappa_{*,\gamma}^{(b)},c_\Gamma)\,\epsilon^{1-\gamma}\norm{\X^{(a)}-\X^{(b)}}_{C^{2,\gamma}}\norm{\bm{\psi}}_{L^\infty}\,.
\end{align*}
It thus remains to bound the middle term of the above expression for $\p_s\mc{R}_{\mc{D},2}[\bm{\psi}]$. Using the form \eqref{eq:barR} of $\barR$, we may write 
\begin{align*}
\textstyle \be_z\cdot\big(\epsilon\wh\kappa\p_s\overline{\bm{K}}_{\mc D}(\Phi^{-1}\bm{\psi})\big) 
 &= \frac{3\epsilon\wh\kappa}{4\pi} \bigg(\frac{10\epsilon\sin^2(\frac{\bartheta}{2})\bars^2\big(\bars\be_{\rm t}\cdot\bm{\psi}+2\epsilon\sin(\frac{\bartheta}{2})\be_\theta\cdot\bm{\psi}\big)}{|\barR|^7} \\
 &\qquad \qquad-\frac{4\epsilon\sin^2(\frac{\bartheta}{2})\big(\bars\be_{\rm t}\cdot\bm{\psi}+\epsilon\sin(\frac{\bartheta}{2})\be_\theta\cdot\bm{\psi}\big)}{|\barR|^5} \bigg)\\
\textstyle \overline\be_r\cdot\big(\epsilon\wh\kappa\p_s \overline{\bm{K}}_{\mc D}(\Phi^{-1}\bm{\psi})\big) 
&= \frac{3\epsilon\wh\kappa}{4\pi} \bigg(\frac{20\epsilon^2\sin^4(\frac{\bartheta}{2})\bars\big(\bars\be_{\rm t}\cdot\bm{\psi}+2\epsilon\sin(\frac{\bartheta}{2})\be_\theta\cdot\bm{\psi}\big)}{|\barR|^7} -\frac{4\epsilon^2\sin^4(\frac{\bartheta}{2})(\be_{\rm t}\cdot\bm{\psi})}{|\barR|^5}\bigg) \\
\textstyle \overline\be_\theta\cdot\big(\epsilon\wh\kappa\p_s \overline{\bm{K}}_{\mc D}(\Phi^{-1}\bm{\psi})\big) 
&= \frac{3\epsilon\wh\kappa}{4\pi} \bigg(\frac{10\epsilon^2\sin^2(\frac{\bartheta}{2})\sin(\bartheta)\bars\big(\bars\be_{\rm t}\cdot\bm{\psi}+2\epsilon\sin(\frac{\bartheta}{2})\be_\theta\cdot\bm{\psi}\big)}{|\barR|^7}\\
&\qquad\qquad -\frac{2\epsilon^2\sin^2(\frac{\bartheta}{2})\sin(\bartheta)(\be_{\rm t}\cdot\bm{\psi})}{|\barR|^5}\bigg)\,.
\end{align*}
By Lemma \ref{lem:alpha_est}, we thus have
\begin{equation}\label{eq:psRD2_Calpha}
\begin{aligned}
&\textstyle \abs{\p_s\mc{R}_{\mc{D},2}[\bm{\psi}]}_{\dot C^{0,\gamma}} \le 
\abs{\bm{W}[\bm{\psi};\epsilon \wh\kappa \p_s\bm{K}_{\mc D}]}_{\dot C^{0,\gamma}} + \displaystyle\bigg|{\rm p.v.}\int_{-1/2}^{1/2}\int_{-\pi}^\pi \textstyle \epsilon\wh\kappa\p_s\overline{\bm{K}}_{\mc D}\,(\Phi^{-1}\bm{\psi})\, \epsilon \,d\bartheta d\bars\bigg|_{\dot C^{0,\gamma}} \\
&\hspace{2.5cm} + \abs{[\p_s,\Phi^{-1}]\int_{-1/2}^{1/2}\int_{-\pi}^\pi\bm{K}_{\mc D}\,\bm{\psi}\,\epsilon^2\wh\kappa\,d\bartheta d\bars}_{\dot C^{0,\gamma}} 
\le c(\kappa_{*,\gamma^+},c_\Gamma)\,\epsilon^{-\gamma^+}\norm{\bm{\psi}}_{C^{0,\gamma}}\,,\\
&\textstyle \abs{\p_s\mc{R}_{\mc{D},2}^{(a)}[\bm{\psi}]-\p_s\mc{R}_{\mc{D},2}^{(b)}[\bm{\psi}]}_{\dot C^{0,\gamma}} 
\le c(\kappa_{*,\gamma^+}^{(a)},\kappa_{*,\gamma^+}^{(b)},c_\Gamma)\,\epsilon^{-\gamma^+}\norm{\X^{(a)}-\X^{(b)}}_{C^{2,\gamma^+}}\norm{\bm{\psi}}_{C^{0,\gamma}} \,. 
\end{aligned}
\end{equation}

Combining the estimates \eqref{eq:RD0_bd}, \eqref{eq:RD0_bd_lip}, \eqref{eq:RD1_inftybd}, \eqref{eq:RD1_inftybd_lip}, \eqref{eq:RD2_linfty}, \eqref{eq:pthetaRD1_Calpha}, \eqref{eq:pthetaRD2_Calpha}, \eqref{eq:psRD1_Calpha}, and \eqref{eq:psRD2_Calpha}, we obtain Lemma \ref{lem:double_layer}.
\end{proof}


\subsection{Single layer applied to constant-in-arclength}\label{subsec:single_layer_mean}
This section is devoted to the proof of Lemma \ref{lem:single_const_in_s} regarding the angle-averaged single layer operator applied to functions that are constant in $s$. We restate the lemma here for convenience. 
\begin{lemma}[Single layer applied to constant-in-$s$]\label{lem:single_const_in_s0}
Let $0<\alpha<\beta<1$ and consider a filament $\Sigma_\epsilon$ with centerline $\X(s)\in C^{2,\beta}(\T)$. Given $\overline{\bm{h}}(\theta)\in C^{0,\alpha}(2\pi\T)$ of the form $\overline{\bm{h}}=h_1(\theta)\be_z+h_2(\theta)\be_x+h_3(\theta)\be_y$, and recalling the map $\Phi$ \eqref{eq:mapPhi_def}, we may decompose
\begin{align*}
\mc{S}[\Phi\overline{\bm{h}}(\theta)] = \mc{H}_\epsilon[\Phi\overline{\bm{h}}(\theta)] + \mc{H}_+[\Phi\overline{\bm{h}}(\theta)]
\end{align*}
where
\begin{equation}
\begin{aligned}
\norm{\mc{H}_\epsilon[\Phi\overline{\bm{h}}]}_{C^{0,\alpha}} &\le c(\kappa_{*,\alpha},c_\Gamma)\,\epsilon^{2-\alpha}\norm{\overline{\bm{h}}}_{L^\infty}\,, \quad
\abs{\mc{H}_\epsilon[\Phi\overline{\bm{h}}]}_{\dot C_s^{1,\alpha}} \le c(\kappa_{*,\alpha^+},c_\Gamma)\,\epsilon^{1-\alpha^+}\norm{\overline{\bm{h}}}_{C^{0,\alpha}} \\
\norm{\mc{H}_+[\Phi\overline{\bm{h}}]}_{C^{0,\gamma}} &\le c(\kappa_{*,\alpha},c_\Gamma)\,\epsilon^{1-\alpha}\norm{\overline{\bm{h}}}_{L^\infty}\,, \quad
\abs{\mc{H}_+[\Phi\overline{\bm{h}}]}_{\dot C_s^{1,\beta}} \le c(\kappa_{*,\beta},c_\Gamma)\,\epsilon^{1-2\beta}\norm{\overline{\bm{h}}}_{C^{0,\alpha}}
\end{aligned}
\end{equation}
for any $\alpha^+\in (\alpha,\beta]$.
In addition, given two nearby curves $\X^{(a)}$, $\X^{(b)}$ satisfying Lemma \ref{lem:XaXb_C2beta}, the corresponding differences $\mc{H}_\epsilon^{(a)}-\mc{H}_\epsilon^{(b)}$ and $\mc{H}_+^{(a)}-\mc{H}_+^{(b)}$ satisfy
\begin{equation}
\begin{aligned}
\norm{\mc{H}_\epsilon^{(a)}[\Phi^{(a)}\overline{\bm{h}}]-\mc{H}_\epsilon^{(b)}[\Phi^{(b)}\overline{\bm{h}}]}_{C^{0,\alpha}} &\le c(\kappa_{*,\alpha}^{(a)},\kappa_{*,\alpha}^{(b)},c_\Gamma)\,\epsilon^{2-\alpha}\norm{\X^{(a)}-\X^{(b)}}_{C^{2,\alpha}}\norm{\overline{\bm{h}}}_{L^\infty} \\
\abs{\mc{H}_\epsilon^{(a)}[\Phi^{(a)}\overline{\bm{h}}]-\mc{H}_\epsilon^{(b)}[\Phi^{(b)}\overline{\bm{h}}]}_{\dot C_s^{1,\alpha}} &\le c(\kappa_{*,\alpha^+}^{(a)},\kappa_{*,\alpha^+}^{(b)},c_\Gamma)\,\epsilon^{1-\alpha^+}\norm{\X^{(a)}-\X^{(b)}}_{C^{2,\alpha^+}}\norm{\overline{\bm{h}}}_{C^{0,\alpha}} \\
\norm{\mc{H}_+^{(a)}[\Phi^{(a)}\overline{\bm{h}}]-\mc{H}_+^{(b)}[\Phi^{(b)}\overline{\bm{h}}]}_{C^{0,\alpha}} &\le c(\kappa_{*,\alpha}^{(a)},\kappa_{*,\alpha}^{(b)},c_\Gamma)\,\epsilon^{1-\alpha}\norm{\X^{(a)}-\X^{(b)}}_{C^{2,\alpha}}\norm{\overline{\bm{h}}}_{L^\infty}\\
\abs{\mc{H}_+^{(a)}[\Phi^{(a)}\overline{\bm{h}}]-\mc{H}_+^{(b)}[\Phi^{(b)}\overline{\bm{h}}]}_{\dot C_s^{1,\beta}} &\le c(\kappa_{*,\beta}^{(a)},\kappa_{*,\beta}^{(b)},c_\Gamma)\,\epsilon^{1-2\beta}\norm{\X^{(a)}-\X^{(b)}}_{C^{2,\beta}}\norm{\overline{\bm{h}}}_{C^{0,\alpha}}\,.
\end{aligned}
\end{equation}
\end{lemma}

\begin{proof}[Proof of Lemma \ref{lem:single_const_in_s}]
For the proof of this lemma, it will be convenient to keep our integrands in terms of $s'$ and $\theta'$ instead of $\bars=s-s'$ and $\bartheta=\theta-\theta'$. The expansions and bounds of section \ref{subsec:setup} will thus be applied using $s-s'$ and $\theta-\theta'$ in place of $\bars,\bartheta$.

We will first define an auxiliary function akin to the function $\bR_{\rm t}$ appearing in the proof of Lemma \ref{lem:single_layer}. Let
\begin{equation}\label{eq:bRr_def}
\begin{aligned}
\bR_{\rm r} &= \X(s)-\X(s')+\epsilon\big(\be_r(s,\theta)-\be_r(s,\theta') \big)\\
&=\textstyle  \X(s)-\X(s') +2\epsilon\sin(\frac{\theta-\theta'}{2})\be_\theta(s,\frac{\theta+\theta'}{2}) \,,
\end{aligned}
\end{equation}
and note that we have
\begin{equation}\label{eq:bRr_expand}
\begin{aligned}
\bR-\bR_{\rm r} &= \epsilon\big(\be_r(s,\theta')-\be_r(s',\theta') \big) = \epsilon(s-s')\bm{Q}_r \\
\abs{\bR}^2-\abs{\bR_{\rm r}}^2 &= 2\epsilon(s-s')\bm{Q}_r\cdot\bR_{\rm r}+\epsilon^2(s-s')^2\abs{\bm{Q}_r}^2\\
\frac{1}{\abs{\bR}^k}-\frac{1}{\abs{\bR_{\rm r}}^k} &= -\frac{2\epsilon(s-s')\bm{Q}_r\cdot\bR_{\rm r}+\epsilon^2(s-s')^2\abs{\bm{Q}_r}^2}{\abs{\bR}|\bR_{\rm r}|(\abs{\bR}+|\bR_{\rm r}|)}\sum_{\ell=0}^{k-1}\frac{1}{|\bR|^\ell|\bR_{\rm r}|^{k-1-\ell}}\\
\frac{1}{|\bR_{\rm r}^{(a)}|^k}-\frac{1}{|\bR_{\rm r}^{(b)}|^k} &= \frac{2\bars^2\epsilon\sin(\frac{\bartheta}{2})A_{ab}+\bars^4B_{ab}}{|\bR_{\rm r}^{(a)}||\bR_{\rm r}^{(b)}|(|\bR_{\rm r}^{(a)}|+|\bR_{\rm r}^{(b)}|)}\sum_{\ell=0}^{k-1}\frac{1}{|\bR_{\rm r}^{(a)}|^\ell|\bR_{\rm r}^{(b)}|^{k-1-\ell}}\\
\end{aligned}
\end{equation}
where $A_{ab}=\be_\theta^{(a)}\cdot\bm{Q}_{\rm t}^{(a)}-\be_\theta^{(b)}\cdot\bm{Q}_{\rm t}^{(b)}$ and $B_{ab}= 2Q_{\rm t}^{(a)}-2Q_{\rm t}^{(b)} +|\bm{Q}_{\rm t}^{(b)}|^2-|\bm{Q}_{\rm t}^{(a)}|^2$. In addition, $\bR_{\rm r}$ may be shown to satisfy each of Lemmas \ref{lem:Rests}, \ref{lem:basic_est}, \ref{lem:odd_nm}, and \ref{lem:alpha_est}.

Using $\bR_{\rm r}$, we will then write $\mc{S}[\Phi\overline{\bm{h}}(\theta)]$ as 
\begin{align*}
\mc{S}[\Phi\overline{\bm{h}}(\theta)] &= \frac{1}{8\pi}\int_\T\int_0^{2\pi}\big(\bm{H}_1+\bm{H}_2 + \bm{H}_3\big)\,(\Phi\overline{\bm{h}})(s',\theta')\,\epsilon\,d\theta'ds' \,,\\
\bm{H}_1 &= \frac{{\bf I}}{\abs{\bR_{\rm r}}}+ \frac{\bR_{\rm r}\otimes\bR_{\rm r}}{\abs{\bR_{\rm r}}^3}\\
\bm{H}_2 &= \frac{\bR\otimes\bR}{\abs{\bR}^3}- \frac{\bR_{\rm r}\otimes\bR_{\rm r}}{\abs{\bR_{\rm r}}^3}  = \frac{(\bR-\bR_{\rm r})\otimes\bR + \bR_{\rm r}\otimes(\bR-\bR_{\rm r})}{\abs{\bR}^3} + \bigg(\frac{1}{\abs{\bR}^3}-\frac{1}{\abs{\bR_{\rm r}}^3}\bigg)\bR_{\rm r}\otimes\bR_{\rm r}\\
\bm{H}_3 &= -\bigg(\frac{{\bf I}}{\abs{\bR}}+ \frac{\bR\otimes\bR}{\abs{\bR}^3}\bigg)\,\epsilon\wh\kappa(s',\theta')\,.
\end{align*}
Using \eqref{eq:bRr_expand} along with Lemmas \ref{lem:basic_est} and \ref{lem:alpha_est}, we have
\begin{align*}
\norm{\int_\T\int_0^{2\pi}\bm{H}_1\,(\Phi\overline{\bm{h}})\,\epsilon\,d\theta'ds'}_{C^{0,\alpha}}&\le c(\kappa_{*,\alpha},c_\Gamma)\,\epsilon^{1-\alpha}\norm{\Phi\overline{\bm{h}}}_{L^\infty} \le c(\kappa_{*,\alpha},c_\Gamma)\,\epsilon^{1-\alpha}\norm{\overline{\bm{h}}}_{L^\infty} \\
\norm{\int_\T\int_0^{2\pi}\bm{H}_2\,(\Phi\overline{\bm{h}})\,\epsilon\,d\theta'ds'}_{C^{0,\alpha}}
&\le c(\kappa_{*,\alpha},c_\Gamma)\,\epsilon^{2-\alpha}\norm{\Phi\overline{\bm{h}}}_{L^\infty} \le c(\kappa_{*,\alpha},c_\Gamma)\,\epsilon^{2-\alpha}\norm{\overline{\bm{h}}}_{L^\infty} \\
\norm{\int_\T\int_0^{2\pi}\bm{H}_3\,(\Phi\overline{\bm{h}})\,\epsilon\,d\theta'ds'}_{C^{0,\alpha}}&\le c(\kappa_{*,\alpha},c_\Gamma)\,\epsilon^{2-\alpha}\norm{\Phi\overline{\bm{h}}}_{L^\infty} \le c(\kappa_{*,\alpha},c_\Gamma)\,\epsilon^{2-\alpha}\norm{\overline{\bm{h}}}_{L^\infty} \,.
\end{align*}
In addition, using \eqref{eq:bRr_expand} along with the fact that each remainder $\bm{Q}$ satisfies \eqref{eq:Qj_Cbeta}, we may show that for two curves $\X^{(a)}$ and $\X^{(b)}$ satisfying Lemma \ref{lem:XaXb_C2beta}, we have
\begin{align*}
\norm{\int_\T\int_0^{2\pi}\big(\bm{H}_1^{(a)}\,(\Phi^{(a)}\overline{\bm{h}})-\bm{H}_1^{(b)}\,(\Phi^{(b)}\overline{\bm{h}})\big)\,\epsilon\,d\theta'ds'}_{C^{0,\alpha}}&\le c(\kappa_{*,\alpha}^{(a)},\kappa_{*,\alpha}^{(b)},c_\Gamma)\,\epsilon^{1-\alpha}\norm{\X^{(a)}-\X^{(b)}}_{C^{2,\alpha}}\norm{\overline{\bm{h}}}_{L^\infty} \\
\norm{\int_\T\int_0^{2\pi}\big(\bm{H}_2^{(a)}\,(\Phi^{(a)}\overline{\bm{h}})-\bm{H}_2^{(b)}\,(\Phi^{(b)}\overline{\bm{h}})\big)\,\epsilon\,d\theta'ds'}_{C^{0,\alpha}}
&\le c(\kappa_{*,\alpha}^{(a)},\kappa_{*,\alpha}^{(b)},c_\Gamma)\,\epsilon^{2-\alpha}\norm{\X^{(a)}-\X^{(b)}}_{C^{2,\alpha}}\norm{\overline{\bm{h}}}_{L^\infty} \\
\norm{\int_\T\int_0^{2\pi}\big(\bm{H}_3^{(a)}\,(\Phi^{(a)}\overline{\bm{h}})-\bm{H}_3^{(b)}\,(\Phi^{(b)}\overline{\bm{h}})\big)\,\epsilon\,d\theta'ds'}_{C^{0,\alpha}}
&\le c(\kappa_{*,\alpha}^{(a)},\kappa_{*,\alpha}^{(b)},c_\Gamma)\,\epsilon^{2-\alpha}\norm{\X^{(a)}-\X^{(b)}}_{C^{2,\alpha}}\norm{\overline{\bm{h}}}_{L^\infty} \,.
\end{align*}

We next turn to estimates for $\p_s\mc{S}[\Phi\overline{\bm{h}}(\theta)]$. We first note that, since $\overline{\bm{h}}(\theta)$ is independent of $s$, the only $s$-dependence in $(\Phi\overline{\bm{h}})(s,\theta)$ is due to the map $\Phi$, given by \eqref{eq:mapPhi_def}. We may calculate 
\begin{align*}
\p_s(\Phi\overline{\bm{h}})(s,\theta) &= (\p_s\be_{\rm t})(\overline{\bm{h}}\cdot\be_z)+(\p_s\be_{\rm n_1})(\overline{\bm{h}}\cdot\be_x)+(\p_s\be_{\rm n_2})(\overline{\bm{h}}\cdot\be_y)\,.
\end{align*}
In particular, we may bound 
\begin{equation}\label{eq:hbound}
\norm{\p_s(\Phi\overline{\bm{h}})}_{L^\infty(\Gamma_\epsilon)} \le c(\kappa_*)\norm{\overline{\bm{h}}}_{L^\infty(2\pi\T)}\,, \quad
\norm{\p_s(\Phi\overline{\bm{h}})}_{C^{0,\alpha}(\Gamma_\epsilon)} \le c(\kappa_{*,\alpha})\norm{\overline{\bm{h}}}_{L^\infty(2\pi\T)}\,.
\end{equation}
In addition, we note the following expansion for $\p_s\bR_{\rm r}$:
\begin{equation}\label{eq:ps_bRr}
\begin{aligned}
-\p_{s'}\bR_{\rm r} &=\be_{\rm t}(s') \\
\p_s\bR_{\rm r} &= \textstyle \be_{\rm t}(s) + 2\epsilon\sin(\frac{\theta-\theta'}{2})\big(\p_\theta\wh\kappa(s,\theta)\be_{\rm t}(s) -\kappa_3\be_r(s,\theta)\big) \\
&=: \textstyle-\p_{s'}\bR_{\rm r}+ (s-s')\bm{Q}_{\rm t}(s,s') + \epsilon\sin(\frac{\theta-\theta'}{2})\bm{Q}_{X1}\,.
\end{aligned}
\end{equation}
The closeness of $\p_s\bR_{\rm r}$ to $-\p_{s'}\bR_{\rm r}$ will be of particular importance in estimating the kernel $\bm{H}_1$. Finally, we record the following expansions of $\p_s\bR$ and $\p_s(\bR-\bR_{\rm r})$:
\begin{equation}\label{eq:ps_bRr_etc}
\begin{aligned}
\p_s\bR &= (1-\epsilon\wh\kappa(s,\theta))\be_{\rm t}(s)+\epsilon\kappa_3\be_\theta(s,\theta) =: \be_{\rm t}(s) + \epsilon\bm{Q}_{X2}(s,\theta) \\
\p_s(\bR-\bR_{\rm r}) &= \epsilon\wh\kappa(s,\theta')\be_{\rm t}(s)+\epsilon\kappa_3\be_\theta(s,\theta')
=: \epsilon \bm{Q}_{X2}(s,\theta')\,.
\end{aligned}
\end{equation}
Here, as usual, the remainders $\bm{Q}_{X1}$ and $\bm{Q}_{X2}$ satisfy \eqref{eq:Qj_Cbeta}.

We may use the expansion \eqref{eq:ps_bRr} of $\p_s\bR_{\rm r}$ to write $\bm{H}_1$ as 
\begin{align*}
\p_s\bm{H}_1&= -\frac{{\bf I}(\p_s\bR_{\rm r}\cdot\bR_{\rm r})}{\abs{\bR_{\rm r}}^3}+ \frac{\p_s\bR_{\rm r}\otimes\bR_{\rm r}+\bR_{\rm r}\otimes\p_s\bR_{\rm r}}{\abs{\bR_{\rm r}}^3} -\frac{3\bR_{\rm r}\otimes\bR_{\rm r}\,(\p_s\bR_{\rm r}\cdot\bR_{\rm r})}{\abs{\bR_{\rm r}}^5}\\
&= -\p_{s'}\bm{H}_1 
+ \frac{\big((s-s')\bm{Q}_{\rm t}+\epsilon\sin(\frac{\theta-\theta'}{2})\bm{Q}_{X1}\big)\otimes\bR_{\rm r}+\bR_{\rm r}\otimes\big((s-s')\bm{Q}_{\rm t}+\epsilon\sin(\frac{\theta-\theta'}{2})\bm{Q}_{X1}\big)}{\abs{\bR_{\rm r}}^3} \\
&\quad -\frac{{\bf I}\big((s-s')\bm{Q}_{\rm t}+\epsilon\sin(\frac{\theta-\theta'}{2})\bm{Q}_{X1}\big)\cdot\bR_{\rm r}}{\abs{\bR_{\rm r}}^3}
-\frac{3\bR_{\rm r}\otimes\bR_{\rm r}\,\big((s-s')\bm{Q}_{\rm t}+\epsilon\sin(\frac{\theta-\theta'}{2})\bm{Q}_{X1}\big)\cdot\bR_{\rm r}}{\abs{\bR_{\rm r}}^5}\,.
\end{align*}
For the kernel $-\p_{s'}\bm{H}_1$, we may integrate by parts in $s'$ and use Lemma \ref{lem:alpha_est}, case (1), along with \eqref{eq:hbound} to obtain 
\begin{align*}
\abs{{\rm p.v.}\int_\T\int_0^{2\pi}\p_{s'}\bm{H}_1\,(\Phi\overline{\bm{h}})\,\epsilon\,d\theta'ds'}_{\dot C^{0,\beta}}
&= \abs{\int_\T\int_0^{2\pi}\bm{H}_1\,\p_{s'}(\Phi\overline{\bm{h}})\,\epsilon\,d\theta'ds'}_{\dot C^{0,\beta}} \\
&\le c(\kappa_{*,\beta},c_\Gamma)\,\epsilon^{1-2\beta}\norm{\overline{\bm{h}}}_{L^\infty}\,.
\end{align*}
For the remaining terms in the above expansion of $\p_s\bm{H}_1$, we may use Lemma \ref{lem:alpha_est}, case (1), to estimate 
\begin{align*}
\abs{{\rm p.v.}\int_\T\int_0^{2\pi}\p_s\bm{H}_1\,(\Phi\overline{\bm{h}})\,\epsilon\,d\theta'ds'}_{\dot C^{0,\beta}} &\le c(\kappa_{*,\beta},c_\Gamma)\,\epsilon^{1-2\beta}\norm{\Phi\overline{\bm{h}}}_{L^\infty}
\le c(\kappa_{*,\beta},c_\Gamma)\,\epsilon^{1-2\beta}\norm{\overline{\bm{h}}}_{L^\infty}\,.
\end{align*}
Using the expansion for $\p_s\bm{H}_1$ and that each remainder $\bm{Q}$ satisfies \eqref{eq:Qj_Cbeta}, we may use a similar series of arguments to show that, for two curves $\X^{(a)}$ and $\X^{(b)}$ satisfying Lemma \ref{lem:XaXb_C2beta}, we have
\begin{align*}
&\abs{{\rm p.v.}\int_\T\int_0^{2\pi}\big(\p_s\bm{H}_1^{(a)}\,(\Phi^{(a)}\overline{\bm{h}})-\p_s\bm{H}_1^{(b)}\,(\Phi^{(b)}\overline{\bm{h}})\big)\,\epsilon\,d\theta'ds'}_{\dot C^{0,\beta}}\\ 
&\qquad \le c(\kappa_{*,\beta}^{(a)},\kappa_{*,\beta}^{(b)},c_\Gamma)\,\epsilon^{1-2\beta}\norm{\X^{(a)}-\X^{(b)}}_{C^{2,\beta}}\norm{\overline{\bm{h}}}_{L^\infty}\,.
\end{align*}

We next consider the kernel $\bm{H}_2$. Using \eqref{eq:bRr_expand} and \eqref{eq:ps_bRr_etc}, we may write
\begin{align*}
&\p_s \bm{H}_2 = \frac{\p_s\bR\otimes(\bR-\bR_{\rm r})}{\abs{\bR}^3}+ \frac{\bR\otimes\p_s(\bR-\bR_{\rm r})}{\abs{\bR}^3} - \frac{3\bR\otimes(\bR-\bR_{\rm r})(\p_s\bR\cdot\bR)}{\abs{\bR}^5} \\
&\qquad +\frac{\p_s(\bR-\bR_{\rm r})\otimes\bR_{\rm r}}{\abs{\bR}^3}  + \frac{(\bR-\bR_{\rm r})\otimes\p_s\bR_{\rm r}}{\abs{\bR}^3} - \frac{3(\bR-\bR_{\rm r})\otimes\bR_{\rm r}(\p_s\bR\cdot\bR)}{\abs{\bR}^5}  \\
&\qquad -3\frac{\p_s(\bR-\bR_{\rm r})\cdot\bR + \p_s\bR_{\rm r}\cdot(\bR-\bR_{\rm r})}{\abs{\bR}^5}\bR_{\rm r}\otimes\bR_{\rm r} \\
&\quad + \bigg(\frac{1}{\abs{\bR}^3}-\frac{1}{\abs{\bR_{\rm r}}^3} \bigg)\big(\p_s\bR_{\rm r}\otimes\bR_{\rm r}+\bR_{\rm r}\otimes\p_s\bR_{\rm r}\big)
 -3\bigg(\frac{1}{\abs{\bR}^5}-\frac{1}{\abs{\bR_{\rm r}}^5} \bigg)(\p_s\bR_{\rm r}\cdot\bR_{\rm r})\bR_{\rm r}\otimes\bR_{\rm r} \\
&= \frac{\epsilon(s-s')\p_s\bR\otimes\bm{Q}_r}{\abs{\bR}^3}
+ \frac{\epsilon\bR\otimes\bm{Q}_{X2}}{\abs{\bR}^3} 
- \frac{3\epsilon(s-s')\bR\otimes\bm{Q}_r(\p_s\bR\cdot\bR)}{\abs{\bR}^5} \\
&\qquad +\frac{\epsilon\bm{Q}_{X2}\otimes\bR_{\rm r}}{\abs{\bR}^3}  + \frac{\epsilon(s-s')\bm{Q}_r\otimes\p_s\bR_{\rm r}}{\abs{\bR}^3} - \frac{3\epsilon(s-s')\bm{Q}_r\otimes\bR_{\rm r}(\p_s\bR\cdot\bR)}{\abs{\bR}^5}  \\
&\qquad -3\frac{\epsilon\bm{Q}_{X2}\cdot\bR + \epsilon(s-s')\p_s\bR_{\rm r}\cdot\bm{Q}_r}{\abs{\bR}^5}\bR_{\rm r}\otimes\bR_{\rm r} \\
&\quad - \bigg(\frac{2\epsilon(s-s')\bm{Q}_r\cdot\bR_{\rm r}+\epsilon^2(s-s')^2\abs{\bm{Q}_r}^2}{\abs{\bR}|\bR_{\rm r}|(\abs{\bR}+|\bR_{\rm r}|)}\sum_{\ell=0}^{2}\frac{1}{|\bR|^\ell|\bR_{\rm r}|^{2-\ell}} \bigg)\big(\p_s\bR_{\rm r}\otimes\bR_{\rm r}+\bR_{\rm r}\otimes\p_s\bR_{\rm r}\big)\\
&\quad +3\bigg(\frac{2\epsilon(s-s')\bm{Q}_r\cdot\bR_{\rm r}+\epsilon^2(s-s')^2\abs{\bm{Q}_r}^2}{\abs{\bR}|\bR_{\rm r}|(\abs{\bR}+|\bR_{\rm r}|)}\sum_{\ell=0}^{4}\frac{1}{|\bR|^\ell|\bR_{\rm r}|^{4-\ell}}\bigg)(\p_s\bR_{\rm r}\cdot\bR_{\rm r})\bR_{\rm r}\otimes\bR_{\rm r} \,.
\end{align*}
Note that each term contains an additional factor of $\epsilon$. Using Lemma \ref{lem:alpha_est}, case (2), we thus have
\begin{align*}
\abs{{\rm p.v.}\int_\T\int_0^{2\pi}\p_s\bm{H}_2\,(\Phi\overline{\bm{h}})\,\epsilon\,d\theta'ds'}_{\dot C^{0,\alpha}} &\le c(\kappa_{*,\alpha^+},c_\Gamma)\,\epsilon^{1-\alpha^+}\norm{\Phi\overline{\bm{h}}}_{C^{0,\alpha}}
\le c(\kappa_{*,\alpha^+},c_\Gamma)\,\epsilon^{1-\alpha^+}\norm{\overline{\bm{h}}}_{C^{0,\alpha}}\,.
\end{align*}
In addition, using that each of the remainder terms $\bm{Q}$ satisfy \eqref{eq:Qj_Cbeta}, we may use the above expansion of $\bm{H}_2$ to show that, for two curves $\X^{(a)}$ and $\X^{(b)}$ satisfying Lemma \ref{lem:XaXb_C2beta}, we have
\begin{align*}
&\abs{{\rm p.v.}\int_\T\int_0^{2\pi}\big(\p_s\bm{H}_2^{(a)}\,(\Phi^{(a)}\overline{\bm{h}})-\p_s\bm{H}_2^{(b)}\,(\Phi^{(b)}\overline{\bm{h}})\big)\,\epsilon\,d\theta'ds'}_{\dot C^{0,\alpha}}\\ 
&\qquad \le c(\kappa_{*,\alpha^+}^{(a)},\kappa_{*,\alpha^+}^{(b)},c_\Gamma)\,\epsilon^{1-\alpha^+}\norm{\X^{(a)}-\X^{(b)}}_{C^{2,\alpha^+}}\norm{\overline{\bm{h}}}_{C^{0,\alpha}}\,.
\end{align*}

Finally, we turn to the kernel $\bm{H}_3$. We have
\begin{align*}
\p_s\bm{H}_3 &= \bigg(\frac{{\bf I}\,\p_s\bR\cdot\bR}{\abs{\bR}^3}+ \frac{\p_s\bR\otimes\bR+\bR\otimes\p_s\bR}{\abs{\bR}^3} + 3\frac{(\bR\otimes\bR)\,\p_s\bR\cdot\bR}{\abs{\bR}^5}\bigg)\epsilon\wh\kappa(s',\theta')\,.
\end{align*}
Noting the additional factor of $\epsilon$ and using the form \eqref{eq:ps_bRr_etc} of $\p_s\bR$, we may use Lemma \ref{lem:alpha_est} to obtain
\begin{align*}
\abs{{\rm p.v.}\int_\T\int_0^{2\pi}\p_s\bm{H}_3\,(\Phi\overline{\bm{h}})\,\epsilon\,d\theta'ds'}_{\dot C^{0,\alpha}} &\le c(\kappa_{*,\alpha^+},c_\Gamma)\,\epsilon^{1-\alpha^+}\norm{\Phi\overline{\bm{h}}}_{C^{0,\alpha}}
\le c(\kappa_{*,\alpha^+},c_\Gamma)\,\epsilon^{1-\alpha^+}\norm{\overline{\bm{h}}}_{C^{0,\alpha}}\,.
\end{align*}
Furthermore, for two nearby curves $\X^{(a)}$ and $\X^{(b)}$, we have
\begin{align*}
&\abs{{\rm p.v.}\int_\T\int_0^{2\pi}\big(\p_s\bm{H}_3^{(a)}\,(\Phi^{(a)}\overline{\bm{h}})-\p_s\bm{H}_3^{(b)}\,(\Phi^{(b)}\overline{\bm{h}})\big)\,\epsilon\,d\theta'ds'}_{\dot C^{0,\alpha}}\\ 
&\qquad \le c(\kappa_{*,\alpha^+}^{(a)},\kappa_{*,\alpha^+}^{(b)},c_\Gamma)\,\epsilon^{1-\alpha^+}\norm{\X^{(a)}-\X^{(b)}}_{C^{2,\alpha^+}}\norm{\overline{\bm{h}}}_{C^{0,\alpha}}\,.
\end{align*}

Defining 
\begin{align*}
\mc{H}_+[\Phi\overline{\bm{h}}] = \frac{1}{8\pi}\int_\T\int_0^{2\pi}\bm{H}_1\,(\Phi\overline{\bm{h}})(s',\theta')\,\epsilon\,d\theta'ds'\,, \quad 
\mc{H}_\epsilon[\Phi\overline{\bm{h}}] = \frac{1}{8\pi}\int_\T\int_0^{2\pi}\big(\bm{H}_2 + \bm{H}_3\big)\,(\Phi\overline{\bm{h}})(s',\theta')\,\epsilon\,d\theta'ds'\,,
\end{align*}
we obtain Lemma \ref{lem:single_const_in_s}.
\end{proof}


\section{Bounds for full Neumann data}\label{sec:full_neumann}
This section is concerned with the bounds for the full Neumann data $\bw(s,\theta)=-\bm{\sigma}[\bu]\bm{n}\big|_{\Gamma_\epsilon}$ stated in Lemma \ref{lem:full_neumann}. We restate the lemma below for convenience. 
\begin{lemma}[Bounds for $\bw(s,\theta)$, restated]\label{lem:full_neumann0}
Let $0<\alpha<\gamma<\beta<1$ and consider a slender body $\Sigma_\epsilon$ with centerline $\X(s)\in C^{2,\beta}(\T)$.
Given $\theta$-independent Dirichlet data $\bv(s)\in C^{1,\alpha}(\T)$ along the filament surface $\Gamma_\epsilon$, let $\bu(\bx)$, $\bx\in\Omega_\epsilon$, denote the corresponding solution to the Stokes equations in $\Omega_\epsilon$, and let $\bw(s,\theta)$ denote the corresponding Neumann boundary value $\bw=-\bm{\sigma}[\bu]\bm{n}\big|_{\Gamma_\epsilon}$. Then $\bw$ may be decomposed as 
\begin{align*}
\bw(s,\theta) = \bw_\epsilon(s,\theta) + \bw_+(s,\theta)
\end{align*}
where
\begin{equation}\label{eq:fullN0}
\begin{aligned}
\norm{\bw_\epsilon}_{C^{0,\alpha}} &\le c(\kappa_*,c_\Gamma)\norm{\bv}_{C^{1,\alpha}}\\
\norm{\bw_+}_{C^{0,\gamma}}&\le c(\epsilon,\kappa_{*,\gamma^+},c_\Gamma)\,\norm{\bv}_{C^{0,\gamma}}
\end{aligned}
\end{equation}
for any $\gamma^+\in(\gamma,\beta]$.
In addition, given two nearby filaments $\Sigma_\epsilon^{(a)},\Sigma_\epsilon^{(b)}$ with centerlines $\X^{(a)}$ and $\X^{(b)}$ satisfying Lemma \ref{lem:XaXb_C2beta}, we have
\begin{equation}\label{eq:fullN_lip0}
\begin{aligned}
\norm{\bw_\epsilon^{(a)}-\bw_\epsilon^{(b)}}_{C^{0,\alpha}}&\le c(\kappa_*^{(a)},\kappa_*^{(b)},c_\Gamma)\norm{\X^{(a)}-\X^{(b)}}_{C^2}\norm{\bv}_{C^{1,\alpha}}\\
\norm{\bw_+^{(a)}-\bw_+^{(b)}}_{C^{0,\gamma}}&\le c(\epsilon,\kappa_{*,\gamma^+}^{(a)},\kappa_{*,\gamma^+}^{(b)},c_\Gamma)\,\norm{\X^{(a)}-\X^{(b)}}_{C^{2,\gamma^+}}\norm{\bv}_{C^{0,\gamma}}\,,
\end{aligned}
\end{equation}
where $\bw^{(a)},\bw^{(b)}$ denote the Neumann data along filaments $\Sigma_\epsilon^{(a)}$ and $\Sigma_\epsilon^{(b)}$, respectively.

In both \eqref{eq:fullN} and \eqref{eq:fullN_lip}, the $\epsilon$-dependence is explicit in the bounds for $\bw_\epsilon$ but not for $\bw_+$.
\end{lemma}

The proof of Lemma \ref{lem:full_neumann} will again rely on a boundary integral formulation of the full Stokes Dirichlet-to-Neumann map along the slender body surface $\Gamma_\epsilon$. However, it will be more convenient to use a different formulation from the one presented in section \ref{subsubsec:SB_DtN_boundaryintegral} in order to extract more precise information.

Following \cite{pozrikidis1992boundary,power1987second}, given Dirichlet data $\bv(\bx)$ on $\Gamma_\epsilon$, the Stokes flow outside of the slender body $\Sigma_\epsilon$ may be represented by 
\begin{equation}\label{eq:u_completeddouble}
\bu(\bx) = \mc{D}[\bm{\varphi}](\bx) + \bm{V}[\bm{\varphi}](\bx)
\end{equation}
where $\mc{D}$ is the double layer operator given by \eqref{eq:S_and_D} and $\bm{V}$ is a \emph{completion flow} serving the following purpose. Along the slender body surface, the density $\bm{\varphi}$ must satisfy the integral equation 
\begin{equation}\label{eq:completed_double}
\bv(\bx) = \textstyle (\frac{1}{2}{\bf I}+\mc{D})[\bm{\varphi}](\bx) + \bm{V}[\bm{\varphi}](\bx)\,, \qquad \bx\in\Gamma_\epsilon\,,
\end{equation}
where the $\frac{1}{2}{\bf I}$ arises from the jump condition for the double layer operator on $\Gamma_\epsilon$. Since the operator $\frac{1}{2}{\bf I}+\mc{D}$ has a well-known 6-dimensional nullspace along $\Gamma_\epsilon$ consisting of rigid motions \cite{pozrikidis1992boundary}, an appropriate choice of completion flow $\bm{V}$ is necessary for invertibility. 

There are many possible choices for $\bm{V}$ (see \cite{corona2017integral,power1987second}). We will use a form of $\bm{V}$ proposed by \cite{keaveny2011applying} that is particularly well-suited for slender filaments. 
Consider $\bx'\in\Gamma_\epsilon$, denoted by $\bx'=\X(s')+\epsilon\be_r(s',\theta')$, and consider the velocity field $\bm{V}$ given by 
\begin{equation}\label{eq:completion}
\bm{V}[\bm{\varphi}](\bx) = \int_{\Gamma_\epsilon}\mc{G}(\bx,\X(s'))\,\bm{\varphi}(\bx')\,dS_{x'} + \int_{\Gamma_\epsilon}\bm{l}_{\mc R}(\bx,\X(s')) \times \big((\bx'-\X(s'))\times\bm{\varphi}(\bx')\big)\,dS_{x'}\,.
\end{equation}
We emphasize that $\X(s')$ corresponds to $\bx'$ while $\bx$ is independent.
Here $\bm{l}_{\mc R}$ is a \emph{rotlet}, given by
\begin{equation}\label{eq:rotlet}
\bm{l}_{\mc R}(\bx,\by) = -\frac{(\bx-\by)}{\abs{\bx-\by}^3}\,,
\end{equation} 
and describes the Stokes flow due to a point torque at $\by$ \cite{chwang1975hydromechanics}. The completion flow \eqref{eq:completion} may be understood as the Stokes flow due to a line density of point forces and point torques distributed along the filament centerline. Note that $\bm{V}$ is a smooth function of $\bx$ along $\Gamma_\epsilon$ since these singularities lie along the centerline of the filament rather than the surface.

Given the representation \eqref{eq:u_completeddouble}, \eqref{eq:completed_double} of the Stokes flow $\bu(\bx)$ in $\Omega_\epsilon$ with Dirichlet data $\bv(\bx)$, we may calculate the corresponding Neumann boundary value for $\bu$ along $\Gamma_\epsilon$. We begin by noting the following expression for the stress due to the double layer $\mc{D}[\bm{\varphi}](\bx)$, $\bx\in\Omega_\epsilon$:
\begin{equation}\label{eq:sigmaD}
\begin{aligned}
\bm{\sigma}[\mc{D}[\bm{\varphi}]](\bx) &= \nabla_x\mc{D}[\bm{\varphi}]+ (\nabla_x\mc{D}[\bm{\varphi}])^{\rm T} - p^\mc{D}[\bm{\varphi}](\bx){\bf I}\,, \\
p^{\mc D}[\bm{\varphi}](\bx) &= \int_{\Gamma_\epsilon}\bm{p}^{\mc D}(\bx,\bx')\cdot \bm{\varphi}(\bx')\, dS_{x'}\,, \\
\bm{p}^{\mc D}(\bx,\bx';\bm{n}(\bx')) &= \frac{1}{2\pi}\bigg(-\frac{\bm{n}(\bx')}{\abs{\bx-\bx'}^3}+3\frac{(\bx-\bx')(\bx-\bx')\cdot\bm{n}(\bx')}{\abs{\bx-\bx'}^5} \bigg)\,.
\end{aligned}
\end{equation}
Here $p^{\mc D}$ is the pressure associated with the double layer representation of $\bu$ and $\bm{p}^{\mc D}$ is the vector-valued pressure kernel for this representation \cite{chwang1975hydromechanics,mitchell2017generalized,pozrikidis1992boundary}.

Letting $\bR=\bx-\bx'$, $\bm{n}=\bm{n}(\bx)$, and $\bm{n}'=\bm{n}(\bx')$, we define the following kernel corresponding to the double layer stress tensor \eqref{eq:sigmaD}: 
\begin{equation}\label{eq:KT}
\begin{aligned}
\bm{K}_{\mc T}(\bx,\bx';\bm{n},\bm{n}') &= \frac{1}{4\pi}\bigg( 2\frac{\bm{n}\otimes\bm{n}'}{\abs{\bR}^3} + 3\frac{(\bR\cdot\bm{n}')\bR\otimes\bm{n} + (\bm{n}\cdot\bm{n}')\bR\otimes\bR}{\abs{\bR}^5} \\
&\qquad + 3\frac{(\bR\cdot\bm{n})(\bR\cdot\bm{n}'){\bf I} + (\bR\cdot\bm{n})\bm{n}'\otimes\bR}{\abs{\bR}^5} -30\frac{(\bR\cdot\bm{n})(\bR\cdot\bm{n}')\bR\otimes\bR}{\abs{\bR}^7}
\bigg)\,.
\end{aligned}
\end{equation}
Then for $\bx\in\Gamma_\epsilon$, on the surface of the filament, we may consider the stress associated to the double layer operator as 
\begin{equation}\label{eq:sigmaDn}
\bm{\sigma}[\mc{D}[\bm{\varphi}]]\bm{n}(\bx) = {\rm p.v.}\int_{\Gamma_\epsilon}\bm{K}_{\mc T}(\bx,\bx';\bm{n},\bm{n}')\big( \bm{\varphi}(\bx')-\bm{\varphi}(\bx)\big)\,dS_{x'}\,.
\end{equation}
Here the normal vectors $\bm{n}$ and $\bm{n}'$ both point out of the slender body $\Sigma_\epsilon$, into the fluid domain $\Omega_\epsilon$. 

To obtain the representation \eqref{eq:sigmaDn}, we use that both the double layer kernel $\bm{K}_{\mc D}$ \eqref{eq:stresslet} and the corresponding pressure kernel $\bm{p}^\mc{D}$ \eqref{eq:sigmaD} integrate to zero within the exterior domain $\Omega_\epsilon$, i.e.
\begin{align*}
\int_{\Gamma_\epsilon}\bm{K}_{\mc{D}}(\by,\bx')\,dS_{x'} = 0\,, \quad
\int_{\Gamma_\epsilon}\bm{p}^{\mc D}(\by,\bx';\bm{n}(\bx'))\,dS_{x'} &= 0 \,,
\qquad \by\in \Omega_\epsilon\,.
\end{align*}
For any constant-in-$\by$ vector-valued surface density $\bm{\varphi}(\bx)$, $\bx\in\Gamma_\epsilon$, we then have
\begin{align*}
\nabla_y\int_{\Gamma_\epsilon}\bm{K}_{\mc{D}}(\by,\bx')\,\bm{\varphi}(\bx)\,dS_{x'}=0\,,
\quad \int_{\Gamma_\epsilon}\bm{p}^{\mc D}(\by,\bx';\bm{n}(\bx'))\cdot\bm{\varphi}(\bx)\,dS_{x'}=0\,;
\qquad \by\in \Omega_\epsilon\,, \; \bx\in\Gamma_\epsilon\,.
\end{align*}
Putting these together, for $\bx\in\Gamma_\epsilon$, we may understand $\bm{\sigma}[\mc{D}[\bm{\varphi}]]\bm{n}(\bx)$ as
\begin{align*}
\bm{\sigma}[\mc{D}[\bm{\varphi}]]\bm{n}(\bx) &= \lim_{\by\to\bx} \bigg( \nabla_y\int_{\Gamma_\epsilon}\bm{K}_{\mc{D}}(\by,\bx')\,\big(\bm{\varphi}(\bx')-\bm{\varphi}(\bx)\big)\,dS_{x'} \\
&\qquad + \bigg(\nabla_y\int_{\Gamma_\epsilon}\bm{K}_{\mc{D}}(\by,\bx')\,\big(\bm{\varphi}(\bx')-\bm{\varphi}(\bx)\big)\,dS_{x'} \bigg)^{\rm T} \\
&\qquad -{\bf I}\int_{\Gamma_\epsilon}\bm{p}^{\mc D}(\by,\bx';\bm{n}(\bx'))\cdot(\bm{\varphi}(\bx')-\bm{\varphi}(\bx))\,dS_{x'}\bigg)\bm{n}(\bx) \\
&= \lim_{\by\to\bx}\int_{\Gamma_\epsilon}\bm{K}_{\mc T}(\by,\bx';\bm{n},\bm{n}')\big( \bm{\varphi}(\bx')-\bm{\varphi}(\bx)\big)\,dS_{x'}\,, 
\quad \by\in\Omega_\epsilon\,.
\end{align*}

We may also calculate the stress corresponding to the completion flow $\bm{V}$, given by
\begin{align*}
\bm{\sigma}[\bm{V}[\bm{\varphi}]](\bx) &=\nabla_x\bm{V}[\bm{\varphi}] + (\nabla_x\bm{V}[\bm{\varphi}])^{\rm T} + p^V(\bx){\bf I}\,,\\
p^V(\bx) &= \frac{1}{4\pi}\int_{\Gamma_\epsilon}\frac{(\bx-\X(s'))\cdot\bm{\varphi}(\bx')}{\abs{\bx-\X(s')}^3}\,dS_{x'}\,.
\end{align*}
Here $p^V$ is the pressure corresponding to the flow $\bm{V}$ \cite{mitchell2017generalized}. For $\bx\in \Gamma_\epsilon$, we define $\bR_X=\bx-\X(s')$ and let $\bm{n}=\bm{n}(\bx)$. Along $\Gamma_\epsilon$, we may then write
\begin{equation}\label{eq:sigmaVn}
\begin{aligned}
\bm{\sigma}[\bm{V}[\bm{\varphi}]]\bm{n}(\bx) &= \frac{3}{8\pi}\int_{\Gamma_\epsilon}\bigg(-2\frac{\big(\bR_X\cdot\bm{n} \big)(\bR_X\cdot\bm{\varphi}(\bx'))\bR_X}{\abs{\bR_X}^5} 
+ \frac{\bR_X\cdot\bm{n}}{\abs{\bR_X}^5}\,\bR_X\times \big((\bx'-\X(s'))\times\bm{\varphi}(\bx') \big) \\
&\qquad + \frac{\bR_X}{\abs{\bR_X}^5}\,\bm{n}\cdot\big(\bR_X\times \big((\bx'-\X(s'))\times\bm{\varphi}(\bx')\big)\big)\bigg)\,dS_{x'}\,.
\end{aligned}
\end{equation}

Given the representations \eqref{eq:sigmaDn} and \eqref{eq:sigmaVn}, we define the hypersingular operator $\mc{T}$ to be the Neumann boundary value associated with the flow \eqref{eq:u_completeddouble}, given by 
\begin{equation}\label{eq:T_def}
\mc{T}[\bm{\varphi}](\bx) = -\bm{\sigma}[\mc{D}[\bm{\varphi}]]\bm{n}(\bx) -\bm{\sigma}[\bm{V}[\bm{\varphi}]]\bm{n}(\bx)\,, \qquad \bx\in \Gamma_\epsilon\,.
\end{equation}
Here we note that the negative sign arises since $\bm{n}$ points into the fluid domain, out of the slender body $\Sigma_\epsilon$.

Given Dirichlet data $\bv(\bx)$ along $\Gamma_\epsilon$, we may use the completed double layer representation \eqref{eq:u_completeddouble}, \eqref{eq:completed_double} of $\bu$ to calculate the corresponding Neumann boundary value $\bw$ via
\begin{equation}\label{eq:eq_for_w}
\bw(\bx) = \textstyle \mc{T}\big[(\frac{1}{2}{\bf I}+\mc{D}+\bm{V})^{-1}[\bv] \big](\bx)\, , \qquad \bx\in \Gamma_\epsilon\,.
\end{equation}

We will use the representation \eqref{eq:eq_for_w} of the full Neumann boundary value $\bw$ to prove Lemma \ref{lem:full_neumann}. In order to do so, we will rely on the following series of lemmas bounding each of the components of \eqref{eq:eq_for_w}. We state each lemma below and then immediately use the lemmas to prove Lemma \ref{lem:full_neumann}. The proof of each component lemma is then given in subsequent sections \ref{subsec:comp_doub_layer}, \ref{subsec:inv_doub_layer}, and \ref{subsec:hyper}.

We begin with bounds for the completed double layer operator \eqref{eq:completed_double} along $\Gamma_\epsilon$. 
\begin{lemma}[Completed double layer smoothing]\label{lem:comp_double_layer}
Let $0<\gamma<\beta<1$ and consider a filament $\Sigma_\epsilon$ with centerline $\X\in C^{2,\beta}$. Given $\bm{\varphi}\in C^{1,\alpha}(\Gamma_\epsilon)$, the double layer operator $\mc{D}$ \eqref{eq:S_and_D} and completion flow $\bm{V}$ \eqref{eq:completion} along $\Gamma_\epsilon$ satisfy
\begin{equation}
\norm{(\mc{D}+\bm{V})[\bm{\varphi}]}_{C^{1,\gamma}(\Gamma_\epsilon)} \le c(\kappa_{*,\gamma^+},c_\Gamma)\,\epsilon^{-1-\gamma}\norm{\bm{\varphi}}_{C^{0,\gamma}(\Gamma_\epsilon)}
\end{equation}
for any $\gamma^+\in(\gamma,\beta]$.
Furthermore, given two nearby filaments with centerlines $\X^{(a)}$, $\X^{(b)}$ satisfying Lemma \ref{lem:XaXb_C2beta}, we have 
\begin{equation}
\norm{(\mc{D}^{(a)}+\bm{V}^{(a)}-\mc{D}^{(b)}-\bm{V}^{(b)})[\bm{\varphi}]}_{C^{1,\gamma}} \le c(\kappa_{*,\gamma^+}^{(a)},\kappa_{*,\gamma^+}^{(b)},c_\Gamma)\,\epsilon^{-1-\gamma}\norm{\X^{(a)}-\X^{(b)}}_{C^{2,\gamma^+}}\norm{\bm{\varphi}}_{C^{0,\gamma}}\,.
\end{equation}
\end{lemma}
The proof of Lemma \ref{lem:comp_double_layer} appears in section \ref{subsec:comp_doub_layer}. 
%
%
%
We will additionally rely on the following uniform bound for $(\frac{1}{2}{\bf I}+\mc{D}+\bm{V})^{-1}$ with respect to the filament centerline shape.
\begin{lemma}[Inverse completed double layer]\label{lem:inv_doub_layer}
Let $0<\gamma<\beta<1$, and let $\epsilon>0$ be fixed. Given any filament $\Sigma_\epsilon$ with centerline $\X(s)\in C^{2,\beta}(\T)$ satisfying the non-self-intersection condition \eqref{eq:cGamma}, the operator $(\frac{1}{2}{\bf I}+\mc{D}+\bm{V})^{-1}$ is uniformly bounded with respect to the filament centerline shape. In particular, for $\bm{\varphi}\in C^{0,\gamma}(\Gamma_\epsilon)$, we have
\begin{equation}\label{eq:inv_doub}
\textstyle \norm{(\frac{1}{2}{\bf I}+\mc{D}+\bm{V})^{-1}[\bm{\varphi}]}_{C^{0,\gamma}}
\le c(\epsilon,\kappa_{*,\gamma^+},c_\Gamma)\norm{\bm{\varphi}}_{C^{0,\gamma}}
\end{equation}
for any $\gamma^+\in(\gamma,\beta]$.
Moreover, given two nearby filaments with centerlines $\X^{(a)}$, $\X^{(b)}$ satisfying Lemma \ref{lem:XaXb_C2beta}, we have 
\begin{equation}\label{eq:inv_doub_lip}
\begin{aligned}
&\textstyle \norm{(\frac{1}{2}{\bf I}+\mc{D}^{(a)}+\bm{V}^{(a)})^{-1}[\bm{\varphi}]-(\frac{1}{2}{\bf I}+\mc{D}^{(b)}+\bm{V}^{(b)})^{-1}[\bm{\varphi}]}_{C^{0,\gamma}}\\
&\qquad \le c(\epsilon,\kappa_{*,\gamma^+}^{(a)},\kappa_{*,\gamma^+}^{(b)},c_\Gamma)\norm{\X^{(a)}- \X^{(b)}}_{C^{2,\gamma^+}}\norm{\bm{\varphi}}_{C^{0,\gamma}}\,.
\end{aligned}
\end{equation}
\end{lemma}
The proof of Lemma \ref{lem:inv_doub_layer} is given in section \ref{subsec:inv_doub_layer}. Note that the $\epsilon$-dependence here is not explicit.

The final ingredient used in the proof of Lemma \ref{lem:full_neumann} involves the mapping properties of the hypersingular operator $\mc{T}$ given by \eqref{eq:T_def}. To determine these mapping properties, we will first need the following auxiliary lemma. The bound below is a more general version of the result appearing in \cite[Lemma 4.3]{laplace}.
\begin{lemma}[Hypersingular operator: estimating $\abs{\cdot}_{\dot C^{0,\alpha}}$ seminorms]\label{lem:new_alpha_est}
Let $\ell,k,n$ be nonnegative integers satisfying $\ell+k+3=n$ and $\ell+k$ is even. Suppose $\varphi(s,\theta)\in C^{1,\alpha}(\Gamma_\epsilon)$ and $g(s,\theta,s-\bars,\theta-\bartheta)\in C^{0,\alpha^+}(\Gamma_\epsilon\times\Gamma_\epsilon)$ for any $0<\alpha<\alpha^+<1$. For $\bR$ as in \eqref{eq:curvedR}, let $I_{\ell kn}(s,\theta,\bars,\bartheta)$ denote 
\begin{equation}\label{eq:I_lkn}
 I_{\varphi;\ell kn}(s,\theta,\bars,\bartheta):= \frac{\bars^\ell(\epsilon\sin(\frac{\bartheta}{2}))^k\,g(s,\theta,s-\bars,\theta-\bartheta)}{|\bR(s,\theta,\bars,\bartheta)|^n}\, \big(\varphi(s-\bars,\theta-\bartheta)-\varphi(s,\theta)\big)\,.
\end{equation} 
Then for any $(s_0,\theta_0)\in[-1/2,1/2]\times[-\pi,\pi]$ with $s_0^2+\theta_0^2\neq 0$ we have 
\begin{equation}\label{eq:new_alphalem}
\begin{aligned}
\bigg|{\rm p.v.}&\int_{-s_0-1/2}^{-s_0+1/2}\int_{-\theta_0-\pi}^{-\theta_0+\pi}I_{\varphi,\ell kn}(s_0+s,\theta_0+\theta,s_0+\bars,\theta_0+\bartheta)\,\epsilon\, d\bartheta d\bars \\
&\qquad\qquad\qquad\qquad\qquad  - {\rm p.v.}\int_{-1/2}^{1/2}\int_{-\pi}^{\pi}I_{\varphi,\ell kn}(s,\theta,\bars,\bartheta)\,\epsilon\, d\bartheta d\bars\,\bigg| \\
&\le c(\kappa_*,c_\Gamma)\bigg(\big(\textstyle \norm{\p_s\varphi}_{L^\infty}+ \norm{\frac{1}{\epsilon}\p_\theta\varphi}_{L^\infty}\big)\big(\norm{g}_{C^{0,\alpha^+}_1}+\norm{g}_{C^{0,\alpha}_2}\big) \\
&\qquad + \textstyle \big(\norm{\p_s\varphi}_{C^{0,\alpha}}+ \norm{\frac{1}{\epsilon}\p_\theta\varphi}_{C^{0,\alpha}}\big)\norm{g}_{L^\infty}\bigg)  \sqrt{s_0^2+\epsilon^2\theta_0^2}^{\,\alpha}\,,
\end{aligned}
\end{equation}
where $\norm{\cdot}_{C^{0,\alpha^+}_1}$ and $\norm{\cdot}_{C^{0,\alpha}_2}$ are as in \eqref{eq:Y_norms}.
\end{lemma}

The proof of Lemma \ref{lem:new_alpha_est} appears in Appendix \ref{app:new_alphalem}. 
Using Lemma \ref{lem:new_alpha_est}, we may then show the final ingredient needed in the proof of Lemma \ref{lem:full_neumann}. In particular, we show the following bounds for the hypersingular operator $\mc{T}$ given by \eqref{eq:T_def}. 
\begin{lemma}[Hypersingular operator bounds]\label{lem:hyper}
Let $0<\alpha<\beta<1$. Given a filament $\Sigma_\epsilon$ with centerline $\X\in C^{2,\beta}$, we consider the hypersingular operator $\mc{T}$ given by \eqref{eq:T_def}. Given $\bm{\varphi}\in C^{1,\alpha}(\Gamma_\epsilon)$, we may decompose $\mc{T}[\bm{\varphi}]$ as
\begin{align*}
\mc{T}[\bm{\varphi}] = \mc{T}_\epsilon[\bm{\varphi}] + \mc{T}_{+}[\bm{\varphi}]
\end{align*}
where 
\begin{equation}\label{eq:hyperest}
\begin{aligned}
\norm{\mc{T}_\epsilon[\bm{\varphi}]}_{C^{0,\alpha}} &\le 
c(\kappa_*,c_\Gamma) \textstyle \big(\norm{\p_s\bm{\varphi}}_{C^{0,\alpha}}+ \norm{\frac{1}{\epsilon}\p_\theta\bm{\varphi}}_{C^{0,\alpha}}\big) \\
\norm{\mc{T}_+[\bm{\varphi}]}_{C^{0,\alpha}} &\le c(\kappa_{*,\alpha^+},c_\Gamma)\,\epsilon^{-1-\alpha^+} \norm{\bm{\varphi}}_{C^1}
\end{aligned}
\end{equation}
for any $\alpha^+\in(\alpha,\beta]$.
In addition, given two nearby curves $\X^{(a)}$ and $\X^{(b)}$ satisfying Lemma \ref{lem:XaXb_C2beta}, we have the Lipschitz estimate 
\begin{equation}\label{eq:hyperest_lip}
\begin{aligned}
\norm{\mc{T}_\epsilon^{(a)}[\bm{\varphi}]-\mc{T}_\epsilon^{(b)}[\bm{\varphi}]}_{C^{0,\alpha}} &\le 
c(\kappa_*^{(a)},\kappa_*^{(b)},c_\Gamma)\, \norm{\X^{(a)}-\X^{(b)}}_{C^2} \textstyle \big(\norm{\p_s\bm{\varphi}}_{C^{0,\alpha}}+ \norm{\frac{1}{\epsilon}\p_\theta\bm{\varphi}}_{C^{0,\alpha}}\big) \\
\norm{\mc{T}_+^{(a)}[\bm{\varphi}]-\mc{T}_+^{(b)}[\bm{\varphi}]}_{C^{0,\alpha}} &\le c(\kappa_{*,\alpha^+}^{(a)},\kappa_{*,\alpha^+}^{(b)},c_\Gamma)\,\epsilon^{-1-\alpha^+} \norm{\X^{(a)}-\X^{(b)}}_{C^{2,\alpha^+}} \norm{\bm{\varphi}}_{C^1}\,.
\end{aligned}
\end{equation}
\end{lemma}
The proof of Lemma \ref{lem:hyper} appears in section \ref{subsec:hyper}. \\

Given Lemmas \ref{lem:comp_double_layer}, \ref{lem:inv_doub_layer}, and \ref{lem:hyper}, we may proceed to the proof of the main result of this section, Lemma \ref{lem:full_neumann}.
\begin{proof}[Proof of Lemma \ref{lem:full_neumann}: bounds on full Neumann boundary value]
Given $\theta$-independent Dirichlet data $\bv(s)$, we may solve for the intermediate density $\bm{\varphi}(\bx)=\bm{\varphi}(s,\theta)$ in the completed double layer representation \eqref{eq:completed_double} as
\begin{align*}
\bm{\varphi}(s,\theta) = 2\bv(s)-\mc{M}[\bm{v}(s)]\,, \qquad \textstyle \mc{M}=(\mc{D}+\bm{V})(\frac{1}{2}{\bf I}+ \mc{D}+\bm{V})^{-1}\,.
\end{align*}
Using Lemmas \ref{lem:comp_double_layer} and \ref{lem:inv_doub_layer}, we may estimate
\begin{align*}
\norm{\mc{M}[\bm{v}]}_{C^{1,\gamma}} &\le c(\kappa_{*,\gamma^+},c_\Gamma)\,\epsilon^{-1-\gamma}\textstyle\norm{(\frac{1}{2}{\bf I}+ \mc{D}+\bm{V})^{-1}[\bv]}_{C^{0,\gamma}} \\
&\le c(\epsilon,\kappa_{*,\gamma^+},c_\Gamma)\norm{\bv}_{C^{0,\gamma}}\,, \qquad 0<\gamma<\gamma^+<1\,.
\end{align*}
In addition, for two curves $\X^{(a)}$, $\X^{(b)}$ satisfying Lemma \ref{lem:XaXb_C2beta}, we obtain the estimate 
\begin{align*}
\norm{\mc{M}^{(a)}[\bm{v}]-\mc{M}^{(b)}[\bm{v}]}_{C^{1,\gamma}} 
&\le c(\epsilon,\kappa_{*,\gamma^+}^{(a)},\kappa_{*,\gamma^+}^{(b)},c_\Gamma)\norm{\X^{(a)}-\X^{(b)}}_{C^{2,\gamma^+}}\norm{\bv}_{C^{0,\gamma}}\,.
\end{align*}

Using Lemma \ref{lem:hyper}, we may write $\bw(s,\theta)$ from \eqref{eq:eq_for_w} as 
\begin{align*}
\bw(s,\theta) &= \mc{T}[2\bv(s)]-\mc{T}[\mc{M}[\bv(s)]] \\
&= 2\mc{T}_\epsilon[\bv(s)]+ 2\mc{T}_+[\bv(s)] -\mc{T}_\epsilon[\mc{M}[\bv(s)]] -\mc{T}_+[\mc{M}[\bv(s)]]\,.
\end{align*}
For $\bv\in C^{1,\alpha}(\T)$, we have  
\begin{align*}
2\norm{\mc{T}_\epsilon[\bv]}_{C^{0,\alpha}}&\le c(\kappa_*,c_\Gamma)\norm{\bv}_{C^{1,\alpha}}\,, \\
2\norm{\mc{T}_\epsilon^{(a)}[\bv]-\mc{T}_\epsilon^{(b)}[\bv]}_{C^{0,\alpha}}&\le c(\kappa_*^{(a)},\kappa_*^{(b)},c_\Gamma)\norm{\X^{(a)}-\X^{(b)}}_{C^2}\norm{\bv}_{C^{1,\alpha}}\,,
\end{align*}
while each of the remaining terms are smoother. In particular, for $0<\alpha<\gamma<\gamma^+<1$, we may estimate 
\begin{align*}
\norm{\mc{T}_+[\bv]}_{C^{0,\gamma}}&\le c(\kappa_{*,\gamma^+},c_\Gamma)\,\epsilon^{-1-\gamma^+}\norm{\bv}_{C^1} \\
\norm{\mc{T}_\epsilon[\mc{M}[\bv]]}_{C^{0,\gamma}}&\le c(\epsilon,\kappa_*,c_\Gamma)\norm{\mc{M}[\bv]}_{C^{1,\gamma}}
\le c(\epsilon,\kappa_{*,\gamma^+},c_\Gamma)\norm{\bv}_{C^{0,\gamma}} \\
\norm{\mc{T}_+[\mc{M}[\bv]]}_{C^{0,\gamma}}&\le c(\kappa_{*,\gamma^+},c_\Gamma)\epsilon^{-1-\gamma^+}\norm{\mc{M}[\bv]}_{C^1}
\le c(\epsilon,\kappa_{*,\gamma^+},c_\Gamma)\norm{\bv}_{C^{0,\gamma}}\,.
\end{align*}
In addition, we obtain the Lipschitz estimates 
\begin{align*}
&\norm{\mc{T}_+^{(a)}[\bv]-\mc{T}_+^{(b)}[\bv]}_{C^{0,\gamma}}\le c(\kappa_{*,\gamma^+}^{(a)},\kappa_{*,\gamma^+}^{(b)},c_\Gamma)\,\epsilon^{-1-\gamma^+}\norm{\X^{(a)}-\X^{(b)}}_{C^{2,\gamma^+}}\norm{\bv}_{C^1} \\
&\norm{\mc{T}_\epsilon^{(a)}[\mc{M}^{(a)}[\bv]]-\mc{T}_\epsilon^{(b)}[\mc{M}^{(b)}[\bv]]}_{C^{0,\gamma}} \\
&\qquad\le \norm{\mc{T}_\epsilon^{(a)}\big[\mc{M}^{(a)}[\bv]-\mc{M}^{(b)}[\bv]\big]}_{C^{0,\gamma}} + \norm{\mc{T}_\epsilon^{(a)}[\mc{M}^{(b)}[\bv]]-\mc{T}_\epsilon^{(b)}[\mc{M}^{(b)}[\bv]]}_{C^{0,\gamma}} \\
&\qquad \le c(\epsilon,\kappa_*^{(a)},\kappa_*^{(b)},c_\Gamma)\bigg(\norm{\mc{M}^{(a)}[\bv]-\mc{M}^{(b)}[\bv]}_{C^{1,\gamma}} +\norm{\X^{(a)}-\X^{(b)}}_{C^2}\norm{\mc{M}^{(b)}[\bv]}_{C^{1,\gamma}}\bigg) \\
&\qquad \le c(\epsilon,\kappa_{*,\gamma^+}^{(a)},\kappa_{*,\gamma^+}^{(b)},c_\Gamma)\,\norm{\X^{(a)}-\X^{(b)}}_{C^{2,\gamma^+}}\norm{\bv}_{C^{0,\gamma}} \\
&\norm{\mc{T}_+^{(a)}[\mc{M}^{(a)}[\bv]]-\mc{T}_+^{(b)}[\mc{M}^{(b)}[\bv]]}_{C^{0,\gamma}}\\
&\qquad \le \norm{\mc{T}_+^{(a)}\big[\mc{M}^{(a)}[\bv]-\mc{M}^{(b)}[\bv]\big]}_{C^{0,\gamma}} + \norm{\mc{T}_+^{(a)}[\mc{M}^{(b)}[\bv]]-\mc{T}_+^{(b)}[\mc{M}^{(b)}[\bv]]}_{C^{0,\gamma}} \\
&\qquad \le c(\kappa_{*,\gamma^+}^{(a)},\kappa_{*,\gamma^+}^{(b)},c_\Gamma)\,\epsilon^{-1-\gamma^+}\bigg(\norm{\mc{M}^{(a)}[\bv]-\mc{M}^{(b)}[\bv]}_{C^1} + \norm{\X^{(a)}-\X^{(b)}}_{C^{2,\gamma^+}}\norm{\mc{M}^{(b)}[\bv]}_{C^1}\bigg) \\
&\qquad \le c(\epsilon,\kappa_{*,\gamma^+}^{(a)},\kappa_{*,\gamma^+}^{(b)},c_\Gamma)\,\norm{\X^{(a)}-\X^{(b)}}_{C^{2,\gamma^+}}\norm{\bv}_{C^{0,\gamma}}\,.
\end{align*}
Combining the above bounds, we obtain Lemma \ref{lem:full_neumann}.
\end{proof}

In the subsequent sections we prove each of the building block lemmas used in the proof of Lemma \ref{lem:full_neumann}. Section \ref{subsec:comp_doub_layer} contains the proof of Lemma \ref{lem:comp_double_layer}, section \ref{subsec:inv_doub_layer} contains the proof of Lemma \ref{lem:inv_doub_layer}, and section \ref{subsec:hyper} contains the proof of Lemma \ref{lem:hyper}.


\subsection{Proof of Lemma \ref{lem:comp_double_layer}: completed double layer}\label{subsec:comp_doub_layer}
We begin by considering the double layer operator $\mc{D}$ as 
\begin{align*}
\mc{D}[\bm{\varphi}]= \overline{\mc{D}}[\Phi^{-1}\bm{\varphi}]+\mc{R}_{\mc{D}}[\bm{\varphi}] 
\end{align*}
where $\overline{\mc{D}}$ and $\mc{R}_{\mc{D}}$ are as in \eqref{eq:decomp_step1}. By Lemma \ref{lem:double_layer}, we have that $\mc{R}_{\mc{D}}$ satisfies
\begin{equation}\label{eq:RD_satisfies}
\begin{aligned}
\norm{\mc{R}_{\mc{D}}[\bm{\varphi}]}_{C^{1,\gamma}} &\le c(\kappa_{*,\gamma^+},c_\Gamma)\,\epsilon^{-\gamma^+}\norm{\bm{\varphi}}_{C^{0,\gamma}} \\
\norm{\mc{R}_{\mc{D}}^{(a)}[\bm{\varphi}]-\mc{R}_{\mc{D}}^{(b)}[\bm{\varphi}]}_{C^{1,\gamma}} &\le c(\kappa_{*,\gamma^+}^{(a)},\kappa_{*,\gamma^+}^{(b)},c_\Gamma)\,\epsilon^{-\gamma^+}\norm{\X^{(a)}-\X^{(b)}}_{C^{2,\gamma^+}}\norm{\bm{\varphi}}_{C^{0,\gamma}}\,.
\end{aligned}
\end{equation}
To estimate $\overline{\mc{D}}$, we use the form \eqref{eq:stresslet} of the kernel $\bm{K}_{\mc D}$ and the form \eqref{eq:RdotN} of $\barR\cdot\overline{\bm{n}}'$ to write 
\begin{align*}
\overline{\mc{D}}[\Phi^{-1}\bm{\varphi}] &= \frac{3}{4\pi}\int_{-1/2}^{1/2}\int_{-\pi}^{\pi}\frac{\barR \, (\barR\cdot\overline{\bm{n}}')\barR\cdot(\Phi^{-1}\bm{\varphi})}{|\barR|^5}\,\epsilon\,d\bartheta d\bars \\
&= -\frac{3}{2\pi}\int_{-1/2}^{1/2}\int_{-\pi}^{\pi}\frac{\epsilon\sin^2(\frac{\bartheta}{2})\,\barR \, \barR\cdot(\Phi^{-1}\bm{\varphi})}{|\barR|^5}\,\epsilon\,d\bartheta d\bars\,.
\end{align*}
Then, by Lemma \ref{lem:basic_est}, we have
\begin{equation}\label{eq:barD_linfty}
\begin{aligned}
\abs{\overline{\mc{D}}[\Phi^{-1}\bm{\varphi}]} &\le c\,\norm{\bm{\varphi}}_{L^\infty}\,,\\
\abs{\overline{\mc{D}}[(\Phi^{(a)})^{-1}\bm{\varphi}-(\Phi^{(b)})^{-1}\bm{\varphi}]} &\le c\,\norm{\X^{(a)}-\X^{(b)}}_{C^1}\norm{\bm{\varphi}}_{L^\infty}\,.
\end{aligned}
\end{equation}
Next, taking $\theta$-derivatives along $\Gamma_\epsilon$ and using the expansions \eqref{eq:barR}, \eqref{eq:RdotN}, \eqref{eq:Rdots}, \eqref{eq:pthetaR}, we may write
\begin{align*} 
\textstyle \frac{1}{\epsilon}\p_\theta\overline{\mc{D}}[\Phi^{-1}\bm{\varphi}] 
&= \frac{3}{4\pi}\,{\rm p.v.}\int_{-1/2}^{1/2}\int_{-\pi}^{\pi}\bigg( 
\frac{\barR\,(\barR\cdot\overline{\bm{n}}')(\be_\theta\cdot\bm{\varphi})}{|\barR|^5}
+ \frac{\barR\,(\barR\cdot(\Phi^{-1}\bm{\varphi})) (\frac{1}{\epsilon}\p_\theta\barR\cdot\overline{\bm{n}}')}{|\barR|^5} \\
&\qquad\qquad + \frac{5\barR\,(\barR\cdot(\Phi^{-1}\bm{\varphi})) (\barR\cdot\overline{\bm{n}}')(\frac{1}{\epsilon}\p_\theta\barR\cdot\barR)}{|\barR|^7}\bigg)\,\epsilon\, d\bartheta d\bars  \\
&=\frac{3}{4\pi}\,{\rm p.v.}\int_{-1/2}^{1/2}\int_{-\pi}^{\pi}\bigg( 
\frac{-2\epsilon\sin^2(\frac{\bartheta}{2})(\bars\be_z+2\epsilon\sin(\frac{\bartheta}{2})\overline\be_\theta)(\be_\theta\cdot\bm{\varphi})}{|\barR|^5} \\
&\qquad\qquad - \frac{\sin(\bartheta)(\bars\be_z+2\epsilon\sin(\frac{\bartheta}{2})\overline\be_\theta)(\bars\be_{\rm t}\cdot\bm{\varphi} + 2\epsilon\sin(\frac{\bartheta}{2})\be_\theta\cdot\bm{\varphi})}{|\barR|^5} \\
&\quad - \frac{10\epsilon^2\sin^2(\frac{\bartheta}{2})\sin(\bartheta)(\bars\be_z+2\epsilon\sin(\frac{\bartheta}{2})\overline\be_\theta)\,(\bars\be_{\rm t}\cdot\bm{\varphi} + 2\epsilon\sin(\frac{\bartheta}{2})\be_\theta\cdot\bm{\varphi})}{|\barR|^7}\bigg)\,\epsilon\, d\bartheta d\bars\,.
\end{align*}
Using Lemma \ref{lem:alpha_est}, case (2), we then obtain
\begin{equation}\label{eq:barD_theta}
\begin{aligned}
\textstyle \abs{\frac{1}{\epsilon}\p_\theta\overline{\mc{D}}[\Phi^{-1}\bm{\varphi}]}_{\dot C^{0,\gamma}} &\le c(\kappa_{*,\gamma},c_\Gamma)\,\epsilon^{-1-\gamma}\norm{\bm{\varphi}}_{C^{0,\gamma}} \\
\textstyle \abs{\frac{1}{\epsilon}\p_\theta\overline{\mc{D}}[(\Phi^{(a)})^{-1}\bm{\varphi}-(\Phi^{(b)})^{-1}\bm{\varphi}]}_{\dot C^{0,\gamma}} &\le c(\kappa_{*,\gamma}^{(a)},\kappa_{*,\gamma}^{(b)},c_\Gamma)\,\epsilon^{-1-\gamma}\norm{\X^{(a)}-\X^{(b)}}_{C^{2,\gamma}}\norm{\bm{\varphi}}_{C^{0,\gamma}}\,.
\end{aligned}
\end{equation}
Here the additional $\epsilon^{-\gamma}$ arises due to the dependence on the curved frame vector $\be_\theta(s,\theta-\frac{\bartheta}{2})$ in the above expansion.
Taking $s$-derivatives along $\Gamma_\epsilon$, we next have
\begin{align*}
\p_s\overline{\mc{D}}[\Phi^{-1}\bm{\varphi}] 
&= \frac{3}{4\pi}\,{\rm p.v.}\int_{-1/2}^{1/2}\int_{-\pi}^{\pi}\bigg( 
\frac{\barR\,(\barR\cdot\overline{\bm{n}}')(\be_{\rm t}\cdot\bm{\varphi})}{|\barR|^5}
+ \frac{\barR\,(\barR\cdot(\Phi^{-1}\bm{\varphi})) (\p_s\barR\cdot\overline{\bm{n}}')}{|\barR|^5} \\
&\qquad\qquad + \frac{5\barR\,(\barR\cdot(\Phi^{-1}\bm{\varphi})) (\barR\cdot\overline{\bm{n}}')(\p_s\barR\cdot\barR)}{|\barR|^7}\bigg)\,\epsilon\, d\bartheta d\bars  \\
&= \frac{3}{4\pi}\,{\rm p.v.}\int_{-1/2}^{1/2}\int_{-\pi}^{\pi}\bigg( 
\frac{-2\epsilon\sin^2(\frac{\bartheta}{2})(\bars\be_z+2\epsilon\sin(\frac{\bartheta}{2})\overline\be_\theta)(\be_{\rm t}\cdot\bm{\varphi})}{|\barR|^5}\\
&\hspace{-1cm} -\frac{10\epsilon\sin^2(\frac{\bartheta}{2})(\bars+2\epsilon^2\sin(\frac{\bartheta}{2})\sin(\bartheta))(\bars\be_z+2\epsilon\sin(\frac{\bartheta}{2})\overline\be_\theta)(\bars\be_{\rm t}\cdot\bm{\varphi} + 2\epsilon\sin(\frac{\bartheta}{2})\be_\theta\cdot\bm{\varphi})}{|\barR|^7}\bigg)\,\epsilon\, d\bartheta d\bars \,.
\end{align*}
Again using Lemma \ref{lem:alpha_est}, case (2), we have
\begin{equation}\label{eq:barD_s}
\begin{aligned}
\abs{\p_s\overline{\mc{D}}[\Phi^{-1}\bm{\varphi}]}_{\dot C^{0,\gamma}} &\le 
c(\kappa_{*,\gamma},c_\Gamma)\,\epsilon^{-1-\gamma}\norm{\bm{\varphi}}_{C^{0,\gamma}} \\
\abs{\p_s\overline{\mc{D}}[(\Phi^{(a)})^{-1}\bm{\varphi}-(\Phi^{(b)})^{-1}\bm{\varphi}]}_{\dot C^{0,\gamma}} &\le c(\kappa_{*,\gamma}^{(a)},\kappa_{*,\gamma}^{(b)},c_\Gamma)\,\epsilon^{-1-\gamma}\norm{\X^{(a)}-\X^{(b)}}_{C^{2,\gamma}}\norm{\bm{\varphi}}_{C^{0,\gamma}}\,.
\end{aligned}
\end{equation}
Here again we lose a factor of $\epsilon^{-\gamma}$ from the single $\be_\theta(s,\theta-\frac{\bartheta}{2})$ term above.
Combining \eqref{eq:RD_satisfies}, \eqref{eq:barD_linfty}, \eqref{eq:barD_theta}, and \eqref{eq:barD_s}, we obtain
\begin{equation}\label{eq:D_smooths}
\begin{aligned}
\norm{\mc{D}[\bm{\varphi}]}_{C^{1,\gamma}} &\le 
c(\kappa_{*,\gamma^+},c_\Gamma)\,\epsilon^{-1-\gamma}\norm{\bm{\varphi}}_{C^{0,\gamma}} \\
\norm{\mc{D}^{(a)}[\bm{\varphi}]-\mc{D}^{(b)}[\bm{\varphi}]}_{C^{1,\gamma}} &\le c(\kappa_{*,\gamma^+}^{(a)},\kappa_{*,\gamma^+}^{(b)},c_\Gamma)\,\epsilon^{-1-\gamma}\norm{\X^{(a)}-\X^{(b)}}_{C^{2,\gamma^+}}\norm{\bm{\varphi}}_{C^{0,\gamma}}\,.
\end{aligned}
\end{equation}

It remains to estimate the completion flow $\bm{V}[\bm{\varphi}]$, given by \eqref{eq:completion}. For $\bx\in\Gamma_\epsilon$, it will be convenient to rewrite $\bR_X=\bx(s,\theta)-\X(s')$ in terms of $\bars=s-s'$ and $s,\theta$. Using the expansions \eqref{eq:s_expand}, we may write $\bR_X$ as 
\begin{align*}
\bR_X(s,\bars,\theta) = \bars\be_{\rm t}(s) - \bars^2\bm{Q}_{\rm t}(s,\bars) + \epsilon\be_r(s,\theta)\,. 
\end{align*}
For two filaments $\X^{(a)}$, $\X^{(b)}$ satisfying Lemma \ref{lem:XaXb_C2beta}, we note the expansions 
\begin{align*}
\bR_X^{(a)} - \bR_X^{(b)} &= \bars\big(\be_{\rm t}^{(a)}-\be_{\rm t}^{(b)} - \bars(\bm{Q}_{\rm t}^{(a)}-\bm{Q}_{\rm t}^{(b)})\big) + \epsilon(\be_r^{(a)}-\be_r^{(b)}) \\
\frac{1}{|\bR_X^{(a)}|^k}- \frac{1}{|\bR_X^{(b)}|^k} &= \bigg(-2\bars(\bm{Q}_{\rm t}^{(b)}\cdot\be_{\rm t}^{(b)}-\bm{Q}_{\rm t}^{(a)}\cdot\be_{\rm t}^{(a)})
-2\epsilon(\bm{Q}_{\rm t}^{(b)}\cdot\be_r^{(b)}-\bm{Q}_{\rm t}^{(a)}\cdot\be_r^{(a)}) \\
& +\bars^2(|\bm{Q}_{\rm t}^{(b)}|^2-|\bm{Q}_{\rm t}^{(a)}|^2)\bigg)
\frac{\bars^2}{|\bR_X^{(a)}||\bR_X^{(b)}|(|\bR_X^{(a)}|+|\bR_X^{(b)}|)} \sum_{\ell=0}^{k-1}\frac{1}{|\bR_X^{(a)}|^\ell|\bR_X^{(b)}|^{k-1-\ell}}  \,.
\end{align*}
From \cite[Lemma 3.1]{closed_loop}, we have 
\begin{equation}
\abs{\bR_X}\ge c(\kappa_*,c_\Gamma)\sqrt{\bars^2+\epsilon^2}\,;
\end{equation}
in particular, $\frac{\abs{\bars}}{|\bR_X|}\le c(\kappa_*,c_\Gamma)$.
Then, using the forms \eqref{eq:stokeslet}, \eqref{eq:rotlet} of $\mc{G}$ and $\bm{l}_\mc{R}$ as well as the fact that $\abs{\bR_X}\ge\epsilon$ while $\abs{\bx'-\X(s')}=\epsilon$, we have
\begin{align*}
\abs{\bm{V}[\bm{\varphi}]}&\le c(\kappa_*)\norm{\bm{\varphi}}_{L^\infty}\int_{\Gamma_\epsilon}\frac{1}{\epsilon}\,dS_{x'}
\le c(\kappa_*)\norm{\bm{\varphi}}_{L^\infty}\,,\\
\abs{\bm{V}^{(a)}[\bm{\varphi}]-\bm{V}^{(b)}[\bm{\varphi}]}
&\le c(\kappa_*^{(a)},\kappa_*^{(b)},c_\Gamma)\norm{\X^{(a)}-\X^{(b)}}_{C^2}\norm{\bm{\varphi}}_{L^\infty}\,.
\end{align*}
Here we recall that $dS_{x'}=\mc{J}_\epsilon(s',\theta')\,d\theta'ds'$.

We next take derivatives of $\bm{V}[\bm{\varphi}]$ along $\Gamma_\epsilon$. Here we use $\p_j$ to denote $\p_s$ or $\frac{1}{\epsilon}\p_\theta$. We have
\begin{align*}
\p_j\bm{V}[\bm{\varphi}] &= \int_{\Gamma_\epsilon}\p_j\mc{G}(\bx,\X(s'))\,\bm{\varphi}(\bx')\,dS_{x'} + \int_{\Gamma_\epsilon}\p_j\bm{l}_{\mc R}(\bx,\X(s')) \times \big((\bx'-\X(s'))\times\bm{\varphi}(\bx')\big)\,dS_{x'} \,.
\end{align*}
From the forms \eqref{eq:stokeslet}, \eqref{eq:rotlet} of $\mc{G}$ and $\bm{l}_\mc{R}$, we may calculate
\begin{align*}
\p_j\mc{G}(\bx,\X(s')) &= -\frac{(\p_j\bR_X\cdot\bR_X){\bf I}}{|\bR_X|^3} + \frac{\p_j\bR_X\otimes\bR_X+\bR_X\otimes\p_j\bR_X}{|\bR_X|^3} - 3\frac{(\p_j\bR_X\cdot\bR_X)\bR_X\otimes\bR_X}{|\bR_X|^5} \,;\\
\p_j{\bm l}_\mc{R}(\bx,\X(s')) &= -\frac{\p_j\bR_X}{|\bR_X|^3} + 3\frac{(\p_j\bR_X\cdot\bR_X)\bR_X}{|\bR_X|^5}\\
\p_s\bR_X &= (1-\epsilon\wh\kappa(s,\theta))\be_{\rm t}(s) + \epsilon\kappa_3\be_\theta(s,\theta)\,, \qquad
\textstyle \frac{1}{\epsilon}\p_\theta\bR_X = \be_\theta(s,\theta)\,.
\end{align*}
Using the above expressions as well as the form \eqref{eq:dot_Calpha_eps} of the $\dot C^{0,\gamma}$ seminorm, we obtain the bounds 
\begin{align*}
&\abs{\p_j\bm{V}[\bm{\varphi}]}_{\dot C^{0,\gamma}} \le c(\kappa_{*,\gamma})\,\norm{\bm{\varphi}}_{L^\infty}\bigg( \int_{\Gamma_\epsilon}\frac{1}{\epsilon^{2+\gamma}}\,dS_{x'} + \int_{\Gamma_\epsilon}\frac{1}{\epsilon^{3+\gamma}}\,\epsilon\,dS_{x'}\bigg) 
\le c(\kappa_{*,\gamma})\,\epsilon^{-1-\gamma}\norm{\bm{\varphi}}_{L^\infty}\,,\\
&\abs{\p_j\bm{V}^{(a)}[\bm{\varphi}]-\p_j\bm{V}^{(b)}[\bm{\varphi}]}_{\dot C^{0,\gamma}} \le c(\kappa_{*,\gamma}^{(a)},\kappa_{*,\gamma}^{(b)},c_\Gamma)\,\epsilon^{-1-\gamma}\norm{\X^{(a)}-\X^{(b)}}_{C^{2,\gamma}}\norm{\bm{\varphi}}_{L^\infty}\,.
\end{align*}
Combining the estimates for $\bm{V}$ with the bounds for $\mc{D}$, we obtain Lemma \ref{lem:comp_double_layer}.
\hfill \qedsymbol

\subsection{Proof of Lemma \ref{lem:inv_doub_layer}: inverse double layer}\label{subsec:inv_doub_layer}
The proof of the bound \eqref{eq:inv_doub} relies on a compactness argument and is nearly identical to the Laplace setting \cite[Section 4.2]{laplace}.

Given fixed constants $\kappa_{*,\gamma^+}$ and $c_\Gamma$, we begin by fixing a nonzero function $\bm{\varphi}=\bm{\varphi}(s,\theta)$ in $C^{0,\gamma}(\Gamma_\epsilon)$, where the space $C^{0,\gamma}(\Gamma_\epsilon)$ is understood with respect to the norm \eqref{eq:dot_Calpha_eps}; in particular, without reference to the filament centerline shape. We will show that 
\begin{equation}\label{eq:doub_inv_bd}
\textstyle \norm{(\frac{1}{2}{\bf I}+\mc{D}+\bm{V})^{-1}[\bm{\varphi}]}_{C^{0,\gamma}} \le c_\varphi(\kappa_{*,\gamma^+},c_\Gamma)\norm{\bm{\varphi}}_{C^{0,\gamma}}
\end{equation}
for any filament $\Sigma_\epsilon$ with centerline curve $\X(s)$ satisfying $\norm{\X_{ss}}_{C^{0,\gamma^+}}\le \kappa_{*,\gamma^+}$ and the non-self-intersection condition \eqref{eq:cGamma} with the fixed choice of $c_\Gamma$. Since $\bm{\varphi}$ is arbitrary, the uniform boundedness principle then implies that the bound \eqref{eq:doub_inv_bd} holds independent of $\bm{\varphi}$.

To show \eqref{eq:doub_inv_bd}, we fix $\norm{\bm{\varphi}}_{C^{0,\gamma}}=1$. Proceeding by contradiction, if the bound \eqref{eq:doub_inv_bd} does not hold, then there exists a sequence of curves $\X_j$, $j\in\N$, satisfying \eqref{eq:cGamma} and $\norm{(\X_j)_{ss}}_{C^{0,\gamma^+}}\le \kappa_{*,\gamma^+}$ such that
\begin{align*}
\textstyle \norm{(\frac{1}{2}{\bf I}+\mc{D}_j+\bm{V}_j)^{-1}[\bm{\varphi}]}_{C^{0,\gamma}} > j\,.
\end{align*}
Here $\mc{D}_j+\bm{V}_j$ denotes the double layer operator and completion flow corresponding to curve $j$. We define
\begin{align*}
\bm{h}_j =\frac{\bm{\varphi}}{\norm{(\frac{1}{2}{\bf I}+\mc{D}_j+\bm{V}_j)^{-1}[\bm{\varphi}]}_{C^{0,\gamma}}}\,, \qquad
\bm{g}_j =\frac{(\frac{1}{2}{\bf I}+\mc{D}_j+\bm{V}_j)^{-1}[\bm{\varphi}]}{\norm{(\frac{1}{2}{\bf I}+\mc{D}_j+\bm{V}_j)^{-1}[\bm{\varphi}]}_{C^{0,\gamma}}}\,.
\end{align*}
Since $\norm{\bm{g}_j}_{C^{0,\gamma}}=1$, we then have
\begin{equation}\label{eq:to_zero}
\textstyle (\frac{1}{2}{\bf I}+\mc{D}_j+\bm{V}_j)[\bm{g}_j] = \bm{h}_j \to 0 \quad \text{in } C^{0,\gamma}\,.
\end{equation}
Using Lemma \ref{lem:comp_double_layer}, we have that $(\mc{D}_j+\bm{V}_j)[\bm{g}_j]$ is smoother, in fact, bounded in $C^{1,\gamma}$ uniformly in $j$. Thus \eqref{eq:to_zero} implies that along a subsequence $j_\ell$, 
\begin{equation}\label{eq:ginfty}
\bm{g}_{j_\ell}-2\bm{h}_{j_\ell} \to \bm{g}_\infty \qquad \text{strongly in }C^{0,\gamma}
\end{equation}
for some limit $\bm{g}_\infty$ with $\norm{\bm{g}_\infty}_{C^{0,\gamma}}=1$.

Since $\norm{(\X_j)_{ss}}_{C^{0,\gamma^+}}\le \kappa_{*,\gamma^+}$, there exists a subsequence of curves $\X_{j_k}$ converging (after possible translation) strongly in $C^{2,\gamma}$, $0<\gamma<\gamma^+$, to a limit curve $\X_\infty(s)$ satisfying \eqref{eq:cGamma}. Using the form \eqref{eq:S_and_D} of $\mc{D}$, for any $\bm{\psi}\in C^{0,\gamma}(\Gamma_\epsilon)$, we have
\begin{align*}
&(\mc{D}_{j_k}- \mc{D}_\infty)[\bm{\psi}] \\
&= \frac{3}{4\pi}\int_{-1/2}^{1/2}\int_{-\pi}^{\pi} \frac{\big(\wh\X_{j_k}+\epsilon\wh\be_{r,j_k}\big)\big(\wh\X_{j_k}+\epsilon\wh\be_{r,j_k}\big)\cdot\be_{r,j_k}\big(\wh\X_{j_k}+\epsilon\wh\be_{r,j_k}\big)\cdot\bm{\psi}(s',\theta')}{|\wh\X_{j_k}+\epsilon\wh\be_{r,j_k}|^5} \,\mc{J}_{\epsilon,j_k}\,d\theta'ds'\\
&-\frac{3}{4\pi}\int_{-1/2}^{1/2}\int_{-\pi}^{\pi} \frac{\big(\wh\X_\infty+\epsilon\wh\be_{r,\infty}\big)\big(\wh\X_\infty+\epsilon\wh\be_{r,\infty}\big)\cdot\be_{r,\infty}\big(\wh\X_\infty+\epsilon\wh\be_{r,\infty}\big)\cdot\bm{\psi}(s',\theta')}{|\wh\X_\infty+\epsilon\wh\be_{r,\infty}|^5} \,\mc{J}_{\epsilon,\infty}\,d\theta'ds'\,,
\end{align*}
where we denote $\wh\X_{j_k}(s,s')=\X_{j_k}(s)-\X_{j_k}(s')$, $\wh\be_{r,j_k}(s,\theta,s',\theta')=\be_{r,j_k}(s,\theta)-\be_{r,j_k}(s',\theta')$, and each of $\wh\X_j$, $\be_{r,j}$, and $\mc{J}_{\epsilon,j}$ are defined along the curve $\X_j$. From the form of these expressions, we can see that $|(\mc{D}_{j_k}- \mc{D}_\infty)[\bm{\psi}]|\to 0$ as $j_k\to\infty$, since $\X_{j_k}\to \X_\infty$.
Similarly, using the form \eqref{eq:completion} of the completion flow $\bm{V}$, we have
\begin{align*}
(\bm{V}_{j_k} - \bm{V}_\infty)[\bm{\psi}] 
&=\int_{-1/2}^{1/2}\int_{-\pi}^{\pi}\bigg(\frac{\mc{J}_{\epsilon,j_k}(s',\theta')}{|\wh\X_{j_k}+\epsilon\be_{r,j_k}(s,\theta)|}
-\frac{\mc{J}_{\epsilon,\infty}(s',\theta')}{|\wh\X_\infty+\epsilon\be_{r,\infty}(s,\theta)|}\bigg)\,\bm{\psi}(s',\theta') \,d\theta'ds'\\
&\quad+ \int_{-1/2}^{1/2}\int_{-\pi}^{\pi}\bigg(\frac{(\wh\X_{j_k}+\epsilon\be_{r,j_k})\otimes(\wh\X_{j_k}+\epsilon\be_{r,j_k})\mc{J}_{\epsilon,j_k}(s',\theta')}{|\wh\X_{j_k}+\epsilon\be_{r,j_k}(s,\theta)|^3} \\
&\quad\qquad -\frac{(\wh\X_\infty+\epsilon\be_{r,\infty})\otimes(\wh\X_\infty+\epsilon\be_{r,\infty})\mc{J}_{\epsilon,\infty}(s',\theta')}{|\wh\X_\infty+\epsilon\be_{r,\infty}(s,\theta)|^3} \,\bigg)\bm{\psi}(s',\theta')\,d\theta'ds'\\
&- \int_{-1/2}^{1/2}\int_{-\pi}^{\pi}\frac{(\wh\X_{j_k}+\epsilon\be_{r,j_k})}{|\wh\X_{j_k}+\epsilon\be_{r,j_k}(s,\theta)|^3} \times \big(\be_{r,j_k}(s',\theta')\times\bm{\psi}(s',\theta')\big)\,\mc{J}_{\epsilon,j_k}(s',\theta')\,d\theta'ds'\\
&+ \int_{-1/2}^{1/2}\int_{-\pi}^{\pi}\frac{(\wh\X_\infty+\epsilon\be_{r,\infty})}{|\wh\X_\infty+\epsilon\be_{r,\infty}(s,\theta)|^3} \times \big(\be_{r,\infty}(s',\theta')\times\bm{\psi}(s',\theta')\big)\,\mc{J}_{\epsilon,\infty}(s',\theta')\,d\theta'ds'\,,
\end{align*}
and again we see that $(\bm{V}_{j_k} - \bm{V}_\infty)[\bm{\psi}] \to 0$ as $j_k\to\infty$. Together, we have $\abs{(\mc{D}_{j_k}+\bm{V}_{j_k})[\bm{\psi}] - (\mc{D}_\infty+\bm{V}_\infty)[\bm{\psi}]}\to 0$. Therefore, by a diagonalization argument, we have 
\begin{align*}
(\mc{D}_j+\bm{V}_j)[\bm{g}_j] \to (\mc{D}_\infty+\bm{V}_\infty)[\bm{g}_\infty]
\end{align*}
along a subsequence, and, using \eqref{eq:ginfty}, $(\frac{1}{2}{\bf I}+\mc{D}_\infty+\bm{V}_\infty)[\bm{g}_\infty]=0$. By injectivity of the completed double layer operator (see discussion below \eqref{eq:completed_double}), we then have $\bm{g}_\infty=0$, contradicting $\norm{\bm{g}_\infty}_{C^{0,\gamma}}=1$. We thus obtain \eqref{eq:inv_doub}. \\

We now turn to the Lipschitz bound \eqref{eq:inv_doub_lip}. Consider two nearby filaments $\X^{(a)}$ and $\X^{(b)}$ satisfying Lemma \ref{lem:XaXb_C2beta}.
We aim to bound the difference $\bv^{(a)}-\bv^{(b)}$ where $\bv^{(a)},\bv^{(b)}$ satisfy
\begin{align*}
\textstyle (\frac{1}{2}{\bf I}+\mc{D}^{(a)}+\bm{V}^{(a)})[\bv^{(a)}] = \bm{\varphi} = (\frac{1}{2}{\bf I}+\mc{D}^{(b)}+\bm{V}^{(b)})[\bv^{(b)}]\,.
\end{align*} 
Rewriting, we have 
\begin{align*}
\bv^{(a)}-\bv^{(b)}&= 2(\mc{D}^{(b)}+\bm{V}^{(b)})[\bv^{(b)}] - 2(\mc{D}^{(a)}+\bm{V}^{(a)})[\bv^{(a)}] \\
&= 2(\mc{D}^{(b)}+\bm{V}^{(b)}-\mc{D}^{(a)}-\bm{V}^{(a)})[\bv^{(b)}]
+ 2(\mc{D}^{(a)}+\bm{V}^{(a)})[\bv^{(b)}-\bv^{(a)}]\,,
\end{align*}
or, moving all $\bv^{(a)}-\bv^{(b)}$ terms to the left and side, 
\begin{align*}
\textstyle \bv^{(a)}-\bv^{(b)} = \big(\frac{1}{2}{\bf I}+ \mc{D}^{(a)}+\bm{V}^{(a)}\big)^{-1}\big(\mc{D}^{(b)}+\bm{V}^{(b)}-\mc{D}^{(a)}-\bm{V}^{(a)}\big)[\bv^{(b)}]\,.
\end{align*}
Then, using the single filament bound \eqref{eq:inv_doub} and Lemma \ref{lem:comp_double_layer}, we have
\begin{align*}
 \norm{\bv^{(a)}-\bv^{(b)}}_{C^{0,\gamma}} &= \textstyle\norm{\big(\frac{1}{2}{\bf I}+ \mc{D}^{(a)}+\bm{V}^{(a)}\big)^{-1}\big(\mc{D}^{(b)}+\bm{V}^{(b)}-\mc{D}^{(a)}-\bm{V}^{(a)}\big)[\bv^{(b)}]}_{C^{0,\gamma}}\\
 &\le c(\epsilon,\kappa_{*,\gamma^+}^{(a)},c_\Gamma)\norm{\big(\mc{D}^{(b)}+\bm{V}^{(b)}-\mc{D}^{(a)}-\bm{V}^{(a)}\big)[\bv^{(b)}]}_{C^{0,\gamma}} \\
 &\le c(\epsilon,\kappa_{*,\gamma^+}^{(a)},\kappa_{*,\gamma^+}^{(b)},c_\Gamma)\norm{\X^{(a)}- \X^{(b)}}_{C^{2,\gamma^+}}\norm{\bv^{(b)}}_{C^{0,\gamma}} \\
 &\le c(\epsilon,\kappa_{*,\gamma^+}^{(a)},\kappa_{*,\gamma^+}^{(b)},c_\Gamma)\norm{\X^{(a)}- \X^{(b)}}_{C^{2,\gamma^+}}\norm{\bm{\varphi}}_{C^{0,\gamma}}\,.
\end{align*}
\hfill \qedsymbol

\subsection{Proof of Lemma \ref{lem:hyper}: hypersingular operator bounds}\label{subsec:hyper}
We use the expansions \eqref{eq:curvedR} and \eqref{eq:RdotN} to write the kernel $\bm{K}_{\mc T}$ \eqref{eq:KT} in terms of $s,\theta,\bars,\bartheta$ as
\begin{align*}
\bm{K}_{\mc T} &= \frac{1}{4\pi}\big( \bm{K}_{\mc T,0} + \bm{K}_{\mc T,1} + \bm{K}_{\mc T,2}\big)\,,\\
\bm{K}_{\mc T,0} &= 2\frac{\be_r\otimes\be_r}{\abs{\bR}^3} + 3\frac{(\bars\be_{\rm t}+2\epsilon\sin(\frac{\bartheta}{2})\be_\theta+\epsilon\bars\bm{Q}_r)\otimes(\bars\be_{\rm t}+2\epsilon\sin(\frac{\bartheta}{2})\be_\theta+\epsilon\bars\bm{Q}_r)}{\abs{\bR}^5} \\
\bm{K}_{\mc T,1} &= 
- 3\frac{\bars^2\big[\bm{Q}_{\rm t}\otimes(\bars\be_{\rm t}+2\epsilon\sin(\frac{\bartheta}{2})\be_\theta+\epsilon\bars\bm{Q}_r)+ (\bars\be_{\rm t}+2\epsilon\sin(\frac{\bartheta}{2})\be_\theta+\epsilon\bars\bm{Q}_r)\otimes \bm{Q}_{\rm t}\big]}{\abs{\bR}^5} \\
& - 2\frac{\bars\be_r\otimes\bm{Q}_r+2\sin(\frac{\bartheta}{2})\be_r\otimes\be_\theta}{\abs{\bR}^3}
-3\frac{(2\epsilon\sin^2(\frac{\bartheta}{2})-\bars^2Q_{\rm Rn'})(\bars\be_{\rm t}+2\epsilon\sin(\frac{\bartheta}{2})\be_\theta+\epsilon\bars\bm{Q}_r)\otimes\be_r }{\abs{\bR}^5}  \\
&\quad
+ 3\frac{(2\epsilon\sin^2(\frac{\bartheta}{2})+\bars^2Q_{\rm Rn})\be_r\otimes(\bars\be_{\rm t}+2\epsilon\sin(\frac{\bartheta}{2})\be_\theta+\epsilon\bars\bm{Q}_r)}{\abs{\bR}^5}  \\
\bm{K}_{\mc T,2} &= 
6\frac{\bars^4\bm{Q}_{\rm t}\otimes\bm{Q}_{\rm t}}{\abs{\bR}^5}
+3\frac{\bars^2\big[ (2\epsilon\sin^2(\frac{\bartheta}{2})-\bars^2Q_{\rm Rn'})\,\bm{Q}_{\rm t}\otimes\be_r -(2\epsilon\sin^2(\frac{\bartheta}{2})+\bars^2Q_{\rm Rn})\,\be_r\otimes\bm{Q}_{\rm t}\big]}{\abs{\bR}^5}\\
&\quad - 3\frac{(2\epsilon\sin^2(\frac{\bartheta}{2})+\bars^2Q_{\rm Rn})\big[(2\epsilon\sin^2(\frac{\bartheta}{2})-\bars^2Q_{\rm Rn'}){\bf I}+(\bars\bm{Q}_r+2\sin(\frac{\bartheta}{2})\be_\theta)\otimes\bR\big]}{\abs{\bR}^5} \\
&\quad 
- 3\frac{\big[(2\sin^2(\frac{\bartheta}{2})-\bars^2Q_{0,5})\abs{\bR}^2-10(2\epsilon\sin^2(\frac{\bartheta}{2})+\bars^2Q_{\rm Rn})(2\epsilon\sin^2(\frac{\bartheta}{2})-\bars^2Q_{\rm Rn'})\big]\bR\otimes\bR}{\abs{\bR}^7}\,.
\end{align*}
Accordingly, for $j=0,1,2$, we define 
\begin{align*}
\mc{T}_j[\bm{\varphi}] = -{\rm p.v.}\int_{-1/2}^{1/2}\int_{-\pi}^\pi\bm{K}_{\mc{T},j}\,\big(\bm{\varphi}(s-\bars,\theta-\bartheta)-\bm{\varphi}(s,\theta) \big)\,\mc{J}_\epsilon(s-\bars,\theta-\bartheta)\,d\bartheta ds\,.
\end{align*}

We begin with bounds for $\mc{T}_2$. Using Lemma \ref{lem:basic_est}, we have
\begin{align*}
\abs{\mc{T}_2[\bm{\varphi}]} &\le 2\norm{\bm{\varphi}}_{L^\infty}\int_{-1/2}^{1/2}\int_{-\pi}^\pi\abs{\bm{K}_{\mc{T},2}}\,\abs{\mc{J}_\epsilon}\,d\bartheta ds \\
&\le c(\kappa_*)\norm{\bm{\varphi}}_{L^\infty}\int_{-1/2}^{1/2}\int_{-\pi}^\pi \frac{1+\epsilon^{-1}+\epsilon^{-2}}{\abs{\bR}}\,\epsilon\,d\bartheta ds
\le c(\kappa_*,c_\Gamma)\,\epsilon^{-1}\norm{\bm{\varphi}}_{L^\infty}\,.
\end{align*}
Furthermore, using that the remainder terms $Q_i$ above each satisfy an estimate of the form \eqref{eq:Qj_Cbeta}, and using the expansion \eqref{eq:Rdiffk}, for two curves $\X^{(a)}$ and $\X^{(b)}$ satisfying Lemma \ref{lem:XaXb_C2beta}, we have
\begin{align*}
\abs{\mc{T}_2^{(a)}[\bm{\varphi}]-\mc{T}_2^{(b)}[\bm{\varphi}]} &\le 2\norm{\bm{\varphi}}_{L^\infty}\int_{-1/2}^{1/2}\int_{-\pi}^\pi\abs{\bm{K}_{\mc{T},2}^{(a)}\,\mc{J}_\epsilon^{(a)}-\bm{K}_{\mc{T},2}^{(b)}\,\mc{J}_\epsilon^{(b)}}\,d\bartheta ds \\
&\le c(\kappa_*^{(a)},\kappa_*^{(b)},c_\Gamma)\,\epsilon^{-1}\norm{\X^{(a)}-\X^{(b)}}_{C^2}\norm{\bm{\varphi}}_{L^\infty}\,.
\end{align*}
In addition, by Lemma \ref{lem:alpha_est}, we have
\begin{align*}
\abs{\mc{T}_2[\bm{\varphi}]}_{\dot C^{0,\alpha}} &\le c(\kappa_{*,\alpha},c_\Gamma)\,\epsilon^{-1-\alpha}\norm{\bm{\varphi}}_{C^{0,\alpha}}\\
\abs{\mc{T}_2^{(a)}[\bm{\varphi}]-\mc{T}_2^{(b)}[\bm{\varphi}]}_{\dot C^{0,\alpha}} &\le c(\kappa_*^{(a)},\kappa_*^{(b)},c_\Gamma)\,\epsilon^{-1-\alpha}\norm{\X^{(a)}-\X^{(b)}}_{C^{2,\alpha}}\norm{\bm{\varphi}}_{C^{0,\alpha}}\,.
\end{align*}
Here we use case (2) of Lemma \ref{lem:alpha_est} to estimate terms of the form $\frac{\sin^2(\frac{\bartheta}{2})}{\abs{\bR}^3}$, treating one factor of $\sin(\frac{\bartheta}{2})$ as a coefficient, which allows us to save a factor of $\epsilon^{-\alpha}$. We may use case (1) to estimate the remaining terms.

We next estimate $\mc{T}_1[\bm{\varphi}]$. We may first use Lemma \ref{lem:odd_nm} to obtain
\begin{align*}
 \abs{\mc{T}_1[\bm{\varphi}]} &\le \abs{{\rm p.v.}\int_{-1/2}^{1/2}\int_{-\pi}^\pi\bm{K}_{\mc{T},1}\big(\bm{\varphi}(s-\bars,\theta-\bartheta)-\bm{\varphi}(s,\theta) \big)\,\mc{J}_\epsilon\,d\bartheta ds }\\
&\le c(\kappa_*,c_\Gamma)\,\epsilon^{\alpha-1}\big(\max_i\norm{Q_i}_{C^{0,\alpha}}\big)\norm{\bm{\varphi}}_{C^{0,\alpha}}
\le c(\kappa_{*,\alpha},c_\Gamma)\,\epsilon^{-1}\norm{\bm{\varphi}}_{C^{0,\alpha}}\,.
 \end{align*} 
 As usual, here we use $Q_i$ as catch-all notation for the remainder coefficients in the above expansion. Similarly, using the expressions \eqref{eq:Qj_Cbeta} and \eqref{eq:Rdiffk}, for two curves $\X^{(a)}$ and $\X^{(b)}$ satisfying Lemma \ref{lem:XaXb_C2beta}, we have
\begin{align*}
\abs{\mc{T}_1^{(a)}[\bm{\varphi}]-\mc{T}_1^{(b)}[\bm{\varphi}]} &\le \abs{{\rm p.v.}\int_{-1/2}^{1/2}\int_{-\pi}^\pi\big(\bm{K}_{\mc{T},1}^{(a)}\,\mc{J}_\epsilon^{(a)}-\bm{K}_{\mc{T},1}^{(b)}\,\mc{J}_\epsilon^{(b)}\big)\big(\bm{\varphi}(s-\bars,\theta-\bartheta)-\bm{\varphi}(s,\theta) \big)\,d\bartheta ds} \\
&\le c(\kappa_{*,\alpha}^{(a)},\kappa_{*,\alpha}^{(b)},c_\Gamma)\,\epsilon^{-1}\norm{\X^{(a)}-\X^{(b)}}_{C^{2,\alpha}}\norm{\bm{\varphi}}_{C^{0,\alpha}}\,.
 \end{align*}
 We next turn to $\dot C^{0,\alpha}$ estimates. Using Lemma \ref{lem:alpha_est}, case (2), we obtain 
\begin{align*}
 \abs{\mc{T}_1[\bm{\varphi}]}_{\dot C^{0,\alpha}} &\le \abs{{\rm p.v.}\int_{-1/2}^{1/2}\int_{-\pi}^\pi\bm{K}_{\mc{T},1}\,\big(\bm{\varphi}(s-\bars,\theta-\bartheta)-\bm{\varphi}(s,\theta) \big)\,\mc{J}_\epsilon\,d\bartheta ds}_{\dot C^{0,\alpha}} \\
&\le c(\kappa_*,c_\Gamma)\max_i\big(\norm{Q_i}_{C^{0,\alpha^+}_1}+\norm{Q_i}_{C^{0,\alpha}_2}\big)\,\epsilon^{-1}\norm{\bm{\varphi}}_{C^{0,\alpha}}\\
&\le c(\kappa_{*,\alpha^+},c_\Gamma)\,\epsilon^{-1-\alpha^+}\norm{\bm{\varphi}}_{C^{0,\alpha}}\,,\\
\abs{\mc{T}_1^{(a)}[\bm{\varphi}]-\mc{T}_1^{(b)}[\bm{\varphi}]}_{\dot C^{0,\alpha}} &\le c(\kappa_{*,\alpha^+}^{(a)},\kappa_{*,\alpha^+}^{(b)},c_\Gamma)\,\epsilon^{-1-\alpha^+}\norm{\X^{(a)}-\X^{(b)}}_{C^{2,\alpha^+}}\norm{\bm{\varphi}}_{C^{0,\alpha}}\,.
 \end{align*}

Finally, we estimate $\mc{T}_0[\bm{\varphi}]$. Here we will use that $\bm{\varphi}\in C^{1,\alpha}$ and we may thus expand
\begin{equation}\label{eq:phi_expand_1}
\begin{aligned}
\bm{\varphi}(s-\bars,\theta-\bartheta)-\bm{\varphi}(s,\theta) &= \bars\bm{Q}_{\varphi,1}+\epsilon\bartheta\bm{Q}_{\varphi,2}\,,\\
\norm{\bm{Q}_{\varphi,1}}_Y&\le c\norm{\p_s\bm{\varphi}}_{Y}\,,  \quad 
\norm{\bm{Q}_{\varphi,2}}_Y\le \textstyle c\norm{\frac{1}{\epsilon}\p_\theta\bm{\varphi}}_{Y}\,,
\end{aligned}
\end{equation}
where the norms $Y$ are as in \eqref{eq:Y_norms} with $\alpha$ in place of $\beta$. Note that since $\bm{\varphi}$ is a function of just $(s,\theta)$, $\norm{\p_s\bm{\varphi}}_{C^{0,\alpha}_\mu}\le \norm{\p_s\bm{\varphi}}_{C^{0,\alpha}}$ for $\mu=1,2$. 
 Using \eqref{eq:phi_expand_1} along with Lemma \ref{lem:odd_nm}, we obtain the $L^\infty$ estimates
\begin{align*}
\abs{\mc{T}_0[\bm{\varphi}]} &\le \abs{{\rm p.v.}\int_{-1/2}^{1/2}\int_{-\pi}^\pi\bm{K}_{\mc{T},0}\,\big(\bars\bm{Q}_{\varphi,1}+\epsilon\bartheta\bm{Q}_{\varphi,2} \big)\,\mc{J}_\epsilon\,d\bartheta ds} \\
&\le c(\kappa_{*,\alpha},c_\Gamma)\,\epsilon^{\alpha}\big(\norm{\bm{Q}_{\varphi,1}}_{C^{0,\alpha}_2}+\norm{\bm{Q}_{\varphi,2}}_{C^{0,\alpha}_2}\big)\\
&\le c(\kappa_{*,\alpha},c_\Gamma)\,\epsilon^\alpha\big(\norm{\p_s\bm{\varphi}}_{C^{0,\alpha}}+\textstyle \norm{\frac{1}{\epsilon}\p_\theta\bm{\varphi}}_{C^{0,\alpha}}\big)\,.
\end{align*}
In addition, for two curves $\X^{(a)}$ and $\X^{(b)}$ satisfying Lemma \ref{lem:XaXb_C2beta}, we have
\begin{align*}
\abs{\mc{T}_0^{(a)}[\bm{\varphi}]-\mc{T}_0^{(b)}[\bm{\varphi}]} &\le \abs{{\rm p.v.}\int_{-1/2}^{1/2}\int_{-\pi}^\pi\big(\bm{K}_{\mc{T},0}^{(a)}\,\mc{J}_\epsilon^{(a)}-\bm{K}_{\mc{T},0}^{(b)}\,\mc{J}_\epsilon^{(b)}\big)\big(\bars\bm{Q}_{\varphi,1}+\epsilon\bartheta\bm{Q}_{\varphi,2} \big)\,d\bartheta ds} \\
&\le c(\kappa_{*,\alpha}^{(a)},\kappa_{*,\alpha}^{(b)},c_\Gamma)\,\epsilon^\alpha\norm{\X^{(a)}-\X^{(b)}}_{C^{2,\alpha}}\big(\norm{\p_s\bm{\varphi}}_{C^{0,\alpha}}+\textstyle \norm{\frac{1}{\epsilon}\p_\theta\bm{\varphi}}_{C^{0,\alpha}}\big)\,.
\end{align*}
Next, using Lemma \ref{lem:new_alpha_est}, we have
\begin{align*}
\abs{\mc{T}_0[\bm{\varphi}]}_{\dot C^{0,\alpha}} 
&\le c(\kappa_*,c_\Gamma)\bigg(\big(\textstyle \norm{\p_s\bm{\varphi}}_{L^\infty}+ \norm{\frac{1}{\epsilon}\p_\theta\bm{\varphi}}_{L^\infty}\big)\displaystyle \max_i\big(\norm{Q_i}_{C^{0,\alpha^+}_1}+\norm{Q_i}_{C^{0,\alpha}_2}\big) \\
&\qquad + \textstyle \big(\norm{\p_s\bm{\varphi}}_{C^{0,\alpha}}+ \norm{\frac{1}{\epsilon}\p_\theta\bm{\varphi}}_{C^{0,\alpha}}\big)\displaystyle \max_i\norm{Q_i}_{L^\infty}\bigg) \\
&\le c(\kappa_*,c_\Gamma) \textstyle \big(\norm{\p_s\bm{\varphi}}_{C^{0,\alpha}}+ \norm{\frac{1}{\epsilon}\p_\theta\bm{\varphi}}_{C^{0,\alpha}}\big) 
+ c(\kappa_{*,\alpha^+},c_\Gamma)\,\epsilon^{-\alpha^+}\big(\textstyle \norm{\p_s\bm{\varphi}}_{L^\infty}+ \norm{\frac{1}{\epsilon}\p_\theta\bm{\varphi}}_{L^\infty}\big)\,.
\end{align*}
Here the coefficients $Q_i$ include $\epsilon^{-1}\mc{J}_\epsilon$, $\epsilon\bm{Q}_r$, and $\be_r,\be_\theta$, which incur an additional $\epsilon^{-\alpha^+}$ due to $\theta$-dependence. 
Similarly, we have 
\begin{align*}
\abs{\mc{T}_0^{(a)}[\bm{\varphi}]-\mc{T}_0^{(b)}[\bm{\varphi}]}_{\dot C^{0,\alpha}} 
&\le c(\kappa_*^{(a)},\kappa_*^{(b)},c_\Gamma) \norm{\X^{(a)}-\X^{(b)}}_{C^2} \textstyle \big(\norm{\p_s\bm{\varphi}}_{C^{0,\alpha}}+ \norm{\frac{1}{\epsilon}\p_\theta\bm{\varphi}}_{C^{0,\alpha}}\big) \\
&\quad + c(\kappa_{*,\alpha^+}^{(a)},\kappa_{*,\alpha^+}^{(b)},c_\Gamma)\,\epsilon^{-\alpha^+}\norm{\X^{(a)}-\X^{(b)}}_{C^{2,\alpha^+}}\big(\textstyle \norm{\p_s\bm{\varphi}}_{L^\infty}+ \norm{\frac{1}{\epsilon}\p_\theta\bm{\varphi}}_{L^\infty}\big)\,.
\end{align*}

In addition to the stress $\bm{\sigma}[\mc{D}[\bm{\varphi}]]\bm{n}(\bx)=\mc{T}_0+\mc{T}_1+\mc{T}_2$, we need to estimate the stress $\bm{\sigma}[\bm{V}[\bm{\varphi}]]\bm{n}(\bx)$ due to the completion flow, given by \eqref{eq:sigmaVn}.

Using that $dS_x= \mc{J}_\epsilon\,d\theta ds$ and $\abs{\bR_x}\ge \epsilon$ while $|\bx'-\X(s')|\le \epsilon$, we have
\begin{align*}
\abs{\bm{\sigma}[\bm{V}[\bm{\varphi}]]\bm{n}} &\le c \int_{\Gamma_\epsilon}\frac{\abs{\bm{\varphi}}}{\epsilon^2}\,dS_{x'}
\le c(\kappa_*)\,\epsilon^{-1}\norm{\bm{\varphi}}_{L^\infty}\,,\\
\abs{\bm{\sigma}[\bm{V}^{(a)}[\bm{\varphi}]]\bm{n}^{(a)}-\bm{\sigma}[\bm{V}^{(b)}[\bm{\varphi}]]\bm{n}^{(b)}} &\le
c(\kappa_*^{(a)},\kappa_*^{(b)})\,\epsilon^{-1}\norm{\X^{(a)}-\X^{(b)}}_{C^2}\norm{\bm{\varphi}}_{L^\infty}\,.
\end{align*}
In addition, we have 
\begin{align*}
\abs{\bm{\sigma}[\bm{V}[\bm{\varphi}]]\bm{n}}_{\dot C^{0,\alpha}} &
\le c(\kappa_*)\,\epsilon^{-1-\alpha}\norm{\bm{\varphi}}_{L^\infty}\,,\\
\abs{\bm{\sigma}[\bm{V}^{(a)}[\bm{\varphi}]]\bm{n}^{(a)}-\bm{\sigma}[\bm{V}^{(b)}[\bm{\varphi}]]\bm{n}^{(b)}}_{\dot C^{0,\alpha}} &\le
c(\kappa_*^{(a)},\kappa_*^{(b)})\,\epsilon^{-1-\alpha}\norm{\X^{(a)}-\X^{(b)}}_{C^2}\norm{\bm{\varphi}}_{L^\infty}\,.
\end{align*}
Combining the estimates for $\bm{\sigma}[\bm{V}]$ with the estimates for $\mc{T}_0$, $\mc{T}_1$ and $\mc{T}_2$, we obtain Lemma \ref{lem:hyper}.
\hfill \qedsymbol


\section{Slender body NtD in H\"older spaces}\label{sec:SBNtD_holder}
Here we prove Lemma \ref{lem:holder_NtD} regarding H\"older space estimates for the slender body Neumann-to-Dirichlet map $\mc{L}_\epsilon$ about a curved filament. As in the Laplace setting \cite[section 5]{laplace}, we will rely on a Campanato-type argument.

We begin by recalling the $L^2$-based solution theory for the slender body boundary value problem \eqref{eq:SB_PDE}, developed in \cite[Theorem 1.2]{closed_loop}:
\begin{lemma}[$L^2$-based SB BVP solution theory]\label{lem:SB_BVP_L2}
Consider a filament $\Sigma_\epsilon$ with centerline $\X(s)\in C^2(\T)$. Given slender body Neumann data $\bm{f}(s)\in L^2(\T)$, the solution $(\bu,p)$ to the slender body boundary value problem \eqref{eq:SB_PDE} belongs to $D^{1,2}(\Omega_\epsilon)\times L^2(\Omega_\epsilon)$, where
\begin{align*}
D^{1,2}(\Omega_\epsilon):= \{\bu\in L^6(\Omega_\epsilon)\;:\; \nabla\bu\in L^2(\Omega_\epsilon) \}\,,
\end{align*}
and satisfies the bound
\begin{equation}\label{eq:D12_est}
\norm{\bu}_{D^{1,2}(\Omega_\epsilon)}+\norm{p}_{L^2(\Omega_\epsilon)}\equiv\norm{\nabla\bu}_{L^2(\Omega_\epsilon)}+\norm{p}_{L^2(\Omega_\epsilon)}\le c(\kappa_*,c_\Gamma)\abs{\log\epsilon}^{1/2}\norm{\bm{f}}_{L^2(\T)}\,.
\end{equation}
\end{lemma}
Recalling the definition \eqref{eq:rstar} of $r_*=r_*(c_\Gamma,\kappa_*)$, we may define the curved annular region 
\begin{equation}\label{eq:O_region}
\mc{O}_{r_*}= \big\{ \bx\in \Omega_\epsilon \;:\; \epsilon<\text{ dist}(\bx,\Gamma_0) < r_*\big\}
\end{equation}
about $\Sigma_\epsilon$. Within $\mc{O}_{r_*}$, the solution $(\bu,p)$ to \eqref{eq:SB_PDE} then satisfies  
\begin{equation}\label{eq:uH1_est}
\norm{\bu}_{H^1(\mc{O}_{r_*})}+\norm{p}_{L^2(\mc{O}_{r_*})}
\le c(\kappa_*,c_\Gamma,\abs{\mc{O}_{r_*}})\abs{\log\epsilon}^{1/2}\norm{\bm{f}}_{L^2(\T)}\,.
\end{equation}
We consider the weak form of the slender body BVP \eqref{eq:SB_PDE} in $\mc{O}_{r_*}$, given by
\begin{equation}\label{eq:O_weakform}
\int_{\mc{O}_{r_*}}\bigg(2\E(\bu):\nabla\bm{\xi}- p\,\div(\bm{\xi})\bigg)\,d\bx = \int_{\T}\bm{f}(s)\cdot{\bm \xi}(s)\,ds\,, \quad 
\E(\bu):= \frac{1}{2}(\nabla\bu+\nabla\bu^{\rm T})\,,
\end{equation}
for any $\bm{\xi}\in H^1(\mc{O}_{r_*})$ satisfying $\bm{\xi}\big|_{\Gamma_\epsilon}=\bm{\xi}(s)$ and $\bm{\xi}\big|_{\p\mc{O}_{r_*}\backslash\Gamma_\epsilon}=0$.

Writing $\bx\in\mc{O}_{r_*}$ as $\bx=\X(s)+r\be_r(s,\theta)$, we define a $C^{1,\beta}$ change of variables 
\begin{align*}
\Psi(\bx) &= s\be_z+r\overline{\be}_r(\theta) = s\be_z+r\cos\theta\,\be_x +r\sin\theta\,\be_y
\end{align*}
mapping the curved annulus $\mc{O}_{r_*}$ to the straight annulus $\Psi(\mc{O}_{r_*})$ about the straight cylinder $\mc{C}_\epsilon$. Here $(\be_z,\be_x,\be_y)$ are the standard Cartesian basis vectors.  We may calculate 
\begin{equation}\label{eq:nablaPhi}
\nabla\Psi^{-1} = (1-r\wh\kappa)\be_{\rm t}\otimes\be_z+r\kappa_3\be_\theta(s,\theta)\otimes\be_z+\be_\theta(s,\theta)\otimes\overline{\be}_\theta(\theta) + \be_r(s,\theta)\otimes\overline{\be}_r(\theta) \,,
\end{equation}
as well as 
\begin{equation}\label{eq:Phimatrix}
\begin{aligned}
\nabla\Psi(\nabla\Psi)^{\rm T}\circ \Psi^{-1} &= \frac{1+r^2\kappa_3^2}{(1-r\wh\kappa)^2}\be_z\otimes\be_z - \frac{r\kappa_3}{1-r\wh\kappa}\big(\overline{\be}_\theta(\theta)\otimes\be_z+\be_z\otimes\overline{\be}_\theta(\theta)\big) \\
&\qquad  +\overline{\be}_\theta(\theta)\otimes\overline{\be}_\theta(\theta) +\overline{\be}_r(\theta)\otimes\overline{\be}_r(\theta)
\end{aligned}
\end{equation}
and $\abs{\det\nabla\Psi}^{-1}=1-r\wh\kappa$.
We denote 
\begin{equation}\label{eq:AB_defs}
\bm{A}(s,r,\theta) = \bigg(\frac{1}{\abs{\det\nabla\Psi}}\nabla\Psi(\nabla\Psi)^{\rm T}\bigg)\circ\Psi^{-1}\,, 
\qquad \bm{B}(s,r,\theta) = \bigg(\frac{1}{\abs{\det\nabla\Psi}}\nabla\Psi\bigg)\circ\Psi^{-1}\,, 
\end{equation} 
and note that, using the additional $C^{3,\alpha}$ regularity of $\X$, we have 
\begin{equation}\label{eq:Breg}
\begin{aligned}
\frac{1}{c(\kappa_{*,\beta})}&\le \norm{\bm{A}}_{C^{0,\beta}(\Psi(\mc{O}_{r_*}))}\le c(\kappa_{*,\beta})\,, \\
\frac{1}{c(\norm{\X_{ss}}_{C^{1,\alpha}} )}&\le \norm{\bm{B}}_{C^{1,\alpha}(\Psi(\mc{O}_{r_*}))}\le c(\norm{\X_{ss}}_{C^{1,\alpha}} )\,,
\end{aligned}
\end{equation}
with smooth dependence on $\theta$. 
Here we have used that, using \eqref{eq:frame}, we may write 
\begin{align*}
  (\kappa_1)_s = \X_{sss}\cdot\be_{\rm n_1}+\kappa_2^2\,, \quad 
  (\kappa_2)_s = \X_{sss}\cdot\be_{\rm n_2}+\kappa_1^2\,,
\end{align*}
from which we may obtain $\norm{\kappa_1}_{C^{1,\alpha}}\,,\,\norm{\kappa_2}_{C^{1,\alpha}}\le c\norm{\X_{ss}}_{C^{1,\alpha}}$.

Let $\bm{\psi}= \bm{\xi}\circ\Psi^{-1}$ and note that along the fiber surface, $\bm{\psi}\big|_{\by\in\Psi(\Gamma_\epsilon)} = \bm{\xi}\big|_{\bx\in\Gamma_\epsilon}$ is still $\theta$-independent.
Letting $\overline\bu=\bu\circ\Psi^{-1}$ and $\overline p = p\circ\Psi^{-1}$,
we may rewrite \eqref{eq:O_weakform} in straightened variables $\by=\Psi(\bx)$ as
\begin{equation}\label{eq:SB_BVP_straight}
\int_{\Psi(\mc{O}_{r_*})}\bigg((\nabla\overline\bu\bm{A} +{\bm B}^{\rm T}\nabla\overline\bu^{\rm T}\nabla\Psi^{\rm T}):\nabla\bm{\psi}-\overline p\, \div(\bm{B}\bm{\psi})\bigg)\,d\by = \int_{\T}\bm{f}(s)\cdot{\bm \psi}(s)\,ds\,.
\end{equation}
We will use the formulation \eqref{eq:SB_BVP_straight} along with the analogous Campanato-type argument used in \cite{laplace} to show that $\bu=\overline\bu\circ\Psi$ satisfies Lemma \ref{lem:holder_NtD} along $\Gamma_\epsilon$.

We will make use of a Campanato-type function space tailored to the slender body BVP, which was defined in \cite{laplace} following the more classical constructions in \cite{giaquinta2013introduction,modica1982non}.
Given $\rho\le \frac{r_*}{2}$ and $s_0\in \T$, we define the following annular region about the straight cylinder $\mc{C}_\epsilon$:
\begin{align*}
A_\rho(s_0) &= \big\{ \by=s\be_z+r\be_r(\theta)\in \Psi(\mc{O}_{r_*}) \; : \; s_0-\rho < s < s_0+\rho\,, \; \epsilon\le r < \rho+\epsilon\, , \; 0\le\theta<2\pi \big\}\,.
\end{align*}
We emphasize that the volume of the region $A_\rho(s_0)$ scales like $\rho^2$ as $\rho\to0$.
Given $\bm{g}(\by)$, $\by\in A_\rho(s_0)$, we use the notation $\bm{g}_{s_0,\rho}$ to denote the mean of $\bm{g}$ over $A_\rho(s_0)$:
\begin{equation}\label{eq:mean_def}
\bm{g}_{s_0,\rho} = \fint_{A_\rho(s_0)}\bm{g}(\by)\,d\by\,.
\end{equation}
For $0<\alpha<1$, we define the seminorm $[ \cdot ]_{\mc{A}^{2,\alpha}}$ and full norm $\norm{\cdot}_{\mc{A}^{2,\alpha}}$ by
 \begin{equation}\label{eq:campanato}
 \begin{aligned}
 \big[ \bm{g} \big]_{\mc{A}^{2,\alpha}} &= \sup_{s_0\in \T,\,0<\rho\le r_*/2}\rho^{-\alpha}\bigg(\fint_{A_\rho(s_0)}\abs{\bm{g} - \bm{g}_{s_0,\rho}}^2\,d\by\bigg)^{1/2}\,, \\
 \norm{\bm{g}}_{\mc{A}^{2,\alpha}} &= \norm{\bm{g}}_{L^2(\Psi(\mc{O}_{r_*}))} + \big[ \bm{g} \big]_{\mc{A}^{2,\alpha}}\,.
 \end{aligned}
 \end{equation}

We may relate the function space defined by the norm $\norm{\cdot}_{\mc{A}^{2,\alpha}}$ to the H\"older space $C^{0,\alpha}$ about the straight filament $\mc{C}_\epsilon$ in the following way.
\begin{lemma}[Campanato-type norm bounds]\label{lem:camp}
Consider $\bm{g}$ in $\Psi(\mc{O}_{r_*})$ satisfying $\bm{g}\big|_{\mc{C}_\epsilon}=\bm{g}(s)$, a function of arclength only.
Then 
\begin{equation}\label{eq:Calpha_Aalpha}
\norm{\bm{g}}_{C^{0,\alpha}(\T)}\le c\,\norm{\bm{g}}_{\mc{A}^{2,\alpha}} \le c\,\norm{\bm{g}}_{C^{0,\alpha}(\Psi(\mc{O}_{r_*}))}\,.
\end{equation}
Here the left-hand-side $C^{0,\alpha}$ norm is taken over the filament centerline ($s\in \T$), while the right-hand-side $C^{0,\alpha}$ norm is over the entire annular region $\Psi(\mc{O}_{r_*})$. 
\end{lemma}
The proof of Lemma \ref{lem:camp} is given in \cite[Lemma 5.2]{laplace}.

We next consider a frozen-coefficients version of the slender body BVP \eqref{eq:SB_BVP_straight}. Taking $0<R\le \frac{r_*}{2}$ and $s_0\in\T$, we consider the annular region $A_R(s_0)$ about $\mc{C}_\epsilon$. Let $\bm{A}_0(\theta)= \bm{A}(s_0,\epsilon,\theta)$, $\bm{B}_0(\theta)=\bm{B}(s_0,\epsilon,\theta)$, and $\nabla\Psi_0(\theta)=\nabla\Psi(s_0,\epsilon,\theta)$, and note that $\bm{A}_0(\theta)$, $\bm{B}_0(\theta)$, and $\nabla\Psi_0(\theta)$ are smooth functions of $\theta$ satisfying $\frac{1}{c(\kappa_*)}\le \abs{\p_\theta^\ell{\bm A}_0},\abs{\p_\theta^\ell{\bm B}_0},\abs{\p_\theta^\ell\nabla\Psi_0}\le c(\kappa_*)$ for $c(\kappa_*)>0$ and $0\le\ell\in\Z$.
We consider $({\bm h},q)$ belonging to the space
 \begin{align*}
 \bigg\{({\bm h},q)\in H^1\times L^2(A_R(s_0))\,:\;\div(\bm{B}_0(\theta)\bm{h})=0\,,\; \bm{h}\big|_{\mc{C}_\epsilon}=\bm{h}(s)\,,\; \bm{h}\big|_{\p A_R\backslash\mc{C}_\epsilon}=\bu\,, \;\int_{A_R}q\,d\by=0 \bigg\}
 \end{align*}
 and satisfying 
\begin{equation}\label{eq:AR_weak}
 \begin{aligned}
\int_{A_R(s_0)}\bigg((\nabla\bm{h}\bm{A}_0(\theta)+\bm{B}_0(\theta)^{\rm T}\nabla\bm{h}^{\rm T}\nabla\Psi_0(\theta)^{\rm T})&:\nabla \bm{\psi} - q\,\div(\bm{B}_0(\theta)\bm{\psi})\bigg)\,d\by \\
& = \bm{f}(s_0)\cdot\int_{\abs{s-s_0}\le R}\bm{\psi}(s)\,ds
\end{aligned}
\end{equation}
for any $\bm{\psi}\in H^1(A_R(s_0))$ with $\bm{\psi}\big|_{\mc{C}_\epsilon}=\bm{\psi}(s)$ and $\bm{\psi}\big|_{\p A_R(s_0)\backslash\mc{C}_\epsilon}=0$.
We may show that $(\bm{h},q)$ satisfies the following proposition. 
\begin{proposition}[Higher regularity for $(\bm{h},q)$]\label{prop:h_highreg}
For $m\in\N$ and any $0<\rho<R$, the solution $(\bm{h},q)$ to \eqref{eq:AR_weak} satisfies 
\begin{equation}
\norm{\bm{h}}_{\dot H^{m+1}(A_\rho(s_0))}+ \norm{q}_{\dot H^m(A_\rho(s_0))} \le c(\epsilon,\kappa_*,c_\Gamma) \frac{1}{(R-\rho)^m}\big(\norm{\bm{h}}_{H^1(A_R(s_0))}+ \norm{q}_{L^2(A_R(s_0))}\big)\,.
\end{equation}
\end{proposition}

\begin{proof}
Fixing $R'<R$, we begin by defining a cutoff function
\begin{equation}\label{eq:eta_R}
\eta(\bx) = \begin{cases}
1\,, & \bx\in A_{R'}(s_0)  \\
0\,, & \bx\not\in A_R(s_0)\,,
\end{cases} 
\quad R'<R\,,
\end{equation}
with $\abs{\nabla\eta}\lesssim \frac{1}{R-R'}$.
We would like to use $\p_k(\eta^2\p_k\bm{h})$, $k=s,\theta$ as test functions in \eqref{eq:AR_weak} (this may be justified rigorously using finite differences), but first it will be convenient to correct for the divergence. 
For $k=s,\theta$, we consider $\bm{d}_k\in H^1_0(A_R(s_0))$ satisfying 
\begin{equation}\label{eq:dk_def}
\begin{aligned}
\div(\bm{B}_0\,\bm{d}_k) &= \div(\bm{B}_0\,\eta^2\p_k \bm{h}) \qquad \text{in }A_R(s_0)\\
\norm{\nabla\bm{d}_k}_{L^2(A_R(s_0))} &\le c(\kappa_*,c_\Gamma)\norm{\div(\bm{B}_0\,\eta^2\p_k \bm{h})}_{L^2(A_R(s_0))}\,.
\end{aligned}
\end{equation}
Such a $\bm{d}_k$ exists by \cite[Section III.3]{galdi2011introduction}, since $\bm{B}_0$ is invertible. In addition, the constant appearing in the bound \eqref{eq:dk_def} is independent of $\epsilon$ due to \cite[Appendix A.2.5]{closed_loop}, although we will not keep track of $\epsilon$-dependence throughout. 
We further note that, since $\div(\bm{B}_0\bm{h})=0$ in $A_R(s_0)$, and $\div$ commutes with $\p_s$ and $\p_\theta$ in cylindrical coordinates, we have 
\begin{align*}
\norm{\div(\bm{B}_0\,\eta^2\p_k \bm{h})}_{L^2(A_R(s_0))} &\le \norm{\div(\bm{B}_0\,\eta^2)\p_k \bm{h}}_{L^2(A_R(s_0))} + \norm{\eta^2\div((\p_k\bm{B}_0) \bm{h})}_{L^2(A_R(s_0))} \\
&\qquad + \norm{\eta^2\p_k\div(\bm{B}_0 \bm{h})}_{L^2(A_R(s_0))}\\
&\le c(\kappa_*)\norm{\eta\,\bm{h}}_{H^1(A_R(s_0))}\,,
\end{align*}
so, in particular, for $k=s,\theta$, we have 
\begin{align*}
\norm{\nabla\bm{d}_k}_{L^2(A_R(s_0))} \le c(\kappa_*,c_\Gamma)\norm{\eta \,\bm{h}}_{H^1(A_R(s_0))}\,.
\end{align*}

 We then take $\p_k(\eta^2\p_k\bm{h}-\bm{d}_k)$, $k=s,\theta$, as a test function in \eqref{eq:AR_weak}. We have 
\begin{equation}\label{eq:AR_test1}
\begin{aligned}
\int_{A_R(s_0)}\bigg((\nabla\bm{h}\bm{A}_0+\bm{B}_0^{\rm T}\nabla\bm{h}^{\rm T}\nabla\Psi_0^{\rm T})&:\nabla \big(\p_k(\eta^2\p_k\bm{h}-\bm{d}_k) \big) - q\,\div\big(\bm{B}_0\p_k(\eta^2\p_k\bm{h}-\bm{d}_k)\big)\bigg)\,d\by \\
& = \bm{f}(s_0)\cdot\int_{\abs{s-s_0}\le R}\p_k\eta^2\p_k\bm{h}(s)\,ds = 0\,.
\end{aligned}
\end{equation}

Before proceeding, we note that, using the expressions \eqref{eq:nablaPhi}-\eqref{eq:AB_defs} for $\bm{A}$, $\bm{B}$, and $\nabla\Psi$, for any $\bv\in L^2(A_R(s_0))$, we may write
\begin{equation}\label{eq:E0_def}
\begin{aligned}
  &(\nabla\bv\bm{A}_0+\bm{B}_0^{\rm T}\nabla\bv^{\rm T}\nabla\Psi_0^{\rm T}) = (\nabla\bv + \nabla\bv^{\rm T}) - \bm{E}_0(\nabla\bv)\,, \\
  &\qquad\abs{\bm{E}_0}\le 2\max\bigg\{ \abs{\epsilon\wh\kappa(s_0,\theta)},\abs{\epsilon\kappa_3}\,, \frac{\abs{\epsilon\wh\kappa(s_0,\theta)}+\epsilon^2\kappa_3^2}{\abs{1-\epsilon\wh\kappa}(s_0,\theta)} \bigg\}\abs{\nabla\bv} \le \epsilon c(\kappa_*)\abs{\nabla\bv}\,.
\end{aligned}
\end{equation}
Furthermore, for any $\bv$ satisfying $\bv=0$ on $\p A_R(s_0)\backslash \mc{C}_\epsilon$, we may use the homogeneous Korn inequality in $A_R(s_0)$ to bound
\begin{equation}\label{eq:KornyR}
  \int_{A_R(s_0)}\abs{\nabla\bv}^2\,d\by \le c(c_\Gamma,\kappa_*)\int_{A_R(s_0)}\abs{\E(\bv)}^2\,d\by\,, \qquad \E(\bv)= \frac{1}{2}(\nabla\bv + \nabla\bv^{\rm T})\,,
\end{equation}
where, by \cite[Lemma 2.6]{closed_loop}, $c(c_\Gamma,\kappa_*)$ is independent of $\epsilon$. 

Applying the above bounds, we may estimate $\eta\nabla\p_k\bm{h}$, $k=s,\theta$, in $A_R(s_0)$ as
\begin{align*}
  &\int_{A_R(s_0)}\eta^2\abs{\nabla\p_k\bm{h}}^2\,d\by 
\le \int_{A_R(s_0)}\abs{\nabla(\eta\p_k\bm{h})}^2\,d\by + \int_{A_R(s_0)}\abs{\nabla\eta}^2\abs{\p_k\bm{h}}^2\,d\by \\
&\quad \le c(c_\Gamma,\kappa_*)\int_{A_R(s_0)}\abs{\E(\eta\p_k\bm{h})}^2\,d\by + \int_{A_R(s_0)}\abs{\nabla\eta}^2\abs{\p_k\bm{h}}^2\,d\by \\
&\quad \le c(c_\Gamma,\kappa_*)\int_{A_R(s_0)}\eta^2\abs{\E(\p_k\bm{h})}^2\,d\by + c(c_\Gamma,\kappa_*)\int_{A_R(s_0)}\abs{\nabla\eta}^2\abs{\p_k\bm{h}}^2\,d\by \\
&\quad =c(c_\Gamma,\kappa_*)\int_{A_R(s_0)}\eta^2\E(\p_k\bm{h}):\nabla(\p_k\bm{h})\,d\by + c(c_\Gamma,\kappa_*)\int_{A_R(s_0)}\abs{\nabla\eta}^2\abs{\p_k\bm{h}}^2\,d\by \\
&\quad = c(c_\Gamma,\kappa_*)\int_{A_R(s_0)}\eta^2(\nabla(\p_k\bm{h})\bm{A}_0+\bm{B}_0^{\rm T}\nabla(\p_k\bm{h})\nabla\Psi_0^{\rm T}):\nabla(\p_k\bm{h})\,d\by \\
&\qquad + c(c_\Gamma,\kappa_*)\int_{A_R(s_0)}\eta^2\bm{E}_0(\p_k\bm{h}):\nabla(\p_k\bm{h})\,d\by + c(\Gamma,\kappa_*)\int_{A_R(s_0)}\abs{\nabla\eta}^2\abs{\p_k\bm{h}}^2\,d\by\,.
\end{align*}
For $\epsilon$ sufficiently small  (depending on $c_\Gamma$ and $\kappa_*$), the middle term may be absorbed into the left hand side to yield
\begin{equation}\label{eq:nablapk_bd}
\begin{aligned}
  \int_{A_R(s_0)}\eta^2\abs{\nabla\p_k\bm{h}}^2\,d\by &\le c(c_\Gamma,\kappa_*)\underbrace{\int_{A_R(s_0)}\eta^2(\nabla(\p_k\bm{h})\bm{A}_0+\bm{B}_0^{\rm T}\nabla(\p_k\bm{h})\nabla\Psi_0^{\rm T}):\nabla(\p_k\bm{h})\,d\by}_{=:J_0} \\
  &\quad + c(c_\Gamma,\kappa_*)\int_{A_R(s_0)}\abs{\nabla\eta}^2\abs{\p_k\bm{h}}^2\,d\by\,.
\end{aligned}
\end{equation}
To estimate $\eta\nabla\p_k\bm{h}$ in $A_R(s_0)$, we will thus require a bound for the integral $J_0$, which we may obtain from \eqref{eq:AR_test1} after some rearrangement. 

Again, since we are using cylindrical coordinates, $\p_s$ and $\p_\theta$ commute with $\nabla$ and $\div$, but we pick up additional terms in moving $\p_\theta$ throughout the coefficients $\bm{A}_0(\theta)$, $\bm{B}_0(\theta)$, and $\nabla\Psi_0(\theta)$. After commuting, integrating by parts, and using that $\div(\bm{B}_0(\eta^2\p_k\bm{h}-\bm{d}_k))=0$ for $k=s,\theta$, we obtain 
\begin{align*}
\abs{J_0}&\le \int_{A_R(s_0)}\bigg((\abs{\bm{A}_0}+\abs{\bm{B}_0}\abs{\nabla\Psi_0})\abs{\nabla\p_k\bm{h}}\big(\abs{\nabla(\eta^2)}\abs{\p_k\bm{h}}+\abs{\nabla\bm{d}_k} \big)\\
&\quad
 + (\abs{\p_k\bm{A}_0}+\abs{\p_k\bm{B}_0}\abs{\nabla\Psi_0}+ \abs{\bm{B}_0}\abs{\p_k\nabla\Psi_0}) \abs{\nabla\bm{h}}\big(\eta^2\abs{\nabla\p_k\bm{h}}+\abs{\nabla(\eta^2)}\abs{\p_k\bm{h}}+\abs{\nabla\bm{d}_k} \big)\bigg)\,d\by\\
&\quad  + \int_{A_R(s_0)}\abs{q}\bigg(\abs{\nabla\p_k\bm{B}_0}(\eta^2\abs{\p_k\bm{h}}+\abs{\bm{d}_k})\\
&\qquad  +\abs{\p_k\bm{B}_0}(\eta^2|\nabla\p_k\bm{h}|+\abs{\nabla(\eta^2)}|\p_k\bm{h}|+\abs{\nabla\bm{d}_k}) \bigg)\,d\by  \\
&\le c(\epsilon,\kappa_*,c_\Gamma)(R-R')^{-1}\bigg(\norm{\bm{h}}_{H^1(A_R)}+ \norm{q}_{L^2(A_R)}\bigg)\bigg(\norm{\eta\nabla\p_k\bm{h}}_{L^2(A_R)} +\norm{\eta\bm{h}}_{H^1(A_R)}\bigg)\,,
 \end{align*} 
 since $\abs{\nabla\eta}\lesssim \frac{1}{R-R'}$.
Combining with \eqref{eq:nablapk_bd} and using Young's inequality, over the smaller annulus $A_{R'}(s_0)$, we obtain the bounds 
 \begin{equation}\label{eq:H2_tang}
\norm{\nabla\p_k\bm{h}}_{L^2(A_{R'}(s_0))} \le c(\epsilon,c_\Gamma,\kappa_*)(R-R')^{-1}\big(\norm{\bm{h}}_{H^1(A_R(s_0))}+ \norm{q}_{L^2(A_R(s_0))}\big)\,, \qquad k=s,\theta\,.
 \end{equation}

We next consider tangential derivatives $\p_kq$, $k=s,\theta$, of the frozen-coefficient pressure $q$. We consider $\bm{g}_k\in H^1_0(A_{R}(s_0))$ satisfying 
\begin{equation}\label{eq:gk_def}
\begin{aligned}
\div(\bm{B}_0\,\bm{g}_k) &= \p_k (\eta^2 q)  \qquad \text{in }A_{R}(s_0)\\
\norm{\nabla\bm{g}_k}_{L^2(A_R (s_0))} &\le c(\kappa_*,c_\Gamma)\norm{\p_k (\eta^2 q)}_{L^2(A_{R}(s_0))}
\end{aligned}
\end{equation} 
Again, this may be justified by finite differences, and, since $\p_k (\eta^2 q)$ has mean zero on $A_{R}(s_0)$ and $\bm{B}_0$ is invertible, such a $\bm{g}_k$ exists by \cite[Section III.3]{galdi2011introduction}.

We now take $\eta^2\p_k\bm{g}_k$, $k=s,\theta$, as test functions in \eqref{eq:AR_weak}:
\begin{align*}
\int_{A_{R}(s_0)}\bigg((\nabla\bm{h}\bm{A}_0+\bm{B}_0^{\rm T}\nabla\bm{h}^{\rm T}\nabla\Psi_0^{\rm T}):\nabla (\eta^2\p_k\bm{g}_k) - q\,\div\big(\bm{B}_0(\eta^2\p_k\bm{g}_k)\big)\bigg)\,d\by = 0\,.
\end{align*}
Commuting $\p_k$ through the equations (and noting that this picks up a factor of $\p_\theta\bm{A}_0$, $\p_\theta\bm{B}_0$, and $\p_\theta\nabla\Psi_0$), after integrating by parts and using the definition \eqref{eq:gk_def} of $\bm{g}_k$, we obtain 
\begin{align*}
\int_{A_{R}(s_0)}\abs{\p_k(\eta^2q)}^2\,d\by 
&\le 
\int_{A_{R}(s_0)}(\abs{\bm{A}_0}+\abs{\bm{B}_0}\abs{\nabla\Psi_0})\bigg(2\abs{\nabla\bm{h}}\abs{\nabla\bm{g}_k}\abs{\nabla(\eta^2)}+ \eta^2\abs{\nabla\bm{g}_k}\abs{\nabla\p_k\bm{h}} \bigg)\,d\by\\
&\quad +\int_{A_{R}(s_0)}\eta^2\abs{\nabla\bm{g}_k}\big(\abs{\p_k\bm{A}_0}+\abs{\p_k\bm{B}_0}\abs{\nabla\Psi_0}+\abs{\bm{B}_0}\abs{\p_k\nabla\Psi_0}\big)\abs{\nabla\bm{h}} \,d\by\\
&\quad + \int_{A_{R}(s_0)}\abs{q}\bigg(\eta^2\big(\abs{\nabla\p_k\bm{B}_0}\abs{\bm{g}_k}+\abs{\p_k\bm{B}_0}\abs{\nabla\bm{g}_k}\big)+\abs{\bm{B}_0}\abs{\p_k\bm{g}_k}\abs{\nabla(\eta^2)} \bigg)\,d\by\,.
\end{align*}
Using Young's inequality and the bound \eqref{eq:gk_def}, and recalling that $\abs{\nabla\eta}\lesssim \frac{1}{R-R'}$, over the smaller annulus $A_{R'}(s_0)$, we then have 
\begin{equation}\label{eq:tang_q}
\begin{aligned}
\norm{\p_k q}_{L^2(A_{R'})} &\le c(\epsilon,\kappa_*,c_\Gamma)\big(\norm{\eta\nabla\p_k\bm{h}}_{L^2(A_R)} +(R-R')^{-1}\big(\norm{\bm{h}}_{H^1(A_R)} + \norm{q}_{L^2(A_{R})}\big) \big)  \\
&\le c(\epsilon,\kappa_*,c_\Gamma)(R-R')^{-1}\big(\norm{\bm{h}}_{H^1(A_{R})} + \norm{q}_{L^2(A_{R})} \big)\,.
\end{aligned}
\end{equation}
for $k=s,\theta$.
Note that for $\p_sq$, instead of \eqref{eq:gk_def}, we may instead consider $\bm{g}_s\in H^1_0(A_{R'}(s_0))$ satisfying 
\begin{equation}\label{eq:gs_def2}
\begin{aligned}
\div(\bm{B}_0\,\bm{g}_s) &= \p_s q-(\p_s q)_{s_0,R'}  \qquad \text{in }A_{R'}(s_0)\\
\norm{\nabla\bm{g}_s}_{L^2(A_{R'} (s_0))} &\le c(\kappa_*,c_\Gamma)\norm{\p_s q-(\p_s q)_{s_0,R'}}_{L^2(A_{R'}(s_0))}
\end{aligned}
\end{equation} 
where $(\p_s q)_{s_0,R'}$ is given by \eqref{eq:mean_def}. Using $\p_s\bm{g}_s$ as our test function in \eqref{eq:AR_weak}, we may then estimate 
\begin{align*}
\int_{A_{R'}(s_0)}\abs{\p_sq-(\p_s q)_{s_0,R'}}^2\,d\by &\le 
\int_{A_{R'}(s_0)}(\abs{\bm{A}_0}+\abs{\bm{B}_0}\abs{\nabla\Psi_0})\abs{\nabla\p_s\bm{h}}\abs{\nabla\bm{g}_s}\,d\by\\
&\le c(\epsilon,\kappa_*,c_\Gamma)\norm{\nabla\p_s\bm{h}}_{L^2(A_{R'})}\norm{\p_s q-(\p_s q)_{s_0,R'}}_{L^2(A_{R'})}\,.
\end{align*}
In particular, over $A_{R'}(s_0)$, we may estimate 
\begin{equation}\label{eq:tang_q_s}
\norm{\p_s q-(\p_s q)_{s_0,R'}}_{L^2(A_{R'})} \le c(\epsilon,\kappa_*,c_\Gamma)\norm{\nabla\p_s\bm{h}}_{L^2(A_{R'})}\,.
\end{equation}

We now turn to the normal direction derivatives $(\p_{rr}\bm{h},\p_rq)$, for which we must use the strong form of the equations \eqref{eq:AR_weak} within $A_R(s_0)$. 
First, using the form \eqref{eq:AB_defs} of $\bm{B}_0=\bm{B}(s_0,\epsilon,\theta)$ and $\nabla \Psi$ \eqref{eq:nablaPhi}, we may write the condition $\p_r\div(\bm{B}_0\bm{h})=0$ in cylindrical coordinates as
\begin{align*}
\p_{rr}h_r = -\p_r\bigg(\frac{1}{\abs{\det\nabla\Psi}}\bigg(\frac{1}{r} h_r + \frac{1}{r}\p_\theta \big( h_\theta +a_1(\theta)h_z\big) +a_2(\theta)\p_s h_z\bigg)\bigg)
\end{align*}
where $a_1(\theta)= -\frac{\epsilon\kappa_3}{1-\epsilon\wh\kappa(s_0,\theta)}$ and $a_2(\theta)=\frac{1}{1-\epsilon\wh\kappa(s_0,\theta)}$.
In particular, by \eqref{eq:H2_tang} and the boundedness of $a_1$ and $a_2$, we have 
\begin{align*}
\norm{\p_{rr}h_r}_{L^2(A_{R'})}\le c(\epsilon,\kappa_*,c_\Gamma)(R-R')^{-1}\big(\norm{\bm{h}}_{H^1(A_{R})} + \norm{q}_{L^2(A_{R})} \big)\,.
\end{align*}

Bounds for the remaining directions $\p_{rr}(\bm{h}\cdot\be_\theta)$, $\p_{rr}(\bm{h}\cdot\be_z)$, and $\p_rq$ come from the form of the frozen-coefficient straightened Stokes equations in $A_R(s_0)$. We have that $(\bm{h},q)$ satisfies 
\begin{equation}\label{eq:frozen_stokes}
\div(\nabla\bm{h}\bm{A}_0)+\bm{F}(\nabla\bm{h}) +\bm{G}(\bm{h})=\bm{B}_0^{\rm T}\nabla q \qquad \text{in }A_R(s_0)\,.
\end{equation}
Here the lower order terms $\bm{F}$ and $\bm{G}$ arise from the frozen-coefficient symmetric gradient term $\div(\bm{B}_0^{\rm T}\nabla\bm{h}^{\rm T}\nabla\Psi_0^{\rm T})$ and the divergence-free condition $\div(\bm{B}_0\bm{h})=0$.
Using the form \eqref{eq:AB_defs} of $\bm{A}_0=\bm{A}(s_0,\epsilon,\theta)$ and $\bm{B}_0$, we may write the normal direction components of the equation \eqref{eq:frozen_stokes} as 
\begin{align*}
\abs{\det\nabla\Psi}\div(\nabla\bm{h}\bm{A}_0)\cdot\be_r &= \frac{1}{r}\p_r(r\p_rh_r)+\frac{1}{r^2}\p_\theta(\p_\theta h_r-h_\theta)+ \p_{ss}h_r\,,\\
\abs{\det\nabla\Psi}(\bm{B}_0^{\rm T}\nabla q)\cdot\be_r &= \p_r q\,,
\end{align*}
from which we may bound $\p_rq$ by $\p_{rr}h_r$, $\p_{\theta\theta}h_r$, $\p_{ss}h_r$, $\nabla\bm{h}$, and $\bm{h}$ to obtain a full gradient bound for $q$:
\begin{equation}\label{eq:full_gradq}
\norm{\nabla q}_{A_{R'}(s_0)} \le c(\epsilon,\kappa_*,c_\Gamma)(R-R')^{-1}\big(\norm{\bm{h}}_{H^1(A_R)}+ \norm{q}_{L^2(A_R)}\big)\,.
\end{equation}
In the $\be_\theta$, $\be_z$ directions, we obtain 
\begin{align*}
\abs{\det\nabla\Psi}\div(\nabla\bm{h}\bm{A}_0)\cdot\be_\theta &= \frac{1}{r}\p_r(r(\p_rh_\theta+a_1\p_rh_z)+\frac{1}{r^2}\p_\theta(\p_\theta h_\theta+h_r + a_1\p_sh_z)+ \p_{ss}(h_\theta+a_1h_z)\,,\\
\abs{\det\nabla\Psi}(\bm{B}_0^{\rm T}\nabla q)\cdot\be_\theta &= \frac{1}{r}\p_\theta q\,,\\
\abs{\det\nabla\Psi}\div(\nabla\bm{h}\bm{A}_0)\cdot\be_r &= \frac{1}{r}\p_r(r(a_1\p_rh_\theta+a_3\p_rh_z)+\frac{1}{r^2}\p_\theta(a_1\p_\theta h_\theta+ a_1h_r + a_3\p_sh_z)+ \p_{ss}(a_1h_\theta+a_3h_z)\,,\\
\abs{\det\nabla\Psi}(\bm{B}_0^{\rm T}\nabla q)\cdot\be_z &=  a_2\p_sq +a_1\p_\theta q\,,
\end{align*}
where $a_1(\theta)$ and $a_2(\theta)$ are as above and $a_3(\theta)=\frac{1+\epsilon^2\kappa_3^2}{(1-\epsilon\wh\kappa(s_0,\theta))^2}$. Using \eqref{eq:frozen_stokes}, we obtain a system of two equations for the unknown directions $\p_{rr}h_\theta$ and $\p_{rr}h_z$ in terms of the remaining directions. Together, we may thus bound 
\begin{align*}
\norm{\p_{rr}h_\theta}_{L^2(A_{R'})}\,, \norm{\p_{rr}h_z}_{L^2(A_{R'})}\le c(\epsilon,\kappa_*,c_\Gamma)(R-R')^{-1}\big(\norm{\bm{h}}_{H^1(A_{R})} + \norm{q}_{L^2(A_{R})} \big)\,.
\end{align*}

Combining the tangential regularity bound \eqref{eq:H2_tang} with the pressure gradient bound \eqref{eq:full_gradq} and the bounds for $\p_{rr}\bm{h}$, we obtain the full $H^2(A_{R'}(s_0))\times H^1(A_{R'}(s_0))$ bound
\begin{equation}\label{eq:H2_est}
\norm{\bm{h}}_{H^2(A_{R'}(s_0))} + \norm{q}_{H^1(A_{R'}(s_0))} \le c(\epsilon,\kappa_*,c_\Gamma)(R-R')^{-1}\big(\norm{\bm{h}}_{H^1(A_R(s_0))}+ \norm{q}_{L^2(A_R(s_0))}\big)\,.
\end{equation}

This procedure may be iterated over smaller and smaller annuli, using test functions of the form $\p_k\p_\ell(\eta_1^2(\p_k\p_\ell\bm{h}))$, $\p_k\p_\ell\p_m(\eta_2^2(\p_k\p_\ell\p_m\bm{h}))$, etc., where $\eta_j$ are appropriate cutoffs and similar corrections for $\div(\bm{B}_0\cdot)$-freeness are made. 
We thus obtain Proposition \ref{prop:h_highreg}.
\end{proof}

Recalling the notation $\bm{g}_{s_0,R}$ for the mean of a function $\bm{g}$ over the annular region $A_R(s_0)$, we may now use Proposition \ref{prop:h_highreg} to show that the solution $(\bm{h},q)$ to \eqref{eq:AR_weak} satisfies the following. 
\begin{corollary}[Caccioppoli inequality]\label{cor:caccio}
For $0<R'<R$, the solution $(\bm{h},q)$ to \eqref{eq:AR_weak} satisfies 
\begin{equation}\label{eq:cacc}
\begin{aligned}
\int_{A_{R'}(s_0)}\abs{\nabla \p_s\bm{h}}^2\,d\by &\le c(\epsilon,\kappa_*)\frac{1}{(R-R')^2}\int_{A_R(s_0)}\abs{\p_s\bm{h}-(\p_s\bm{h})_{s_0,R}}^2\,d\by \,.
\end{aligned}
\end{equation}
\end{corollary}

\begin{proof}
Using that the coefficients $\bm{A}_0(\theta)$, $\bm{B}_0(\theta)$, $\nabla\Psi_0(\theta)$ are independent of $s$, for $(\bm{h},q)$ satisfying \eqref{eq:AR_weak}, we have that $(\p_s\bm{h},\p_sq-(\p_sq)_{s_0,R})$ satisfies 
\begin{equation}\label{eq:AR_hs_weak}
\begin{aligned}
\int_{A_R(s_0)}\bigg((\nabla\p_s\bm{h}\bm{A}_0(\theta)&+\bm{B}_0^{\rm T}(\theta)(\nabla\p_s\bm{h})^{\rm T}\nabla\Psi_0^{\rm T}):\nabla \bm{\psi} \\
&- (\p_sq-(\p_sq)_{s_0,R})\,\div(\bm{B}_0(\theta)\bm{\psi})\bigg)\,d\by = 0
\end{aligned}
\end{equation}
for any $\bm{\psi}\in H^1(A_R(s_0))$ with $\bm{\psi}\big|_{\mc{C}_\epsilon}=\bm{\psi}(s)$ and $\bm{\psi}\big|_{\p A_R(s_0)\backslash\mc{C}_\epsilon}=0$. Furthermore, since adding a constant does not affect the $\theta$-independence of $\p_s\bm{h}\big|_{\mc{C}_\epsilon}$, we have that $(\p_s\bm{h}-\bm{\lambda},\p_sq-(\p_sq)_{s_0,R})$ also satisfies \eqref{eq:AR_hs_weak} for any constant $\bm{\lambda}$. 

For $0<R'<R$, taking $\eta$ to be a cutoff function as in \eqref{eq:eta_R}, we will use $\bm{\psi}=\eta^2(\p_s\bm{h}-\bm{\lambda})$ as a test function in \eqref{eq:AR_hs_weak} with $(\p_s\bm{h}-\bm{\lambda},\p_sq-(\p_sq)_{s_0,R})$. We then have 
\begin{align*}
0 &= \int_{A_R(s_0)}\bigg(\big(\nabla(\p_s\bm{h}-\bm{\lambda})\bm{A}_0(\theta) + \bm{B}_0^{\rm T}(\theta)\nabla(\p_s\bm{h}-\bm{\lambda})^{\rm T}\nabla\Psi_0^{\rm T}\big):\nabla \big(\eta^2(\p_s\bm{h}-\bm{\lambda}) \big) \\
&\qquad - (\p_sq-(\p_sq)_{s_0,R})\,\div\big(\bm{B}_0(\theta)\eta^2(\p_s\bm{h}-\bm{\lambda})\big)\bigg)\,d\by \\
&= \int_{A_R(s_0)}\bigg(\eta^2\big(\nabla(\p_s\bm{h}-\bm{\lambda})\bm{A}_0(\theta) + \bm{B}_0^{\rm T}(\theta)\nabla(\p_s\bm{h}-\bm{\lambda})^{\rm T}\nabla\Psi_0^{\rm T}\big):\nabla (\p_s\bm{h}-\bm{\lambda}) \\
&\quad + \big(\nabla(\p_s\bm{h}-\bm{\lambda})\bm{A}_0(\theta)+ \bm{B}_0^{\rm T}(\theta)\nabla(\p_s\bm{h}-\bm{\lambda})\nabla\Psi_0^{\rm T}\big):\big(\nabla(\eta^2) \otimes(\p_s\bm{h}-\bm{\lambda}) \big) \\
&\quad - \eta^2(\p_sq-(\p_sq)_{s_0,R})\,\div\big(\bm{B}_0(\theta)(\p_s\bm{h}-\bm{\lambda})\big)
- (\p_sq-(\p_sq)_{s_0,R})\,\nabla(\eta^2)\cdot\big(\bm{B}_0(\p_s\bm{h}-\bm{\lambda})\big)\bigg)\,d\by\,.
\end{align*}
Since $\div(\bm{B}_0(\theta)(\p_s\bm{h}-\bm{\lambda}))=0$, after using an analogous argument to \eqref{eq:nablapk_bd} to bound $\eta\nabla(\p_s\bm{h}-\bm{\lambda})$ by the first term in the final expression above, we may estimate 
\begin{align*}
&\int_{A_R(s_0)}\eta^2\abs{\nabla(\p_s\bm{h}-\bm{\lambda})}^2\,d\by \\
&\qquad\le c(\epsilon,\kappa_*,c_\Gamma)\int_{A_R(s_0)}\abs{\nabla\eta}\abs{\p_s\bm{h}-\bm{\lambda}}\bigg(\abs{\eta\nabla(\p_s\bm{h}-\bm{\lambda})} + \abs{\p_sq-(\p_sq)_{s_0,R}}\bigg)\,d\by\,.
\end{align*}
Using Young's inequality, the support of $\eta$, and $\nabla\eta\lesssim\frac{1}{\abs{R-R'}}$ from \eqref{eq:eta_R}, we obtain 
\begin{align*}
\int_{A_{R'}(s_0)}\abs{\nabla\p_s\bm{h}}^2\,d\by 
&\le \frac{c(\delta)c(\epsilon,\kappa_*,c_\Gamma)}{(R-R')^2}\int_{A_R(s_0)}\abs{\p_s\bm{h}-\bm{\lambda}}^2\,d\by
+ \delta c(\epsilon,\kappa_*,c_\Gamma)\int_{A_R(s_0)}\abs{\p_sq-(\p_sq)_{s_0,R}}^2\,d\by
\end{align*}
for any small $\delta>0$. Using the estimate \eqref{eq:tang_q_s}, we have 
\begin{align*}
\int_{A_R(s_0)}\abs{\p_sq-(\p_s q)_{s_0,R}}^2\,d\by \le c(\epsilon,\kappa_*,c_\Gamma)\int_{A_R(s_0)}\abs{\nabla\p_s\bm{h}}^2\,d\by\,;
\end{align*}
so, in particular, 
\begin{equation}\label{eq:almost_there}
\int_{A_{R'}(s_0)}\abs{\nabla\p_s\bm{h}}^2\,d\by 
\le \frac{c(\delta)c(\epsilon,\kappa_*,c_\Gamma)}{(R-R')^2}\int_{A_R(s_0)}\abs{\p_s\bm{h}-\bm{\lambda}}^2\,d\by
+ \delta c(\epsilon,\kappa_*,c_\Gamma)\int_{A_R(s_0)}\abs{\nabla\p_s\bm{h}}^2\,d\by\,.
\end{equation}

We now recall the following useful proposition.
\begin{proposition}[Lemma 6.1 from \cite{giusti2003direct}]\label{prop:epsilon_lemma}
Let $Z(\rho)$ be a bounded, nonnegative function in the interval $[R_1,R_2]$. If for $R_1\le R'\le R\le R_2$ we have 
\begin{align*}
Z(R')\le M(R-R')^{-m} + \delta Z(R)
\end{align*}
for $M,m>0$ and $0\le \delta<1$, then
\begin{equation}
Z(R_1)\le c(m,\delta)M(R_2-R_1)^{-m}\,.
\end{equation}
\end{proposition}
Taking $\delta$ sufficiently small, we may apply Proposition \ref{prop:epsilon_lemma} to \eqref{eq:almost_there} with $Z(\rho)=\int_{A_\rho(s_0)}\abs{\nabla\p_s\bm{h}}^2\,d\by$. Taking $\bm{\lambda}=(\p_s\bm{h})_{s_0,R}$, we thus obtain \eqref{eq:cacc}.
\end{proof}

Finally, using Proposition \ref{prop:h_highreg} and Corollary \ref{cor:caccio}, we show that oscillation of $\p_s\bm{h}$ satisfies the following scaling.
\begin{lemma}[Oscillation of $\p_s\bm{h}$]\label{lem:osc_h}
Consider $(\bm{h},q)$ satisfying \eqref{eq:AR_weak}. For $0<\rho<R/2$, we have that $\p_s\bm{h}$ satisfies
\begin{equation}
\begin{aligned}
\fint_{A_\rho(s_0)}\abs{\p_s\bm{h}-(\p_s\bm{h})_{s_0,\rho}}^2\,d\by &\le 
c(\epsilon,\kappa_*,c_\Gamma)\bigg(\frac{\rho}{R}\bigg)^2\fint_{A_R(s_0)}\abs{\p_s \bm{h} - (\p_s \bm{h})_{s_0,R}}^2\,d\by \,.
\end{aligned}
\end{equation}
\end{lemma}

\begin{proof}
Using a Poincar\'e inequality over $A_\rho(s_0)$, Sobolev embedding, Proposition \ref{prop:h_highreg}, estimate \eqref{eq:tang_q_s}, and Corollary \ref{cor:caccio}, we may estimate 
\begin{align*}
\fint_{A_\rho(s_0)} &\abs{\p_s {\bm h} - (\p_s \bm{h})_{s_0,\rho}}^2 \,d\by \le c(\epsilon)\,\rho^2\fint_{A_\rho(s_0)} \abs{\nabla\p_s \bm{h}}^2 \,d\by \\
&\le c(\epsilon,\kappa_*)\,\rho^2 \int_0^{2\pi}\bigg(\sup_{\abs{s-s_0}\le \rho,\, 0\le r\le \rho}\abs{\nabla\p_s {\bm h}}^2\bigg)\,d\theta \\ 
&\le c(\epsilon,\kappa_*)\,\rho^2 \int_0^{2\pi}\bigg(\sup_{\abs{s-s_0}\le R/2,\, 0\le r\le R/2}\abs{\nabla\p_s {\bm h}}^2\bigg)\,d\theta \\
&\le c(\epsilon,\kappa_*)\,\rho^2\big(R^{-2}\norm{\nabla\p_s \bm{h}}_{L^2(A_{R/2}(s_0))}^2 + R^2\norm{\nabla\p_s \bm{h}}_{\dot H^2(A_{R/2}(s_0))}^2\big)\\
&
\le c(\epsilon,\kappa_*,c_\Gamma)\,\rho^2\fint_{A_{3R/4}(s_0)}\bigg( \abs{\nabla\p_s \bm{h}}^2 + \abs{\p_s q-(\p_s q)_{s_0,3R/4}}^2 \bigg)\,d\by \\
&\le c(\epsilon,\kappa_*,c_\Gamma)\,\rho^2\fint_{A_{3R/4}(s_0)}\abs{\nabla\p_s \bm{h}}^2\,d\by  
\le c(\epsilon,\kappa_*,c_\Gamma)\bigg(\frac{\rho}{R}\bigg)^2\fint_{A_R(s_0)}\abs{\p_s \bm{h} - (\p_s \bm{h})_{s_0,R}}^2\,d\by \,.
\end{align*}
In the fourth inequality we have used scaling along each $(s,r)$ slice in the annulus $A_{R/2}(s_0)$ to obtain the $R$-dependence of the Sobolev inequality. In the fifth inequality we have used that $(\p_s\bm{h},\p_s q-(\p_sq)_{s_0,3R/4})$ satisfies \eqref{eq:AR_weak} with zero boundary data, and we may thus apply the bound from Proposition \ref{prop:h_highreg} at the level of $(\p_s\bm{h},\p_s q-(\p_sq)_{s_0,3R/4})$.
\end{proof}

Equipped with Lemmas \ref{lem:camp} and \ref{lem:osc_h}, we may proceed to the proof of Lemma \ref{lem:holder_NtD}.
\begin{proof}[Proof of Lemma \ref{lem:holder_NtD}]
Consider the weak solution $(\overline\bu,\overline p)$ to the slender body BVP satisfying \eqref{eq:O_weakform} within the straightened region $\Psi(\mc{O}_{r_*})$. Let $0<R\le\frac{r_*}{2}$ and $s_0\in\T$, and consider the annular region $A_R(s_0)\subset \Psi(\mc{O}_{r_*})$. 

Within $A_R(s_0)$, we decompose $(\overline\bu,\overline p)$ as $\overline\bu=\bm{h}+\bm{b}$ and $\overline p=q+\zeta$, where $(\bm{h},q)$ satisfies the frozen-coefficient slender body PDE \eqref{eq:AR_weak}. 
Rewriting \eqref{eq:O_weakform} as 
\begin{align*}
&\int_{A_R}\big(\nabla\overline\bu({\bm A}-{\bm A}_0) + ({\bm B}-{\bm B}_0)^{\rm T}\nabla\overline\bu^{\rm T}\nabla\Psi^{\rm T} + \bm{B}_0^{\rm T}\nabla\overline\bu^{\rm T}(\nabla\Psi-\nabla\Psi_0)^{\rm T}\big):\nabla\bm{\psi}\,d\by \\
&+\int_{A_R}\big((\nabla\bm{h}+\nabla\bm{b}){\bm A}_0 + \bm{B}_0^{\rm T}(\nabla\bm{h}+\nabla\bm{b})^{\rm T}\nabla\Psi_0^{\rm T}\big):\nabla\bm{\psi}\,d\by \\
&- \int_{A_R}(\overline p-\overline p_{s_0,R})\,\div((\bm{B}-\bm{B}_0)\bm{\psi})\,d\by
- \int_{A_R}(q+\zeta-\overline p_{s_0,R})\,\div(\bm{B}_0\bm{\psi})\,d\by
= \int_{\T}\bm{f}(s)\bm{\psi}(s)\,ds
\end{align*}
for any $\bm{\psi}$ with $\bm{\psi}\big|_{\mc{C}_\epsilon}=\bm{\psi}(s)$ and $\bm{\psi}\big|_{\p A_R(s_0)\backslash\mc{C}_\epsilon}=0$, we may use that $(\bm{h},q)$ satisfies \eqref{eq:AR_weak} to obtain an equation for ${\bm b}$:
\begin{equation}\label{eq:OG_beqn}
\begin{aligned}
&\int_{A_R}\big(\nabla\bm{b}{\bm A}_0+ \bm{B}_0^{\rm T}\nabla\bm{b}^{\rm T}\nabla\Psi_0^{\rm T}\big):\nabla\bm{\psi}\,d\by
= \int_{\T}\big(\bm{f}(s)-\bm{f}(s_0)\big)\bm{\psi}(s)\,ds\\
&\qquad -\int_{A_R}\big(\nabla\overline\bu({\bm A}-{\bm A}_0)+(\bm{B}-\bm{B}_0)^{\rm T}\nabla\overline\bu^{\rm T}\nabla\Psi^{\rm T}
+\bm{B}_0^{\rm T}\nabla\overline\bu^{\rm T}(\nabla\Psi-\nabla\Psi_0)^{\rm T}\big):\nabla\bm{\psi}\,d\by \\
&\qquad + \int_{A_R}(\overline p-\overline p_{s_0,R})\,\div((\bm{B}-\bm{B}_0)\bm{\psi})\,d\by + \int_{A_R}(\zeta-\overline p_{s_0,R})\,\div(\bm{B}_0\bm{\psi})\,d\by
\,.
\end{aligned}
\end{equation}
If, in addition, we require that the test function $\bm{\psi}$ satisfies $\div(\bm{B}_0\bm{\psi})=0$, we obtain the equation
\begin{equation}\label{eq:divBfree_b}
\begin{aligned}
&\int_{A_R}\big(\nabla\bm{b}{\bm A}_0+ \bm{B}_0^{\rm T}\nabla\bm{b}^{\rm T}\nabla\Psi_0^{\rm T}\big):\nabla\bm{\psi}\,d\by
=  \int_{\T}\big(\bm{f}(s)-\bm{f}(s_0)\big)\bm{\psi}(s)\,ds\\
&\qquad -\int_{A_R}\big(\nabla\overline\bu({\bm A}-{\bm A}_0)+(\bm{B}-\bm{B}_0)^{\rm T}\nabla\overline\bu^{\rm T}\nabla\Psi^{\rm T}
+\bm{B}_0^{\rm T}\nabla\overline\bu^{\rm T}(\nabla\Psi-\nabla\Psi_0)^{\rm T}\big):\nabla\bm{\psi}\,d\by \\
&\qquad + \int_{A_R}(\overline p-\overline p_{s_0,R})\,\div((\bm{B}-\bm{B}_0)\bm{\psi})\,d\by
\,.
\end{aligned}
\end{equation}

Noting that $\bm{b}=\overline\bu-\bm{h}$ already satisfies the boundary conditions $\bm{b}\big|_{\mc{C}_\epsilon}=\bm{b}(s)$ and $\bm{b}\big|_{\p A_R(s_0)\backslash\mc{C}_\epsilon}=0$, we would like to use $\bm{b}$ as a test function above. It will be convenient to use \eqref{eq:divBfree_b}, so we must first correct for the divergence. We consider $\bm{g}\in H^1_0(A_R(s_0))$ satisfying
\begin{equation}\label{eq:g_correction}
\begin{aligned}
\div(\bm{B}_0 \,\bm{g}) &= \div(\bm{B}_0\,\bm{b}) \qquad \text{in }A_R(s_0)\\
\norm{\nabla \bm{g}}_{L^2(A_R(s_0))} &\le c(\kappa_*,c_\Gamma)\norm{\div(\bm{B}_0\,\bm{b})}_{L^2(A_R(s_0))}\,.
\end{aligned}
\end{equation}
As before, such a $\bm{g}$ exists due to \cite[Section III.3]{galdi2011introduction} and \cite[Appendix A.2.5]{closed_loop}. Using that $\div(\bm{B}\,\overline\bu)=0$ and $\div(\bm{B}_0\,\bm{h})=0$ in $A_R(s_0)$, we may write  
\begin{align*}
\div(\bm{B}_0\,\bm{b})=-\div((\bm{B}-\bm{B}_0)\,\overline\bu)\,,
\end{align*}
so, making use of the extra regularity assumption \eqref{eq:Breg} on the coefficient $\bm{B}$, we have  
\begin{equation}\label{eq:gcorrection_est}
\norm{\nabla \bm{g}}_{L^2(A_R(s_0))} \le c(\kappa_*,c_\Gamma)\norm{\div(\bm{B}_0\,\bm{b})}_{L^2(A_R(s_0))} \le c(\epsilon,\norm{\X_{ss}}_{C^{1,\alpha}},c_\Gamma)\,R^\alpha\norm{\overline\bu}_{H^1(A_R(s_0))} \,.
\end{equation}

We may now use $\bm{b}-\bm{g}$ as a test function in \eqref{eq:divBfree_b}. We first note that
\begin{align*}
   &\abs{\int_{A_R}\big(\nabla\bm{b}{\bm A}_0+ \bm{B}_0^{\rm T}\nabla\bm{b}^{\rm T}\nabla\Psi_0^{\rm T}\big):\nabla\bm{b}\,d\by} \le
   \int_{A_R}(\abs{{\bm A}_0}+\abs{\bm{B}_0}\abs{\nabla\Psi_0})\abs{\nabla\bm{b}}\abs{\nabla\bm{g}}\,d\by \\
  &\qquad +\int_{A_R}(\abs{{\bm A}-{\bm A}_0}+\abs{\bm{B}-\bm{B}_0}\abs{\nabla\Psi}+\abs{\bm{B}_0}\abs{\nabla\Psi-\nabla\Psi_0})\abs{\nabla\overline\bu}(\abs{\nabla\bm{b}}+\abs{\nabla\bm{g}})\,d\by\\
  &\qquad + \int_{A_R}\abs{\overline p-\overline p_{s_0,R}}\big(\abs{\bm{B}-\bm{B}_0}+\abs{\nabla(\bm{B}-\bm{B}_0)} \big)\big(\abs{\nabla\bm{b}}+\abs{\nabla\bm{g}}+\abs{\bm{b}}+\abs{\bm{g}}\big)\,d\by\\
  &\qquad+ \int_{\T}\abs{\bm{f}(s)-\bm{f}(s_0)}\abs{\bm{b}\big|_{\mc{C}_\epsilon}}\,ds \,.
 \end{align*} 
 Using the decomposition \eqref{eq:E0_def} and the homogeneous Korn inequality \eqref{eq:KornyR}, we may bound 
\begin{align*}
  \int_{A_R}\abs{\nabla\bm{b}}^2\,d\by \le c(c_\Gamma,\kappa_*)\int_{A_R}\big(\nabla\bm{b}{\bm A}_0+ \bm{B}_0^{\rm T}\nabla\bm{b}^{\rm T}\nabla\Psi_0^{\rm T}\big):\nabla\bm{b}\,d\by\,.
\end{align*}
Then, using the regularity of the coefficients $\bm{A}$, $\bm{B}$, and $\nabla\Psi$ as well as the bound \eqref{eq:gcorrection_est} for $\bm{g}$, we obtain 
\begin{align*}
\int_{A_R}\abs{\nabla\bm{b}}^2\,d\by
&\le c(\epsilon,\norm{\X_{ss}}_{C^{1,\alpha}},c_\Gamma)\bigg(R^\alpha\norm{\nabla\bm{b}}_{L^2(A_R)}\norm{\overline\bu}_{H^1(A_R)}+ R^{2\alpha}\norm{\overline\bu}_{H^1(A_R)}^2
\\
&\hspace{-1.6cm}\quad + R^\alpha\norm{\overline p-\overline p_{s_0,R}}_{L^2(A_R)}(\norm{\nabla\bm{b}}_{L^2(A_R)}+R^\alpha\norm{\overline\bu}_{H^1(A_R)}) + R^{1/2+\alpha}\norm{\bm{f}}_{C^{0,\alpha}}\norm{\bm{b}\big|_{\mc{C}_\epsilon}}_{L^2(|s-s_0|<R)}\bigg)\,.
\end{align*}
By a scaling argument (see \cite[proof of Lemma 1.8]{laplace}), the following trace inequality holds in $A_R(s_0)$:
\begin{align*}
\norm{\bm{b}\big|_{\mc{C}_\epsilon}}_{L^2(|s-s_0|<R)}^2 \le c(\epsilon)\,R\norm{\nabla{\bm b}}_{L^2(A_R(s_0))}^2\,.
\end{align*}
Using Young's inequality, we thus obtain the bound
\begin{equation}\label{eq:bestimate}
\int_{A_R}\abs{\nabla\bm{b}}^2\,d\by
\le c(\epsilon,\norm{\X_{ss}}_{C^{1,\alpha}},c_\Gamma)\,R^{2\alpha}
\bigg(\norm{\overline\bu}_{H^1(A_R)} + \norm{\overline p-\overline p_{s_0,R}}_{L^2(A_R)}^2 + R^2\norm{\bm{f}}_{C^{0,\alpha}(\T)}^2\bigg)\,.
\end{equation}

Returning to the equation \eqref{eq:OG_beqn} for $\bm{b}$, we may estimate the pressure term $\zeta-\overline p_{s_0,R}$ by choosing a test function $\bm{\psi}\in H^1_0(A_R(s_0))$ satisfying 
\begin{align*}
\div(\bm{B}_0\bm{\psi})&=\zeta-\overline p_{s_0,R} \qquad \text{in }A_R(s_0)\\
\norm{\nabla\bm{\psi}}_{L^2(A_R(s_0))}&\le c(\kappa_*,c_\Gamma)\norm{\zeta-\overline p_{s_0,R}}_{L^2(A_R(s_0))}\,.
\end{align*}
Using this $\bm{\psi}$ in \eqref{eq:OG_beqn}, we obtain 
\begin{align*}
\int_{A_R(s_0)}\abs{\zeta-\overline p_{s_0,R}}^2\,d\by 
&\le \int_{A_R(s_0)}\bigg((\abs{{\bm A}_0}+\abs{\bm{B}_0}\abs{\nabla\Psi_0})\abs{\nabla\bm{b}}\abs{\nabla\bm{\psi}} + \abs{\overline p-\overline p_{s_0,R}}\abs{\div((\bm{B}-\bm{B}_0)\bm{\psi})}\\
&\qquad +(\abs{{\bm A}-{\bm A}_0}+\abs{{\bm B}-{\bm B}_0}\abs{\nabla\Psi}+\abs{\bm{B}_0}\abs{\nabla\Psi-\nabla\Psi_0})\abs{\nabla\overline\bu}\abs{\nabla\bm{\psi}}\bigg)\,d\by \\
&\hspace{-1cm}\le c(\epsilon,\norm{\X_{ss}}_{C^{1,\alpha}},c_\Gamma)\bigg(\norm{\nabla\bm{b}}_{L^2(A_R)}^2+R^{2\alpha}\norm{\nabla\overline\bu}_{L^2(A_R)}^2+R^{2\alpha}\norm{\overline p-\overline p_{s_0,R}}_{L^2(A_R)}^2 \bigg)\,.
\end{align*}
Combining with \eqref{eq:bestimate}, we have 
\begin{equation}\label{eq:bzeta_estimate}
\begin{aligned}
&\int_{A_R(s_0)}\abs{\nabla\bm{b}}^2\,d\by + \int_{A_R(s_0)}\abs{\zeta-\overline p_{s_0,R}}^2\,d\by\\
&\qquad\le c(\epsilon,\norm{\X_{ss}}_{C^{1,\alpha}},c_\Gamma)
\bigg(R^{2\alpha}\norm{\overline \bu}_{H^1(A_R)} + R^{2\alpha}\norm{\overline p-\overline p_{s_0,R}}_{L^2(A_R)}^2 + R^{2+2\alpha}\norm{\bm{f}}_{C^{0,\alpha}(\T)}^2\bigg)\,.
\end{aligned}
\end{equation}

Next, using Proposition \ref{prop:h_highreg} and recalling that $A_\rho$ scales like $\rho^2$ for small $\rho$, for $\rho<R/2$, we have that $\bm{h}=\overline\bu-\bm{b}$ and $q=\overline p-\zeta$ satisfy 
\begin{align*}
&\int_{A_\rho(s_0)}\bigg(\abs{\nabla \bm{h}}^2+\abs{\bm{h}}^2+\abs{q}^2\bigg)\,d\by \le c\,\rho^2\int_0^{2\pi}\sup_{\abs{s-s_0}\le \rho,\,0\le r\le \rho}\bigg(\abs{\nabla \bm{h}}^2+\abs{\bm{h}}^2+\abs{q}^2\bigg)\,d\theta \\
&\quad \le c\,\rho^2\int_0^{2\pi}\sup_{\abs{s-s_0}\le R/2,\,0\le r\le R/2}\bigg(\abs{\nabla \bm{h}}^2+\abs{\bm{h}}^2+\abs{q}^2\bigg)\,d\theta \\
&\quad
\le c(\epsilon,\kappa_*,R)\,\rho^2\big(\norm{\bm{h}}_{H^3(A_{R/2}(s_0))}^2+\norm{q}_{H^2(A_{R/2}(s_0))}^2\big)\\
&\quad\le c(\epsilon,\kappa_*)\bigg(\frac{\rho}{R}\bigg)^2\bigg(\norm{\bm{h}}_{H^1(A_{R}(s_0))}^2 + \norm{q}_{L^2(A_{R}(s_0))}^2\bigg)\,,
\end{align*}
by a similar scaling argument to the proof of Lemma \ref{lem:osc_h}. We thus have 
\begin{equation}\label{eq:up_est}
\begin{aligned}
&\norm{\overline\bu}_{H^1(A_\rho)}^2 + \norm{\overline p-\overline p_{s_0,\rho}}_{L^2(A_\rho)}^2
\le 2\big(\norm{\bm{h}}_{H^1(A_\rho)}^2 + \norm{\bm{b}}_{H^1(A_\rho)}^2 + \norm{q}_{L^2(A_\rho)}^2+ \norm{\zeta-\overline p_{s_0,\rho}}_{L^2(A_\rho)}^2\big) \\
&\quad\le c(\epsilon,\kappa_*)\bigg(\frac{\rho}{R}\bigg)^2\bigg(\norm{\bm{h}}_{H^1(A_R)}^2+ \norm{q}_{L^2(A_R)}^2 \bigg) + 2\norm{\bm{b}}_{H^1(A_\rho)}^2+ 2\norm{\zeta-\overline p_{s_0,\rho}}_{L^2(A_\rho)}^2\\
%
%
&\quad\le c(\epsilon,\kappa_*)\bigg(\frac{\rho}{R}\bigg)^2\bigg(\norm{\overline \bu}_{H^1(A_R)}^2+ \norm{\overline p-\overline p_{s_0,R}}_{L^2(A_R)}^2 \bigg) \\
&\qquad + c(\epsilon,\norm{\X_{ss}}_{C^{1,\alpha}},c_\Gamma)
\bigg(R^{2\alpha}\big(\norm{\overline \bu}_{H^1(A_R)} +\norm{\overline p-\overline p_{s_0,R}}_{L^2(A_R)}^2\big) + R^{2+2\alpha}\norm{\bm{f}}_{C^{0,\alpha}(\T)}^2\bigg)\,,
\end{aligned}
\end{equation}
where we have used the estimate \eqref{eq:bzeta_estimate}.

We may now apply a useful proposition from \cite{giaquinta2013introduction}: 
\begin{proposition}[Lemma 5.13 from \cite{giaquinta2013introduction}]\label{prop:giaq}
Given $m>\nu>0$ and a nondecreasing function $\Xi:\R^+\to\R^+$ satisfying
\begin{align*}
\Xi(\rho) \le c_1\bigg(\bigg(\frac{\rho}{R}\bigg)^m+\omega\bigg)\Xi(R)+c_2R^\nu
\end{align*}
for all $0<\rho\le R\le \frac{r_*}{2}$, for $\omega$ sufficiently small (depending on $c_1,m,\nu$), we in fact have
\begin{equation}\label{eq:giaquinta_est}
\Xi(\rho) \le c_3\bigg(\frac{\Xi(R)}{R^\nu}+c_2\bigg)\rho^\nu 
\end{equation}
for all $0\le \rho\le R\le \frac{r_*}{2}$.
\end{proposition}

Applying Proposition \ref{prop:giaq} to \eqref{eq:up_est} with $\Xi(\rho)=\norm{\overline\bu}_{H^1(A_\rho)}^2 + \norm{\overline p-\overline p_{s_0,\rho}}_{L^2(A_\rho)}^2$, $m=2$, and $\nu=2-\delta$ for $\delta>0$, for $R$ sufficiently small we obtain
\begin{equation}\label{eq:up_linfty}
\begin{aligned}
&\norm{\overline\bu}_{H^1(A_\rho(s_0))}^2 + \norm{\overline p}_{L^2(A_\rho(s_0))}^2\\
&\le c(\epsilon,\norm{\X_{ss}}_{C^{1,\alpha}},c_\Gamma) 
\bigg(R^{-(2-\delta)}\big(\norm{\overline\bu}_{H^1(A_R)} +\norm{\overline p-\overline p_{s_0,R}}_{L^2(A_R)}^2\big) + R^{2\alpha+\delta}\norm{\bm{f}}_{C^{0,\alpha}(\T)}^2\bigg) \rho^{2-\delta}
\end{aligned}
\end{equation}
for $0\le\rho\le R\le\frac{r_*}{2}$.

We may now estimate the oscillation of $\p_s\overline\bu$ over $A_\rho(s_0)$. For $0<\rho\le R' \le R\le \frac{r_*}{2}$, we have 
\begin{equation}\label{eq:main_comp}
\begin{aligned}
&\int_{A_\rho(s_0)} \abs{\p_s \overline\bu - (\p_s \overline\bu)_{s_0,\rho}}^2\,d\by \le 
\int_{A_\rho(s_0)} 2\bigg(\abs{\p_s \bm{h} - (\p_s \bm{h})_{s_0,\rho}}^2 +\abs{\p_s \bm{b} - (\p_s \bm{b})_{s_0,\rho}}^2\bigg)\,d\by \\
& \le c(\epsilon,\kappa_*,c_\Gamma)\bigg(\frac{\rho}{R'} \bigg)^4 \int_{A_{R'}(s_0)}\abs{\p_s \bm{h} - (\p_s \bm{h})_{s_0,R'}}^2\,d\by + 4\int_{A_\rho(s_0)}\abs{\nabla \bm{b}}^2\,d\by \\
& \le c(\epsilon,\kappa_*,c_\Gamma) \bigg(\frac{\rho}{R'} \bigg)^4 \int_{A_{R'}(s_0)}\abs{\p_s \overline\bu - (\p_s \overline\bu)_{s_0,R'}}^2\,d\by + c(\epsilon,\kappa_*)\int_{A_{R'}(s_0)}\abs{\nabla {\bm b}}^2\,d\by \\
& \le c(\epsilon,\kappa_*,c_\Gamma) \bigg(\frac{\rho}{R'} \bigg)^4 \int_{A_{R'}(s_0)}\abs{\p_s \overline\bu - (\p_s \overline\bu)_{s_0,R'}}^2\,d\by \\
&\quad + c(\epsilon,\norm{\X_{ss}}_{C^{1,\alpha}},c_\Gamma)
\bigg((R')^{2\alpha}\big(\norm{\overline\bu}_{H^1(A_{R'})} + \norm{\overline p-\overline p_{s_0,R'}}_{L^2(A_{R'})}^2\big) + (R')^{2+2\alpha}\norm{\bm{f}}_{C^{0,\alpha}(\T)}^2\bigg) \\
& \le  c(\epsilon,\kappa_*,c_\Gamma) \bigg(\frac{\rho}{R'} \bigg)^4 \int_{A_{R'}(s_0)}\abs{\p_s \overline\bu - (\p_s \overline\bu)_{s_0,R'}}^2\,d\by 
 + c(\epsilon,\norm{\X_{ss}}_{C^{1,\alpha}},c_\Gamma)\,(R')^{2+2\alpha-\delta}\,M  \,,
\end{aligned}
\end{equation}
where $M = R^{-(2-\delta)}\big(\norm{\overline\bu}_{H^1(A_{R})} + \norm{\overline p-\overline p_{s_0,R}}_{L^2(A_{R})}^2\big)+ R^\delta\norm{\bm{f}}_{C^{0,\alpha}(\T)}^2$.

Using Proposition \ref{prop:giaq} with $\Xi(\rho)= \int_{A_\rho(s_0)} \abs{\p_s \overline\bu - (\p_s \overline\bu)_{s_0,\rho}}^2\,d\by$, $m=4$, and $\nu=2+2\alpha-\delta$, we obtain
\begin{align*}
&\int_{A_\rho(s_0)} \abs{\p_s \overline\bu - (\p_s \overline\bu)_{s_0,\rho}}^2\,d\by\\
&\qquad\le \rho^{2+2\alpha-\delta}c(\epsilon,\norm{\X_{ss}}_{C^{1,\alpha}} ,c_\Gamma)\bigg(R^{-(2+2\alpha-\delta)} \int_{A_{R}(s_0)}\abs{\p_s \overline\bu - (\p_s \overline\bu)_{s_0,R}}^2\,d\by + M\bigg)
\end{align*}
for any $0\le \rho\le R\le\frac{r_*}{2}$. Using Lemma \ref{lem:camp}, we have $\p_s\overline\bu\in C^{0,\alpha-\frac{\delta}{2}}$ for $\delta>0$ small. Furthermore, using that $\overline\bu\big|_{\Psi(\Gamma_\epsilon)}$ is independent of $\theta$, we also have $\p_\theta\overline\bu\big|_{\Psi(\Gamma_\epsilon)\cap\p A_R(s_0)}\in C^{0,\alpha-\frac{\delta}{2}}$. Using classical elliptic regularity theory, we then have $(\overline\bu,\overline p-\overline p_{s_0,R})\in C^{1,\alpha-\frac{\delta}{2}}(A_R(s_0))\times C^{0,\alpha-\frac{\delta}{2}}(A_R(s_0))$ with
\begin{equation}\label{eq:bu_AR_est}
\begin{aligned}
&\norm{\overline\bu}_{C^{1,\alpha-\frac{\delta}{2}}(A_R(s_0))} + \norm{\overline p-\overline p_{s_0,R}}_{C^{0,\alpha-\frac{\delta}{2}}(A_R(s_0))} \\
&\qquad \le 
c(\epsilon,\norm{\X_{ss}}_{C^{1,\alpha}} ,c_\Gamma)\big(\norm{\overline\bu}_{H^1(\Psi(\mc{O}_{r_*}))} + \norm{\overline p}_{L^2(\Psi(\mc{O}_{r_*}))} +\norm{\bm{f}}_{C^{0,\alpha}(\T)} \big)\,.
\end{aligned}
\end{equation}
In particular, we thus have $(\nabla\overline\bu,\overline p-\overline p_{s_0,R})\in L^\infty(A_R(s_0))\times L^\infty(A_R(s_0))$ and we may estimate
\begin{align*}
\int_{A_R(s_0)}\bigg( \abs{\nabla \overline\bu}^2+\abs{\overline \bu}^2+\abs{\overline p-\overline p_{s_0,R}}^2\bigg)\,d\by &\le c\,R^2\sup_{\by\in A_R(s_0)}\bigg( \abs{\nabla \overline\bu}^2+\abs{\overline\bu}^2+\abs{\overline p-\overline p_{s_0,R}}^2\bigg)\,.
\end{align*}
The estimate \eqref{eq:main_comp} may thus be improved to 
\begin{equation}\label{eq:main_comp2}
\begin{aligned}
&\int_{A_\rho(s_0)} \abs{\p_s \overline\bu - (\p_s \overline\bu)_{s_0,\rho}}^2\,d\by \\
&\qquad \le  c(\epsilon,\kappa_*,c_\Gamma) \bigg(\frac{\rho}{R'} \bigg)^4 \int_{A_{R'}(s_0)}\abs{\p_s \overline\bu - (\p_s \overline\bu)_{s_0,R'}}^2\,d\by 
 + c(\epsilon,\norm{\X_{ss}}_{C^{1,\alpha}} ,c_\Gamma)\,(R')^{2+2\alpha}\,M_2  \,,
\end{aligned}
\end{equation}
where $M_2=\sup_{\by\in A_R(s_0)}\big( \abs{\nabla \overline\bu}^2+\abs{\overline\bu}^2+\abs{\overline p-\overline p_{s_0,R}}^2\big)+ \norm{\bm{f}}_{C^{0,\alpha}(\T)}^2$. By Proposition \ref{prop:giaq}, we then have
\begin{equation}
\begin{aligned}
&\fint_{A_\rho(s_0)}\abs{\p_s\overline\bu-(\p_s\overline\bu)_{s_0,\rho}}^2\,d\by\\
&\qquad\le \rho^{2\alpha}\,c(\epsilon,\norm{\X_{ss}}_{C^{1,\alpha}} ,c_\Gamma)\bigg(R^{-2\alpha}\fint_{A_R(s_0)}\abs{\p_s\overline\bu-(\p_s\overline\bu)_{s_0,\rho}}^2\,d\by+M_2 \bigg)\,.
\end{aligned}
\end{equation}
By covering the region $\Psi(\mc{O}_{r_*})$ with annuli $A_R(s_0)$, $s_0\in \T$, we thus have that $\p_s\overline\bu$ belongs to the Campanato-type space $\mc{A}^{2,\alpha}$ \eqref{eq:campanato}. Using Lemma \ref{lem:camp} along with $\bu=\overline\bu\circ\Psi$, we obtain that $\bu\big|_{\Gamma_\epsilon} = \bu(s)\in C^{1,\alpha}(\T)$. The estimate \eqref{eq:holderNTD1} follows from the bounds \eqref{eq:bu_AR_est} and \eqref{eq:uH1_est}. \\

Finally, we show the Lipschitz estimate \eqref{eq:holderNTD2}. Given two nearby filaments with centerlines $\X^{(a)}$, $\X^{(b)}$ satisfying Lemma \ref{lem:XaXb_C2beta}, we aim to bound the difference $\bv^{(a)}(s)-\bv^{(b)}(s)$ for $\bv^{(a)}(s),\bv^{(b)}(s)$ satisfying
\begin{align*}
(\mc{L}_\epsilon^{(a)})^{-1}[\bv^{(a)}](s) = \bm{f}(s) = (\mc{L}_\epsilon^{(b)})^{-1}[\bv^{(b)}](s)\,.
\end{align*}
Rearranging, we may write 
\begin{align*}
\bv^{(a)}-\bv^{(b)} = \mc{L}_\epsilon^{(a)}\big[(\mc{L}_\epsilon^{(b)})^{-1}[\bv^{(b)}]-(\mc{L}_\epsilon^{(a)})^{-1}[\bv^{(b)}] \big] \,.
\end{align*}
We may then apply the estimate \eqref{eq:holderNTD1} for $\mc{L}_\epsilon^{(a)}$ and $\mc{L}_\epsilon^{(b)}$, and the Lipschitz estimate \eqref{eq:thm_DtN_ests_lip} for $\mc{L}_\epsilon^{-1}$ to obtain 
\begin{align*}
\norm{\bv^{(a)}-\bv^{(b)}}_{C^{1,\alpha}(\T)} &= \norm{\mc{L}_\epsilon^{(a)}\big[(\mc{L}_\epsilon^{(b)})^{-1}[\bv^{(b)}]-(\mc{L}_\epsilon^{(a)})^{-1}[\bv^{(b)}] \big] }_{C^{1,\alpha}(\T)} \\
&\le c(\epsilon,\|\X_{ss}^{(a)}\|_{C^{1,\alpha}},c_\Gamma)\,\norm{(\mc{L}_\epsilon^{(b)})^{-1}[\bv^{(b)}]-(\mc{L}_\epsilon^{(a)})^{-1}[\bv^{(b)}] }_{C^{0,\alpha}(\T)} \\
&\le c(\epsilon,\|\X_{ss}^{(a)}\|_{C^{1,\alpha}},\kappa_{*,\alpha^+}^{(b)},c_\Gamma)\norm{\X^{(a)}-\X^{(b)}}_{C^{2,\alpha^+}}\norm{\bv^{(b)}}_{C^{1,\alpha}(\T)} \\
&\le c(\epsilon,\|\X_{ss}^{(a)}\|_{C^{1,\alpha}},\|\X_{ss}^{(b)}\|_{C^{1,\alpha}},c_\Gamma)\norm{\X^{(a)}-\X^{(b)}}_{C^{2,\alpha^+}}\norm{\bm{f}}_{C^{0,\alpha}(\T)}\,.
\end{align*}
\end{proof}


\appendix
\section{Proof of Lemma \ref{lem:XaXb_C2beta}: nearby curves}\label{app:XaXb} 
Given the two curves $\X^{(a)}(s)$, $\X^{(b)}(s)$ satisfying \eqref{eq:XaXb_close}, we begin by framing both curves with a corresponding Bishop frame $(\be_{\rm t},\wt{\be}_1,\wt{\be}_2)$ \cite{bishop1975there}, which, given an initial choice of normal vector $\wt\be_1(0)$, satisfies the ODE 
\begin{equation}\label{eq:Bishop}
\frac{d}{ds} \begin{pmatrix}
\be_{\rm t}\\
\wt{\be}_1\\
\wt{\be}_2
\end{pmatrix} = 
\begin{pmatrix}
0 & \wt\kappa_1 & \wt\kappa_2\\
-\wt\kappa_1 & 0 & 0\\
-\wt\kappa_2 & 0 & 0
\end{pmatrix}
\begin{pmatrix}
\be_{\rm t}\\
\wt{\be}_1\\
\wt{\be}_2
\end{pmatrix}\,.
\end{equation}
Here for curve (a) we choose the initial normal vector $\wt\be_1^{(a)}(0)$ arbitrarily, while $\wt{\be}^{(b)}_1(0)$ is chosen to satisfy
\begin{align*}
\wt{\be}^{(b)}_1(0) = \inf_{\bv\perp\be_{\rm t}^{(b)}(0),\,\abs{\bv}=1}\abs{\bv-\wt{\be}^{(a)}_1(0)}\,;
\end{align*}
i.e. at the closest pair of points $\X^{(a)}(0)$ and $\X^{(b)}(0)$, we choose the normal vectors to be as close as possible. Given that $\X^{(a)}$ and $\X^{(b)}$ satisfy \eqref{eq:XaXb_close}, the difference in initial normal vectors must then satisfy 
\begin{equation}\label{eq:init_cond_close}
\abs{\wt{\be}^{(a)}_1(0)-\wt{\be}^{(b)}_1(0)}\le\delta\,.
\end{equation} 

Although the pure Bishop frame is not necessarily periodic in $s$, its coefficients $\wt\kappa_1$ and $\wt\kappa_2$ control the coefficients from the periodicized frame \eqref{eq:frame}. 
In particular, using the Bishop frame \eqref{eq:Bishop}, we define the angle $-\pi\le\phi< \pi$ to be
\begin{align*}
\phi = \arccos(\wt{\be}_1(0)\cdot\wt{\be}_1(1))\,,
\end{align*}
i.e. the mismatch from periodicity. 
Take $\kappa_3=-\phi$ and note that for the curves (a) and (b) we have 
\begin{equation}\label{eq:kappa3ab}
\abs{\kappa_3^{(a)}-\kappa_3^{(b)}}\le c\,\norm{\wt{\be}_1^{(a)}-\wt{\be}_1^{(b)}}_{L^\infty}\,.
\end{equation}
Furthermore, given $\kappa_3$ as above, we may define
\begin{align*}
\begin{pmatrix}
\be_{\rm n_1}(s) \\
\be_{\rm n_2}(s)
\end{pmatrix}
&= \begin{pmatrix}
\cos(s\kappa_3) & -\sin(s\kappa_3)\\
\sin(s\kappa_3) & \cos(s\kappa_3)
\end{pmatrix}
\begin{pmatrix}
\wt{\be}_1(s) \\
\wt{\be}_2(s)
\end{pmatrix}\,.
\end{align*}
It may be seen that $(\be_{\rm t}(s),\be_{\rm n_1}(s),\be_{\rm n_2}(s))$ satisfy the ODE \eqref{eq:frame} with
\begin{align*}
\begin{pmatrix}
\kappa_1(s) \\
\kappa_2(s)
\end{pmatrix}
&= \begin{pmatrix}
\cos(s\kappa_3) & -\sin(s\kappa_3)\\
\sin(s\kappa_3) & \cos(s\kappa_3)
\end{pmatrix}
\begin{pmatrix}
\wt\kappa_1(s) \\
\wt\kappa_2(s)
\end{pmatrix}\,.
\end{align*}
For both $\ell=1,2$, we then have 
\begin{equation}\label{eq:kappa12ab}
\begin{aligned}
\norm{\kappa_\ell^{(a)}-\kappa_\ell^{(b)}}_{C^{0,\beta}} &\le c(\kappa_{*,\beta}^{(a)},\kappa_{*,\beta}^{(b)})\abs{\kappa_3^{(a)}-\kappa_3^{(b)}}+c\norm{\wt\kappa_1^{(a)}-\wt\kappa_1^{(b)}}_{C^{0,\beta}}+c\norm{\wt\kappa_2^{(a)}-\wt\kappa_2^{(b)}}_{C^{0,\beta}}\\
&\le c(\kappa_{*,\beta}^{(a)},\kappa_{*,\beta}^{(b)})\left(\norm{\wt\kappa_1^{(a)}-\wt\kappa_1^{(b)}}_{C^{0,\beta}}+\norm{\wt\kappa_2^{(a)}-\wt\kappa_2^{(b)}}_{C^{0,\beta}}\right)\,.
\end{aligned}
\end{equation}
It thus suffices to show the desired bound \eqref{eq:XaXb_C2beta} for the coefficients $\wt\kappa_1$, $\wt\kappa_2$ of the pure Bishop frame. For both curves (a) and (b), we may rewrite the ODEs \eqref{eq:Bishop} for $\wt{\be}_1$ and $\wt{\be}_2$ as 
\begin{equation}\label{eq:Bishop2}
\begin{aligned}
\dot{\wt{\be}}_1 &= -(\wt{\be}_1\cdot\dot{\be}_{\rm t})\,\be_{\rm t}\\
\dot{\wt{\be}}_2 &= -(\wt{\be}_2\cdot\dot{\be}_{\rm t})\,\be_{\rm t}\,,
\end{aligned}
\end{equation}
where we use $\dot{\be}$ to denote $\frac{d}{ds}\be$.
Letting $\bw=\wt{\be}_1^{(a)}-\wt{\be}_1^{(b)}$, we then have
\begin{equation}\label{eq:dot_bw}
\begin{aligned}
\dot\bw &= -(\bw\cdot\dot{\be}_{\rm t}^{(a)})\,\be_{\rm t}^{(a)}
-(\wt{\be}_1^{(b)}\cdot(\dot{\be}_{\rm t}^{(a)}-\dot{\be}_{\rm t}^{(b)}))\,\be_{\rm t}^{(a)}
-(\wt{\be}_1^{(b)}\cdot\dot{\be}_{\rm t}^{(b)})\,(\be_{\rm t}^{(a)}-\be_{\rm t}^{(b)})\,.
\end{aligned}
\end{equation}
In particular, using that $\norm{\be_{\rm t}^{(a)}-\be_{\rm t}^{(b)}}_{C^{1,\beta}}\le \norm{\X^{(a)}-\X^{(b)}}_{C^{2,\beta}}=\delta$, we have 
\begin{align*}
\frac{1}{2}\frac{d}{ds}\abs{\bw}^2 &= c(\kappa_*^{(a)})\abs{\bw}^2+ \delta\abs{\bw}+c(\kappa_*^{(b)})\delta\abs{\bw}\\
&\le  c(\kappa_*^{(a)},\kappa_*^{(b)})\big(\abs{\bw}^2+ \delta^2\big)\,.
\end{align*}
Using Gr\"onwall's inequality and \eqref{eq:init_cond_close}, we then have
\begin{align*}
\abs{\bw}^2 &\le  c(\kappa_*^{(a)},\kappa_*^{(b)})\big(\abs{\bw(0)}^2+ \delta^2\big)
\le c(\kappa_*^{(a)},\kappa_*^{(b)})\,\delta^2\,.
\end{align*}
Returning to the ODE \eqref{eq:dot_bw} for $\bw$, we also obtain
\begin{align*}
\abs{\dot\bw} \le c(\kappa_*^{(a)})\abs{\bw}+ \delta +c(\kappa_*^{(b)})\delta
\le c(\kappa_*^{(a)},\kappa_*^{(b)})\,\delta\,.
\end{align*}
In total, recalling that $\delta=\norm{\X^{(a)}-\X^{(b)}}_{C^{2,\beta}}$ and $\bw=\wt{\be}_1^{(a)}-\wt{\be}_1^{(b)}$, we have
\begin{align*}
\norm{\wt{\be}_1^{(a)}-\wt{\be}_1^{(b)}}_{C^1} \le c(\kappa_*^{(a)},\kappa_*^{(b)})\,\norm{\X^{(a)}-\X^{(b)}}_{C^{2,\beta}}\,.
\end{align*}
We may then use the ODEs \eqref{eq:Bishop}, \eqref{eq:Bishop2} to obtain bounds on the coefficients $\wt\kappa_1$. We have
\begin{align*}
\wt\kappa_1^{(a)}-\wt\kappa_1^{(b)} = (\wt{\be}_1^{(a)}-\wt{\be}_1^{(b)})\cdot\dot\be_{\rm t}^{(a)} + \wt{\be}_1^{(b)}\cdot(\dot\be_{\rm t}^{(a)}-\dot\be_{\rm t}^{(b)})\,,
\end{align*}
and thus
\begin{align*}
\norm{\wt\kappa_1^{(a)}-\wt\kappa_1^{(b)}}_{C^{0,\beta}} \le c(\kappa_{*,\beta}^{(a)},\kappa_{*,\beta}^{(b)})\, \norm{\X^{(a)}-\X^{(b)}}_{C^{2,\beta}}\,.
\end{align*}
Using the ODE \eqref{eq:Bishop2} for $\wt\be_2$, the same bound may be shown for $\wt\kappa_2^{(a)}-\wt\kappa_2^{(b)}$. Then by \eqref{eq:kappa3ab} and \eqref{eq:kappa12ab}, we obtain Lemma \ref{lem:XaXb_C2beta}.
\hfill \qedsymbol


\section{Calculation of straight single and double layer symbols}\label{app:symbol_calc}
Here we derive the expressions stated in Lemma \ref{lem:straight_SD} for the single and double layer operators $\overline{\mc{S}}$ and $\overline{\mc{D}}$ about the straight filament $\mc{C}_\epsilon$.

\begin{proof}[Proof of Lemma \ref{lem:straight_SD}]
For the straight filament, periodicity of the operators $\overline{\mc{S}}$ and $\overline{\mc{D}}$ is enforced by considering only densities $\bm{\varphi}(s,\theta)$ which are 1-periodic in $s$. For such a density $\bm{\varphi}$, we write $\overline{\mc{S}}$ and $\overline{\mc{D}}$ as
\begin{align*}
\overline{\mc{S}}[\bm{\varphi}](s,\theta) &= \int_{-\infty}^{\infty}\int_0^{2\pi}\overline{\mc{G}}(s,\theta,s',\theta')\,\bm{\varphi}(s',\theta')\,\epsilon\, d\theta'ds' \\
\overline{\mc{D}}[\bm{\varphi}](s,\theta) &= \int_{-\infty}^{\infty}\int_0^{2\pi}\overline{\bm{K}}_{\mc D}(s,\theta,s',\theta')\,\bm{\varphi}(s',\theta')\,\epsilon \,d\theta'ds'\,,
\end{align*}
where the kernels $\overline{\mc{G}}$ and $\overline{\bm{K}}_{\mc D}$ will be written in more detail below. Recalling the general form of the Stokeslet $\mc{G}$ \eqref{eq:stokeslet} and double layer kernel $\bm{K}_{\mc D}$ \eqref{eq:stresslet}, we note that about $\mc{C}_\epsilon$ we may write 
\begin{equation}\label{eq:straightx_diff}
\begin{aligned}
\overline{\bx}-\overline{\bx}' &= \textstyle(s-s')\be_z+2\epsilon\sin(\frac{\theta-\theta'}{2})\be_\theta(\frac{\theta+\theta'}{2})  \\
(\overline{\bx}-\overline{\bx}')\cdot\be_r(\theta') &= \textstyle -2\epsilon\sin^2(\frac{\theta-\theta'}{2}) \\
\textstyle \be_\theta(\frac{\theta+\theta'}{2})&=\textstyle -\sin(\frac{\theta+\theta'}{2})\be_x+\cos(\frac{\theta+\theta'}{2})\be_y\,.
\end{aligned}
\end{equation}
Here again we use overline notation to emphasize that $\overline{\bx}$ lies along the straight filament $\mc{C}_\epsilon$.
Given a continuous density $\bm{\varphi}$ along $\mc{C}_\epsilon$, we may use \eqref{eq:straightx_diff} to write out the integrand \eqref{eq:stokeslet} of the single layer operator about $\mc{C}_\epsilon$ as 
\begin{equation}\label{eq:straight_SL_kernel}
\begin{aligned}
8\pi(\overline{\mc{G}}\bm{\varphi})\cdot\be_z &= \textstyle \frac{(\bm{\varphi}\cdot\be_z)}{\sqrt{(s-s')^2+4\epsilon^2\sin^2(\frac{\theta-\theta'}{2})}}
+ \frac{(s-s')^2(\bm{\varphi}\cdot\be_z)}{\sqrt{(s-s')^2+4\epsilon^2\sin^2(\frac{\theta-\theta'}{2})}^{\,3}}\\
&\qquad \textstyle +\frac{2\epsilon(s-s')\sin(\frac{\theta-\theta'}{2})\big(-\sin(\frac{\theta+\theta'}{2})(\bm{\varphi}\cdot\be_x)+\cos(\frac{\theta+\theta'}{2})(\bm{\varphi}\cdot\be_y)\big)}{\sqrt{(s-s')^2+4\epsilon^2\sin^2(\frac{\theta-\theta'}{2})}^{\,3}} \\
8\pi(\overline{\mc{G}}\bm{\varphi})\cdot\be_x &= \textstyle \frac{(\bm{\varphi}\cdot\be_x)}{\sqrt{(s-s')^2+4\epsilon^2\sin^2(\frac{\theta-\theta'}{2})}}
 -\frac{2\epsilon(s-s')\sin(\frac{\theta-\theta'}{2})\sin(\frac{\theta+\theta'}{2})(\bm{\varphi}\cdot\be_z) }{\sqrt{(s-s')^2+4\epsilon^2\sin^2(\frac{\theta-\theta'}{2})}^{\,3}} \\
 &\qquad \textstyle +\frac{4\epsilon^2\sin^2(\frac{\theta-\theta'}{2})\big(\sin^2(\frac{\theta+\theta'}{2})(\bm{\varphi}\cdot\be_x)-\sin(\frac{\theta+\theta'}{2})\cos(\frac{\theta+\theta'}{2})(\bm{\varphi}\cdot\be_y)\big)}{\sqrt{(s-s')^2+4\epsilon^2\sin^2(\frac{\theta-\theta'}{2})}^{\,3}}
 \\
8\pi(\overline{\mc{G}}\bm{\varphi})\cdot\be_y &= \textstyle \frac{(\bm{\varphi}\cdot\be_y)}{\sqrt{(s-s')^2+4\epsilon^2\sin^2(\frac{\theta-\theta'}{2})}}
+\frac{2\epsilon(s-s')\sin(\frac{\theta-\theta'}{2})\cos(\frac{\theta+\theta'}{2})(\bm{\varphi}\cdot\be_z) }{\sqrt{(s-s')^2+4\epsilon^2\sin^2(\frac{\theta-\theta'}{2})}^{\,3}}  \\
&\qquad \textstyle +\frac{4\epsilon^2\sin^2(\frac{\theta-\theta'}{2})\big(\cos^2(\frac{\theta+\theta'}{2})(\bm{\varphi}\cdot\be_y)-\sin(\frac{\theta+\theta'}{2})\cos(\frac{\theta+\theta'}{2})(\bm{\varphi}\cdot\be_x)\big)}{\sqrt{(s-s')^2+4\epsilon^2\sin^2(\frac{\theta-\theta'}{2})}^{\,3}}\,.
\end{aligned}
\end{equation}
Similarly, using \eqref{eq:straightx_diff} in \eqref{eq:stresslet} the integrand of the straight double layer operator may be written 
\begin{equation}\label{eq:straight_DL_kernel}
\begin{aligned}
(\overline{\bm K}_{\mc{D}}\bm{\varphi})\cdot\be_z &=  -\textstyle \frac{3}{\pi}
\frac{\epsilon^2(s-s')\sin^3(\frac{\theta-\theta'}{2})\big(-\sin(\frac{\theta+\theta'}{2})(\bm{\varphi}\cdot\be_x)+\cos(\frac{\theta+\theta'}{2})(\bm{\varphi}\cdot\be_y)\big)}{\sqrt{(s-s')^2+4\epsilon^2\sin^2(\frac{\theta-\theta'}{2})}^{\,5}} \\
&\qquad -\textstyle \frac{3}{4\pi}
\frac{2\epsilon(s-s')^2\sin^2(\frac{\theta-\theta'}{2})(\bm{\varphi}\cdot\be_z)}{\sqrt{(s-s')^2+4\epsilon^2\sin^2(\frac{\theta-\theta'}{2})}^{\,5}} \\
(\overline{\bm K}_{\mc{D}}\bm{\varphi})\cdot\be_x &= -\textstyle \frac{6}{\pi}
 \frac{\epsilon^3\sin^4(\frac{\theta-\theta'}{2})\big(\sin^2(\frac{\theta+\theta'}{2})(\bm{\varphi}\cdot\be_x)-\sin(\frac{\theta+\theta'}{2})\cos(\frac{\theta+\theta'}{2})(\bm{\varphi}\cdot\be_y)\big) }{\sqrt{(s-s')^2+4\epsilon^2\sin^2(\frac{\theta-\theta'}{2})}^{\,5}}
 \\
 &\qquad +\textstyle \frac{3}{\pi}
 \frac{ \epsilon^2(s-s')\sin^3(\frac{\theta-\theta'}{2})\sin(\frac{\theta+\theta'}{2})(\bm{\varphi}\cdot\be_z) }{\sqrt{(s-s')^2+4\epsilon^2\sin^2(\frac{\theta-\theta'}{2})}^{\,5}}
 \\
(\overline{\bm K}_{\mc{D}}\bm{\varphi})\cdot\be_y &= -\textstyle \frac{6}{\pi}
\frac{\epsilon^3\sin^4(\frac{\theta-\theta'}{2})\big(\cos^2(\frac{\theta+\theta'}{2})(\bm{\varphi}\cdot\be_y)-\sin(\frac{\theta+\theta'}{2})\cos(\frac{\theta+\theta'}{2})(\bm{\varphi}\cdot\be_x)\big) }{\sqrt{(s-s')^2+4\epsilon^2\sin^2(\frac{\theta-\theta'}{2})}^{\,5}}\\
&\qquad -\textstyle \frac{3}{\pi}
\frac{\epsilon^2(s-s')\sin^3(\frac{\theta-\theta'}{2})\cos(\frac{\theta+\theta'}{2})(\bm{\varphi}\cdot\be_z) }{\sqrt{(s-s')^2+4\epsilon^2\sin^2(\frac{\theta-\theta'}{2})}^{\,5}}\,.
\end{aligned}
\end{equation}
Before computing anything involving \eqref{eq:straight_SL_kernel} and \eqref{eq:straight_DL_kernel}, we make note of a series of identities involving integrals of the second-kind modified Bessel functions $K_0$ and $K_1$. In particular, the following identities hold: 
\begin{equation}\label{eq:sinKj_IDs}
\begin{aligned}
\int_0^{2\pi}\textstyle K_0(z\sin(\frac{\theta}{2}))\,d\theta &= \textstyle2\pi I_0(\frac{z}{2})K_0(\frac{z}{2}) \\
\int_0^{2\pi}\textstyle \sin(\frac{\theta}{2})\,K_1(z\sin(\frac{\theta}{2}))\,d\theta &= \textstyle \pi \big( I_0(\frac{z}{2})K_1(\frac{z}{2})-I_1(\frac{z}{2})K_0(\frac{z}{2})\big)\\
\int_0^{2\pi}\textstyle \sin^2(\frac{\theta}{2})\,K_0(z\sin(\frac{\theta}{2}))\,d\theta &= \textstyle \pi\big( I_0(\frac{z}{2})K_0(\frac{z}{2})-I_1(\frac{z}{2})K_1(\frac{z}{2})\big) \\
\int_0^{2\pi}\textstyle \sin^3(\frac{\theta}{2})\,K_1(z\sin(\frac{\theta}{2}))\,d\theta &= \textstyle \pi\big(I_0(\frac{z}{2})K_1(\frac{z}{2}) -I_1(\frac{z}{2})K_0(\frac{z}{2}) \big) - \frac{2\pi}{z}I_1(\frac{z}{2})K_1(\frac{z}{2})\\
\int_0^{2\pi}\textstyle \sin^4(\frac{\theta}{2})\,K_0(z\sin(\frac{\theta}{2}))\,d\theta &= \textstyle\pi(I_0(\frac{z}{2})K_0(\frac{z}{2})-I_1(\frac{z}{2})K_1(\frac{z}{2})) + \frac{\pi}{z}(I_0(\frac{z}{2})K_1(\frac{z}{2})-I_1(\frac{z}{2})K_0(\frac{z}{2}))  \\
&\qquad \textstyle -\frac{4\pi}{z^2}I_1(\frac{z}{2})K_1(\frac{z}{2}) \\
\int_0^{2\pi}\textstyle \sin^5(\frac{\theta}{2})\,K_1(z\sin(\frac{\theta}{2}))\,d\theta &= \textstyle \pi(I_0(\frac{z}{2})K_1(\frac{z}{2})-I_1(\frac{z}{2})K_0(\frac{z}{2})) +\frac{\pi}{z}(I_0(\frac{z}{2})K_0(\frac{z}{2})-3I_1(\frac{z}{2})K_1(\frac{z}{2}))  \\
&\qquad \textstyle  +\frac{4\pi}{z^2}(I_0(\frac{z}{2})K_1(\frac{z}{2})-I_1(\frac{z}{2})K_0(\frac{z}{2})) -\frac{16\pi}{z^3}I_1(\frac{z}{2})K_1(\frac{z}{2})\,.
\end{aligned}
\end{equation}
Here $I_0$ and $I_1$ are first-kind modified Bessel functions. The first integral identity is shown in \cite{laplace}; the rest may be seen by differentiating the first identity with respect to $z$, using that $K_0'(z)=-K_1(z)$ and $K_1'(z)=-K_0(z)-\frac{K_1(z)}{z}$.

Equipped with the expressions \eqref{eq:straight_SL_kernel} and \eqref{eq:straight_DL_kernel} and the identities \eqref{eq:sinKj_IDs}, we may show Lemma \ref{lem:straight_SD}. 
We begin with the expression \eqref{eq:DL_tangential}, i.e. how the double layer acts on $\theta$-independent functions $\bv(s) = e^{2\pi iks}\be_z$ tangent to the filament centerline. Using \eqref{eq:straight_DL_kernel} and \eqref{eq:sinKj_IDs}, we may calculate 
\begin{align*}
\overline{\mc{D}}[e^{2\pi iks}\be_z] \cdot\be_z &= -\frac{3}{2\pi}\int_{-\infty}^\infty\int_0^{2\pi}\textstyle \frac{\epsilon\sin^2(\frac{\theta'}{2})(s')^2}{\sqrt{(s')^2+4\epsilon^2\sin^2(\frac{\theta'}{2})}^{\,5}}\, e^{2\pi ik(s-s')} \,\epsilon\,d\theta'ds'\\
&= -\int_0^{2\pi}\bigg(\textstyle \epsilon\abs{k}\sin(\frac{\theta'}{2})K_1\big(4\pi\epsilon\abs{k}\sin(\frac{\theta'}{2})\big) - 4\pi\epsilon^2\abs{k}^2\sin^2(\frac{\theta'}{2}) K_0\big(4\pi\epsilon\abs{k}\sin(\frac{\theta'}{2})\big)\bigg)\,d\theta'\, e^{2\pi iks} \\
&= -\left(-4\pi^2\epsilon^2\abs{k}^2 \big( I_0K_0-I_1K_1\big) +\pi\epsilon\abs{k}\big(I_0K_1-I_1K_0\big) \right)\, e^{2\pi iks} \,, \\
\overline{\mc{D}}[e^{2\pi iks}\be_z] \cdot\be_r &= -\frac{3}{\pi}\int_{-\infty}^\infty\int_0^{2\pi} \textstyle \frac{\epsilon^2\,s'\sin^4(\frac{\theta'}{2}) }{\sqrt{(s')^2+4\epsilon^2\sin^2(\frac{\theta'}{2})}^{\,5}}\,e^{2\pi ik(s-s')}\epsilon \,d\theta'ds' \\
&=  \int_0^{2\pi} \textstyle4\pi i \epsilon^2k^2\sin^3(\frac{\theta'}{2}) K_1(4\pi\epsilon\abs{k}\sin(\frac{\theta'}{2}))\,d\theta'\,  e^{2\pi iks}  \\
&= -i\left(4\pi^2\epsilon^2k^2(I_1K_0-I_0K_1)+2\pi\epsilon\abs{k}I_1K_1\right)\,e^{2\pi iks} \,.
\end{align*}
Here each $I_j$, $K_j$ in the final expressions is evaluated at $2\pi\epsilon\abs{k}$. We also note that $\overline{\mc{D}}[e^{2\pi iks}\be_z] \cdot\be_\theta=0$.

To calculate the expression \eqref{eq:Sinv_tangential1} for how $\overline{\mc{S}}^{-1}$ acts on zero modes in $\theta$, we first determine how $\overline{\mc{S}}$ acts on functions of the form $e^{2\pi iks}\be_z$ and $e^{2\pi iks}\be_r$. Using \eqref{eq:straight_SL_kernel} and \eqref{eq:sinKj_IDs}, for the $\be_z$ direction we may calculate 
\begin{align*}
\overline{\mc{S}}[e^{2\pi iks}\be_z] \cdot\be_r &= \frac{1}{8\pi}\int_{-\infty}^\infty\int_0^{2\pi} \textstyle\frac{2\epsilon\,s'\sin^2(\frac{\theta'}{2})}{\sqrt{(s')^2+4\epsilon^2\sin^2(\frac{\theta'}{2})}^{\,3}}\,e^{2\pi i k(s-s')}\,\epsilon\,d\theta'ds' \\
&=-\int_0^{2\pi}\textstyle i\epsilon^2\abs{k}\sin^2(\frac{\theta'}{2})K_0(4\pi\epsilon\abs{k}\sin(\frac{\theta'}{2}))\,d\theta' \,e^{2\pi i ks}\\
&= -i\pi\epsilon^2\abs{k}(I_0K_0-I_1K_1)\,e^{2\pi i ks} \,,\\
\overline{\mc{S}}[e^{2\pi iks}\be_z] \cdot\be_z 
&=\frac{1}{8\pi}\int_{-\infty}^\infty\int_0^{2\pi}\textstyle\bigg(\frac{2}{\sqrt{(s')^2+4\epsilon^2\sin^2(\frac{\theta'}{2})}}-\frac{4\epsilon^2\sin^2(\frac{\theta'}{2})}{\sqrt{(s')^2+4\epsilon^2\sin^2(\frac{\theta'}{2})}^{\,3}}\bigg)\,e^{2\pi i k(s-s')}\,\epsilon\,d\theta'ds' \\
&= \bigg(\frac{1}{2\pi}\int_0^{2\pi}\textstyle K_0(4\pi\epsilon\abs{k}\sin(\frac{\theta'}{2}))\,\epsilon\,d\theta'
\displaystyle -\int_0^{2\pi} \textstyle\epsilon^2\abs{k}\sin(\frac{\theta'}{2})\,K_1(4\pi\epsilon\abs{k}\sin(\frac{\theta'}{2}))\,d\theta'\bigg)\,e^{2\pi i ks} \\
&= \epsilon \left(I_0K_0 + \pi\epsilon\abs{k}(I_1K_0- I_0K_1)\right)\, e^{2\pi iks}\,.
\end{align*}
For the $\be_r$ direction, we may calculate
\begin{align*}
\overline{\mc{S}}[e^{2\pi iks}\be_r] \cdot\be_r &= \frac{1}{8\pi}\int_{-\infty}^\infty\int_0^{2\pi} \textstyle\bigg( \frac{1-2\sin^2(\frac{\theta'}{2})  }{\sqrt{(s')^2+4\epsilon^2\sin^2(\frac{\theta'}{2})}} - \frac{4\epsilon^2\sin^4(\frac{\theta'}{2})}{\sqrt{(s')^2+4\epsilon^2\sin^2(\frac{\theta'}{2})}^{\,3}}\bigg)\,e^{2\pi i k(s-s')}\,\epsilon\,d\theta'ds'\\
&= \frac{1}{4\pi}\int_0^{2\pi} \textstyle K_0(4\pi\epsilon\abs{k}\sin(\frac{\theta'}{2}))\left(1-2\sin^2(\frac{\theta'}{2})\right)\,\epsilon\,d\theta' \,e^{2\pi iks} \\
& \qquad - \int_0^{2\pi} \textstyle \epsilon^2\abs{k}\sin^3(\frac{\theta'}{2})K_1(4\pi\epsilon\abs{k}\sin(\frac{\theta'}{2})) \,d\theta' \,e^{2\pi iks}\\
&=  \epsilon\left( I_1K_1 + \pi\epsilon\abs{k}(I_1K_0-I_0K_1)\right) e^{2\pi iks}\,, \\
\overline{\mc{S}}[e^{2\pi iks}\be_r] \cdot\be_z &= -\frac{1}{8\pi}\int_{-\infty}^\infty\int_0^{2\pi} \textstyle\frac{2\epsilon\,s'\sin^2(\frac{\theta'}{2})}{\sqrt{(s')^2+4\epsilon^2\sin^2(\frac{\theta'}{2})}^{\,3}} \,e^{2\pi i k(s-s')}\,\epsilon\,d\theta'ds'\\
&= \int_0^{2\pi}\textstyle i\epsilon^2\abs{k}\sin^2(\frac{\theta'}{2})K_0(4\pi\epsilon\abs{k}\sin(\frac{\theta'}{2}))\,d\theta' \,e^{2\pi iks}\\
&= i\pi\epsilon^2\abs{k}(I_0K_0-I_1K_1)\,e^{2\pi iks}\,.
\end{align*}
Each $I_j$ and $K_j$ in the final expressions is again evaluated at $2\pi\epsilon\abs{k}$. We further note that $\overline{\mc{S}}[e^{2\pi iks}\be_z] \cdot\be_\theta=\overline{\mc{S}}[e^{2\pi iks}\be_r] \cdot\be_\theta=0$.

In total, we obtain the following expression for the action of $\overline{\mc{S}}$ on the relevant zero modes in $\theta$:
\begin{equation}\label{eq:SL_tangential}
\overline{\mc{S}} \begin{pmatrix}
e^{2\pi iks} \be_z \\
e^{2\pi iks} \be_r
\end{pmatrix}= \epsilon\begin{pmatrix}
I_0K_0+\pi\epsilon\abs{k}(I_1K_0-I_0K_1)  & i\pi\epsilon\abs{k}(I_0K_0-I_1K_1)\\
-i\pi\epsilon\abs{k}(I_0K_0-I_1K_1) &  I_1K_1+ \pi\epsilon\abs{k}(I_1K_0-I_0K_1)
\end{pmatrix}\begin{pmatrix}
e^{2\pi iks} \be_z \\
e^{2\pi iks} \be_r
\end{pmatrix}\,.
\end{equation}
Here $I_j=I_j(2\pi\epsilon\abs{k})$ and $K_j=K_j(2\pi\epsilon\abs{k})$. Inverting the matrix \eqref{eq:SL_tangential}, we obtain the expression \eqref{eq:Sinv_tangential1} and \eqref{eq:Sinv_tangential2} for the multiplier $\bm{M}_{\rm S,t}^{-1}(k)$ in the tangential direction along the filament.

We next consider $\theta$-independent data in directions normal to the filament centerline. We consider data of the form $e^{2\pi iks}\be_x$ and note that, by symmetry, data of the form $e^{2\pi iks}\be_y$ behaves analogously. Note that, using the identities $\be_x=\cos\theta\be_r-\sin\theta\be_\theta$ and $\be_y=\sin\theta\be_r+\cos\theta\be_\theta$, functions of this form actually correspond to one-modes in $\theta$. 

We begin by calculating the expressions \eqref{eq:DL_normal} for the double layer multipliers.
Using the form \eqref{eq:straight_DL_kernel} of the integrand of $\overline{\mc{D}}$ as well as the identities \eqref{eq:sinKj_IDs}, we may calculate
\begin{align*}
\overline{\mc{D}}[e^{2\pi iks}\be_x]\cdot\be_z &= -\frac{3}{4\pi}\int_{-\infty}^\infty\int_0^{2\pi}\textstyle\frac{4\epsilon^2\,s'\sin^4(\frac{\theta'}{2})}{\sqrt{(s')^2+4\epsilon^2\sin^2(\frac{\theta'}{2})}^{\,5}}\, e^{2\pi i k(s-s')}\,\epsilon\,d\theta'ds'\, \cos\theta \\
&= \int_0^{2\pi} \textstyle i 4\pi \epsilon^2k^2\sin^3(\frac{\theta'}{2})K_1(4\pi\epsilon\abs{k}\sin(\frac{\theta'}{2}))\,d\theta'\, e^{2\pi iks}\cos\theta \\
&= i\left(4\pi^2\epsilon^2k^2(I_0K_1-I_1K_0) - 2\pi\epsilon\abs{k}I_1K_1\right) \, e^{2\pi iks}\cos\theta\,, \\
\overline{\mc{D}}[e^{2\pi iks}\be_x]\cdot\be_x &= -\frac{3}{4\pi}\int_{-\infty}^\infty\int_0^{2\pi}\textstyle\frac{8\epsilon^3\sin^4(\frac{\theta'}{2})\big(\sin^2\theta-\sin^2(\frac{\theta'}{2}) (\sin^2\theta-\cos^2\theta)\big)}{\sqrt{(s')^2+4\epsilon^2\sin^2(\frac{\theta'}{2})}^{\,5}}\, e^{2\pi i k(s-s')}\,\epsilon\,d\theta'ds'\\
&= -\int_0^{2\pi}\textstyle 4\pi \epsilon^2k^2\sin^2(\frac{\theta'}{2}) K_2(4\pi\epsilon\abs{k}\sin(\frac{\theta'}{2}))\,d\theta'\, e^{2\pi iks}\sin^2\theta \\
&\qquad + \int_0^{2\pi}\textstyle 4\pi \epsilon^2k^2\sin^4(\frac{\theta'}{2}) K_2(4\pi\epsilon\abs{k}\sin(\frac{\theta'}{2}))\,d\theta'\,e^{2\pi iks}(\sin^2\theta-\cos^2\theta)\\
&= -\left(\pi\epsilon\abs{k}(I_1K_0-I_0K_1)+2I_1K_1\right) e^{2\pi iks}\sin^2\theta \\
&\qquad  + \left(4\pi^2\epsilon^2\abs{k}^2(I_1K_1-I_0K_0)+\textstyle 3\pi\epsilon\abs{k}(I_1K_0-I_0K_1) + 2I_1K_1 \right)e^{2\pi iks}\cos^2\theta \,.
\end{align*}
Here we have also used the Bessel function identity $K_2(z)=K_0(z)+2\frac{K_1(z)}{z}$. As before, in the final expressions, each $I_j$ and $K_j$ is evaluated at $2\pi\epsilon\abs{k}$. Altogether, the double layer operator $\overline{\mc{D}}$ can be seen to act on $\theta$-independent functions in directions normal to the filament centerline via 
\begin{equation}\label{eq:Dform}
\begin{aligned}
\overline{\mc{D}}[e^{2\pi iks}\be_x] &= \overline{\mc{D}}[e^{2\pi iks}\big(\cos\theta\be_r - \sin\theta\be_\theta\big)] 
= \left( Q_N\cos\theta\be_r - Q_O\sin\theta\be_\theta + iQ_P\cos\theta\be_z\right)e^{2\pi iks}\,, \\
Q_N &= 4\pi^2\epsilon^2\abs{k}^2(I_1K_1-I_0K_0)+\textstyle 3\pi\epsilon\abs{k}(I_1K_0-I_0K_1) + 2I_1K_1  \\
Q_O &= -\pi\epsilon\abs{k}(I_1K_0-I_0K_1)-2I_1K_1\\
Q_P &=4\pi^2\epsilon^2k^2(I_0K_1-I_1K_0) - 2\pi\epsilon\abs{k}I_1K_1 \,.
\end{aligned}
\end{equation}

We next turn to the single layer operator. Given the form of the expressions arising in the double layer operator, we will need to calculate a full matrix-valued multiplier for $\overline{\mc{S}}$ in order to determine how $\overline{\mc{S}}^{-1}$ acts on functions of the form \eqref{eq:Dform}. Recalling the series of identities \eqref{eq:sinKj_IDs} as well as the form \eqref{eq:straight_SL_kernel} of the integrand of $\overline{\mc{S}}$, we begin by calculating  
\begin{align*}
\overline{\mc{S}}[ &e^{2\pi i k s}\be_x]\cdot\be_x =\overline{\mc{S}}[ e^{2\pi i k s}\cos\theta\be_r]\cdot\be_x - \overline{\mc{S}}[ e^{2\pi i k s}\sin\theta\be_\theta]\cdot\be_x\\
&= \frac{1}{8\pi}\int_{-\infty}^{\infty}\int_0^{2\pi}\bigg(\textstyle\frac{1}{\sqrt{(s')^2+4\epsilon^2\sin^2(\frac{\theta'}{2})}} + \frac{4\epsilon^2\sin^2(\frac{\theta'}{2}) \big(\sin^2\theta - \sin^2(\frac{\theta'}{2})(\sin^2\theta-\cos^2\theta)\big)}{\sqrt{(s')^2+4\epsilon^2\sin^2(\frac{\theta'}{2})}^{\,3}}\bigg)\,e^{2\pi i k (s-s')} \,\epsilon\,d\theta'ds' \\
&= \frac{1}{4\pi}\int_0^{2\pi} \textstyle K_0(4\pi\epsilon\abs{k}\sin(\frac{\theta'}{2})) \,\epsilon\,d\theta'ds'\, e^{2\pi iks}
\displaystyle
+ \int_0^{2\pi} \textstyle \epsilon^2\abs{k}\sin(\frac{\theta'}{2})K_1(4\pi\epsilon\abs{k}\sin(\frac{\theta'}{2}))\,d\theta'\, e^{2\pi iks}\sin^2\theta \\
&\qquad \displaystyle
+ \int_0^{2\pi} \textstyle \epsilon^2\abs{k}\sin^3(\frac{\theta'}{2})K_1(4\pi\epsilon\abs{k}\sin(\frac{\theta'}{2}))\,d\theta'\, e^{2\pi iks}(\cos^2\theta -\sin^2\theta) \\
&= \frac{\epsilon}{2}(I_0K_0+I_1K_1)\, e^{2\pi iks}\sin^2\theta + \epsilon\left(\pi\epsilon\abs{k}(I_0K_1-I_1K_0) +\frac{1}{2}(I_0K_0-I_1K_1) \right)\, e^{2\pi iks}\cos^2\theta \,.
\end{align*}
As usual, $I_j=I_j(2\pi\epsilon\abs{k})$ and $K_j=K_j(2\pi\epsilon\abs{k})$ in the final expressions. We may also calculate
\begin{align*}
\overline{\mc{S}}[&e^{2\pi iks}\cos\theta\be_r]\cdot\be_x\\
&= \frac{1}{8\pi}\int_{-\infty}^\infty\int_0^{2\pi} \textstyle \bigg( \frac{8\epsilon^2(\sin^4(\frac{\theta'}{2})-\sin^6(\frac{\theta'}{2}))}{\sqrt{(s')^2+4\epsilon^2\sin^2(\frac{\theta'}{2})}^{\,3}} + \frac{4\sin^2(\frac{\theta'}{2})-4\sin^4(\frac{\theta'}{2})}{\sqrt{(s')^2+4\epsilon^2\sin^2(\frac{\theta'}{2})}}\bigg)\,e^{2\pi i k(s-s')}\,\epsilon\,d\theta'ds'\, \sin^2\theta \\
&\quad - \frac{1}{8\pi}\int_{-\infty}^\infty\int_0^{2\pi}\textstyle \bigg( \frac{4\epsilon^2\sin^4(\frac{\theta'}{2}) - 8\epsilon^2\sin^6(\frac{\theta'}{2})}{\sqrt{(s')^2+4\epsilon^2\sin^2(\frac{\theta'}{2})}^{\,3}}- \frac{1-4\sin^2(\frac{\theta'}{2})+4\sin^4(\frac{\theta'}{2}) )}{\sqrt{(s')^2+4\epsilon^2\sin^2(\frac{\theta'}{2})}}\bigg)\,e^{2\pi i k(s-s')}\,\epsilon\,d\theta'ds'\, \cos^2\theta\\
&= \int_0^{2\pi}\bigg( \textstyle
\frac{\epsilon}{\pi} (\sin^2(\frac{\theta'}{2})-\sin^4(\frac{\theta'}{2}))K_0(4\pi\epsilon\abs{k}\sin(\frac{\theta'}{2})) \\
&\qquad  + \textstyle
 2\epsilon^2\abs{k}(\sin^3(\frac{\theta'}{2})-\sin^5(\frac{\theta'}{2}))K_1(4\pi\epsilon\abs{k}\sin(\frac{\theta'}{2})) \bigg)\,d\theta'\, e^{2\pi iks}\sin^2\theta \\
 &\quad + \int_0^{2\pi}\bigg( \textstyle
\frac{\epsilon}{4\pi} (1-4\sin^2(\frac{\theta'}{2})+4\sin^4(\frac{\theta'}{2}))K_0(4\pi\epsilon\abs{k}\sin(\frac{\theta'}{2})) \\
&\qquad - \textstyle
 \epsilon^2\abs{k}(\sin^3(\frac{\theta'}{2})-2\sin^5(\frac{\theta'}{2}))K_1(4\pi\epsilon\abs{k}\sin(\frac{\theta'}{2})) \bigg)\,d\theta' \,e^{2\pi iks}\cos^2\theta \\
 &= \textstyle \epsilon\left(\frac{3}{4\pi\epsilon\abs{k}}(I_1K_0-I_0K_1)+\frac{3}{4\pi^2\epsilon^2k^2}I_1K_1 - \frac{1}{2}(I_0K_0-I_1K_1) \right) \,e^{2\pi iks}\sin^2\theta \\
&\quad - \epsilon\left( \textstyle \frac{3}{4\pi\epsilon\abs{k}}(I_1K_0-I_0K_1)+\frac{3}{4\pi^2\epsilon^2k^2}I_1K_1  -\pi\epsilon\abs{k}(I_0K_1-I_1K_0) - (I_0K_0-I_1K_1)  \right) \, e^{2\pi iks}\cos^2\theta\,,
\end{align*}
from which we may then obtain 
\begin{align*}
\overline{\mc{S}}[ e^{2\pi i k s}\sin\theta\be_\theta]\cdot\be_x
&=  \epsilon\left( \textstyle \frac{3}{4\pi\epsilon\abs{k}}(I_1K_0-I_0K_1)+\frac{3}{4\pi^2\epsilon^2k^2}I_1K_1 - I_0K_0 \right) e^{2\pi iks}\,\sin^2\theta  \\
&\quad + \epsilon\left( \textstyle -\frac{3}{4\pi\epsilon\abs{k}}(I_1K_0-I_0K_1)-\frac{3}{4\pi^2\epsilon^2k^2}I_1K_1   + \frac{1}{2}(I_0K_0-I_1K_1)  \right) \, e^{2\pi iks}\cos^2\theta \,.
\end{align*}

In the $\be_z$ direction, we additionally have
\begin{align*}
\overline{\mc{S}}[e^{2\pi iks}\cos\theta\be_r]\cdot\be_z
&= -\frac{1}{8\pi}\int_{-\infty}^\infty\int_0^{2\pi} \textstyle \frac{2\epsilon\,s'(\sin^2(\frac{\theta'}{2})-2\sin^4(\frac{\theta'}{2}))}{\sqrt{(s')^2+4\epsilon^2\sin^2(\frac{\theta'}{2})}^{\,3}} \,e^{2\pi i k(s-s')}\,\epsilon\,d\theta'ds'\, \cos\theta \\
&= -\int_0^{2\pi}\textstyle i\epsilon^2k\big(-\sin^2(\frac{\theta'}{2})+2\sin^4(\frac{\theta'}{2}) \big) K_0(4\pi\epsilon\abs{k}\sin(\frac{\theta'}{2}))\, d\theta'\, e^{2\pi i ks}\cos\theta\\
&=-\epsilon i\left(\textstyle\pi\epsilon\abs{k}(I_1K_1-I_0K_0) + \frac{1}{2}(I_1K_0-I_0K_1) + \frac{1}{2\pi\epsilon\abs{k}}I_1K_1 \right)e^{2\pi i ks}\cos\theta\,, \\
\overline{\mc{S}}[ e^{2\pi i k s}\sin\theta\be_\theta]\cdot\be_z
&= -\frac{1}{8\pi}\int_{-\infty}^\infty\int_0^{2\pi} \textstyle \frac{4\epsilon\,s'(\sin^2(\frac{\theta'}{2})-\sin^4(\frac{\theta'}{2}))}{\sqrt{(s')^2+4\epsilon^2\sin^2(\frac{\theta'}{2})}^{\,3}} \,e^{2\pi i k(s-s')}\,\epsilon\,d\theta'ds'\, \cos\theta \\
&= -\int_0^{2\pi}\textstyle i2\epsilon^2k\big(-\sin^2(\frac{\theta'}{2})+\sin^4(\frac{\theta'}{2}) \big) K_0(4\pi\epsilon\abs{k}\sin(\frac{\theta'}{2}))\, d\theta'\, e^{2\pi i ks}\cos\theta\\
&=-\epsilon i\left(\textstyle \frac{1}{2}(I_1K_0-I_0K_1) + \frac{1}{2\pi\epsilon\abs{k}}I_1K_1 \right)\, e^{2\pi i ks}\cos\theta\,, \\
\overline{S}[ e^{2\pi i k s}\cos\theta\be_z]\cdot\be_z
&=  \frac{1}{8\pi} \int_{-\infty}^\infty\int_0^{2\pi}\textstyle\bigg(\frac{2(1-2\sin^2(\frac{\theta'}{2}))}{\sqrt{(s')^2+4\epsilon^2\sin^2(\frac{\theta'}{2})}}
- \frac{4\epsilon^2(\sin^2(\frac{\theta'}{2})-2\sin^4(\frac{\theta'}{2}))}{\sqrt{(s')^2+4\epsilon^2\sin^2(\frac{\theta'}{2})}^{\,3}} \bigg)\, \,e^{2\pi i k(s-s')}\,\epsilon\,d\theta'ds'\, \cos\theta  \\
&= \frac{1}{2\pi}\int_0^{2\pi} \textstyle (1-2\sin^2(\frac{\theta'}{2}))K_0(4\pi\epsilon\abs{k}\sin(\frac{\theta'}{2}))\, \epsilon\,d\theta' \,e^{2\pi i ks}\cos\theta \\
&\qquad - \int_0^{2\pi} \textstyle \epsilon^2\abs{k}(\sin(\frac{\theta'}{2})-2\sin^3(\frac{\theta'}{2}))K_1(4\pi\epsilon\abs{k}\sin(\frac{\theta'}{2}))\, d\theta'\, e^{2\pi i ks}\cos\theta\\
&= \epsilon\left( \pi\epsilon\abs{k}(I_0K_1-I_1K_0) \right) \,e^{2\pi i ks}\cos\theta\,.
\end{align*}

To summarize, we may write the action of $\overline{\mc{S}}$ on the vectors $(\cos\theta \be_r,\sin\theta \be_\theta,\cos\theta \be_z)$ associated with the $\be_x$ direction as
\begin{equation}\label{eq:straight_SL_n}
\begin{aligned}
\overline{\mc{S}} &\begin{pmatrix}
e^{2\pi iks}\cos\theta \be_r \\
e^{2\pi iks}\sin\theta \be_\theta \\
e^{2\pi iks}\cos\theta\be_z
\end{pmatrix}
= \epsilon\begin{pmatrix}
Q_H & Q_I & -iQ_J \\
Q_I & Q_K & -iQ_L \\
iQ_J & iQ_L & Q_M
\end{pmatrix} \begin{pmatrix}
e^{2\pi iks}\cos\theta \be_r \\
e^{2\pi iks}\sin\theta \be_\theta \\
e^{2\pi iks}\cos\theta\be_z
\end{pmatrix}\,,\\
Q_H &=  \textstyle -\frac{3}{4\pi\epsilon\abs{k}}(I_1K_0-I_0K_1)-\frac{3}{4\pi^2\epsilon^2k^2}I_1K_1  +\pi\epsilon\abs{k}(I_0K_1-I_1K_0) + (I_0K_0-I_1K_1) \\
Q_I &= \textstyle -\frac{3}{4\pi\epsilon\abs{k}}(I_1K_0-I_0K_1)-\frac{3}{4\pi^2\epsilon^2k^2}I_1K_1 + \frac{1}{2}(I_0K_0-I_1K_1)  \\
Q_J &= \textstyle \pi\epsilon\abs{k}(I_1K_1-I_0K_0) + \frac{1}{2}(I_1K_0-I_0K_1) + \frac{1}{2\pi\epsilon\abs{k}}I_1K_1  \\
Q_K &= \textstyle -\bigg(\frac{3}{4\pi\epsilon\abs{k}}(I_1K_0-I_0K_1)+\frac{3}{4\pi^2\epsilon^2k^2}I_1K_1 - I_0K_0 \bigg)\\
Q_L &= \textstyle \frac{1}{2}(I_1K_0-I_0K_1) + \frac{1}{2\pi\epsilon\abs{k}}I_1K_1 \\
Q_M &= \pi\epsilon\abs{k}(I_0K_1-I_1K_0)\,,
\end{aligned}
\end{equation}
with each $I_j$, $K_j$ evaluated at $2\pi\epsilon\abs{k}$. Inverting the matrix \eqref{eq:straight_SL_n}, we arrive at the expression \eqref{eq:Sinv_normal2}, \eqref{eq:Sinv_components} for $\bm{M}_{\rm S,n}^{-1}(k)$ stated in Lemma \ref{lem:straight_SD}. The same expression for the $y$-direction follows by symmetry. 
\end{proof}

\section{Proof of Proposition \ref{prop:Phi_comm}: commutator estimate}\label{app:Phi_comm}
Here we show the commutator estimate \eqref{eq:Phi_comm_est} for the multiplier $\overline{\mc{L}}_\epsilon\p_s^2$. We will rely on the tools of section \ref{subsec:besov}. Before proving Proposition \ref{prop:Phi_comm}, we show the following auxiliary lemma. 

\begin{lemma}[Order one commutator estimate]\label{lem:symbol_comm}
Let $\wh{\bm{M}}(\xi)$, $\xi\in \R$, be a matrix-valued $W^{2,\infty}_{\rm loc}$ function satisfying 
\begin{equation}\label{eq:M_conditions}
\abs{\phi_j(\xi)\wh{\bm{M}}(\xi)} \le b_0 \,2^{j}\,, \qquad 
\abs{\p_\xi^2(\phi_j(\xi)\wh{\bm{M}}(\xi))} \le b_1 \,2^{-j}
\end{equation}
for $\phi_j(\xi)$ as in \eqref{eq:phij_def}.
Given matrix-valued $\bm{A}\in C^{n+1,\alpha}(\R)$ and vector-valued $\bm{g}\in C^{n,\alpha}(\R)$, $0\le n\in \Z$, $0<\alpha<1$, the commutator $T_{\wh{\bm{M}}} (\bm{A}\bm{g}) - \bm{A}(T_{\wh{\bm{M}}}\bm{g})$ satisfies 
\begin{equation}\label{eq:symbol_comm}
\norm{T_{\wh{\bm{M}}} (\bm{A}\bm{g}) - \bm{A}(T_{\wh{\bm{M}}}\bm{g})}_{C^{n,\alpha}} 
\le c\,(\sqrt{b_0b_1}+b_1)\norm{\bm{A}}_{C^{n+1,\alpha}}\norm{\bm{g}}_{C^{n,\alpha}} \,.
\end{equation}
\end{lemma}

\begin{proof}
Following \cite[Chapter 2.8]{bahouri2011fourier}, we begin by defining a decomposition of the product $\bm{A}\bm{g}$ into three pieces capturing interactions among different modes in the Littlewood-Paley decomposition \eqref{eq:littlewoodp}. We write 
\begin{equation}\label{eq:mode_decomp}
\bm{A}\bm{g} = \mc{T}_{LH}(\bm{A},\bm{g}) + \mc{T}_{HL}(\bm{A},\bm{g}) + \mc{T}_{HH}(\bm{A},\bm{g})\,,
\end{equation} 
where the low-high, high-low, and high-high interactions, respectively, are given by 
\begin{equation}\label{eq:T_op_defs}
\begin{aligned}
\mc{T}_{LH}(\bm{A},\bm{g}) &= \sum_{j\ge 1} P_j\bm{A} \,P_{\le j-2}\bm{g} \\
\mc{T}_{HL}(\bm{A},\bm{g}) &= \sum_{j\ge 1} P_{\le j-2}\bm{A} \,P_j\bm{g} \\
\mc{T}_{HH}(\bm{A},\bm{g}) &= P_{\le 0}\bm{A}\,P_{\le 0}\bm{g}+ \sum_{\abs{i-j}\le 1,\, i+j\ge 1} P_j\bm{A} \,P_i\bm{g} \,.
\end{aligned}
\end{equation}
Using the definitions \eqref{eq:T_op_defs}, we decompose the commutator expression \eqref{eq:symbol_comm} into three parts: 
\begin{equation}\label{eq:comm_decomp}
\begin{aligned}
T_{\wh{\bm{M}}} (\bm{A}\bm{g}) - \bm{A}(T_{\wh{\bm{M}}}\bm{g}) &= J_{LH}+J_{HL}+J_{HH}\,,\\
J_{LH} &= T_{\wh{\bm{M}}} \mc{T}_{LH}(\bm{A},\bm{g}) - \mc{T}_{LH}(\bm{A},T_{\wh{\bm{M}}}\bm{g})  \\
J_{HL}&= T_{\wh{\bm{M}}} \mc{T}_{HL}(\bm{A},\bm{g}) - \mc{T}_{HL}(\bm{A},T_{\wh{\bm{M}}}\bm{g})  \\
J_{HH} &= T_{\wh{\bm{M}}} \mc{T}_{HH}(\bm{A},\bm{g}) - \mc{T}_{HH}(\bm{A},T_{\wh{\bm{M}}}\bm{g}) \,.
\end{aligned}
\end{equation}

We first compute the low-high interactions $J_{LH}$. For $k\ge 1$, we have 
\begin{align*}
P_k J_{LH} &= P_k\bigg(T_{\wh{\bm{M}}}\sum_{j=k-2}^{k+2}P_{\le j-2}\bm{A}\,P_j\bm{g}
-\sum_{j=k-2}^{k+2}P_{\le j-2}\bm{A}\,P_j(T_{\wh{\bm{M}}}\bm{g})\bigg) \\
&= P_k\bigg(P_{k-5\le \ell\le k+5}T_{\wh{\bm{M}}}\sum_{j=k-2}^{k+2}P_{\le j-2}\bm{A}\,P_j\bm{g}
-\sum_{j=k-2}^{k+2}P_{\le j-2}\bm{A}\, P_{k-5\le \ell\le k+5}T_{\wh{\bm{M}}} P_j\bm{g}\bigg)\,.
\end{align*}
Let $\bm{M}_k$ denote the physical space kernel corresponding to the operator $P_{k-5\le j\le k+5}T_{\wh{\bm{M}}}$. For each $j$, we may write
\begin{align*}
&P_{k-5\le \ell\le k+5}T_{\wh{\bm{M}}}\,P_{\le j-2}\bm{A}\,P_j\bm{g}
-P_{\le j-2}\bm{A}\,P_{k-5\le \ell\le k+5}T_{\wh{\bm{M}}} \,P_j\bm{g}\\
&\qquad = \int_\R\bigg(\bm{M}_k(s-s')\,P_{\le j-2}\bm{A}(s')\,P_j\bm{g}(s') 
- P_{\le j-2}\bm{A}(s)\,\bm{M}_k(s-s')\,P_j\bm{g}(s') \bigg)\,ds'\,.
\end{align*}
We may then estimate 
\begin{equation}\label{eq:JLH_integral_est}
\begin{aligned}
&\abs{\int_\R \bm{M}_k(s-s')\,\big(P_{\le j-2}\bm{A}(s')- P_{\le j-2}\bm{A}(s)\big) \,P_j\bm{g}(s') \,ds'} \\
&\qquad \le c\int_\R\abs{\bm{M}_k(s-s')}\abs{s-s'}\norm{\p_s(P_{\le j-2}\bm{A})}_{L^\infty}\norm{P_j\bm{g}}_{L^\infty} \,ds' \\
&\qquad \le c\,\norm{\bm{M}_k\,\abs{s}\,}_{L^1}\, 2^j\norm{P_{\le j-2}\bm{A}}_{L^\infty}\norm{P_j\bm{g}}_{L^\infty}\,.
\end{aligned}
\end{equation}
We note that \eqref{eq:M_conditions} implies that
\begin{align*}
\int_\R \abs{\bm{M}_k(s)}\abs{s}\,ds &\le R_k\int_{\abs{s}\le R_k}\abs{\bm{M}_k}\,ds + \int_{\abs{s}>R_k}\abs{\bm{M}_k}\abs{s}\,ds \\
&\le R_k\int_{\abs{s}\le R_k}\abs{\bm{M}_k}\,ds + \bigg(\int_{\abs{s}>R_k}\abs{\bm{M}_k}^2\abs{s}^4\,ds\bigg)^{1/2}\bigg(\int_{\abs{s}>R_k}\abs{s}^{-2}\,ds\bigg)^{1/2}\\
&\le R_k\norm{\bm{M}_k}_{L^1(\R)} + c\,R_k^{-1/2}\norm{\p_\xi^2\wh{\bm{M}}_k}_{L^2(\R)}
\le c\, R_k \,2^k\sqrt{b_0b_1} + c\, R_k^{-1/2}\,2^{-k/2}b_1\,,
\end{align*}
where we have used Lemma \ref{lem:besov}. Choosing $R_k\sim 2^{-k}$, we have
\begin{align*}
\int_\R \abs{\bm{M}_k(s)}\abs{s}\,ds \le c\,(\sqrt{b_0b_1}+b_1)\,,
\end{align*}
so, in particular,
\begin{align*}
\abs{\int_\R \bm{M}_k(s-s')\,\big(P_{\le j-2}\bm{A}(s')- P_{\le j-2}\bm{A}(s)\big) \,P_j\bm{g}(s') \,ds'} 
\le c\,(\sqrt{b_0b_1}+b_1)\,2^j\norm{P_{\le j-2}\bm{A}}_{L^\infty}\norm{P_j\bm{g}}_{L^\infty}\,.
\end{align*}
Summing in $j$ and applying $P_k$, $k\ge1$, we thus obtain  
\begin{align*}
2^{k(n+\alpha)}\norm{P_k J_{LH}}_{L^\infty} &\le c\,(\sqrt{b_0b_1}+b_1)\,2^{k(n+\alpha)}\sum_{j=k-2}^{k+2} 2^j\norm{P_{\le j-2}\bm{A}}_{L^\infty}\norm{P_j\bm{g}}_{L^\infty} \\
&\le c\,(\sqrt{b_0b_1}+b_1)\,\norm{\bm{A}}_{C^1}\norm{\bm{g}}_{C^{n,\alpha}}\,.
\end{align*}
For $P_{\le 0}J_{LH}$, we have $\norm{P_{\le 0} J_{LH}}_{L^\infty}\le c\,(\sqrt{b_0b_1}+b_1)\,\norm{\bm{A}}_{L^\infty}\norm{\bm{g}}_{L^\infty}$. 

We next consider the high-low interactions $J_{HL}$ in \eqref{eq:comm_decomp}. Similar to before, for $k\ge 1$, we have
\begin{align*}
P_k J_{HL} &= P_k\bigg(P_{k-5\le \ell\le k+5}T_{\wh{\bm{M}}}\sum_{j=k-2}^{k+2}P_j\bm{A}P_{\le j-2}\bm{g}
-\sum_{j=k-2}^{k+2}P_j\bm{A}P_{k-5\le \ell\le k+5}T_m P_{\le j-2}\bm{g}\bigg)\,.
\end{align*}
For each $j$, the analogous integral to \eqref{eq:JLH_integral_est} is then
\begin{align*}
&\abs{\int_\R\bigg(\bm{M}_k(s-s')\,P_j\bm{A}(s')\,P_{\le j-2}\bm{g}(s') 
- P_j\bm{A}(s)\,\bm{M}_k(s-s')\,P_{\le j-2}\bm{g}(s') \bigg)\,ds'} \\
&\qquad \le 2\norm{\bm{M}_k}_{L^1}\norm{P_j\bm{A}}_{L^\infty}\norm{P_{\le j-2}\bm{g}}_{L^\infty} 
\le c\,\sqrt{b_0b_1}\, 2^k\norm{P_j\bm{A}}_{L^\infty}\norm{P_{\le j-2}\bm{g}}_{L^\infty}\,,
\end{align*}
where we have used Lemma \ref{lem:besov} to bound $\norm{\bm{M}_k}_{L^1}$.
Summing in $j$ and applying $P_k$, we have 
\begin{align*}
2^{k(n+\alpha)}\norm{P_k J_{HL}}_{L^\infty} &\le c\, 2^{k(n+\alpha)}\sum_{j=k-2}^{k+2}c\,\sqrt{b_0b_1}\, 2^k\norm{P_j\bm{A}}_{L^\infty}\norm{P_{\le j-2}\bm{g}}_{L^\infty}\\
&\le c\,\sqrt{b_0b_1}\,\norm{\bm{A}}_{C^{n+1,\alpha}}\norm{\bm{g}}_{L^\infty}\,.
\end{align*}
We additionally have $\norm{P_{\le 0} J_{HL}}_{L^\infty}\le c\,\sqrt{b_0b_1}\,\norm{\bm{A}}_{L^\infty}\norm{\bm{g}}_{L^\infty}$.

Similarly, for $J_{HH}$, for $k\ge 3$ we have
\begin{align*}
P_k J_{HH} &= P_k\bigg(P_{k-5\le \ell\le k+5}T_{\wh{\bm{M}}}\sum_{j\ge k-2,\,\abs{i-j}\le 1,\,i+j\ge1}P_j\bm{A}\,P_i\bm{g}
-\sum_{j\ge k-2,\,\abs{i-j}\le 1,\,i+j\ge1}P_j\bm{A}\,P_{k-5\le \ell\le k+5}T_{\wh{\bm{M}}} P_i\bm{g}\bigg)\,.
\end{align*}
For each $j,i$ satisfying $j\ge k-2$, $\abs{i-j}\le 1$, $i+j\ge1$, as in \eqref{eq:JLH_integral_est}, we must estimate the integral 
\begin{align*}
&\abs{\int_\R\bigg(\bm{M}_k(s-s')\,P_j\bm{A}(s')\,P_i\bm{g}(s') 
- P_j\bm{A}(s)\,\bm{M}_k(s-s')\,P_i\bm{g}(s') \bigg)\,ds'} \\
&\qquad \le 2\norm{\bm{M}_k}_{L^1}\norm{P_j\bm{A}}_{L^\infty}\norm{P_i\bm{g}}_{L^\infty} 
\le c \,\sqrt{b_0b_1}\, 2^k\norm{P_j\bm{A}}_{L^\infty}\norm{P_i\bm{g}}_{L^\infty}\,.
\end{align*}
Again summing in $j$ and $i$ and applying $P_k$, we have
\begin{align*}
2^{k(n+\alpha)}\norm{P_k J_{HH}}_{L^\infty} &\le c \,\sqrt{b_0b_1}\, 2^{k(n+\alpha)}\sum_{j\ge k-2,\,\abs{i-j}\le 1,\,i+j\ge1} 2^k\norm{P_j\bm{A}}_{L^\infty}\norm{P_i\bm{g}}_{L^\infty}\\
&\le c\,\sqrt{b_0b_1} \sum_{j\ge k-2} 2^{k(n+1+\alpha)}2^{-j(n+1+\alpha)}\norm{\bm{A}}_{C^{n+1,\alpha}} 2^{-j\alpha}\norm{\bm{g}}_{C^{0,\alpha}}\\
&\le c\,\sqrt{b_0b_1} \,\norm{\bm{A}}_{C^{n+1,\alpha}}\norm{\bm{g}}_{C^{0,\alpha}}\,.
\end{align*}
Furthermore, for $P_{\le 2}J_{HH}$ we have $\norm{P_{\le 2} J_{HH}}_{L^\infty}\le c\,\sqrt{b_0b_1}\,\norm{\bm{A}}_{L^\infty}\norm{\bm{g}}_{L^\infty}$.

In total, we obtain Lemma \ref{lem:symbol_comm}.
\end{proof}

We may now prove Proposition \ref{prop:Phi_comm} by showing that the operator $\overline{\mc{L}}_\epsilon\p_s^2$ satisfies the assumptions \eqref{eq:M_conditions} of Lemma \ref{lem:symbol_comm}.

\begin{proof}[Proof of Proposition \ref{prop:Phi_comm}]
Recalling the definition \eqref{eq:Meps12_defs} of the multiplier $\wh{\bm{M}}_\epsilon^{(1)}(\xi)$, we may write $\overline{\mc{L}}_\epsilon\p_s^2$ as $T_{\wh{\bm{M}}}$ where 
\begin{equation}\label{eq:our_whM}
\wh{\bm{M}}(\xi)=\wh{\bm{M}}_\epsilon^{(1)}(\xi)\xi^2\,.
\end{equation}
Using Lemmas \ref{lem:mt_bounds} and \ref{lem:mn_bounds}, we immediately have that 
\begin{equation}\label{eq:mult_assumptions}
\abs{\phi_j(\xi)\wh{\bm{M}}(\xi)} \le c\,\epsilon^{-1}\abs{\log\epsilon}\,2^{j}\,, \qquad
\abs{\p_\xi^2(\phi_j(\xi)\wh{\bm{M}}(\xi))} \le c\,\epsilon^{-1}\abs{\log\epsilon}\,2^{-j} \,.
\end{equation}
We may thus apply Lemma \ref{lem:symbol_comm} with $b_0=b_1=c\,\epsilon^{-1}\abs{\log\epsilon}$ to obtain Proposition \ref{prop:Phi_comm}.
\end{proof}


\section{Proof of Lemma \ref{lem:new_alpha_est}: hypersingular bounds}\label{app:new_alphalem}

Using that $\varphi\in C^{1,\alpha}(\Gamma_\epsilon)$, we may expand $\varphi$ about $(s,\theta)$ as
\begin{equation}\label{eq:varphi_expand}
\begin{aligned}
\varphi(s-\bars,\theta-\bartheta)- \varphi(s,\theta) &= \textstyle -\bars \,\p_s\varphi(s,\theta) - \epsilon\bartheta\,\frac{1}{\epsilon}\p_\theta\varphi(s,\theta) \\
&\qquad+ \abs{\bars}^{1+\alpha} Q_{\varphi,s}(s,\theta,\bars,\bartheta) + |\epsilon\bartheta|^{1+\alpha} Q_{\varphi,\theta}(s,\theta,\bars,\bartheta) \,,
\end{aligned}
\end{equation} 
where 
\begin{equation}\label{eq:Qvarphi}
\norm{Q_{\varphi,s}}_{L^\infty}\le c\norm{\p_s\varphi}_{C^{0,\alpha}} \,, 
\qquad
\norm{Q_{\varphi,\theta}}_{L^\infty}\le \textstyle c\norm{\frac{1}{\epsilon}\p_\theta\varphi}_{C^{0,\alpha}} \,.
\end{equation}
Similarly, we may expand $\varphi$ about $(s_0+s,\theta_0+\theta)$ as
\begin{equation}\label{eq:varphi_expand2}
\varphi(s-\bars,\theta-\bartheta)- \varphi(s_0+s,\theta_0+\theta) 
=(s_0+\bars) H_{\varphi,s} + \epsilon(\theta_0+\bartheta)H_{\varphi,\theta}\,,
\end{equation}
but make note of two different expressions for the remainder terms $H_\varphi$. The first is given by
\begin{equation}\label{eq:rem1}
\begin{aligned}
(s_0+\bars)H_{\varphi,s} &= -(s_0+\bars)\p_s\varphi(s_0+s,\theta_0+\theta)\\
&\qquad +\abs{s_0+\bars}^{1+\alpha}Q_{\varphi,s}(s_0+s,\theta_0+\theta,s_0+\bars,\theta_0+\bartheta)\\ 
\epsilon(\theta_0+\bartheta)H_{\varphi,\theta}&= \textstyle -\epsilon(\theta_0+\bartheta)\,\frac{1}{\epsilon}\p_\theta\varphi(s_0+s,\theta_0+\theta)\\
&\qquad +\abs{\epsilon(\theta_0+\bartheta)}^{1+\alpha}Q_{\varphi,\theta}(s_0+s,\theta_0+\theta,s_0+\bars,\theta_0+\bartheta)\,,
\end{aligned}
\end{equation}
i.e. the expansion \eqref{eq:varphi_expand} evaluated at $(s_0+s,\theta_0+\theta,s_0+\bars,\theta_0+\bartheta)$. The second is given by 
\begin{equation}\label{eq:rem2}
\begin{aligned}
(s_0+\bars)H_{\varphi,s} &= -(s_0+\bars)\p_s\varphi(s,\theta)+\abs{s_0}^{1+\alpha}Q_{\varphi,s}(s,\theta,s_0+s,\theta_0+\theta)\\
&\qquad +\abs{\bars}^{1+\alpha}Q_{\varphi,s}(s,\theta,\bars,\bartheta)\\ 
\epsilon(\theta_0+\bartheta)H_{\varphi,\theta}&=\textstyle -\epsilon(\theta_0+\bartheta)\,\frac{1}{\epsilon}\p_\theta\varphi(s,\theta)+\abs{\epsilon\theta_0}^{1+\alpha}Q_{\varphi,\theta}(s,\theta,s_0+s,\theta_0+\theta) \\
&\qquad+ |\epsilon\bartheta|^{1+\alpha}Q_{\varphi,\theta}(s,\theta,\bars,\bartheta)\,.
\end{aligned}
\end{equation}
In both \eqref{eq:rem1} and \eqref{eq:rem2}, the terms $Q_{\varphi,\mu}$ are as in \eqref{eq:Qvarphi}.

Let $\wh I_\varphi$ denote the quantity of interest, i.e.
\begin{align*}
\wh I_\varphi&= \bigg|{\rm p.v.}\int_{-1/2-s_0}^{1/2-s_0}\int_{-\pi-\theta_0}^{\pi-\theta_0}I_{\varphi,\ell kn}(s_0+s,\theta_0+\theta,s_0+\bars,\theta_0+\bartheta)\,\epsilon\, d\bartheta d\bars \\
&\qquad\qquad\qquad\qquad - {\rm p.v.}\int_{-1/2}^{1/2}\int_{-\pi}^{\pi}I_{\varphi,\ell kn}(s,\theta,\bars,\bartheta)\,\epsilon\, d\bartheta d\bars\,\bigg|\,.
\end{align*} 
We begin by considering the case $\sqrt{s_0^2+\epsilon^2\theta_0^2}\ge \epsilon$. Using the first expansion \eqref{eq:rem1} for $\varphi$, we have 
\begin{align*}
\abs{\wh I_\varphi}
&\le 2\norm{{\rm p.v.}\int_{-1/2}^{1/2}\int_{-\pi}^{\pi}\frac{\bars^\ell(\epsilon\sin(\frac{\bartheta}{2}))^k\,g}{\abs{\bR}^n}\bigg(\bars\p_s\varphi(s,\theta)+\epsilon\bartheta\, \textstyle \frac{1}{\epsilon}\p_\theta\varphi(s,\theta) \bigg)\,\epsilon\, d\bartheta d\bars\,}_{L^\infty}\\
&\qquad +  2\norm{\int_{-1/2}^{1/2}\int_{-\pi}^\pi\frac{\bars^\ell(\epsilon\sin(\frac{\bartheta}{2}))^k\,g}{\abs{\bR}^n}\big( \abs{\bars}^{1+\alpha} Q_{\varphi,s} + |\epsilon\bartheta|^{1+\alpha} Q_{\varphi,\theta}\big)\,\epsilon\, d\bartheta d\bars }_{L^\infty} \\
&
\le c(\kappa_*,c_\Gamma)\,\epsilon^\alpha\bigg(\norm{g}_{C^{0,\alpha}_2}\big(\norm{\p_s\varphi}_{L^\infty}+\textstyle \norm{\frac{1}{\epsilon}\p_\theta\varphi}_{L^\infty}\big)+\norm{g}_{L^\infty}\big(\norm{\p_s\varphi}_{C^{0,\alpha}}+\textstyle \norm{\frac{1}{\epsilon}\p_\theta\varphi}_{C^{0,\alpha}}\big)\bigg) \\
&
\le c(\kappa_*,c_\Gamma)\,\sqrt{s_0^2+\epsilon^2\theta_0^2}^{\,\alpha}\bigg(
\norm{g}_{C^{0,\alpha}_2}\big(\norm{\p_s\varphi}_{L^\infty}+\textstyle \norm{\frac{1}{\epsilon}\p_\theta\varphi}_{L^\infty}\big)\\
&\qquad\qquad  + \norm{g}_{L^\infty}\big(\norm{\p_s\varphi}_{C^{0,\alpha}}+\textstyle \norm{\frac{1}{\epsilon}\p_\theta\varphi}_{C^{0,\alpha}}\big)\bigg)\,.
\end{align*}
Here we have used Lemmas \ref{lem:odd_nm} and \ref{lem:basic_est} and the fact that $\epsilon\le \sqrt{s_0^2+\epsilon^2\theta_0^2}$.

We next consider the case $\sqrt{s_0^2+\epsilon^2\theta_0^2}< \epsilon$. Throughout, we will use the notation
\begin{align*}
 \bR &= \bR(s,\theta,\bars,\bartheta)\,, \qquad \bR_0 = \bR(s_0+s,\theta_0+\theta,s_0+\bars,\theta_0+\bartheta)\,,\\
 g_0 &= g(s_0+s,\theta_0+\theta,s-\bars,\theta-\bartheta)\,, \qquad 
 g_0^0=g_0\big|_{(s_0+\bars,\theta_0+\bartheta)=(0,0)}\,, \qquad g^0=g\big|_{(\bars,\bartheta)=(0,0)}
 \end{align*} 
 to condense the exposition. 
Using the above notation as well as the expansions \eqref{eq:varphi_expand} and \eqref{eq:varphi_expand2}, we may write 
\begin{align*}
\wh I_\varphi &= J_{\varphi,s}+J_{\varphi,\theta} \,,\\
J_{\varphi,s}&= {\rm p.v.}\int_{-1/2-s_0}^{1/2-s_0}\int_{-\pi-\theta_0}^{\pi-\theta_0}\frac{(s_0+\bars)^\ell(\epsilon\sin(\frac{\theta_0+\bartheta}{2}))^k\,g_0}{\abs{\bR_0}^n}(s_0+\bars)H_{\varphi,s}\,\epsilon\,d\bartheta d\bars \\
&\qquad\qquad - {\rm p.v.}\int_{-1/2}^{1/2}\int_{-\pi}^{\pi}\frac{\bars^\ell(\epsilon\sin(\frac{\bartheta}{2}))^k\,g}{\abs{\bR}^n}\big(-\bars\p_s\varphi(s,\theta)+\abs{\bars}^{1+\alpha}Q_{\varphi,s}(s,\theta,\bars,\bartheta) \big) \,\epsilon\,d\bartheta d\bars\\
J_{\varphi,\theta}&= {\rm p.v.}\int_{-1/2-s_0}^{1/2-s_0}\int_{-\pi-\theta_0}^{\pi-\theta_0}\frac{(s_0+\bars)^\ell(\epsilon\sin(\frac{\theta_0+\bartheta}{2}))^k\,g_0}{\abs{\bR_0}^n}\epsilon(\theta_0+\bartheta)H_{\varphi,\theta}\,\epsilon\,d\bartheta d\bars \\
&\qquad  - {\rm p.v.}\int_{-1/2}^{1/2}\int_{-\pi}^{\pi}\frac{\bars^\ell(\epsilon\sin(\frac{\bartheta}{2}))^k\,g}{\abs{\bR}^n}\bigg(\textstyle -\epsilon\bartheta \,\frac{1}{\epsilon}\p_\theta\varphi(s,\theta)+|\epsilon\bartheta|^{1+\alpha}Q_{\varphi,\theta}(s,\theta,\bars,\bartheta) \bigg)\,\epsilon\, d\bartheta d\bars\,.
\end{align*}
In order to bound these terms, we will need to introduce a function $\bR_{\rm even}$ satisfying 
\begin{equation}\label{eq:Reven}
\abs{\bR_{\rm even}}^2 := \abs{\overline{\bm{R}}}^2 +\epsilon\bars^2Q_{R,0}(s,\theta) +\kappa_3\epsilon^2\bars\sin(\bartheta)\,,
\end{equation}
where $Q_{R,0}$ is as in \eqref{eq:Q_Rj}, and crucially does not depend on $\bars$ or $\bartheta$. In particular, since $|\barR|$ is even about both $\bars=0$ and $\bartheta=0$, this means that $\abs{\bR_{\rm even}}$ is also even about $\bars=0$ and $\bartheta=0$. For nonnegative integers $n$ and $m$, we thus have that the expression
\begin{equation}\label{eq:odd_kernel}
\frac{\bars^n(\epsilon\sin(\frac{\bartheta}{2}))^m}{\abs{\bR_{\rm even}}^{n+m+2}}\,, \quad n+m \text{ odd},
\end{equation}
is globally odd about $\bars=\bartheta=0$. Therefore we have that
\begin{equation}\label{eq:oddness}
{\rm p.v.}\int_{-1/2}^{1/2}\int_{-\pi}^{\pi}\frac{\bars^n(\epsilon\sin(\frac{\bartheta}{2}))^m }{\abs{\bR_{\rm even}}^{n+m+2}}\, \epsilon d\bartheta d\bars = 0\,,
\end{equation}
and thus kernels of the form \eqref{eq:odd_kernel} with constant-in-$\bars$-and-$\bartheta$ coefficients may be inserted into the integral expressions for $J_{\varphi,s}$ and $J_{\varphi,\theta}$.
We additionally note that 
\begin{equation}\label{eq:Rsq2}
\begin{aligned}
\abs{\bm{R}}^2 &=\textstyle \abs{\bR_{\rm even}}^2 + \bars^4Q_{R,1} +\epsilon\bars^3 Q_{R,2}+ \epsilon^2 \bars^2\sin(\frac{\bartheta}{2})Q_{R,3}\,,\\
\frac{1}{\abs{\bm{R}}}-\frac{1}{\abs{\bR_{\rm even}}} &= \frac{\bars^4Q_{R,1} +\epsilon\bars^3 Q_{R,2}+ \epsilon^2 \bars^2\sin(\frac{\bartheta}{2})Q_{R,3}}{\abs{\bR_{\rm even}}\abs{\bm{R}}(\abs{\bm{R}}+\abs{\bR_{\rm even}})}\,,
\end{aligned}
\end{equation}
where each $Q_{R,j}$ is as in \eqref{eq:Q_Rj} and, as in Lemma \ref{lem:Rests}, 
\begin{equation}\label{eq:Reven_low}
\abs{\bR_{\rm even}} \ge c(\kappa_*,c_\Gamma)|\barR|\,.
\end{equation}
We will use the notation $\bR_{\rm even,0}$ to denote $\bR_{\rm even}(s_0+s,\theta_0+\theta,s_0+\bars,\theta_0+\bartheta)$.

Now, in order to bound $J_{\varphi,s}$ and $J_{\varphi,\theta}$, we split the integrals into two regions:
\begin{align*}
R_1 &= \big\{ (\bars,\bartheta)\; : \; \sqrt{(s_0+\bars)^2+\epsilon^2(\theta_0+\bartheta)^2}\le 4\sqrt{s_0^2+\epsilon^2\theta_0^2} \big\}\,,\\
R_2 &= \big\{ (\bars,\bartheta)\; : \; \sqrt{(s_0+\bars)^2+\epsilon^2(\theta_0+\bartheta)^2}> 4\sqrt{s_0^2+\epsilon^2\theta_0^2} \big\}\,.
\end{align*}
For $j=s,\theta$, let $J_{\varphi,j,1}$ denote the integral of the integrand of $J_{\varphi,j}$ over the region $R_1$. Note that $\sqrt{\bars^2+\epsilon^2\bartheta^2}$ then belongs to the region $R_1'$ where $\sqrt{\bars^2+\epsilon^2\bartheta^2}\le 5\sqrt{s_0^2+\epsilon^2\theta_0^2}$.

After inserting $|\bR_{\rm even}|$ and $|\bR_{\rm even,0}|$ and using the first expansion \eqref{eq:rem1} for $\varphi$ about $(s_0+s,\theta_0+\theta)$, we may bound
\begin{align*}
\abs{J_{\varphi,s,1}} 
&\le \abs{\p_s\varphi(s_0+s,\theta_0+\theta)\,{\rm p.v.}\iint_{R_1} \textstyle (s_0+\bars)^{\ell+1}(\epsilon\sin(\frac{\theta_0+\bartheta}{2}))^k \displaystyle \bigg(\frac{g_0}{\abs{\bR_0}^n}-\frac{g_0^0}{\abs{\bR_{\rm even,0}}^n}\bigg)\,\epsilon\, d\bartheta d\bars} \\
&\qquad + \abs{\p_s\varphi(s,\theta)\,{\rm p.v.}\iint_{R_1'} \textstyle \bars^{\ell+1}(\epsilon\sin(\frac{\bartheta}{2}))^k \displaystyle \bigg(\frac{g}{\abs{\bR}^n}-\frac{g^0}{\abs{\bR_{\rm even}}^n}\bigg)\,\epsilon\, d\bartheta d\bars} \\
&\quad + \iint_{R_1}\frac{|s_0+\bars|^{\ell+1+\alpha}|\epsilon\sin(\frac{\theta_0+\bartheta}{2})|^k\,|g_0|}{\abs{\bR_0}^n}\abs{Q_{\varphi,s}} \,\epsilon\, d\bartheta d\bars
 + \iint_{R_1'} \frac{|\bars|^{\ell+1+\alpha}|\epsilon\sin(\frac{\bartheta}{2})|^k\,|g|}{\abs{\bR}^n}\abs{Q_{\varphi,s}}\,\epsilon\, d\bartheta d\bars \\
&\le \norm{\p_s\varphi}_{L^\infty}\norm{g}_{C^{0,\alpha}_2}\iint_{R_1} \frac{1}{\abs{\bR_0}^{2-\alpha}}\,\epsilon\, d\bartheta d\bars 
+ \norm{\p_s\varphi}_{L^\infty}\norm{g}_{C^{0,\alpha}_2}\iint_{R_1'} \frac{1}{\abs{\bR}^{2-\alpha}}\, \epsilon\, d\bartheta d\bars  \\
&\quad + \norm{\p_s\varphi}_{L^\infty}\abs{\bigg({\rm p.v.}\iint_{R_1} (s_0+\bars)^{\ell+1}\textstyle (\epsilon\sin(\frac{\theta_0+\bartheta}{2}))^k\displaystyle \bigg(\frac{1}{\abs{\bR_0}^n}-\frac{1}{\abs{\bR_{\rm even,0}}^n}\bigg) \,\epsilon\,d\bartheta d\bars\bigg) g_0^0}\\
&\quad + \norm{\p_s\varphi}_{L^\infty}\abs{\bigg({\rm p.v.}\iint_{R_1'} \bars^{\ell+1}\textstyle(\epsilon\sin(\frac{\bartheta}{2}))^k\displaystyle \bigg(\frac{1}{\abs{\bR}^n}-\frac{1}{\abs{\bR_{\rm even}}^n}\bigg)\,\epsilon\, d\bartheta d\bars\bigg)g^0 }  \\
&\qquad + \norm{\p_s\varphi}_{C^{0,\alpha}}\norm{g}_{L^\infty}\bigg(\iint_{R_1}\frac{1}{\abs{\bR_0}^{2-\alpha}} \,\epsilon d\bartheta d\bars + \iint_{R_1'} \frac{1}{\abs{\bR}^{2-\alpha}}\,\epsilon d\bartheta d\bars\bigg) \\
&\le c\big(\norm{\p_s\varphi}_{L^\infty}\norm{g}_{C^{0,\alpha}_2}+\norm{\p_s\varphi}_{C^{0,\alpha}}\norm{g}_{L^\infty}\big)\bigg(\iint_{R_1}\frac{1}{\abs{\bR_0}^{2-\alpha}} \,\epsilon\, d\bartheta d\bars + \iint_{R_1'} \frac{1}{\abs{\bR}^{2-\alpha}}\, \epsilon\, d\bartheta d\bars\bigg) \\
&\le c(\kappa_*,c_\Gamma)\big(\norm{\p_s\varphi}_{L^\infty}\norm{g}_{C^{0,\alpha}_2}+\norm{\p_s\varphi}_{C^{0,\alpha}}\norm{g}_{L^\infty}\big)\bigg(\iint_{\rho\le4\sqrt{s_0^2+\epsilon^2\theta_0^2}}\frac{1}{\rho^{2-\alpha}} \,\rho d\rho d\phi \\
&\hspace{5cm}  + \iint_{\rho\le 5\sqrt{s_0^2+\epsilon^2\theta_0^2}} \frac{1}{\rho^{2-\alpha}}\,\rho d\rho d\phi\bigg) \\
&\le c(\kappa_*,c_\Gamma)\big(\norm{\p_s\varphi}_{L^\infty}\norm{g}_{C^{0,\alpha}_2}+\norm{\p_s\varphi}_{C^{0,\alpha}}\norm{g}_{L^\infty}\big)\,\sqrt{s_0^2+\epsilon^2\theta_0^2}^{\,\alpha} \,.
\end{align*}
Here in the second-to-last line we have used that $\abs{\bR}\ge c(\kappa_*,c_\Gamma)\sqrt{\bars^2+\epsilon^2\bartheta^2}$ (see \cite[Lemma 3.1]{laplace}) to convert to polar coordinates $(\rho,\phi)$ over the patch $\rho\le c\sqrt{s_0^2+\epsilon^2\theta_0^2}<c\epsilon$. 
Similarly, we may show that $J_{\varphi,\theta,1}$ satisfies 
\begin{align*}
\abs{J_{\varphi,\theta,1}} &\le 
c(\kappa_*,c_\Gamma)\big(\textstyle \norm{\frac{1}{\epsilon}\p_\theta\varphi}_{L^\infty}\norm{g}_{C^{0,\alpha}_2}+\norm{\frac{1}{\epsilon}\p_\theta\varphi}_{C^{0,\alpha}}\norm{g}_{L^\infty}\big)\,\sqrt{s_0^2+\epsilon^2\theta_0^2}^{\,\alpha} \,.
\end{align*}

Finally, we bound $J_{\varphi,s,2}$ and $J_{\varphi,\theta,2}$ over the region $R_2$. We now use the second expansion \eqref{eq:rem2} of $\varphi$ about $(s_0+s,\theta_0+\theta)$ to write $J_{\varphi,s,2}$ as
\begin{align*}
J_{\varphi,s,2}&= 
-\p_s\varphi(s,\theta)\iint_{R_2} \bigg(\frac{(s_0+\bars)^{\ell+1}(\epsilon\sin(\frac{\theta_0+\bartheta}{2}))^k\,(g_0-g_0^0)}{\abs{\bR_0}^n}-\frac{\bars^{\ell+1}(\epsilon\sin(\frac{\bartheta}{2}))^k\,(g-g^0)}{\abs{\bR}^n} \bigg)\,\epsilon\, d\bartheta d\bars \\
&\qquad -\p_s\varphi(s,\theta)\iint_{R_2} \frac{(s_0+\bars)^{\ell+1}(\epsilon\sin(\frac{\theta_0+\bartheta}{2}))^k\,(g_0^0-g^0)}{\abs{\bR_0}^n} \,\epsilon\, d\bartheta d\bars \\
&\qquad -\p_s\varphi(s,\theta)\,g^0 \iint_{R_2} \bigg( \textstyle (s_0+\bars)^{\ell+1}(\epsilon\sin(\frac{\theta_0+\bartheta}{2}))^k \displaystyle \bigg(\frac{1}{\abs{\bR_0}^n}-\frac{1}{\abs{\bR_{\rm even,0}}^n}\bigg) \\
&\qquad\qquad + \textstyle \bars^{\ell+1}(\epsilon\sin(\frac{\bartheta}{2}))^k\displaystyle \bigg(\frac{1}{\abs{\bR_{\rm even}}^n}-\frac{1}{\abs{\bR}^n}\bigg) \bigg)\,\epsilon\, d\bartheta d\bars \\
&\qquad + \abs{s_0}^{1+\alpha}Q_{\varphi,s}(s,\theta,s_0+s,\theta_0+\theta)\iint_{R_2} \frac{(s_0+\bars)^{\ell}(\epsilon\sin(\frac{\theta_0+\bartheta}{2}))^k\,g_0}{\abs{\bR_0}^n}\,\epsilon\, d\bartheta d\bars\,.
\end{align*}
Using \eqref{eq:Rsq2} to expand the differences $\frac{1}{\abs{\bR_0}^n}-\frac{1}{\abs{\bR_{\rm even,0}}^n}$ and $\frac{1}{\abs{\bR_{\rm even}}^n}-\frac{1}{\abs{\bR}^n}$, we have
\begin{align*}
\abs{J_{\varphi,s,2}} 
&\le \norm{\p_s\varphi}_{L^\infty}\norm{g}_{C^{0,\alpha}_2}\bigg|\iint_{R_2}\bigg(\frac{(s_0+\bars)^{\ell+1+\alpha}(\epsilon\sin(\frac{\theta_0+\bartheta}{2}))^k}{\abs{\bR_0}^n}-\frac{\bars^{\ell+1+\alpha}(\epsilon\sin(\frac{\bartheta}{2}))^k}{\abs{\bR}^n} \bigg)\,\epsilon\, d\bartheta d\bars\bigg|  \\
&\quad + c\norm{\p_s\varphi}_{L^\infty}\norm{g}_{C^{0,\alpha^+}_1}\iint_{R_2}\frac{\sqrt{s_0^2+\epsilon^2\theta_0^2}^{\,\alpha^+}\,|s_0+\bars|^{\ell+1}|\epsilon\sin(\frac{\theta_0+\bartheta}{2})|^k}{\abs{\bR_0}^n}\,\epsilon\, d\bartheta d\bars  \\
&\quad + \norm{\p_s\varphi}_{L^\infty}\norm{g}_{L^\infty}\bigg|\iint_{R_2} \bigg(\frac{\bars^5Q_{R,1} +\epsilon\bars^4 Q_{R,2}+ \epsilon^2 \bars^3\sin(\frac{\bartheta}{2})Q_{R,3} }{\abs{\bR_{\rm even}}\abs{\bm{R}}(\abs{\bm{R}}+\abs{\bR_{\rm even}})}\sum_{j=0}^{n-1}\frac{\bars^{\ell+1}(\epsilon\sin(\frac{\bartheta}{2}))^k}{\abs{\bR}^j\abs{\bR_{\rm even}}^{n-1-j}} \\
& - \frac{(s_0+\bars)^5Q_{R,1} +\epsilon(s_0+\bars)^4 Q_{R,2}+ \epsilon^2 (s_0+\bars)^3\sin(\frac{\theta_0+\bartheta}{2})Q_{R,3} }{\abs{\bR_{\rm even,0}}\abs{\bm{R}_0}(\abs{\bm{R}_0}+\abs{\bR_{\rm even,0}})}\sum_{j=0}^{n-1}\frac{(s_0+\bars)^{\ell+1}(\epsilon\sin(\frac{\theta_0+\bartheta}{2}))^k}{\abs{\bR_0}^j\abs{\bR_{\rm even,0}}^{n-1-j}}\bigg) \,\epsilon\, d\bartheta d\bars\bigg| \\
&\quad + c\abs{s_0}^{1+\alpha}\norm{\p_s\varphi}_{C^{0,\alpha}}\norm{g}_{L^\infty}\iint_{R_2} \frac{1}{\abs{\bR_0}^3}\,\epsilon d\bartheta d\bars\\
&\le c(\kappa_*)\norm{\p_s\varphi}_{L^\infty}\bigg(\norm{g}_{C^{0,\alpha}_2} \iint_{R_2}\frac{\sqrt{s_0^2+\epsilon^2\theta_0^2}}{\abs{\bR_0}^{3-\alpha}}\,\epsilon\, d\bartheta d\bars  
+ \norm{g}_{C^{0,\alpha^+}_1}\iint_{R_2}\frac{\sqrt{s_0^2+\epsilon^2\theta_0^2}^{\,\alpha^+}}{\abs{\bR_0}^2}\,\epsilon\, d\bartheta d\bars\bigg) \\
&\quad 
+ c(\kappa_*)\norm{\p_s\varphi}_{L^\infty}\norm{g}_{L^\infty}\bigg(\iint_{R_2} \frac{\sqrt{s_0^2+\epsilon^2\theta_0^2}}{\abs{\barR_0}^2}\,\epsilon \,d\bartheta d\bars
+\iint_{R_2} \frac{\sqrt{s_0^2+\epsilon^2\theta_0^2}^{\,\alpha}}{\abs{\barR_0}}\,\epsilon \, d\bartheta d\bars \bigg)\\
&\qquad + c\norm{\p_s\varphi}_{C^{0,\alpha}}\norm{g}_{L^\infty}\iint_{R_2} \frac{\abs{s_0}^{1+\alpha}}{\abs{\bR_0}^3}\,\epsilon \,d\bartheta d\bars\\
&\le c(\kappa_*)\norm{\p_s\varphi}_{L^\infty}\big(\norm{g}_{C^{0,\alpha}_2}+\norm{g}_{C^{0,\alpha^+}_1}\big)\iint_{\rho>4\sqrt{s_0^2+\epsilon^2\theta_0^2}}\bigg(\frac{\sqrt{s_0^2+\epsilon^2\theta_0^2}}{\rho^{3-\alpha}}+ \frac{\sqrt{s_0^2+\epsilon^2\theta_0^2}^{\,\alpha^+}}{\rho^2}\bigg)\,\rho d\rho d\phi \\
&\quad 
+ c(\kappa_*)\norm{\p_s\varphi}_{C^{0,\alpha}}\norm{g}_{L^\infty}\iint_{\rho>4\sqrt{s_0^2+\epsilon^2\theta_0^2}}\bigg( \frac{\sqrt{s_0^2+\epsilon^2\theta_0^2}}{\rho^2}+\frac{\sqrt{s_0^2+\epsilon^2\theta_0^2}^{\,\alpha}}{\rho}+\frac{\abs{s_0}^{1+\alpha}}{\rho^3} \bigg)\,\rho d\rho d\phi \\
&\le c(\kappa_*)\norm{\p_s\varphi}_{L^\infty}\big(\norm{g}_{C^{0,\alpha}_2}+\norm{g}_{C^{0,\alpha^+}_1}\big)\bigg(\sqrt{s_0^2+\epsilon^2\theta_0^2}\big(1+\sqrt{s_0^2+\epsilon^2\theta_0^2}^{\,-1+\alpha}\big) \\
&\qquad\qquad\qquad  +\sqrt{s_0^2+\epsilon^2\theta_0^2}^{\,\alpha^+}\big(1+ \abs{\log(s_0^2+\epsilon^2\theta_0^2)}\big)\bigg) \\
&\quad 
+ c(\kappa_*)\norm{\p_s\varphi}_{C^{0,\alpha}}\norm{g}_{L^\infty}\bigg(\sqrt{s_0^2+\epsilon^2\theta_0^2}\big(1+ \abs{\log(s_0^2+\epsilon^2\theta_0^2)}\big) \\
&\qquad\qquad\qquad  +\sqrt{s_0^2+\epsilon^2\theta_0^2}^{\,\alpha}\big(1+\sqrt{s_0^2+\epsilon^2\theta_0^2}\big)+ \abs{s_0}^{1+\alpha}\big(1+\sqrt{s_0^2+\epsilon^2\theta_0^2}^{\,-1}\big)\bigg) \\
&\le c(\kappa_*)\bigg(\norm{\p_s\varphi}_{L^\infty}\big(\norm{g}_{C^{0,\alpha}_2}+\norm{g}_{C^{0,\alpha^+}_1}\big) + \norm{\p_s\varphi}_{C^{0,\alpha}}\norm{g}_{L^\infty}\bigg)  \sqrt{s_0^2+\epsilon^2\theta_0^2}^{\,\alpha}\,.
\end{align*}
Here in the second inequality we have used the closeness of the kernels with respect to $s_0$ and $\epsilon\theta_0$, and in the third inequality we have again switched to polar coordinates. In the fourth inequality we use that $\alpha^+>\alpha$ to absorb the logarithmic term in $\sqrt{s_0^2+\epsilon^2\theta_0^2}^{\,\alpha^+}\abs{\log(s_0^2+\epsilon^2\theta_0^2)}$. 

By an analogous series of estimates, it may be shown that $J_{\varphi,\theta,2}$ satisfies 
\begin{align*}
c(\kappa_*)\bigg(\textstyle \norm{\frac{1}{\epsilon}\p_\theta\varphi}_{L^\infty}\big(\norm{g}_{C^{0,\alpha}_2}+\norm{g}_{C^{0,\alpha^+}_1}\big) + \norm{\frac{1}{\epsilon}\p_\theta\varphi}_{C^{0,\alpha}}\norm{g}_{L^\infty}\bigg)  \sqrt{s_0^2+\epsilon^2\theta_0^2}^{\,\alpha}\,.
\end{align*}
Combining the estimates from cases 1, 2a, and 2b, we have  
\begin{align*}
\abs{\wh I_\varphi} &\le c(\kappa_*,c_\Gamma)\bigg(\big(\textstyle \norm{\p_s\varphi}_{L^\infty}+ \norm{\frac{1}{\epsilon}\p_\theta\varphi}_{L^\infty}\big)\big(\norm{g}_{C^{0,\alpha}_2}+\norm{g}_{C^{0,\alpha^+}_1}\big) \\
&\qquad + \textstyle \big(\norm{\p_s\varphi}_{C^{0,\alpha}}+ \norm{\frac{1}{\epsilon}\p_\theta\varphi}_{C^{0,\alpha}}\big)\norm{g}_{L^\infty}\bigg)  \sqrt{s_0^2+\epsilon^2\theta_0^2}^{\,\alpha}\,.
\end{align*}
\hfill \qedsymbol


\subsubsection*{Acknowledgments} LO acknowledges support from NSF grant DMS-2406003 and from the Wisconsin Alumni Research Foundation.


\bibliographystyle{abbrv} 
\bibliography{StokesBib}


\end{document}